\documentclass[a5paper,twoside,openright,11pt,titlepage]{scrbook}
\usepackage[paperheight=24cm,paperwidth=17cm]{geometry}
\usepackage[T1]{fontenc}
\usepackage[utf8]{inputenc}
\usepackage[english]{babel}
\usepackage{amsmath}
\usepackage{amssymb}
\usepackage{amsrefs}
\usepackage{amsthm,thmtools}
\usepackage{setspace}
\usepackage{graphicx}
\usepackage{tikz-cd}
\usepackage{comment}
\usepackage{csquotes}
\usepackage{url}
\usepackage[a-2b]{pdfx}

\tikzcdset{arrow style=tikz, diagrams={>=Straight Barb}} % per avere le punte delle frecce in tikzcd simili a quelle nel resto del testo
\usepackage{xcolor}
\usepackage{bbm}
\usepackage{mdframed}
\usepackage{libertine}
\usepackage[libertine]{newtxmath}
\hypersetup{colorlinks=true, linkcolor=blue, citecolor=blue, filecolor=black, urlcolor=blue}
\everymath{\displaystyle}
\let\cleardoublepage=\clearpage

\declaretheoremstyle[headfont=\bfseries,bodyfont=\itshape]{myplain}
\declaretheoremstyle[headfont=\bfseries,bodyfont=\normalfont]{mydefinition}
\declaretheoremstyle[headfont=\bfseries,bodyfont=\normalfont,qed=$\diamond$]{myexample}
\declaretheoremstyle[headfont=\itshape,bodyfont=\normalfont, qed=$\diamond$]{myremark}

\declaretheorem[style=plain, numberwithin=section, name=Theorem]{theo}
\declaretheorem[style=plain, sharenumber=theo, name=Lemma]{lemma}
\declaretheorem[style=plain, sharenumber=theo, name=Corollary]{coroll}
\declaretheorem[style=plain, sharenumber=theo, name=Proposition]{prop}
\declaretheorem[style=plain, sharenumber=theo, name=Proposition/Definition]{prop/def}
\declaretheorem[style=plain, sharenumber=theo, name=Principle]{principle}

\declaretheorem[style=definition, sharenumber=theo, name=Definition]{definition}

\declaretheorem[style=myexample, sharenumber=theo, name=Example]{example}

\declaretheorem[style=myremark, sharenumber=theo, name=Remark]{rem}

\numberwithin{equation}{section}

\everymath{\displaystyle}
\newcommand{\dA}{d_{D}}
\newcommand{\OA}{\Omega_{D}}
\newcommand{\rank}{\operatorname{rank}}
\newcommand{\im}{\operatorname{im}}
\newcommand{\mk}{\mathsf{M}K}
\newcommand{\mc}{\mathsf{M}R}
\newcommand{\pr}{\mathrm{pr}}
\newcommand{\filleddiamond}{\text{\scalebox{0.65}{$\blacksquare$}}}
\newcommand{\la}{\langle\hspace{-3pt}\langle}
\newcommand{\ra}{\rangle\hspace{-3pt}\rangle}
\newcommand{\leftq}{[\hspace{-3pt}[}
\newcommand{\rightq}{]\hspace{-3pt}]}
\newcommand{\tot}{\operatorname{Tot}}
\newcommand{\ttot}{\operatorname{tTot}}

\title{Shifted Contact Structures on Differentiable Stacks}
\author{Antonio Maglio}
\begin{document}
	
	\frontmatter
		%\maketitle
	%\input{copertina}
	%\newgeometry{left=2.5cm, right=2.5cm, top=3cm, bottom=3cm} 

	% ------------------- FRONTESPIZIO -------------------
	\begin{titlepage}
		\centering
		% Logo dell'Università
		\includegraphics[width=0.22\textwidth]{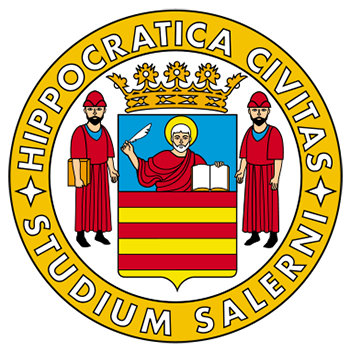}\\[0.8cm] % Inserisci il logo della tua università
		
		% Nome dell'Università e Dipartimento
		{\Large \textbf{Università degli Studi di Salerno}}\\[0.2cm]
		%\rule{\linewidth}{0.2mm}\\[0.3cm]
		{\large \textbf{Dipartimento di Matematica}}\\[0.8cm]
		
		% Denominazione del Corso di Dottorato e Ciclo
		{\Large \textbf{Dottorato di Ricerca in Matematica, Fisica ed Applicazioni}}\\[0.2 cm]
		{\large \textbf{Curriculum Matematica (MATH-02/B)}}\\[0.2cm]
		{\large \textbf{Ciclo XXXVII}}\\[0.7cm]
		
		{\Large \textbf{Tesi di Dottorato}}\\[0.3cm]
		% Titolo della tesi
		\rule{\linewidth}{0.5mm}\\[0.3cm]
		{\Huge \textbf{Shifted Contact Structures on}}\\[0.3cm]
		{\Huge \textbf{Differentiable Stacks}}\\[0.3cm]
		\rule{\linewidth}{0.5mm}\\[0.8cm]
		
%		\begin{minipage}{0.3\textwidth}
%			{\large\textbf{Tutor:}}\\
%			{\large Prof.~Luca Vitagliano}\\[0.8cm]
%			\rule{\linewidth}{0.2 mm}\\
%			{\large Dott.~Alfonso G.~Tortorella}\\[0.8cm]
%			\rule{\linewidth}{0.2 mm}\\[0.5cm]
%		\end{minipage}%	
%		\hfill
%		\begin{minipage}{0.3\textwidth}
%			{\large\textbf{Dottorando:}}\\
%			{\large Antonio Maglio}\\[0.8cm]
%			\rule{\linewidth}{0.2 mm}\\
%			{\large Matricola: 8803000024}\\[1.3cm]
%		\end{minipage}%\hspace{-1.2cm}
		
		\begin{tabular}{p{0.4\textwidth} @{\hspace{4.5cm}} p{0.3\textwidth}}
			\textbf{\large Tutor:} & \textbf{\large Dottorando:} \\[0.2cm]
			Prof.~Luca Vitagliano & Antonio Maglio \\[0.2cm]
			%\rule{\linewidth}{0.2mm} & \rule{\linewidth}{0.2mm} \\[0.2cm]
			Dott.~Alfonso G.~Tortorella %& Matricola: 8803000024 \\[1cm]
			%\rule{\linewidth}{0.2mm} & \\[0.5cm]
		\end{tabular}
		\vspace{1cm}

		\begin{minipage}{0.3\textwidth}
			\centering
			{\large\textbf{Coordinatore:}}\\[0.2cm]
			{Prof.~Vincenzo Tibullo}\\%[0.5cm]
			%\rule{\linewidth}{0.2 mm}\\
		\end{minipage}\\
		\vspace{1.2cm}
		% Anno accademico
		{\Large \textbf{Anno Accademico: 2023/2024}}
	\end{titlepage}

	\newpage
	\thispagestyle{empty}
	\null
	\newpage
	
	%Dedication
	\cleardoublepage
	\thispagestyle{empty}
	\vspace*{\stretch{1}}
	\begin{flushright}
		\itshape
		A Grottolella
	\end{flushright}
	\vspace{\stretch{3}}
	
	\newpage
	\thispagestyle{empty}
	\null
	\newpage
	
	\chapter*{Abstract}
	This thesis focuses on developing "stacky" versions of contact structures, extending the classical notion of contact structures on manifolds. A fruitful approach is to study contact structures using line bundle-valued $1$-forms. Specifically, if $L$ is a line bundle over a manifold $M$, a contact structure on $M$ is described by a nowhere-zero $L$-valued $1$-form, ensuring that its kernel defines a well-defined distribution on $M$. Additionally, the curvature of this $1$-form must be non-degenerate.
	
	On the other hand, differentiable stacks are mathematical objects used to model singular spaces, including orbifolds, leaf spaces of foliations, and orbit spaces of Lie group actions. These stacks can be described as equivalence classes of Lie groupoids under Morita equivalence, meaning that geometry on differentiable stacks involves considering Morita-invariant geometric structures on Lie groupoids.
	
	We introduce the notions of $0$ and $+1$-shifted contact structures on Lie groupoids. Since the property of being nowhere-zero is not Morita invariant, we seek a Morita invariant replacement for the kernel of a generic line bundle-valued $1$-form. To define the kernel of a line bundle-valued $1$-form $\theta$ on a Lie groupoid, we draw inspiration from the concept of the \emph{homotopy kernel} in Homological Algebra. That kernel is essentially given by a representation up to homotopy (RUTH). Similarly, the curvature is described by a specific RUTH morphism. The RUTH defining the kernel is concentrated in degrees $-1$, $0$, and $1$ in the $0$-shifted case, and in degrees $-1$ and $0$ in the $+1$-shifted case. Notably, similarly to the symplectic case, using the correspondence between 2-term RUTHs and vector bundle groupoids (VBGs), we can promote the RUTH in the $+1$-shifted case to a plain VBG, and the curvature can be promoted to a VBG morphism. All these constructions are Morita invariant, ensuring their well-definedness in the context of differentiable stacks, and they simplify when $\theta$ is nowhere-zero.
	
	Both the definitions are motivated by the Symplectic-to-Contact Dictionary, which establishes a relationship between Symplectic and Contact Geometry.
	
	Examples of $0$-shifted contact structures can be found in contact structures on orbifolds, while examples of $+1$-shifted contact structures include the prequantization of $+1$-shifted symplectic structures and the integration of Dirac-Jacobi structures.
%\end{document}
	\newpage
	\thispagestyle{empty}
	\null
	\newpage
	\chapter*{Acknowledgments}
This thesis would not have been possible without the help of my advisors, Luca Vitagliano and Alfonso Tortorella. Thank you for the problem you proposed to me. I did not expect it to be so interesting or that I would enjoy working on it so much.

I am deeply grateful to Luca for guiding me throughout my academic journey, for all the things he has taught me over the years, for his patience with me 
and for the trust he has shown in me every time we met. I promise that in my next life, I will write the famous ``quadernetto'' you asked for throughout my PhD.

I am also very grateful to Alfonso for all the time we spent discussing my doubts, for his kindness, and for always listening to me when I needed to talk about my problems. Alfonso has given me a lot of valuable advice over the years, especially every time I was complaining about all the things I had to do.

A huge thank you to my family. My father, my mother, and my sisters have been incredibly patient and kind to me. Thank you for the lunches you prepared for me every day so that I could eat at the university.

I am also grateful to my friends from my hometown. Going out and playing football with you was the best way to forget about work during my free time.

I am thankful to the rest of the Geometry group at the University of Salerno, especially Antonio and Chiara. Thanks also to all the PhD students and postdocs in the Math Department. In particular, I would like to thank Marco, Marialaura, Martina, Mikel, and Vincenzo for the chess games during lunch, all the coffees together, and many other memorable moments.

I would like to acknowledge the Math Department in Göttingen, where I spent three months and learned a lot. Thanks to Chenchang Zhu for giving me the opportunity to be there and to Miquel Cueca for the very interesting questions he asked me. Thanks also to the Math Department in São Paulo, especially to Cristian Ortiz for giving me the opportunity to visit and for all the stimulating discussions we had.

During these experiences, I also made two great friends, Ilias Ermeidis and Fabricio Valencia, with whom I spent a lot of time. My stays in Göttingen and São Paulo were much more enjoyable thanks to you.

A special thanks again to Cristian Ortiz and Henrique Bursztyn for taking the time to read my thesis and for the very useful suggestions they gave me.

Finally, I would like to thank GNSAGA of INdAM for financially supporting me, allowing me to attend many conferences, learn a lot, and meet new people.

	\tableofcontents
	
%	\cleardoublepage
%	\newpage
%	\thispagestyle{empty}
%	\null
%	\newpage

	\mainmatter
	\pagestyle{plain}      % revert to "plain" page style
	\pagenumbering{arabic}
	%\addcontentsline{toc}{chapter}{Introduction}

\chapter*{Introduction}\label{ch:intro}
\addcontentsline{toc}{chapter}{Introduction}

Manifolds exhibit two primary types of singularities. First, the category of manifolds is not closed under fiber products or intersections of submanifolds. Second, it is not closed under quotients by Lie group actions. Indeed, both fiber products and quotients of manifolds are generically singular spaces. To address these limitations, the category of manifolds is enlarged. For the first type of singularities, those coming from fiber products, the category of manifolds is replaced by the broader category of \emph{derived manifolds}, which accommodates fiber products. For the second type of singularities, those coming from quotients, manifolds are replaced by \emph{(higher) differentiable stacks}. 

In recent decades, significant progress has been made in extending the definitions and properties of geometric objects from the realm of manifolds to that of derived manifolds and differentiable stacks. This PhD thesis focuses on the second type of singularity, i.e., quotient manifolds, and thus centers on the study of differentiable stacks. 

Differentiable stacks are the Differential Geometry analogue of \emph{algebraic stacks} in \emph{Algebraic Geometry} \cite{BXu11}. As we already mentioned, they provide models for certain singular spaces in Differential Geometry, such as orbifolds, leaf spaces of foliations, and orbit spaces of smooth Lie group actions on manifolds. A very useful description of differentiable stacks is through equivalence classes of \emph{Lie groupoids} under \emph{Morita equivalence}. Essentially, making geometry on a differentiable stack means making \emph{Morita invariant} geometry on Lie groupoids. Consequently, geometric structures on differentiable stacks roughly correspond to Morita invariant geometric structures on Lie groupoids. 

Structures on differentiable stacks do usually possess a degree, determined by the categorical structure of the Lie groupoids through its nerve. This degree is commonly referred to as a \emph{shift}. Hence, the terminology \emph{shifted structures}. For example, when the stack corresponds to a manifold, $0$-shifted structures reduce to the familiar geometric structures on the manifold.

Numerous notions have already been generalized from the context of manifolds and Lie groups to the broader framework of differentiable stacks. These include concepts such as volume \cite{We09}, vector fields \cite{He09, OW19}, representations \cite{AC13}, Riemannian metrics \cite{dHF19}, measures \cite{CM19}, vector bundles \cite{dHO20}, equivariant cohomology \cite{Be04}, Poisson structures \cite{BCGX22}, coisotropic structures \cite{Ma24}, gerbes \cite{BXu03, LSX09, BXu11}, \emph{K}-theory \cite{TXL04}, symplectic structures \cite{Ge14, CZ23}, Morse theory \cite{OV24}, and others. 

The main goal of this thesis is to develop the ``stacky'' version of \emph{contact structures}, in particular we define \emph{$0$ and $+1$-shifted contact structures} on a differentiable stack. The origins of Contact Geometry are rooted in the geometric theory of PDEs, and their symmetries, as initiated by Lie and continued by Engel, Poincarè, and E.~Cartan among others. In the late 50s, it was formalized (Boothby-Wang, Gray, Reeb) as the geometry of a contact structure, i.e., a maximally non-integrable hyperplane distribution. 

A particularly effective formulation of contact structures involves line bundle-valued $1$-forms. Specifically, if $M$ is a manifold, then a contact structure is determined by a line bundle $L$ over $M$ and an $L$-valued $1$-form $\theta\in \Omega^1(M,L)$ that is nowhere-zero, with a non-degenerate curvature. The nowhere-zero property ensures that the kernel $K_\theta=\ker \theta$ of $\theta$ is a well-defined distribution on $M$. The curvature of $\theta$ is defined as
\[
	R_\theta\colon \wedge^2 K_\theta \to L, \quad R_\theta(X,Y)=\theta[X,Y].
\]

There is also a symplectic like-approach to Contact Geometry: if $M$ is a manifold, a contact structure can be equivalently described by a line bundle $L$ over $M$ and an $L$-valued $2$-form on the \emph{Atiyah algebroid}, the \emph{Lie algebroid} of \emph{derivations} of $L$, which is \emph{closed} and \emph{non-degenerate} in an suitable sense. This approach inspires a dictionary between Symplectic and Contact Geometry that we use to motivate the definition of \emph{shifted contact structures} that we introduce. Notably, shifted contact structures have recently been studied within the context of derived algebraic geometry as well \cite{BerktavA, BerktavB, BerktavC}.

To define shifted contact structures, we introduce a ``stacky'' version of all the constructions involved in the definition of contact structures on manifolds. Specifically, we make sense of line bundles, kernel and curvature of a line bundle-valued multiplicative $1$-form on a Lie groupoid in a Morita-invariant manner.

The Morita theory of vector bundles over Lie groupoids (VBGs) has been extensively studied in \cite{dHO20}. Building on this, we introduce the notion of a line bundle-groupoid (LBG) and interpret its Morita equivalence class as a line bundle in the category of stacks. The situation becomes more intricate when dealing with the kernel of a line bundle-valued $1$-form $\theta$ on a Lie groupoid. Since the property of $\theta$ being nowhere-zero is not Morita invariant, we do not rely on this property. Consequently, we cannot simply take the classical kernel of $\theta$. 

To address this challenge, we observe that, in both the $0$ and the $+1$-shifted cases, $\theta$ determines a cochain map between two appropriate complexes. Inspired by the construction of \emph{homotopy kernel} in \emph{Homological Algebra}, we define the kernel of $\theta$ as the mapping cone of this cochain map. The resulting complex is actually a representation up to homotopy (RUTH). In the $0$-shifted case, this RUTH is concentrated in degrees $-1$, $0$ and $+1$, while in the $+1$-shifted case, it is concentrated in degrees $-1$ and $0$. This allows us to use the correspondence between $2$-term RUTHs and VBGs \cite{GSM17}, enabling us to promote the RUTH in the $+1$-shifted case to a VBG. The definition of the ``kernel'' of $\theta$ as a VBG is more intrinsic then that as a RUTH. Similarly, the curvature of $\theta$ is a RUTH morphism, which, in the $+1$-shifted case, can be promoted to a VBG morphism.% As a side resul, we obtain a description of the kernel of a multiplicative line bundle-valued $1$-form. All these results closely parallel the symplectic case.

When $\theta$ is nowhere-zero all the constructions simplify significantly in a Morita invariant way. Contact structures on \emph{orbifolds} \cite{H13} provide natural examples of $0$-shifted contact structures, while the \emph{prequantization} of a $+1$-shifted symplectic structure \cite{LGXu05} and the integration of \emph{Dirac-Jacobi structures} \cite{V18} yield examples of $+1$-shifted contact structures.

The thesis is organized into five chapters. Chapter \ref{ch:preliminaries} provides an overview of Lie groupoids, their infinitesimal counterparts, Lie algebroids, and Morita equivalence between Lie groupoids. Special attention is given to the \emph{Atiyah algebroid} of a vector bundle and its properties, as this structure is pivotal for the symplectic perspective on Contact Geometry that we present in Chapter \ref{ch:ssas}. 

Chapter \ref{ch:vbg} serves as a blend of foundational review and the introduction of new constructions, focusing on \emph{VB-groupoids} (vector bundles in the category of Lie groupoids). The chapter begins by recalling the concepts of VB-groupoids (VBGs) and \emph{VB-groupoid morphisms}, with an emphasis on the dual of a VB-groupoid and the relationships between VB-groupoid morphisms and their duals. A particular class of VB-groupoids, termed \emph{trivial core VB-groupoids}, is examined, with several key properties established. This section concludes with the introduction of \emph{line bundle-groupoids (LBGs)}, which, from our perspective, serve as the analog of line bundles in the category of Lie groupoids (Section \ref{sec:trivial_core}). The previous discussion allows us to introduce a new construction: dual VB-groupoids twisted by a trivial core one (Section \ref{sec:twisted_dual}). This is motivated by the fact that the role of twisted dual bundles in Contact Geometry is analogue to the role of the dual bundle in Symplectic Geometry. The section further explores how twisted dual VB-groupoids of related VB-groupoids interact under morphisms. The approach proceeds in two stages: 
\begin{itemize}
	\item We first define the \emph{tensor product} of a VB-groupoid and a trivial core VB-groupoid.
	\item We then combine the latter construction with the dualization process to define twisted dual VB-groupoids.
\end{itemize}
Similarly, for morphisms, tensor products of VB-groupoid morphisms with trivial core morphisms are defined and then combined with the dualization of morphisms.  The section concludes with a discussion of the \emph{Atiyah VB-groupoid} associated with a line bundle-groupoid. This structure plays a role in Contact Geometry analogous to the \emph{tangent VB-groupoid} in Symplectic Geometry.

Following the correspondence established in \cite{GSM17}, VBGs are linked to $2$-term \emph{representations up to homotopy (RUTHs)}. Section \ref{sec:RUTHs} recalls this correspondence and introduces operations such as \emph{tensor product RUTHs} and \emph{dual RUTHs twisted by a representation}, showing how these constructions align with the analogous operations on VB-groupoids.

The Morita theory of VB-groupoids, as presented in \cite{dHO20}, is revisited, with particular attention to line bundle-groupoids and twisted dual VB-groupoids. The chapter concludes by introducing a notion of \emph{linear transformations} in the context of VB-groupoids, culminating in a characterization theorem for VB-groupoid morphisms that cover the same base map.

Chapter \ref{ch:sss} focuses on the theory of $0$ and $+1$-shifted symplectic structures, emphasizing their definitions and Morita invariance. %The discussion begins with the construction of the \emph{Bott-Shulman-Stasheff double complex} associated with Lie groupoids, which provides the natural framework for defining closed shifted $2$-forms. These tools are essential for developing the concept of shifted symplectic structures.

Section \ref{sec:0_shifted} introduces $0$-shifted symplectic structures by recalling their \emph{non-degeneracy condition}, formulated as a quasi-isomorphism between the \emph{tangent complex} of a Lie groupoid and its dual. We recall in which precise sense $0$-shifted symplectic structures are Morita invariant, ensuring that they are well-defined on differentiable stacks.

In Section \ref{sec:1_shifted_s}, the focus moves to $+1$-shifted symplectic structures, first introduced by Xu \cite{Xu03} and independently by Bursztyn, Crainic, Weinstein and Zhu \cite{BCWZ04}, though with different terminology and motivations. Xu uses the terminology \emph{quasi-symplectic structures} and unifies, under one Morita invariant theory, different momentum map theories. Bursztyn et al. use the terminology \emph{twisted presymplectic structures} and prove that they are the global counterpart of \emph{Dirac structures} twisted by a closed $3$-form. Following \cite{Ge14, CZ23}, the non-degeneracy condition of a $+1$-shifted symplectic structure is similarly expressed through a quasi-isomorphism, this time between the tangent complex and its dual shifted by $+1$, or, following \cite{dHO20}, equivalently, through a global formulation involving a VBG morphism that is a \emph{Morita map} in the sense of \cite{dHO20}. In order to discuss $+1$-shifted symplectic structures, the section examines multiplicative $2$-forms, recalling their properties and their interplay with VBG morphisms and pullbacks induced by Morita maps. This prepares the groundwork for revisiting the Morita invariance of $+1$-shifted symplectic structures, a key result that extends Symplectic Geometry to differentiable stacks. This extension leads to the notion of $+1$-shifted symplectic stack. Lastly, the section recalls the relationship between twisted Dirac structures and $+1$-shifted symplectic groupoids, including a brief discussion of the \emph{AMM-groupoid} with its infinitesimal structure.

In Chapter \ref{ch:ssas}, we introduce a preliminary definition of shifted contact structures. Specifically, in Section \ref{sec:c_manifolds}, we recall the definition of contact structures and explain the Symplectic-to-Contact Dictionary. This dictionary has two equivalent formulations. The first employs \emph{symplectic Atiyah forms}, which are closed and non-degenerate $L$-valued $2$-forms on the Atiyah algebroid $DL$ of a line bundle $L$ \cite{V18}. The second uses \emph{homogeneous symplectic manifolds}, defined as principal bundles with structure group the multiplicative non-zero reals, equipped with a symplectic structure that is homogeneous of degree $1$ with respect to the principal action \cite{BGG17}. We present both approaches and recall their relationship, as detailed in \cite{VW20}. For our purposes, we prefer to adopt the Atiyah form perspective.

In the rest of the chapter we apply the Symplectic-to-Contact Dictionary to Chapter \ref{ch:sss}. Sections \ref{sec:Ashifted}, \ref{sec:0_shifted_A}, and \ref{sec:A+1} correspond to Sections \ref{sec:shifted_structures}, \ref{sec:0_shifted}, and \ref{sec:1_shifted_s}, respectively. Indeed, we begin by introducing the analogue of the Bott-Shulman-Stasheff double complex and defining \emph{closed shifted Atiyah $2$-forms}. We then focus on \emph{$0$-shifted symplectic Atiyah structures}, providing their definition and proving their Morita invariance. In the final section, we turn to \emph{$+1$-shifted symplectic Atiyah structures}, providing the definition and analyzing multiplicative Atiyah $2$-forms in detail. Finally, we address Morita equivalence by defining an appropriate notion of \emph{Symplectic Atiyah Morita equivalence} and proving results analogous to those established in the plain symplectic case.

The main aim of the last chapter is to provide a translation in more classical terms of the definitions of $0$ and $+1$-shifted symplectic Atiyah structures. Chapter \ref{ch:scs}, largely based on \cite{MTV24}, introduces the notions of $0$ and \emph{$+1$-shifted contact structures} using an independent approach and establishes their equivalence with the corresponding definitions involving Atiyah forms.

We begin by defining \emph{shifted LBG-valued $1$-forms}, introducing a complex associated with the nerve of an LBG and recalling by \cite{DE19} that this complex is Morita invariant. Using this, we define \emph{multiplicative} and \emph{basic} LBG-valued forms. The non-degeneracy condition for a closed $0$ or $+1$-shifted structure, as introduced in Chapter \ref{ch:sss}, is expressed as a quasi-isomorphism between the tangent complex and its twisted dual, or its shifted twisted dual in the $+1$-shifted case. Finding the notion of shifted contact structures requires precise definitions for the kernel and curvature of $0$ and $+1$-shifted LBG-valued $1$-forms.

To define the kernel, as we already mentioned, we draw inspiration from the \emph{homotopy kernel} construction in \emph{Homological Algebra} and define an appropriate RUTH to serve as the kernel of a shifted LBG-valued $1$-form. In the $0$-shifted case, this RUTH is concentrated in degrees $-1$, $0$, and $1$, while in the $+1$-shifted case, it is concentrated in degrees $-1$ and $0$, allowing it to be promoted to a VBG. These constructions are Morita invariant and simplify significantly when $\theta$ is nowhere-zero.

The curvature is then introduced as a RUTH morphism, which in the $+1$-shifted case can also be promoted to a VBG morphism. We analyze its Morita invariance and its behavior in the case when $\theta$ is nowhere-zero.

Finally, we present the formal definitions of $0$- and $+1$-shifted contact structures and illustrate them with examples. % Contact structures on \emph{orbifolds} \cite{H13} provide natural examples of $0$-shifted contact structures, while the \emph{prequantization} of a $+1$-shifted symplectic structure \cite{LGXu05} and the integration of \emph{Dirac-Jacobi structures} \cite{V18} yield examples of $+1$-shifted contact structures.

This work not only fills a significant gap in the theory of differentiable stacks but also lays the groundwork for future studies on contact geometry in this generalized settings, with potential applications in Mathematical Physics and Geometry.

The definitions of $0$- and $+1$-shifted contact structures form a foundational aspect of Shifted Geometry on differentiable stacks. These structures enable the exploration of new objects and properties within the realm of differentiable stacks. For example, the notion of shifted contact structures can be extended to the broader contexts of \emph{higher Lie groupoids} and \emph{higher differentiable stacks}, following similar advancements made for shifted symplectic structures in \cite{CZ23}. Alternatively, we can generalize other constructs from Contact Geometry, such as \emph{Jacobi structures}, the contact analogue of \emph{Poisson structures}, to develop a framework for \emph{shifted Jacobi structures}.

Another promising direction is to extend classical theorems about contact structures on manifolds to this higher framework. For instance, one can work in extending the Marsden-Weinstein reduction to differentiable stacks equipped with shifted contact structures, as already done for $0$-shifted symplectic structures in \cite{HS21}.

	\chapter{Lie Groupoids and Morita equivalence}\label{ch:preliminaries}

	In this chapter we introduce the main geometric objects of the thesis, i.e., \emph{differentiable stacks}. They serve as model for singular spaces, such as orbifolds, orbit spaces of Lie groups actions, leaf spaces of foliations, and others. Among the various ways to define differentiable stacks, we adopt the one that, from our perspective, is most practical and insightful: defining them as equivalence classes of Lie groupoids under an appropriate equivalence relation, namely \emph{Morita equivalence}. 
	
	To provide the necessary background, we first review Lie groupoids and their infinitesimal counterparts, Lie algebroids, in the first two sections. Finally, in the last section, we recall the notion of Morita equivalence between Lie groupoids and conclude with the definition of differentiable stacks. The main references for this chapter are \cites{CF11, MM03, dH13}.
	
	\section{Lie groupoids}
	\label{sec:Lie_groupoids}
	In this section, we introduce the notions of Lie groupoids and their morphisms. We provide several examples of Lie groupoids that will be relevant in the subsequent chapters. Furthermore, we introduce the concept of \emph{representation} of a Lie groupoid, which will play a central role in the developments presented in the rest of the thesis.
	
	To define the notion of a Lie groupoid, we first recall the concept of a groupoid and then introduce a differentiable structure that is compatible with the algebraic one.
	\begin{definition}
		A \emph{groupoid} is a (small) category in which every arrow is invertible.
	\end{definition}
	We indicate a groupoid by $G\rightrightarrows M$, where $G$ is the set of arrows and $M$ is the set of objects, and we say that $G$ is a groupoid over $M$. Sometimes we indicate a groupoid only by $G$ if it is clear from the context what is the set of objects. From the definition some underlying \emph{structure maps} come: 
	\begin{itemize}
		\item the \emph{source} and the \emph{target} maps $s,t\colon G\to M$ associate to every arrow $g\in G$ its source object $s(g)$ and its target object $t(g)$, respectively;
		\item  the \emph{composition map} $m\colon G^{(2)} \to G$, where
		\begin{equation*}
			G^{(2)} = G \mathbin{{}_{s}\times_{t}} G = \{(g,h) \, | \, s(g)= t(h)\},
		\end{equation*}
		is the set of composable arrows. We also indicate $m(g,h)=gh$ and we call $gh$ the multiplication of $g$ and $h$;
		\item the \emph{unit map} $u\colon M\to G$ sends $x$ to the identity arrow $1_x\in G$ at $x$;
		\item  the \emph{inverse map} $i\colon G\to G$ sends an arrow $g$ to its inverse $g^{-1}$.
	\end{itemize}  
	
	These structure maps satisfy the following properties:
	\begin{itemize}
		\item if $(g,h)\in G^{(2)}$, then $s(gh)=s(h)$ and $t(gh)=t(g)$;
		\item if $(g,h)$, $(h,k)\in G^{(2)}$, then $(gh)k=g(hk)$;
		\item if $g\colon x\to y$, then $1_xg=g$ and $g1_y=g$;
		\item if $g\colon x\to y$, then $g^{-1}\colon y\to x$, $g^{-1}g=1_x$ and $gg^{-1}=1_y$.
	\end{itemize}
	We indicate by $g\colon x\to y$ an arrow from $x$ to $y$, with $x,y\in M$ and we often identify $u(M)$ with $M$ and we indicate $1_x$ just by $x\in M$.
	
	If $G\rightrightarrows M$ is a groupoid then we get an equivalence relation on $M$: two objects $x,y\in M$ are equivalent if there exists an arrow $g\in G$ such that $s(g)=x$ and $t(g)=y$. The equivalence class of $x\in M$ is called the \emph{orbit} through $x$ that is denoted by $O_x= t(s^{-1}(x))= s(t^{-1}(x))$, or simply by $O$, and the space of orbits is indicated by $M/G=\{O_x \, | \, x\in M\}$ and called the \emph{orbit space}. For any $x\in M$, the preimages $s^{-1}(x)$ and $t^{-1}(x)$ are called the \emph{$s$-fiber} and the \emph{$t$-fiber} over $x$ respectively. Finally, the intersection $G_x=s^{-1}(x)\cap t^{-1}(x)$ of the $s$ and $t$-fibers over $x$ is a group called the \emph{isotropy group} at $x$. 	
	
	For any $g\colon x\to y\in G$, the \emph{right multiplication by $g$} is the bijection
	\begin{equation}
		\label{eq:right_multiplication}
		R_g\colon s^{-1}(y)\to s^{-1}(x), \quad h\mapsto hg, 
	\end{equation}
	and the \emph{left multiplication by g} is the bijection
	\begin{equation}
		\label{eq:left_multiplication}
		L_g \colon t^{-1}(x) \to t^{-1}(y), \quad h \mapsto gh.
	\end{equation}
	
	Let $G\rightrightarrows M$ and $H\rightrightarrows N$ be groupoids.
	\begin{definition}
		A \emph{groupoid morphism} $f\colon (H\rightrightarrows N) \to (G\rightrightarrows M)$ is a functor between the categories $H\rightrightarrows N$ and $G\rightrightarrows M$.
	\end{definition}
	From the definition, a groupoid morphism from $H\rightrightarrows N$ to $G\rightrightarrows M$ is a pair of maps $f\colon H\to G$, $f\colon N\to M$, that we indicate with the same letter, such that the following conditions are satisfied:
	\begin{itemize}
		\item if $h\in H$, then $s(f(h))= f(s(h))$ and $t(f(h))=f(t(h))$;
		\item if $(h,h')\in H^{(2)}$, then $f(hh')=f(h)f(h')$;
		\item if $y\in N$, then $f(1_y)= 1_{f(y)}$;
		\item if $h\in H$, then $f(h^{-1})= f(h)^{-1}$.
	\end{itemize}
	The last condition is just a consequence of the others.    
	\begin{rem}
		\label{rem:function_orbit_spaces}
		A groupoid morphism $f\colon (H\rightrightarrows N) \to (G\rightrightarrows M)$ induces a map between the orbit spaces $N/H$ and $M/G$ and group homomorphisms between the isotropy groups $H_y$ and $G_{f(y)}$, with $y\in N$, in the obvious way. %Indeed, let $O_x$ be the orbit in $H$ through $x\in N$. For any $y\in O_x$, there exists an arrow $h\colon x \to y$ in $H$, then $f(h)\colon f(x)\to f(y)$ and $f(y)\in O_{f(x)}$, the orbit in $G$ through $f(x)\in M$. Hence $f(O_x)\subseteq O_{f(x)}$, and $f$ induces the function $f\colon N/H\to M/G$ between the spaces of orbits, that maps any orbit $O_x$ to the orbit $O_{f(x)}$. 
		%Moreover, from the conditions satisfied by a groupoid morphism we easly see that, for any $x\in N$, $f(H_x)\subseteq G_{f(x)}$ and the restriction $f\colon H_x \to G_{f(x)}$ is a group homomorphism
	\end{rem}
	\begin{definition}
		A \emph{groupoid isomorphism} is an invertible groupoid morphism $f\colon (H\rightrightarrows N)\to (G\rightrightarrows M)$.
	\end{definition}
	%\begin{rem}
	%	\label{rem:inverse_groupoid_morphism}
	%	If $f\colon (H\rightrightarrows N)\to (G\rightrightarrows M)$ is a groupoid isomorphism then $f^{-1}\colon (G\rightrightarrows M)\to (H\rightrightarrows N)$ is a groupoid isomorphism. Indeed, let $g\in G$, then $s(f^{-1}(g))\in N$ is the only object such that
	%	\[
	%		f(s(f^{-1}(g)))= f(f^{-1}(s(g)))= s(g),
	%	\]
	%	then, $s(f^{-1}(g))=f^{-1}(s(g))$. Moreover, $t(f^{-1}(g))\in N$ is the only object such that
	%	\[
	%		f(t(f^{-1}(g)))= f(f^{-1}(t(g)))= t(g),
	%	\]
	%	then $t(f^{-1}(g))=f^{-1}(t(g))$. 
		
	%	If $(g,g')\in G^{(2)}$, then $f^{-1}(g)f^{-1}(g')\in H$ is the only arrow such that
	%	\[
	%		f(f^{-1}(g)f^{-1}(g'))= f(f^{-1}(g))f(f^{-1}(g'))=gg',
	%	\]
	%	then $f^{-1}(gg')=f^{-1}(g)f^{-1}(g')$.
		
	%	If $x\in M$ then $f^{-1}(1_x)\in H$ is the onlyh arrow such that $f(F^{-1}(1_x))=1_x$, then $f^{-1}(1_x)=1_{f^{-1}(x)}$. 
	%\end{rem}
	
	As already explained, a Lie groupoid is a groupoid together with a differentiable structure compatible with the algebraic structure. Precisely, we have the following 
	\begin{definition}
		\label{def:lie_morphism}
		A \emph{Lie groupoid} is a groupoid $G\rightrightarrows M$ whose underlying sets $G$ and $M$ are both manifolds, and whose structure maps $s,t,m,u,i$ are smooth and such that $s$ and $t$ are submersions. A \emph{Lie groupoid morphism} is a groupoid morphism between two Lie groupoids that is also a smooth map. A \emph{Lie groupoid isomorphism} is a Lie groupoid morphism $f\colon (H\rightrightarrows N)\to (G\rightrightarrows M)$ such that $f\colon H\to G$ and $f\colon N\to M$ are diffeomorphisms.
	\end{definition}
	%From Remark \ref{rem:inverse_groupoid_morphism} the inverse $f^{-1}$ of a Lie groupoid isomorphism $f$, is a Lie groupoid isomoprhism.
	
	\begin{rem}
		\label{rem:Husdorff}
		Sometimes, in the literature, the manifold $G$ is not required to be Hausdorff in such a way to include more natural examples and, when $G$ is Hausdorff, the Lie groupoid $G\rightrightarrows M$ is called \emph{Hausdorff Lie groupoid}. For simplicity, we work by default with Hausdorff Lie groupoids, but most of our results are also valid for the larger class of Lie groupoids admitting an Ehresmann connection (see Definition \ref{def:Ehresmann_connection}).
	\end{rem}
	
	Concerning the previous definition of Lie groupoid, notice that, since $s$ and $t$ are submersions, then $G^{(2)}$ is a submanifold of $G\times G$, and it makes sense to require that $m\colon G^{(2)}\to G$ is a smooth map. Moreover, for any $k\in \mathbbm{N}$, the set of \emph{$k$-composable arrows}
	\begin{equation*}
		G^{(k)} := \{(g_1, \dots,g_k) | \, s(g_i)=t(g_{i+1}), i=1,\dots, k-1\}
	\end{equation*}
	is a submanifold of $G^k$. Notice that, in particular, $G^{(1)}=G$. For any $k\geq 2$, $G^{(k)}$ comes together with $k+1$ smooth maps $\partial_i\colon G^{(k)}\to G^{(k-1)}$, $i=0, \dots, k$, defined by setting
	\begin{equation}
		\label{eq:facemaps}
		\partial_i(g_1, \dots,g_k) =
		\begin{cases}
			(g_2, \dots, g_k) & \text{when } i=0\\
			(g_1, \dots, g_{k-1}) & \text{when } i=k\\
			(g_1, \dots, g_ig_{i+1}, \dots, g_k) & \text{when } 0 < i<k
		\end{cases},
	\end{equation}
	and $k+1$ smooth maps $d_i\colon G^{(k)}\to G^{(k+1)}$, $i=0, \dots, k$, defined by setting
	\begin{equation}
		\label{eq:degmaps}
		d_i(g_1, \dots,g_k) =(g_1, \dots, g_i, 1_{s(g_i)}, g_{i+1}, \dots,g_k).
	\end{equation}

	If we set $G^{(0)}:=M$, $\partial_0 :=s$, $\partial_1:=t \colon G\to M$, and $d_0:=u\colon M \to G$, then the sequence of manifolds $G^{(\bullet)}$ together with the maps \eqref{eq:facemaps} and \eqref{eq:degmaps} is a \emph{simplicial manifold} with face maps $\partial_i$ and degeneracy maps $d_i$. It is called the \emph{nerve} of the Lie groupoid $G$. 
	
	\begin{rem}
		\label{rem:multiplication_diffeomorphism}
		Another easy consequence of $s$ and $t$ being submersions is that the $s$-fibers and $t$-fibers are manifolds and, for any $g\colon x\to y\in G$, the right multiplication $R_g\colon s^{-1}(y)\to s^{-1}(x)$, described in Equation \eqref{eq:right_multiplication}, and the left multiplication $L_g\colon t^{-1}(x)\to t^{-1}(y)$, described in Equation \eqref{eq:left_multiplication}, by $g$ are diffeomorphisms.
	\end{rem}
	
	\begin{rem}
		\label{rem:basic_facts}
		Let $G\rightrightarrows M$ be a Lie groupoid. Recall from \cite[Theorem 5.4]{MM03} that the orbit $O_x$ through $x\in M$ is not just a set but it is a (generically non-embedded) initial submanifold of $M$, and for any $x\in M$ the isotropy group $G_x$ at $x$ is a Lie group. Moreover, the orbit space $M/G$ is a topological space with the quotient topology.
		
		Finally, let $f\colon (H\rightrightarrows N)\to (G\rightrightarrows M)$ be a Lie groupoid morphism. The map $f\colon N/H \to M/G$ induced by $f$ (see Remark \ref{rem:function_orbit_spaces}) is a continous map, and, for any $y\in N$, the group homomorphism $f\colon H_y\to G_{f(y)}$ induced by $f$ (see Remark \ref{rem:function_orbit_spaces} again) is a Lie group morphism.
	\end{rem}
	\begin{rem}
		\label{rem:morphism_nerve}
		A Lie groupoid morphism $f\colon (H\rightrightarrows N)\to (G\rightrightarrows M)$ determines a morphism of simplicial manifolds between the nerve $H^{(\bullet)}$ of $H$ and the nerve $G^{(\bullet)}$ of $G$. For any $k\in \mathbbm{N}$, the map between $H^{(k)}$ and $G^{(k)}$, again denoted by $f$, is simply given by
		\[
			f(h_1, \dots, h_k)= (f(h_1), \dots, f(h_k)), \quad (h_1,\dots,h_k)\in H^{(k)}.
			\qedhere
		\]
		%It is easy too see that $f$ commutes with the face and degeneracy maps. Indeed, for any $(h_1,\dots, h_k)\in H^{(k)}$, we have
		%\[
		%\begin{aligned}
		%	f(\partial_0(h_1, \dots, h_k))= f(h_2, \dots, h_k)&= (f(h_2), \dots, f(h_k)) %\\
		%	&= \partial_0(f(h_1), \dots, f(h_k))= \partial_0(f(h_1, \dots, h_k)), \\
		%	f(\partial_k(h_1, \dots, h_k))= f(h_1, \dots, h_{k-1})&= (f(h_1), \dots, f(h_{k-1})) \\
		%	&= \partial_k(f(h_1), \dots, f(h_k))= \partial_k(f(h_1, \dots, h_k)),
		%\end{aligned}
		%\]
		%and 
		%\[
		%\begin{aligned}
		%	f(\partial_i(h_1, \dots, h_k))&= f(h_1, \dots, h_ih_{i+1}, \dots, h_k)\\
		%	&= (f(h_1),\dots, f(h_ih_{i+1}), \dots, f(h_{k-1})) \\
		%	&= (f(h_1), \dots, f(h_i)f(h_{i+1}), \dots ,f(h_k))\\
		%	&= \partial_i(f(h_1), \dots, f(h_k))\\
		%	&=\partial_i(f(h_1, \dots, h_k)),
		%\end{aligned}
		%\]
		%for all $0<i<k$. Finally, for any $(h_1,\dots,h_k)\in H^{(k)}$, we have
		%\[
		%\begin{aligned}
		%	f(d_i(h_1, \dots, h_k))&=f(h_1, \dots,s(h_i), \dots, h_k )\\
		%	& = (f(h_1), \dots, f(s(h_i)), \dots, f(h_k))\\
		%	&= (f(h_1), \dots, s(f(h_i)), \dots, f(h_k))\\
		%	&= d_i(f(h_1), \dots, f(h_k))\\
		%	&= d_i(f(h_1, \dots, h_k)),
		%\end{aligned}
		%\]
		%for all $0<i<k$.
	\end{rem}
	
	Now we provide several examples of Lie groupoids. We begin with the most classical examples and then introduce those that will be relevant in the rest of the thesis.
	
	First, observe that Lie groupoids and their morphisms form a category. This category generalizes those of Lie groups and manifolds, as illustrated by the following two examples.
	\begin{example}[Lie groups]
		Any Lie group $G$ is a Lie groupoid over a point $G\rightrightarrows \{\ast\}$. Source and target map every element in $G$ to the point $\ast$, the multiplication and the inverse are just the multiplication and the inverse of the Lie group $G$, and the unit maps the point $\ast$ to the unit in $G$. Lie groupoid morphisms between Lie groups are just Lie group morphisms.
	\end{example}
	\begin{example}[Unit groupoid]
		For any manifold $M$, $M\rightrightarrows M$ is a Lie groupoid whose structure maps are all the identity $\operatorname{id}_M\colon M \to M$. This is called the \emph{unit groupoid}. Orbits and isotropy groups of the unit groupoid are points and Lie groupoid morphisms between two unit groupoids are just smooth maps between manifolds.
	\end{example}
	
	The unit groupoid is not the only Lie groupoid coming from a manifold. Another example is discussed in the following
	\begin{example}[Pair groupoid]
		\label{ex:pair_groupoid}
		Let $M$ be a manifold. The \emph{pair groupoid} is the Lie groupoid $M\times M\rightrightarrows M$ where there is an unique arrow from $x\in M$ to $y\in M$ given by the pair $(y,x)$. The source and the target are the projections on the second and first factor respectively. The multiplication between $(z,y)$ and $(y,x)$ is the pair $(z,x)$. The inverse of a pair $(x,y)$ is the pair $(y,x)$ and, finally, the unit maps $x\in M$ to the pair $(x,x)$. This is an example of \emph{transitive} Lie groupoid, i.e., there is only one orbit given by the whole $M$. Isotropy groups are trivial and morphisms between pair groupoids are just smooth maps between manifolds.
	\end{example}
	 
	The latter example can be generalized to a more general one:
	\begin{example}[Submersion groupoid]
		\label{ex:submersion_groupoid}
		Let $f\colon M \to B $ be a surjective submersion, then $M\times_B M=\{(x,y)\in M\times M \, | \, f(x)=f(y)\}$ is a Lie groupoid over $M$ with the obvious structure maps giving by restricting those of the pair groupoid. This Lie groupoid is called the \emph{submersion groupoid}, and when $f$ is the identity $\operatorname{id}_M$ we recover the pair groupoid $M\times M$. Isotropy groups of the submersion groupoid are again trivial, and the orbits are just the fibers of $f\colon M\to B$.
	\end{example}
	 
	Another important property of Lie groupoids is that they capture the concept of Lie group actions. Specifically, any Lie group action gives rise to an example of a Lie groupoid. We illustrate this construction in the following
	\begin{example}[Action groupoid]
		\label{ex:action_groupoid}
		Let $G\times M\to M$, $(g,x)\mapsto g.x$, be a smooth action of the Lie group $G$ on the manifold $M$. Then $G\times M \rightrightarrows M$ is a Lie groupoid, called the \emph{action groupoid}, and denoted by $G\ltimes M$. The source is the projection on $M$ and the target is the action itself. The multiplication between $(h,g.x)$ and $(g,x)$ is the pair $(hg,x)$, the unit of $x\in M$ is $(e,x)$, where $e$ is the unit in $G$ and the inverse of $(g,x)$ is $(g^{-1},g.x)$. For each $x\in M$, the isotropy group $G_x$ agrees with the stabilizer group of $x$ with respect to the $G$-action and the orbit $O_x$ is just the orbit of $x$ under the $G$-action on $M$.
	\end{example}
	
	Sometimes it is useful to pull-back a Lie groupoid along a smooth map. However, the result is not always a Lie groupoid. We discuss it in the next example.
	\begin{example}[Pullback groupoid]
		\label{ex:pullback_groupoid}
		Let $G\rightrightarrows M$ be a Lie groupoid and let $f\colon N\to M$ be a smooth map. The pullback groupoid is the groupoid $f^{!}G \rightrightarrows N$ whose arrows are triples $(y_1, g, y_2)\in N \mathbin{{}_{f}\times_{t}} G \mathbin{{}_{s}\times_{f}} N$, the source and target maps are the projections on the last and on the first factor respectively, the multiplication is just the multiplication in $G$, the inverse of $(y_1,g,y_2)$ is the triple $(y_2,g^{-1},y_1)$, and the unit of $y\in N$ is the triple $(y,1_{f(y)}, y)$. Notice that the set $N \mathbin{{}_{f}\times_{t}} G \mathbin{{}_{s}\times_{f}} N$ is not always a manifold, but, it is so when the map $(f,f)\colon N \to M\times M$ is transverse to $(s,t)\colon G \to M\times M$. In this case $f^{!}G\rightrightarrows N$ is a Lie groupoid and the pair of smooth maps $(\pr_2, f)$, where $\pr_2\colon N \mathbin{{}_{f}\times_{t}} G \mathbin{{}_{s}\times_{f}} N \to G$ is the projection on $G$, is a Lie groupoid morphism from $f^{!}G$ to $G$.
	\end{example}

	If $f\colon H\to G$ and $f'\colon H'\to G$ are Lie groupoid morphisms, then, following \cite[Section 5.3]{MM03} we can define a groupoid given by the fiber product of $f$ and $f'$. This construction is too strong for our purposes, for this reason we recall the construction of a \lq \lq weak \rq \rq fiber product of $f$ and $f'$ that appears in \cite[Section 4.4]{dH13} under the name of \emph{homotopy pullback} and in \cite[Section 5.3]{MM03} under the name of \emph{weak fibred product}.
	\begin{example}[Homotopy fiber product]
		\label{ex:homotopy_pullback}
		Let $f\colon (H\rightrightarrows N)\to (G\rightrightarrows M)$, and $f'\colon (H' \rightrightarrows N')\to (G\rightrightarrows M)$ be two Lie groupoid morphisms. Following \cite[Section 4.4]{dH13}, or \cite[Section 5.3]{MM03}, we define the groupoid $H\times^h_G H'$ where the objects are triples $$\big(y_1, f(y_1)\leftarrow f'(y_2), y_2\big)\in N\mathbin{{}_{f}\times_{t}} G \mathbin{{}_{s}\times_{f'}} N'$$ and an arrow from $\big(x_1, f(x_1)\xleftarrow{g_1} f'(x_2), x_2\big)$ to $\big(y_1, f(y_1)\xleftarrow{g_2}f'(y_2), y_2\big)$ is a triple $$\left(y_1 \xleftarrow{k_1} x_1,k_2, y_2\xleftarrow{k_3} x_2\right)\in H \mathbin{{}_{s\circ f}\times_{t}} G \mathbin{{}_{s}\times_{t\circ f'}} H'$$ such that
		\begin{equation*}
			g_1= k_2 \cdot f'(k_3) \quad \text{and} \quad g_2= f(k_1) \cdot k_2.
		\end{equation*}
		The groupoid $H\times_G^h H'$ comes together with two canonical groupoid morphisms: the projections on $H$ and on $H'$.
		
		The groupoid $H\times^h_G H'$ is not a Lie groupoid in general. See \cite[Section 5.3]{MM03} for more details on when it is a Lie groupoid.
		%The set of objects $N_1\mathbin{{}_{f_1}\times_{t}} G \mathbin{{}_{s}\times_{f_2}} N_2$ is not a manifold in general (see \cite[Section 5.3]{MM03} for more details ). It is a manifold for instance when either $f_1\colon N_1\to M$ or $f_2\colon N_2\to M$ is a submersion. If $N_1\mathbin{{}_{f_1}\times_{t}} G \mathbin{{}_{s}\times_{f_2}} N_2$ is a manifold then $H_1 \mathbin{{}_{s\circ f_1}\times_{t}} G \mathbin{{}_{s}\times_{f_2}} N_2$ is a manifold because it is the fiber product between $s\colon H_1\to N_1$ and $\pr_1\colon N_1\mathbin{{}_{f_1}\times_{t}} G \mathbin{{}_{s}\times_{f_2}} N_2 \to N_1$ and $H_1 \mathbin{{}_{s\circ f_1}\times_{t}} G \mathbin{{}_{s}\times_{t\circ f_2}} H_2$ is a manifold because it is the fiber product between $s\colon H_2\to N_2$ and $\pr_3\colon H_1 \mathbin{{}_{s\circ f_1}\times_{t}} G \mathbin{{}_{s}\times_{f_2}} N_2\to N_2$. 
		When $H\times_G^h H'$ is a Lie groupoid we call it the \emph{homotopy fiber product} of $f$ and $f'$, and the two projections $\pr_1\colon H\times_G^h H'\to H$ and $\pr_3\colon H\times_G^h H'\to H'$ are Lie groupoid morphisms.
		
		The square
		\[
			\scriptsize
			\begin{tikzcd}
				& H\times_G^h H' \arrow[dl, "\pr_1"'] \arrow[dr, "\pr_3"] \\
				H \arrow[dr, "f"'] \arrow[rr, double] && H'\arrow[dl, "f'"] \\
				& G
			\end{tikzcd}
		\]
		commutes up the natural transformation $T\colon f\circ \pr_1\Rightarrow f'\circ \pr_3$ defined by 
		\[
			T\big(x_1, f(x_1)\xleftarrow{g} f'(x_2), x_2\big)\to g, \quad  \text{for all } (x_1, f(x_1)\xleftarrow{g} f'(x_2), x_2)\in N\mathbin{{}_{f}\times_{t}} G \mathbin{{}_{s}\times_{f'}} N'.
			\qedhere
		\]
		%Indeed, 
		%\[
		%	s(T(x_1, f_1(x_1)\xleftarrow{g} f_2(x_2), x_2))= s(g)= f_1(x_1)= f_1(\pr_1((x_1, f_1(x_1)\xleftarrow{g} f_2(x_2), x_2))),
		%\]
		%and
		%\[
		%	t(T(x_1, f_1(x_1)\xleftarrow{g} f_2(x_2), x_2))=t(g)= f_2(x_2)= f_2(\pr_3(x_1, f_1(x_1)\xleftarrow{g} f_2(x_2), x_2)).
		%\]
		%Moreover, for any $(x_2 \xleftarrow{k_1} x_1, k_2, y_2 \xleftarrow{k_3} y_1)\in H_1 \mathbin{{}_{s\circ f_1}\times_{t}} G \mathbin{{}_{s}\times_{t\circ f_2}} H_2$ arrow in $H_1\times_G^h H_2$ from $(x_1, f_1(x_1)\xleftarrow{g_1} f_2(x_2), x_2)$ to $(y_1, f_1(y_1)\xleftarrow{g_2}f_2(y_2), y_2)$, we have
		%\[
		%	f_1(k_1)\cdot g_1= f_1(k_1)\cdot k_2 \cdot f_2(k_3)= g_2\cdot f_2(k_3),
		%\]
		%and the naturality of $T$ is proved.
	\end{example}
	 
	The next two examples will be useful in Section \ref{sec:dictionary} for an alternative description of contact structures on manifolds.
	\begin{example}[General linear groupoid]
		\label{ex:general_linear_groupoid}
		Let $E$ be a vector bundle over the manifold $M$. The \emph{general linear groupoid} of $E\to M$ is the Lie groupoid $\operatorname{GL}(E) \rightrightarrows M$ where an arrow between $x,y\in M$ is an isomorphism $E_x\to E_y$ between the fibers over $x$ and $y$. The multiplication is just the composition of maps, the inverse of an arrow $\psi\colon E_x\to E_y$ is the inverse isomorphism $\psi^{-1}\colon E_y\to E_x$, and the unit of $x\in M$ is just the identity $\operatorname{id}_{E_x}\colon E_x\to E_x$.
	\end{example}
 
 	\begin{example}[Gauge groupoid]
 		Let $P\to M$ be a principal $G$-bundle, where $G$ is a Lie group. Then $G$ acts principally on $P\times P$ through the diagonal action and all the structure maps of the pair groupoid $P\times P \rightrightarrows P$ are $G$-equivariant. Hence, the quotient of the pair groupoid $P\times P \rightrightarrows P$ by the $G$-action is a Lie groupoid over $P/G=M$, that is called the \emph{gauge groupoid} of the principal $G$-bundle $P\to P/G=M$.
 	\end{example}
 
 	\begin{rem}
 		When the Lie group $G$ is the general linear group $GL(n)$, with $n\ge 0$, and the principal $G$-bundle $P\to M$ is the frame bundle of a rank $n$ vector bundle $E$, then the gauge groupoid agrees with the general linear groupoid $\operatorname{GL}(E)$. %Indeed, let $\operatorname{Fr}(E)\to M$ be the frame bundle of $E$. For any $x\in M$, the fiber $\operatorname{Fr}(E)_x$ consists of linear isomorphisms $\phi_x\colon \mathbbm{R}^n\to E_x$. Then we can consider the Lie groupoid morphism from the pair groupoid $\operatorname{Fr}(E)\times \operatorname{Fr}(E) \rightrightarrows \operatorname{Fr}(E)$ to the general linear groupoid $\operatorname{GL(E)}\rightrightarrows M$, where the map between the objects is just the projection of the frame bundle $\operatorname{Fr}(E)\to M$, and, between the arrows, the morphism maps the pair $(\phi_x, \psi_y)$, with $\phi_x\colon \mathbbm{R}^n\to E_x$ and $\psi_y\colon \mathbbm{R}^n\to E_y$, to the linear isomorphism $\psi_y\circ \phi_x^{-1}\colon E_x\to E_y$. 
 	\end{rem}
 	
 	Another important example of Lie groupoids is given by \emph{orbifolds}. These are a generalization of manifolds used as models for certain singular spaces, which can locally be seen as quotients of an open subset of $\mathbbm{R}^n$, with $n\geq 0$, by a finite group. Orbifolds correspond to a particular class of Lie groupoids. We discuss this in the following
 	\begin{example}[Orbifolds]
 		\label{ex:orbifolds}
 		Let $Q$ be a topological space. An \emph{orbifold chart} of dimension $n\geq 0$ on $Q$ is a triple $(U,G,\phi)$, where $U$ is a connected open subset of $\mathbbm{R}^n$, $G$ is a finite subgroup of the group of diffeomorphisms of $U$, and $\phi\colon U\to Q$ is an open map which induces a homeomorphism $U/G\to \phi(U)$. Two orbifold charts $(U,G,\phi)$ and $(U',G',\phi')$ of dimension $n$ on $Q$ are \emph{compatible} if, for any $x\in \phi(U)\cap \phi'(U')$, there exist an orbifold chart $(V,H,\psi)$ on $Q$ with $x\in \psi(V)$ and embeddings $f\colon V\to U$ and $f'\colon V\to U'$ such that 
 		\[
 			\phi\circ f = \psi, \quad \text{and} \quad \phi'\circ f'= \psi.
 		\]
 		A collection of pairwise compatible orbifold charts $\mathcal{U}= \{(U_i,G_i,\phi_i)\}$ of dimension $n$ on $Q$ such that $\bigcup_i\phi_i(U_i)=Q$ is called an \emph{orbifold atlas} of dimension $n$ of $Q$. An \emph{orbifold} of dimension $n$ is a pair $(Q,\mathcal{U})$, where $Q$ is a second-countable Hausdorff topological space and $\mathcal{U}$ is a maximal orbifold atlas of dimension $n$ on $Q$.
 		
 		Orbifold atlases form an important class of examples of Lie groupoids. Indeed, any orbifold atlas correspond to a \emph{proper} and \emph{\'etale} Lie groupoid $G\rightrightarrows M$, where ``proper'' means that $G$ is Hausdorff (but see Remark \ref{rem:Husdorff}) and the map $(s,t)\colon G\to M\times M$ is a proper map, and ``\'etale'' means that the source and the target maps are local diffeomorphisms. See \cite[Section 5.6]{MM03} for more details on the equivalence between orbifold atlases and proper and \'etale Lie groupoids. Here we just want to sketch how to obtain an orbifold from a proper and \'etale groupoid. Let $G\rightrightarrows M$ be a proper and \'etale groupoid. The orbit space $X := M/G$ is an orbifold and $G$ defines an orbifold atlas on $X$ as follows. Let $x \in M$, and let $U \subseteq M$ be an open neighborhood of $x$ such that the restricted groupoid $G_U = (s, t)^{-1} (U \times U) \rightrightarrows U$ identifies with an action groupoid $G_x \ltimes U \rightrightarrows U$, where $G_x$ acts (linearly) on $U$ via a diffeomorphism $U \cong T_x M$ (see \cite[Proposition 5.30]{MM03} for the existence of such $U$). The projection $U \to U/G_U \subseteq X$, together with the $G_x$-action on $U$, is an orbifold chart, and $X$ is covered by such charts. If $U, V \to X$ are two such charts, and $x \in U \cap V$, then a \emph{chart compatibility} is provided by any open subset $W \subseteq (s, t)^{-1}(U \times V)$ such that $s \colon W \to U$ and $t \colon W \to V$ are both embeddings around $x$. 
 	\end{example}
 	
 	In the next example we apply the tangent functor to both the manifolds $G$ and $M$, as well as to all the structure maps defining a Lie groupoid $G\rightrightarrows M$. 
 	\begin{example}[Tangent groupoid]
 		\label{ex:tangent_groupoid}
 		Let $G\rightrightarrows M$ be a Lie groupoid. Then $TG$ is a Lie groupoid over $TM$ with the following structure maps: source, target, inverse and unit maps are simply given by the differential maps $ds$, $dt$, $di$ and $du$ of the source, target, inverse and unit maps of $G$ respectively. The space of the composable arrows $(TG)^{(2)}= TG\mathbin{{}_{ds}\times_{dt}} TG$ is a vector bundle over $G^{(2)}$, isomorphic to $TG^{(2)} = T(G\mathbin{{}_{s}\times_{t}} G) \to G^{(2)}$ throught 
 		\[
 			d\pr_1\times d\pr_2\colon T(G\mathbin{{}_{s}\times_{t}}G) \to TG \mathbin{{}_{ds}\times_{dt}} TG.
 		\]
 		%where $\pr_i\colon G\mathbin{{}_{s}\times_{t}} G\to G$ is the projection on the $i$-th factor, for $i=1,2$. Namely, for any $v\in T_{(g,h)}G^{(2)}$, with $(g,h)\in G^{(2)}$, then $v$ is the velocity of a curve in $G^{(2)}$ starting from $(g,h)$, then $v=\tfrac{d}{d\varepsilon}|_{\varepsilon=0} (\gamma(\varepsilon), \eta(\varepsilon))$, such that $s(\gamma(\varepsilon))=t(\eta(\varepsilon))$, for all $\varepsilon$ and $(\gamma(0), \eta(0))=(g,h)$. If $v$ is in the kernel of $d\pr_1\times d\pr_2$, then 
 		%\[
 		%	\left( \frac{d}{d\varepsilon}|_{\varepsilon=0} \, \gamma(\varepsilon), \frac{d}{d\varepsilon}|_{\varepsilon=0} \, \eta(\varepsilon)\right)
 		%\]
 		%is the zero vector in $T_gG\mathbin{{}_{ds}\times_{dt}} T_hG$, so $\gamma(\varepsilon)=g$ and $\eta(\varepsilon)=h$, for all $\varepsilon$, and $v$ is the zero vector. This proves the injectivity of $d\pr_1\times d\pr_2$. The surjectivity follows from dimensional reasons. 
 		Actually, for any $k\in \mathbbm{N}$, $k\geq 2$, $(TG)^{(k)}$ is a vector bundle over $G^{(k)}$ isomorphic to $TG^{(k)}$ throught
 		\[
 			d\pr_1 \times \dots \times d\pr_k\colon TG^{(k)} \to (TG)^{(k)}.
 		\] 
 		Hence, the nerve of $TG$ is the simplicial manifold $TG^{(\bullet)}$ obtained by applying the tangent functor to the nerve $G^{(\bullet)}$ of $G$.
 		
 		Notice that, for any $g\in G$, the differential of the right multiplication $R_g$ \eqref{eq:right_multiplication} is just the right multiplication in $TG$ by the zero vector $0^{TG}_g\in T_gG$, and the differential of the left multiplication by $g$ \eqref{eq:left_multiplication} is just the left multiplication in $TG$ by the zero vector $0^{TG}_g\in T_gG$. %Indeed, let $v= \tfrac{d}{d\varepsilon}|_{\varepsilon=0} \gamma(\varepsilon) \in TG$, with $\gamma(\varepsilon)$ a curve in the $s$-fiber over $t(g)$, then
 		%\[
 		%	dR_g(v)= dR_g \left(\frac{d}{d\varepsilon}|_{\varepsilon=0} \gamma(\varepsilon)\right)= \frac{d}{d\varepsilon}|_{\varepsilon=0} \gamma(\varepsilon) \cdot g = v \cdot 0_g.
 		%\]
 	\end{example}
 
 	The notion of action of a Lie group extends to Lie groupoids. Let $G\rightrightarrows M$ be a Lie groupoid, and let $X$ be a manifold together with a smooth map $\mu\colon X\to M$.
 	\begin{definition}
 		\label{def:action_groupoid}
 		A \emph{left Lie groupoid action} of $G\rightrightarrows M$ on $X$ with \emph{moment map} $\mu\colon X\to M$ is given by a smooth map
 		\begin{equation}
 			\label{eq:action}
 			G\mathbin{{}_{s}\times_{\mu}} X \to X, \quad (g,x)\mapsto g.x,
 		\end{equation}
 	such that the following conditions are satisfied:
 	\begin{itemize}
 		\item for any $g\in G$ and $x\in X$ such that $s(g)=\mu(x)$, then $\mu(g.x)= t(g)$;
 		\item for any $(g,h)\in G^{(2)}$ and $x\in X$ such that $s(h)=\mu(x)$, then $g.(h.x)= (gh).x$;
 		\item for any $x\in X$, then $1_{\mu(x)}.x=x$.
 	\end{itemize}
 	\end{definition}
 	%As in the previous definition, often in the follows we indicate by ``.'' the $G$-action on $X$.
 	Similarly we can define \emph{right Lie groupoid action}.
 	\begin{definition}
 		A \emph{right Lie groupoid action} of $G\rightrightarrows M$ on $X$ with \emph{moment map} $\mu\colon X\to M$ is given by a smooth map
 		\begin{equation*}
 			X \mathbin{{}_{\mu}\times_{t}} G \to X, \quad (x,g) \mapsto x.g,
 		\end{equation*}
 		such that the following condition are satisfied:
 		\begin{itemize}
 			\item for any $g\in G$ and $x\in X$ such that $\mu(x)=t(g)$, then $\mu(x.g)= s(g)$;
 			\item for any $(g,h)\in G^{(2)}$ and $x\in X$ such that $\mu(x)=t(g)$, then $(x.g). h= x.(gh)$;
 			\item for any $x\in X$, then $x.1_{\mu(x)} =x$.
 		\end{itemize}
 	\end{definition}
 	In the following, when we write \emph{Lie groupoid action} we mean the left one.
 
 	As discussed in Example \ref{ex:action_groupoid}, any Lie group action on a manifold gives rise to an example of a Lie groupoid. Similarly, any Lie groupoid action gives rise to another example of Lie groupoid. 
 	\begin{example}[Action groupoid of a Lie groupoid action]
 		\label{ex:action_groupoid2}
 		Let $G\rightrightarrows M$ be a Lie groupoid acting on $X$ with moment map $\mu$. The fiber product $G\mathbin{{}_{s}\times_{\mu}} X$ is a Lie groupoid over $X$, again called the \emph{action groupoid} and denoted by $G\ltimes X$. The structure maps are the following: the source is simply the projection on $X$, the target is the action, the multiplication between $(h, g.x)$ and $(g,x)$ is $(hg, x)$, where $s(h)=\mu(g.x)=t(g)$, the inverse of $(g,x)$ is $(g^{-1}, g.x)$ and the unit of $x$ is $(\mu(x),x)$. Similarly we can build the action groupoid $X\rtimes G$ associated to a right Lie groupoid action.
 	\end{example}  
 	%Any Lie groupoid $G\rightrightarrows M$ acts on $G$ from left and right as showed in the next
 	%\begin{example}
 	%	Let $G\rightrightarrows M$ be a Lie groupoid. The multiplication
 	%	\[
 	%		G\mathbin{{}_{s}\times_{t}} G \to G, \quad (g,h)\mapsto gh,
 	%	\]
 	%	is a left Lie groupoid action of $G\rightrightarrows M$ on $G$ with moment map the target $t\colon G\to M$ and a right Lie groupoid action of $G\rightrightarrows M$ on $G$ with moment map the source $s\colon G\to M$. Notice that, for any $g\in G$, the left action gives the left multiplication \eqref{eq:left_multiplication} by $g$, and the right action gives the right multiplication \eqref{eq:right_multiplication} by $g$.
 	%\end{example}
 	
 	We are more interested in fiberwise linear Lie groupoid action on vector bundle.
 	\begin{definition}
 		\label{def:representation}
 		A \emph{representation} of a Lie groupoid $G\rightrightarrows M$ on a vector bundle $E\to M$ is a Lie groupoid action of $G$ on $E$ such that the moment map is the vector bundle projection $E\to M$ and the action \eqref{eq:action} is linear on the fibers i.e., for any arrow $g\colon x\to y$, the map $E_x\to E_y$, $e\mapsto g.e$ is linear, where $E_x$ is the fiber over $x$ and likewise for $y$. In this case, we also say that $E$ is a representation of $G$
 	\end{definition}
 	\begin{rem}
 		If $E$ is a representation of $G$, then for any $g\colon x\to y \in G$ we get a linear isomorphism $E_x\to E_y$. Indeed, any representation $E$ of $G$ can be seen as a Lie groupoid morphism from $G\rightrightarrows M$ to the general linear groupoid $\operatorname{GL}(E) \rightrightarrows M$ (see Example \ref{ex:general_linear_groupoid}) covering the identity $\operatorname{id}_M$.
 	\end{rem}
	
	When the Lie groupoid is just a Lie group $G\rightrightarrows \ast$ and the vector bundle is just a vector space $V$, then a representation of the Lie groupoid is just a representation of the Lie group $G$ on $V$.
	
	An important example of representation is the following
	\begin{example}[Normal representation]
		\label{ex:normal_representation}
		Let $O$ be an orbit of the Lie groupoid $G\rightrightarrows M$ and let $NO=TM|_O / TO$ be the normal bundle of $O$. The restriction $G_O\rightrightarrows O$, where $G_O=s^{-1}(O)\subseteq G$, is a well-defined Lie groupoid (see \cite[Proposition 3.4.2]{dH13}). We denote by $[v]$ the equivalence class in $N_xO$ of a vector $v\in T_xM$, $x\in O$. The \emph{normal representation} is the linear action of $G_O\rightrightarrows O $ on $NO$ defined in the following way: let $g\colon x\to y\in G_O$ and $[v]\in N_xO$, with $v\in T_xM$, then there exists $u\in T_gG$ such that $ds(u)=v\in T_xM$ and we put $g. [v]:=[dt(u)]$. %The action is well-defined. First, $g.[v]$ does not depend on the choice of $u\in T_gG$. Indeed, let $v\in T_xM$ and let $u,u'\in T_gG$ be such that $ds(u)=ds(u')=v'$. Then the difference $u-u'\in \ker d_gs$. But, the right multiplication by $g$, $R_g\colon s^{-1}(y) \to s^{-1}(x)$, determines an isomorphism
		%\[
		%	d_yR_g\colon T_ys^{-1}(y) \to T_gs^{-1}(x).
		%\]
		%Since $s$ is a submersion, $T_ys^{-1}(y)= \ker d_ys$ and $T_gs^{-1}(x)= \ker d_gs$, so $u-u'=dR_g(a)$ for an unique $a\in \ker d_ys$, and $dt(u)-dt(u')=dt(a)$. Noting that, since $O=\im(t\colon s^{-1}(y)\to M)$, then $T_yO = \im(d_yt\colon T_ys^{-1}(y)\to T_yM)$ and $dt(a)\in T_yO$.
		%
		%Now, let $v, v'\in T_xM$ be such that $[v]=[v']$ and let $u\in T_gG$ such that $ds(u)=v$. Then there exists $a\in \ker(d_xs)$ such that $v'=v+dt(a)$. From the first part of the proof, $g.[v']$ does not depend on the choice of $u'\in T_gG$ such that $ds(u')=v'$. Take $u'= u + dL_g(di(a))$. Indeed 
		%\[
		%	ds(u')=ds(u)+ ds(dL_g(di(a)))= v+ dt(a)=v'.
		%\]
		%Now we have
		%\[
		%	dt(u')= dt(u)+ dt(dL_g(di(a)))= dt(u),
		%\]
		%where we used that $t\circ L_g\colon t^{-1}(x)\to M$ is constantly equal to $y$.
		%
		%In order to show that it is really an action we need to prove the conditions in Definition \ref{def:action_groupoid}. The first and the third ones are obvious. For the second one, let $(g,h)\in G^{(2)}$ and $[v]\in N_{s(h)}O$, with $v\in T_{s(h)}M$. There exists $u\in T_hG$ such that $ds(u)=v$ and $h.[v]=[dt(u)]\in N_{t(h)}O$. But $s(g)=t(h)$, then there exists $w\in T_gG$ such that $ds(w)=dt(u)$ and $g.(h.[v])=[dt(w)]$. If we take $w\cdot u\in T_{gh}G$, where ``$\cdot$'' is the multiplication in the tangent groupoid $TG\rightrightarrows TM$ (see Example \ref{ex:tangent_groupoid}), then $ds(w\cdot u))=ds(u)$ and $dt(w\cdot u)=dt(w)$.
	\end{example}
	\begin{rem}
		\label{rem:normal_representation_morphism}
		Let $G\rightrightarrows M$ be a Lie groupoid. Following \cite[Section 3.4]{dHO20} the normal representation (see Example \ref{ex:normal_representation}) can be restricted to isotropy groups. Let $O$ be the orbit through $x\in M$. For any $g\colon x\to x$ in the isotropy group $G_x$ and $[ds(u)]\in N_xO=T_xM/T_xO$, with $u\in T_gG$, then $dt(u)\in T_xM$ and $N_xO$ is a representation of $G_x$ called the \emph{normal representation} of $G$ at $x$.
		
		If $x,y\in O$, then the normal representations at $x$ and $y$ are (non-canonically) isomorphic. Indeed, for any arrow $g\colon x\to y$ we have the Lie group representation isomorphism $G_x\times N_xO\to G_y\times N_yO$, $(h,[v])\mapsto (ghg^{-1}, g.[v])$, covering $N_xO\to N_yO$, $[v]\mapsto g.[v]$. %Its iverse is the Lie group representation morphism determined by $g^{-1}\colon y\to x$.
		
		Finally, let $f\colon (H\rightrightarrows N)\to (G\rightrightarrows M)$ be a Lie groupoid morphism. For any $y\in N$, let $O$ be the orbit in $H$ through $y$ and $O'$ the orbit in $G$ through $f(y)\in M$. By Remark \ref{rem:function_orbit_spaces}, $f(O)\subseteq O'$, then $df\colon TN\to TM$ descends to a map between the normal spaces $df\colon N_yO \to N_{f(y)}O'$, $[v]\mapsto [df(v)]$. The map 
		\[
			f\colon H_y\times N_yO\to G_{f(y)}\times N_{f(y)}O', \quad (h, [v])\mapsto (f(h), [df(v)])
		\]
		is a Lie group representation morphism covering $f\colon N_yO\to N_{f(y)}O'$, $[v]\mapsto [df(v)]$.% Indeed, if $u\in T_xH$ such that $ds(u)=v$, then $df(u)\in T_{f(x)}G$ is such that $ds(df(u))= df(v)$, and
		%\[
		%	f(h,[v])=f(h). [df(v)]= [dt(df(u))]= [df(dt(u))]= f([dt(u)])= f(h.[v]).
		%\]
	\end{rem}
	
	We conclude this section recalling the notions of \emph{right} and \emph{left-invariant vector field} on a Lie groupoid and the differential graded algebra of functions on the nerve of a Lie groupoid. Let $G\rightrightarrows M$ be a Lie groupoid. By Remark \ref{rem:multiplication_diffeomorphism}, for any $g\colon x\to y$, the right multiplication $R_g\colon s^{-1}(y)\to s^{-1}(x)$ (see \eqref{eq:right_multiplication}) and the left multiplication $L_g\colon t^{-1}(x)\to t^{-1}(y)$ (see \eqref{eq:left_multiplication}) by $g$ are diffeomorphisms, so we can consider the pullback of vector fields on $G$ that are tangent to either the $s$ or the $t$-fibers. 
	\begin{definition}
		\label{def:right_invariant_vector_fields}
		A \emph{right-invariant vector field} on $G$ is a vector field $X$ on $G$ such that $X$ is a section of the subbundle $\ker ds\to G$ and, for any $(g,h)\in G^{(2)}$, then
		\[
			X_{gh} = dR_h(X_g),
		\]
		where $R_h$ is right multiplication by $h$ (see \eqref{eq:right_multiplication}). The set of right-invariant vector fields will be denoted by $\mathfrak{X}_r(G)$.
		
		A \emph{left-invariant vector field} on $G$ is a vector field $X$ on $G$ such that $X$ is a section of the subbundle $\ker dt\to G$ and, for any $(g,h)\in G^{(2)}$, then
		\[
		X_{gh} = dL_g(X_h),
		\]
		where $L_g$ is left multiplication by $g$ (see \eqref{eq:left_multiplication}). Similarly, the set of left-invariant vector fields will be denoted by $\mathfrak{X}_l(G)$.
	\end{definition}
	\begin{rem}
		\label{rem:lie_subalgebra}
		The set $\mathfrak{X}_r(G)$ of the right-invariant vector fields is a $C^{\infty}(M)$-module, where the product of a function $f\in C^\infty(M)$ with a right-invariant vector field $X\in \mathfrak{X}_r(G)$ is the vector field $(t^{\ast}f) X$. %The latter vector field is right-invariant. Namely, for any $(g,h)\in G^{(2)}$, we get
		%\[
		%	\left((t^\ast f)X\right)_{gh}= f(t(gh)) X_{gh}= f(t(g)) dR_h(X_g)=dR_h\left(f(t(g)) X_g\right)= dR_h\left(((t^\ast f)X)_g\right).
		%\]
		Similarly, $\mathfrak{X}_l(G)$ is a $C^\infty(M)$-module, where the product of a function $f\in C^\infty(M)$ with a left-invariant vector field $X\in \mathfrak{X}_l(G)$ is the left-invariant vector field $(s^\ast f)X$.
		
		Moreover, the $C^{\infty}(M)$-module $\mathfrak{X}_r(G)$ (respectively $\mathfrak{X}_l(G)$) of the right-invariant (respectively left-invariant) vector fields is a Lie subalgebra of the Lie algebra $\mathfrak{X}(G)$ of vector fields on $G$. %This easily follows from the naturality of the Lie bracket of vector fields with respect to diffeomorphisms, i.e., for any $X,Y\in \mathfrak{X}_r(G)$ and $h\in G$, we have
		%\[
		%	R_h^\ast [X,Y] = [R_h^\ast X, R_h^\ast Y]=[X,Y],
		%\]
		%and similar for left-invariant vector fields.
	\end{rem}
	
	Finally we can associate to any Lie groupoid $G$ a differential graded algebra of functions on the nerve $G^{(\bullet)}$.
	\begin{rem}
		\label{rem:dg_algebra}
		Let $G\rightrightarrows M$ be a Lie groupoid. For any $k\in\mathbbm{N}$ we can cosider the algebra $C^\infty(G^{(k)})$ of smooth functions on the manifold $G^{(k)}$ in the nerve $G^{(\bullet)}$of $G$. We set $C^k(G):=C^{\infty}(G^{(k)})$. The face maps $\partial_i$ in the nerve $G^{(\bullet)}$ define a differential $\partial\colon C^k(G) \to C^{k+1}(G)$: the alternating sum of the pullbacks along the face maps, i.e., $$\partial=\sum (-1)^i \partial_i^\ast.$$ %More precisely, for any $f\in C^0(G)=C^\infty(M)$, then $\partial f\in C^1(G)=C^\infty(G)$ is given by 
		%\[
		%(\partial f)(g)=f(s(g))- f(t(g)), \quad \text{for all } g\in G, 
		%\]
		%and, for any $f\in C^k(G)$, with $k\geq 1$, then $\partial f\in C^{k+1}(G)$ is given by
		%\[
		%(\partial f)(g_0, \dots, g_k)= f(g_1, \dots, g_k) + \sum_{i=1}^{k} (-1)^if(g_0, \dots, \hat{g_i}, \dots, g_k) + (-1)^{k+1}f(g_0, \dots,g_{k-1}),
		%\]
		%for all $(g_0, \dots, g_k)\in G^{(k+1)}$. The cohomology of the complex $(C^{\bullet},\partial)$ is called the \emph{smooth groupoid cohomology} of $G$.
		%
		We set $C(G)=\bigoplus_{k\in \mathbbm{N}} C^k(G)$. The graded space $C(G)$ comes together with a cup product
		\[
		\cup\colon C^k(G)\times C^l(G)\to C^{k+l}(G), \quad (f_1,f_2) \mapsto f_1\cup f_2,
		\]
		defined in the following way: when $k,l>0$, then 
		\[
		(f_1\cup f_2)(g_1, \dots, g_{k+l})= f_1(g_1, \dots, g_k) f_2(g_{k+1}, \dots, g_{k+l}),
		\]
		for all $(g_1, \dots , g_{k+l})\in G^{(k+l)}$. When $k=0$ and $l>0$, then
		\[
		(f_1\cup f_2)(g_1, \dots, g_l)= f_1(t(g_1))f_2(g_1, \dots, g_l), 
		\]
		for all $(g_1, \dots, g_l)\in G^{(l)}$. When $k>0$ and $l=0$, then
		\[
		(f_1\cup f_2)(g_1, \dots, g_k)= f_1(g_1, \dots, g_k)f_2(s(g_k)), 
		\]
		for all $(g_1, \dots, g_k)\in G^{(k)}$. Finally, when $k=l=0$, then $f_1\cup f_2=f_1f_2$ is the usual product between functions.
		%
		%The differential $\partial$ is a graded derivation with respect to the cup product, i.e.,
		%\[
		%\partial(f_1\cup f_2)= (\partial f_1)\cup f_2 + (-1)^k f_1\cup (\partial f_2),
		%\]
		%for all $f_1\in C^k(G)$ and $f_2\in C(G)$. 
		The vector space $C(G)$ with the differential and the cup product is a differential graded algebra called the \emph{differential algebra} of $G$.
	\end{rem}
	\begin{rem}
		\label{rem:section_representation}
		A representation $E$ of $G\rightrightarrows M$ gives rise to a differential graded module over the differential graded algebra described in Remark \ref{rem:dg_algebra}. For any $k\in \mathbbm{N}$, we denote by $t\colon G^{(k)}\to M$ the projection onto the first factor followed by the target map. Then $t^\ast E\to G^{(k)}$ is a vector bundle and, following \cite[Section 1.2]{Cr03}, sections of $t^\ast E$ are degree $k$ cochains in an appropriate complex. We prefer not to discuss this construction in detail as it will only appear in what follows under a slightly different disguise. 
	\end{rem}

	%If $E$ is a representation of $G\rightrightarrows M$ there is a version for the first line of the BSS double complex with values in $E$. Following \cite{???}, for any $k$ we denote $t\colon G^{(k)}\to M$, $(g_1,\dots , g_k)\mapsto t(g_1)$, then $t^{\ast}E$ is a vector bundle over $G^{(k)}$. Cochains are sections of these vector bundles and the differential of a section $\lambda$ of $t^{\ast}E\to G^{(k)}$ is
	%\begin{equation*}
		%d\lambda (g_1, \dots, g_{k+1}) = g_1 . \lambda(g_2,\dots, g_{k+1}) + \sum_{i=1}^{k} (-1)^i \lambda(g_1, \dots, g_i g_{i+1}, \dots, g_{k+1}) + (-1)^{k+1} \lambda(g_1, \dots, g_k).
	%\end{equation*}
	%Moreover, from \cite[Definition 2.1]{CSS} an \emph{$E$-valued multiplicative $k$-form} is a form $\theta\in \Omega^k(G,t^{\ast}E)$ such that
	%\begin{equation*}
	%	\theta_{gh}(v_1u_1, \dots, v_ku_k)= \theta_g(v_1, \dots,v_k) + g.\theta_h(u_1, \dots,u_k)
	%\end{equation*}
	\section{Lie algebroids}\label{sec:Lie_algebroids}
	In this section we recall the infinitesimal version of Lie groupoids, i.e., \emph{Lie algebroids}. We begin with the definition, provide examples and we recall the notion of representation of a Lie algebroid. Finally, in Section \ref{sec:Atiyah_algebroid}, we focus on the Lie algebroid associated with a general linear groupoid (example \ref{ex:general_linear_groupoid}). This will play a key role in Section \ref{sec:dictionary}, where we provide an alternative description of contact structures on manifolds.

	Let $M$ be a manifold.
	\begin{definition}
		A \emph{Lie algebroid} over $M$ is a triple $(A,\rho, [-,-])$, where $A$ is a vector bundle over $M$, $\rho\colon A\to TM$ is a vector bundle (VB in the follows) morphism covering the identity map, and $[-,-]$ is a Lie bracket on the module of sections $\Gamma(A)$, satisfying the following Leibniz rule: for any $a,b\in \Gamma(A)$ and $f\in C^{\infty}(M)$
		\begin{equation}
			\label{eq:leibniz_rule}
			[a, fb]= f[a, b] + \rho(a)(f) b.
		\end{equation}
		The vector bundle morphism $\rho\colon A \to TM$ is called the \emph{anchor map}. A \emph{Lie algebroid morphism} is a vector bundle morphism between two Lie algebroids which is compatible with the anchors and the brackets (see, e.g., \cite[Section 2.2]{CF11} for a detailed explanation of these compatibility conditions).
	\end{definition}
	In the following we often just say that $A$ is a Lie algebroid.
	
	In the next remark we recall that it is always possible to associate a Lie algebroid to any Lie groupoid. This process is called \emph{differentiation}. The converse, the \emph{integration}, unlike in the Lie group case, is not always possible: there exist Lie algebroids that are not integrable, i.e., they do not arise from a Lie groupoid as their infinitesimal counterparts. The integration problem for Lie algebroids is discussed in \cite{CF11}.

	\begin{rem}[The Lie algebroid associated with a Lie groupoid]
		\label{rem:Lie_algebroid}
		Let $G\rightrightarrows M$ be a Lie groupoid. We recall from \cite[Section 1.4]{CF11} that the Lie algebroid $(A_G, \rho, [-,-])$ associated to $G$ is defined as follows: $A_G$ is the vector bundle over $M$ whose fiber over $x\in M$ is the tangent space at the unit $1_x$ of the $s$-fiber $s^{-1}(x)$ over $x$, i.e., $A_G=(\ker ds)|_M$. The anchor map $\rho\colon A_G \to TM$ is the restriction to $A_G$ of $dt\colon TG\to TM$, i.e., $\rho=dt|_{A_G}$. In order to introduce the Lie bracket $[-,-]$ on $\Gamma(A_G)$ notice that the $\mathbbm{R}$-vector space $\Gamma(A_G)$ is isomorphic to $\mathfrak{X}_r(G)$ via the map $a \mapsto \overrightarrow{a}$ defined by setting 
		\[
			\overrightarrow{a}_g=dR_g(a_{t(g)}).
		\]
		%Conversly, for any right-invariant vector field $X\in \mathfrak{X}_r(G)$ there exist a section $a\in \Gamma(A_G)$ such that $X=\overrightarrow{a}$. Indeed, the section $a$ is simply the restriction $X|_M$. But, 
		From Remark \ref{rem:lie_subalgebra}, the space $\mathfrak{X}_r(G)$ is a Lie subalgebra of the Lie algebra $\mathfrak{X}(G)$ of vector fields on $G$, so we can define a Lie bracket on $\Gamma(A_G)$ from the Lie bracket on $\mathfrak{X}_r(G)$, in other words, for any $a,b\in \Gamma(A_G)$, the Lie bracket $[a,b]\in \Gamma(A_G)$ is implicitly defined by
		\[
			\overrightarrow{[a,b]}=[\overrightarrow{a},\overrightarrow{b}].
		\]
		The defined triple $(A_G,\rho, [-,-])$ is a Lie algebroid over $M$. %Indeed, first notice that for any $a\in \Gamma(A_G)$ and $f\in C^{\infty}(M)$,
		%\[
		%	(\overrightarrow{af})_g=dR_g(f(t(g))a_{t(g)})= ((t^{\ast}f)\overrightarrow{a})_g,
		%\]
		%for all $g\in G$, and the right-invatiant vector field $\overrightarrow{fa}$ agrees with $(t^{\ast}f)\overrightarrow{a}$. Then, for any $a,b\in \Gamma(A_G)$ and $f\in C^{\infty}(M)$, we have
		%\begin{align}
		%	\label{eq:lie_algebroid}
		%	\overrightarrow{[a,fb]}= [\overrightarrow{a}, \overrightarrow{fb}] = [\overrightarrow{a}, (t^\ast f) \overrightarrow{b}] = t^\ast f[\overrightarrow{a},\overrightarrow{b}] + \overrightarrow{a}(t^\ast f)\overrightarrow{b} = t^\ast f \overrightarrow{[a,b]} + \overrightarrow{a}(t^\ast f)\overrightarrow{b},
		%\end{align}
		%but, for any $g\in G$, 
		%\begin{align*}
		%	\overrightarrow{a}(t^{\ast} f)(g) &= d_g(f\circ t)(\overrightarrow{a}_g)= (d_{t(g)} f \circ d_g t)(\overrightarrow{a}_g) \\
		%	&= d_{t(g)}f(\rho(a)_{t(g)})= \rho(a)(f)(t(g)) = t^\ast (\rho(a)(f)) (g),
		%\end{align*}
		%where we used that $d_gt(\overrightarrow{a}_g)= d_gt(dR_g(a_{t(g)}))= d_{t(g)}t(a_{t(g)})=\rho(a)_{t(g)}$. Then $\overrightarrow{a}(t^\ast f)= t^\ast(\rho(a)(f))$, \eqref{eq:lie_algebroid} reduces to
		%\[
		%	\overrightarrow{[a,fb]} = (t^\ast f)\overrightarrow{[a,b]} + t^\ast(\rho(a)(f))\overrightarrow{b} = \overrightarrow{f[a,b]} + \overrightarrow{\rho(a)(f)b},
		%\]
		%and the Leibniz rule \eqref{eq:leibniz_rule} is satisfied.
		
		Notice that, from the Lie groupoid $G$, one can define another Lie algebroid $(A_G^L,\rho, [-,-])$ over $M$ changing the role of the source and target map: $A_G^L$ is the vector bundle over $M$ whose fiber over $x\in M$ is the tangent space at the unit $1_x$ of the $t$-fiber $t^{-1}(x)$ over $x$. The anchor map $\rho$ is the restriction to $A_G^L$ of $ds\colon TG \to TM$. Finally, we can see that the space of sections $\Gamma(A^L)$ is isomorphic to the Lie algebra $\mathfrak{X}_l(G)$ of the left-invariant vector fields on $G$ (Remark \ref{rem:lie_subalgebra}), and the Lie bracket on $\mathfrak{X}_l(G)$ induces a Lie bracket on $\Gamma(A_G^L)$. Actually this Lie algebroid is not a new structure. In fact, $A_G$ and $A_G^L$ are isomorphic through the restriction of the VB morphism $di\colon TG \to TG$.
		
		Conventionally, the Lie algebroid induced by $G$ is $(A_G,\rho,[-,-])$. We refer to it as the \emph{Lie algeborid of $G$} and we denote it simply by $A$ if there is no risk of confusion.
	\end{rem}

	Let $A$ be a Lie algebroid over $M$. For any $x\in M$, the kernel $\ker (\rho_x\colon A_x\to T_xM)$ is a Lie algebra called the \emph{isotropy Lie algebra} at $x$ and denoted by $\mathfrak{g}_x$.
	
	Now we provide some examples of Lie algebroids. We begin with two trivial case: Lie algebras and the tangent bundle.
	\begin{example}[Lie algebras]
		A Lie algebra is a Lie algebroid over a point. This is the Lie algebroid associated to a Lie group.
	\end{example}
	
	\begin{example}[Tangent bundle]\label{ex:tangent_bundle}
		Let $M$ be a manifold. The tangent bundle $TM$ is a Lie algebroid with anchor map the identity and Lie bracket the commutator of vector fields. This is the Lie algebroid associated to the pair groupoid $M\times M\rightrightarrows M$. More generally, any regular foliation determines a Lie algebroid. Indeed, an involutive subbundle $A\subseteq TM$, is a Lie algebroid with anchor map the inclusion and Lie bracket again the commutator. In particular, in the latter case, the anchor map is injective.
	\end{example}

	\begin{rem}
		\label{rem:foliation_groupoid}
		As explained in Example \ref{ex:tangent_bundle}, any foliation on a manifold $M$ induces a Lie algebroid structure $A$ over $M$. An integration $G\rightrightarrows M$ of $A$ is called a \emph{foliation groupoid}. By \cite[Theorem 1]{CM01}, a Lie groupoid is a foliation groupoid if and only if the anchor map is injective or all the isotropy Lie groups of $G$ are discrete. 
	\end{rem}
	
	An important class of examples of Lie algebroids is given by Poisson manifolds.
	\begin{example}[Poisson manifold]
		\label{ex:poisson_manifold}
		Let $\pi$ be a Poisson bivector on $M$. Then the cotangent bundle $T^\ast M$ is equipped with a Lie algebroid structure: the anchor map is given by the sharp map $\pi\colon T^\ast M \to TM$ of $\pi$, and, for any $a,b\in \Omega^1(M)$, the Lie bracket is given by
		\[
			 [a,b]= \mathcal{L}_{\pi(a)} b - \mathcal{L}_{\pi(b)}a - d\pi(a,b),
		\]
		where $\mathcal{L}$ is the Lie derivative of differential forms along vector fields.
	\end{example}

	\begin{rem}
		\label{rem:Lie_functor}
		The correspondence $G \leadsto A_G$ is functorial. Indeed, if $f\colon (H\rightrightarrows N)\to (G\rightrightarrows M)$ is a Lie groupoid morphism, then the restriction of the map $df\colon TH \to TG$ to $A_H$, the Lie algebroid of $H$, takes values in $A_G$, the Lie algebroid of $G$, and the map $df\colon A_H \to A_G$ is a Lie algebroid morphism. This functor from the category of Lie groupoids with Lie groupoid morphisms to the category of Lie algebroids with Lie algebroid morphisms is called the \emph{Lie functor}. %For this reason, sometimes, in the literature, the Lie algebroid of a Lie groupoid $G$ is denoted by $Lie(G)$ and the Lie algebroid morphism induced by a Lie groupoid morphism $f$ is denoted by $Lie(f)$.
	\end{rem}

	\begin{rem}
		\label{rem:SES_Lie_algebroid}
		Let $G\rightrightarrows M$ be a Lie groupoid. The Lie algebroid $A$ of $G$ fits in the following short exact sequence of VBs over $M$
		\begin{equation}
			\label{eq:core_tangentSES}
			\begin{tikzcd}
				0 \arrow[r] & A \arrow[r] & TG|_M \arrow[r, "ds"] & TM\arrow[r] & 0,
			\end{tikzcd}
		\end{equation}
		where the VB morphism $A\to TG|_M$ is just the inclusion. The differential $du\colon TM \to TG$ of the unit map is a canonical right splitting of \eqref{eq:core_tangentSES}, then we get $TG|_M \cong A \oplus TM$.
		
		The short exact sequence \ref{eq:core_tangentSES} is the restriction to the units of the short exact sequence of VBs over $G$
		\begin{equation}
			\label{eq:tangentSES}
			\begin{tikzcd}
				0 \arrow[r] & t^\ast A \arrow[r] & TG\arrow[r, "ds"] & s^\ast TM\arrow[r] & 0,
			\end{tikzcd}
		\end{equation}
		where the VB morphism $t^\ast A \to TG$ is just the differential of the right multiplication \eqref{eq:right_multiplication}, namely, it maps any $(g,a)\in G\times_M A$, with $a\in A_{t(g)}$, to $ dR_g(a)= a\cdot 0^{TG}_g\in T_gG$.
	\end{rem}
	Let $G\rightrightarrows M$ be a Lie groupoid. Following \cite{AC13} we introduce the notion of \emph{Ehresmann connection} that will be useful in Section \ref{sec:RUTHs}.
	\begin{definition}[{\cite[Definition 2.8]{AC13}}]
		\label{def:Ehresmann_connection}
		An \emph{Ehresmann connection} on $G$ is a right splitting $$h\colon s^\ast TM \to TG$$ of the short exact sequence of VBs over $G$ \eqref{eq:tangentSES}	which agrees with $du \colon TM \to TG$ on units.
	\end{definition}
	
	\begin{rem}
		Because of Remark \ref{rem:Husdorff}, all the Lie groupoids that we are considering admit an Ehresmann connection.
	\end{rem}

	Let $A$ be a Lie algebroid over $M$, and let $X$ be a manifold together with a smooth map $\mu\colon X\to M$.
	\begin{definition}
		A \emph{Lie algebroid action} of $A$ on $X$ along the map $\mu$ is a Lie algebra morphism
		\[
			\chi\colon\Gamma(A)\to \mathfrak{X}(X),
		\]
		such that, for any $a\in \Gamma(A)$, $f\in C^\infty(M)$ and $x\in M$
		\[
			d_x\mu(\chi(a)_x)= \rho (a_{\mu(x)}), \quad \chi(fa)= (\mu^\ast f) \chi(a).
		\]
	\end{definition}

	Let $A$ be a Lie algebroid over $M$.
	\begin{definition}
		A \emph{representation} of $A$ is a VB $E\to M$ together with a \emph{flat $A$-connection} $\nabla$, i.e., an $\mathbbm{R}$-bilinear map $\nabla\colon \Gamma(A)\times\Gamma(E) \to \Gamma(E)$, $(a,e)\mapsto \nabla_{a} e$, which is $C^{\infty}(M)$-linear in the first entry, satisfies the Leibniz rule
		\begin{equation*}
			\nabla_{a}(fe)= f\nabla_{a}e + \rho(a)(f)e, \quad a\in \Gamma(A), \, e\in \Gamma(E), \, f\in C^{\infty}(M),
		\end{equation*}
		and it is flat in the sense that
		\begin{equation*}
			\nabla_{a}\nabla_{b} - \nabla_{b}\nabla_{a}=\nabla_{[a,b]}, \quad a,b\in \Gamma(A).
		\end{equation*}
	\end{definition}

	\begin{rem}
		Let $A$ be a Lie algebroid acting on a vector bundle $E$. The action is \emph{linear} if, for any $a\in \Gamma(A)$, $\chi(a)\in \mathfrak{X}(E)$ is a \emph{linear vector field}, i.e., there exists $X\in \mathfrak{X}(M)$ such that $(\chi(a), X)\colon (E\to M)\to (TE\to TM)$ is a VB morphism. If the action of $A$ on $E$ is linear, then it is just a representation of $A$ on $E$. This is discussed in detail in, e.g., \cite{KM02} or \cite[Section 3.4]{M05}. Here we just mention that a representation of a Lie algebroid is equivalent to a Lie algebroid morphism from $A$ to the Lie algebroid of the general linear groupoid $GL(E)$ (see Section \ref{sec:Atiyah_algebroid} for a more precise description of this Lie algebroid). The latter Lie algebroid is in fact equivalent to the Lie subalgebroid of $TE$ given by linear vector fields, whence the claim. 
	\end{rem}
	\begin{rem}
		\label{rem:Lie_functor_actions}
		The Lie functor extends to actions and representations. Indeed, let $G\rightrightarrows M$ be a Lie groupoid acting on $X$ with moment map $\mu\colon X\to M$ by the smooth map $\phi\colon G\mathbin{{}_{s}\times_{\mu}} X\to X$, and let $A$ be the Lie algebroid of $G$.
		We define a Lie algebra morphism $\chi\colon \Gamma(A)\to \mathfrak{X}(X)$ by setting
		\[
			\chi(a)_x= d_{(\mu(x), x)}\phi (a_{\mu(x)}, 0_x).
		\]
		This action is called the \emph{infinitesimal action}.
		
		Let $G\rightrightarrows M$ be a Lie groupoid with a representation $E\to M$. The induced representation of $A$, the Lie algebroid of $G$, on $E$ is given by the flat connection $\nabla\colon \Gamma(A)\times \Gamma(E) \to \Gamma(E)$ defined by
		\[
			(\nabla_a e)_x= \frac{d}{d\varepsilon}|_{\varepsilon=0} \left(\phi^a(\varepsilon)(x)\right)^{-1} . e_{t(\phi^a(\varepsilon)(x))},
		\]
		where $\phi^a_{\varepsilon}(x)$ is the value of the flow of the right-invariant vector field $\overrightarrow{a}$ at the point $x$, for all $x\in M$. Similarly, this representation of $A$ on $E$ is just the Lie algebroid morphism obtained applying the Lie functor to the Lie groupoid morphism $G\to GL(E)$ given by the representation of $G$ on $E$.
	\end{rem}
	
	Let $E$ be a representation of a Lie algebroid $A\to M$ with $A$-connection $\nabla$. Then we can introduce a cochain complex whose cochains are $E$-valued differential forms on $A$, i.e., the elements of the graded vector space $\Omega^\bullet(A,E):=\Gamma(\wedge^{\bullet} A^{\ast}\otimes E)$, and whose differential is defined by the usual Koszul-type formula: if $\omega$ is an $E$-valued $k$-form on $A$, then 
	\begin{equation}
		\label{eq:differential}
		\begin{aligned}
			d\omega(a_0, \dots, a_k)&= \sum_{i=0}^{k} (-1)^i \nabla_{a_i}(\omega(a_0, \dots, \hat{a}_i, \dots, a_k)) \\ 
			&\quad -\sum_{i<j} (-1)^{i+j} \omega([a_i,a_j], a_0, \dots, \hat{a}_i, \dots, \hat{a}_j, \dots, a_k)
		\end{aligned}
	\end{equation} 
	for all $a_0, \dots, a_k\in \Gamma(A)$. Above, as in the following, placing on hat over entries of a sequence means that those entries are removed from the sequence.

%	\begin{rem}
%		In the case when $E$ is a representation of a Lie groupoid $G$ and $\nabla$ is the infinitesimal representation (see Remark \ref{rem:Lie_functor_actions}) of $A$, the Lie algebroid of $G$, on $E$, then it is possible to define a Van Est map from the cohomology of the complex discussed in Remark \ref{rem:section_representation} to the cohomology of the complex discussed above. This is discussed in \cite{Cr03}.
%	\end{rem}
	
	\begin{rem}
		\label{rem:definition_differential_nabla}
		When the Lie algebroid $A\to M$ is the tangent bundle $TM$, then a $TM$-connection is a usual connection on $E$, and so a representation of $TM$ on $E$ is simply given by a flat connection on $E$. In this case we denote by $\Omega^\bullet(M,E)= \Gamma(\wedge^\bullet T^\ast M\otimes E)$ the graded vector space of $E$-valued differential forms on $M$. For any connection $\nabla$ on $M$, the $\mathbbm{R}$-linear operator $d^\nabla$ defined by Equation \eqref{eq:differential} is a differential on $\Omega^\bullet(M,E)$ if and only if $\nabla$ is flat.
	\end{rem}

	VB-valued forms, discussed in Remark \ref{rem:definition_differential_nabla}, will play a key role in the definition of contact structure. Here we just want to recall when we can pullback VB-valued forms. First, let $(F, f)\colon (E\to M)\to (E'\to M')$ be a VB morphism.
	\begin{definition}
		\label{def:regular_VBMorphism}
		The VB morphism $(F,f)$ is said \emph{regular} if $F\colon E\to E'$ is fiberwise isomorphism, i.e., for any $x\in M$, the restriction $F_x\colon E_x\to E'_{f(x)}$ to the fibers is an isomorphism.
	\end{definition}
	A regular VB morphism $(F,f)\colon (E\to M)\to (E'\to M')$ determines a \emph{pullback of VB-valued forms}, $F^{\ast}\colon \Omega^{\bullet}(M',E')\to \Omega^{\bullet}(M,E)$: for any $x \in M$,
	\begin{equation}
		\label{eq:pullback_vector_valued_forms}
		(F^{\ast}\theta)_x(v_1, \dots, v_k):=F_x^{-1}\big(\theta_{f(x)}(df(v_1), \dots,df(v_k))\big),
	\end{equation}
	for all $v_1,\dots, v_k\in T_xM$ and $\theta\in \Omega^k(M',E')$.
	
	\subsection{The Atiyah algebroid}\label{sec:Atiyah_algebroid}
	
	An important Lie algebroid in Contact Geometry is given by the Atiyah Lie algebroid. Its role will be clear in Section \ref{sec:dictionary}, while in this paragraph we describe it precisely.

	Let $P$ be a principal $G$-bundle over $M$, with $G$ a Lie group. The Atiyah (or gauge) algebroid is the Lie algebroid $A=TP/G$ over $P/G=M$ where the anchor map is induced by the differential of the projection $P\to M$ and the Lie bracket is induced by noticing that a section of $A$ correspond to a $G$-invariant vector field on $P$. This is the Lie algebroid associated to the gauge groupoid of $P\to M$.

	When the Lie group $G$ is the general linear group $GL(n)$, with $n\geq 0$, and the principal bundle is the frame bundle of a rank $n$ vector bundle $E\to M$, the Atiyah algebroid is the Lie algebroid of the general linear groupoid $\operatorname{GL}(E)\rightrightarrows M$. There is another and easier description of this Lie algebroid that involves the notion of \emph{derivation} of a vector bundle.
	
	Let $E\to M$ be a vector bundle. We can consider the following Lie algebroid $DE\to M$: for any $x\in M$, the fiber $D_xE$ of $DE$ over $x$ consists of derivations of $E\to M$ at the point $x$, i.e., $\mathbbm{R}$-linear maps $\delta\colon\Gamma(E)\to E_x$ such that, for any $f\in C^{\infty}(M)$ and $e\in \Gamma(E)$,
	\begin{equation*}
		\delta(fe)= f(x)\delta(e) + v(f)e_x
	\end{equation*}
	for a, necessarily unique, tangent vector $v\in T_xM$, called the \emph{symbol} of $\delta$, which is also denoted by $\sigma(\delta)$. Sections of $DE$ are derivations of $E$, i.e., $\mathbbm{R}$-linear maps $\Delta\colon \Gamma(E)\to \Gamma(E)$ such that, for any $f\in C^{\infty}(M)$ and $e\in \Gamma(E)$,
	\begin{equation*}
		\Delta(fe)= f\Delta(e) + X(f)e 
	\end{equation*}
	for a, necessarily unique, vector field $X\in \mathfrak{X}(M)$, called the \emph{symbol} of $\Delta$, which is also denoted by $\sigma(\Delta)$. The Lie bracket of $DE$ is just the commutator of derivations and the anchor map is the symbol map $\sigma\colon DE\to TM$. The Leibniz rule \eqref{eq:leibniz_rule} is satisfied. Indeed, for any $\Delta$ and $\Delta'$ derivations of $E$ and $f\in C^\infty(M)$ we have
	\begin{align*}
		[\Delta, f\Delta'] (e) &=\Delta(f\Delta'(e)) - f\Delta'(\Delta(e)) \\ &=f\Delta(\Delta'(e)) + \sigma(\Delta)(f)\Delta'(e) -f\Delta'(\Delta(e)) \\
		&= f[\Delta,\Delta'](e) +\sigma(\Delta)(f) \Delta'(e),
	\end{align*}
	for all $e\in \Gamma(E)$, so $[\Delta, f\Delta']= f[\Delta, \Delta'] +\sigma(\Delta)(f)\Delta'$.
	
	\label{subsec:Atiyah_algebroid}
	The canonical Lie algebroid isomorphism over $\operatorname{id}_M$ existing between the Atiyah algebroid of the principal $\operatorname{GL}(n)$-bundle $\operatorname{Fr}(E)\to M$ and the Lie algebroid $DE$ is explained in \cite[Lemma 2.33]{CF11}. In the following, we refer to the Lie algebroid $DE$ simply as the Atiyah algebroid of $E$. It follows that any derivation $\delta\in D_xE$ is the velocity of a curve of isomorphisms $\Upsilon(\varepsilon)\colon E_x\to E_{\gamma(\varepsilon)}$ with $\Upsilon(0)=\operatorname{id}_{E_x}$, where $\gamma(\varepsilon)$ is a curve on $M$ with $\gamma(0)=x$, i.e., 		
	\begin{equation*}
		\delta(e)= \frac{d}{d\varepsilon}|_{\varepsilon=0} \Upsilon(\varepsilon)^{-1}(e_{\gamma(\varepsilon)}), \quad \text{for all }e\in \Gamma(E).
	\end{equation*}
	In this case we write $\delta= \tfrac{d}{d\varepsilon}|_{\varepsilon=0} \Upsilon(\varepsilon)$ and its symbol is $\sigma (\delta) =\tfrac{d}{d\varepsilon}|_{\varepsilon=0} \gamma(\varepsilon)$. This alternative description of a derivation at a point will be often useful in what follows.
	
	%The symbol map $\sigma \colon DL \to TM$ preserves the Lie bracket. Namely, for any $\Delta, \Delta' \in \Gamma(DE)$, $f\in C^{\infty}(M)$ and $e\in \Gamma(E)$ we get
	%\begin{align*}
	%	[\Delta,\Delta'](fe)&= \Delta(\Delta'(fe)) - \Delta'(\Delta (fe)) \\
	%	&= \Delta(f\Delta'(e)+ X'(f)e) - \Delta'(f\Delta(e)+ X(f)e)\\
	%	&= f(\Delta(\Delta'(e))) + X(f) \Delta'(e) + X'(f) \Delta(e) + X(X'(f)) e \\
	%	&-f(\Delta'(\Delta(e))) - X'(f) \Delta(e) + X(f) \Delta'(e) + X'(X(f)) e \\
	%	&=f[\Delta, \Delta'](e) + [X,X'](f) e, 
	%\end{align*}
	%where $X=\sigma(\Delta)$, and $X'=\sigma(\Delta')$. Then $\sigma([\Delta,\Delta'])=[X,X']\in \mathfrak{X}(M)$.
	
	It is important to note that the symbol map fits in the following short exact sequence of VBs over $M$
	\begin{equation}\label{eq:Spencer}
		\begin{tikzcd}
			0 \arrow[r] &\operatorname{End}E \arrow[r] &DE \arrow[r, "\sigma"] &TM \arrow[r] &0
		\end{tikzcd} ,
	\end{equation}
	where $\operatorname{End}E$ is the VB of endomorphisms, and the map $\operatorname{End}E\to DE$ is the inclusion, i.e., it maps the endomorphism $\phi_x\colon E_x\to E_x$, with $x\in M$, to the derivation $\delta^\phi\in D_xE$ defined by
	\begin{equation}
		\label{eq:delta_phi}
		\delta^\phi(e) = \phi_x(e_x), \quad \text{for all } e\in \Gamma(E).
	\end{equation}
	A connection on $E$ is a right splitting $\nabla\colon TM \to DE$ of \eqref{eq:Spencer}, so it determines a direct sum decomposition
	\begin{equation*}
		DE \cong TM \oplus \operatorname{End}E, \quad \delta \mapsto \big(\sigma(\delta), f_{\nabla}(\delta)\big),
	\end{equation*}
	where $f_{\nabla}\colon DE \to \operatorname{End}E$ is the associated left splitting: 
	\begin{equation*}
		f_{\nabla}(\delta) e_x= \delta(e)- \nabla_{\sigma(\delta)}e, \quad \text{for all } \delta \in D_xE, \quad e\in \Gamma(E).
	\end{equation*}
	It follows that $\rank DE= \dim M + (\rank E)^2$.
	
	\begin{rem}
		\label{rem:regular_VB_morphisms}
		The correspondence $E\leadsto DE$ is functorial in the following sense. Let $E\to M$ and $E' \to M'$ be two VBs. First of all, we say that a VB morphism $(F, f)\colon E \to E'$ covering a smooth map $f \colon M \to M'$, is \emph{regular} if, for any $x \in M$, the restriction $F_x := F|_{E_x} \colon E_x \to E'_{f(x)}$ of $F$ to fibers is an isomorphism. Given a regular VB morphism $(F, f) \colon E\to E'$, a section $e'\in \Gamma(E')$ can be pulled-back to a section $F^{\ast}e'\in \Gamma(E)$, defined by $(F^{\ast}e')_x = F_x^{-1}(e'_{f(x)})$, $x\in M$. The map $DF \colon DE \to DE'$ defined by 
		\begin{equation}
			\label{eq:def_DF}
			DF(\delta)e':=F\big(\delta(F^{\ast}e')\big), \quad \text{for all }\delta\in DE, \quad \text{and } e'\in \Gamma(E'),
		\end{equation}
		is a Lie algebroid morphism covering $f$.
		
		If we see $\delta\in D_xE$ as the velocity $\tfrac{d}{d\varepsilon}|_{\varepsilon=0} \Upsilon(\varepsilon)$ of a curve of isomorphisms $\Upsilon(\varepsilon)\colon E_x\to E_{\gamma(\varepsilon)}$, with $x \in M$, then the derivation $DF(\delta)\in D_{f(x)}E'$ is given by 
		\begin{equation}
			\label{eq:derivation}
			\frac{d}{d\varepsilon}|_{\varepsilon=0}\, F_{\gamma(\varepsilon)}\circ \Upsilon(\varepsilon)\circ F_x^{-1}.
		\end{equation}
		Indeed, if $e'\in \Gamma(E')$, then the derivation \eqref{eq:derivation} on $e'$ is
		\begin{equation*}
			\frac{d}{d\varepsilon}|_{\varepsilon=0} \, F_x\left(\Upsilon(\varepsilon)^{-1}\left(F_{\gamma(\varepsilon)}^{-1}\left(e'_{f(\gamma(\varepsilon))}\right)\right)\right) = F_x\left(\frac{d}{d\varepsilon}|_{\varepsilon=0} \, \Upsilon(\varepsilon)^{-1}\left((F^{\ast}e')_{\gamma(\varepsilon)}\right)\right) = F(\delta(F^{\ast}e')).
		\end{equation*}
		
		It easy to see that
		\begin{equation}
			\label{eq:DFcommutes}
			\sigma \circ DF = df \circ \sigma.
		\end{equation}
		Indeed let $\delta= \tfrac{d}{d\varepsilon}|_{\varepsilon=0} \Upsilon(\varepsilon) \in D_xE$, with $x\in M$, then
		\begin{align*}
			\sigma(DF(\delta)) &= \sigma\left(\frac{d}{d\varepsilon}|_{\varepsilon=0} F_{\gamma(\varepsilon)}\circ \Upsilon(\varepsilon)\circ F_x^{-1}\right) =\frac{d}{d\varepsilon}|_{\varepsilon=0} f(\gamma(\varepsilon))\\ &=df\left(\frac{d}{d\varepsilon}|_{\varepsilon=0} \gamma(\varepsilon)\right) = df(\sigma(\delta)).
		\end{align*}
		Another possible way to prove \eqref{eq:DFcommutes}, without using curves of isomorphisms, is the following: let $\delta \in D_xE$, for any $e'\in \Gamma(E')$ and $g\in C^{\infty}(M')$ we get
		\begin{align*}
			DF(\delta)(ge')&= F(\delta(F^{\ast}(ge')))\\
			&= F \left(\delta ((f^{\ast}g) (F^{\ast}e'))\right)\\
			& = F\big((f^{\ast}g)(x)\delta(F^{\ast}e') + \sigma(\delta)(f^{\ast}g)(F^{\ast}e')_x\big) \\
			& = g(f(x)) DF(\delta) (e') + df(\sigma(\delta))(g)e'_{f(x)},
		\end{align*}
		then, from the Leibniz rule, the symbol of $DF(\delta)\in D_{f(x)}E'$ is $df(\sigma(\delta))\in T_{f(x)}M'$.
	\end{rem}
	
	Let $(F,f)\colon (E\to M)\to (E'\to M')$ be a regular VB morphism. Consider the diagram
	\begin{equation}
		\label{eq:diagram:DF}
		\begin{tikzcd}
			0 \arrow[r] &\operatorname{End}E \arrow[r] \arrow[d, "\operatorname{End}F"] &DE \arrow[r, "\sigma"] \arrow[d, "DF"] &TM \arrow[r] \arrow[d, "df"] &0 \\
			0 \arrow[r] &\operatorname{End}E' \arrow[r] &DE' \arrow[r, "\sigma"'] &TM' \arrow[r] &0
		\end{tikzcd},
	\end{equation}
	where $\operatorname{End}F$ maps $\phi_x\colon E_x\to E_x$ to $F_x\circ \phi_x\circ F_x^{-1}\colon E'_{f(x)}\to E'_{f(x)}$, with $x\in M$. Using Equations \eqref{eq:delta_phi} and \eqref{eq:def_DF} we get
	\begin{align*}
		DF(\delta^\phi)(e')&= F\left(\delta^\phi\left(F^\ast e'\right)\right)= F\left(\phi_x\left(F_x^{-1}\left(e'_{f(x)}\right)\right)\right)\\&= \operatorname{End}F\left(\phi_x\right) \left(e'_{f(x)}\right)= \delta^{\operatorname{End}F(\phi)}(e'),
	\end{align*}
	for all $e'\in \Gamma(E')$ and $x\in M$. Then the leftmost square in Diagram \eqref{eq:diagram:DF} commutes. The rightmost square in Diagram \eqref{eq:diagram:DF} commutes as well because of Equation \eqref{eq:DFcommutes}. The VB morphism $\operatorname{End}F\colon \operatorname{End}E\to \operatorname{End}E'$ is regular. Indeed, for any endomorphism $\phi'_{f(x)}\colon E'_{f(x)}\to E'_{f(x)}$, with $x\in M$, there exists a unique endomorphims $\phi_x=F_x^{-1}\circ \phi'_{(x)}\circ F_x\colon E_x\to E_x$ such that $\operatorname{End}F(\phi_x)=\phi'_{f(x)}$. As explained in \cite[Section 2.1]{ETV19} the rightmost square in Diagram \eqref{eq:diagram:DF} is a pullback diagram. Then we have the following
	\begin{prop}
		\label{prop:derivation_symbol}
		Let $(F,f)\colon (E\to M)\to (E'\to M')$ be a regular VB morphism. The map
		\[
		DF \times \sigma\colon DE \to DE'\mathbin{{}_{\sigma}\times_{df}} TM, \quad \delta \mapsto (DF(\delta), \sigma(\delta)),
		\]
		is a VB isomorphism.
	\end{prop}
	%\begin{proof}
	%	Let $\delta'\in D_{f(x)}E'$ and $v\in T_xM$ be such that $\sigma(\delta')=df(v)\in T_{f(x)}M'$, with $x\in M$. Since $\sigma\colon DE\to TM$ is surjective, there exists a derivation $\widetilde{\delta}\in D_xE$ such that $\sigma(\widetilde{\delta})=v$. The derivation $\delta' -DF(\widetilde{\delta})\in D_{f(x)}E'$ is such that
	%	\[
	%		\sigma(\delta' -DF(\widetilde{\delta}))= \sigma(\delta') - df(\sigma(\widetilde{\delta}))= \sigma(\delta') - df(v)=0,
	%	\]
	%	and so there exists $\phi_{f(x)}\in \operatorname{End}E'_{f(x)}$ such that $\delta^{\phi'}= \delta' -DF(\widetilde{\delta})$. As $\operatorname{End}F$ is regular, there exists (a unique) $\phi_x\in \operatorname{End}E_x$ such that $\operatorname{End}F(\phi_x)=\phi'_{f(x)}$. Now, let $\delta=\widetilde{\delta} + \delta^\phi \in D_xE$. Then
	%	\[
	%		DF(\delta) = DF(\widetilde{\delta}) + DF(\delta^\phi)= DF(\widetilde{\delta}) + \delta^{\phi'} = DF(\widetilde{\delta}) + \delta' - DF(\widetilde{\delta})= \delta',
	%	\]
	%	and
	%	\[
	%		\sigma(\delta)= \sigma(\widetilde{\delta}) + \sigma(\delta^\phi)= v,
	%	\]
	%	And the VB morphism $DF\times \sigma$ is surjective. Moreover,
	%	\begin{align*}
	%		\rank (DE'\mathbin{{}_{\sigma}\times_{df}} TM) &= \rank DE' + \rank TM - \rank TM' \\
	%		&=\dim M' + (\rank E')^2 + \dim M - \dim M' \\
	%		&= \dim M +(\rank E)^2 =\rank DE,
	%	\end{align*}
	%	where we used that $\rank E=\rank E'$ becuase $F$ is regular. Hence, $DF\times \sigma $ is a VB isomorphism.
	%\end{proof}
	An easy consequence of Proposition \ref{prop:derivation_symbol} is that any derivation in $DE$ is completely determined by its symbol and the action on pullback section $F^\ast e'\in \Gamma(E)$, with $e'\in \Gamma(E')$.

	Next remark will be useful in the sequel.
	\begin{rem}
		\label{rem:f_nabla_fiberwise}
		Let $(K,k)\colon (E\to M)\to (E'\to M')$ be a regular VB morphism and let $\nabla$ be a connection on $E'$. The left splitting $f_{k^\ast \nabla}$ associated with the pullback connection $k^\ast\nabla$ on $E$ is related to the left splitting $f_\nabla$ through $DK$. Indeed, for any $\delta\in D_xE$ and $e'\in \Gamma(E')$, we have
		\begin{align*}
			\big(DK\circ f_{k^\ast\nabla}\big)(\delta)(e')&= K\big(f_{k^\ast\nabla}(\delta) (K^\ast e')\big)\\
			&= K\big(\big(\delta-(k^\ast\nabla)_v\big)(K^\ast e')\big)\\
			&= K\big(\delta(K^\ast e')\big) + \nabla_{dk(v)}e'.
		\end{align*}
		On the other hand, we have
		\[
			(f_\nabla \circ DK)(\delta)(e')= \big(DK(\delta) - \nabla_{df(v)}\big)e' = K(\delta(K^\ast e')) - \nabla_{dk(v)} e',
		\]
		where $v=\sigma(\delta)\in T_xM$ is the symbol of $\delta$, hence $DK\circ f_{k^\ast\nabla}=f_\nabla\circ DK$ as announced.
	\end{rem}

	Let $E\to M$ be a vector bundle. The Atiyah algebroid $DE$ naturally acts on $E$: the action of a derivation on a section is just the tautological one. Then there is a differential $\dA$ on the graded vector space $\OA^{\bullet}(E):= \Gamma(\operatorname{Alt}^{\bullet}(DE, E))$ of $E$-valued alternating forms on $DE$ defined by Equation \eqref{eq:differential}. The cochain complex $(\OA^{\bullet}(E), \dA)$ is sometimes called the \emph{der-complex} \cite{R80} and it is actually acyclic. Even more, it possesses a canonical contracting homotopy given by the contraction $\iota_{\mathbbm{I}}\colon \OA^{\bullet}(E)\to \OA^{\bullet-1}(E)$ with the \emph{identity derivation} $\mathbbm{I}\colon \Gamma(E)\to \Gamma(E)$. In the sequel, cochains in  $(\OA^{\bullet}(E), \dA)$ will be called \emph{Atiyah forms} (on $E$), and $\dA$ will be called the \emph{Atiyah differential}.
	
	Any regular VB morphism $(F, f)\colon E\to E'$ induces a \emph{pullback of Atiyah forms}, $F^{\ast}\colon \OA^{\bullet}(E')\to \OA^{\bullet}(E)$: for any $x \in M$,
	\begin{equation}
		\label{eq:pullback_Atiyah_forms}
		(F^{\ast}\omega)_{x}(\delta_1, \dots, \delta_k):=F_x^{-1}\big(\omega_{f(x)}(DF(\delta_1), \dots, DF(\delta_k))\big),
	\end{equation}
	for all $\delta_1, \dots, \delta_k \in D_xE$ and $\omega \in \OA^k(E')$.
	
	The relation between pullback of Atiyah forms and pullback of vector valued forms (Equation \eqref{eq:pullback_vector_valued_forms}) will be explained in Section \ref{sec:dictionary} in the case when $E=L\to M$ is a line bundle.
	
	We are particularly interested in the case when the vector bundle $E=L\to M$ is a line bundle. Recall that a \emph{first order differential operator} on $L$ (see \cite[Definition 9.57]{Ne20} for a more general notion of differential operator) is an $\mathbbm{R}$-linear map $\square\colon \Gamma(L)\to \Gamma(L)$ such that, for any $f,g\in C^{\infty}(M)$,
	\begin{equation}
		\label{eq:bracket_do}
		[[\square,f], g](\lambda)=0,
	\end{equation}
	where the bracket in \eqref{eq:bracket_do} is defined by setting
	\[
		[\square, f](\lambda):=\square(f\lambda)-f\square(\lambda), \quad f\in C^\infty(M), \lambda\in \Gamma(L).
	\]
	By \cite[Theorem 11.64]{Ne20}, first order differential operators on $L$ form a $C^{\infty}(M)$-module that is isomorphic to $\operatorname{Hom}(J^1L,L)$, where $J^1L$ is the \emph{first jet bundle} of $L$ (see \cite{Sa89} for more details on jet bundles). Moreover, derivations on a line bundle $L$ exhaust first order differential operators on $L$ (see \cite[Section 1.1]{To17} for more details). 
	
	Summarizing, the Atiyah algebroid $DL$ of a line bundle $L$ is isomorphic to $\operatorname{Hom}(J^1L,L)$.
%	\begin{rem}
%		\label{rem:jet_bundle}
%		Let $E\to M$ be a VB and let $x\in M$. Two local sections $e,e'$ of $E$ around $x$ are \emph{$1$-equivalent at $x$} if $e(x)=e'(x)$ and there exists an adapted chart $(x_i,u_j)$ around $e(x)$ such that 
%		\[
%			\frac{\partial e_j}{\partial x_i}|_x = \frac{\partial e_j'}{\partial x_i}|_x,
%		\]
%		for all $i$ and $j$, where $e_j= u_j\circ e$ (likewise for $e'$). This is an equivalence relation and the equivalence class of $e$ is denoted by $j_x^1e$ and is called the \emph{first jet} of $e$ at $x$.
%		
%		Denoted by $J^1_xE$ the space of all first jets of sections of $E$ at $x$, then
%		\[
%			J^1E= \bigsqcup_{x\in M} J^1_x E
%		\]
%		is a VB over $M$, called the \emph{first jet bundle} of $E$. See, e.g., \cite{Sa89} for more details and the definition of a general order jet bundle.
%	\end{rem}
	
	\section{Morita equivalence and differentiable stacks}\label{sec:stack}
	Two Lie groupoids are isomorphic if there exists a Lie groupoid ismorphism between them. This equivalence relation between Lie groupoids is very strict. A weaker equivalence relation is given by the \emph{Morita equivalence}. The latter, for instance, indentifies two Lie groupoids that \lq \lq have the same orbit spaces and the same \emph{transverse geometry}\rq \rq, see Theorem \ref{theo:Morita}, below for a more precise statement.
	
	There are several equivalent definitions of Morita equivalence between Lie groupoids. Indeed one can define Morita equivalence using Lie groupoid actions (see Definition \ref{def:principal_bibundle} below) or using particular Lie groupoid morphisms called \emph{Morita maps}. There are also different notions of Morita maps and we recall the relation between them in Remark \ref{rem:Morita_maps}. The notion that we prefer to fix is that used in \cite[]{MM03} under the name of \emph{weak equivalence}.
	
	Let $H\rightrightarrows N$ and $G\rightrightarrows M$ be Lie groupoids.
	\begin{definition}
		\label{def:Morita_map}
		A Lie groupoid morphism $f\colon H\to G$ is a \emph{Morita map} if
		\begin{itemize}
			\item $f$ is \emph{fully faithful}, i.e., the diagram
			\begin{equation}
				\label{eq:fully_faithful}
				\begin{tikzcd}
					H \arrow[r, "f"] \arrow[d, "(s{,}t)"'] & G \arrow[d, "(s{,}t)"] \\
					N\times N \arrow[r, "f\times f"'] & M\times M
				\end{tikzcd}
			\end{equation}
			is a pullback diagram;
			\item $f$ is \emph{essentially surjective}, i.e., the map
			\begin{equation*}
				G\mathbin{{}_{s}\times_{f}} N \to M, \quad (x\leftarrow f(y), y) \mapsto x
			\end{equation*}
			is a surjective submersion.
		\end{itemize}
	\end{definition}
	In particular Lie groupoid isomorphisms (Definition \ref{def:lie_morphism}) are Morita maps.
	
	We recall an useful characterization of the Morita maps given by del Hoyo. First, a Lie groupoid morphism $f\colon H\to G$, by Remark \ref{rem:basic_facts}, induces a continuous map between the orbit spaces $N/H$ and $M/G$ and, by Remark \ref{rem:normal_representation_morphism}, induces a Lie group representation morphism between the normal representations of $H_y$ and $G_{f(y)}$ on $N_yO$ and $N_{f(y)}O'$ respectively, where $O$ is the orbit through $y\in N$ and $O'$ is the orbit through $f(y)\in M$. Then we have the following 
	\begin{theo}[{\cite[Theorem 4.3.1]{dH13}}]
		\label{theo:Morita}
		A Lie groupoid morphism $f\colon (H\rightrightarrows N) \to (G\rightrightarrows M)$ is a Morita map if and only if it yields an homeomorphism between the orbit spaces $N/H$ and $M/G$ and, for each $y\in N$, an isomorphism between the normal representations of $H_y$ and $G_{f(y)}$ on $N_yO$ and $N_{f(y)}O'$ respectively.
	\end{theo}
	
	An easy consequence of Theorem \ref{theo:Morita} is that composition of Morita maps is again a Morita map. Another consequence is the following
	\begin{lemma}[Two-out-of-three]
		\label{lemma:two-out-of-three}
		In a commutative triangle of Lie groupoid morphisms, if two of the three morphisms are Morita maps, then the third is so as well.
	\end{lemma}
	Moreover, we also have the following
	\begin{lemma}[{\cite[Corollary 4.3.2]{dH13}}]
		If $f,f'\colon H\to G$ are two Lie groupoid morphisms such that there exists a natural transformation between them, then $f$ is a Morita map if and only if $f'$ is so.
	\end{lemma}
	
	Now we are ready to recall the Morita equivalence between Lie groupoids.
	\begin{definition}
		Two Lie groupoids $G$ and $G'$ are \emph{Morita equivalent} if there exists a Lie groupoid $H$ and two Morita maps $G\leftarrow H \rightarrow G'$.
	\end{definition}

	Morita equivalence is an equivalence relation. This is discussed in the remark following Proposition 5.12 in \cite{MM03} but we prefer to recall it here, as it serves as inspiration for other equivalence relations established later in the thesis.
	\begin{prop}
		\label{prop:Morita_equivalence_relation}
		Morita equivalence between Lie groupoids is an equivalence relation.
	\end{prop}
	\begin{proof}
		Reflexivity and symmetry are obvious. For the transitivity we need the homotopy fiber product construction (see Example \ref{ex:homotopy_pullback}). If $G_1\leftarrow H_1 \to G_2$, $G_2 \leftarrow H_2 \to G_3$ are Morita equivalences, then we consider the homotopy fiber product $H_1\times_{G_2}^h H_2$ of $f_1\colon H_1\to G_2$ and $f_2 \colon H_2\to G_2$. If $f_1$ (respectively $f_2$) is a Morita map, then, from \cite[Proposition 5.12]{MM03} the homotopy fiber product exists and the projection $H_1\times^h_{G_2} H_2 \to H_2$ (respectively $H_1\times^h_{G_2} H_2 \to H_1$) is a Morita map. Hence every Lie groupoid morphism in the diagram
		\begin{equation*}
			\scriptsize
			\begin{tikzcd}%[column sep=-3]
				& & H_1\times^h_{G_2} H_2 \arrow[dr] \arrow[dl] \\
				& H_1 \arrow[dr] \arrow[dl] & & H_2 \arrow[dr]\arrow[dl] \\
				G_1 & & G_2 & & G_3
			\end{tikzcd}
		\end{equation*}
		is a Morita map and $G_1$ and $G_3$ are Morita equivalent as well.
	\end{proof}
	
	\begin{example}
		Following \cite[Section 4.5]{dH13} we have that a Lie groupoid $G\rightrightarrows M$ is Morita equivalent to a unit groupoid $B\rightrightarrows B$ if and only if $G$ is a submersion groupoid (see Example \ref{ex:submersion_groupoid}) with quotient $M/G\cong B$. Moreover, a Lie groupoid $G\rightrightarrows M$ is Morita equivalent to a Lie group $H$ if and only if $G$ is transitive (see Example \ref{ex:pair_groupoid} for the definition of transitive Lie groupoid) and all the isotropy groups are isomorphic to $H$.
	\end{example}
		
	As already mentioned, in the literature are different definitions of the Morita equivalence using Lie groupoid morphisms. We discuss the relation between those in the following
	\begin{rem}
		\label{rem:Morita_maps}
		In \cite{dH13} the author defines a fully faithful morphism as a Lie groupoid morphism $f\colon H\to G$ such that the diagram \eqref{eq:fully_faithful} is a \emph{good pullback diagram}, meaning that it is a pullback diagram of the underlying topological spaces and it induces pullbacks between the tangent spaces. This notion is stronger than that used in Definition \ref{def:Morita_map}, but, as explained in \cite[Remark 4.2.4]{dH13}, when the Lie groupoid morphism $f$ is also essentially surjective, then the map $(s,t)\colon G\to M\times M$ and $f\times f\colon N\times N \to M\times M$ are transverse and the two definitions of Morita map agree.
		
		In \cite{BXu11} the authors define a Morita map as a fully faithful Lie groupoid morphism $f\colon H\to G$ such that the map between the bases $f\colon N\to M$ is a surjective submersion. The equivalence relation generated by this notion agrees with the one generated by Morita maps in the sense of Definition \ref{def:Morita_map}. Indeed, following \cite[Section 5.4]{MM03}, if $G$ and $G'$ are Morita equivalent Lie groupoids, then it is always possible to find a Lie groupoid $H$ and Morita maps $H\to G$ and $H\to G'$ that are surjective submersions at the level of the bases.
		
		Morita equivalences can be further simplified: let $G, G'$ and $H$ be Lie groupoids and let $G\xleftarrow{f} H \xrightarrow{f'} G'$ Morita maps. From the fully faithfulness we get that $H$ is isomorphic to the pullback groupoids $f^{!}G$ and $f'^{!}G'$ (see Example \ref{ex:pullback_groupoid}). This means that it is possible to break any Morita equivalence in a chain of simpler Morita equivalences: $G$ is Morita equivalent to $f^{!}G \cong f'^{!}G'$, which is Morita equivalent to $G'$.
	\end{rem}
	
	Another approach to Morita equivalence is by using \emph{principal bibundles}. Let $G'$ and $G$ be Lie groupoids.
	\begin{definition}
		\label{def:principal_bibundle}
		A \emph{principal bibundle} is a manifold $P$ with a left action of $G$ on $P$ and a right action of $G'$ on $P$ such that the two actions commute and the moment map of one is invariant for the other one.
	\end{definition}
	\begin{theo}[{\cite[Theorem 2.26]{BXu11}}]
		Two Lie groupoids are Morita equivalent if and only if there exists a principal bibundle between them.
	\end{theo}
	
	We conclude this chapter introducing the notion of \emph{differentiable stack}.
	\begin{definition}
		A \emph{differentiable stack} is a Morita equivalence class of Lie groupoids. We indicate by $[M/G]$ the differentiable stacks consisting of Lie groupoids Morita equivalent to $G\rightrightarrows M$. In this case we also say that $G\rightrightarrows M$ is a \emph{presentation} of $[M/G]$.
	\end{definition}
	\begin{rem}
		Differentiable stacks can be also defined in a different way: as an algebraic stacks equipped with a presentation. See, e.~g., \cite{BXu11} for this approach and for the equivalence of the definitions.
	\end{rem}

	\chapter{VB-groupoids and VB-Morita equivalence}\label{ch:vbg}
	
This chapter is devoted to VB-groupoids, which can be understood either as vector bundles in the category of Lie groupoids or as Lie groupoids in the category of vector bundles. The chapter is organized into four sections.

In the first and largest section, we recall the definition of VB-groupoids and present several useful examples. Motivated by the role of twisted dual vector bundles in Contact Geometry, we introduce the notion of dual VB-groupoids twisted by a trivial core one. To this end, we first examine \emph{trivial core VB-groupoids}, in particular VB-groupoids that are also line bundles (Section \ref{sec:trivial_core}). Then, we define the tensor product between a VB-groupoid and a trivial core one and then we combine the latter construction with the one of the dual VB-groupoid (Section \ref{sec:twisted_dual}). Finally, the Atiyah VB-groupoid induced by a \emph{line bundle-groupoid} (Section \ref{sec:Atiyah_VBG}).

The second section recalls the notion of \emph{representations up to homotopy}, a generalization of ordinary representations (Definition\ref{def:representation}), and explores their relationship with VB-groupoids. Additionally, we introduce new operations on representations up to homotopy, such as tensor product and twisted dual.

In Section \ref{sec:VB-Morita} we recall from \cite{dHO20} the Morita theory between VB-groupoids and we prove some results on the twisted dual ones. 

Finally, the last section introduces the concept of \emph{natural transformation} between VB-groupoid morphisms, culminating in a useful characterization theorem (Section \ref{sec:lni}). 

\section{VB-groupoids}
We begin this section by recalling \emph{VB-groupoids}, i.e., VBs in the category of Lie groupoids, along with some of their properties. We then focus on specific examples that will be useful later. In particular, we discuss a particular class of VBGs, the ones with trivial core and VB-groupoids that are, in a certain sense, line bundles over Lie groupoids (Section \ref{sec:trivial_core}), then we introduce dual VB-groupoids twisted by a trivial core one (Section \ref{sec:twisted_dual}), and, finally, the VB-groupoid obtained by applying the functor $D$ (see Remark \ref{rem:regular_VB_morphisms}) to a line bundle groupoid (Section \ref{sec:Atiyah_VBG}).
\begin{definition}
	A \emph{VB-groupoid (VBG)} $(V\rightrightarrows V_M; G\rightrightarrows M)$ is a diagram
	\begin{equation*}
		\begin{tikzcd}
			V \arrow[r,shift left=0.5ex] \arrow[r, shift right=0.5 ex] \arrow[d] &V_M \arrow[d] \\
			G \arrow[r,shift left=0.5ex] \arrow[r, shift right=0.5 ex] & M
		\end{tikzcd}
	\end{equation*}
	where $V\rightrightarrows V_M$ and $G\rightrightarrows M$ are Lie groupoids, $V\to G$ and $V_M\to M$ are vector bundles, and all the structure maps of the Lie groupoid $V\rightrightarrows V_M$ are VB morphisms.
\end{definition}
Sometimes, in the following, we simply say that $V$ is a VBG over $G$, or even just that $V$ is a VBG, when the context makes it clear what $V_M$ and $M$ are.

In what follows, abusing the notation, we denote the structure maps of both groupoids $G\rightrightarrows M$ and $V\rightrightarrows V_M$ in the same way: $s$, $t$, $m$, $i$, and $u$ for the source, target, multiplication, inverse and unit, respectively. The vector bundle $V_M\to M$ is called the \emph{side bundle} of $V$ and the vector bundle $V\to G$ is called the \emph{total bundle} of $V$. For any $k\in \mathbbm{N}$, the manifold $V^{(k)}$ in the nerve $V^{(\bullet)}$ of the Lie groupoid $V\rightrightarrows V_M$ is a VB over $G^{(k)}$, where the projection $V^{(k)}\to G^{(k)}$ maps $k$-composable arrows $(v_1, \dots, v_k)\in V_{g_1}\oplus \dots \oplus V_{g_k}$ to the $k$-composable arrows $(g_1, \dots, g_k)\in G^{(k)}$. Moreover, the degeneracy and face maps of the nerve $V^{(\bullet)}$ are VB morphisms. Then, $V^{(\bullet)}\to G^{(\bullet)}$ together with the face and degeneracy maps is a simplicial VB, i.e., a simplicial object in the category of VBs (or a VB object in the category of simplicial manifolds). We call it the \emph{nerve} of the VBG $V$.

\begin{rem}
	Let $(V\rightrightarrows V_M;G\rightrightarrows M)$ be a VBG. We recall from \cite[Section 3.1]{GSM17} that, for any $(g,h)\in G^{(2)}$ and $x\in M$, we have
	\[
	0^V_g\cdot 0^V_h=0^V_{gh}, \quad \text{and} \quad 0^V_{u(x)}=u(0^{V_M}_x),
	\]
	where, here and in what follows, the dot ``$\cdot$'' indicates the multiplication in $V$ and $0^V_g$ is the zero vector in $V_g$, likewise for $V_M$. Let $(g,h)\in G^{(2)}$. For any $v_1, v_2\in V_g$ and $u_1, u_2\in V_h$ such that $s(v_1)= t(u_1)$ and $s(v_2)=t(u_2)$, we have
	\begin{equation}
		\label{eq:interchange_law}
		(v_1+v_2)\cdot (u_1+u_2)= v_1\cdot u_1 + v_2\cdot u_2.
	\end{equation}
	Equation \eqref{eq:interchange_law} is called the \emph{interchange law}.		
\end{rem}

Let $(V\rightrightarrows V_M;G\rightrightarrows M)$ be a VBG.
\begin{definition}
	The \emph{(right) core} $C$ of $V$ is the VB over $M$ given by $\ker(s\colon V\to V_M)|_M$. The \emph{core-anchor} of $V$ is the VB morphism $t|_C\colon C\to V_M$ given by the restriction of the target map of $V$ to $C$.
\end{definition}
%Associated to any VBG $(V\rightrightarrows V_M; G\rightrightarrows M)$ there is the VB over $M$ given by $\ker(s\colon V\to V_M)|_M$. This VB is called the \emph{(right) core} of $V$ and it is also indicated by $C$. Moreover, the core $C$ comes together with the VB morphism $t|_C\colon C\to V_M$, the restriction of the target map of $V$ to $C$, which is sometimes called the \emph{core-anchor}. 
The core-anchor of a VBG $V$ is often regarded as a $2$-term cochain complex of VBs over $M$
\begin{equation*}
	\begin{tikzcd}
		0 \arrow[r] & C \arrow[r, "t|_C"] & V_M \arrow[r] & 0,
	\end{tikzcd}
\end{equation*}
called the \emph{core complex}. It is standard to assume that the core complex is concentrated in degrees $-1$,$0$. The value of the core complex at the point $x\in M$
\begin{equation*}
	\begin{tikzcd}
		0 \arrow[r] & C_x \arrow[r, "t|_C"] & V_{M,x} \arrow[r] & 0,
	\end{tikzcd}
\end{equation*}
is also called the \emph{fiber of $V$ over $x$}.

Let $(V\rightrightarrows V_M; G\rightrightarrows M)$ be a VBG. The core $C$ and the side bundle $V_M$ fit in the following short exact sequence of VBs over $G$
\begin{equation}
	\label{eq:core_SES}
	\begin{tikzcd}
		0 \arrow[r] &t^\ast C \arrow[r] & V \arrow[r, "s"] & s^\ast V_M \arrow[r] &0
	\end{tikzcd},
\end{equation}
where the VB morphism $t^\ast C\to V$ is just the right multiplication by the zero vector, namely, it maps the pair $(g,v)\in G\times_M C$, with $v\in C_{t(g)}$, to $v\cdot 0^V_g\in V_g$. The unit map $u\colon V_M\to V$ is a canonical splitting of the short exact sequence of VBs over $M$
\begin{equation}
	\label{eq:canonical splitting}
	\begin{tikzcd}
		0 \arrow[r] & C \arrow[r] & V|_M \arrow[r, "s"] &  V_M \arrow[r] \arrow[l, bend left, "u"] &0
	\end{tikzcd},
\end{equation}
got as the restriction of the short exact sequence \eqref{eq:core_SES} to $M$. Then $V|_M\cong V_M\oplus C$. The sequence \ref{eq:core_SES} does not possess a canonical splitting, so we recall the following	
\begin{definition}[{\cite[Definition 3.8]{GSM17}}]
	\label{def:right_horizontal_splitting}
	A \emph{right-horizontal lift} is a right splitting $h\colon s^\ast V_M\to V$ of the short exact sequence of VBs over $G$ \eqref{eq:core_SES},	which agrees with $u \colon V_M \to V$ on units.
\end{definition} 

\begin{rem}
	\label{rem:left_core}
	Similarly to the definition of \emph{right core} of a VBG $(V\rightrightarrows V_M; G\rightrightarrows M)$, one can introduce the notion of \emph{left core} $C^L$ as the vector bundle over $M$ given by the restriction to $M$ of the kernel of the target map, i.e., $C^L=\ker (t\colon V\to V_M)|_M$. The left core is isomorphic to the right core through the inverse map $i\colon V\to V$. Moreover, there is a left core version of the short exact sequence \eqref{eq:core_SES}:
	\begin{equation}
		\label{eq:left_core_SES}
		\begin{tikzcd}
			0 \arrow[r] & s^\ast C^L \arrow[r] & V \arrow[r, "t"] & t^\ast V_M \arrow[r] & 0
		\end{tikzcd},
	\end{equation}
	where the VB morphism $s^\ast C^L\to V$ is just the left multiplication by the zero vector, namely, it maps the pair $(g,v)\in G\times_M C^L$, with $v\in C^L_{s(g)}$, to $0^V_g\cdot v\in V_g$.
\end{rem}

Let $V$ be a VBG with core $C$.
\begin{definition}
	\label{def:right_invariant_section}
	A \emph{right-invariant section} of $V$ is a section $\chi \in \Gamma(V)$ such that the restriction $\chi|_M$ to $M$ is a section of the core $C$ and, for any $(g,h)\in G^{(2)}$, 
	\[
	\chi_{gh}= \chi_g\cdot 0^V_h.
	\]
	A \emph{left-invariant section} of $V$ is a section $\chi\in \Gamma(V)$ such that the restriction $\chi|_M$ to $M$ is a section of the left core $C^L$ (see Remark \ref{rem:left_core}) and, for any $(g,h)\in G^{(2)}$,
	\[
	\chi_{gh}= 0^V_g \cdot \chi_h.
	\]
\end{definition}
\begin{rem}
	\label{rem:right_inv_section}
	Using the right (respectively left) multiplication by the zero vector (as in Equation \ref{eq:core_SES} (respectively \ref{eq:left_core_SES})) we can promote a section of the core $C$ to a right-invariant (respectively left-invariant) section of $V$. Indeed,	for any $\chi\in \Gamma(C)$, the section $\overrightarrow{\chi}\in \Gamma(V)$, defined by
	\[
	\overrightarrow{\chi}_g= \chi_{t(g)} \cdot 0^V_g, \quad g\in G,
	\]
	is a right-invariant section of $V$ and the section $\overleftarrow{\chi}\in \Gamma(V)$, defined by
	\[
	\overleftarrow{\chi}_g= 0^V_g \cdot \chi_{t(g)}^{-1}, \quad g\in G,
	\]
	is a left-invariant section of $V$.
\end{rem}

Now we present some examples of VBGs. We begin with the VBG obtained by applying the tangent functor to a Lie groupoid. 
\begin{example}[Tangent VB-groupoid]
	\label{ex:tangent_VBG}
	Let $G\rightrightarrows M$ be a Lie groupoid. As discussed in Example \ref{ex:tangent_groupoid}, $TG\rightrightarrows TM$ is a Lie groupoid whose structure maps are the differential of the structure maps of $G$, which nerve is $TG^{(\bullet)}$ obtained by applying the tangent functor to the nerve $G^{(\bullet)}$ of $G$.
	
	Actually, $TG$ is not just a Lie groupoid, but it is a VBG over $G$ called the \emph{tangent VB-groupoid}. The core of $TG$ is $A$, the Lie algebroid of $G$ (see Remark \ref{rem:Lie_algebroid}), and the core-anchor is just the anchor map $\rho\colon A\to TM$. Sometimes the core complex associated to the tangent VBG
	\begin{equation*}
		\begin{tikzcd}
			0 \arrow[r] &A \arrow[r, "\rho"] & TM \arrow[r] &0
		\end{tikzcd}
	\end{equation*}
	is called the \emph{tangent complex} of $G$.
	
	Notice that the left core of $TG$ agrees with the Lie algebroid $A^L$ introduced at the end of Remark \ref{rem:Lie_algebroid}. The short exact sequence \eqref{eq:core_SES} in the case of the tangent VBG $TG$ reduces to the short exact sequence \eqref{eq:tangentSES} and an Ehresmann connection on $G$ (see Definition \ref{def:Ehresmann_connection}) is just a right-horizontal lift of the tangent VBG $TG$. Finally, a right-invariant (respectively left-invariant) section of $TG$ (see Definition \ref{def:right_invariant_section}) is just a right-invariant (respectively left-invariant) vector field (see Definition \ref{def:right_invariant_vector_fields}) on $G$. %Indeed, the differential of the right multiplication \eqref{eq:right_multiplication} (respectively left multiplication \eqref{eq:left_multiplication}) is simply the right (respectively left) multiplication by the zero vector (see the last part of Example \ref{ex:tangent_VBG}).
\end{example}
The next example is the direct sum between two VBGs.
\begin{example}[Direct sum VBG]
	\label{ex:direct_sum_VBG}
	Let $(V\rightrightarrows V_M; G\rightrightarrows M)$ and $(W\rightrightarrows W_M; G\rightrightarrows M)$ be VBGs over the same base. Then we can define the \emph{direct sum VBG} given by $(V\oplus W\rightrightarrows V_M\oplus W_M; G\rightrightarrows M)$, whose structure maps are the direct sum of the structure maps of the components. The core is simply the direct sum $C\oplus D$ of the cores $C$ and $D$ of $V$ and $W$ respectively and the core complex
	\begin{equation*}
		\begin{tikzcd}
			0 \arrow[r] & C\oplus D \arrow[r] & V_M \oplus W_M \arrow[r] & 0
		\end{tikzcd}
	\end{equation*}
	is the direct sum of the core complexes of $V$ and $W$.
\end{example}
In the next example we present the pullback of a VBG along a Lie groupoid morphism. This will be useful in order to discuss the Morita equivalence between the dual and twisted dual VBGs (see Section \ref{sec:VB-Morita}).
\begin{example}[Pullback VBG]
	\label{ex:pullback_VBG}
	Let $(V\rightrightarrows V_M; G\rightrightarrows M)$ be a VBG and let $f\colon (H\rightrightarrows N)\to (G\rightrightarrows M)$ be a Lie groupoid morphism. The \emph{pullback VBG} of $V$ along $f$ is the VBG $(f^\ast V \rightrightarrows f^\ast V_M;H \rightrightarrows N)$ where $f^\ast V =H\times_G V$ and $f^\ast V_M= N\times_M V_M$ are the pullback bundles along $f$ over $H$ and $N$ respectively, and the structure maps of $f^\ast V$ are simply given by the structure maps of $H$ in the first entry and the structure maps of $V$ in the second entry. The core of $f^\ast V$ is given by
	\[
	\big\{(x,v)\in N\times_M V_M \, | \, s(v)=0\big\} = f^\ast C,
	\]
	where $C$ is the core of $V$. The core anchor of $f^\ast V$ is fiberwise given by the core anchor of $V$.
\end{example}

Now we recall the construction of the dual VBG of a VBG.
\begin{example}[Dual VBG] \label{ex:dual_VBG}
	Let $(V\rightrightarrows V_M; G\rightrightarrows M)$ be a VBG with core $C$. The \emph{dual VB-groupoid} is the VB-groupoid $(V^{\ast}\rightrightarrows C^{\ast}; G\rightrightarrows M)$ whose structure maps are the following:
	\begin{itemize}
		\item for any $\psi\in V^{\ast}_g$, with $g\in G$, the source and the target maps are defined by
		\begin{equation*}
			\left\langle s(\psi), c\right\rangle = - \left\langle \psi, 0^V_g \cdot c^{-1} \right\rangle, \quad \left\langle t(\psi), c'\right\rangle = \left\langle \psi, c'\cdot 0^V_g\right\rangle,
		\end{equation*}
		for all $c\in C_{s(g)}$ and $c'\in C_{t(g)}$, where the bracket $\langle -,-\rangle$ denote the duality pairing between $V$ and $V^{\ast}$, i.e., $\langle-,-\rangle \colon V^{\ast} \times_G V\to \mathbbm{R}$, $\langle \psi,v\rangle :=\psi(v)$;
		\item for any $\psi\in C^{\ast}_x$, with $x\in M$, the unit maps $\psi$ to $u(\psi)=\psi \circ \pr_C\in V|_M^{\ast}\subseteq V^{\ast}$, where $\pr_C\colon V|_M\to C$ is given by the canonical splitting \eqref{eq:canonical splitting};
		\item for any two composable arrows $\psi\in V_g^{\ast}$ and $\psi'\in V_{g'}^{\ast}$, with $(g,g')\in G^{(2)}$, the multiplication is defined by
		\begin{equation*}
			\left\langle \psi\psi', vv'\right\rangle = \left\langle \psi, v\right\rangle + \left\langle \psi', v'\right\rangle,
		\end{equation*}
		where we used that any vector in $V_{gg'}$ can be seen as the multiplication of $v\in V_g$ and $v'\in V_{g'}$;
		\item for any $\psi\in V_{g}^{\ast}$, the inverse $\psi^{-1}\in V_{g^{-1}}^{\ast}$ is defined by
		\begin{equation*}
			\left\langle \psi^{-1}, v\right\rangle = -\left\langle \psi, v^{-1}\right\rangle, \quad \text{for all } v\in V_{g^{-1}}.
		\end{equation*}
	\end{itemize}
	The core of the dual VB-groupoid $(V^{\ast}\rightrightarrows C^{\ast}; G\rightrightarrows M)$ agrees with $V_M^{\ast}$. Indeed, $V|_M=V_M\oplus C$ then $V^{\ast}|_M\cong V_M^{\ast}\oplus C^{\ast}$. Hence, the core complex
	\begin{equation*}
		\begin{tikzcd}
			0 \arrow[r] &V_M^{\ast} \arrow[r, "t|_C^{\ast}"] & C^{\ast} \arrow[r] &0
		\end{tikzcd}
	\end{equation*}
	is just the dual complex of the core complex of $V$ shifted by $+1$.
	
	Finally, if $h\colon s^\ast V_M \to V$ is a right-horizontal lift for $V$, then $\varpi= h\circ ds -\operatorname{id}\colon V\to t^\ast C$, the left splitting of Sequence \eqref{eq:core_SES} corresponding to $h$, determines a right-horizontal lift $h^\ast\colon s^\ast C^\ast\to V^\ast$ of $V^\ast$ simply by putting
	\begin{equation}
		\label{eq:dual_splitting}
		\langle h^\ast_g(\psi), v\rangle = \langle \psi, \varpi_{g^{-1}}(v^{-1})\rangle,
	\end{equation}
	for all $\psi\in C^\ast_{s(g)}$ and $v\in V_g$, with $g\in G$.
	
	A particular case is given by the dual of the tangent VB-groupoid $TG$. In this case we get $(T^{\ast}G\rightrightarrows A^{\ast}; G\rightrightarrows M)$, where $A$ is the Lie algebroid of $G$, which is called the \emph{cotangent VB-groupoid} of $G$. Its core complex is simply given by
	\begin{equation*}
		\begin{tikzcd}
			0 \arrow[r] & T^{\ast} M \arrow[r, "\rho^{\ast}"] & A^{\ast} \arrow[r] &0.
		\end{tikzcd}\qedhere
	\end{equation*}
\end{example}

Now we recall the notion of morphism between two VBGs. Let $(W\rightrightarrows W_N; H\rightrightarrows N)$ and $(V\rightrightarrows V_M; G\rightrightarrows M)$ be two VBGs.
\begin{definition}
	A \emph{VBG morphism} $(F,f)\colon(W\rightrightarrows W_N; H\rightrightarrows N)\to(V\rightrightarrows V_M; G\rightrightarrows M)$ from $W$ to $V$ is a commutative diagram 
	\begin{equation*}
		{\scriptsize
			\begin{tikzcd}%[row sep=scriptsize, column sep=scriptsize]
				W \arrow[rr, shift left=0.5ex] \arrow[rr, shift right=0.5ex] \arrow[dd] \arrow[dr, "F"] & &W_N \arrow[dd] \arrow[dr, "F"] \\
				&  V \arrow[rr, shift left= 0.5ex, crossing over] \arrow[rr, shift right =0.5ex, crossing over] & &V_M \arrow[dd]\\
				H \arrow[rr, shift left=0.5ex] \arrow[rr, shift right=0.5ex] \arrow[dr, "f"] & & N \arrow[dr, "f"] \\ 
				&  G \arrow[from=uu, crossing over]\arrow[rr, shift left= 0.5ex] \arrow[rr, shift right =0.5ex] & &M
		\end{tikzcd}}
	\end{equation*}
	where $F\colon (W\rightrightarrows W_N)\to (V\rightrightarrows V_M)$ and $f\colon (H\rightrightarrows N)\to (G\rightrightarrows M)$ are Lie groupoid morphisms, and $(F,f)\colon (W\to H)\to (V\to G)$ and $(F,f)\colon (W_N\to N)\to (V_M\to M)$ are VB morphisms. A \emph{VBG isomorphism} is a VBG morphism $(F,f)$ where the VB morphisms $(F,f)$ are VB isomorphisms.
\end{definition}
In the sequel, we sometimes simply the notation by writing $(F,f)\colon (W\to H)\to (V\to G)$, or simply $F\colon W\to V$ when the context makes the other objects clear.

Notice that a VBG morphism $F\colon W\to V$ induces a cochain map 
\[
\begin{tikzcd}
	0 \arrow[r] & D \arrow[r] \arrow[d, "F"] & W|_N \arrow[r] \arrow[d, "F"] & W_N \arrow[r] \arrow[d, "F"] & 0\\
	0 \arrow[r] & C \arrow[r] & V|_M \arrow[r] & V_M \arrow[r] & 0
\end{tikzcd}
\]
between the short exact sequences \eqref{eq:core_SES} for $W$ and $V$ restricted to the units $N$ and $M$ respectively. Let $\pr_D\colon W|_N\to D$ and $\pr_C\colon V|_M\to C$ be the associated left splittings. Then we have
\[
F(\pr_D(w)) = F(w-s(w))= F(w)-F(s(w))= F(w)- s(F(w))= \pr_C(F(w)),
\]
for all $w\in W|_N$. Hence, we proved that
\begin{equation}
	\label{eq:Fcommutes_projection}
	F\circ \pr_D=\pr_C\circ F.
\end{equation}

In Example \ref{ex:homotopy_pullback} we presented the homotopy fiber product of two Lie groupoid morphisms. Similarly, we can define the homotopy fiber product of two VBG morphisms. The resulting VBG will play a key role in proving certain Morita equivalence results between VBGs in what follows.
\begin{example}[The homotopy fiber product VBG]
	\label{ex:homotopy_pullback_VBG}
	For any two VBG morphisms $(F,f)\colon (W\rightrightarrows W_N; H\rightrightarrows N) \to (V\rightrightarrows V_M;G\rightrightarrows M)$ and $(F',f')\colon (W'\rightrightarrows W'_{N'}; H'\rightrightarrows N') \to (V\rightrightarrows V_M;G\rightrightarrows M)$, the homotopy fiber product $W\times_V^h W'$ of $F\colon (W\rightrightarrows W_N)\to (V\rightrightarrows V_M)$ and $F'\colon (W'\rightrightarrows W'_{N'})\to (V\rightrightarrows V_M)$ (see Example \ref{ex:homotopy_pullback}), when it exists, is a VBG over the homotopy fiber product $H\times_G^h H'$ of $f\colon (H\rightrightarrows N)\to (G\rightrightarrows M)$ and $f'\colon (H'\rightrightarrows N')\to (G\rightrightarrows M)$ and it comes together with two VBG morphisms: the projections $\pr_1$ and $\pr_2$ on $W$ and $W'$ respectively.
\end{example}

We now provide some simple examples of VBG morphisms that naturally arise from Lie groupoid morphisms. Specifically, we discuss the VBG morphism between the tangent VBGs (Example \ref{ex:tangent_VBG}) of two Lie groupoids related by a Lie groupoid morphism, and the VBG morphism between a VBG and its pullback along a Lie groupoid morphism(Example \ref{ex:pullback_VBG}).
\begin{example}
	\label{ex:differentialVBGmorphism}
	Let $f\colon H \to G$ be a Lie groupoid morphism. Then $(df, f)$ is a VBG morphism between the tangent VBGs (see Example \ref{ex:tangent_VBG}) $TH$ and $TG$.
\end{example}
\begin{example}
	\label{ex:pullbackVBGmorphism}
	Let $(V\rightrightarrows V_M; G\rightrightarrows M)$ be a VBG and let $f\colon (H\rightrightarrows N)\to (G\rightrightarrows M)$ be a Lie groupoid morphism. Then we can consider the pullback VBG $f^\ast V\rightrightarrows f^\ast V_M $ over $H$ (see Example \ref{ex:pullback_VBG}). The projection $\pr_2\colon f^\ast V\to V$ is a VBG morphism covering $f$.
\end{example}

Finally, it will be useful in the sequel to establish VBG morphisms between the dual VBGs (see Example \ref{ex:dual_VBG}) of two VBGs related by a VBG morphism. Let $(F,f)\colon (W\to H)\to (V\to G)$ be a VBG morphism. If $f$ is not a Lie groupoid isomorphism (Definition \ref{def:lie_morphism}), there is no direct way to relate $W^\ast$ and $V^\ast$ through a VBG morphism. Inspired by \cite{dHO20}[Proposition 5.2], we instead relate the VBG $W^\ast$ with the pullback VBG $f^\ast V^\ast$ of $V^\ast$ along $f$, while $f^\ast V^\ast$ and $V^\ast$ are connected via the VBG morphism described in Example \ref{ex:pullbackVBGmorphism}.

\begin{prop}
	\label{prop:dualVBG_morphism}
	If $(F,f)\colon (W\rightrightarrows W_N;H\rightrightarrows N)\to (V\rightrightarrows V_M; G\rightrightarrows M)$ is a VBG morphism, then the map $F^{\ast}\colon (f^\ast V^{\ast}\rightrightarrows f^\ast C^{\ast})\to (W^{\ast}\rightrightarrows D^{\ast})$ given by $F^{\ast}(h, \psi)=\psi \circ F$, for all $\psi\in V^{\ast}_{f(h)}$, is a VBG morphism (over the identity $\operatorname{id}_H$), from the pullback VBG (see Example \ref{ex:pullback_VBG}) of the dual VBG $V^\ast$ along $f$ to the dual VBG $W^\ast$, where $C$ and $D$ are the cores of $V$ and $W$ respectively.
\end{prop}
\begin{proof}
	The proof is an easy computation of the conditions of a Lie groupoid morphism using the definions of the structure maps of the dual VBG (Example \ref{ex:dual_VBG}) and that $F$ is a VBG morphism. For any $(h,\psi)\in f^{\ast}V^\ast= H\times_G V^\ast$, we have
	\begin{align*}
		\langle s(F^\ast(h,\psi)), d\rangle &= - \langle F^\ast (h,\psi), 0_h\cdot d^{-1}\rangle = -\langle \psi, F(0_h \cdot d^{-1})\rangle \\
		&=-\langle \psi, 0_{f(h)}\cdot F(d)^{-1}\rangle =  \langle s(\psi), F(d)\rangle  \\
		&= \langle F^\ast(s(h),s(\psi)) , d\rangle= \langle F^\ast(s(h,\psi)) , d\rangle,
	\end{align*}
	for all $d\in D_{s(h)}$, and 
	\begin{align*}
		\langle t(F^\ast(h,\psi)), d\rangle &= \langle F^\ast(h,\psi), d\cdot 0_h\rangle = \langle \psi, F(d\cdot 0_h)\rangle \\
		&= \langle \psi, F(d)\cdot 0_{f(h)} \rangle = \langle t(\psi), F(d)\rangle \\
		& = \langle F^\ast (t(h), t(\psi)), d\rangle = \langle F(t(h,\psi)), d\rangle,
	\end{align*}
	for all $d\in D_{t(h)}$. Then $F^\ast$ commutes with the source and the target map.
	
	For any $(y, \psi)\in f^\ast C^\ast$, with $y\in N$, we have
	\begin{align*}
		\langle u(F^\ast(y,\psi)), w\rangle &= \langle F^\ast(y,\psi), \pr_D(w)\rangle \\
		& = \langle \psi, F(\pr_D(w))\rangle \\ 
		&=\langle \psi, \pr_C(F(w))\rangle\\
		& = \langle u(\psi), F(w)\rangle \\ 
		&= \langle F^\ast(y, u(\psi)), w\rangle\\
		&= \langle F^\ast(u(y,\psi)), w\rangle,
	\end{align*}
	for all $w\in W_y$, where we used \eqref{eq:Fcommutes_projection}, and so $u\circ F^\ast= F^\ast \circ u$.
	
	For any composable arrows $(h, \psi), (h',\psi')\in f^\ast V^\ast$, we have
	\begin{align*}
		\langle F^\ast(h,\psi) \cdot F^\ast(h',\psi'), w\cdot w'\rangle &= \langle F^\ast(h,\psi), w\rangle + \langle F^\ast(h',\psi'), w'\rangle \\
		& = \langle \psi, F(w)\rangle + \langle \psi', F(w')\rangle \\
		&= \langle \psi\cdot \psi', F(w)\cdot F(w')\rangle \\
		&= \langle \psi\cdot \psi', F(w\cdot w')\rangle \\
		&= \langle F^\ast (hh', \psi\cdot \psi'), w\cdot w'\rangle\\
		&=\langle F^\ast ((h, \psi)\cdot (h'\psi')), w\cdot w'\rangle,
	\end{align*}
	for all $(w,w')\in W^{(2)}_{h,h'}$, and so $m\circ F^\ast= F^\ast \circ m$.
	The case of the inverse maps is a consequence of the other cases.% but we prefer to write it for completeness. For any $(h,\psi)\in f^\ast V^\ast$, we have
%	\begin{align*}
%		\langle (F^\ast(h,\psi))^{-1}, w\rangle &= -\langle F^\ast(h,\psi), w^{-1}\rangle = -\langle \psi, F(w^{-1})\rangle = -\langle \psi, F(w)^{-1}\rangle = \\
%		&=\langle \psi^{-1}, F(w)\rangle =\langle F^\ast(h^{-1},\psi^{-1}), w\rangle =\langle F^\ast((h,\psi)^{-1}), w\rangle,
%	\end{align*}
%	for all $w\in W_{h^{-1}}$, and so $i\circ F^\ast= F^\ast \circ i$.
\end{proof}
\begin{rem}
	\label{rem:dualF_on_iso}
	Proposition \ref{prop:dualVBG_morphism} simplifies in the case when $F\colon (W\rightrightarrows W_N)\to (V\rightrightarrows V_M)$ is a VBG morphism covering a Lie groupoid isomorphism $f\colon (H\rightrightarrows N)\to (G\rightrightarrows M)$. Indeed, the pullback bundles $f^\ast V^\ast$ and $f^\ast C^\ast$ are VB isomorphic to $V^\ast$ and $C^\ast$ respectively, through the projection on the second component $\pr_2$. The latter projection is a VBG isomorphism covering $f$ (see Example \ref{ex:pullbackVBGmorphism}). Hence we can consider the inverse $(\pr_2^{-1}, f^{-1})\colon V^\ast \to f^\ast V^\ast$ that is a VBG isomorphism as well. The composition between $F^\ast\colon f^\ast V^\ast \to W^\ast$ in Proposition \ref{prop:dualVBG_morphism} and $\pr_2^{-1}$ is a VBG morphism from $V^\ast$ to $W^\ast$ covering $f^{-1}$, again denoted by $F^\ast$, and simply given by $F^\ast (\psi) = \psi\circ F\in W^\ast_{f^{-1}(g)}$, for all $\psi\in V^\ast_g$, with $g\in G$.
\end{rem}

\subsection{Trivial core VB-groupoids}
\label{sec:trivial_core}
Here we focus on VBGs with trivial core. We present some properties of this class of VBGs, noting that they are equivalent to representations of Lie groupoids (Definition \ref{def:representation}). Finally, we consider a specific subclass of trivial core VBGs: those that are also line bundles in a certain sense. In what follows we will usually denote by $(E\rightrightarrows E_M;G\rightrightarrows M)$ a trivial core VBG.

We begin by proving that the structure maps of a trivial core VBG are regular VB morphisms (Definition \ref{def:regular_VBMorphism}).
\begin{lemma}
	\label{lemma:structure_maps_trivial_coreVBG}
	If $(E\rightrightarrows E_M;G\rightrightarrows M)$ is a trivial core VBG, then every structure map of $E$ is a regular VB morphism.
\end{lemma}
\begin{proof}
	First of all, for any $x\in M$ the source map $s_x\colon E_x\to E_{M,x}$ is bijective, because its kernel is the fiber over $x$ of the core of $E$ which is trivial. For any $g\in G$, $s_g\colon E_g\to E_{M,s(g)}$ is surjective and its kernel is trivial because the right multiplication by the zero vector $0^E_g$ determines a bijection between $\ker s_{t(g)}$ and $\ker s_g$.
	
	For any $g\in G$, $t_g\colon E_g\to E_{M,t(g)}$ is surjective and its kernel is trivial because the inverse map $i\colon E\to E$ determines a bijection between $\ker s_{g^{-1}}$ and $\ker t_g$.
	
	For any $x\in M$, the unit $u_x\colon E_{M,x}\to E_x$ is just the inverse of $s_x$ (or $t_x$), then it is bijective.
	
	Finally, for any $(g,g')\in G^{(2)}$, we already used in the definition of the multiplication of the dual VBG (Example \ref{ex:dual_VBG}) that $m_{(g,g')} \colon E_g\times E_{g'}\to E_{gg'}$ is surjective. Indeed, if $\xi\in E_{gg'}$, then, from the surjectivity of $t\colon E_g\to E_{M,t(g)}$, we can take any $v\in E_g$ such that $t(v)=t(\xi)$ and $v'=v^{-1}\xi\in E_{g'}$ in such a way $vv'=\xi$. Moreover, as $t_g\colon E_g\to E_{M,t(g)}$ is bijective, then, for any $\xi\in E_{gg'}$ there exists a unique $v\in E_g$ such that $t(v)=t(\xi)$, so there exists a unique $v'\in E_{g'}$ such that $vv'=\xi$, i.e., $v'=v^{-1}\xi$.
\end{proof}

\begin{rem}
	\label{rem:face_maps_trivial_coreVBG}
	Let $(E\rightrightarrows E_M;G\rightrightarrows M)$ be a trivial core VBG. In Lemma \ref{lemma:structure_maps_trivial_coreVBG}, we proved that the structure maps of $E$ are regular VB morphisms. We can prove more, i.e., every face and degeneracy map in the nerve $E^{(\bullet)}\to G^{(\bullet)}$ of $E$ is a regular VB-morphism. Let $k\geq 2$. For any $(g_1, \dots, g_k)\in G^{(k)}$ the map
	\[
	\partial_0\colon E^{(k)}_{(g_1, \dots, g_k)}\to E^{(k-1)}_{(g_2, \dots, g_k)}, \quad (v_1,\dots, v_k)\mapsto (v_2, \dots, v_k)
	\] 
	is bijective. Indeed, by Lemma \ref{lemma:structure_maps_trivial_coreVBG}, $s_{g_1}\colon E_{g_1}\to E_{M,s(g_1)}$ is bijective, then, for any $$(v_2, \dots, v_k)\in E^{(k-1)}_{(g_2, \dots, g_k)},$$ there exists a unique $v_1\in E_{g_1}$ such that $s(v_1)=t(v_2)$, and so there exists a unique $k$-tuple of $k$-composable arrows $(v_1, v_2, \dots,v_k )\in E^{(k)}_{(g_1, \dots,g_k)}$, such that $\partial_0(v_1, \dots, v_k)= (v_2, \dots, v_k)$.
	
	For any $(g_1, \dots, g_k)\in G^{(k)}$ the map
	\[
	\partial_k\colon E^{(k)}_{(g_1, \dots, g_k)}\to E^{(k-1)}_{(g_1, \dots, g_{k-1})}, \quad (v_1,\dots, v_k)\mapsto (v_1, \dots, v_{k-1})
	\]
	is bijective as well. Indeed, by Lemma \ref{lemma:structure_maps_trivial_coreVBG}, $t_{g_k}\colon E_{g_k}\to E_{M,t(g_k)}$ is bijective, then, for any $$(v_1, \dots, v_{k-1})\in E^{(k-1)}_{(g_1, \dots, g_{k-1})},$$ there exists a unique $v_k\in E_{g_k}$ such that $t(v_k)=s(v_{k-1})$, and so there exists a unique $k$-tuple of $k$-composable arrows $(v_1, \dots,v_{k-1},v_k )\in E^{(k)}_{(g_1, \dots,g_k)}$, such that $\partial_k(v_1, \dots, v_k)= (v_1, \dots, v_{k-1})$.
	
	For any $0<i<k$ and $(g_1, \dots, g_k)\in G^{(k)}$, the map
	\[
	\partial_i\colon E^{(k)}_{(g_1, \dots, g_k)}\to V^{(k-1)}_{(g_1, \dots,g_ig_{i+1}, \dots, g_k)}, \quad (v_1,\dots, v_k)\mapsto (v_1, \dots,v_iv_{i+1}, \dots, v_k)
	\]
	is bijective. Indeed, by Lemma \ref{lemma:structure_maps_trivial_coreVBG}, $m_{(g_i, g_{i+1})}\colon E^{(2)}_{(g_i, g_{i+1})}\to E_{g_ig_{i+1}}$ is bijective, then, for any $$(v_1, \dots, v_{i-1}, \xi, v_{i+2}, \dots , v_k)\in E^{(k-1)}_{(g_1, \dots, g_ig_{i+1}, \dots, g_k)},$$ there exists a unique $(v_i, v_{i+1})\in E^{(2)}_{(g_i, g_{i+1})}$ such that $v_iv_{i+1}=\xi$, and so there exists a unique $(v_1, \dots,v_i, v_{i+1}, \dots,v_k )\in E^{(k)}_{(g_1, \dots,g_k)}$, such that $\partial_i(v_1, \dots, v_k)= (v_1, \dots,v_{i-1}, \xi, v_{i+2}, \dots, v_k)$.
	
	Finally, for any $0\leq i\leq k$ and $(g_1, \dots, g_k)\in G^{(k)}$, the degeneracy map
	\[
	d_i\colon E^{(k)}_{(g_1, \dots, g_k)}\to E^{(k+1)}_{(g_1, \dots, g_i, s(g_i), g_{i+1}, \dots, g_k)},
	\]
	given by
	\[
		 d_i(v_1, \dots, v_k)=(v_1, \dots, v_i, s(v_i), v_{i+1}, \dots , v_k),
	\]
	is the right inverse of $\partial_i\colon E^{(k+1)}_{(g_1, \dots, g_i, s(g_i), g_{i+1}, \dots, g_k)}\to E^{(k)}_{(g_1, \dots, g_k)} $ (or of  $\partial_{i+1}$), then it is bijective.
\end{rem}

\begin{rem}
	\label{rem:trvial_core_VBG-representation}
	We recall from \cite[Proposition 3.3.5]{dH13} (or from \cite[Example 3.2]{ETV19}) that there is a bijection between trivial core VBGs over $G\rightrightarrows M$ and representations of $G$ (see Definition \ref{def:representation}). If $E_M$ is a representation of $G\rightrightarrows M$ then the associated action groupoid $G\ltimes E_M\rightrightarrows E_M$ (see Example \ref{ex:action_groupoid2}) is a trivial core VBG over $G\rightrightarrows M$.
	
	On the other hand, if $E\rightrightarrows E_M$ is a trivial core VBG over $G\rightrightarrows M$, then its structure maps are regular VB morphisms and we get the following VB isomorphism
	\begin{equation*}
		G\mathbin{{}_{s}\times_{}} E_M \to E, \quad (g,v)\mapsto s_g^{-1}(v).
	\end{equation*}
	Hence $E$ is isomorphic to the pullback of $E_M$ along the source $s\colon G\to M$ and, under this identification, the target map $t\colon E\to E_M$ determines a representation of $G$ on $E_M$.
	
	There is another description of representations through VBGs. Indeed, there is a bijection between representations of $G$ and VBGs over $G$ with trivial side bundle. Namely, if $E_M$ is a representation of $G\rightrightarrows M$, then the pullback $G\mathbin{{}_{t}\times_{}} E_M$ of $E_M$ along the target $t\colon G\to M$ is a Lie groupoid over the trivial bundle $0_M$, where the source, target and unit maps are trivial, the multiplication of $(g_1,v_1)$ and $(g_2,v_2)$ is $(g_1g_2,v_1 + g_1.v_2)$ and the inverse of $(g,v)$ is $(g^{-1}, -g^{-1}.v)$. Actually, $G\mathbin{{}_{t}\times_{}} E_M \rightrightarrows 0_M$ is a VBG over $G$ with core $E_M$. On the other hand, if $E\rightrightarrows 0_M$ is a VBG over $G$, then its core $C$ is a representation of $G$ where the action of $G$ on $C$ is defined by
	\begin{equation*}
		g.c=0_g \cdot c \cdot 0_{g^{-1}}, \quad \text{for all } c\in C_{s(g)}.
	\end{equation*}
	In particular we get that there is a bijection between trivial core VBGs and VBGs with trivial side bundle, so they contain the same information, i.e., a representation of $G$. In the following we will monstly use trivial core VBGs. 
\end{rem}

In Remark \ref{rem:trvial_core_VBG-representation} we recalled that any representation can be seen as a trivial core VBG. An example is given by the trivial core VBG coming the normal representation (see Example \ref{ex:normal_representation}). 
\begin{example}[Normal representation]
	\label{ex:VBG_normal_representation}
	Let $G\rightrightarrows M$ be a Lie groupoid and let $O$ be the orbit through $x\in M$. As already discussed in Example \ref{ex:normal_representation}, the restriction $G_O\rightrightarrows O$, with $G_O=s^{-1}(O)$ is a well-defined Lie groupoid. Following \cite[Section 3.4]{dH13} we can consider the VBG $NG_O\rightrightarrows NO$ over $G_O$, where $NG_O=TG|_{G_O}/TG_O$ is the normal bundle over $G_O$ and $NO=TM|_O/TO$ is the normal bundle over $O$, and, as $G_O\rightrightarrows O$ is well-defined, then the structure maps of $TG$ descend to $NG_O$. The submanifolds $G_O\subseteq G$ and $O\subseteq M$ have the same codimension, namely
	\[
	\dim G - \dim G_O= \dim G -(\dim G - \dim M + \dim O)= \dim M -\dim O.
	\]
	Hence, the structure maps of $NG_O$ are regular VB morphims and $NG_O$ has trivial core. The representation of $G_O$ on $NO$ coming from $NG_O$ is the following: if $g\colon x \to y$ in $G_O$ and $[v]\in N_xO$, with $v\in T_xM$, since $NG_O$ is trivial core, there exists a unique $[u]\in N_gG_O$ such that 
	\[
	ds([u])= [ds(u)] =[v], \quad \text{with } u\in T_gG,
	\]	
	and, $g.[v]=[dt(u)]$. Hence, this representation is just the normal representation of $G_O$ (see Example \ref{ex:normal_representation}).		
\end{example}

VBG morphisms between trivial core VBGs have the following property.
\begin{lemma}
	\label{lemma:VBG_morphism_regular}
	Let $(f_E,f)\colon (E'\rightrightarrows E'_N;H\rightrightarrows N)\to (E\rightrightarrows E_M; G\rightrightarrows M)$ be a VBG morphism between trivial core VBGs $E'$ and $E$. Then $(f_E,f)\colon (E'\to H)\to (E\to G)$ is a regular VB morphism if and only if $(f_E,f)\colon (E'_N\to N)\to (E_M\to M)$ is so.
\end{lemma}
\begin{proof}
	For any $h\in H$, we have the following commutative diagram
	\[
	\begin{tikzcd}
		E'_h \arrow[r, "s_h"] \arrow[d, "F_h"'] & E'_{N,s(h)} \arrow[d, "F_{s(h)}"]\\
		E_{f(h)}\arrow[r, "s_{f(h)}"'] & E_{M,f(s(h))}
	\end{tikzcd},
	\]
	where the horizontal arrows, from Lemma \ref{lemma:structure_maps_trivial_coreVBG}, are isomorphisms. Then $F_h$ is an isomorphism if and only if $F_{s(h)}$ is so, whence the claim. 
\end{proof}

We are mainly interested in VBGs that are also line bundles. Then we introduce the following
\begin{definition}
	\label{def:LBG}
	A \emph{line bundle-groupoid (LBG)} $(L\rightrightarrows L_M; G\rightrightarrows M)$ is a VBG 
	\begin{equation*}
		\begin{tikzcd}
			L \arrow[r,shift left=0.5ex] \arrow[r, shift right=0.5 ex] \arrow[d] &L_M \arrow[d] \\
			G \arrow[r,shift left=0.5ex] \arrow[r, shift right=0.5 ex] & M
		\end{tikzcd}
	\end{equation*}
	where $L\to G$ and $L_M\to M$ are line bundles. An \emph{LBG morphism} is a VBG morphism $(F,f)\colon (L'\rightrightarrows L'_{N}; H\rightrightarrows N)\to (L\rightrightarrows L_M, G\rightrightarrows M)$ between LBGs such that $F\colon L' \to L$ and $F\colon L'_N \to L_M$ are LB morphisms (i.e.~regular VB morphisms between line bundles).
\end{definition}

Notice that, given an LBG $(L\rightrightarrows L_M; G\rightrightarrows M)$, for any $x\in M$, the restriction of the source map $s_x\colon L_x \to L_{M,x}$ is a surjective linear map between $1$-dimensional vector spaces, so its kernel is trivial. Hence the core of an LBG $(L\rightrightarrows L_M; G \rightrightarrows M)$ is automatically trivial, and, by Remark \ref{rem:face_maps_trivial_coreVBG}, all the simplicial structure maps of the nerve $L^{(\bullet)}\to G^{(\bullet)}$ are LB morphisms. Since the core of $L$ is trivial, then, by Remark \ref{rem:trvial_core_VBG-representation}, $G$ acts on $L_M$ and $L \rightrightarrows L_M$ is isomorphic to the action groupoid $G \ltimes L_M \rightrightarrows L_M$ (Example \ref{ex:action_groupoid2}). We stress that the $G$-action on $L_M$ is given by
\begin{equation}
	\label{eq:G-action}
	G\mathbin{{}_{s}\times_{}} L_M \to L_M, \quad (g,v)\mapsto g.v =t(s_g^{-1}(v)),
\end{equation}
and the isomorphism is given by
\begin{equation*}
	G\mathbin{{}_{s}\times_{}} L_M \to L, \quad (g,v)\mapsto s_g^{-1}(v).
\end{equation*}
We prefer not to take the action groupoid point of view, despite the fact that we will occasionally use the $G$-action on $L_M$. 

Again By Remark \ref{rem:trvial_core_VBG-representation}, since an LBG $(L\rightrightarrows L_M; G\rightrightarrows M)$ is just a representation of $G\rightrightarrows M$ on the line bundle $L_M\to M$, then $L$ contains the same information of a VBG of the type
\begin{equation*}
	\begin{tikzcd}
		L \arrow[r,shift left=0.5ex] \arrow[r, shift right=0.5 ex] \arrow[d] &0_M \arrow[d] \\
		G \arrow[r,shift left=0.5ex] \arrow[r, shift right=0.5 ex] & M
	\end{tikzcd}
\end{equation*}
with core $L_M$.

We conclude this subsection with a remark on the homotopy fiber product of two LBG morphisms.
\begin{rem}
	\label{rem:hom_LBG}
	Let $L$, $L'$ and $L''$ be LBGs over $G$, $G'$ and $G''$ respectively, and let $(F,f)\colon (L'\to G')\to (L\to G)$ and $(K,k)\colon (L''\to G'')\to (L\to G)$ be two LBG morphisms. The homotopy fiber product $L'\times_L^h L''$ of $F$ and $K$, when exists, is a VBG over $G'\times_G^h G''$ (see Example \ref{ex:homotopy_pullback_VBG}). Actually, $L'\times_L^h L''$ is an LBG over $G'\times_G^h G''$. Indeed, the rank of the side bundle of $L'\times_L^h L''$ is
	\begin{align*}
		\rank (L'_{M'}\mathbin{{}_{F}\times_{t}} L \mathbin{{}_{s}\times_{K}} L''_{M''})&= \dim (L'_{M'}\mathbin{{}_{F}\times_{t}} L \mathbin{{}_{s}\times_{K}} L''_{M''}) - \dim (M'\mathbin{{}_{f}\times_{t}} G \mathbin{{}_{s}\times_{k}} M'')\\
		&= \dim L'_{M'} + \dim L''_{M''} +\dim L - 2\dim L_M\\
		& \quad  - \dim M' -\dim M'' - \dim G + 2\dim M\\
		&= \rank L'_{M'} + \rank L''_{M''} + \rank L -2\rank L_M= 1.
	\end{align*}
	On the other hand, the total bundle of $L'\times_L^h L''$ is a subbundle of $L'\mathbin{{}_{s\circ F}\times_{t}} L \mathbin{{}_{s}\times_{t\circ K}} L''$. But the rank of the latter is
	\begin{align*}
		\rank (L'\mathbin{{}_{s\circ F}\times_{t}} &L \mathbin{{}_{s}\times_{t\circ K}} L'')\\
		&= \dim (L'\mathbin{{}_{s\circ F}\times_{t}} L \mathbin{{}_{s}\times_{t\cir K}} L'') - \dim (G'\mathbin{{}_{s\circ f}\times_{t}} G \mathbin{{}_{s}\times_{t\circ k}} G'')\\
		&= \dim L' + \dim L'' +\dim L - 2\dim L_M\\
		& \quad  - \dim G' -\dim G'' - \dim G + 2\dim M\\
		&= \rank L' + \rank L'' + \rank L -2\rank L_M= 1.
	\end{align*}
	Since $L'\times_L^h L''$ is a Lie groupoid, then the total bundle agrees with the line bundle $L'\mathbin{{}_{s\circ F}\times_{t}} L \mathbin{{}_{s}\times_{t\circ K}} L''$.
\end{rem}

\subsection{Twisted dual VB-groupoid}\label{sec:twisted_dual}
Dual VBs twisted by a line bundle play in Contact Geometry a role analogous to that of dual VBs in Symplectic Geometry. Motivated by this parallel, in this subsection, we introduce the construction of the dual VBG twisted by a trivial core one. To achieve this, we first define the tensor product between a VBG and a trivial core, and then we combine this construction with that of the dual VBG. Finally, we discuss how to relate twisted dual VBGs of two VBGs connected by a VBG morphism.

In general the tensor product of two VBGs is no longer a VBG but if one of this is trivial core then the tensor product is a VBG as well. We introduce the tensor product between a VBG and a trivial core VBG over the same Lie groupoid in the following 
\begin{example}[Tensor product VBG]
	\label{ex:tensor_product_VBG}
	Let $(V\rightrightarrows V_M; G\rightrightarrows M)$ be a VBG with core $C$ and let $(E\rightrightarrows E_M; G\rightrightarrows M)$ be a trivial core VBG. The \emph{tensor product VBG} between $V$ and $E$ is the VBG $V\otimes E \rightrightarrows V_M\otimes E_M$ over $G$, whose structure maps are simply the tensor product of the structure maps of $V$ and the structure maps of $E$. The core of $V\otimes E$ is the kernel of
	\[
	s\otimes s\colon V|_M\otimes E|_M \to V_M\otimes E_M,
	\]
	but, as the core of $E$ is trivial, $s\colon E|_M\to E_M$ is a VB isomorphism covering $\operatorname{id}_M$. Then the core of $V\otimes E$ is simply given by
	\[
	\ker (s\colon V|_M\to V_M)\otimes E_M=C\otimes E_M,
	\]
	and the core-anchor of $V\otimes E$, that is the restriction of the target map $t\otimes t\colon V\otimes E\to V_M\otimes E_M$ to the core, is given by
	\begin{equation}
		\label{eq:core_anchor_tensor}
		t|_C\otimes \operatorname{id}_{E_M}\colon C\otimes E_M\to V_M\otimes E_M.
	\end{equation}
	
	Finally, the short exact sequence \eqref{eq:core_SES} associated to $V\otimes E$ is
	\[
		\begin{tikzcd}
			0 \arrow[r] & t^\ast(C\otimes E_M)\arrow[r] & V\otimes E\arrow[r] & s^\ast(V_M\otimes E_M) \arrow[r] & 0,
		\end{tikzcd}
	\]
	and a right-horizontal lift of $V\otimes E$ is simply given by one of $V$. Indeed, if $h\colon s^\ast V_M\to V$ is a right-horizontal lift of $V$, then $h\colon s^\ast(V_M\otimes E_M)\to V\otimes E$, defined by
	\begin{equation}
		\label{eq:tensor_splitting}
		h_g(v\otimes e)= h_gv\otimes s_g^{-1}(e),
	\end{equation}
	for all $v\in V_{M,s(g)}$ and $e\in E_{M,s(g)}$, with $g\in G$, is a right-horizontal lift of $V\otimes E$. And, since $s\colon E\to E_M$ is a fiberwise isomorphism, every right-horizontal lift of $V\otimes E$ comes from one of $V$.
\end{example}

Combining the construction of the dual VBG (Example \ref{ex:dual_VBG}) with the one of the tensor product (Example \ref{ex:tensor_product_VBG}) we get the following
\begin{example}[Twisted dual VBG]
	\label{ex:twisted_dual}
	Let $(V\rightrightarrows V_M; G\rightrightarrows M)$ be a VBG with core $C$ and let $(E\rightrightarrows E_M; G\rightrightarrows M)$ be a trivial core VBG. Applying the construction of the dual VBG (see Example \ref{ex:dual_VBG}) and then that of the tensor product VBG (see Example \ref{ex:tensor_product_VBG}), we get the VBG $(V^{\dag}\rightrightarrows C^{\dag}; G\rightrightarrows M)$, where $V^{\dag}:=V^\ast \otimes E$ and $C^{\dag}:= C^\ast \otimes E_M $, that we call the \emph{$E$-twisted dual VBG} of $V$ (or simply the \emph{twisted dual VBG}, if there is no risk of confusion). Under the isomorphisms $V^\dag \cong \operatorname{Hom}(V, E)$ and $C^\dag\cong \operatorname{Hom}(C, E_M)$, the structure maps of $V^\dag$ are the following:
	\begin{itemize}
		\item for any $\psi\in V^{\dag}_g$, with $g\in G$, the source and target are defined by
		\begin{align*}
			\big\langle s(\psi), c \big\rangle = - s\big\langle \psi , 0_g \cdot c^{-1}\big\rangle, \quad \text{and} \quad \big\langle t(\psi), c' \big\rangle = t \big\langle \psi, c' \cdot 0_g \big\rangle,
		\end{align*}
		for all $c\in C_{s(g)}$ and $c'\in C_{t(g)}$, where we are using $\langle -,-\rangle$ to denote the $E$-valued duality pairing between $V$ and $V^{\dag}$, i.e., $\langle -,-\rangle\colon V^{\dag}\times_G V\to E$, $\langle\phi,v\rangle := \phi(v)$, as well as the $E_M$-valued duality pairing between $C$ and $C^{\dag}$;
		\item for any $\psi\in C^{\dag}_x$, with $x\in M$, the unit maps $\psi$ to $\psi \circ \pr_C \in \operatorname{Hom}(V|_M, E_M) \subseteq V^{\dag}$, where the projection $\pr_C\colon V|_M \to C$ is again given by the canonical splitting \eqref{eq:canonical splitting};
		\item for any two composable arrows $\psi\in V_g^{\dag}$ and $\psi'\in V_{g'}^{\dag}$ the multiplication is defined by
		\begin{equation*}
			\langle \psi  \psi', v v'\rangle =s_{gg'}^{-1}\left(g'^{-1} . s\langle \psi, v \rangle + s\langle \psi', v'\rangle \right)
		\end{equation*}
		for all composable arrows $v\in V_g$ and $v'\in V_{g'}$, where the dot ``$.$'' indicates the $G$-action on $E_M$;
		\item finally, for any $\psi\in V_g^{\dag}$ the inverse is defined by
		\begin{equation*}
			\langle \psi^{-1}, v\rangle = - s_{g^{-1}}^{-1}(t\langle\psi, v^{-1} \rangle), \quad v\in V_{g^{-1}}.
		\end{equation*}
	\end{itemize}
	
	By Example \ref{ex:tensor_product_VBG}, the core of the twisted dual VBG $V^\dag$ is the tensor product between the core of $V^\ast$ and $E_M$. Since the core of $V^\ast$, from Example \ref{ex:dual_VBG}, is $V_M^\ast$, then the core of $V^\dag$ is $V_M^\dag=V_M^\ast \otimes E_M$ and the core anchor is $t|_C^\ast \otimes E_M\colon V_M^\dag\to C^\dag$. Under the isomorphisms $V_M^{\dag}\cong \operatorname{Hom}(V_M, E_M)$ and $C^\dagger \cong \operatorname{Hom}(C,E_M)$, the core complex is
	\begin{equation*}
		\begin{tikzcd}
			0 \arrow[r] & V_M^{\dag}\arrow[r, "t|_C^{\dag}"] & C^{\dag}\arrow[r]&0,
		\end{tikzcd}
	\end{equation*}
	where $t|_C^{\dag}$ is the \emph{twisted transpose map} to $t|_C$, i.e., $t|_C^{\dag}(\psi)= \psi \circ t|_C$ for all $\psi\in V_M^{\dag}$. Notice that the core complex of $V^\dag$ is just the dual of the core complex of $V$ twisted by $E_M$ and shifted by $+1$.

	Finally, from the last parts of the Examples \ref{ex:tensor_product_VBG} and \ref{ex:dual_VBG}, it follows that any right-horizontal lift of $V^\dagger$ comes from one of $V$. Namely, the short exact sequence \eqref{eq:core_SES} associated to $V^\dagger$ is
	\[
	\begin{tikzcd}
		0 \arrow[r] & t^\ast V_M^\dagger\arrow[r] & V^\dagger\arrow[r] & s^\ast C^\dagger \arrow[r] & 0,
	\end{tikzcd}
	\]
	and, if $h\colon s^\ast V_M \to V$ is a right-horizontal lift of $V$, then we get a right-horizontal lift $h^\dagger\colon s^\ast C^\dagger \to V^\dagger$ of $V^\dagger$ by putting
	\begin{equation}
		\label{eq:splitting_twisted_dual}
		\langle h^\dagger_g(\psi), v\rangle = s_g^{-1}\langle \psi, \varpi_{g^{-1}}(v^{-1})\rangle ,
	\end{equation}
	for all $\psi\in C^\dagger_{s(g)}$ and $v\in V_g$, with $g\in G$, where $\varpi= h\circ ds -\operatorname{id}$ is the left splitting corresponding to $h$.
\end{example}

We now discuss VBG morphisms between twisted dual VBGs of VBGs connected by a VBG morphism. As in the case of dual VBGs, the approach involves passing to the pullback VBG. We start by examining the case of the tensor product VBG. First, we present the following straightforward proposition, which relates pullback VBGs and tensor product VBGs.
\begin{prop}
	\label{prop:pullback_tensor_commute}
	Let $(f_E,f)\colon (E'\rightrightarrows E'_N;H\rightrightarrows N)\to (E\rightrightarrows E_M;G\rightrightarrows M)$ be a VBG morphism between trivial core VBGs $E'$ and $E$, such that $f_E\colon (E'\to H)\to (E\to G)$ (and so, by Lemma \ref{lemma:VBG_morphism_regular}, $f_E\colon( E'_N\to N)\to (E_M\to M)$) is a regular VB morphism. The map
	\[
	F\colon f^\ast V\otimes E'\to f^\ast(V\otimes E), \quad (h,v,e')\mapsto (h,v,f_E(e')),
	\]
	for all $v\in V_h$, $e'\in E'_h$ and $h\in H$, is a VBG isomorphism from the tensor product VBG (see Example \ref{ex:tensor_product_VBG}) $f^\ast V\otimes E'$ between the pullback VBG (see Example \ref{ex:pullback_VBG}) $f^\ast V$ and $E'$ to the pullback VBG $f^\ast (V\otimes E)$ of the tensor product VBG $V\otimes E$.
\end{prop}
\begin{proof}
	As $f_E$ is regular, then $E'\to N$ is isomorphic to $f^\ast E\to N$ and the map in the statement is a VBG isomorphism.
\end{proof}

The next proposition is fairly straightforward; however, we prefer to provide a detailed proof for the sake of completeness.
\begin{prop}
	\label{prop:tensor_product_VBGmorphism}
	If $(F,f)\colon (W\rightrightarrows W_N;H\rightrightarrows N)\to (V\rightrightarrows V_M;G\rightrightarrows M)$ is a VBG morphism and $f_E\colon (E'\rightrightarrows E'_N)\to (E\rightrightarrows E_M)$ is a VBG morphism covering $f$ between trivial core VBGs $E'$ and $E$, then the map $F\otimes f_E\colon W\otimes E'\to V\otimes E$ is a VBG morphism covering $f$. 
\end{prop}
\begin{proof}
	The proof easily follows from the definition of the structure maps of the tensor VBG (see Example \ref{ex:tensor_product_VBG}) and from $F$ and $f_E$ being VBG morphism. We discuss it in detail for completeness. Let $h\in H$. For any $w\in W_h$ and $e'\in E'_h$, we have
	\begin{align*}
		s(F\otimes f_E(w\otimes e'))&= s(F(w)\otimes f_E(e'))= s(F(w))\otimes s(f_E(e')) \\ &= F(s(w))\otimes f_E(s(e'))
		= F\otimes f_E(s(w)\otimes s(e'))\\&= F\otimes f_E(s(w\otimes e')),
	\end{align*}
	and 
	\begin{align*}
		t(F\otimes f_E(w\otimes e'))&= t(F(w)\otimes f_E(e'))= t(F(w))\otimes t(f_E(e')) \\
		&= F(t(w))\otimes f_E(t(e'))= F\otimes f_E(t(w)\otimes t(e'))\\
		&= F\otimes f_E(t(w\otimes e')).
	\end{align*}
	Then $s\circ F\otimes f_E= F\otimes f_E\circ s$ and $t\circ F\otimes f_E= F\otimes f_E\circ t$. 
	
	For any composable arrows $w_1\otimes e'_1\in W_{h_1}\otimes E'_{h_1}$ and $w_2\otimes e'_2\in W_{h_2}\otimes E'_{h_2}$, with $(h_1,h_2)\in H^{(2)}$, we have
	\begin{align*}
		F\otimes f_E(w_1\otimes e_1')\cdot F\otimes f_E(w_2\otimes d_2')&= (F(w_1)\otimes f_E(e_1'))\cdot (F(w_2)\otimes f_E(e_2')) \\
		&= F(w_1)F(w_2)\otimes f_E(e_1')f_E(e'_2) \\
		&=F(w_1w_2)\otimes f_E(e_1'e_2') \\
		&= F\otimes f_E(w_1w_2\otimes e_1'e_2')\\
		&= F\otimes f_E((w_1\otimes e_1')\cdot (w_2\otimes e_2')),
	\end{align*}
	and $m \circ F\otimes f_E=F\otimes f_E\circ m$.
	
	For any $w\in W_{N,y}$ and $e'\in E'_{N,y}$, with $y\in N$, we have
	\begin{align*}
		u(F\otimes f_E(w\otimes e'))&=u(F(w)\otimes f_E(e'))= u(F(w))\otimes u(f_E(e')) \\
		&= F(u(w))\otimes f_E(u(e'))= F\otimes f_E(u(w)\otimes u(e'))\\
		&= F\otimes f_E(u(w\otimes e')),
	\end{align*}
	and $u\circ F\otimes f_E=F\otimes f_E \circ u$.
	The case of the inverse maps follows from the other cases.% but we prefer to prove it for completeness. For any $w\in W_h$ and $e'\in E'_h$, with $h\in H$, we have
%	\begin{align*}
%		(F\otimes f_E(w\otimes e'))^{-1}&= (F(w)\otimes f_E(e'))^{-1}= F(w)^{-1}\otimes f_E(e')^{-1} = F(w^{-1})\otimes f_E(e'^{-1})\\
%		&= F\otimes f_E(w^{-1}\otimes e'^{-1})= F\otimes f_E((w\otimes e')^{-1}).\qedhere
%	\end{align*} 
\end{proof}

By Propositions \ref{prop:pullback_tensor_commute} and \ref{prop:tensor_product_VBGmorphism} we get the following
\begin{coroll}
	\label{coroll:tensor_product_VBGmorphism}
	If $F\colon (W\rightrightarrows W_N;H\rightrightarrows N)\to (V\rightrightarrows V_M;G\rightrightarrows M)$ is a VBG morphism and $E'\rightrightarrows E'_N$ is a trivial core VBG over $H$, then the map $F^\ast\otimes \operatorname{id}_H\colon f^\ast V^\ast\otimes E'\to W^\ast\otimes E'$ is a VBG morphism, where $F^\ast \colon f^\ast V^\ast \to W^\ast$ is the VBG morphism defined in Proposition \ref{prop:dualVBG_morphism}.
\end{coroll}
\begin{proof}
	The statement is a consequence of Proposition \ref{prop:tensor_product_VBGmorphism} for the VBG morphism $F^\ast\colon f^\ast V^\ast\to W^\ast$ (covering $\operatorname{id}_H$) defined in Proposition \ref{prop:dualVBG_morphism} and the VBG morphism $\operatorname{id}_{E'}\colon E'\to E'$ (covering $\operatorname{id}_H$).
\end{proof}

Now we are ready to consider the case of the twisted dual VBGs, aiming to obtain results analogous to Proposition \ref{prop:dualVBG_morphism} and Remark \ref{rem:dualF_on_iso}.
\begin{rem}
	\label{rem:twisted_dual_VBGmorphism}
	Let $(F,f)\colon (W\rightrightarrows W_N;H\rightrightarrows N)\to (V\rightrightarrows V_M;G\rightrightarrows M)$ be a VBG morphism and let $(f_E,f)\colon (E'\rightrightarrows E'_N;H\rightrightarrows N)\to (E\rightrightarrows E_M;G\rightrightarrows M)$ be a VBG morphism between trivial core VBGs $E'$ and $E$, such that $(f_E,f)\colon (E'\to H)\to (E\to G)$ (and so, by Lemma \ref{lemma:VBG_morphism_regular}, $(f_E,f)\colon (E'_N\to N)\to (E_M\to M)$) is a regular VB morphism. By Proposition \ref{prop:pullback_tensor_commute}, the VBG $f^\ast V^\ast\otimes E'$ is isomorphic to $f^\ast(V^\ast\otimes E)= f^\ast V^\dag$, then $F^\ast\otimes \operatorname{id}_{E'}$ in Corollary \ref{coroll:tensor_product_VBGmorphism} is a VBG morphism from $f^\ast V^\dag$ to $W^\dag=W^\ast\otimes E$. Under the isomorphisms $V^\dag\cong\operatorname{Hom}(V,E)$ and $W^\dag\cong\operatorname{Hom}(W,E')$ the VBG morphism $F\otimes \operatorname{id}_{E'}$ becomes $F^\dag\colon f^\ast V^\dag\to W^\dag$ and it is simply given by
	\[
	\left\langle F^\dag(h,\psi), w\right\rangle = f_{E,h}^{-1}\big\langle \psi, F(w)\big\rangle, \quad \psi\in V^\dag_{f(h)}, w\in W_h, \, h\in H.\qedhere
	\]
\end{rem}

\begin{rem}
	\label{rem:twisted_dualF_on_iso}
	Just as in the case of the dual VBG (see Remark \ref{rem:dualF_on_iso}), the VBG morphism $F^\dag\colon f^\ast V^\dag\to W^\dag$, discussed in Remark \ref{rem:twisted_dual_VBGmorphism}, also simplifies if $f$ is a Lie groupoid isomorphism. Indeed, the pullback bundles $f^\ast V^\dag$ and $f^\ast C^\dag$ are VB isomorphic to $V^\dag$ and $C^\dag$ respectively, through the projection on the second component $\pr_2$. The latter projection is a VBG morphism covering $f$ (see Example \ref{ex:pullbackVBGmorphism}). Hence we can consider the inverse $(\pr_2^{-1}, f^{-1})\colon V^\dag \to f^\ast V^\dag$ that is a VBG isomorphism as well. The composition between the VBG morphism $F^\dag\colon f^\ast V^\dag \to W^\dag$ in Remark \ref{rem:twisted_dual_VBGmorphism} and $\pr_2^{-1}$ is a VBG morphism from $V^\dag$ to $W^\dag$ covering $f^{-1}$, again denoted by $F^\dag$, and simply given by
	\[
	\left\langle F^\dag (\psi) , w\right\rangle= f_{E, f^{-1}(g)}^{-1}\big\langle \psi, F(w)\big\rangle ,
	\]
	for all $\psi\in V^\ast_g$ and $w\in W_{f^{-1}(g)}$, with $g\in G$. Notice that, a regular VB morphism covering a diffeomorphism is a VB isomorphism, then we can consider the VBG isomorphim $(F^{-1}, f^{-1})\colon E\to E'$. Hence, $F^\dag\colon V^\dagger \to W^\dagger$ can be also obtained by applying Proposition \ref{prop:tensor_product_VBGmorphism} to the VBG $F^\ast \colon V^\ast \to W^\ast$ defined in Remark \ref{rem:dualF_on_iso} and the VBG isomorphism $(F^{-1}, f^{-1})\colon E\to E'$.
\end{rem}

\subsection{The Atiyah VB-groupoid}
\label{sec:Atiyah_VBG}
In this subsection we apply the functor $D$ (Remark \ref{rem:regular_VB_morphisms}) to an LBG (Definition \ref{def:LBG}), resulting in another VBG. Finally, we use the twisted dual construction (Example \ref{ex:twisted_dual}) to this VBG and an LBG to obtain yet another VBG, which plays a role in Contact Geometry analogous  to that of the cotangent VBG (Example \ref{ex:dual_VBG}).

The functor $D$ can be applied to any trivial core VBG. However, we focus on an LBG, as this is the case of interest to us. Let $(L\rightrightarrows L_M; G\rightrightarrows M)$ be an LBG. Since every structure map is a regular VB morphism, applying the functor $D$ to get another VBG, that we call the \emph{Atiyah VBG},
\begin{equation*}
	\begin{tikzcd}
		DL \arrow[r,shift left=0.5ex] \arrow[r, shift right=0.5 ex] \arrow[d] &DL_M \arrow[d] \\
		G \arrow[r,shift left=0.5ex] \arrow[r, shift right=0.5 ex] & M
	\end{tikzcd}.
\end{equation*}
Here $DL$ and $DL_M$ are the Atiyah algebroid of $L\to G$ and $L_M\to M$, respectively (see Section \ref{sec:Atiyah_algebroid}), and $DL$ is a Lie groupoid over $DL_M$ with structure maps given by $Ds$, $Dt$, $Di$, $Du$, and $Dm$. As for the tangent VBG we used the identification $(DL)^{(2)}\cong DL^{(2)}$ where $L^{(\bullet)}\to G^{(\bullet)}$ is the nerve of $L$. Moreover, by Remark \ref{rem:regular_VB_morphisms}, every vector bundle $L^{(\bullet)}$ in the nerve of $L$ is a line bundle over $G^{(\bullet)}$ and the nerve of $DL$ is simply given by $DL^{(\bullet)}\to G^{(\bullet)}$. Namely, for any $k$, we have a diagram os short exact sequences
\begin{equation*}
	\begin{tikzcd}
		0 \arrow[r] & \mathbbm{R}_{G^{(k)}} \arrow[r] \arrow[d, equal] & DL^{(k)} \arrow[r, "\sigma"] \arrow[d, "D\pr_1\times\dots\times D\pr_k"] & TG^{(k)} \arrow[r] \arrow[d, "d\pr_1\times \dots \times d\pr_k"]& 0 \\
		0 \arrow[r] & \mathbbm{R}_{G^{(k)}} \arrow[r] & (DL)^{(k)} \arrow[r, "\sigma\times \dots \times \sigma"'] & (TG)^{(k)} \arrow[r]& 0
	\end{tikzcd}.
\end{equation*}
As the left and right vertical arrows are VB isomorphisms, then the central vertical arrow is a VB isomorphism as well.

Notice that, from Proposition \ref{prop:derivation_symbol}, the Atiyah algebroid $DL$ is isomorphic to $DL_M \mathbin{{}_{\sigma}\times_{ds}} TG$ and to $DL_M \mathbin{{}_{\sigma}\times_{dt}} TG$.

\begin{rem}
	Let $(L\rightrightarrows L_M; G\rightrightarrows M)$ be an LBG. From Equation \eqref{eq:DFcommutes} follows that the symbol map induce a VBG morphism covering the identity
	\begin{equation}
		\label{eq:symbol_morphism}
		{\scriptsize
			\begin{tikzcd}%[row sep=scriptsize, column sep=scriptsize]
				DL \arrow[rr, shift left=0.5ex] \arrow[rr, shift right=0.5ex] \arrow[dd] \arrow[dr, "\sigma"] & &DL_M \arrow[dd] \arrow[dr, "\sigma"] \\
				&  TG \arrow[rr, shift left= 0.5ex, crossing over] \arrow[rr, shift right =0.5ex, crossing over] & &TM \arrow[dd]\\
				G \arrow[rr, shift left=0.5ex] \arrow[rr, shift right=0.5ex] \arrow[dr, equal] & & M \arrow[dr, equal] \\ 
				&  G \arrow[from=uu, crossing over]\arrow[rr, shift left= 0.5ex] \arrow[rr, shift right =0.5ex] & &M
		\end{tikzcd}}
	\end{equation}
	between the Atiyah VBG $DL$ and the tangent VBG $TG$.
\end{rem}

\begin{prop}
	\label{prop:core_isomorphic_A}
	Let $(L\rightrightarrows L_M; G\rightrightarrows M)$ be an LBG. The VBG morphism \eqref{eq:symbol_morphism} induces a VB isomorphism between the core of the Atiyah VBG $DL$ and the core of the tangent VBG $TG$.
\end{prop}
\begin{proof}
	Let $C$ be the core of the Atiyah VBG $DL$. We recall from Example \ref{ex:tangent_VBG} that the core of the tangent VBG $TG$ is $A$, the Lie algeborid of $G$. As $\sigma$ is a VBG morphism, then the restriction of the symbol map $\sigma \colon DL \to TG$ to $C$ takes values in $A$. Let $\delta\in C_x$ be such that $\sigma(\delta)=0\in A_x$, with $x\in M$. %Then, for any $\lambda\in \Gamma(L_M)$, we have
%	\[
%	\delta(s^\ast \lambda)= s_x^{-1}(Ds(\delta)\lambda) =0.
%	\]
	As, by Proposition \ref{prop:derivation_symbol}, $DL\cong DL_M \mathbin{{}_{\sigma}\times_{ds}} TG$, and $Ds(\delta)=0$, we have that $\delta=0$ and $\sigma\colon C\to A$ is injective. The surjectivity follows from dimensional reasons. Indeed,
	\[
	\rank C = \rank DL - \rank DL_M = \rank TG +1 - \rank TM -1 = \rank A.
	\]
	Hence $\sigma \colon C\to A$ is a VB isomorphism.
\end{proof}
By Proposition \ref{prop:core_isomorphic_A}, the core of $DL$ is isomorphic to $A$, the Lie algebroid of $G$. For this reason, in what follows, we identify the core of $DL$ with $A$ and we denote the core complex of $DL$ by $\mathcal{D}\colon A\to DL_M$ and the image of an element $a\in A$ simply by $\mathcal{D}_a$.

%Similarly as in the manifold case, this VBG morphism fits in a short exact sequence of VBGs over $G\rightrightarrows M$:
%\begin{equation*}{\scriptsize
%	\begin{tikzcd}
	%		0 \arrow[rr] &&\mathbbm{R}_G \arrow[rr] \arrow[dr, shift left=0.5ex] \arrow[dr, shift right=0.5ex] &&DL \arrow[rr, "\sigma"] \arrow[dr, shift left=0.5ex] \arrow[dr, shift right=0.5ex] &&TG \arrow[rr] \arrow[dr, shift left=0.5ex] \arrow[dr, shift right=0.5ex] && 0 \\
	%		& 0 \arrow[rr] &&\mathbbm{R}_M \arrow[rr] &&DL_M \arrow[rr, "\sigma"] &&TM \arrow[rr] && 0
	%	\end{tikzcd}}
	%\end{equation*}
	%where $\mathbbm{R}_G \rightrightarrows \mathbbm{R}_M$ is the trivial LBG whose structure maps are fiberwise indentities. Now, let $\nabla$ a connection on $L_M$ and let $t^{\ast}\nabla$ be the pullback connection on $L$. The pair $(t^{\ast}\nabla, \nabla)$ is a right splitting of the sequence before and it induces the VBG isomorphism
	%\begin{equation*}{\scriptsize
	%	\begin{tikzcd}%[row sep=scriptsize, column sep=scriptsize]
%		DL \arrow[rr, shift left=0.5ex] \arrow[rr, shift right=0.5ex] \arrow[dd] \arrow[dr, "(\sigma {,}f_{t^{\ast}\nabla})"] & &DL_M \arrow[dd] \arrow[dr, "(\sigma{,} f_{\nabla})"] \\
%		&  TG\oplus \mathbbm{R}_G \arrow[rr, shift left= 0.5ex, crossing over] \arrow[rr, shift right =0.5ex, crossing over] & &TM\oplus \mathbbm{R}_M \arrow[dd]\\
%		G \arrow[rr, shift left=0.5ex] \arrow[rr, shift right=0.5ex] \arrow[dr, equal] & & M \arrow[dr, equal] \\ 
%		&  G \arrow[from=uu, crossing over]\arrow[rr, shift left= 0.5ex] \arrow[rr, shift right =0.5ex] & &M
%	\end{tikzcd}}
%\end{equation*}

\begin{rem}
We can explicitly write the inverse of the VB isomorphism $\sigma\colon C \to A$ in Proposition \ref{prop:core_isomorphic_A}, where $C$ is the core of the Atiyah VBG $DL$. Let $a\in A_x$, $x\in M$, then, $a=\tfrac{d}{d\varepsilon}|_{\varepsilon=0} \gamma(\varepsilon)$, where $\gamma$ is a curve in $s^{-1}(x)$ such that $\gamma(0)=x$. For any $\gamma(\varepsilon)\in G$ the map $s_{\gamma(\varepsilon)}\colon L_{\gamma(\varepsilon)}\to L_{M,x}$ is an isomorphism. Hence,
\begin{equation*}
\Upsilon(\varepsilon):= s_{\gamma(\varepsilon)}^{-1} \circ s_x\colon L_x\to L_{\gamma(\varepsilon)}
\end{equation*}
is a curve of isomorphisms such that $\Upsilon(0)= \operatorname{id}_{L_x}$. The derivation $\delta^a=\tfrac{d}{d\varepsilon}|_{\varepsilon=0}\Upsilon(\varepsilon)$ is in $C_x$. Indeed, from \eqref{eq:derivation}
\begin{equation*}
Ds(\delta^a)= Ds\left(\frac{d}{d\varepsilon}|_{\varepsilon=0}\Upsilon(\varepsilon)\right) = \frac{d}{d\varepsilon}|_{\varepsilon=0} \, s_{\gamma(\varepsilon)}\circ s_{\gamma(\varepsilon)}^{-1} \circ s_x \circ s_x^{-1}=  \frac{d}{d\varepsilon}|_{\varepsilon=0}\, \operatorname{id}_{L_x} =0.
\end{equation*}
The map $A\to C$, $a\mapsto \delta^a$, is the inverse of $\sigma\colon C\to A$. Indeed, as discussed in Section \ref{sec:Atiyah_algebroid}, the symbol of $\delta^a$ is $a$. On the other hand, if $\delta= \tfrac{d}{d\varepsilon}|_{\varepsilon=0} \Upsilon(\varepsilon)\colon L_x\to L_{\gamma(\varepsilon)}$ is in $C$, then $Ds(\delta)=0$, i.e., the curve of isomorphisms $s_{\gamma(\varepsilon)}\circ \Upsilon(\varepsilon)\circ s_x^{-1}\colon L_{M,x}\to L_{M,x}$ is constant. Hence $\Upsilon(\varepsilon)=s_{\gamma(\varepsilon)}^{-1} \circ s_x$ and $\delta=\delta^a$, with $a=\tfrac{d}{d\varepsilon}|_{\varepsilon=0}\gamma(\varepsilon)\in A$.

Under the identification $C\cong A$, the core-anchor $\mathcal{D}\colon A \to DL_M$ of $DL$ maps an element $a=\tfrac{d}{d\varepsilon}|_{\varepsilon=0}\gamma(\varepsilon)\in A_x$ to the derivation
\begin{equation*}
\mathcal{D}_a=Dt\left(\frac{d}{d\varepsilon}|_{\varepsilon=0} \, s_{\gamma(\varepsilon)}^{-1}\circ s_x\right)= \frac{d}{d\varepsilon}|_{\varepsilon=0} \, t_{\gamma(\varepsilon)}\circ s_{\gamma(\varepsilon)}^{-1} \circ s_x \circ t_x^{-1}= \frac{d}{d\varepsilon} |_{\varepsilon=0} \, t_{\gamma(\varepsilon)}\circ s_{\gamma(\varepsilon)}^{-1} ,
\end{equation*}
where, in the last step, we used that source and target agree on units.

Notice that $\mathcal{D}\colon A\to DL_M$ actually agrees with the infinitesimal action (see Remark \ref{rem:Lie_functor_actions}) of $A$ on $L_M$ corresponding to the $G$-action on $L_M$. Indeed, from Equation \eqref{eq:G-action},
the $G$-action on $L_M$ is given by the Lie groupoid morphism $G\to \operatorname{GL}(L_M)$ that maps an arrow $g\colon x \to y$ in $G$ to the linear map $t \circ s_g^{-1}\colon L_{M,x} \to L_{M,y}$.
Applying the Lie functor, we get the infinitesimal action $A\to DL_M$ that maps $a=\tfrac{d}{d\varepsilon} |_{\varepsilon=0} \, \gamma(\varepsilon)$ to $\tfrac{d}{d\varepsilon} |_{\varepsilon=0} \, t_{\gamma(\varepsilon)}\circ s_{\gamma(\varepsilon)}^{-1} = \mathcal{D}_a$. 
\end{rem}

We already discussed in Section \ref{sec:Atiyah_algebroid} that a connection $\nabla$ on $L_M$ is a left splitting of the short exact sequence \eqref{eq:Spencer_L}.
%\[
%	\begin{tikzcd}
%		0 \arrow[r] & \mathbbm{R}_M\arrow[r] & DL_M \arrow[r] & TM \arrow[r] & 0,
%	\end{tikzcd}
%\]
In what follows we will need the linear form $F_\nabla=f_{\nabla}\circ \mathcal{D}\colon A \to \mathbbm R_M$, where $f_\nabla \colon DL_M \to \mathbbm R_M$ is the right splitting of \eqref{eq:Spencer_L} corresponding to $\nabla$. The following diagram summarizes the situation
\begin{equation}
	\label{eq:fnabla}
	\begin{tikzcd}
		& & A \arrow[d, "\mathcal{D}"] \arrow[dl, " F_\nabla"'] \arrow[dr, "\rho"]\\
		0 \arrow[r] & \mathbbm{R}_M\arrow[r] & DL_M \arrow[r, "\sigma"] \arrow[l, bend left, "f_{\nabla}"] & TM \arrow[r] \arrow[l, "\nabla", bend left] & 0 
	\end{tikzcd}.
\end{equation}
On the other hand we can consider the pullback connections $s^\ast \nabla$ and $t^\ast \nabla$ on $L$ and their difference $\eta_\nabla:= s^\ast\nabla -t^\ast\nabla$ can be seen as a plain $1$-form on $G$. Indeed, for any $v\in T_gG$, with $g\in G$, $(s^\ast\nabla)_v$ and $(t^\ast \nabla)_v$ are two derivations in $D_gL$ with the same symbol $v$. Then the difference $\eta_\nabla(v)$ is in $\ker \sigma$ so it takes values in $\mathbbm{R}$.
 
In the next Lemma we relate the linear form $F_\nabla$ with the $1$-form $\eta_\nabla$.
\begin{lemma}
	\label{lemma:restriction_fnabla}
	The restriction of the $1$-form $\eta_\nabla$ to the Lie algebroid $A$ agrees with $F_\nabla$.
\end{lemma}
\begin{proof}	
	Let $\lambda_M\in \Gamma(L_M)$, $\lambda = s^{\ast}\lambda_M$ and $x\in M$. Then locally, around $x$, $s^{\ast}\lambda_M = ft^{\ast}\lambda_M$, for some function $f\in C^{\infty}(G)$ such that $f|_M=1$. Hence, for any $a\in A_x$, we have
	\begin{equation*}
		\eta_{\nabla}(a)\lambda_x = (s^{\ast}\nabla)_a\lambda - (t^{\ast}\nabla)_a\lambda= - (t^{\ast}\nabla)_a(ft^{\ast}\lambda_M)= - a(f)\lambda_{M,x} - \nabla_{\rho(a)} \lambda_M ,
	\end{equation*}	
	and
	\begin{equation*}
		F_\nabla(a)\lambda_x=F_\nabla(a)\lambda_{M,x}=\mathcal{D}_a\lambda_M - \nabla_{\rho(a)}\lambda_M.
	\end{equation*}
	Now let $a=\tfrac{d}{d\varepsilon}|_{\varepsilon=0}\, g(\varepsilon)$ be the velocity of a curve $g(\varepsilon)\colon x\to x(\varepsilon)$ in the $s$-fiber over $x$. Then
	\begin{align*}
		\mathcal{D}_a \lambda_M &= \frac{d}{d\varepsilon}|_{\varepsilon=0} \, g(\varepsilon)^{-1}. \lambda_{M,x(\varepsilon)}\\
		&= \frac{d}{d\varepsilon}|_{\varepsilon=0} \, t\left(s_{g(\varepsilon)^{-1}}^{-1}\left(\lambda_{M,s(g(\varepsilon)^{-1})}\right)\right)\\
		&=\frac{d}{d\varepsilon}|_{\varepsilon=0} \, t\left(\left(s^{\ast}\lambda_M\right)_{g(\varepsilon)^{-1}}\right)\\
		&= \frac{d}{d\varepsilon}|_{\varepsilon=0} \, t\left(f\left(g(\varepsilon)^{-1}\right)\left(t^{\ast}\lambda_M\right)_{g(\varepsilon)^{-1}}\right) \\
		&=\frac{d}{d\varepsilon}|_{\varepsilon=0} \, f\left(g(\varepsilon)^{-1}\right) \lambda_{M,x}\\
		&=di(a)\left(f\right)\lambda_{M,x}\\
		&= \left(\rho(a)(f)- a(f)t\right)\lambda_{M,x}\\
		&=-a(f)\lambda_{M,x},
	\end{align*}
	where, in the last step, we used that $f$ is constant on $M$. This concludes the proof.
\end{proof}

We can consider right-invariant and left-invariant sections (see Definition \ref{def:right_invariant_section}) of the Atiyah VBG $DL$ generated by sections of $A$ (see Remark \ref{rem:right_inv_section}). The right-invariant section of $DL$ generated by $a\in \Gamma(A)$ is the derivarion on $L$ that we denote by $\overrightarrow{\Delta^a} \in \Gamma(DL)$ whose value at $g\in G$ is 
\begin{equation}
	\label{eq:right_invariant_derivation}
\overrightarrow{\Delta^a}_g= a_{t(g)} \cdot 0^{DL}_g.
\end{equation}
We call $\overrightarrow{\Delta^a}$ the \emph{right-invariant derivation} on $DL$ generatad by $a$. Similarly, the left-invariant section of $DL$ generated by $a\in \Gamma(A)$ is the derivation on $L$ that we denote by $\overleftarrow{\Delta^a} \in \Gamma(DL)$ whose value at $g\in G$ is 
\begin{equation}
	\label{eq:left_invariant_derivation}
\overleftarrow{\Delta^a}_g= 0^{DL}_g \cdot a_{s(g)}^{-1}.
\end{equation}
We call $\overleftarrow{\Delta^a}$ the \emph{left-invariant derivation} on $DL$ generated by $a$.

\begin{rem}
	\label{rem:jet_VBG}
	Applying the construction of the twisted dual VBG (see Example \ref{ex:twisted_dual}) to the Atiyah VBG $DL\rightrightarrows DL_M$ and the LBG $L\rightrightarrows L_M$ we get a VBG
	\begin{equation*}
		\begin{tikzcd}
			J^1L \arrow[r,shift left=0.5ex] \arrow[r, shift right=0.5 ex] \arrow[d] &A^{\dag} \arrow[d] \\
			G \arrow[r,shift left=0.5ex] \arrow[r, shift right=0.5 ex] & M
		\end{tikzcd},
	\end{equation*}
	where $A^{\dagger}= \operatorname{Hom}(A,L_M)$ and $J^1L$ is the first jet bundle of $L$. We call the VBG $J^1L$ the \emph{jet VBG}. The role of $J^1L$ in Contact Geometry is the same as that of the cotangent VBG $T^\ast G\rightrightarrows A^{\ast}$ in Symplectic Geometry. As discussed in Example \ref{ex:twisted_dual}, the core of $J^1L$ is simply $J^1L_M$ and the core-anchor is $\mathcal{D}^{\dagger}\colon J^1L_M \to A^{\dagger}$, that maps an element $\psi\in J^1_xL_M=\operatorname{Hom}(D_xL_M,L_{M,x})$ to $\mathcal{D}^\dag(\psi)\in A^\dag_x$ defined by setting
	\[
		\langle \mathcal{D}^\dag(\psi),a\rangle=\psi(\mathcal{D}_a),
	\]
	for all $a\in A_x$.
\end{rem} 

\section{Representations up to homotopy}
\label{sec:RUTHs}
As highlighted in Remark \ref{rem:trvial_core_VBG-representation}, trivial core VBGs correspond to representations of Lie groupoids. Following \cite{AC13}, we recall a generalization of representations, namely, \emph{representation up to homotopy}. It was shown in \cite{GSM17} that VBGs correspond to $2$-term representations up to homotopy, a correspondence that extends to an equivalence of categories \cite[Theorem 2.7]{dHO20}. Subsequently, we introduce new operations on representations up to homotopy, such as the tensor product with a plain representation and the twisted dual operation, and we show the relation with the analogous operations on VBGs.

We start by recalling the definition of representations up to homotopy. Let $G\rightrightarrows M$ be a Lie groupoid and let 
\[
V = \bigoplus_{i \in \mathbbm Z} V^i \to M
\] 
a graded vector bundle over $M$. We assume that $V$ is bounded from both sides. As we did in Remark \ref{rem:section_representation}, we denote by $t \colon G^{(k)} \to M$ the composition of the projection $\mathrm{pr}_1 \colon G^{(k)} \to G$ onto the first factor followed by the target, so the pullback $t^\ast V^i $ is a VB over $G^{(k)}$. We consider
\[
C(G; V) = \bigoplus_{n \in \mathbbm Z} C(G; V)^n, \quad C(G; V)^n := \bigoplus_{k+i = n} \Gamma (t^\ast V^i \to G^{(k)}).
\]
The latter is a graded $C(G)$-module, where $C(G)$ is the differential algebra of $G$ descibed in Remark \ref{rem:dg_algebra}. Finally, a section $\lambda\in \Gamma(t^\ast V^i \to G^{(k)})$ is called \emph{normalized} if $\lambda_{(g_1, \dots, g_k)}=0$ for all $(g_1, \dots, g_k)\in G^{(k)}$ in the image of the degeneracy maps $d\colon G^{(k-1)}\to G^{(k)}$, or equivalently, whenever $g_j\in M$, for at least one $j$.

Now we are ready to recall the definition of representation up to homotopy from \cite[Definition 3.1]{AC13} (see also \cite[Definition 2.8]{GSM17}). Let $G\rightrightarrows M$ be a Lie groupoid.
\begin{definition}
A \emph{representation up to homotopy (RUTH)} of $G$ on a graded vector bundle $V$ over $M$ is a differential $\partial^V \colon C(G; V)^\bullet  \to C(G; V)^{\bullet +1}$ on the $C(G)^\bullet := C^\infty (G^{(\bullet)})$-module $C(G;V)$, preserving normalized cochains, and satisfying the following Leibniz identity:
\[
\partial^V( \eta f)=  \partial^V(\eta) f + (-1)^k \eta \partial(f),
\]
for all $\eta\in C(G;V)^k$ and $f\in C(G)$, where $\partial$ is the differential defined in Remark \ref{rem:dg_algebra}.

A RUTH morphism $\Phi\colon W\to V$ between two RUTHs $W$ and $V$ of $G$ is a $C(G)$-linear map 
\[
\Phi\colon C(G;W)^\bullet \to C(G,V)^\bullet,
\]
that commutes with the differentials.
\end{definition}

\begin{rem}
The differential $\partial^V$ gives to $C(G; V)$ the structure of a DG module over $C(G)$. Then a RUTH morphism is just a morphism of the DG modules determined by the RUTHs. 
\end{rem}

Similarly to the map $t \colon G^{(k)} \to M$, in the following, we will denote by $s \colon G^{(k)} \to M$ the composition of the projection $\pr_k\colon G^{(k)}\to G$ onto the last factor followed by the source. If $V$ is a graded vector bundle on $M$, then, for any $k\in \mathbbm{N}$, $s^\ast V$ and $t^\ast V$ are graded vector bundles over $G^{(k)}$, and, in particular, we can consider the vector bundle
\begin{equation}
\label{eq:RUTH_hom}
\operatorname{Hom}^{-k + 1}\big(s^\ast V, t^\ast V \big) \to G^{(k)}
\end{equation}	
of degree $-k + 1$ morphisms from $s^\ast V$ to $t^\ast V$. Similarly, if $W$ and $V$ are two graded vector bundles over $M$, then, for any $k\in \mathbbm{N}$ $s^\ast W$ and $t^\ast V$ are graded vector bundles over $G^{(k)}$ and, in particular, we can consider the vector bundle
\begin{equation}
\label{eq:RUTH_morphism_hom}
\operatorname{Hom}^{-k}\big(s^\ast W, t^\ast V \big) \to G^{(k)},
\end{equation}
of degree $-k$ morphisms from $s^\ast W$ to $t^\ast V$.

An useful characterization of RUTHs and RUTH morphisms is given by the following
\begin{prop}[{\cite[Proposition 3.2]{AC13}}]
\label{prop:RUTH_characterization}
Let $G\rightrightarrows M$ be a Lie groupoid and let $V$ be a graded vector bundle over $M$. A RUTH of $G$ on $V$ is equivalent to a sequence $\{R_k\}_{k \geq 0}$, where, for any $k\geq 0$, $R_k$ is a section of the vector bundle \eqref{eq:RUTH_hom}, $R_1(x)=\operatorname{id}$, with $x\in M$, $R_k(g_1, \dots, g_k)=0$, with $k>1$, whenever $g_j\in M$, for at least one $j$, and, finally, the $R_k$'s satisfy the following identities: for any $k \geq 0$ and any $(g_1, \ldots, g_k) \in G^{(k)}$, 
\begin{equation}
	\label{eq:struct_sect_RUTH}
	\begin{aligned}
		&\sum_{j=1}^{k-1}(-)^j R_{k-1}(g_1, \ldots, g_jg_{j+1}, \ldots, g_k)\\
		&\quad  = \sum_{j= 0}^k (-)^j R_j (g_1, \ldots, g_j) \circ R_{k-j} (g_{j+1}, \ldots, g_k).
	\end{aligned}
\end{equation}

Moreover, let $W$ and $V$ two RUTHs of $G$. A RUTH morphism $\Phi\colon W \to V$ is equivalent to a sequence $\{\Phi_k\}_{k \geq 0}$, where, for any $k\geq 0$, $\Phi_k$ is a section of the vector bundle \ref{eq:RUTH_morphism_hom}, $\Phi_k(g_1, \dots,g_k)=0$, with $k>0$, whenever $g_j\in M$, for at least one $j$, and, finally, the $\Phi_k$'s satisfy the following identities: for all $k\geq 0$ and all $(g_1, \ldots, g_k) \in G^{(k)}$,
\begin{equation}\label{eq:struct_sect_RUTH_mor}
\begin{aligned}
	& \sum_{i+j = k}(-)^j \Phi_j (g_1, \ldots, g_j) \circ R_i (g_{j+1}, \ldots, g_k) \\
	& = \sum_{i+j = k} R'_j (g_1, \ldots, g_j) \circ \Phi_i (g_{j+1}, \ldots, g_k) + \sum_{j = 1}^{k-1} (-)^j \Phi_{k-1}(g_1, \ldots, g_jg_{j+1}, \ldots, g_k),
\end{aligned}
\end{equation}
where $R_k, R_k'$ are the sections corresponding to $W$ and $V$ respectively.
\end{prop}

In light of Proposition \ref{prop:RUTH_characterization}, we give the following definition: let $W$ and $V$ two RUTHs of $G$ and $\Phi\colon W\to V$ a RUTH morphism.
\begin{definition}
Sections $R_k$'s of the vector bundle $\eqref{eq:RUTH_hom}$, coming from the RUTH $V$, will be called the \emph{structure operators} of the RUTH $V$, and sections $\Phi_k$'s of the vector bundle \eqref{eq:RUTH_morphism_hom}, coming from the RUTH morphism $\Phi$, will be called the \emph{components} of the RUTH morphism $\Phi$.
\end{definition}

In particular, by Proposition \ref{prop:RUTH_characterization}, the $0$-th structure operator, $R_0\colon V^\bullet \to V^{\bullet + 1}$, of a RUTH $V$ is a differential, hence it gives to $V$ the structure of a complex of vector bundles over $M$.
\begin{rem}
	\label{rem:quasi_actions_quis}
	When $k=1$, Equation \eqref{eq:struct_sect_RUTH} means that, for any $g\in G$, $R_1(g)$ is a cochain map. When $k=2$, Equation \eqref{eq:struct_sect_RUTH} says that, for any $(g_1,g_2)\in G^{(2)}$, $R_2(g_1,g_2)$ is a homotopy between the cochain maps $R_1(g_1)\circ R_1(g_2)$ and $R_1(g_1g_2)$. Since $R_1(x)= \operatorname{id}$, with $x\in M$, it follows, for any $g\in G$, there is a homotopy between $R_1(g)\circ R_1(g^{-1})$ and the identity. Hence, $R_1(g)$ is a quasi-isomorphism for all $g$. Finally, the $0$-th component, $\Phi_0 \colon (W, R_0) \to (V, R'_0)$, of a RUTH morphism $\Phi\colon W \to V$, is a cochain map.
\end{rem}
 
As we already mentioned, RUTHs are a generalization of representations. We discuss it in detail in the following
\begin{rem}
Let $E$ be a representation of the Lie groupoid $G\rightrightarrows M$. Then $E=E^0\to M$ is a RUTH of $G$ whose structure operators are: $R_k=0$ for all $k\neq 1$ and, for any $g\in G$, $R_1(g)\colon E_{s(g)} \to E_{t(g)}$ is simply given by the $G$-action on $E$. Indeed, for degree reasons, identities \eqref{eq:struct_sect_RUTH} are trivial except for the case $k=2$, when we have
\begin{equation}
\label{eq:RUTH_representation}
R_1(g_1g_2)= R_1(g_1)\circ R_1(g_2),
\end{equation}
for all $(g_1,g_2)\in G^{(2)}$. But, Equation \eqref{eq:RUTH_representation} is just the second condition in the definition of representation (see Definition \ref{def:action_groupoid} and Definition \ref{def:representation}), then it is satisfied.

Conversly, if $E$ is a RUTH of $G$ concentrated in degree $0$, then $E=E^0$ is just a non-graded VB over $M$ and, for degree reason, the only non-trivial structure operator is $R_1\in \Gamma(\operatorname{Hom}(s^\ast E^0, t^\ast E^0))$, so, for any $g\in G$, $R_1(g)\colon E^0_{s(g)}\to E^0_{t(g)}$ is a linear map. Moreover, again, identities \eqref{eq:struct_sect_RUTH} are trivial except for the case $k=2$, when we have Equation \eqref{eq:RUTH_representation}. Finally, since  $R_1(x)=\operatorname{id}_{E_x}$, $E$ is a representation of $G$.
\end{rem}

Now, following \cite[Section 3.2]{AC13}, we provide an Example of RUTH that has the role of \emph{adjoint RUTH}.
\begin{example}[The adjoint RUTH]
\label{ex:adjointRUTH}
Using an Ehresmann connection $h$ on $G$ (see Definition \ref{def:Ehresmann_connection}) we can promote the core complex of $TG$
\begin{equation}\label{eq:core_compl_ad_RUTH:bis}
\begin{tikzcd}
	0 \arrow[r] & A \arrow[r, "\rho"] & TM \arrow[r]&0
\end{tikzcd}
\end{equation}
to a RUTH $(C(G; A[-1] \oplus TM), \partial^{\mathrm{Ad}})$ as follows. Besides the differential $R_0^T = \rho$ of the core complex \eqref{eq:core_compl_ad_RUTH:bis}, the latter RUTH has got only two more non-trivial structure operators, namely the $1$-st and the $2$-nd, denoted by $R^T_1, R^T_2$, which are defined as follows: let $\varpi = h \circ ds - \operatorname{id}\colon TG \to t^\ast A$ be the left splitting corresponding to $h$, then
\[
R^T_1(g)a =  - \varpi_g \big(dl_g (a)\big)\quad \text{and} \quad R^T_1 (g) v = dt \big( h_g (v)\big), \quad a \in A_{s(g)}, \quad v \in T_{s(g)}M,
\]
where $l_g$ is the left translation along $g \in G$, i.e., $dl_g(a)= 0_g^{TG}\cdot a^{-1}$, and, moreover
\[
R^T_2 (g_1, g_2) v = - \varpi_{g_1g_2} \Big(h_{g_1}\big(R^T_1 (g_2)v\big) \cdot h_{g_2} (v) \Big), \quad v \in T_{s(g_2)} M,
\]
where we used the multiplication in $TG \rightrightarrows TM$.

In what follows, we often denote the first structure operator $R_1^T(g)$ simply by $g_T .$, with $g\in G$. Let $g\colon x\to y\in G$, notice that, for any $a\in A_x$, we have
\begin{equation}
\label{eq:quasi_action_onA}
\begin{aligned}
	g_T.a &= dR_{g^{-1}}\left(h_g(\rho(a))- dL_g(a))\right) = \left(h_g(\rho(a))- 0^{TG}_g\cdot a^{-1}\right) \cdot 0^{TG}_{g^{-1}} \\
	&= \left(h_g(\rho(a))- 0^{TG}_g\cdot a^{-1}\right) \cdot \left(a\cdot 0^{TG}_{g^{-1}} - a\cdot 0^{TG}_{g^{-1}}\right) = h_g(\rho(a))\cdot a \cdot 0^{TG}_{g^{-1}} \in A_y,
\end{aligned}
\end{equation}
where we used the interchange law \eqref{eq:interchange_law}. Since $R_1(x)= \operatorname{id}$, then the cochain map
\begin{equation}
\label{eq:adjoint_cochainmap}
\begin{tikzcd}
	0 \arrow[r] & A_x\arrow[r, "\rho"] \arrow[d, "g_T."'] & T_xM\arrow[r] \arrow[d, "g_T."] & 0\\
	0 \arrow[r] & A_y\arrow[r, "\rho"'] & T_yM\arrow[r] & 0
\end{tikzcd}
\end{equation}
is a quasi-isomorphism. %Indeed, $\ker(\rho_x)=\mathfrak{g}_x$ is the isotropy Lie algebra at the point $x$ and in cohomology
%\begin{equation*}
%g_T.\colon \mathfrak{g}_x\to \mathfrak{g}_y
%\end{equation*}
%agrees with the adjoint action of the isotropy Lie group $G_x$ on $\mathfrak{g}_x$. This is obvious from Equation \eqref{eq:quasi_action_onA} in the case $\rho(a)=0$. Then $ T_xM/\im \rho_x=T_xM/T_xO:=N_xO$ and in cohomology
%\begin{equation*}
%g_T.\colon N_xO\to N_yO
%\end{equation*}
%is the normal action.
\end{example}
We are mainly interested in $2$-term RUTHs concentrated in degrees $-1$ and $0$. Indeed, following \cite[Section 3]{GSM17} and \cite[Theorem 2.7]{dHO20}, for any Lie groupoid $G\rightrightarrows M$ there is an equivalence of categories between the category of $2$-term RUTHs concentrated in degrees $-1$ and $0$ of $G$ with RUTH morphisms and VBGs over $G$ with VBG morphisms. We quickly recall this equivalence at the level of objects in Remark \ref{rem:VBG-RUTH}, but, before we need an easy remark on $2$-term RUTHs.

First we recall the definition of \emph{quasi action}. Let $G\rightrightarrows M$ be a Lie groupoid and let $E$ be a vector bundle over $M$.
\begin{definition}
\label{def:quasi_action}
A \emph{quasi action} of $G$ on $E$ is a smooth map
\[
G\mathbin{{}_{s}\times_{\mu}} E \to E, \quad (g,e)\mapsto g.e,
\] 
where $\mu\colon E\to M$ is the bundle map, such that $\mu(g.e)= t(g)$, and, for any $g\in G$, the map $E_{s(g)}\to E_{t(g)}$, $e\mapsto g.e$, is linear. 

A quasi action is \emph{unital} if $\mu(e).e=e$, and it is \emph{flat} if $g.(h.e)=(gh).e$, for all $(g,h)\in G^{(2)}$ and $e\in E_{s(h)}$. 
\end{definition} 
A unital and flat quasi action is just a representation of $G$.

In the next remark, following \cite[Section 2.7]{GSM17}, we discuss the $2$-term RUTHs, making explicit the structure operators of a $2$-term RUTH concentrated in degrees $-1$ and $0$.
\begin{rem}
\label{rem:2-term_RUTH}
Let $C$ and $V_M$ be VBs over $M$. For degree reasons a RUTH of $G$ on $C[-1]\oplus V_M$ consists in just four non-trivial maps: the $0$-th structure operator $R_0\colon C\to V_M$ that we indicate by $\delta$. The first structure operator $R_1$ splits in two maps, one is a quasi action $\Delta^C$ of $G$ on $C$ and the other one is a quasi action $\Delta^V$ of $G$ on $V_M$. From $R_1(x)=\operatorname{id}$ for any $x\in M$, we have that $\Delta^V$ and $\Delta^C$ are unital. The second structure operator $R_2$ is a section of $\operatorname{Hom}(s^{\ast}V_M, t^{\ast}C)\to G$ that we indicate by $\Omega$. Moreover, $\Omega_{g_1,g_2}=0$ whenever $g_1$ or $g_2$ are in the image of the unit map $u\colon M\to G$. These four maps $(\delta, \Delta^C, \Delta^V, \Omega)$ satisfy the following conditions:
\begin{align*}
\Delta^V_{g_1}\delta - \delta\Delta^C_{g_1}&=0,\\
\Delta^C_{g_1}\Delta^C_{g_2}- \Delta^C_{g_1g_2} + \Omega_{g_1,g_2}\delta&=0, \\
\Delta^V_{g_1}\Delta^V_{g_2}- \Delta^V_{g_1g_2} + \delta\Omega_{g_1,g_2}&=0,\\
\Delta^C_{g_1}\Omega_{g_2,g_3} - \Omega_{g_1g_2,g_3} + \Omega_{g_1,g_2g_3} -\Omega_{g_1,g_2}\Delta^V_{g_3}&=0,
\end{align*}
for all $(g_1,g_2,g_3)\in G^{(3)}$.
\end{rem}

Now we recall the equivalence of categories between the cateogry of $2$-terms RUTH of $G$ concentrated in degrees $-1$ and $0$ with RUTH morphisms and VBGs over $G$ with a right-horizontal lift and VBG morphisms.
\begin{rem}
\label{rem:VBG-RUTH}
Let $C[-1]\oplus V_M$ be a $2$-terms RUTH of $G$ concentrated in degrees $-1$ and $0$ and let $(\delta, \Delta^C, \Delta^V, \Omega)$ be the maps describing such a RUTH, as discussed in Remark \ref{rem:2-term_RUTH}. Following \cite[Example 3.16]{GSM17}, we can build the \emph{semidirect product} VBG over $G$ as follows. The total bundle is 
\[
V=t^{\ast}C\oplus s^{\ast}V_M=\{(c,g,v) \, | \, c\in C_{t(g)}, \, v\in V_{M,s(g)}\}.
\]
The structure maps are the following:
\begin{itemize}
\item the source and target maps $s,t\colon V\to V_M$ are given by
\begin{equation*}
	s(c,g,v)= v, \quad t(c,g,v)= \delta c+ \Delta^V_gv;
\end{equation*}
\item the multiplication of composable arrows is given by
\begin{equation*}
	(c_1,g_1,v_1) \cdot (c_2, g_2, v_2)= (c_1+ \Delta^C_{g_1} c_2- \Omega_{g_1,g_2}v_2, g_1g_2, v_2);
\end{equation*}
\item the unit of $v\in V_{M,x}$ is $(0,x,v)$;
\item the inverse is given by
\[
(c,g,v)^{-1}=(-\Delta^C_{g^{-1}}c+ \Omega_{g^{-1},g}v, g^{-1}, \delta c+ \Delta^V_g v).
\]
\end{itemize}

Let $(V\rightrightarrows V_M;G\rightrightarrows M)$ be a VBG and let $h\colon s^\ast V_M \to V$ a right-horizontal lift (see Definition \ref{def:right_horizontal_splitting}). Following \cite[Section 3.3]{GSM17}, $V$ induces a RUTH on $C[-1]\oplus V_M$ whose described by:
\begin{itemize}
\item the $0$-th structure operator is the core-anchor $t|_C\colon C\to V_M$;
\item the quasi action $\Delta^C$ on $C$ is given by
\begin{equation*}
	\Delta^C_g c= h_g(t(c))\cdot c \cdot 0_{g^{-1}}, \quad c\in C_{t(g)}, \, g\in G;
\end{equation*}
\item the quasi action $\Delta^V$ on $V_M$ is given by
\begin{equation*}
	\Delta^V_g v= t(h_g(v)), \quad v\in V_{M,s(g)}, \, g\in G;
\end{equation*}
\item the second structure operator $\Omega$ is given by
\begin{equation*}
	\Omega_{g_1,g_2} v=\left(h_{g_1g_2}v - h_{g_1}(t(h_{g_2}v))\cdot h_{g_2}v\right)\cdot 0_{{g_1g_2}^{-1}}, \quad v\in V_{M,s(g_2)}, \, (g_1,g_2)\in G^{(2)}.
\end{equation*}
\end{itemize}
It is easy to see that the RUTH associated to the tangent VBG is the adjoint RUTH discussed in Example \ref{ex:adjointRUTH}. These constructions are one the inverse of the other one. Finally, this correspondence extends to RUTH morphisms and VBG morphisms, this is proved in \cite[Theorem 2.7]{dHO20}. 
\end{rem}

%\begin{rem}
%	
%	Let $(V\rightrightarrows V_M;G\rightrightarrows M)$ be a VBG with core $C$. The quasi-actions $\Delta^C$ and $\Delta^V$ are unital. Indeed, for any $x\in M$, $h_x$ is the unit map and so
%	\[
%		\Delta^C_x c= t(c)\cdot c \cdot 0_x= c,
%	\] 
%	for all $c\in C_x$, and
%	\[
%		\Delta^V_x v= t(v)= v,
%	\]
%	for all $v\in V_{M,x}$. Hence, for any $g\in G$, the quasi-actions $\Delta^C_g$ and $\Delta^V_g$ form a quasi-isomorphism between the fiber of $V$ at $s(g)$ and the fiber of $V$ at $t(g)$.
%\end{rem}

Using the equivalence of categories between $2$-terms RUTHs concentrated in degrees $-1$ and $0$ and VBGs, described in Remark \ref{rem:VBG-RUTH}, we can get some results on VBGs, for instance, as discussed in \cite[Section 6.2]{GSM17}, we can classify \emph{regular} VBGs. We recall this classification in the following
\begin{rem}
\label{rem:regularVBG}
A VBG $V\rightrightarrows V_M$ is \emph{regular} if its core-anchor $t|_C\colon C\to V_M$ has constant rank. Following \cite[Section 6.2]{GSM17} two particular cases are VBG \emph{of type $1$}, i.e., VBGs whose core-anchor $t|_C$ is a VB isomorphism, and VBG \emph{of type $0$}, i.e., VBGs whose core-anchor $t|_C$ is the zero map. In \cite[Lemma 6.10]{GSM17} regular VBGs are classified through VBGs of type $1$ and of type $0$. Specifically any regular VBG $V\rightrightarrows V_M$ is isomorphic to the direct sum (see Example \ref{ex:direct_sum_VBG}) of unique (up to isomorphisms) VBGs of type $0$ and of type $1$. Indeed, let $V\rightrightarrows V_M$ be a regular VBG with core $C$. Let $K:=\ker(t|_C)$, $F:=\im(t|_C)$, and $\nu := V_M/F$. Then choosing splittings of the short exact sequences of VBs over $M$
\begin{equation*}
%\scriptsize
\begin{tikzcd}
	0 \arrow[r] & K \arrow[r] & C \arrow[r] & F \arrow[r] &0, \\
	0 \arrow[r] & F \arrow[r] & V_M \arrow[r] & \nu \arrow[r] &0,
\end{tikzcd}
\end{equation*}
we get isomorphisms $C\cong K\oplus F$ and $V_M\cong \nu\oplus F$. With an appropriate choice of the right horizontal lift for $V$, the associated $2$-terms RUTH decomposes as the direct sum of a $2$-terms RUTH on $\nu[-1]\oplus K$, that comes from a type $0$ VBG, and a $2$-terms RUTH on $F[-1]\oplus F$, that comes from a type $1$ VBG, hence the result.
\end{rem}

We recall some operations on RUTHs, introduced in \cite[Section 3.3]{AC13}, that will be useful later. Specifically, we recall the \emph{dual RUTH} of a RUTH and the \emph{mapping cone} of RUTH morphism. 
\begin{example}[{\cite[Example 3.19]{AC13}}]
\label{ex:dual_RUTH}
Let $V$ be a RUTH of $G\rightrightarrows M$ with structure operators $\{R_k\}_{k\geq 0}$. Then the \emph{dual RUTH} is the RUTH of $G$ on $V^{\ast}$, where $(V^\ast)^i=(V^{-i})^\ast$, whose structure operators are $\{R^{\ast}_k\}_{k\geq 0}$ defined by
\begin{equation*}
R^{\ast}_k(g_1, \dots,g_k) = (-1)^{k+1} (R_k(g_k^{-1}, \dots, g_1^{-1}))^{\ast},
\end{equation*}
for all $(g_1,\dots, g_k)\in G^{(k)}$. 

When $\{R_k\}_{k\geq 0}$ is the RUTH coming from a VBG $V$ and a right-horizontal lift $h\colon s^\ast V_M\to V$, and $\{R'_k\}_{k\geq 0}$ are the structure operators coming from the dual VBG $V^\ast$ (Example \ref{ex:dual_VBG}) and the right-horizontal lift $h^\ast$ (Equation \eqref{eq:dual_splitting}), then $R_k'$ agrees with the structure operators $R_k^\ast$ of the dual RUTH. %Indeed, the $0$-th one is simply the core-anchor $\rho^\ast\colon T^\ast M\to A^\ast$ of $T^\ast G$. For any $g\in G$, the $1$-st one is equivalent to two quasi-actions on $T^\ast M$ and $A^\ast$ given by
%\begin{align*}
%	\langle R_1'(g)(\psi) , v\rangle = \langle h^\ast_g(t(\psi))\cdot \psi\cdot 0_{g^{-1}}^{T^\ast G}, w\cdot \rangle  
%\end{align*}

A particular case is given by the adjoint RUTH (Example \ref{ex:adjointRUTH}). The dual of the adjoint RUTH is called \emph{coadjoint RUTH} and it is the RUTH associated to the cotangent VBG $T^\ast G$.
\end{example}

Before describing the mapping cone of a RUTH morphism, we recall the construction in the simpler case of a cochain map between two complexes of VBs. 
\begin{rem}
	\label{rem:mapping_cone}
	Let $\Phi_0\colon (V,\delta^V)\to (W,\delta^W)$ be a cochain map between complexes of VBs over the same base, then the mapping cone of $\Phi_0$ is a new complex of VBs $(\mathcal{C}(\Phi_0), \delta)$, where
	\[
	\mathcal{C}(\Phi_0)^n= V^n\oplus W^{n-1} \quad \text{and} \quad \delta(v,w)=(-\delta^V(v), \Phi_0(v)+\delta^W(w)).
	\]
	The construction can be extended to other categories. Moreover, we recall from \cite[Corollary 1.5.4]{We94} that a cochain map is a quasi-isomorphism if and only if its mapping cone is acyclic.
\end{rem}

\begin{example}[{\cite[Example 3.21]{AC13}}]
\label{ex:mp_RUTH}
Let $\Phi\colon V \to W$ be a RUTH morphism between the RUTH $V$ with structure operators $\{R_k^V\}_{k\geq 0}$ and the RUTH $W$ with structure operators $\{R_k^W\}_{k\geq 0}$. The \emph{mapping cone} of $\Phi$ is the RUTH on $\mathcal{C}(\Phi_0)$ with structure operators $\{R_k\}_{k\geq0}$ defined by
\[
R_k(g_1, \dots,g_k)(v,w)=\left(R_k^V(g_1,\dots,g_k)v, \Phi_k(g_1,\dots,g_k)v-(-1)^kR_k^W(g_1,\dots,g_k)w\right),
\]
for all $v\in V_{s(g_k)}$ and $w\in W_{s(g_k)}$, with $(g_1, \dots, g_k)\in G^{(k)}$, where $\Phi_k$ is the $k$-th component of the RUTH morphism $\Phi$.
\end{example}

%\begin{rem}
%	The mapping cone of a cochain map, as discussed in Remark \ref{rem:mapping_cone}, and the mapping cone of a RUTH morphism, as discussed in Example \ref{ex:mp_RUTH}, use two different conventional choices of sign. However, it is straightforward to verify that if $\{R_k\}_{k\geq 0}$ are the structure operators of a RUTH of a Lie groupoid $G$, then the sequence $\{(-1)^{k-1}R_k\}_{k\geq 0}$ also determines a RUTH of $G$. Furthermore, if $\{\Phi_k\}_{k\geq 0}$ are the components of a RUTH morphism between two RUTHs with structure components $\{R_k\}_{k\geq 0}$ and $\{R'_k\}_{k\geq 0}$, then the sequence $\{(-1)^k\Phi_k\}_{k\geq 0}$ determines a RUTH morphism between the RUTHs with structure operators $\{(-1)^{k-1}R_k\}_{k\geq 0}$ and $\{(-1)^{k-1}R'_k\}_{k\geq 0}$. These two RUTH morphisms encode the same information, but the mapping cone of the RUTH morphism with components $\{(-1)^k\Phi_k\}_{k\geq 0}$, as described in Example \ref{ex:mp_RUTH}, is consistent with the mapping cone of complexes, as discussed in Remark \ref{rem:mapping_cone}. Nevertheless, we choose to adopt the standard convention for signs, as outlined in Example \ref{ex:mp_RUTH}.
%\end{rem}

As discussed in \cite[Example 3.22]{AC13}, the tensor product between RUTHs is a complicated construction, but fortunately we are only interested in the tensor product between a RUTH and a plain representation. We introduce this construction in the following
\begin{example}[Tensor product]
\label{ex:RUTH_tensor_product}
Let $V=\bigoplus_{i\in \mathbbm{Z}} V^i$ be a RUTH of $G\rightrightarrows M$ with structure operators $\{R_k\}_{k\geq 0}$ and let $E_M$ be a representation of $G$. We denote by $g.\colon E_{M,s(g)}\to E_{M,t(g)}$ the isomorphism determined by the action of $g\in G$. We define the \emph{tensor product RUTH} between $V$ and $E_M$ as the RUTH on the graded vector bundle over $M$
\[
V\otimes E_M = \bigoplus_{i\in \mathbbm{Z}} V^i\otimes E_M,
\]
whose structure operators $\{R_k^\otimes\}_{k\geq 0}$ are defined as follows: for any $(g_1, \dots, g_k)\in G^{(k)}$ and any $i\in \mathbbm{Z}$, we set
\begin{equation}
	\label{eq:so_tensor}
	R_k^\otimes (g_1, \dots, g_k):= R_k(g_1,\dots,g_k)\otimes(g_1\dots,g_k).\colon V^i_{s(g_k)}\otimes E_{M,s(g_k)}\to V^{i-k+1}_{t(g_1)}\otimes E_{M,t(g_1)}. 
\end{equation}
It is clear that $R_1^\otimes (x)=\operatorname{id}$ for all $x\in M$ and $R_k^\otimes(g_1,\dots,g_k)=0$, with $k>1$, whenever $g_j\in M$, for at least one $j$.

The identities \eqref{eq:struct_sect_RUTH} are satisfied, Indeed, for any $k\geq 0$ and $(g_1,\dots,g_k)\in G^{(k)}$ we have
\begin{align*}
\sum_{j=1}^{k-1}(-)^j &R_{k-1}^\otimes(g_1, \ldots, g_jg_{j+1}, \ldots, g_k)(v\otimes e) \\
&=\sum_{j=1}^{k-1}(-)^j R_{k-1}(g_1, \ldots, g_jg_{j+1}, \ldots, g_k)(v)\otimes (g_1,\dots,g_k).e\\
&= \sum_{j= 0}^k (-)^j R_j (g_1, \ldots, g_j) \left(R_{k-j} (g_{j+1}, \ldots, g_k)(v)\right)\\
&\quad \otimes(g_1,\dots,g_j).\left((g_{j+1}, \dots, g_k).e\right) \\
&=\sum_{j= 0}^k (-)^j R_j^\otimes (g_1, \ldots, g_j) \left(R_{k-j} (g_{j+1}, \ldots, g_k)(v)\otimes\left((g_{j+1}, \dots, g_k).e\right) \right) \\
&= \sum_{j= 0}^k (-)^j R_j^\otimes (g_1, \ldots, g_j) \left(R_{k-j}^\otimes (g_{j+1}, \ldots, g_k)(v\otimes e)\right),
\end{align*}
for all $v\in V^i_{s(g_k)}$ and $e\in E_{M,s(g_k)}$, with $i\in \mathbbm{Z}$.

Let $\{R_k\}_{k\geq 0}$ be the structure operators of a RUTH coming from the VBG $V$ choosing a right-horizontal lift $h\colon s^\ast V_M \to V$, and let $E$ be a trivial core VBG. Then the structure operators $\{R'_k\}_{k\geq 0}$ coming from the tensor product VBG $V\otimes E$ (Example \ref{ex:tensor_product_VBG}) choosing the right-horizontal lift $h\colon s^\ast(V_M\otimes E_M)\to V\otimes E$ defined by Equation \eqref{eq:tensor_splitting}, agree with the structure operators $\{R_k^\otimes\}_{k\geq 0}$ of the tensor product RUTH between $V$ and $E_M$. Indeed, $R_0'$ is the core-anchor \eqref{eq:core_anchor_tensor} that agrees with the structure operator $R_0^\otimes$ described in Equation \eqref{eq:so_tensor}. Following Remark \ref{rem:VBG-RUTH}, for any $g\in G$, $R_1'(g)$ is given by the quasi-actions on the core $C\otimes E_M$ and the side bundle $V_M\otimes E_M$ given by:
\begin{align*}
	R_1'(g)(c\otimes e) &= h_g(t(c\otimes e)) \cdot c\otimes e \cdot 0_{g^{-1}}^{V\otimes E} \\&= \big(h_g(t(c)) \otimes s_g^{-1}(e)\big) \cdot \big(c\otimes e\big) \cdot \big(0_{g^{-1}}^V\otimes s_{g^{-1}}^{-1}(e)\big)\\
	&= (h_g(t(c))\cdot c\cdot 0_{g^{-1}}^V)\otimes (s_g^{-1}(e)\cdot e\cdot s_{g^{-1}}^{-1}(e))\\
	&= R_1(g)(c)\otimes g.e\\
	&= R_1^\otimes (g)(c\otimes e),
\end{align*}
for all $c\in C_{s(g)}$ and $e\in E_{M,s(g)}$, and
\begin{align*}
	R_1'(g)(v\otimes e)= t(h_g(v\otimes e))= t(h_g(v)\otimes s_g^{-1}(e))= R_1(g)(v)\otimes g.e = R_1^\otimes (g)(v\otimes e),
\end{align*}
for all $v\in V_{M,s(g)}$ and $e\in E_{M,s(g)}$. Finally, for any $(g,g')\in G^{(2)}$, we have
\begin{align*}
	R_2'(g,g')(v\otimes e)&= \left(h_{gg'}(v\otimes e) - h_{g}(t(h_{g'}(v\otimes e)))\cdot h_{g'}(v\otimes e)\right)\cdot 0^{V\otimes E}_{(gg')^{-1}}\\
	&=\left(h_{gg'}(v)\otimes s_{gg'}^{-1}(e) - h_{g}(t(h_{g'}(v))\otimes g'.e)\cdot (h_{g'}(v)\otimes s_{g'}^{-1}(e))\right) \\
	&\quad \cdot 0^{V}_{(gg')^{-1}}\otimes s_{(gg')^{-1}}^{-1}(e)\\
	&= \left(h_{gg'}(v)\otimes s_{gg'}^{-1}(e) - \big(h_g(t(h_{g'}(v)))\cdot h_{g'}(v)\big)\otimes s_{gg'}^{-1}(e)\right)\\ 
	&\quad\cdot 0^{V}_{(gg')^{-1}}\otimes s_{(gg')^{-1}}^{-1}(e)\\
	&= \left(\left(h_{gg'}(v) - \big(h_g(t(h_{g'}(v)))\cdot h_{g'}(v)\big)\right) \otimes s_{gg'}^{-1}(e)\right) \\
	&\quad \cdot 0^{V}_{(gg')^{-1}}\otimes s_{(gg')^{-1}}^{-1}(e)\\
	&= \left(\left(h_{gg'}(v) - \big(h_g(t(h_{g'}(v)))\cdot h_{g'}(v)\big)\right)\cdot 0^V_{(gg')^{-1}}\right)\otimes gg'.e\\
	&= R_2(g,g')(v)\otimes gg'.e\\
	&= R_2^\otimes (g,g')(v\otimes e),
\end{align*}
for all $v\in V_{M,s(g)}$ and $e\in E_{M,s(g)}$.
\end{example}

Combining Example \ref{ex:dual_RUTH} with Example \ref{ex:RUTH_tensor_product} we get the dual RUTH twisted by a representation, discussed in the following
\begin{example}[Twisted dual RUTH]
\label{ex:twisted_dual_RUTH}
Let $G\rightrightarrows M$ be a Lie groupoid. Let $E_M$ be a representation of $G$ and $V$ be a RUTH of $G\rightrightarrows M$ with structure operators $\{R_k\}_{k\geq 0}$. We define the \emph{$E_M$-twisted dual RUTH} as the RUTH of $G$ that we get, considering first the dual RUTH of $G$ on $V^\ast$ (see Example \ref{ex:dual_RUTH}) and, then, the tensor product with the representation of $G$ on $E_M$ (see Example \ref{ex:RUTH_tensor_product}). The graded vector bundle on $M$ that supports the twisted dual RUTH is
\[
V^\dag=\bigoplus_{i\in \mathbbm{Z}} (V^{-i})^\ast \otimes E_M \cong \bigoplus_{i\in \mathbbm{Z}} \operatorname{Hom}(V^{-i}, E_M)= \operatorname{Hom}(V, E_M).
\]
Under the latter identification, the structure operators $\{R^\dag_k\}_{k \geq 0}$ of the twisted dual RUTH are the following: for any $(g_1, \dots, g_k)\in G^{(k)}$
\[
R^\dag_k (g_1, \ldots, g_k) = (-1)^{k+1} R_k (g_k^{-1}, \ldots, g_1^{-1})^\dag \colon (V_{s(g_k)}^\dag)^\bullet \to (V_{t(g_1)}^\dag)^{1-k+\bullet},
\]
maps any $\psi\in V^\dag_{s(g_k)}=\operatorname{Hom}(V_{s(g_k)}, E_{M,s(g_k)})$ to $$R_k^\dag(g_1, \dots, g_k)(\psi)\in \psi\in V^\dag_{t(g_1)}=\operatorname{Hom}(V_{t(g_1)}, E_{M,t(g_1)})$$ defined by
\begin{equation}
	\label{eq:adjoint_structure_op}
R_k^\dag(g_1, \dots, g_k)(\psi)(v)= (-1)^{k+1}(g_1\cdots g_k). \psi(R_k(g_k^{-1}, \dots, g_1^{-1})(v)) \in E_{M,t(g_1)},
\end{equation}
for all $v\in V_{t(g_1)}$.

Combining the last part of Examples \ref{ex:dual_RUTH} and \ref{ex:RUTH_tensor_product}, we get that, if $\{R_k\}_{k\geq 0}$ are the structure operators of the RUTH coming from a VBG $V$ and a right-horizontal lift $h\colon s^\ast V_M\to V$ and $E$ is a trivial core VBG, then the structure operators of the RUTH coming from the twisted dual VBG $V^\dagger$ and the right-horizontal lift $h^\dagger$ (Equation \eqref{eq:splitting_twisted_dual}) agree with the structure operators of the $E_M$-twisted dual RUTH.
\end{example} 

For future purposes we describe the RUTH coming from the Atiyah VBG $DL \rightrightarrows DL_M$. 
\begin{example}[Atiyah RUTH]
\label{ex:Atiyah_RUTH}
An Ehresmann connection $h$ in $G$ (see Definition \ref{def:Ehresmann_connection}) also induces a right splitting $h^D\colon s^\ast DL_M \to DL$ of the short exact sequence
\begin{equation}\label{eq:SES_Appendix}
\begin{tikzcd}
	0 \arrow[r] & t^\ast A \arrow[r] & DL \arrow[r] & s^\ast DL_M \arrow[r]  \arrow[l, bend left, "h^D"] & 0
\end{tikzcd}
\end{equation}
as follows. Given $\delta \in D_xL_M$, $x \in M$, by Proposition \ref{prop:derivation_symbol} applied to the source map, we get that for every arrow $g\colon x \to y$, there exists a unique derivation $h^D_g (\delta) \in D_g L$ such that $\sigma (h^D_g (\delta)) = h_g (\sigma (\delta))$, and $Ds (h^D_g (\delta)) = \delta$. In its turn $h^D$ determines a $2$-term RUTH concentrated in degrees $-1$ and $0$ that we call the \emph{Atiyah RUTH}. The underlying cochain complex of vector bundles is the core complex
\begin{equation}\label{eq:core_compl_DL_RUTH}
\begin{tikzcd}
	0 \arrow[r] & A \arrow[r, "\mathcal D"] & DL_M \arrow[r] &  0,
\end{tikzcd}
\end{equation}
and besides the differential $R_0^D = \mathcal D$ of  \eqref{eq:core_compl_DL_RUTH}, there are only two more non-trivial structure operators, the $1$-st and the $2$-nd, denoted $R^D_1, R^D_2$, which are defined in a similar way as $R_1^T, R_2^T$ for the adjoint RUTH:
\[
R^D_1(g)a =  R^T_1 (g)a \quad \text{and} \quad R^D_1 (g) \delta = Dt \big( h^D_g (\delta)\big), \quad \delta \in D_{s(g)}L_M, \quad g\in G,
\]
for all $a \in A_{s(g)}$, moreover
\begin{equation}\label{eq:gst36}
R^D_2 (g_1, g_2) \delta = R^T_2 (g_1, g_2) \sigma(\delta), \quad \delta \in D_{s(g_1)}L_M, \quad (g_1,g_2)\in G^{(2)}.
\end{equation}

The maps $R^D$ satisfy some easy properties. The symbol intertwines the structure operators $R^D$ with the structure operators $R^T$ (see Example \ref{ex:adjointRUTH}). Indeed, $R_0^D=\mathcal{D}$ and, from Equation \eqref{eq:DFcommutes}, $\sigma \circ \mathcal{D}=\rho =R_0^T$. Then, the quasi action of $G$ on $A$ is the same as the quasi action of the adjoint RUTH, and for any $\delta\in D_{s(g)}L_M$, with $g\in G$, we have
\[
\sigma\big(R_1^D(g)\delta\big)= \sigma\big(Dt\big(h^D_g(\delta)\big)\big)= dt\big(\sigma \big(h^D_g(\delta)\big)\big)= dt\left(h_g(\sigma(\delta))\right)= R_1^T(\sigma(\delta)),
\]
where we used again Equation \eqref{eq:DFcommutes} and the definition of $h^D$. Finally, by the definition of $R_2^D$, the symbol intertwines $R_2^D$ with $R_1^T$.

Moreover $R_1^D(g)$ acts as the identity on endomorphisms. First, for any $g\in G$, $h^D_g\left(\mathbbm{I}_{s(g)}\right)= \mathbbm{I}_g\in D_gL$. Indeed,
\[
\sigma\big(h^D_g\big(\mathbbm{I}_{s(g)}\big)\big)=h_g\big(\sigma\big(\mathbbm{I}_{s(g)}\big)\big)=0_g=\sigma\left(\mathbbm{I}_g\right),
\]
and, for any $\lambda\in \Gamma(L_M)$,
\[
h^D_g\left(\mathbbm{I}_{s(g)}\right)(s^\ast \lambda)= s_g^{-1}\big(Ds\big(h^D_g\left(\mathbbm{I}_{s(g)}\right)\big)\lambda\big)= s_g^{-1}\left(\lambda_{s(g)}\right)= (s^\ast\lambda)_g=\mathbbm{I}_g(s^\ast \lambda).
\]
Hence, for any $r\in \mathbbm{R}$, if we apply $R_1^D(g)$, with $g\in G$, to the derivation $r\mathbbm{I}_{s(g)}\in D_{s(g)}L_M$, we get
\[
R_1^D(g)\left(r\mathbbm{I}_{s(g)}\right)=Dt\big(h^D_g(r\mathbbm{I}_{s(g)})\big)=r Dt\left(\mathbbm{I}_g\right)=r\mathbbm{I}_{t(g)} \in D_{t(g)}L_M,
\]
where we used that $Dt\left(\mathbbm{I}_g\right)=\mathbbm{I}_{t(g)}$. Indeed, for any $\lambda\in\Gamma(L_M)$
\[
Dt\left(\mathbbm{I}_g\right)\lambda= t\left(\mathbbm{I}_g (t^\ast\lambda)\right)= \lambda_{t(g)}= \mathbbm{I}_{t(g)}\lambda.
\]
We proved that, under the identifications $\mathbbm R \cong \operatorname{End} L_{M, s(g)} \subseteq D_{s(g)}L_M$ and $\mathbbm{R}\cong \operatorname{End} L_{M, t(g)}\subseteq D_{t(g)}L_M$, 
\begin{equation}
\label{eq:R1onendomorphsims}
R_1^D(r)=r, \quad \text{for all } r\in \mathbbm{R}.
\end{equation}

Finally, from the definition of $R_2^D$, it follows that
\begin{equation}
R_2^D(g_1, g_2)(r)= 0, \quad \text{for all } (g_1,g_2)\in G^{(2)}, \quad r\in \mathbbm{R}.\qedhere
\end{equation}
%		
%The identities \eqref{eq:struct_sect_RUTH} are satisfied. For $k=0$ this is trivial. For $k=1$, for any $a\in A_{s(g)}$, with $g\in G$, the derivations $\mathcal{D}_{R_1^D(g)a}$ and $R_1^D(g)\mathcal{D}_a$ in $D_{t(g)}L_M$ agree. Indeed, the symbols of the two derivations agree, namely
%\[
%\sigma\left(\mathcal{D}_{R_1^D(g)a}\right)= \rho\left(R_1^D(g)a\right)=\rho\left(R_1^T(g)a\right) = R_1^T(g)\rho(a)= R_1^T(g)\sigma\left(\mathcal{D}_a\right)= \sigma \left(R_1^D(g)\mathcal{D}_a\right),
%\]
%where we used that $\sigma$ intertwines $R^D$ with $R^T$. Moreover, for any $\lambda\in \Gamma(L_M)$,
%\[
%	\mathcal{D}_{R_1^D(g)a}= \mathcal{D}_{R_1^T(g)a}=\rho(R_1^T(g)a)= R_1^T(g)\rho(a)= R_1^T(g)\sigma(\mathcal{D}_a)= \sigma(R_1^D(g)\mathcal{D}_a)
%\]
\end{example}

\begin{example}[Jet RUTH]
	\label{ex:jet_RUTH}
	Let $(L\rightrightarrows L_M;G\rightrightarrows M)$ be an LBG. Applying the $L_M$-twisted dual construction (Example \ref{ex:twisted_dual_RUTH}) to the Atiyah RUTH (Example \ref{ex:Atiyah_RUTH}) we get a RUTH of $G$ on $J^1L_M[-1]\oplus A^\dagger$ that we call the \emph{jet RUTH}. Let $\{R_k^D\}_{k\geq0}$ be the structure operators of the Atiyah RUTH. Then, the structure operators $\{R_k^\dagger\}_{k\geq 0}$ of the jet RUTH are the following: the $0$-th one is given by
	\[
		R_0^\dagger= (R_0^D)^\dagger =\mathcal{D}^\dagger \colon J^1L_M \to A^\dagger, \quad \langle \mathcal{D}^\dagger(\psi), a\rangle = \langle \psi, \mathcal{D}(a)\rangle,
	\]
	for all $\psi\in J^1_xL_M$ and $a\in A_x$, with $x\in M$, and where $\mathcal{D}\colon A \to DL_M$ is the core-anchor of $DL$.
	
	For any $g\in G$, $R_1^\dagger(g)$ is given by two quasi actions on $J^1L_M$ and $A^\dagger$ that we denote by $g_\dagger .$ and are given by
	\[
		\langle g_\dagger . \psi , \delta\rangle= g. \langle \psi , g^{-1}_D . \delta\rangle ,
	\]
	for all $\psi\in J^1_{s(g)}L_M$ and $\delta \in D_{t(g)}L_M$, and
	\[
		\langle g_\dagger .\psi , a\rangle = g. \langle \psi, g^{-1}_T.a\rangle,
	\]
	for all $\psi \in A^\dagger_{s(g)}$ and $a\in A_{t(g)}$.
	
	Finally, for any $(g,g')\in G^{(2)}$, $R_2^\dagger(g,g')\colon A^\dagger_{s(g')}\to J^1_{t(g)}L_M$ is given by
	\[
		\left\langle R_2^\dagger(g,g')(\psi),\delta\right\rangle = gg'. \left\langle \psi , R_2^D(g'^{-1}, g^{-1}) (\delta) \right\rangle,
	\]
	for all $\psi\in A^\dagger_{s(g')}$ and $\delta\in D_{t(g)}L_M$. From the last part of Example \ref{ex:twisted_dual_RUTH} it follows that this is exactly the RUTH coming from the jet VBG (see Remark \ref{rem:jet_VBG}).
\end{example}

\begin{rem}
	\label{rem:higherVBGs}
The correspondence between $2$-term RUTHs and VBGs can be generalized. Indeed, there is an equivalence between the category of RUTHs of a fixed Lie groupoid $G$, concentrated in non-positive degrees, and the category of \emph{higher VB-groupoids} on $G$, i.e., simplicial vector bundles over $G^{(\bullet)}$ (up to some technical aspects, see \cite{dHT23}).
\end{rem}

\section{VB-Morita equivalence and VB-stacks}
\label{sec:VB-Morita}
The concept of Morita equivalence for Lie groupoids can be extended to VBGs. An approach based on principal bibundles (Definition \ref{def:principal_bibundle}) is presented in \cite[Section 4.2]{BCGX22}. Here, we follows the approach in \cite[Section 3]{dHO20}, recalling the notion of \emph{VB-Morita map}. Furthermore, we specialize this equivalence to the case of LBGs and we discuss the Morita equivalence between the twisted dual VBGs of Morita equivalent VBGs.

Let $(W\rightrightarrows W_N;H\rightrightarrows N)$, $(V\rightrightarrows V_M;G\rightrightarrows M)$ be two VBGs.
\begin{definition}
\label{def:VB-Morita}
A VBG morphism $(F,f)\colon W\to V$ is a \emph{VB-Morita map} if $F\colon (W\rightrightarrows W_N)\to (V\rightrightarrows V_M)$ is a Morita map.
\end{definition}

Notice that VBG isomorphisms are VB-Morita maps. Moreover, a straightforward property of VB-Morita maps is described in the following
\begin{lemma}[Two-out-of-three]
\label{lemma:two-out-of-three-VBG}
In a commutative triangle of VBG morphisms, if two of the three VBG morphisms are VB-Morita maps, then the third is so as well.
\end{lemma}
\begin{proof}
The proof follows from Lemma \ref{lemma:two-out-of-three} and Definition \ref{def:VB-Morita}.
\end{proof}
An useful characterization of VB-Morita maps is given by del Hoyo and Ortiz in \cite{dHO20}. First, notice that, as already discussed in Section \ref{sec:RUTHs}, any VBG morphism $(F,f)\colon (W\rightrightarrows W_N; H\rightrightarrows N)\to (V\rightrightarrows V_M; G\rightrightarrows M)$ (or, equivalently, RUTH morphism) induces a cochain map between the core complexes
\begin{equation}\label{eq:fiber_VBGM}
	\begin{tikzcd}
		0 \arrow[r] &D \arrow[r] \arrow[d, "F"'] &W_N \arrow[r] \arrow[d, "F"]&0 \\
		0 \arrow[r]& C \arrow[r] &V_M \arrow[r] &0
	\end{tikzcd},
\end{equation}
where $C$ and $D$ are the cores of $V$ and $W$ respectively. Both components $F \colon D \to C$ and $F\colon W_N \to V_M$ are VB morphisms covering $f \colon N \to M$, and we will often consider the restriction 
\begin{equation}\label{eq:fiber_x_VBGM}
	\begin{tikzcd}
		0 \arrow[r] &D_x \arrow[r] \arrow[d, "F_x"'] &W_{N, x} \arrow[r] \arrow[d, "F_x"]&0 \\
		0 \arrow[r]& C_{f(x)} \arrow[r] &V_{M, f(x)} \arrow[r] &0
	\end{tikzcd},
\end{equation}
of \eqref{eq:fiber_VBGM} to the fibers over $x \in N$ and $f(x) \in M$. 

\begin{theo}[{\cite[Theorem 3.5]{dHO20}}] \label{theo:caratterizzazioneVBmorita}
Let $(F,f)\colon (W\rightrightarrows W_N; H\rightrightarrows N)\to (V\rightrightarrows V_M; G\rightrightarrows M)$ be a VBG morphism. The following conditions are equivalent:
\begin{enumerate}
\item $(F,f)$ is a VB-Morita map;
\item $f\colon (H\rightrightarrows N)\to (G\rightrightarrows M)$ is a Morita map and for any $y\in N$ the cochain map \eqref{eq:fiber_x_VBGM} between the fibers of $W$ over $y\in N$ and of $V$ over $f(y)\in M$ is a quasi-isomorphism.			
\end{enumerate}
\end{theo}
In particular, by Theorem \ref{theo:caratterizzazioneVBmorita}, if $(F,f)$ is a VB-Morita map then $f$ is a Morita map. Moreover, composition of VB-Morita maps is again a VB-Morita map.

VB-Morita maps equivalence between trivial core VBGs (in particular LBGs) have an easier description given by the following corollary of Theorem \ref{theo:caratterizzazioneVBmorita}.
\begin{coroll}
\label{coroll:VBMorita}
Let $(F,f)\colon (E'\rightrightarrows E'_N; H \rightrightarrows N)\to (E\rightrightarrows E_M; G\rightrightarrows M)$ be a VBG morphism between trivial core VBGs. The following conditions are equivalent
\begin{itemize}
\item[\emph{i)}] $(F,f)$ is a VB-Morita map;
\item[\emph{ii)}] $f \colon H \to G$ is Morita and $F\colon (E'\to H)\to (E\to G)$ is a regular VB morphism;
\item[\emph{iii)}] $f \colon H \to G$ is Morita and $F\colon (E'_N\to N)\to (E_M\to M)$ is a regular VB morphism.
\end{itemize}
\end{coroll}
\begin{proof}
From Theorem \ref{theo:caratterizzazioneVBmorita}, $(F,f)$ is a VB-Morita map if and only if $f$ is a Morita map and 
\[
\begin{tikzcd}
0 \arrow[r] & 0 \arrow[r] \arrow[d] & E'_{N,y}\arrow[r] \arrow[d, "F"] & 0\\
0 \arrow[r] & 0 \arrow[r] & E_{M,f(y)}\arrow[r] & 0
\end{tikzcd}
\] 
is a quasi isomorphism for all $y\in N$. The second condition is equivalent to $F_y\colon E'_{N,y}\to E_{M,f(y)}$ being an isomorphism for all $y\in N$, i.e., $F\colon (E'_N\to N)\to (E_M\to M)$ is a regular VB morphism, and $i)$ is equivalent to $iii)$. 

By Lemma \ref{lemma:VBG_morphism_regular} $F\colon(E'\to H)\to (E\to G)$ is a regular VB morphism if and only if $F\colon (E'_N\to N)\to (E_M\to M)$ is so, then $ii)$ is equivalent to $iii)$ as well.
\end{proof}
In particular we get that any LBG morphism $(F,f)$ (Definition \ref{def:LBG}) is a VB-Morita map between LBGs if and only if $f$ is Morita.

Now, we present some straightforward examples of VB-Morita maps involving the tangent VBGs (Example \ref{ex:tangent_VBG}) and the pullback VBG (Example \ref{ex:pullback_VBG}).
\begin{example}[{\cite[Corollary 3.8]{dHO20}}]
\label{ex:df_VB_Morita}
Let $f\colon (H\rightrightarrows N)\to (G\rightrightarrows M)$ be a Morita map. It follows from Theorem \ref{theo:caratterizzazioneVBmorita} that the VBG morphism $(df,f)\colon TH\to TG$ discussed in Example \ref{ex:differentialVBGmorphism} is a VB-Morita map if and only if $f$ is Morita. %Indeed, from Theorem \ref{theo:caratterizzazioneVBmorita}, if $(df,f)$ is a VB-Morita map, then $f$ is a Morita map. Conversly, let $f$ be a Morita map. For any $x\in n$ the cochain maps
%\[
%	\begin{tikzcd}
%		0 \arrow[r] & A_{H,x} \arrow[r, "\rho_H"] \arrow[d, "df"'] & T_xN \arrow[r] \arrow[d, "df"] & 0\\
%		0 \arrow[r] & A_{G,f(x)} \arrow[r, "\rho_G"'] & T_{f(x)}M \arrow[r] & 0
%	\end{tikzcd}
%\]
%is a quasi-isomorphism, where $A_H$ and $A_G$ are the Lie alegbroids of $H$ and $G$, and $\rho_H$ and $\rho_G$ are the anchor-maps of $H$ an $G$ respectively. Indeed, $\ker \rho_H=\mathfrak{h}_x$ is the isotropy Lie algebra at $x$ and $\ker \rho_G= \mathfrak{g}_{f(x)}$  is the isotropy Lie algebra at $f(x)$. The cockernel of $\rho_H$ and $\rho_G$ are $T_xN/\im \rho_H= N_xO'$ and $T_{f(x)}M/\im \rho_G=N_{f(x)}O$ respectively, where $O'$ is the orbit through $x$ and $O$ is the orbit through $f(x)$. Since $f$ is a Morita map, from Theorem \ref{theo:Morita}, $df\colon N_xO'\to N_{f(x)}O$ is an isomorphism and $f \colon H_x \to G_{f(x)}$ is a Lie group diffeomeophism that induces the isomorphism $df\colon \mathfrak{h}_x\to \mathfrak{g}_{f(x)}$. Hence $(F,f)$ is a VB-Morita map by Theorem \ref{theo:caratterizzazioneVBmorita}.
\end{example}

\begin{example}[{\cite[Corollary 3.7]{dHO20}}]
\label{ex:pullbackVBG_VB_Morita}
Let $f\colon (H\rightrightarrows N)\to (G\rightrightarrows M)$ be a Morita map and let $V\rightrightarrows V_M$ be a VBG over $G$. The VBG morphism $(\pr_2, f)\colon f^\ast V\to V$ discussed in Example \ref{ex:pullbackVBGmorphism} is a VB-Morita map. Indeed, by Example \ref{ex:pullback_VBG}, the fiber of $f^\ast V$ over a point $y\in N$ agrees with the fiber of $V$ over $f(y)$, then for any $y\in N$ the cochain map induced by $\pr_2$ is simply the identity.
\[
\begin{tikzcd}
0 \arrow[r] & C_{f(y)}\arrow[r, "t|_V"] \arrow[d, "\operatorname{id}"'] & T_{f(y)}M \arrow[r] \arrow[d, "\operatorname{id}"]& 0 \\
0 \arrow[r] & C_{f(y)}\arrow[r, "t|_V"'] & T_{f(y)}M \arrow[r] & 0
\end{tikzcd}.
\]
Moreover, $f$ is a Morita map, and the result is a consequence of Theorem \ref{theo:caratterizzazioneVBmorita}.
\end{example}

Before discussing the case of the twisted dual VBGs we recall the case of the dual VBGs. In Proposition \ref{prop:dualVBG_morphism} and Remark  \ref{rem:dualF_on_iso} we already discussed how dual VBGs are related. Now, we proceed to examine in detail the situation when the VBG morphisms defined therein are VB-Morita maps.
\begin{prop}
\label{prop:dual_VB-Morita_map}
Let $(F,f)\colon (W\rightrightarrows W_N;H\rightrightarrows N)\to (V\rightrightarrows V_M; G\rightrightarrows M)$ be a VBG morphism and let $F^\ast \colon (f^\ast V^\ast\rightrightarrows f^\ast C^\ast) \to (W^\ast \rightrightarrows D^\ast)$ the VBG morphism (covering $\operatorname{id}_H$) discussed in Proposition \ref{prop:dualVBG_morphism}, with $D$ and $C$ cores of $W$ and $V$ respectively. If $F$ is a VB-Morita map then $F^\ast$ is so.
\end{prop}
\begin{proof}
For any $y\in N$, the cochain map induced by $F^\ast$ on the fibers is
\begin{equation}
\label{eq:dual_cochain_complex}
\begin{tikzcd}
	0 \arrow[r] & V_{M,f(y)}^\ast \arrow[r, "t|_C^\ast"] \arrow[d, "F^\ast"'] & C_{f(y)}^\ast \arrow[r] \arrow[d, "F^\ast"] & 0 \\
	0 \arrow[r] & W_{N,y}^\ast \arrow[r, "t|_D^\ast"'] & D_y^\ast \arrow[r] & 0
\end{tikzcd}.
\end{equation}
Notice that $\ker t|_C^\ast$ is the annihilator of $\im t|_C$ in $V_{M,f(y)}$, that is isomorphic to $(V_{M,f(y)}/\im t|_C))^\ast$, the dual of the cokernel of $t|_C$. Similarly, $\ker t|_D^\ast$ is isomorphic to $(W_{N,y}/\im t|_D)^\ast$, the dual of the cokernel of $t|_D$. Hence
\[
F^\ast\colon \ker t|_C^\ast \to \ker t|_D^\ast
\]
is an isomorphism if and only if its dual
\[
F\colon \frac{W_N,y}{\im t|_D}\to \frac{V_{M,f(y)}}{\im t|_C}
\]
is so. Moreover $C^\ast_{f(y)}/\im t|_C^\ast$ is isomorphic to $(\ker t|_C)^\ast$ and $D^\ast_y/\im t|_D^\ast$ is isomoprhic to $(\ker t|_D)^\ast$ and
\[
F^\ast \colon \frac{C^\ast_{f(y)}}{\im t|_C^\ast} \to \frac{D^\ast_y}{\im t|_D^\ast}
\]
is an isomorphism if and only if its dual map
\[
F\colon \ker t|_D \to \ker t|_C
\]
is so. In other words, the cochain map \eqref{eq:dual_cochain_complex} is a quasi isomorphism if and only if the cochain map induced by $F$ \eqref{eq:fiber_x_VBGM} is so. Since $F$ is a VB-Morita map, then, by Theorem \ref{theo:caratterizzazioneVBmorita}, the cochain map \eqref{eq:fiber_x_VBGM} (and so the cochain map \eqref{eq:dual_cochain_complex}) is a quasi isomorphism for all $y\in N$. The identity $\operatorname{id}_H$ is a Morita map. Hence the statement follows by using again Theorem \ref{theo:caratterizzazioneVBmorita}.
\end{proof}

\begin{rem}
\label{rem:dual_VBMorita_on_iso}
In the case when $F\colon (W\rightrightarrows W_N)\to (V\rightrightarrows V_M)$ is a VBG morphism covering a Lie groupoid isomorphism $f\colon (H\rightrightarrows N)\to (G\rightrightarrows M)$, then we can simply consider the VBG morphism $F^\ast\colon V^\ast\to W^\ast$ covering $f^{-1}$ (see Remark \ref{rem:dualF_on_iso}). Every Lie groupoid isomorphism is a Morita map, then $f$ and $f^{-1}$ are Morita maps. Hence, by the proof of Proposition \ref{prop:dual_VB-Morita_map}, $F^\ast$ is a VB-Morita map if and only if $F$ is so. When $f$ is the identity this result has been proved in \cite[Corollary 3.9]{dHO20}. 

Finally, notice that the fact that $F^\ast \colon V^\ast \to W^\ast $ is a VB-Morita map can also be established through a different argument: we can simply consider the commutative triangle of VBG morphisms
\[
\begin{tikzcd}
& f^\ast V^\ast \arrow[dl] \arrow[dr] \\
W^\ast & & V^\ast \arrow[ll, "F^\ast"]
\end{tikzcd},
\]
where $f^\ast V^\ast \to W^\ast$ is the VB-Morita map discussed in Proposition \ref{prop:dual_VB-Morita_map} and $f^\ast V^\ast\to V^\ast$ is the VB-Morita map discussed in Example \ref{ex:pullbackVBG_VB_Morita}. By the Lemma \ref{lemma:two-out-of-three-VBG}, $F^\ast \colon V^\ast \to W^\ast$ is a VB-Morita map as well.
\end{rem}

In order to discuss the twisted dual VBGs we first consider the case of the tensor product VBG (Example \ref{ex:tensor_product_VBG}).
\begin{prop}
\label{prop:tensor_product_VB-Morita}
Let $(F,f)\colon (W\rightrightarrows W_N;H\rightrightarrows N)\to (V\rightrightarrows V_M; G\rightrightarrows M)$ be a VB-Morita map and let $f_E\colon E'\rightrightarrows E'_N\to E\rightrightarrows E_M$ be a VB-Morita map, covering $f$, between trivial core VBGs $E'$ and $E$. Then The VBG morphism $F\otimes f_E\colon W\otimes E'\to V\otimes E$, defined in Proposition \ref{prop:tensor_product_VBGmorphism}, is a VB-Morita map.
\end{prop}
\begin{proof}
By Theorem \ref{theo:caratterizzazioneVBmorita} $f\colon H\to G$ is a Morita map. Recalling the core complex of the tensor VBG from Example \ref{ex:tensor_product_VBG}, we have that the cochain map determined by $F\otimes f_E$ between the fibers over the point $y\in N$ and $f(y)\in M$ is
\begin{equation}
\label{eq:tensor_cochain_map}
\begin{tikzcd}
	0 \arrow[r] & D_y\otimes E'_{N,y}\arrow[r, "t|_D\otimes \operatorname{id}_{E'_N}"] \arrow[d, "F\otimes f_E"'] & W_{N,y}\otimes E'_{N,y}\arrow[r] \arrow[d, "F\otimes f_E"] & 0\\
	0 \arrow[r] & C_{f(y)}\otimes E_{M,f(y)}\arrow[r, "t|_C\otimes \operatorname{id}_{E_M}"'] & V_{M,f(y)}\otimes E_{M,f(y)}\arrow[r] & 0
\end{tikzcd}.
\end{equation}
In degree $-1$ cohomology we get the map
\[
F_y\otimes f_{E,y} \colon (\ker t|_D)\otimes E'_{N,y} \to (\ker t|_C)\otimes E_{M,y}.
\]
But, for any $y\in N$, by Theorem \ref{theo:caratterizzazioneVBmorita}, the map $F_y\colon \ker t|_D\to \ker t|_C$ is an isomorphism and, by Corollary \ref{cor:VB_Mor_equiv_LBGs}, the map $f_{E,y}\colon E'_{N,y}\to E_{M,f(y)}$ is an isomorphism, then the tensor product $F_y\otimes f_{E,y}$ is an isomorphism as well.

In degree $0$ cohomology we get the map
\[
F_y\otimes f_{E,y} \colon \frac{W_{N,y}}{\im t|_D}\otimes E'_{N,y} \to \frac{V_{M,y}}{\im t|_C}\otimes E_{M,y}.
\]
But, for any $y\in N$, by Theorem \ref{theo:caratterizzazioneVBmorita}, the map $F_y\colon \tfrac{W_{N,y}}{\im t|_D}\to \tfrac{V_{M,y}}{\im t|_C}$ is an isomorphism and, by Corollary \ref{cor:VB_Mor_equiv_LBGs} again, the map $f_{E,y}\colon E'_{N,y}\to E_{M,f(y)}$ is an isomorphism, then the tensor product $F_y\otimes f_{E,y}$ is an isomorphism as well.

Hence the cochain map \eqref{eq:tensor_cochain_map} is a quasi isomorphism for all $y\in N$, and $F\otimes f_E\colon W\otimes E'\to V\otimes E$ is a VB-Morita map by Theorem \ref{theo:caratterizzazioneVBmorita}.
\end{proof}

As a corollary, we obtain the following results on twisted dual VBGs, which are analogous to Proposition \ref{prop:dual_VB-Morita_map} and Remark \ref{rem:dual_VBMorita_on_iso} in the dual VBG case.
\begin{coroll}
\label{coroll:twisted_dual_VB-Morita_map}
Let $(F,f)\colon (W\rightrightarrows W_N;H\rightrightarrows N)\to (V\rightrightarrows V_M; G\rightrightarrows M)$ be a VBG morphism, let $f_E\colon (E'\rightrightarrows E'_N)\to (E\rightrightarrows E_M)$ be a VBG morphism covering $f$ between trivial core VBGs $E'$ and $E$ such that $f_E\colon (E'\to H)\to (E\to M)$ is regular, and let $F^\dag \colon (f^\ast V^\dag\rightrightarrows f^\ast C^\dag) \to (W^\dag \rightrightarrows D^\dag)$ be the VBG morphism (covering $\operatorname{id}_H$) discussed in Remark \ref{rem:twisted_dual_VBGmorphism}, with $D$ and $C$ the cores of $W$ and $V$ respectively. If $F$ is a VB-Morita map then $F^\dag$ is so.
\end{coroll}
\begin{proof}
	Applying Proposition \ref{prop:tensor_product_VB-Morita} to the VBG $F^\ast \colon f^\ast V^\ast \to W^\ast$, which is VB-Morita by Proposition \ref{prop:dual_VB-Morita_map}, and the identity VBG morphism $\operatorname{id}_{E'}\colon E'\to E'$, we have that the VBG morphism $F^\ast\otimes \operatorname{id}_{E'}\colon f^\ast V^\ast\otimes E'\to W^\ast \otimes E'$ (see Corollary \ref{coroll:tensor_product_VBGmorphism}) is a VB-Morita map. By Proposition \ref{prop:pullback_tensor_commute} the VBG $f^\ast V^\ast\otimes E'$ is isomorphic to the VBG $f^\ast(V^\ast\otimes E)=f^\ast V^\dag$, whence the claim.
\end{proof}

\begin{rem}
\label{rem:twisteddual_VBMorita_on_iso}
In the case when $F\colon (W\rightrightarrows W_N)\to (V\rightrightarrows V_M)$ is a VBG morphism covering a Lie groupoid isomorphism $f\colon (H\rightrightarrows N)\to (G\rightrightarrows M)$, then we can simply consider the VBG morphism $F^\dag\colon V^\dag\to W^\dag$ covering $f^{-1}$ obtained as the tensor product of $F^\ast\colon V^\ast \to W^\ast$ (see Remark \ref{rem:dualF_on_iso}) and $(f_E^{-1}, f^{-1})$, as discussed in Remark \ref{rem:twisted_dualF_on_iso}. Every Lie groupoid isomorphism is a Morita map, then $f$ and $f^{-1}$ are Morita maps. Applying Proposition \ref{prop:tensor_product_VB-Morita} to $F^\ast$ and $f_E^{-1}$ we have that $F^\dag$ is a VB-Morita map. 

Notice that the fact that $F^\dag \colon V^\dag \to W^\dag $ is a VB-Morita map can also be established through a different argument: we can simply consider the commutative triangle of VBG morphism
\[
\begin{tikzcd}
& f^\ast V^\dag \arrow[dl] \arrow[dr] \\
W^\dag & & V^\dag \arrow[ll, "F^\dag"]
\end{tikzcd},
\]
where $f^\ast V^\dag \to W^\dag$ is the VB-Morita map discussed in Corollary \ref{coroll:twisted_dual_VB-Morita_map} and $f^\ast V^\dag\to V^\dag$ is the VB-Morita map discussed in Example \ref{ex:pullbackVBG_VB_Morita}. By Lemma \ref{lemma:two-out-of-three-VBG}, $F^\dag \colon V^\dag \to W^\dag$ is a VB-Morita map as well.

The converse is also true when $E'=L'$ and $E=L$ are LBGs. Indeed, let $F\colon W\to V$ be a VBG morphism covering the Lie gropoid isomorphism $f\colon H\to G$ and let $f_L\colon L'\to L$ be an LBG morphism covering $f$. If $F^\dag =F^\ast \otimes f_L^{-1}\colon V^\ast \otimes L \to W^\ast\otimes L'$ is a VB-Morita map covering $f^{-1}$, then, applying Proposition \ref{prop:tensor_product_VB-Morita} to $(F^\dag)^\ast\colon W\otimes L'^\ast \to V\otimes L^\ast$ (that is a VB-Morita map because it is dual of a VB-Morita map) and $f_L\colon L'\to L$ we have that $F=(F^\dag)^\ast\otimes f_L\colon W\to V$ is a VB-Morita map. The case when $f=\operatorname{id}_G$ and $L=L'$ has been discussed in \cite[Proposition 3.13]{MTV24}.
\end{rem}

Now we recall the equivalence generated by VB-Morita maps and prove that it indeed constitutes an equivalence relation. The proof closely parallels the case of Lie groupoids, with the homotopy fiber product VBG (Example \ref{ex:homotopy_pullback_VBG}) replacing the homotopy fiber product.
\begin{definition}
Two VBGs $V$ and $V'$ are \emph{VB-Morita equivalent} if there exist a VBG $W$ and two VB-Morita maps $V\leftarrow W \rightarrow V'$.
\end{definition}

\begin{prop}
\label{prop:VBequiv_relation}
VB-Morita equivalence between VBGs is an equivalence relation.
\end{prop}
\begin{proof}
The proof is similar to the proof of Proposition \ref{prop:Morita_equivalence_relation}. Reflexivity and simmetry are obvious. For the transitivity we need the homotopy fiber product VBG construction (see Example \ref{ex:homotopy_pullback_VBG}). If $V_1\leftarrow W_1 \to V_2$, $V_2 \leftarrow W_2 \to V_3$ are VB-Morita equivalences, then we consider the homotopy fiber product VBG $W_1\times_{V_2}^h W_2$ of $F_1\colon W_1\to V_2$ and $F_2 \colon W_2\to V_2$. If $F_1$ (respectively $F_2$) is a VB-Morita map, then, by \cite[Proposition 5.12]{MM03}, the homotopy fiber product exists and the projection $W_1\times^h_{V_2} W_2 \to W_2$ (respectively $W_1\times^h_{V_2} W_2 \to W_1$) is a VB-Morita map. Hence every VBG morphism in the diagram
\begin{equation*}
\scriptsize
\begin{tikzcd}%[column sep=-3]
	& & W_1\times^h_{V_2} W_2 \arrow[dr] \arrow[dl] \\
	& W_1 \arrow[dr] \arrow[dl] & & W_2 \arrow[dr]\arrow[dl] \\
	V_1 & & V_2 & & V_3
\end{tikzcd}
\end{equation*}
is a VB-Morita map and $V_1$ and $V_3$ are VB-Morita equivalent as well.
\end{proof}

\begin{definition}
A \emph{VB-stack} is a VB-Morita equivalence class of VBGs. We indicate by $[V_M/V]\to [M/G]$ the VB-stack consisting of the VBGs VB-Morita equivalent to $(V\rightrightarrows V_M; G\rightrightarrows M)$. In this case we also say that $V$ is a presentation of $[V_M/V]$. Finally, we call \emph{LB-stack} the VB-stack $[L_M/L]\to [M/G]$ presented by an LBG $(L\rightrightarrows L_M;G\rightrightarrows M)$.
\end{definition}

A VB-stack should be considered, from our point of view, as a vector bundle in the category of differentiable stacks, and an LB-stack as a line bundle in that category. Note that an LB-stack $[L_M/L] \to [M/G]$, presented by an LBG $(L \rightrightarrows L_M; G \rightrightarrows M)$, does not exclusively consist of LBGs. Indeed, a VBG $V$ VB-Morita equivalent to $L$ is not necessarily an LBG. However, as we demonstrate in the next result, such a VBG is always (non-canonically) isomorphic to the direct sum (Example \ref{ex:direct_sum_VBG}) of an LBG and another VBG whose core anchor is a VB isomorphism. Moreover, the latter does not play any role in the theory, as its core-anchor has trivial cohomology.
\begin{prop}
\label{prop:Morita_map_to _LBG}
Let $(F,f)$ be a VB-Morita map from a VBG $(V\rightrightarrows V_N;H\rightrightarrows N)$ to an LBG $(L\rightrightarrows L_M; G\rightrightarrows M)$. Then $V$ is isomorphic to the direct sum of an LBG and a VBG whose core-anchor is a VB isomorphism.
\end{prop}
\begin{proof}
	 Since $(F,f)$ is a VB-Morita map between $V$ and $L$, then $f$ is a Morita map between $H$ and $G$ and the core complexes
	\begin{equation*}
		\begin{tikzcd}
			0 \arrow[r] & 0 \arrow[r] & L_M \arrow[r] & 0, &
			0 \arrow[r] & C \arrow[r, "t|_C"] & V_M \arrow[r] & 0
		\end{tikzcd}
	\end{equation*}
	are quasi-isomorphic, where $C$ is the core of $V$. This means that $t|_C$ is injective and $\rank(V_M)= \rank(C) +1$. Then $V$ is a regular VBG and, following Remark \ref{rem:regularVBG}, $V$ is (non-canonically) isomorphic to the direct sum of a VBG with core-anchor
	\[
	\begin{tikzcd}
		0 \arrow[r] & C \arrow[r, "t|_C"] & \im t|_C \arrow[r] & 0,
	\end{tikzcd}
	\] 
	and a VBG with core-anchor
	\begin{equation}
		\label{eq:core_anchor}
	\begin{tikzcd}
		0 \arrow[r] & 0 \arrow[r] & \nu \arrow[r] & 0,
	\end{tikzcd}
	\end{equation} 
	where $\nu$ is the cokernel of $t|_C\colon C\to V_M$. The VBG whose core-anchor is \eqref{eq:core_anchor} is an LBG.
\end{proof}

As a consequence of Proposition \ref{prop:Morita_map_to _LBG} we get the next two results. These clarify that any VB-Morita map from a VBG to an LBG can be replaced by an LBG morphism. Consequently, the VB-Morita equivalence between LBGs can be studied exclusively by considering LBGs.
\begin{lemma}
Let $(L\rightrightarrows L_M; G \rightrightarrows M)$ be an LBG, let $(V \rightrightarrows V_N; H \rightrightarrows N)$ be a VBG and let $(F, f) \colon V \to L$ be a VB-Morita map. Then there exists an LBG $L'\to H$ and a VB-Morita map $L' \to V$.
\end{lemma}

\begin{proof}
From Proposition \ref{prop:Morita_map_to _LBG}, the VBG $V$ is isomorphic to the direct sum of an LBG $L' \to H$ and a VBG $W \to H$ whose core-anchor is bijective. The embedding $j\colon L' \to V$ is clearly a VB-Morita map. Then the composition $F\circ j\colon L'\to L$ is a VB-Morita map.
\end{proof}

\begin{coroll}\label{cor:VB_Mor_equiv_LBGs}
Let $L_1, L_2$ be VB-Morita equivalent LBGs. Then the VB-Morita equivalence can be realized via an LBG, i.e., there exists an LBG $L'$ togheter with VB-Morita maps
\[
\begin{tikzcd}
& L' \arrow[dl] \arrow[dr] \\
L_1 & & L_2
\end{tikzcd}.
\]
\end{coroll}
\begin{proof}
	Let $L_1 \xleftarrow{F_1} V \xrightarrow{F_2} L_2$ be a VB-Morita equivalence between $L_1$ and $L_2$. The proof follows by noting that the direct sum decomposition of $V$, as described in Proposition \ref{prop:Morita_map_to _LBG}, is independent of the VB-Morita maps $F_1$ and $F_2$.
\end{proof}

Before discussing VB-Morita equivalence between twisted dual VBGs of Morita equivalent VBG we prove the following result on VB-Morita equivalence between Atiyah VBGs. This statement is analogous to Example \ref{ex:df_VB_Morita}.
\begin{prop}
\label{prop:DFMorita}
Let $(F,f)\colon (L'\rightrightarrows L'_N;H\rightrightarrows N)\to (L\rightrightarrows L_M; G\rightrightarrows M)$ be an LBG morphism. Then the VBG morphism $(DF,f)\colon (DL'\rightrightarrows DL'_N, H\rightrightarrows N)\to (DL\rightrightarrows DL_M, G\rightrightarrows M)$ is a VB-Morita map if and only if $f\colon (H\rightrightarrows N)\to (G\rightrightarrows M)$ is a Morita map.
\end{prop}
\begin{proof}
If $(DF, f)$ is a VB-Morita map, then $f$ is a Morita map. For the converse, suppose that $f$ is a Morita map and let $y \in N$. The map induced by $DF$ between the fibers of $DL'$ and $DL$ over $y$ and $f(y)$ is
\begin{equation}\label{eq:DF_fiber}
\begin{tikzcd}
	0 \arrow[r] &A_{H,y} \arrow[r, "\mathcal D_H"] \arrow[d, "D F"'] & D_y L'_N \arrow[r] \arrow[d, "DF"]&0 \\
	0 \arrow[r]& A_{G,f(y)} \arrow[r, "\mathcal D_G", swap] &D_{f(y)} L_M \arrow[r] &0
\end{tikzcd},
\end{equation}
where $A_H,A_G$ are the Lie algebroids, and $\mathcal{D}_H$, $\mathcal{D}_G$ the core-anchor of $H$ and $G$ respectively. But \eqref{eq:DF_fiber} fits in the following commutative diagram of $\mathbbm{R}$-linear maps where the horizontal lines are short exact sequence
\begin{equation}
{\scriptsize
	\begin{tikzcd}[column sep={{{{2.5em,between origins}}}},% column sep=1ex,
		row sep=1ex]
		&&&&&&&0\arrow[ddr]&&&&0\arrow[ddr]&&&&&&\\%12
		&\phantom{x}&&&&&&&&&&&&&&&&\\%11
		0\arrow[rrrr]&&&&0\arrow[ddr]\arrow[rrrr]\arrow[ddddd]&&&&\smash[b]{A_{H,y}}\arrow[ddr, "\mathcal{D}_H"]\arrow[rrrr, equal]\arrow[ddddd, "DF"' near end, swap]&&&&\smash[b]{A_{H,y}}\arrow[ddr, "\rho_H"]\arrow[rrrr]\arrow[ddddd, "df" near end]&&&&0&\\%10
		&\phantom{x}&&\phantom{x}&\phantom{x}&\phantom{x}&&&&&&&&&&&&\\%9
		&0\arrow[rrrr, crossing over]&&&&\mathbbm{R}\arrow[rrrr, crossing over]&&&&\smash[b]{D_yL_N'}\arrow[ddr]\arrow[rrrr, crossing over, "\sigma"]&&&&\smash[b]{T_yN}\arrow[ddr]\arrow[rrrr]&&&&0\\%8
		&&&&&&&0\arrow[ddr]&&&&0\arrow[ddr]&&&&&&\\%7
		&&&&&&\phantom{x}&&&&0&&&&0&&&\\%6
		0\arrow[rrrr]&&&&0\arrow[ddr]\arrow[rrrr]&&&&\smash[b]{A_{G,f(y)}}\arrow[ddr, "\mathcal{D}_G", swap]\arrow[rrrr, equal]&&&&\smash[b]{A_{G,f(y)}}\arrow[ddr, "\rho_G", swap]\arrow[rrrr]&&&&0&\\%5
		&&\phantom{x}&&&&&&&&&&&&&&&\\%4
		&0\arrow[rrrr]&&&&\mathbbm{R}\arrow[rrrr]\arrow[from=uuuuu, crossing over, equal]&&&&\smash[b]{D_{f(y)}L_M}\arrow[ddr]\arrow[rrrr, "\sigma", swap]\arrow[from=uuuuu, crossing over, "DF"', swap]&&&&\smash[b]{T_{f(y)}M}\arrow[ddr]\arrow[rrrr]\arrow[from=uuuuu, crossing over, "df"]&&&&0\\%3
		&&&&\phantom{x}&&&&&&&&&&&&&\\%2
		&&&&&&&&&&0&\phantom{x}&&&0&&&\\%1
\end{tikzcd}}.
\end{equation}
Since $f$ is Morita, the vertical arrows in (both the leftmost and) the rightmost square form a quasi-isomorphism (Example \ref{ex:df_VB_Morita}). It follows that \eqref{eq:DF_fiber} is a quasi-isomorphism as well, hence $(DF, f)$ is a VB-Morita map by Theorem \ref{theo:caratterizzazioneVBmorita}.
\end{proof}
In particular, by Corollary \ref{coroll:VBMorita} and Proposition \ref{prop:DFMorita}, we get that two LBGs are VB-Morita equivalent if and only if so are the induced Atiyah VBGs.

Now we are ready to discuss the relation between Morita equivalence and the construction of the twisted dual VBG (see Example \ref{ex:twisted_dual}). First, we recall the case of the dual VBGs (see Example \ref{ex:dual_VBG}). 
\begin{rem}
\label{rem:dualVBG_VB-Morita_equivalent}
If two VBGs are Morita equivalent, then their dual VBGs are also Morita equivalent. Indeed, if $(F,f)\colon (W\rightrightarrows W_N;H\rightrightarrows N)\to (V\rightrightarrows V_M;G\rightrightarrows M)$ is a VB-Morita map, and $C$ is the core of $V$, we can consider the pullback VBG $f^{\ast}V^\ast\rightrightarrows f^{\ast}C^{\ast}$ of the dual VBG $V^\ast$ along $f$ (see Example \ref{ex:pullback_VBG}), and the VBG morphism $f^\ast V^\ast \to V^\ast$ is a VB-Morita map (see Example \ref{ex:pullbackVBG_VB_Morita}). On the other hand, since $F$ is a VB-Morita map, the VBG morphism $F^\ast\colon f^\ast V^\ast \to W^\ast$, defined by
\[
\langle F^\ast(h,\psi),w \rangle= \langle \psi, F(w)\rangle, \quad \psi \in V^\ast_{f(h)}, w\in W_h, h\in H,
\]
is also a VB-Morita map (see Proposition \ref{prop:dual_VB-Morita_map}).  

Hence, the dual VBGs $W^{\ast}$ and $V^{\ast}$ are related by VB-Morita maps
\begin{equation*}
\begin{tikzcd}
	& f^{\ast}V^\ast \arrow[dl] \arrow[dr] \\
	W^{\ast} & & V^{\ast}
\end{tikzcd},
\end{equation*}
and they are VB-Morita equivalent.

In the case when $f$ is a Lie groupoid isomorphism, we can simply consider the VBG morphism $(F^\ast, f^{-1})\colon (V^\ast\rightrightarrows C^\ast; G\rightrightarrows M)\to (W^\ast\rightrightarrows D^\ast; H\rightrightarrows N)$ defined in Remark \ref{rem:dualF_on_iso}. By Remark \ref{rem:dual_VBMorita_on_iso}, the latter VBG morphism is a VB-Morita map.
\end{rem}

Now, let $(W\rightrightarrows W_N; H\rightrightarrows N), (V\rightrightarrows V_M; G\rightrightarrows M)$ be Morita equivalent VBGs and let $(E'\rightrightarrows E'_N; H\rightrightarrows N), (E\rightrightarrows E_M; G\rightrightarrows M)$ be Morita equivalent trivial core VBGs. First we focus on the tensor product VBGs. The tensor product VBGs $W\otimes E'$ and $V\otimes E$ are VB-Morita equivalent as well. To see this it is enough to consider the case when $W, V$ and $E', E$ are related by a VB-Morita map. Let $F\colon W\to V$ be a VB-Morita map covering $f\colon H\to G$ and let $f_E\colon E'\to E$ be a VB-Morita map covering $f\colon N\to M$. Then the tensor product VBGs $W\otimes E'$ and $V\otimes E$ are related by the VB-Morita map $F\otimes f_E$ (see Proposition \ref{prop:tensor_product_VB-Morita}). 

Finally, the twisted dual VBGs (see Example \ref{ex:twisted_dual}) $W^\dag=W^\ast \otimes E'$ and $V^\dag=V^\ast \otimes E$ are also Morita equivalent as well. Again it is enough to assume that $W, V$ and $E', E$ are related by a VB-Morita map and this case is discussed in the following 

\begin{rem}
\label{rem:twisted_dualVBG_VB-Morita_equivalent}
Let $W,V$, $E', E$ be as above, and let $(F,f)\colon (W\rightrightarrows W_N; H \rightrightarrows N)\to (V\rightrightarrows V_M; G\rightrightarrows M)$ and $(f_E,f)\colon (E'\rightrightarrows E'_N; H \rightrightarrows N)\to (E\rightrightarrows E_M; G\rightrightarrows M)$ be VB-Morita maps. Then the twisted dual VBGs $W^{\dag}, V^\dag$ are Morita equivalent. Indeed, let $C$ be the core of $V$, we can consider the pullback VBG $f^{\ast}V^\dag\rightrightarrows f^{\ast}C^{\dag}$ of the twisted dual VBG $V^\dag$ along $f$ (see Example \ref{ex:pullback_VBG}), and the VBG morphism $f^\ast V^\dag \to V^\dag$ is a VB-Morita map (see Example \ref{ex:pullbackVBG_VB_Morita}). On the other hand, since $F$ is a VB-Morita map, the VBG morphism $F^\dag\colon f^\ast V^\dag \to W^\dag$ is a VB-Morita map (see Corollary \ref{coroll:twisted_dual_VB-Morita_map}).  

Hence, the twisted dual VBGs $W^{\dag}$ and $V^{\dag}$ are related by VB-Morita maps
\begin{equation*}
\begin{tikzcd}
	& f^{\ast}V^\dag \arrow[dl] \arrow[dr] \\
	W^{\dag} & & V^{\dag}
\end{tikzcd},
\end{equation*}
and they are VB-Morita equivalent.

In the case when $f$ is a Lie groupoid isomorphism, we can simply consider the VBG morphism $(F^\dag, f^{-1})\colon (V^\dag\rightrightarrows C^\dag; G\rightrightarrows M)\to (W^\dag\rightrightarrows D^\dag; H\rightrightarrows N)$ defined in Remark \ref{rem:twisted_dualF_on_iso}. By Remark \ref{rem:twisteddual_VBMorita_on_iso}, the latter VBG morphism is a VB-Morita map.
\end{rem}

An easy consequence is that, if two LBGs $L$ and $L'$ are VB-Morita equivalent if and only if the VBGs $J^1L$ and $J^1L'$ are VB-Morita equivalent as well.

\section{Linear natural isomorphisms}\label{sec:lni}
We conclude this chapter discussing natural isomorphisms in the VBG setting (see \cite[Section 6.1]{dHO20} for the special case of VBG morphisms covering the identity). We introduce the notion of natural transformation between VBG morphisms and the concept of \emph{homotopic} VBG morphisms (see \cite[Section 4.1]{BCGX22} for the special case of VBG morphisms covering the identity). Finally, we relate these notions in a Theorem \ref{theo:VBtransformation}. 

Let $(F,f), (F',f')\colon (W\rightrightarrows W_N; H\rightrightarrows N)\to (V\rightrightarrows V_M; G\rightrightarrows M)$ be VBG morphisms. 
\begin{definition}
\label{def:LNT}
A \emph{linear natural isomorphism} $(T, \tau) \colon (F,f) \Rightarrow (F',f')$ from $(F,f)$ to $(F',f')$ is a VB morphism $(T, \tau) \colon (W_N \to N) \to (V \to G)$ such that $T \colon F \Rightarrow F'$ is a natural isomorphism.
\end{definition}
Notice that if $(T,\tau)\colon (F,f) \Rightarrow (F',f')$ is a linear natural isomorphism, then $\tau\colon f\Rightarrow f'$ is a natural isomorphism as well, and we say that $T$ \emph{covers} $\tau$.

\begin{example}
	\label{ex:lni_homotopy}
	Let $(F,f)\colon (W\to H)\to (V\to G)$ and $(F',f')\colon (W'\to H')\to (V\to G)$ be VBG morphisms and suppose that the homotopy fiber product $W\times^h_V W'$ between $F$ and $F'$ exists. Then, by Example \ref{ex:homotopy_pullback_VBG}, $W\times^h_V W'$ is a VBG over the homotopy fiber product $H\times^h_G H'$ between $f$ and $f'$. Moreover $W\times^h_V W'$ comes with two VBG morphisms $\pr_1\colon W\times^h_V W' \to W$ and $\pr_2\colon W\times^h_V W'\to W'$ and, by Example \ref{ex:homotopy_pullback}, there exists a natural isomorphism $T$ between $F\circ \pr_1$ and $F'\circ \pr_1$. Actually $T$ is a VB morphism, whence it is a linear natural isomorphism.
\end{example}

Now we introduce the notion of homopic VBG morphisms covering the same map. The case of VBG morphisms covering the identity has already been discussed in \cite[Section 4.1]{BCGX22}. Let $(W\rightrightarrows W_N; H\rightrightarrows N)$ and $(V\rightrightarrows V_M; G\rightrightarrows M)$ be VBGs and let $C$ be the core of $V$. We start with a technical result.
\begin{lemma}
Let $\mathcal H\colon W_N\to C$ be a VB morphism covering $f\colon N\to M$. Then, for any $\omega\in W_h$, with $h\in H$, 
\begin{equation}
\label{eq:homotopy_morphism}
0_{f(h)}^V\cdot \mathcal H\big(s(\omega)\big)^{-1} + \mathcal H\big(t(\omega)\big)\cdot 0_{f(h)}^V = \mathcal H\big(t(\omega)\big) \cdot 0_{f(h)}^V\cdot \mathcal H\big(s(\omega)\big)^{-1}.
\end{equation}
\end{lemma}
\begin{proof}
Let $\omega\in W_h$, with $h\colon x\to y\in H$. Then, we have
\begin{align*}
0_{f(h)}^V\cdot \mathcal H\big(s(\omega)\big)^{-1} &+ \mathcal H\big(t(\omega)\big)\cdot 0_{f(h)}^V\\
&= 0_{f(h)}^V\cdot \mathcal H\big(s(\omega)\big)^{-1} + \mathcal H\big(t(\omega)\big)\cdot 0_{f(h)}^V\cdot 0_{f(x)}^V \\
&= \left(0_{f(h)}^V + \mathcal H\big(t(\omega)\big)\cdot 0_{f(h)}^V\right)\cdot \left(\mathcal H\big(s(\omega)\big)^{-1} + 0_{f(x)}^V\right)\\
&= \mathcal H\big(t(\omega)\big)\cdot 0_{f(h)}^V \cdot \mathcal H\big(s(\omega)\big)^{-1},
\end{align*}
where we used the interchange law \eqref{eq:interchange_law}.
\end{proof}

Starting from a VB morphism $(\mathcal{H}, f)\colon (W_N\to H)\to (C\to M)$ we can define a VBG morphism as discussed in the following
\begin{prop}
Let $\mathcal H\colon W_N\to C$ be a VB morphism covering $f\colon N\to M$. The VB morphism $J_{\mathcal{H}}\colon W\to V$ defined by setting
\[
J_{\mathcal{H}}(\omega)= 0_{f(h)}^V\cdot \mathcal H\big(s(\omega)\big)^{-1} + \mathcal H\big(t(\omega)\big)\cdot 0_{f(h)}^V, \quad \omega\in W_h,
\] 
with $h\in H$, and the VB morphism $J_{\mathcal{H}}\colon W_N\to V_M$ defined by setting
\[
J_{\mathcal{H}}(w)= t\big(\mathcal H(w)\big), \quad w\in W_{N,y},
\] 
with $y\in N$, form a VBG morphism from $W$ to $V$ covering $f$.
\end{prop}
\begin{proof}
For any $\omega\in W_h$, with $h\colon x\to y\in H$, we have
\begin{align*}
s\big(J_{\mathcal H}(\omega)\big)= s\left(0_{f(h)}^V\cdot \mathcal H\big(s(\omega)\big)^{-1} + \mathcal H\big(t(\omega)\big)\cdot 0_{f(h)}^V\right) =t\left(\mathcal H\big(s(\omega)\big)\right) = J_{\mathcal H}\big(s(\omega)\big), 
\end{align*}
and
\begin{align*}
t\big(J_{\mathcal H}(\omega)\big)= t\left(0_{f(h)}^V\cdot \mathcal H\big(s(\omega)\big)^{-1} + \mathcal H\big(t(\omega)\big)\cdot 0_{f(h)}^V\right) = t\left(\mathcal H\big(t(\omega)\big)\right)= J_{\mathcal H}\big(t(\omega)\big).
\end{align*}
For any $w\in W_{N,y}$, with $y\in N$, we have
\begin{align*}
u\big(J_{\mathcal H}(w)\big)&= u\left(t\big(\mathcal H(w)\big)\right)= u\left(t\big(\mathcal H(w)\big)\right)= \mathcal H(w)\cdot  \mathcal H(w)^{-1}\\
&=\mathcal H(w)\cdot 0_{f(x)}^V \cdot \mathcal H(w)^{-1}=J_{\mathcal H}\big(u(w)\big),
\end{align*}
where we used Equation \eqref{eq:homotopy_morphism}.
For any $(\omega, \omega')\in W^{(2)}_{(h,h')}$, with $(h,h')\in H^{(2)}$, we have
\begin{align*}
J_{\mathcal H}(\omega)&\cdot J_{\mathcal H}(\omega')\\
&= \mathcal H\big(t(\omega)\big)\cdot 0_{f(h)}^V \cdot \mathcal H\big(s(\omega)\big)^{-1}\cdot \mathcal H\big(t(\omega')\big)\cdot 0_{f(h')}^V \cdot \mathcal H\big(s(\omega')\big)^{-1} \\
&=\mathcal H\big(t(\omega)\big)\cdot 0_{f(h)}^V \cdot 0_{f(h')}^V \cdot \mathcal H\big(s(\omega')\big)^{-1}\\
&= \mathcal H\big(t(\omega\cdot \omega')\big)\cdot 0_{f(hh')}^V \cdot \mathcal H\big(s(\omega\cdot \omega')\big)^{-1}\\
&=J_{\mathcal H}(\omega\cdot \omega').\qedhere
\end{align*}
%Finally, for any $\omega\in W_h$, with $h\in H$, we have
%\begin{align*}
%J_{\mathcal H}(\omega)^{-1}&= \left(\mathcal H\big(t(\omega)\big)\cdot 0_{f(h)}^V \cdot \mathcal H\big(s(\omega)\big)^{-1}\right)^{-1}\\
%&= \mathcal H\big(t(\omega^{-1})\big)\cdot 0_{f(h^{-1})}^V \cdot \mathcal H\big(s(\omega^{-1})\big)^{-1}\\
%&= J_{\mathcal H}\big(\omega^{-1}\big).\qedhere
%\end{align*}
\end{proof}

\begin{rem}
	Let $W$ and $V$ be VBGs over $G$. VBG morphisms from $W$ to $V$ are $1$-cochains in an appropriate cochain complex associated to $W$ and $V$. Moreover, VB morphisms $\mathcal{H}$ from $W_M$ to the core $C$ of $V$ are $0$-cochains in that complex and the VBG morphism $J_\mathcal H$ is, up to a sign, the $1$-coboundary obtained applying the differential to $\mathcal{H}$. In the case when $W$ and $V$ are VBGs over $H$ and $G$ respectively, and $f\colon H\to G$ is a Lie groupoid morphism, then we can simply consider the pullback VBG $f^\ast V$. Moreover, VBG morphisms $W\to V$ covering $f$ are equivalent to VBG morphisms $W\to f^\ast V$ covering $\operatorname{id}_H$, and we have the analogous discussion. See \cite{MOV24} for more details.
\end{rem}

Let $F,F'\colon W\to V$ be VBG morphisms covering the same map $f$ and let $\mathcal H\colon W_N\to C$ be a VB morphism covering $f$, where $C$ is the core of $V$.
\begin{definition}
The VBG morphism $F'$ is \emph{homotopic} to $F$ through $\mathcal H$ if
\[
F' - F= J_{\mathcal H}.
\]
\end{definition}
\begin{comment}
The terminology is motivated by the following
\begin{prop}
\label{prop:homotopy}
Let $F,F'\colon W\to V$ be VBG morphisms. If $F'$ is homotopic to $F$ through $\mathcal H\colon W_N\to C$, with $C$ the core of $V$, then, for any $y\in N$, $\mathcal H_y\colon W_{N,x}\to C_{f(y)}$ is a well-defined homotopy between the cochains map $F,F'$ from the fiber of $W$ at $y\in N$ to the fiber of $V$ at $f(y)\in M$:
\begin{equation*}
	\begin{tikzcd}
		0 \arrow[r] &D_y \arrow[d, "F", shift left=0.5ex] \arrow[d, "F'"', shift right=0.5ex]\arrow[r]& W_{N,y} \arrow[d, "F", shift left=0.5ex] \arrow[d,"F'"', shift right=0.5ex] \arrow[r] \arrow[dl, "\mathcal{H}_y"'] &0\\
		0 \arrow[r] &C_{f(y)}\arrow[r] &V_{M,{f(y)}} \arrow[r]&0
	\end{tikzcd}.
\end{equation*}
\end{prop}
\begin{proof}
For the homotopy condition $\mathcal{H}\circ t|_D = F' - F$ in degree $-1$, take $d \in D_y$, then
\begin{align*}
	(F'-F)(d)= J_{\mathcal H}(d) = 0_{f(y)}^V \cdot \left(\mathcal H(s(d))^{-1} + \mathcal H(t(d))\right) \cdot 0_{f(y)}^V = \mathcal H(t(d)).
\end{align*}
For the homotopy condition $t|_C \circ \mathcal{H}= F'-F$ in degree $0$, take $w\in W_{N,y}$, then
\begin{align*}
	(F'-F)(w)= J_{\mathcal H}(w)= 0_{f(y)}^V \cdot (\mathcal H(w)^{-1} + \mathcal H(w)) \cdot 0_{f(y)}^V,
\end{align*}
and, applying the target map we get
\begin{equation*}
	(F'-F)(w)= t\left(0_{f(y)}^V \cdot (\mathcal H(w)^{-1} + \mathcal H(w)) \cdot 0_{f(y)}^V\right)= t(\mathcal H(w)).
\end{equation*}
\end{proof}
\end{comment}
There is a useful characterization of linear natural isomorphisms in the case when the VBGs cover the same Lie groupoid morphism.
\begin{theo}
\label{theo:VBtransformation}
Let $F, F'\colon W\to V$ be VBG morphisms covering the same map $f\colon H\to G$. Let $(T,f)\colon (W_N\to N)\to (V\to G)$ be a VB morphism and let $(\mathcal{H}_T,f) \colon (W_N\to N) \to (C\to M)$ be the VB morphism defined by setting $\mathcal{H}_T(w)= T(w)-F(w)$, for all $w\in W_N$, where $C$ is the core of $V$. The following conditions are equivalent:
\begin{itemize}
\item[i)] $(T,f)\colon (F,f)\Rightarrow (F',f)$ is a linear natural isomorphism;
\item[ii)] $F'$ is homotopic to $F$ through $\mathcal{H}_T$;
\item[iii)] for any $y\in N$, $\mathcal{H}_T\colon W_{N,y}\to C_{f(y)}$ is a well-defined homotopy between the cochain maps on the fibers over $y, f(y)$ induced by $F, F'$:
\begin{equation*}
	\begin{tikzcd}
		0 \arrow[r] &D_y \arrow[d, "F", shift left=0.5ex] \arrow[d, "F'"', shift right=0.5ex]\arrow[r]& W_{N,y} \arrow[d, "F", shift left=0.5ex] \arrow[d,"F'"', shift right=0.5ex] \arrow[r] \arrow[dl, "\mathcal{H}_T"'] &0\\
		0 \arrow[r] &C_{f(y)}\arrow[r] &V_{M,{f(y)}} \arrow[r]&0
	\end{tikzcd},
\end{equation*}
and additionally, for any $\omega\in W_h$, with $h\in H$, the following condition is satisfied
\begin{equation}
	\label{eq:point_3}
(F'-F)(\omega)\cdot \mathcal{H}_T\big(s(\omega)\big)= \mathcal H_T\big(t(\omega)\big)\cdot 0_{f(h)}^V.
\end{equation}
\end{itemize}
Moreover the assignment $(T, f) \mapsto \mathcal{H}_T$ establishes a bijection between linear natural isomorphisms $(T, f)$ from $F$ to $F'$ and VB morphism from $W_N$ to $C$ covering $f$ that make $F'$ homotopic to $F$.
\end{theo}
\begin{proof}
Let $(T, f)$ be a linear natural isomorphism as in the statement. For every $w \in W_N$, $T(w)$ is an arrow in $V$ from $F(w)$ to $F'(w)$, $T(w)\colon F(w) \to F'(w)$. The map $\mathcal{H}_T$ in the statement is well-defined, indeed, for any $w \in W_N$, $s (\mathcal{H}_T (w)) = s (F(w)) - s(T(w)) = F(w) - F(w) = 0$, i.e., $\mathcal{H}_T(w) \in C$. For any $\omega\colon w \to w' \in W_h$, with $h\in H$, from the naturality of $T$ we get $F'(\omega)=T(w')F(\omega)T(w)^{-1}$, and then
\begin{align*}
(F'-F)(\omega) &= T(w')F(\omega) T(w)^{-1} - F(\omega)\\
&= \big(T(w')- F(w')\big)\cdot \big(F(\omega) T(w)^{-1} - F(\omega)\big)\\
&= \big(T(w')-F(w')\big)\cdot \big(F(\omega)-F(\omega)\big)\cdot \big(T(w)^{-1}-F(w)\big)\\
&= \mathcal H_T(w')\cdot 0_{f(h)}^V\cdot \mathcal H_T(w)^{-1}\\
&= J_{\mathcal H_T}(\omega ),
\end{align*}
where we used two times the interchange law \eqref{eq:interchange_law} and Equation \eqref{eq:homotopy_morphism}. Hence $i)$ implies $ii)$.

Next we prove that $ii)$ implies $iii)$. Let $\mathcal{H}\colon W_N \to C$ be a VB morphism covering $f$ such that $F'$ is homotopic to $F$ through $\mathcal{H}$. For the homotopy condition $\mathcal{H}\circ t|_D = F' - F$ in degree $-1$, take $d \in D_y$, then
\begin{align*}
(F'-F)(d)= J_{\mathcal H}(d) = 0_{f(y)}^V \cdot \left(\mathcal H\big(s(d)\big)^{-1} + \mathcal H\big(t(d)\big)\right) \cdot 0_{f(y)}^V = \mathcal H\big(t(d)\big).
\end{align*}
For the homotopy condition $t|_C \circ \mathcal{H}= F'-F$ in degree $0$, take $w\in W_{N,y}$, then
\begin{align*}
(F'-F)(w)= J_{\mathcal H}(w)= 0_{f(y)}^V \cdot \left(\mathcal H(w)^{-1} + \mathcal H(w)\right) \cdot 0_{f(y)}^V,
\end{align*}
and, applying the target map, we get
\begin{equation*}
(F'-F)(w)= t\left(0_{f(y)}^V \cdot \left(\mathcal H(w)^{-1} + \mathcal H(w)\right) \cdot 0_{f(y)}^V\right)= t(\mathcal H(w)).
\end{equation*}
Finally, for any $\omega\in W_h$, from Equation \eqref{eq:homotopy_morphism}, we have
\[
(F'-F)(\omega)\mathcal{H}\big(s(\omega)\big)= J_{\mathcal{H}}(\omega)\mathcal{H}\big(s(\omega)\big)= \mathcal{H}\big(t(\omega)\big)\cdot 0_{f(h)}^V.
\]

Now let $\mathcal{H}\colon W_N\to C$ be a VB morphism covering $f$ such that, for any $y\in N$, $H\colon W_{N,y}\to C_{f(y)}$ is a homotopy between the cochain maps on the fiber over $y$ and $f(y)$ induced by $F$ and $F'$. We simply define $T\colon W_N \to V$ by setting
\[
T(w) = \mathcal{H}(w) + F(w), \quad w\in W_{N,x}.
\]
The map $T$ is a VB morphism covering $f\colon N\to G$ and it is a natural transformation between $W\rightrightarrows W_N$ and $V\rightrightarrows V_M$. Indeed, for any $w\in W_{N,x}$ we have
\[
s\big(T(w)\big) = s\big(\mathcal{H}(w)\big)+ s\big(F(w)\big) = F(w),
\]
where we used that $\mathcal{H}(w)\in C_{f(x)}$, and
\[
t\big(T(w)\big)= t\big(\mathcal{H}(w)\big) + t\big(F(w)\big)= F'(w) - F(w) + F(w)= F'(w),
\]
because of the homotopy condition for $\mathcal{H}$ in degree $0$. Moreover, for any $\omega \colon w \to w'$ in $W_h$, with $h\in H$ we have
\begin{align*}
F'(\omega)T(w)&= F'(\omega)\big(\mathcal{H}(w) + F(w)\big)= \big((F'(\omega)-F(\omega))+ F(\omega)\big)\big(h(w) + F(w)\big) \\
&=\big(F'(\omega)- F(\omega)\big)h(w) + F(\omega)\\
&= \mathcal{H}(w')0^V_{f(h)} + F(\omega) \\
&= \big(\mathcal{H}(w') + F(w')\big)\big(0^V_{f(h)} + F(\omega)\big)\\
&=T(w')F(\omega).
\end{align*}
Hence $iii) \Rightarrow i)$. The last parte of the statement is clear.
\end{proof}

\begin{rem}
In the case when $f$ is just the identity the equivalence between linear natural isomorphisms and VB morphisms that make $F'$ homotopic to $F$ has been already mentioned in \cite[Remark 4.12]{BCGX22}.
\end{rem}

\begin{rem}
\label{rem:VB_morita_natural_iso}
Given a linear natural isomorphism between two VBG morphisms $(F, f),(F', f')$ from $W$ to $V$, then $F$ is a VB-Morita map if and only if so is $F'$. This simply follows from the similar property for Morita maps. Moreover, in the case when $f=f'$, the result also follows from Theorem \ref{theo:VBtransformation} and Theorem \ref{theo:caratterizzazioneVBmorita}. Indeed, let $(T,f)\colon (F,f)\Rightarrow (F',f)$ a linear natural isomorphism. Then, for any $y\in N$, we have a homotopy between the cochain maps induced by $F$ and $F'$ on the fibers of $W$ and $V$ over $y$ and $f(y)$ respectively. Then one is a quasi-isomorphism if and only if the other one is so. Moreover, $F$ and $F'$ cover the same base, so the claim follows from Theorem \ref{theo:caratterizzazioneVBmorita}.
\end{rem}

We conclude this section by recalling a characterization of VB-Morita maps covering the identity that will be useful in what follows.
\begin{prop}[{\cite[Proposition 6.2]{dHO20}}]
\label{prop:VB_morita_on_identity}
Let $F\colon (W\to H) \to (V\to G)$ be a VBG morphism covering the identity $\operatorname{id}_G$. Then $F$ is a VB-Morita map if and only if there exist a VBG morphism $F'\colon V\to W$ and two linear natural isomorphisms $(T, \operatorname{id}_G)\colon F\circ F' \Rightarrow \operatorname{id}_V$ and $(T', \operatorname{id}_G)\colon F'\circ F \Rightarrow \operatorname{id}_W$.
\end{prop}

	\chapter{Shifted symplectic structures}\label{ch:sss}

	In Poisson Geometry, Lie groupoids equipped with geometric structures play a crucial role. For instance, an integrable Poisson manifold (see Example \ref{ex:poisson_manifold}) gives rise to a Lie groupoid equipped with a compatible symplectic structure \cite{CF04}. Similarly, an integrable Jacobi manifold corresponds to a Lie groupoid equipped with a compatible contact structure \cite{CZ07}. However, the notion of a symplectic groupoid (a Lie groupoid with a compatible symplectic structure) is not Morita invariant. Therefore, to define a symplectic structure on a differentiable stack, it becomes necessary to relax the closure and non-degeneracy conditions.
	
	The first steps in this direction were made two decades ago by Xu \cite{Xu03} and Bursztyn, Crainic, Weinstein, and Zhu \cite{BCWZ04}. Xu introduced the notion of a \emph{quasi-symplectic groupoid} to unify various moment map theories under a Morita invariant framework. Meanwhile, Bursztyn and collaborators introduced the concept of a \emph{twisted presymplectic groupoid} as the integration of \emph{twisted Dirac structures}. Actually these two notions agree and such a structure is now called \emph{$+1$-shifted symplectic}.
	
	\emph{Shifted symplectic structures} in Algebraic Geoemtry were studied, e.g., by Pantev, To\"en, Vaqui\'e and Vezzosi in \cite{PTVV13}. In Differential Geometry shifted symplectic structures on (higher) Lie groupoids were introduced for the first time by Getzler in the slides presented in ``Les Diablerets'' ten years ago \cite{Ge14} and they were studied later by Cueca and Zhu in \cite{CZ23}. In Section \ref{sec:shifted_structures} we recall the definition of shifted $2$-forms and in the other two sections we recall the definition and the properties of $0$ and $+1$-shifted symplectic structures respectively.
	
	\section{Shifted $2$-forms}
	\label{sec:shifted_structures}
	In this section we recall the definition of shifted forms on Lie groupoids. As explained in Section \ref{sec:Lie_groupoids}, the nerve $G^{(\bullet)}$ of a Lie groupoid $G\rightrightarrows M$ is a simplicial manifold. Then we can consider the following complex
	\begin{equation}
		\label{eq:partial_complex}
		\begin{tikzcd}
			0 \arrow[r] & \Omega^{\bullet}(M) \arrow[r, "\partial"] &\Omega^{\bullet}(G)\arrow[r, "\partial"] &\Omega^{\bullet}(G^{(2)})\arrow[r,"\partial"] &\cdots
		\end{tikzcd},
	\end{equation}
	concentrated in non-negative degrees, where the differential $\partial$ is
	\[
		\partial = \sum_{i=0}^k (-1)^i \partial_i^\ast\colon \Omega^\bullet(G^{(k-1)})\to\Omega^\bullet(G^{(k)}),
	\] 
	the alternating sum of the pullbacks along the face maps $\partial_i$.
	
	In the next remark, we explain in which precise sense Complex \eqref{eq:partial_complex}, associated with a Lie groupoid $G\rightrightarrows M$, is Morita invariant. This property, in turn, allows it to be viewed as a complex associated with the stack $[M/G]$.
	\begin{rem}
		\label{rem:partial_Morita_invariant}
		From Remark \ref{rem:morphism_nerve}, a Lie groupoid morphism $f\colon (H\rightrightarrows N)\to (G\rightrightarrows M)$ determines a morphism $f\colon H^{(\bullet)}\to G^{(\bullet)}$ between the nerves of $H$ ad $G$. Then, for any $k$, the pullbacks along $f\colon H^{(k)}\to G^{(k)}$ determine a cochain map
		\begin{equation*}
			\begin{tikzcd}
				0 \arrow[r] & \Omega^{\bullet}(M) \arrow[r, "\partial"] \arrow[d, "f^\ast"] &\Omega^{\bullet}(G)\arrow[r, "\partial"] \arrow[d, "f^\ast"] &\Omega^{\bullet}(G^{(2)})\arrow[r,"\partial"] \arrow[d, "f^\ast"] &\cdots \\
				0 \arrow[r] & \Omega^{\bullet}(N) \arrow[r, "\partial"] &\Omega^{\bullet}(H)\arrow[r, "\partial"] &\Omega^{\bullet}(H^{(2)})\arrow[r,"\partial"] &\cdots
			\end{tikzcd}.
		\end{equation*}
		When $f$ is a Morita map, then the cochain map $f^\ast$ is a quasi-isomorphism (see \cite[Corollary 3]{Be04}). In this sense the complex $(\Omega^\filleddiamond (G^{(\bullet)}), \partial)$ is Morita invariant up to quasi-isomorphisms.
	\end{rem}
	
	Using Complex \eqref{eq:partial_complex} we can recall the definitions of \emph{basic} and \emph{multiplicative} forms on a Lie groupoid. Let $G\rightrightarrows M$ be a Lie groupoid.
	\begin{definition}
		\label{def:multiplicative_forms}
		A form $\omega\in \Omega^\bullet(M)$ on $M$ is \emph{basic} if $\omega$ is a $0$-cocycle in \eqref{eq:partial_complex}. A form $\omega\in \Omega^\bullet(G)$ on $G$ is \emph{multiplicative} if $\omega$ is $1$-cocycle in \eqref{eq:partial_complex}. We denote by $\Omega^k_{\operatorname{mult}}(G)$ the space of multiplicative $k$-forms on $G$.
	\end{definition}
	
	An example of multiplicative $1$-form on a Lie groupoid $G$ has already appeared in Section \ref{sec:Atiyah_VBG}, as we explain in the following
	\begin{rem}
		\label{rem_eta_mult}
		Let $(L\rightrightarrows L_M;G\rightrightarrows M)$ be a LBG. For any connection $\nabla$ on $L_M$ the $1$-form $\eta_\nabla=(s^\ast\nabla)-(t^\ast\nabla)\in \Omega^1(G)$, defined right before Lemma \ref{lemma:restriction_fnabla}, is multiplicative. Indeed 
		\begin{align*}
			m^{\ast} \eta_{\nabla}&= m^{\ast}s^{\ast} \nabla - m^{\ast}t^{\ast}\nabla =\pr_2^{\ast}s^{\ast}\nabla -\pr_1^{\ast} t^{\ast} \nabla\\
			&=\pr_2^{\ast}s^{\ast}\nabla - \pr_2^{\ast} t^{\ast}\nabla + \pr_1^{\ast}s^{\ast}\nabla -\pr_1^{\ast} t^{\ast} \nabla = \pr_2^{\ast} \eta_{\nabla} + \pr_1^{\ast} \eta_{\nabla},
		\end{align*}
		where we used that $s \circ m = s \circ \pr_2$ and $t \circ m = t \circ \pr_1$.
	\end{rem}
	
	Using complex \eqref{eq:partial_complex} once more, we can now introduce the following
	\begin{definition}
		A \emph{$k$-shifted $m$-form} on a Lie groupoid $G$ is a $\partial$-closed $m$-form on $G^{(k)}$.
	\end{definition}
%	\begin{rem}
%		From Remark \ref{rem:partial_Morita_invariant} the notion of shifted form is Morita invariant, then the $\partial$-cohomology class $[\omega]$ of a $k$-shifted $m$-form $\omega\in \Omega^m(G^(k))$ can be considered as a $k$-shifted $m$-form on the stack $[M/G]$.
%	\end{rem}
	Notice that $0$ and $+1$-shifted forms on $G$ are just basic and multiplicative forms respectively.
	
	The differential $\partial$ commutes with the de Rham differential $d$, giving rise to the following double complex:
	\begin{equation}
		\label{eq:BSS}
		\begin{tikzcd}
			&\vdots &\vdots &\vdots
			\\
			0\arrow[r] &\Omega^2(M) \arrow[u, "d"] \arrow[r, "\partial"] &\Omega^2(G)\arrow[u,"d"] \arrow[r,"\partial"] &\Omega^2(G^{(2)})\arrow[u,"d"] \arrow[r,"\partial"] &\cdots
			\\
			0\arrow[r] &\Omega^1(M) \arrow[u, "d"]\arrow[r,"\partial"] &\Omega^1(G)\arrow[u,"d"] \arrow[r,"\partial"]&\Omega^1(G^{(2)})\arrow[u,"d"]\arrow[r,"\partial"] &\cdots
			\\
			0\arrow[r] &C^{\infty}(M) \arrow[u, "d"]\arrow[r,"\partial"] &C^{\infty}(G)\arrow[u,"d"] \arrow[r,"\partial"]& C^{\infty}(G^{(2)})\arrow[u,"d"]\arrow[r,"\partial"] &\cdots
			\\
			&0 \arrow[u] &0\arrow[u] &0\arrow[u]
		\end{tikzcd}.
	\end{equation}
	The latter is called the \emph{Bott-Shulman-Stasheff (BSS)} double complex \cite{BSS76}. The associated total complex is the complex $(\tot(G)^{\bullet}, D)$ where $$\tot(G)^n= \bigoplus_{k+l=n}\Omega^l(G^{(k)})$$ and $D = \partial + (-1)^kd$. Notice that the bottom row of \eqref{eq:BSS} is the underlying complex of the differential algebra of $G$ discussed in Remark \ref{rem:dg_algebra}.
	\begin{rem}
		\label{rem:total_complex_Morita_invariant}
		The complex $(\tot(G)^{\bullet}, D)$ is Morita invariant up to quasi-isomorphims. Indeed, the pullback $f^\ast$ along a Lie groupoid morphism $f\colon (H\rightrightarrows N) \to (G\rightrightarrows M)$ determines a cochain map of double complexes. Then $f^\ast $ determines a cochain map between the total complexes $(\tot (G)^\bullet, D)$ and $(\tot (H)^\bullet, D)$. When $f$ is a Morita map, by Remark \ref{rem:partial_Morita_invariant} and spectral sequence arguments, $f^\ast\colon (\tot (G)^\bullet, D) \to (\tot (H)^\bullet, D)$ is a quasi-isomorphism (see \cite[Proposition 2.1]{BXu03}).
	\end{rem}
	
	Using the degeneracy maps of the nerve $G^{(\bullet)}$ of a Lie groupoid $G$ (see Equation \eqref{eq:degmaps}) we recall the following 
	\begin{definition}
		\label{def:shif_forms}
		A form $\omega\in \Omega(G^{(k)})$ is \emph{normalized} if the pullback of $\omega$ along all the degeneracy maps $d_i$ is zero:
		\[
			d_i^\ast \omega= 0 \in \Omega(G^{(i-1)}),
		\]
		for all $i=0, \dots, k-1$.
	\end{definition}
	If $\omega \in \Omega^l(G^{(k)})$ is normalized, then clearly $d\omega\in \Omega^{l+1}(G^{(k)})$ is also normalized. Moreover, $\partial\omega\in \Omega^l(G^{(k+1)})$ is normalized as well. Indeed, for any $i=0, \dots, k$, we have
	\[
	\begin{aligned}
		d_i^\ast (\partial \omega) &= \left(\sum_{j=0}^{l+1} (-1)^j d_i^\ast \partial_j^\ast\right)\omega= \sum_{j=0}^{l+1} (-1)^j (\partial_j d_i)^\ast \omega \\
		&= \sum_{j=0}^{i-1} (-1)^j (d_{i-1}\partial_j)^\ast \omega + (-1)^i \omega + (-1)^{i+1} \omega + \sum_{j=i+2}^l (d_i\partial_{j-1})^\ast \omega \\
		&= \sum_{j=0}^{i-1} (-1)^j \partial_j^\ast (d_{i-1}^\ast \omega) + \sum_{j=i+2}^l \partial_{j-1}^\ast( d_i^\ast\omega) =0,
	\end{aligned}
	\]
	where we used the simplicial identities. Hence, if $\omega\in \Omega^\filleddiamond(G^{(\bullet)})$ is normalized, then $D\omega$ is normalized and we can consider the subcomplex of $(\tot(G)^\bullet, D)$ whose cochains are $n$-tuple
	\[
		(\omega_n, \dots, \omega_0)\in \tot(G)^n= \bigoplus_{k+l=n}\Omega^k(G^{(l)}),
	\]
	i.e., for any $i=0, \dots, n$, $\omega_i\in \Omega^{n-i}(G^{(i)})$, such that $\omega_i$ is normalized for all $i=0, \dots, n$. The latter subcomplex is called the \emph{normalized subcomplex}. 
	
	\begin{rem}
		\label{rem:normalized}
		%The normalized subcomplex is quasi-isomorphic to the total complex $\tot(G)$ \textcolor{red}{Non ho una referenza}. 
		Any form on $M$ is trivially normalized, and any multiplicative form $\omega\in \Omega(G)$ is normalized (see point $i)$ of Proposition \ref{prop:formule} below). In particular any $0$ and $+1$-shifted form is normalized. 
	\end{rem}
	In \cite{CZ23}, shifted symplectic structures are described using a sequence of forms that are closed, non-degenerate in an appropriate sense, and normalized. However, since we focus solely on the $0$ and $+1$-shifted cases, by Remark \ref{rem:normalized} we will omit this last condition in what follows.
	
	%\label{def:shif_sympl_structures}
	We are more interested in closed $k$-shifted $2$-forms. Let $G$ be a Lie groupoid.
	\begin{definition}
		\label{def:closed_shifted}
		A \emph{closed $k$-shifted $2$-form} on $G$ is a $(k+1)$-tuple $(\omega_k, \dots, \omega_0)$ consisting of normalized differential forms, where, for any $i=0,\dots, k$,
		\[
		\omega_i\in \Omega^{2+k-i}(G^{(i)}),
		\] 
		and such that $D(\omega_k, \dots, \omega_0)=0$.
	\end{definition} 
	From Definition \ref{def:closed_shifted} it follows that, if $(\omega_k, \dots, \omega_0)$ ia a closed $k$-shifted $2$-form, then, in particular, $\omega_k$ is a $k$-shifted $2$-form.
		
	In the next sections we recall the definitions of $0$ and $+1$-shifted symplectic structures. Here we just want to recall that in general a $k$-shifted symplectic structure on a Lie groupoid $G$ is a closed $k$-shifted $2$-form on $G$ in the sense of Definition \ref{def:closed_shifted} that is non-degenerate in a suitable sense. 
	In particular this non-degeneracy condition of a $k$-shifted $2$-form is expressed in terms of an appropriate cochain map between the tangent complex of $G$,
	\begin{equation*}
		\begin{tikzcd}
			0 \arrow[r] & A \arrow[r, "\rho"] & TM \arrow[r] & 0,
		\end{tikzcd}
	\end{equation*}
	which is concentrated in degrees $-1,0$, and its dual complex
	\begin{equation*}
		\begin{tikzcd}
			0 \arrow[r] & T^{\ast}M \arrow[r, "\rho^{\ast}"] & A^{\ast} \arrow[r] & 0,
		\end{tikzcd}
	\end{equation*}
	which is concentrated in degrees $0,1$, shifted by $k$. Notice that there are non trivial cochain maps between those for $k=0,1,2$ only. For this reason, on a Lie groupoid, it only makes sense to consider $k$-shifted symplectic structures with $k=0,1,2$.
	
	\section{$0$-shifted symplectic structures}
	\label{sec:0_shifted}
	In this subsection, we focus on the simpler case of shifted symplectic structures, specifically the $0$-shifted ones.  These structures were studied in \cite{HS21}, where their reduction in the context of differentiable stacks was also addressed. Here, we recall their definition and provide a detailed explanation of the precise sense in which $0$-shifted symplectic structures are Morita invariant. To the best of our knowledge, the results proved in this subsection do not appear in this form elsewhere. Since they serve as inspiration for the contact case, we prefer to prove them in detail.
	
	By Definition \ref{def:closed_shifted}, a closed $0$-shifted $2$-form on a Lie groupoid $G\rightrightarrows M$ is simply a basic and $d$-closed $2$-form $\omega\in \Omega^2(M)$ on the base $M$. The condition $\partial\omega=0$ simply means that $s^{\ast}\omega= t^{\ast}\omega$. This implies that the flat map of $\omega$, that we denote again by $\omega\colon TM\to T^{\ast}M$, is a cochain map between the tangent complex and its dual
	\begin{equation}
		\label{eq:complex_0-shifted}
		\begin{tikzcd}
			0 \arrow[r] & A \arrow[r, "\rho"] \arrow[d] & TM \arrow[r] \arrow[d, "\omega"] & 0 \arrow[r] \arrow[d] & 0 \\
			0 \arrow[r] & 0 \arrow[r] & T^{\ast}M \arrow[r, "\rho^{\ast}"'] & A^{\ast} \arrow[r] & 0
		\end{tikzcd}.
	\end{equation}
	Indeed, for any $a\in A_x$ and $v=dt(\tilde{v})\in T_xM$, with $\tilde{v}\in T_xG$ and $x\in M$, we get
	\begin{equation*}
		\langle \omega(\rho(a)), v\rangle =\omega(\rho(a),v)= t^{\ast}\omega(a,\tilde{v})=s^{\ast}\omega(a,\tilde{v})=\omega(s(a), s(\tilde{v}))=0,
	\end{equation*}
	then $\omega\circ \rho=0$. Changing the roles of $a$ and $v$, we also obtain $\rho^{\ast}\circ \omega=0$.
	
	We are now ready to recall the definition of $0$-shifted symplectic structures. Let $G\rightrightarrows M$ be a Lie groupoid.
	\begin{definition}
		\label{def:0_shifted_sympl}
		A \emph{$0$-shifted symplectic structure} on $G$ is a $2$-form $\omega\in \Omega^2(M)$, such that $d\omega=0$, $\partial\omega=0$ and $\omega$ is \emph{non-degenerate}, meaning that, for any $x\in M$, the value of the cochain map \eqref{eq:complex_0-shifted} at the point $x$
		\begin{equation}
			\label{eq:0_shifted_Atiyah_non_deg}
			\begin{tikzcd}
				0 \arrow[r] & A_x \arrow[r, "\rho"] \arrow[d] & T_xM \arrow[r] \arrow[d, "\omega"] & 0 \arrow[r] \arrow[d] & 0 \\
				0 \arrow[r] & 0 \arrow[r] & T^{\ast}_xM \arrow[r, "\rho^{\ast}"'] & A_x^{\ast} \arrow[r] & 0
			\end{tikzcd}
		\end{equation}
		is a quasi-isomorphism.
	\end{definition}
	
	Notice that, by the non-degeneracy condition, $\rho\colon A_x\to T_xM$ has to be injective, for all $x\in M$. This means that the Lie groupoid $G\rightrightarrows M$ is a foliation groupoid (see Remark \ref{rem:foliation_groupoid}). Moreover, $\ker\omega = \im\rho$, hence $\omega$ induces an isomorphism between $T_xM/\im\rho$ and $\ker\rho^{\ast}$.
	
	In the last part of this subsection we want to recall in which precise sense the notion of $0$-shifted symplectic structure is Morita invariant. This is essentially discussed for general shifted symplectic structures in \cite[Lemma 2.30]{CZ23} using \emph{hypercovers} between \emph{higher Lie groupoids}. In the context of Lie groupoids, these hypercovers correspond to Morita maps which are surjective and submersive on the bases (see Remark \ref{rem:Morita_maps} for the relation between Morita maps and those which are surjective and submersive on the bases). We prove a specific version of this Morita invariance (see Proposition \ref{prop:0-shifted_Morita_inv}) for $0$-shifted symplectic strucutres on a Lie groupoid in the case of a general Morita map. To achieve this, we first need some preliminary results.
	
	We begin with the following lemma, in which we prove that the non-degeneracy condition of a basic $2$-form on $G$ an be interpreted as a condition on the orbits of $G$.
	\begin{lemma}
		\label{lemma:omega_orbit}
		Let $G\rightrightarrows M$ be a Lie groupoid and let $\omega\in \Omega^2(M)$ be such that $\partial \omega=0$. If the cochain map \eqref{eq:complex_0-shifted} is a quasi-isomorphism at the point $x\in M$, then it is a quasi-isomorphism at all points in the orbit through $x$.
	\end{lemma}
	\begin{proof}
		Let $x,y\in M$ be two points in the same orbit and let $g\colon x\to y\in G$. Let $h\colon s^\ast TM\to TG$ be an Ehresmann connection (see Definition \ref{def:Ehresmann_connection}). For simplicity, we denote by $g_T.$ the quasi actions of $G$ on $A$ and on $TM$ defining the adjoint RUTH (see Example \ref{ex:adjointRUTH}). The fibers of the tangent VBG $TG$ over the points $x$ and $y$ are related by the cochain map
		\begin{equation}
			\label{eq:adjoint_cochain_map}
			\begin{tikzcd}
				0 \arrow[r] & A_x\arrow[r, "\rho"] \arrow[d, "g_T."'] & T_xM\arrow[r] \arrow[d, "g_T."] & 0\\
				0 \arrow[r] & A_y\arrow[r, "\rho"'] & T_yM\arrow[r] & 0
			\end{tikzcd}.
		\end{equation}
		The dual complexes are related by the cochain map given by the first structure operator corresponding to $g^{-1}$ of the coadjoint RUTH (see Example \ref{ex:dual_RUTH}) that we denote by $g^{-1}_\ast .$:
		\begin{equation}
			\label{eq:dual_adjoint_cochain_map}
			\begin{tikzcd}
				0 \arrow[r] & T_y^\ast M\arrow[r, "\rho^\ast"] \arrow[d, "g^{-1}_\ast."'] & A_y^\ast\arrow[r] \arrow[d, "g^{-1}_\ast."] & 0 \\
				0 \arrow[r] & T_x^\ast M\arrow[r, "\rho^\ast"'] & A_x^\ast\arrow[r] & 0
			\end{tikzcd}.
		\end{equation}
		The cochain maps \eqref{eq:adjoint_cochain_map} and \eqref{eq:dual_adjoint_cochain_map} fit in the diagram
		\begin{equation}
			\label{eq:omega_orbit}
			{\scriptsize
			\begin{tikzcd}
				0 \arrow[rr] && A_x \arrow[rr, "\rho"] \arrow[dd] \arrow[dr, "g_T."] && T_xM \arrow[rr] \arrow[dd, "\omega", near start] \arrow[dr, "g_T."] && 0 \arrow[rr] \arrow[dd] \arrow[dr] && 0 \\
				&0 \arrow[rr, crossing over] && A_y \arrow[rr, "\rho", near start, crossing over] && T_yM \arrow[rr, crossing over] && 0 \arrow[rr] && 0 \\
				0 \arrow[rr] && 0 \arrow[rr] && T^{\ast}_xM \arrow[rr, "\rho^{\ast}"', near start] && A_x^{\ast} \arrow[rr] && 0 \\
				&0 \arrow[rr] && 0 \arrow[rr] \arrow[ul] \arrow[from=uu, crossing over] && T^{\ast}_yM \arrow[rr, "\rho^{\ast}"'] \arrow[ul, "g^{-1}_\ast."] \arrow[from=uu, "\omega", near start, crossing over] && A_y^{\ast} \arrow[rr] \arrow[ul, "g^{-1}_\ast."] \arrow[from=uu, crossing over] && 0
			\end{tikzcd}}.
		\end{equation}
		Diagram \eqref{eq:omega_orbit} is commutative. Indeed, for any $v,w\in T_xM$, we have
		\[
		\begin{aligned}
			\langle g^{-1}_\ast. \omega(g_T.v) ,w\rangle &= \omega(g_T.v, g_T.w)= \omega\big(dt(h_g(v)), dt(h_g(w))\big)\\
			&= t^\ast\omega \big(h_g(v), h_g(w)\big)= s^\ast\omega \big(h_g(v), h_g(w)\big)\\
			&=\omega\big(ds(h_g(v)), ds(h_g(w))\big)= \omega(v,w)\\
			&=\langle \omega(v),w\rangle .
		\end{aligned}
		\]
		By Remark \ref{rem:quasi_actions_quis}, $g_T.$ and $g^{-1}_\ast.$ are quasi-isomorphisms. Hence the cochain map determined by $\omega$ is a quasi-isomorphism at the point $x$ if and only if it is so at the point $y$ and the statement is proved.
	\end{proof}
	
	The pullback $f^\ast\omega$ of a basic form $\omega$ on $G$ along a Lie groupoid morphism $f\colon H\to G$ is a basic form on $H$. Consequently, $f^\ast\omega$ induces a cochain map between the core complex of $TH$ and its dual. We now examine the relation between the cochain maps induced by $\omega$ and $f^\ast\omega$ under the assumption that $f$ is a Morita map.
	\begin{prop}
		\label{prop:non_degeneracy_0-shifted}
		Let $f\colon (H\rightrightarrows N) \to (G\rightrightarrows M)$ be a Morita map and let $\omega\in \Omega^2(M)$ be such that $\partial \omega=0$. Then the cochain map determined by $\omega$ is a quasi-isomorphism at all points in $M$ if and only if the cochain map determined by $f^\ast \omega$ is so at all points in $N$.
	\end{prop}
	\begin{proof}
		For any $y\in N$, the flat maps $f^\ast\omega\colon T_yN \to T_y^\ast N$ and $\omega\colon T_{f(y)} M \to T_{f(y)}^\ast M$ fit in the following diagram
		\begin{equation}
		\label{eq:omega_Morita_inv}
		\resizebox{\textwidth}{!}{
			\begin{tikzcd}[ampersand replacement=\&] % <=== CORRETTO!
				0 \arrow[rr] \&\& A_{H,y} \arrow[rr, "\rho_H"] \arrow[dd] \arrow[dr, "df"] \&\& T_yN \arrow[rr] \arrow[dd, "f^\ast\omega", near start] \arrow[dr, "df"] \&\& 0 \arrow[rr] \arrow[dd] \arrow[dr] \&\& 0 \\
				\&0 \arrow[rr, crossing over] \&\& A_{G,f(y)} \arrow[rr, "\rho_G", near start, crossing over] \&\& T_{f(y)}M \arrow[rr, crossing over] \&\& 0 \arrow[rr] \&\& 0 \\
				0 \arrow[rr] \&\& 0 \arrow[rr] \&\& T^{\ast}_yN \arrow[rr, "\rho_H^{\ast}"', near start] \&\& A_{H,y}^{\ast} \arrow[rr] \&\& 0 \\
				\&0 \arrow[rr] \&\& 0 \arrow[rr] \arrow[ul] \arrow[from=uu, crossing over] \&\& T^{\ast}_{f(y)}M \arrow[rr, "\rho_G^{\ast}"'] \arrow[ul, "df^\ast"] \arrow[from=uu, "\omega", near start, crossing over] \&\& A_{G,f(y)}^{\ast} \arrow[rr] \arrow[ul, "df^\ast"] \arrow[from=uu, crossing over] \&\& 0
			\end{tikzcd}}.
		\end{equation}
		Diagram \eqref{eq:omega_Morita_inv} is commutative. Indeed, for any $v,w\in T_yN$, we have
		\begin{align*}
		\langle df^\ast(\omega (df(v))),w\rangle&= \langle \omega(df(v)),df(w)\rangle = \omega(df(v), df(w))\\
		&= (f^\ast \omega) (v,w)= \langle f^\ast\omega(v),w\rangle.
		\end{align*}
		
		Since $f$ is Morita, by Example \ref{ex:df_VB_Morita} and Proposition \ref{prop:dual_VB-Morita_map}, the cochain maps determined by $df$ and $df^\ast$ in Diagram \eqref{eq:omega_Morita_inv} are quasi-isomorphisms. Then the cochain map determined by $f^\ast\omega$ is a quasi-isomorphism at the point $y\in N$ if and only if the cochain map determined by $\omega$ is a quasi-isomorphism at the point $f(y)\in M$. Hence if the cochain map determined by $\omega$ is a quasi-isomorphism at all points in $M$, then the one determined by $f^\ast \omega$ is so at all points in $N$. For the converse, we have that if the cochain map determined by $f^\ast\omega$ is a quasi-isomorphism at all points $y\in N$, then the one determined by $\omega$ is a quasi-isomorphism at the points $f(y)\in M$. Since $f$ is essentially surjective, then, for any $x\in M$ there exists $y\in N$ such that $f(y)$ and $x$ are in the same orbit. Finally, from Lemma \ref{lemma:omega_orbit}, we get that the cochain map determined by $\omega$ is a quasi-isomorphism at all points in $M$.
	\end{proof}

	Finally, we are ready to prove the following result of Morita invariance.
	\begin{prop}
		\label{prop:0-shifted_Morita_inv}
		Let $f\colon H\to G$ be a Morita map. The pullback $f^\ast\colon \Omega^\filleddiamond (G^{(\bullet)})\to \Omega^\filleddiamond (H^{(\bullet)})$ induces a bijection between $0$-shifted symplectic structures on $G$ and on $H$.
	\end{prop}
	\begin{proof}
		From Remark \ref{rem:partial_Morita_invariant}, the cochain map $f^\ast \colon (\Omega^\filleddiamond (G^{(\bullet)}), \partial) \to (\Omega^\filleddiamond (H^{(\bullet)}), \partial)$ is a quasi-isomorphism. Then $f^\ast$ maps bijectively $\partial$-closed $2$-forms on $M$ to $\partial$-closed $2$-forms on $N$. Moreover, $f^\ast$ preserves the closure with respect the de Rham differential, i.e., $d\omega=0$ if and only if $df^{\ast}\omega=0$. Indeed, if $d\omega=0$, then 
		\[
			df^\ast\omega=f^\ast d\omega=0.
		\]
		Conversly, if $df^\ast \omega=0$, then $f^\ast d\omega=0$. But, $\partial d\omega=d\partial \omega=0$ and, since $f$ is an isomorphism on $\ker \partial$, then $d\omega=0$.
		
		Finally, by Proposition \ref{prop:non_degeneracy_0-shifted}, $\omega$ is non-degenerate if and only if $f^\ast\omega$ is so. Hence, the claim is proved.
	\end{proof}

	From Proposition \ref{prop:0-shifted_Morita_inv} the following Definition makes sense.
	\begin{definition}
		A \emph{$0$-shifted symplectic structure} on a stack $[M/G]$ is a $0$-shifted symplectic structure on a Lie groupoid $G\rightrightarrows M$ presenting $[M/G]$.
	\end{definition}
	
	\section{$+1$-shifted symplectic structures}
	\label{sec:1_shifted_s}
	In this final section, we focus on the most relevant shifted symplectic structures in Poisson Geometry: the $+1$-shifted ones. They were studied for the first time twenty years ago by Ping Xu \cite{Xu03}, who unified different momentum map theories under one single framework, and by Bursztyn, Crainic, Weinstein and Zhu \cite{BCWZ04} as the integration of \emph{twisted Dirac structures}. 
	
	We begin by recalling the definition of $+1$-shifted symplectic structures. Then, in Section \ref{sec:mult}, we explore properties of multiplicative $2$-forms on a Lie groupoid, demonstrating in particular that the non-degeneracy condition of a $+1$-shifted symplectic structure is a Morita invariant notion. In Section \ref{sec:sme} we describe Morita equivalence between Lie groupoids equipped with such structures. Finally, in Section \ref{sec:twisted}, we review the infinitesimal counterpart there structures, namely twisted Dirac structures.
	
	By Definition \ref{def:closed_shifted}, a $+1$-shifted $2$-form on a Lie groupoid $G\rightrightarrows M$ is a $D$-closed pair $(\omega, \Omega)\in \Omega^2(G)\oplus \Omega^3(M)$. This simply means
	\begin{equation*}
		\partial \omega=0, \quad d\omega=\partial \Omega, \quad d\Omega=0,
	\end{equation*}
	i.e., $\omega$ is multiplicative, $\omega$ is $d$-closed up to the $\partial$-coboundary of $\Omega$, and $\Omega$ is $d$-closed.
	
	Multiplicativity of $2$-forms on Lie groupoids has some useful characterizations that we recall from, e.g., \cite[Theorem 5.1]{Ko16} in the following 
	\begin{rem}\label{rem:omega_VBGmorphism}
		Let $G\rightrightarrows M$ be a Lie groupoid and let $\omega\in \Omega^2(G)$. The $2$-form $\omega\in \Omega^2(G)$ is multiplicative if and only if the graph of the multiplication
		\[
		\{(g,h,gh) \in G\times G \times G \, | \, s(g)= t(h)\}
		\]
		is an isotropic submanifold in $(G\times G\times G, (\omega,\omega,-\omega))$. %Indeed, $\omega$ is  multiplicative if and only if, for any $(v,v'), (u,u')\in T_{(g,g')}G^{(2)}$, with $(g,g')\in G^{(2)}$, 
		%\begin{align*}
		%	0&=(\pr_1^{\ast} \omega)((v,v'), (u,u')) + (\pr_2^\ast \omega)((v,v'), (u,u')) -(m^{\ast} \omega)((v,v'), (u,u')) \\
		%	&= \omega(v,u) + \omega(v',u') -\omega(v\cdot u,v'\cdot u') \\
		%	&=(\omega,\omega, -\omega)((v,u,v\cdot u), (v',u',v'\cdot u')).
		%\end{align*}
		%But $(v,u,v\cdot u)$ and $(v',u',v'\cdot u')$ are the generic tangent vectors to $\Lambda$. Then the statement.
		
		Moreover, the $2$-form $\omega$ is multiplicative if and only if its flat map, again denoted by $\omega\colon TG\to T^{\ast}G$, determines a VBG morphism covering the identity $\operatorname{id_G}\colon G\to G$ from the tangent VBG $TG$ to the cotangent VBG $T^\ast G$.
	\end{rem}
	
	\begin{rem}
		\label{rem:cochain_map}
		From Remark \ref{rem:omega_VBGmorphism} it follows that, if $\omega\in \Omega^2(G)$ is multiplicative, then the flat map of $\omega$
		\begin{equation*}
			\scriptsize
			\begin{tikzcd}
				TG \arrow[rr, shift left=0.5ex] \arrow[rr, shift right=0.5ex] \arrow[dd] \arrow[dr, "\omega"] & &TM \arrow[dd] \arrow[dr, "\omega"] \\
				&  T^{\ast}G \arrow[rr, shift left= 0.5ex, crossing over] \arrow[rr, shift right =0.5ex, crossing over] & &A^{\ast} \arrow[dd]\\
				G \arrow[rr, shift left=0.5ex] \arrow[rr, shift right=0.5ex] \arrow[dr, equal] & & M \arrow[dr, equal]\\ 
				&  G \arrow[from=uu, crossing over]\arrow[rr, shift left= 0.5ex] \arrow[rr, shift right =0.5ex] & &M
			\end{tikzcd}
		\end{equation*}
		determines a VBG morphism. Hence, $\omega$ induces a cochain map between the core complex of $TG$ and the core complex of $T^\ast G$:
		\begin{equation}
			\label{eq:complex_1-shifted}
			\begin{tikzcd}
				0 \arrow[r] & A \arrow[r, "\rho"] \arrow[d, "\omega"'] & TM \arrow[r] \arrow[d, "\omega"] & 0 \\
				0 \arrow[r] & T^{\ast}M \arrow[r, "\rho^\ast"'] & A^\ast \arrow[r] & 0
			\end{tikzcd}. \qedhere
		\end{equation}
	\end{rem}
	\begin{rem}
		We recall that a \emph{symplectic groupoid} is a Lie groupoid $G\rightrightarrows M$ equipped with a symplectic form $\omega\in \Omega^2(G)$ on $G$ that is, additionally, multiplicative. The multiplicativity of $\omega$ implies that its flat map is a VBG morphism and the non-degeneracy condition of $\omega$ can be expressed saying that the VBG morphism $\omega\colon TG\to T^\ast G$ is a VBG isomorphism.  
	\end{rem}
	
	We are now ready to recall the notion of a $+1$-shifted symplectic structure, adopting the perspective of \cite{CZ23} for the non-degeneracy condition. Let $G\rightrightarrows M$ be a Lie groupoid.
	\begin{definition}
		\label{def:+1_shifted_sympl}
		A \emph{$+1$-shifted symplectic structure} on $G$ is a pair $(\omega, \Omega)\in \Omega^2(G)\oplus \Omega^3(M)$ such that $\partial \omega=0$, $d\omega=\partial \Omega$, $d\Omega=0$, and $\omega$ is non-degenerate meaning that, for any $x\in M$, the value of the cochain map \eqref{eq:complex_1-shifted} at the point $x$
		\begin{equation}
			\label{eq:nondegenerate_1}
			\begin{tikzcd}
				0 \arrow[r] & A_x \arrow[r, "\rho"] \arrow[d, "\omega"'] & T_xM \arrow[r] \arrow[d, "\omega"] & 0 \\
				0 \arrow[r] & T^{\ast}_xM \arrow[r, "\rho^\ast"'] & A_x^\ast \arrow[r] & 0
			\end{tikzcd}
		\end{equation}
		is a quasi-isomorphism.
		
		A \emph{$+1$-shifted symplectic groupoid} is a triple $(G, \omega, \Omega)$ where $G$ is a Lie groupoid and $(\omega, \Omega)$ is a $+1$-shifted symplectic structure on $G$.
	\end{definition}
	\begin{rem}
		\label{rem:quasi_iso_-1_0}
		Notice that the map $\omega\colon T_xM\to A_x^\ast$ is just the opposite of the dual of $\omega\colon A_x\to T_x^\ast M$. Indeed, for any $a\in A_x$ and $v\in T_xM$, we have
		\[
		\langle -\omega^\ast(v),a\rangle=- \langle \omega(a),v\rangle = -\omega(a,v)= \omega(v,a)= \langle \omega(v),a\rangle. 
		\]
		Then, the cochain map determined by $\omega$ is a quasi-isomorphism in degree $-1$ if and only if $\omega\colon \ker \rho \to \ker \rho^\ast$ is an isomorphism, if and only if $-\omega^\ast \colon (\ker \rho^\ast)^\ast \to (\ker \rho)^\ast$, or equivalently $\omega\colon T_xM/\im \rho \to A_x^\ast/\im \rho^\ast$, is an isomorphism, if and only if the cochain map determined by $\omega$ is a quasi-isomorphism in degree $0$. Hence, it is enough to check that the cochain map induced by $\omega$ is a quasi-isomorphism only in one degree.
	\end{rem}
	
	By Theorem \ref{theo:caratterizzazioneVBmorita}, $\omega$ is non-degenerate if and only if the VBG morphism $\omega\colon TG\to T^{\ast}G$ is a VB-Morita map \cite[Proposition 5.4]{dHO20}. Hence, we have this alternative, but equivalent, definition of $+1$-shifted symplectic structure:
	\begin{definition}
		A \emph{$+1$-shifted symplectic structure} on $G$ is a pair $(\omega, \Omega)\in \Omega^2(G)\oplus \Omega^3(M)$, such that $\partial \omega=0$, $d\omega=\partial \Omega$, $d\Omega=0$, and $\omega$ is non-degenerate, i.e., $\omega\colon TG\to T^\ast G$ is a VB-Morita map.
	\end{definition}

	\subsection{Multiplicative $2$-forms}\label{sec:mult}
 	In this subsection we recall some standard properties of multiplicative $2$-forms, with particular focus on the non-degeneracy condition of $+1$-shifted symplectic structures. Indeed, this only depend on the multiplicative $2$-form on the Lie groupoid. We establish that this non-degeneracy condition is a Morita invariant property. To the best of our knowledge, this result does not appear in this specific form elsewhere. Furthermore, we present two distinct proofs: the first follows an approach similar to the one used in the $0$-shifted case, while the second relies exclusively on VB-groupoid morphisms and linear natural isomorphisms.
	
	Let $G\rightrightarrows M$ be a Lie groupoid. By Definition \ref{def:multiplicative_forms}, a $2$-form $\omega\in \Omega^2(G)$ on $G$ is multiplicative if it is closed with respect to the differential $\partial$, i.e.,
	\begin{equation}
		\label{eq:mult}
		m^{\ast} \omega = \pr_1^{\ast} \omega + \pr_2^\ast \omega\in \Omega^2(G^{(2)}),
	\end{equation}
	where $\pr_i\colon G^{(2)}\to G$ is the projection on the $i$-th factor, with $i=1,2$. 
	
	In the next Proposition we summarize some computational rules for multiplicative $2$-forms that have been discussed in \cite[Proposition 2.3]{Xu03} and \cite[Lemma 3.1]{BCWZ04} and will be useful in the follows.
	\begin{prop}
		\label{prop:formule}
		Let $\omega\in \Omega^2(G)$ be a multiplicative form on the Lie groupoid $G\rightrightarrows M$. Then
		\begin{itemize}
			\item[i)] the pullback $u^\ast \omega$ of $\omega$ along the unit map $u\colon M\to G$ is zero;
			\item[ii)] the pullback $i^\ast \omega$ of $\omega$ along the inverse map $i\colon G\to G$ is $-\omega$.
%			\item[iii)] for any $a,b\in \Gamma(A)$, we have
%			\[
%				\omega(\overrightarrow{a},\overrightarrow{b}) = - \omega(\overleftarrow{a},\overleftarrow{b}), \quad \omega(\overrightarrow{a},\overleftarrow{b})=0,
%			\]
%			where $\overrightarrow{a},\overrightarrow{b}$ and $\overleftarrow{a},\overleftarrow{b}$ are the right and left-invariant vector fields respectively, generated by $a$ and $b$ (see Definition \ref{def:right_invariant_vector_fields}).
		\end{itemize}
	\end{prop}
	%\begin{proof}
	%	Let $v,u\in T_xM$, with $x\in M$. Then, applying \eqref{eq:mult} to $(v,v),(u,u)\in T_{(x,x)}G^{(2)}$, we get
	%	\[
	%		\omega(v,u) = \omega(v,u)+\omega(v,u),
	%	\]
	%	and so $\omega(v,u)=0$ and the $i)$ is proved. In order to prove $ii)$ let $v,u\in T_gG$ and applying \eqref{eq:mult} to $(v, v^{-1}), (u,u^{-1})\in T_{(g,g^{-1})}G^{(2)}$, we get
	%	\begin{equation}
	%		\label{eq:inv_omega}
	%		\omega(v\cdot v^{-1}, u\cdot u^{-1}) = \omega(v,u) +\omega(v^{-1}, u^{-1}),
	%	\end{equation}
	%	but $v\cdot v^{-1}= dt(v), u\cdot u^{-1}=dt(u)\in T_{t(g)}M$ and, from $i)$, the equation \eqref{eq:inv_omega} reduces to $\omega(v^{-1}, u^{-1})=-\omega(v,u)$ and so the statement.
	%	The first part of $iii)$ follows from $ii)$ noting that, for any $g\in G$, $a^R_g=di(a^L_g)$ and the same for $b$. Finally, for any $g\in G$, applying \eqref{eq:mult} to $(a^R_g, 0_{s(g)}), (0_g, b^L_{s(g)})\in T_{(g,t(g))}G^{(2)}$, we get
	%	\[
	%		\omega(a^R_g, b^L_g) = \omega(a^R_g, 0_g) + \omega(0_{s(g)}, b^L_{s(g)}),
	%	\]
	%	where we used that $a^R_g \cdot 0_{s(g)}= a^R_g$ and $0_g\cdot b^L_{s(g)}= b^L_g$.
	%\end{proof}
	
	\begin{rem}
		\label{rem:eta_nabla_formule}
		Identities described in Proposition \ref{prop:formule} hold for a generic multiplicative form. In particular, in what follows we need these identities for the multiplicative $1$-form $\eta_\nabla\in \Omega^1(G)$, where $\nabla$ is a connection on $L_M$ and $L\rightrightarrows L_M$ is an LBG over $G$.
	\end{rem}

	Before examining the relation between Morita maps and the cochain maps induced by multiplicative $2$-forms, we establish the following result, which connects a multiplicative $2$-form to the quasi-actions coming from the tangent and cotangent VBGs.
	\begin{prop}
		Let $\omega\in \Omega^2(G)$ be a multiplicative $2$-form on the Lie groupoid $G$. Denote by $g_T.$ the first structure operator of the adjoint RUTH (Example \ref{ex:adjointRUTH}) determined by an Ehresmann connection $h\colon s^\ast TM \to TG$, then 
		\begin{equation}
			\label{eq:omega_adjointRUTH}
			\omega(g_T.a, g_T.v)= \omega(a,v) + \omega\left(h_g(\rho(a)), h_g(v)\right),
		\end{equation}
		for all $a\in A_x$ and $v\in T_xM$, with $x\in M$.
	\end{prop}
	\begin{proof}
		Let $g\colon x\to y\in G$. From Equation \eqref{eq:quasi_action_onA}, for any $a\in A_x$, we have
		\begin{equation*}
			g_T.a= h_g(\rho(a))\cdot a \cdot 0_{g^{-1}}  \in A_y.
		\end{equation*}
		Remembering that $g_T.v= dt(h_g(v))= h_g(v)\cdot h_g(v)^{-1}\in T_yM$, we have
		\begin{align*}
			\omega(g_T.a, g_T.v) &= \omega\left(h_g(\rho(a))\cdot a \cdot 0_{g^{-1}}, h_g(v)\cdot h_g(v)^{-1}\right)\\
			&=\omega\left(h_g(\rho(a))\cdot a, h_g(v)\cdot v\right) \\
			&=\omega\left(h_g(\rho(a)), h_g(v)\right) + \omega\left(a, v\right).\qedhere
		\end{align*}
	\end{proof}
	
	We now prove that the quasi-isomorphism of a cochain map induced by a multiplicative $2$-form is a condition on the orbits of the Lie groupoid. This result is analogous to Lemma \ref{lemma:omega_orbit} and follows a similar reasoning.
	\begin{lemma}
		\label{lemma:non-deg_orbit}
		Let $G\rightrightarrows M$ be a Lie groupoid and let $\omega\in \Omega^2(G)$ be a multiplicative $2$-form on $G$. If the cochain map \eqref{eq:nondegenerate_1} is a quasi-isomorphism at a point $x\in M$, then it is a quasi-isomorphism at all points in the orbit through $x$.
	\end{lemma}
	\begin{proof}
		Let $g\colon x\to y\in G$, with $x,y\in M$. As discussed in the proof of Lemma \ref{lemma:omega_orbit}, the fibers over the points $x$ and $y$ are related by the cochain map \eqref{eq:adjoint_cochain_map}:
		\begin{equation*}
			\begin{tikzcd}
				0 \arrow[r] & A_x\arrow[r, "\rho"] \arrow[d, "g_T."'] & T_xM\arrow[r] \arrow[d, "g_T."] & 0\\
				0 \arrow[r] & A_y\arrow[r, "\rho"'] & T_yM\arrow[r] & 0
			\end{tikzcd}
		\end{equation*}
		given by the adjoint RUTH (Example \ref{ex:adjointRUTH}), and the dual complexes are related by the cochain map \eqref{eq:dual_adjoint_cochain_map}, given by the coadjoint RUTH (Example \ref{ex:dual_RUTH}):
		\begin{equation*}
			\begin{tikzcd}
				0 \arrow[r] & T_y^\ast M\arrow[r, "\rho^\ast"] \arrow[d, "g^{-1}_\ast."'] & A_y^\ast\arrow[r] \arrow[d, "g^{-1}_\ast."] & 0 \\
				0 \arrow[r] & T_x^\ast M\arrow[r, "\rho^\ast"'] & A_x^\ast\arrow[r] & 0
			\end{tikzcd}.
		\end{equation*}
	%	where we denote by $g^{-1}_\ast .$ the quasi actions determined by $g^{-1}$ and defining the $1$-st structure operator of the coadjoint RUTH.
		These cochain maps fit in the diagram
		\begin{equation}
			\label{eq:diagram_orbit}
			{\scriptsize
			\begin{tikzcd}
				0 \arrow[rr] & & A_x \arrow[rr, "\rho"] \arrow[dd, "\omega"' near start] & & T_xM \arrow[rr] \arrow[dd, "\omega"' near start] & & 0 \\
				& 0 \arrow[rr, crossing over] & & A_y \arrow[rr, crossing over, "\rho" near start] \arrow[from=ul, "g_T."] & & T_yM \arrow[rr] \arrow[from=ul, "g_T."]& & 0 \\
				0 \arrow[rr] & & T_x^{\ast}M \arrow[rr, "\rho^\ast"' near end] \arrow[from= dr, "g^{-1}_\ast."] & & A_x^{\ast} \arrow[rr] \arrow[from=dr, "g^{-1}_\ast."] & & 0 \\
				& 0 \arrow[rr] & & T_y^{\ast}M \arrow[rr,  "\rho^\ast"'] \arrow[from=uu, crossing over, "\omega"' near start]  & & A_y^{\ast} \arrow[rr] \arrow[from=uu, crossing over, "\omega" near start] & & 0
			\end{tikzcd}}.
		\end{equation}
		Diagram \eqref{eq:diagram_orbit} does not commute, but, by Equation \eqref{eq:omega_adjointRUTH}, it commutes in cohomology. Moreover, by Remark \ref{rem:quasi_actions_quis}, $g_T.$ and $g^{-1}_\ast.$ are quasi-isomorphisms. Hence the cochain map determined by $\omega$ is a quasi-isomorphism at the point $x$ if and only if it is so at the point $y$ and the statement is proved.
	\end{proof}
	
	We are now ready to present the Morita invariance result, which is analogous to Proposition \ref{prop:0-shifted_Morita_inv}. Indeed, pullback of multiplicative forms are again multiplicative. Then in the next result we relate the cochain maps determined by a multiplicative form and by its pullback along a Morita map.
	\begin{prop}
		\label{prop:+1-shifted_Morita_map}
		Let $f\colon (H\rightrightarrows N)\to (G\rightrightarrows M)$ be a Morita map between Lie groupoids and let $\omega\in \Omega^2(G)$ be a multiplicative $2$-form on $G$. Then $f^{\ast}\omega$ is a quasi-isomorphism between the fibers over all points in $N$ if and only if $\omega$ is so over all points in $M$.
	\end{prop}
	\begin{proof}
		By Example \ref{ex:df_VB_Morita} the VBG morphism $(df,f)\colon (TH\rightrightarrows TN;H\rightrightarrows N)\to (TG\rightrightarrows TM;G\rightrightarrows M)$ is a VB-Morita map, then, by Theorem \ref{theo:caratterizzazioneVBmorita}, for any $y\in N$, if we set $x=f(y)\in M$, then the cochain map
		\begin{equation}
			\label{eq:df_quis}
			\begin{tikzcd}
				0 \arrow[r] & A_{H,y}\arrow[r, "\rho_H"] \arrow[d, "df"'] & T_yN \arrow[r] \arrow[d, "df"] & 0\\
				0 \arrow[r] & A_{G,x}\arrow[r, "\rho_G"'] & T_xM \arrow[r] & 0
			\end{tikzcd}
		\end{equation}
		is a quasi-isomorphism. Hence, the cochain map
		\begin{equation}
			\label{eq:dualdf_quis}
			\begin{tikzcd}
				0 \arrow[r] & T^\ast_xM \arrow[r, "\rho_G^\ast"] \arrow[d, "df^\ast"'] & A_{G,x}^\ast \arrow[r] \arrow[d, "df^\ast"] & 0\\
				0 \arrow[r] & T^\ast_yN \arrow[r, "\rho_H^\ast"'] & A_{H,y}^\ast \arrow[r] & 0
			\end{tikzcd},
		\end{equation}
		given by the dual of the cochain map \eqref{eq:df_quis} is a quasi-isomorphism as well.
		
		The quasi-isomorphisms \eqref{eq:df_quis} and \eqref{eq:dualdf_quis} fit in the following diagram
		\begin{equation}
			\label{eq:df}
			{\scriptsize
			\begin{tikzcd}
				0 \arrow[rr] & & A_{H,y} \arrow[rr, "\rho_H"] \arrow[dd] & & T_yN \arrow[rr] \arrow[dd] & & 0 \\
				& 0 \arrow[rr, crossing over] & & A_{G,x} \arrow[rr, crossing over, "\rho_G" near start] \arrow[from=ul, "df"] & & T_xM \arrow[rr] \arrow[from=ul, "df"]& & 0 \\
				0 \arrow[rr] & & T_y^{\ast}N \arrow[rr, "\rho_H^\ast"' near end] \arrow[from= dr, "df^\ast"] & & A_{G,y}^{\ast} \arrow[rr] \arrow[from=dr, "df^\ast"'] & & 0 \\
				& 0 \arrow[rr] & & T_x^{\ast}M \arrow[rr,  "\rho_G^\ast"'] \arrow[from=uu, crossing over]  & & A_{G,x}^{\ast} \arrow[rr] \arrow[from=uu, crossing over] & & 0
			\end{tikzcd}}
		\end{equation}
		where the vertical arrows are the cochain maps induced by $f^\ast \omega$ and $\omega$. Diagram \eqref{eq:df} commutes. Indeed, for any $a\in A_{H,y}$ and $v\in T_yN$, we get
		\begin{align*}
			\langle f^\ast \omega(a),v\rangle& = f^\ast\omega(a,v)= \omega(df(a), df(v)) \\
			&=\langle \omega(df(a)),df(v)\rangle = \langle df^\ast(\omega(df(a))),v\rangle,
		\end{align*}
		and the left square in \eqref{eq:df} commutes. Exchanging the role of $a$ and $v$ we see that the right square in \eqref{eq:df} commutes as well.
		
		Since $df$ and $df^\ast$ in \eqref{eq:df} are quasi-isomorphisms, then the cochain map induced by $f^\ast \omega$ is a quasi-isomorphism at the point $y\in N$ if and only if the one induced by $\omega$ is so at the point $x=f(y)\in M$. Hence, if the cochain map induced by $\omega$ is a quasi-isomorphism at all points in $M$, then the one induced by $f^\ast\omega$ is a quasi-isomorphism at all points in $N$. For the converse, if the cochain map induced by$f^\ast \omega$ is a quasi-isomorphism at all points in $N$, then the one induced by $\omega$ is a quasi-isomorphism at all points $f(y)\in M$. But, since $f$ is essentially surjective, it follows that, for any $x\in M$, there exists $y\in N$ such that $f(y)$ and $x$ are in the same orbit. The statement follows from Lemma \ref{lemma:non-deg_orbit}.
	\end{proof}
	
	By the second part of Remark \ref{rem:omega_VBGmorphism}, every multiplicative $2$-form determines a VBG morphism. If $\omega\in \Omega^2(G)$ is multiplicative, then $\omega$ determines the VBG morphism $\omega\colon TG\to T^\ast G$, and, if $f\colon H\to G$ is a Lie groupoid morphism, then $f^\ast \omega\in \Omega^2(H)$ is also multiplicative, and it determines the VBG morphim $f^\ast \omega\colon TH\to T^\ast H$. Hence, by Theorem \ref{theo:caratterizzazioneVBmorita}, Proposition \ref{prop:+1-shifted_Morita_map} says that, if $f\colon H\to G$ is a Morita map, then $\omega\colon TG\to T^\ast G$ is a VB-Morita map if and only if $f^\ast \omega\colon TH\to T^\ast H$ is so. In Proposition \ref{prop:omega_VB-Morita_if_fomega} we provide an alternative and more conceptual proof of the latter using linear natural isomorphisms (see Definition \ref{def:LNT}).

	\begin{prop}
		\label{prop:omega_VB-Morita_if_fomega}
		Let $f\colon H\to G$ be a Morita map and let $\omega\in \Omega^2(G)$ be a multiplicative form. Then $\omega\colon TG\to T^\ast G$ is a VB-Morita map if and only if $f^\ast \omega\colon TH \to T^\ast H$ is so.
	\end{prop}
	\begin{proof}
		The VBG morphisms $f^\ast\omega\colon TH\to T^\ast H$ and $\omega\colon TG \to T^\ast G$ fit in the following diagram
		\begin{equation}
			\label{eq:pentagon}
			\begin{tikzcd}
				 & TH \arrow[rr, "df"] \arrow[dl, "f^\ast\omega"'] & & TG\arrow[dr, "\omega"]\\
				 T^\ast H &&&& T^\ast G \\
				 && f^\ast T^\ast G \arrow[ull, "df^\ast"] \arrow[urr, "\pr_2"']
			\end{tikzcd},
		\end{equation}
		where $df\colon TH\to TG$ is a VB-Morita map because of Example \ref{ex:df_VB_Morita}, and $T^\ast H \xleftarrow{df^\ast} f^\ast T^\ast G \xrightarrow{ \pr_2} T^\ast G$ is the Morita equivalence (between the dual VBGs $T^\ast H$ and $T^\ast G$) discussed in Remark \ref{rem:dualVBG_VB-Morita_equivalent}. 
		
		Since $df^\ast\colon f^\ast T^\ast G \to T^\ast H$ is a VB-Morita map covering the identity $\operatorname{id}_H$, by Proposition \ref{prop:VB_morita_on_identity}, there exist a VBG morphism $F\colon T^\ast H \to f^\ast T^\ast G$ and two linear natural isomorphisms $(T,\operatorname{id}_H)\colon df^\ast \circ F\Rightarrow \operatorname{id}_{T^\ast H}$ and $(T',\operatorname{id}_H)\colon F\circ df^\ast \Rightarrow \operatorname{id}_{f^\ast T^\ast G}$. In particular, by Theorem \ref{theo:VBtransformation}, $\mathcal{H}'=T'- F\circ df^\ast\colon f^\ast A_G^\ast\to f^\ast T^\ast M$ is a smooth map covering the identity $\operatorname{id}_N$ that makes $F'=F\circ df^\ast$ homotopic to $\operatorname{id}_{f^\ast T^\ast G}$.
		%for any $y\in N$, there exists a homotopy $\mathcal{H'}$ between $F'=F\circ df^\ast$ and $\operatorname{id}$
		%\begin{equation*}
		%	\begin{tikzcd}
		%		0 \arrow[r] & T_x^\ast M \arrow[r, "\rho_G^\ast"] \arrow[d, "\operatorname{id}", shift left=0.5ex] \arrow[d, "F'"', shift right=0.5ex] & A_{G,x}^\ast \arrow[r] \arrow[d, "\operatorname{id}", shift left=0.5ex] \arrow[d, "F'"', shift right=0.5ex] \arrow[dl, "\mathcal{H'}"'] & 0\\
		%		0 \arrow[r] & T_x^\ast M \arrow[r, "\rho_G^\ast"'] & A_{G,x}^\ast \arrow[r] & 0
		%	\end{tikzcd},
		%\end{equation*}
		%where $x=f(y)\in M$.
		
		Replacing $df^\ast$ with $F$ in Diagram \eqref{eq:pentagon} we get the diagram
		\begin{equation*}
			\label{eq:pentagon2}
			\begin{tikzcd}
				& TH \arrow[rr, "df"] \arrow[dl, "f^\ast\omega"'] & & TG\arrow[dr, "\omega"]\\
				T^\ast H &&&& T^\ast G \\
				&& f^\ast T^\ast G \arrow[from=ull, "F"'] \arrow[urr, "\pr_2"']
			\end{tikzcd}.
		\end{equation*}
		The latter commutes up to a linear natural isomorphism. In order to see this, first notice that both $K'=\pr_2\circ F\circ f^\ast \omega$ and $K= \omega \circ df$ are VBG morphisms from $TH$ to $T^\ast G$ covering $f\colon H\to G$. In order to apply Theorem \ref{theo:VBtransformation}, we require a VB morphism $\mathcal{H}\colon TN\to T^\ast M$ covering $f$ that makes $K'$ homotopic to $K$.
		%to define a smooth family of homotopies $\mathcal{H}$ between $K$ and $K'$:
		%\begin{equation*}
		%	\begin{tikzcd}
		%		0 \arrow[r] & A_{H,y} \arrow[r, "\rho_H"] \arrow[d, "K", shift left=0.5ex] \arrow[d, "K'"', shift right=0.5ex] & T_yN \arrow[r] \arrow[d, "K", shift left=0.5ex] \arrow[d, "K'"', shift right=0.5ex] \arrow[dl, "\mathcal{H}"'] & 0\\
		%		0 \arrow[r] & T_x^\ast M \arrow[r, "\rho_G^\ast"'] & A_{G,x}^\ast \arrow[r] & 0
		%	\end{tikzcd},
		%\end{equation*}
		%where $y\in N$ and $x=f(y)\in M$. We define $\mathcal{H}\colon T_yN \to T_X^\ast M$ by setting
		%\[
		%	\mathcal{H}(v)= \mathcal{H'}(\omega(df(v))), \quad v\in T_yN.
		%\]
	
		%For any $a\in A_{H,y}$, we have
		%\begin{align*}
		%	\mathcal{H}(\rho_H(a))&= \mathcal{H'}(\omega(df(\rho_H(a)))) \\
		%	&=\mathcal{H'}(\omega(\rho_G(df(a))))\\
		%	&=\mathcal{H'}(\rho_G^\ast(\omega(df(a))))\\
		%	&= F(df^\ast(\omega(df(a)))) - \omega(df(a)) \\
		%	&= F((f^\ast\omega)_\flat(a))- \omega(df(a))\\
		%	&= (K'-K)(a),
		%\end{align*}
		%and the homotopy condition in degree $-1$ is proved.
		
		%For any $v\in T_yN$, we have
		%\begin{align*}
		%	\rho_G^\ast (\mathcal{H}(v))&= \rho_G^\ast(\mathcal{H'}(\omega(df(v))))\\
		%	&= F(df^\ast(\omega(df(v)))) - \omega(df(v))\\
		%	&= F((f^\ast\omega)_\flat (v))- \omega (df(v))\\
		%	&= (K'-K)(v),
		%\end{align*}
		%and the homotopy condition in degree $0$ is proved. 
		
		The composition $\omega\circ df\colon TN\to A_G^\ast$ is a VB morphism covering $f\colon N\to M$. Then we can define a VB morphism from $TN$ to $f^\ast A_G^\ast$ covering the identity $\operatorname{id}_N$, again denoted by $\omega\circ df$, simply setting 
		\[
			(\omega\circ df) (w)= (y,\omega(df(w))), \quad w\in T_yN,
		\]
		with $y\in N$. Now we set $\mathcal{H}= \pr_2\circ \mathcal{H'}\circ \omega\circ df$. Then $K'$ is homotopic to $K$ through $\mathcal{H}$. Indeed, for any $w\in T_hH$, with $h\in H$, we have
		\begin{align*}
			(K'-K)(w)&= (\pr_2\circ F \circ f^\ast \omega - \omega\circ df)(w)\\
			&= (\pr_2\circ F \circ df^\ast)(\omega(df(w)))- \omega (df(w))\\
			&=(\pr_2\circ J_{\mathcal{H'}}- \pr_2\circ \operatorname{id}_{f^\ast T^\ast G})(\omega(df(w))) - \omega( df(w)) \\
			&=\pr_2(J_{\mathcal{H}'}(\omega(df(w))))\\
			&= J_{\mathcal{H}}(w).
		\end{align*}
		
		Finally, $ F\colon T^\ast H\to f^\ast T^\ast G$ is a VB-Morita map because of Remark \ref{rem:VB_morita_natural_iso} and Lemma \ref{lemma:two-out-of-three-VBG}. Then $\omega\colon TG\to T^\ast G$ is a VB-Morita map if and only if $K$ is so (Lemma \ref{lemma:two-out-of-three-VBG}), if and only if $K'$ is so (Remark \ref{rem:VB_morita_natural_iso}), if and only if $f^\ast \omega\colon TH\to T^\ast H$ is so (Lemma \ref{lemma:two-out-of-three-VBG} again), whence the claim. 
	\end{proof}

	\subsection{Symplectic Morita equivalence}
	\label{sec:sme}
	In this section we prove that the notion of $+1$-shifted symplectic structures is Morita invariant. In order to do this, following \cite[Section 2.4]{CZ23}, we introduce a suitable notion of \emph{symplectic Morita equivalence} that uses \emph{gauge transformations} of $+1$-shifted symplectic structures. Then in Proposition \ref{prop:symplecticMoritaequiv} we prove that any Lie groupoid Morita equivalent to a $+1$-shifted symplectic groupoid, is itself a $+1$-shifted symplectic groupoid. This allows us to define a $+1$-shifted symplectic stack.
	
	Following \cite[Section 2.4]{CZ23} we introduce gauge transformations of $+1$-shifted symplectic structures. Let $G\rightrightarrows M$ be a Lie groupoid.
	\begin{definition}
		The \emph{gauge transformation} of a $+1$-shifted symplectic structure $(\omega, \Omega)$ on $G\rightrightarrows M$ by a $2$-form $\alpha\in \Omega^2(M)$ is the pair $(\omega + \partial \alpha, \Omega + d\alpha)\in \Omega^2(G)\oplus \Omega^3(M)$.
	\end{definition}
	In the next proposition we prove that a gauge transformation of a $+1$-shifted symplectic structure is again a $+1$-shifted symplectic structure. this is also proved in \cite[Proposition 4.6]{Xu03}, where gauge transformations are called \emph{gauge transformations of the first type}, but we provide a new proof using the different, but equivalent, non-degeneracy condition, i.e., that the cochain map \eqref{eq:nondegenerate_1} is a quasi-isomorphism. The same result is also proved in \cite[Proposition 2.35]{CZ23} for gauge transformations of general shifted symplectic structure on higher Lie groupoids. 
	\begin{prop}
		\label{prop:gauge_transf_1-shift}
		Let $(G\rightrightarrows M, \omega, \Omega)$ be a $+1$-shifted symplectic groupoid. The gauge transformation of $(\omega, \Omega)$ by a $2$-form $\alpha\in \Omega^2(M)$ is a $+1$-shifted symplectic structure on $G$.
	\end{prop}
	\begin{proof}
		Since $\Omega$ is closed, then the $3$-form $\Omega+d\alpha\in \Omega^3(M)$ is closed, and since $\omega$ is multiplicative, then $\omega+\partial \alpha$ is multiplicative. Moreover, $d(\omega+\partial \alpha)= \partial(\Omega+d\alpha)$.

		We need to prove that $\omega+\partial \alpha$ is non-degenerate. Notice that, for any $x\in M$, the cochain map induced by $\omega+\partial \alpha$ between the fibers of $TG$ and $T^\ast G$ over $x$ is
		\begin{equation*}
			\begin{tikzcd}
				0 \arrow[r] & A_x \arrow[r, "\rho"] \arrow[d, "\omega + \partial \alpha"'] & T_xM \arrow[r] \arrow[d, "\omega + \partial \alpha"] &0 \\
				0 \arrow[r] & T_x^{\ast}M \arrow[r, "\rho^{\ast}"'] & A_x^{\ast} \arrow[r] &0
			\end{tikzcd},
		\end{equation*}
		but $\partial \alpha = s^\ast \alpha - t^\ast \alpha \colon A_x \to T_x M$ vanishes on $\ker \rho$, showing that $\omega$ and $\omega+\partial \alpha$ do actually induce the same map in the cohomology of the fibers in degree $-1$, hence, by Remark \ref{rem:quasi_iso_-1_0}, in degree $0$ as well.
	\end{proof}
	
	The next step consists in proving that the pullbacks of $+1$-shifted symplectic structures along Morita maps are again $+1$-shifted symplectic structures. This is proved in \cite[Lemma 2.28]{CZ23} for the more general notion of shifted symplectic higher Lie groupoids and hypercovers instead of Morita maps.
	\begin{prop}
		\label{prop:pullback_1shifted}
		Let $(G\rightrightarrows M, \omega, \Omega)$ be a $+1$-shifted symplectic groupoid and let $f\colon (H\rightrightarrows N) \to (G\rightrightarrows M)$ be a Morita map. Then $(H\rightrightarrows N, f^{\ast}\omega, f^{\ast}\Omega)$ is a $+1$-shifted symplectic groupoid as well.
	\end{prop}
	\begin{proof}
		Since the de Rham differential is natural and $f\colon H\to G$ is a Lie groupoid morphism, we get $df^{\ast}\Omega=0$, $$\partial f^{\ast}\omega= f^{\ast}\partial\omega=0,$$so $f^\ast \omega$ is multiplicative, and
		\begin{equation*}
			df^\ast \omega= f^\ast d\omega= f^\ast \partial \Omega=\partial f^\ast \Omega.
		\end{equation*}
		Finally, by Proposition \ref{prop:+1-shifted_Morita_map}, or Proposition \ref{prop:omega_VB-Morita_if_fomega}, we have that $f^\ast \omega$ is non-degenerate.
	\end{proof}
	In Remark \ref{rem:Morita_maps} we discussed that the Morita equivance can be realized through Morita maps that are surjective submersions on bases, and pullback groupoids (see Example \ref{ex:pullback_groupoid}). There is a version of Proposition \ref{prop:pullback_1shifted} using surjective submersions and pullback groupoids, see \cite[Proposition 4.8]{Xu03}.
		
	Now we discuss Morita equivalence between $+1$-shifted symplectic Lie groupoids. This notion was introduced in \cite[Definition 4.3]{Xu03} using principal bibundles and in \cite[Definition 2.29]{CZ23} for shifted symplectic higher Lie groupoids.
	\begin{definition}
		Two $+1$-shifted symplectic Lie groupoids $(G_1, \omega_1, \Omega_1)$ and $(G_2, \omega_2, \Omega_2)$ are \emph{symplectic Morita equivalent} is there exist a Lie groupoid $H$, and Morita maps
		\begin{equation*}
			\begin{tikzcd}
				& H \arrow[dl, "f_1"'] \arrow[dr, "f_2"] \\
				G_1 & & G_2
			\end{tikzcd}
		\end{equation*}
		such that the $+1$-shifted symplectic structures $(f_1^\ast \omega_1, f_1^\ast \Omega_1), (f_2^\ast \omega_2, f_2^\ast \Omega_2)$ agree up to a gauge transformation.
	\end{definition}
	That the symplectic Morita equivalence is an equivalence relation is proved in \cite[Theorem 4.5]{Xu03} where the equivalence is introduced using bibundles and an appropriate notion of Hamiltonian action for $+1$-shifted symplectic structures. It is also proved in \cite[Proposition 2.31]{CZ23} in the realm of higher Lie groupoids and hypercovers (which, in the case of Lie groupoids, are surjective submersions on bases Morita maps). In the next proposition we recall this result and prove it for a generic Morita map. 
	\begin{prop}
		\label{prop:sympl_Morita_equiv}
		Symplectic Morita equivalence is an equivalence relation.
	\end{prop}
	\begin{proof}
		Reflexivity and simmetry are obvious. For the transitivity, let $(G_1, \omega_1, \Omega_1), (G_2, \omega_2, \Omega_2)$ and $(G_3, \omega_3, \Omega_3)$ be $+1$-shifted symplectic groupoids and let
		\begin{equation*}
			\begin{tikzcd}[column sep=-3]
				& (H_1, \alpha_1) \arrow[dl, "f_1"'] \arrow[dr, "f_2"]  & & (H_2, \alpha_2)   \arrow[dl, "l_1"'] \arrow[dr, "l_2"]  &\\
				(G_1, \omega_1, \Omega_1) & & (G_2, \omega_2, \Omega_2) & & (G_3, \omega_3, \Omega_3)
			\end{tikzcd}
		\end{equation*}
		be two symplectic Morita equivalences, i.e.,
		\begin{equation}
			\label{eq:transitivity}
			\begin{aligned}
				f_2^{\ast}\omega_2 - f_1^\ast\omega_1 &= \partial \alpha_1, \quad f_2^\ast\Omega_2 - f_1^\ast\Omega_1 = d\alpha_1, \\
				l_2^\ast\omega_3 -l_1^\ast\omega_2 &= \partial \alpha_2, \quad
				l_2^\ast \Omega_3 -l_1^\ast\Omega_2 = d\alpha_2.
			\end{aligned}
		\end{equation}
		As discussed in Proposition \ref{prop:Morita_equivalence_relation}, the homotopy fiber product $H_1 \times_{G_2}^h H_2$ of $H_1\xrightarrow{f_2} G_2 \xleftarrow{l_1} H_2$ (see Example \ref{ex:homotopy_pullback}) exists and it comes with two Morita maps $ H_1 \xleftarrow{f} H_1 \times_{G_2}^h H_2 \xrightarrow{l} H_2$ fitting in the following diagram:
		\begin{equation}
			\label{eq:homot_fiber_prod_sympl}
			\begin{tikzcd}[column sep=-3]
				& &H_1 \times_{G_2}^h H_2 \arrow[dl, "f"'] \arrow[dr, "l"]  & & \\
				& \big(H_1, \alpha_1\big) \arrow[dl, "f_1"'] \arrow[dr, "f_2"]  & & \big(H_2, \alpha_2\big)   \arrow[dl, "l_1"'] \arrow[dr, "l_2"]  &\\
				\big(G_1, \omega_1, \Omega_1\big) & & \big(G_2, \omega_2, \Omega_2\big) & & \big(G_3, \omega_3, \Omega_3\big)
			\end{tikzcd},
		\end{equation}
		where every Lie groupoid morphism is a Morita map. The middle square in \eqref{eq:homot_fiber_prod_sympl} commutes up to a natural transformations $\tau\colon f_2\circ f \Rightarrow l_1\circ l$ (see again Example \ref{ex:homotopy_pullback}). Consider $\beta=-\tau^\ast \omega_2$, a $2$-form on the base of the homotopy pullback $H_1 \times_{G_2}^h H_2$.
		
		From the naturality of $\tau$ we have that $f_2\circ f = m\circ \left(i\circ \tau \circ t, m\circ (l_1\circ l, \tau \circ s) \right)$, then, using the multiplicativity of $\omega_2$, it follows that 
		\begin{align*}
			(f_2\circ f)^\ast \omega_2 &= \left(m\circ \left(i\circ \tau \circ t, m\circ (l_1\circ l, \tau \circ s) \right)\right)^\ast \omega_2 \\
			&=\left(i\circ \tau \circ t, m\circ (l_1\circ l, \tau \circ s) \right)^\ast (m^\ast \omega_2)\\ 			
			&= \left(i\circ \tau \circ t, m\circ (l_1\circ l, \tau \circ s) \right)^\ast (\pr_1^\ast\omega_2 +\pr_2^\ast \omega_2)\\
			& =(i\circ \tau\circ t)^\ast \omega_2 + \left(m\circ (l_1\circ l, \tau \circ s)\right)^\ast\omega_2\\
			& = (i\circ \tau\circ t)^\ast \omega_2 + (l_1\circ l, \tau \circ s)^\ast(m^\ast\omega_2)\\
			&=(i\circ \tau\circ t)^\ast \omega_2 + (l_1\circ l, \tau \circ s)^\ast(\pr_1^\ast\omega_2 +\pr_2^\ast\omega_2)\\
			&=(i\circ \tau\circ t)^\ast \omega_2 + (l_1\circ l)^\ast \omega_2 + (\tau\circ s)^\ast \omega_2\\
			&= -(\tau\circ t)^\ast \omega_2  + (l_1\circ l)^\ast \omega_2 + (\tau\circ s)^\ast \omega_2,
		\end{align*}
		where, in the last step, we used that $i^\ast\omega=-\omega$ from Proposition \ref{prop:formule}$.ii)$. Hence
		\begin{equation}
			\label{eq:partialbeta}
			(l_1\circ l)^\ast \omega_2 - (f_2\circ f)^\ast \omega_2 = (\tau\circ t)^\ast \omega_2 -(\tau\circ s)^\ast \omega_2=t^\ast \tau^\ast \omega_2 - s^\ast \tau^\ast \omega_2= \partial \beta.
		\end{equation}
		Moreover, we get 
		\begin{equation}
			\label{eq:differentialbeta}
			\begin{aligned}
				(l_1\circ l)^\ast \Omega_2 - (f_2\circ f)^\ast \Omega_2 &=(t\circ \tau)^\ast \Omega_2 - (s\circ \tau)^\ast \Omega_2 \\
				&=\tau^\ast t^\ast \Omega_2 - \tau^\ast s^\ast \Omega_2 = -\tau^\ast \partial\Omega_2 =- d\tau^\ast \omega_2= d\beta.
			\end{aligned}
		\end{equation}
		By \eqref{eq:transitivity} and \eqref{eq:partialbeta} we have
		\begin{align*}
			(l_2\circ l)^\ast \omega_3 - (f_1\circ f)^\ast \omega_1 &= l^\ast (l_2^\ast\omega_3) - f^\ast(f_1^\ast\omega_1) \\
			&= l^\ast\left(l_1^\ast\omega_2 +\partial \alpha_2\right) - f^\ast \left(f_2^\ast \omega_2 -\partial \alpha_1\right) \\
			&=l^\ast(l_1^\ast \omega_2) + l^\ast (\partial \alpha_2) -f^\ast(f_2^\ast\omega_2) +f^\ast (\partial \alpha_1) \\
			&= (l_1\circ l)^\ast \omega_2 - (f_2\circ f)^\ast \omega_2 + \partial (l^\ast \alpha_2) +\partial (f^\ast \alpha_1) \\
			&=\partial \beta +\partial (l^\ast \alpha_2) + \partial (f^\ast \alpha_1) \\
			&=\partial (\beta + l^\ast\alpha_2 +f^\ast\alpha_1),
		\end{align*}
		and, by \eqref{eq:transitivity} and \eqref{eq:differentialbeta} we get
		\begin{align*}
			(l_2\circ l)^\ast \Omega_3 - (f_1\circ f)^\ast \Omega_1 &= l^\ast (l_2^\ast\Omega_3) - f^\ast(f_1^\ast\Omega_1) \\
			&=l^\ast\left(l_1^\ast\Omega_2 +d \alpha_2\right) - f^\ast \left(f_2^\ast \Omega_2 -d \alpha_1\right) \\
			&=l^\ast(l_1^\ast \Omega_2) + l^\ast (d\alpha_2) -f^\ast(f_2^\ast\Omega_2) +f^\ast (d\alpha_1) \\
			&=(l_1\circ l)^\ast \Omega_2 - (f_2\circ f)^\ast \Omega_2 + d (l^\ast \alpha_2) + d (f^\ast \alpha_1) \\ 
			&=d\beta + d (l^\ast \alpha_2) + d (f^\ast \alpha_1) \\ 
			& = d(\beta + l^\ast \alpha_2 + f^\ast \alpha_1).
		\end{align*}
		Hence the $+1$-shifted symplectic structures $((f_1\circ f)^\ast\omega_1, (f_1\circ f)^\ast \Omega_1)$ and $((l_2\circ l)^\ast\omega_3, (l_2\circ l)^\ast \Omega_3) $ agree up to the gauge transformation by $\beta+l^\ast \alpha_2 + f^\ast \alpha_1$.
	\end{proof}
	\begin{rem}
		Using Morita maps which are surjective submersions on bases the proof of Proposition \ref{prop:sympl_Morita_equiv} simplifies. Indeed, in this case, the middle square in \eqref{eq:homot_fiber_prod_sympl} is commutative and the $+1$-shifted symplectic structures  $((f_1\circ f)^\ast\omega_1, (f_1\circ f)^\ast \Omega_1)$ and $((l_2\circ l)^\ast\omega_3, (l_2\circ l)^\ast \Omega_3)$ agree up to the gauge transformation by $l^\ast \alpha_2 + f^\ast \alpha_1$ (see \cite[Proposition 2.31]{CZ23}).
	\end{rem}
	
	In \cite[Theorem 6.4]{Ma24} the author showed that if a Lie groupoid $G_2$ is Morita equivalent to a $+1$-shifted symplectic groupoid $(G_1, \omega_1, \Omega_1)$, then there exists a $+1$-shifted symplectic structure $(\omega_2,\Omega_2)$ on $G_2$ in such a way that $(G_1, \omega_1, \Omega_1)$ and $(G_2,\omega_2, \Omega_2)$ are symplectic Morita equivalent, and, moreover, such $(\omega_2,\Omega_2)$ is unique up to gauge transformations. %A part of this result was already proved in \cite[Theorem 4.9]{Xu03}, where, more precisely, the author showed that, if $G_1$ and $G_2$ are Morita equivalent Lie groupoids:
	%\begin{equation*}
	%	\begin{tikzcd}
	%		& H \arrow[dr, "f_2"] \arrow[dl, "f_1"'] \\
	%		G_1 & & G_2
	%	\end{tikzcd},
	%\end{equation*} 
	%and $(\omega_i,\Omega_i)\in \Omega^2(G_i)\oplus \Omega^3(M_i)$ is a pair of forms on $G_i$, $i=1,2$, which are closed with respect to the total differential. Then, if one is non-degenerate (and so a $+1$-shifted symplectic structure) and $f_1^\ast (\omega_1,\Omega_1)$ and $f_2^\ast (\omega_2, \Omega_2)$ differ for a gauge transformation, then the other one is non-degenerate (and so a $+1$-shifted symplectic structure), and so $(G_1, \omega_1,\Omega_1)$ and $(G_2,\omega_2,\Omega_2)$ are symplectic Morita equivalent. 
	We now recall \cite[Theorem 6.4]{Ma24} together with its proof as it will be of inspiration for the analogue statement in the contact case.
	\begin{prop}
		\label{prop:symplecticMoritaequiv}
		Let $(G_1, \omega_1, \Omega_1)$ be a $+1$-shifted symplectic groupoid and let $G_2$ a Morita equivalent Lie groupoid. Then there exists a $+1$-shifted symplectic structure $(\omega_2, \Omega_2)$ on $G_2$ such that $(G_1, \omega_1, \Omega_1)$ and $(G_2, \omega_2, \Omega_2)$ are symplectic Morita equivalent. Moreover such $(\omega_2, \Omega_2)$ is unique up to gauge transformations.
	\end{prop}
	\begin{proof}
		The Lie groupoids $G_1$ and $G_2$ are Morita equivalent, then there exist a Lie groupoid $H$ and two Morita maps
		\begin{equation*}
			\begin{tikzcd}
				& H \arrow[dr, "f_2"] \arrow[dl, "f_1"'] \\
				G_1 & & G_2
			\end{tikzcd}.
		\end{equation*}
		We consider the truncated BSS double complex of $H$
		\begin{equation}
			\label{eq:truncatedBSS}
			{\scriptsize
			\begin{tikzcd}
				&\vdots&\vdots 
				\\
				0\arrow[r] &\Omega^4(N) \arrow[u, "d"] \arrow[r, "\partial"] &\Omega^4_{\operatorname{mult}}(H)\arrow[u, "d"] \arrow[r] & 0
				\\
				0\arrow[r] &\Omega^3(N) \arrow[u, "d"] \arrow[r, "\partial"] &\Omega^3_{\operatorname{mult}}(H)\arrow[u, "d"] \arrow[r] & 0
				\\
				0\arrow[r] &\Omega^2(N) \arrow[u, "d"] \arrow[r, "\partial"] &\Omega^2_{\operatorname{mult}}(H)\arrow[u,"d"] \arrow[r] & 0
				\\
				&0 \arrow[u] &0\arrow[u] 
			\end{tikzcd}}
		\end{equation}
		obtained from the double complex \eqref{eq:BSS} by deleting the first two rows and every column except the first two. We denote by $(\ttot(H), D)$ the total complex of the double complex \eqref{eq:truncatedBSS} (likewise for $G_1$ and $G_2$). By Remark \ref{rem:partial_Morita_invariant}, every row in the truncated BSS double complex is Morita invariant up to quasi-isomorphisms, then, from standard spectral sequence arguments, the double complex \eqref{eq:truncatedBSS} and then its total complex are Morita invariant up to quasi-isomorphisms. In other words the pullback maps $$f_1^\ast \colon (\ttot (G_1), D) \to (\ttot (H), D)$$ and $$f_2^\ast \colon (\ttot (G_2), D) \to (\ttot (H), D)$$ are quasi-isomorphisms. The pair $(\omega_1,\Omega_1)$ is a $3$-cocycle in the total complex $(\ttot(G_1), D)$, then there exists a $3$-cocycle $(\omega_2, \Omega_2)\in \Omega^2(G_2)\oplus \Omega^3(M_2)$ in $(\ttot(G_2), D)$ such that the cohomology classes of $(f_1^\ast\omega_1, f_1^\ast \Omega_1)$ and $(f_2^\ast \omega_2, f_2^\ast\Omega_2)$ in $(\ttot(H), D)$ agree. This means that there exist $\alpha\in \Omega^2(N)$ such that
		\[
			f_1^\ast \Omega_1 - f_2^\ast \Omega_2 = d\alpha, \quad f_1^\ast \omega_1 - f_2^\ast \omega_2= \partial \alpha -d\beta.
		\]
%		Since $f_2^\ast \colon (\Omega^\filleddiamond(G_2^{(\bullet)}), \partial)\to (\Omega^\filleddiamond (H^{(\bullet)}), \partial)$ is a quasi-isomorphism, then there exists $\widetilde{\beta}\in \Omega^1(G_2)$, such that $\partial \widetilde{\beta}=0$ and $f_2^\ast \widetilde{\beta}$ and $\beta$ are $\partial$-cohomologous, i.e., there exists $\gamma\in \Omega^1(N)$ such that $f_2^\ast \widetilde{\beta}= \beta + \partial \gamma$.
%		
%		We set $\omega_2=\widetilde{\omega}_2 + d\widetilde{\beta}$ and $\Omega_2=\widetilde{\Omega}_2$. The pair $(\omega_2, \Omega_2)\in \Omega^2(G_2)\oplus\Omega^3(M_2)$ is a $3$-cocycle in $(\ttot (G_2), D)$. Indeed,
%		\[
%			\partial \omega_2=\partial \widetilde{\omega}_2 + d\partial \widetilde{\beta} =0, \quad d\Omega_2= d\widetilde{\Omega}_2=0, \quad d\omega_2= d\widetilde{\omega}_2 = \partial \widetilde{\Omega}_2 = \partial \Omega_2.
% 		\]
		Then  $(f_1^\ast\omega_1, f_1^\ast \Omega_1)$ and $(f_2^\ast \omega_2, f_2^\ast\Omega_2)$ agree up to the gauge transformation by $\alpha$. %Indeed, 
%		\begin{align*}
%			f_1^\ast \omega_1 - f_2^\ast\omega_2 &= f_1^\ast\omega_1 - f_2^\ast \widetilde{\omega}_2 + df_2^\ast \widetilde{\beta} \\
%			&=\partial \alpha -d \beta + d\beta + d \partial \gamma \\
%			&= \partial (\alpha+ d\gamma),
%		\end{align*}
%		and
%		\begin{equation*}
%			f_1^\ast \Omega_1 - f_2^\ast \Omega_2= f_1^\ast \Omega_1 - f_2^\ast \widetilde{\Omega}_2 = d\alpha = d(\alpha + d \gamma).
%		\end{equation*}
		Since $(\omega_1,\Omega_1)$ is a $+1$-shifted symplectic structure, then, by Proposition \ref{prop:pullback_1shifted}, $(f_1^\ast\omega_1, f_1^\ast\Omega_1)$ is so, by Proposition \ref{prop:gauge_transf_1-shift}, $(f_2^\ast\omega_2, f_2^\ast\Omega_2)$ is so, by Proposition \ref{prop:+1-shifted_Morita_map}, $(\omega_2, \Omega_2)$ is so, and, the $+1$-shifted symplectic groupoids $(G_1,\omega_1,\Omega_1)$ and $(G_2, \omega_2,\Omega_2)$ are symplectic Morita equivalent.
		
		For the second part of the statement let $(\omega_2', \Omega_2')$ be another $+1$-shifted symplectic structure on $G_2$ such that $G_1\xleftarrow{f_1} H \xrightarrow{f_2} G_2$ is a symplectic Morita equivalence between $(G_2, \omega_2', \Omega_2')$ and $(G_1,\omega_1, \Omega_1)$. Then $(\omega_2, \Omega_2)$ and $(\omega_2', \Omega_2')$ are in the same cohomology class of the total complex $(\ttot(G_2), D)$. This means that there exist $\alpha\in \Omega^2(M_2)$ such that
		\[
			\omega_2-\omega_2'= \partial \alpha, \quad \text{and}\quad \Omega_2-\Omega_2'=d \alpha.
		\]
		Hence $(\omega_2,\Omega_2)$ and $(\omega_2', \Omega_2')$ agree up to the gauge transformation by $\alpha$.
	\end{proof}
	\begin{rem}
		A similar statement as Proposition \ref{prop:symplecticMoritaequiv} is proved in \cite[Theorem 3.12]{BCGX22} for \emph{$+1$-shifted Poisson structures}.
	\end{rem}
	
	Proposition \ref{prop:symplecticMoritaequiv} motivates the following 
	\begin{definition}
		A \emph{$+1$-shifted symplectic stack} is a symplectic Morita equivalence class of $+1$-shifted symplectic groupoids. 
	\end{definition}
	
	\begin{rem}
		It is clear by the proof of Proposition \ref{prop:symplecticMoritaequiv} that the pullback of a Morita map $f\colon H\to G$ determines a bijection between the cohomology classes in the total complex of the truncated double complex \eqref{eq:truncatedBSS} of $+1$-shifted symplectic structures on $G$ and $H$. Moreover, two $+1$-shifted symplectic groupoids $(G_1,\omega_1,\Omega_1)$ and $(G_2,\omega_2, \Omega_2)$ are symplectic Morita equivalent if and only if $G_1$ and $G_2$ are Morita equivalent and this equivalence maps the cohomology class $[(\omega_1, \Omega_1)]$ in the $(\ttot(G_1), D)$ to the cohomology class $[(\omega_2,\Omega_2)]$ in $(\ttot(G_2), D)$. Hence, a $+1$-shifted symplectic structure on a stack $[M/G]$ can be considered as a cohomology class $[(\omega,\Omega)]$ in $(\ttot(G), D)$ of a $+1$-shifted symplectic structure $(\omega,\Omega)$ on an Lie groupoid $G$ presenting $[M/G]$.
	\end{rem}
	
%	\begin{prop}
%		Two $+1$-shifted symplectic groupoids $(G_1,\omega_1,\Omega_1)$, $(G_2,\omega_2,\Omega_2)$ are symplectic Morita equivalent if and only if there is a Morita equivalence 
%		\[
%			\begin{tikzcd}
%				& H\arrow[dl, "f_1"'] \arrow[dr, "f_2"] \\
%				G_1 && G_2
%			\end{tikzcd}
%		\]
%		such that the composition between . 
%	\end{prop}
%	\begin{proof}
%		If $(G_1,\omega_1,\Omega_1)$ and $(G_2,\omega_2, \Omega_2)$ are symplectic Morita equivalent, then there exist a Morita equivalence $G_1\leftarrow H\rightarrow G_2$ and a $2$-form $\alpha\in \Omega^2(H)$ such that $(f_1^\ast\omega_1, f_1^\ast\Omega_1)$ and $(f_2^\ast\omega_2, f_2^\ast\Omega_2)$ are gauge equivalent by $\alpha$. In particular $f_1^\ast \omega_1$ and $f_2^\ast\omega_2$ are $\partial$-cohomologous. Then the isomorphism coming from $f_1^\ast$ maps $[\omega_1]$ to $[f_1^\ast\omega_1]=[f_2^\ast\omega_2]$ and the one coming from $f_2^\ast$ maps $[\omega_2]$ to $[f_2^\ast\omega_2]$.
%		
%		On the other hand, if $G_1\leftarrow H\rightarrow G_2$ is a Morita equivalence between $g_1$ and $G_2$ such that the inverse of the isomorphism coming from $f_1^\ast\colon (\Omega^2(G_1), \partial)\to (\Omega^2(H), \partial)$ and the isomorphism coming from $f_2^\ast\colon (\Omega^2(G_2), \partial)\to (\Omega^2(H), \partial)$ maps the $\partial$-cohomology class $[\omega_2]$ to the $\partial$-cohomology class $[\omega_1]$, then $f_1^\ast\omega_1$ and $f_2^\ast\omega_2$ are $\partial$-cohomologous. This meas that there exists $\alpha\in \Omega^2(H)$ such that $f_2^\ast\omega_2-f_1^\ast\omega_1 = \partial \alpha$.
%	\end{proof}

	\subsection{Twisted Dirac structures}\label{sec:twisted}
 	In \cite{BCWZ04} the authors showed that $+1$-shifted symplectic structures are the global counterparts of \emph{twisted Dirac structures}. In this subsection we recall what a twisted Dirac structure is and some examples. We start recalling plain Dirac manifolds. In order to do that we follow \cite[Section 3]{Bu13}. Let $M$ be a manifold. The \emph{generalized tangent bundle} $\mathbbm{T}M:= TM\oplus T^\ast M$ comes together with the simmetric pairing
	\begin{equation}
		\label{eq:pairing}
		\la - , - \ra\colon  \mathbbm{T}M \otimes \mathbbm{T}M \to \mathbbm{R}, \quad \la (X,\alpha), (Y, \beta)\ra =\alpha(Y) + \beta(X),
	\end{equation}
	and a bracket on sections
	\begin{equation}
		\label{eq:courant}
		\leftq-,-\rightq\colon \Gamma(\mathbbm{T}M)\times \Gamma(\mathbbm{T}M)\to \Gamma(\mathbbm{T}M), \quad \leftq (X,\alpha), (Y,\beta)\rightq = ([X,Y], \mathcal{L}_X \beta - \iota_Y d\alpha),
	\end{equation}
	called the \emph{Courant bracket}. With these structures $\mathbbm{T}M$ is a \emph{Courant algebroid}
	\begin{definition}\label{def:dirac}
		A \emph{Dirac structure} on $M$ is a subbundle $\mathbbm{L}\subseteq \mathbbm{T}M$ of the generalized tangent bundle, such that
		\begin{itemize}
			\item $\mathbbm{L}$ is Lagrangian with respect to the pairing $\la -, -\ra$;
			\item $\mathbbm{L}$ is involutive with respect to the Courant bracket $\leftq -,-\rightq$.
		\end{itemize}
	\end{definition}
	Dirac structures were introduced in \cite{Co90, CW88} in order to generalize Poisson and presymplectic structures, where for the latter we mean closed $2$-forms.
	\begin{example}[Poisson structures]
		\label{ex:Poisson_dirac}
		Let $M$ be a manifold. For any bivector field $\pi\in \Gamma(\wedge^2TM)$, the graph of $\pi^\sharp \colon T^\ast M \to TM$, $\alpha \mapsto \iota_\alpha \pi$, given by
		\begin{equation*}
			\mathbbm{L}_\pi:=\left\{\left(\alpha, \pi^\sharp(\alpha)\right) \, | \, \alpha \in T^{\ast}M \right\}\subseteq \mathbbm{T}M,
		\end{equation*}
		is a subbundle of the generalized tangent bundle that is Lagrangian with respect to the pairing \eqref{eq:pairing}. Moreover, $\mathbbm{L}_\pi$ is involutive with respect to the Courant bracket \eqref{eq:courant}, and so it is a Dirac structure, if and only if $\pi$ is a Poisson bivector field. Actually, Poisson structures can be identified with Dirac structures $\mathbbm{L}\subseteq \mathbbm{T}M$ with the additional property that $\mathbbm{L}\cap TM=\{0\}$.
	\end{example}
	\begin{example}[Presymplectic structures]
		\label{ex:presym_dirac}
		Let $M$ be a manifold. For any $2$-form $\omega\in \Omega^2(M)$, the graph of $\omega \colon TM\to T^\ast M$, given by
		\begin{equation*}
			\mathbbm{L}_\omega:=\left\{\left(\omega(X), X\right) \, | \, X \in TM \right\} \subseteq \mathbbm{T}M,
		\end{equation*}
		is a subbundle of the generalized tangent bundle that is Lagrangian with respect to the pairing \eqref{eq:pairing}. Moreover, $\mathbbm{L}_\omega$ is involutive with respect to the Courant bracket \eqref{eq:courant}, and so it is a Dirac structure, if and only if $\omega$ is closed. Actually, presymplectic structures can be identified with Dirac structures $\mathbbm{L}\subseteq \mathbbm{T}M$ with the additional property $\mathbbm{L}\cap T^\ast M=\{0\}$.
	\end{example}
	Another example of Dirac structures is given by regular foliations.
	\begin{example}[Regular foliations]
		Let $M$ be a manifold. For any regular distribution $F\subseteq TM$, let
		\begin{equation*}
			\mathbbm{L}= F\oplus F^0\subseteq \mathbbm{T}M,
		\end{equation*}
		where $F^0\subseteq T^\ast M$ is the annihilator of $F$, is a subbundle of the generalized tangent bundle that is Lagrangian with respect to the pairing \eqref{eq:pairing}. Moreover, $\mathbbm{L}$ is involutive with respect to the Courant bracket \eqref{eq:courant}, and so it is a Dirac structure, if and only if $F$ is involutive.
	\end{example}
	
	Let $\mathbbm{L}\subseteq \mathbbm{T}M$ be a Dirac structure on the manifold $M$. The vector bundle $\mathbbm{L}\to M$ comes together with a Lie algebroid structure: the anchor map is given by the restriction to $\mathbbm{L}$ of the projection $\mathbbm{T}M\to TM$ onto the first factor, and the Lie bracket is simply the restriction of the Courant bracket to $\Gamma(\mathbbm{L})$.
	
	One can consider a more general Courant algebroid structure on $\mathbbm{T}M$. Namely, let $\Omega^3(M)$ be a closed $3$-form on the manifold $M$. The Courant bracket \eqref{eq:courant} can be modified in the following way: $\leftq-,-\rightq \leadsto \leftq -,-\rightq_\Omega$ where
	\begin{equation*}
		\leftq-,-\rightq_\Omega\colon \Gamma(\mathbbm{T}M)\times \Gamma(\mathbbm{T}M)\to \Gamma(\mathbbm{T}M), 
	\end{equation*}
	is given by
	\begin{equation}
		\label{eq:courant_twisted}
		\leftq (X,\alpha), (Y,\beta)\rightq_\Omega = \left([X,Y], \mathcal{L}_X \beta - \iota_Y d\alpha + \iota_X\iota_Y\Omega\right).
	\end{equation}
	\begin{definition}
		An \emph{$\Omega$-twisted Dirac structure} on $M$ is a subbundle $\mathbbm{L}\subseteq \mathbbm{T}M$ of the generalized tangent bundle, such that
		\begin{itemize}
			\item $\mathbbm{L}$ is Lagrangian with respect to the pairing $\la -, -\ra$;
			\item $\mathbbm{L}$ is involutive with respect to the $\Omega$-twisted Courant bracket $\leftq -,-\rightq_\Omega$.
		\end{itemize}
	\end{definition}
	\begin{example}
		As explained in Example \ref{ex:Poisson_dirac}, any bivector field $\pi\in \Gamma(\wedge^2TM)$ on a manifold $M$ determines a Lagrangian subbundle $\mathbbm{L}_\pi\subseteq \mathbbm{T}M$ of the generalized tangent bundle. Let $\Omega\in \Omega^3(M)$ be a closed $3$-form on $M$. It is easy to see that $\mathbbm{L}_\pi$ is an $\Omega$-twisted Dirac structure if and only if $[\pi,\pi]=2 \wedge^3 \pi^\sharp(\Omega)$. These structures are studied in \cite{SW01}, where they are called \emph{twisted Poisson structures}.
	\end{example}
	\begin{example}
		As explained in Example \ref{ex:presym_dirac}, any $2$-form $\omega\in \Omega^2(M)$ on a manifold $M$ determines a Lagrangian subbundle $\mathbbm{L}_\omega\subseteq \mathbbm{T}M$ of the generalized tangent bundle. Let $\Omega\in \Omega^3(M)$ be a closed $3$-form on $M$. It is easy to see that $\mathbbm{L}_\omega$ is an $\Omega$-twisted Dirac structure if and only if $d\omega=\Omega$. %\textcolor{red}{referenza}
	\end{example}
	As for Dirac structures, any $\Omega$-twisted Dirac structure $\mathbbm{L}$ on a manifold $M$ possesses a Lie algebroid structure: the anchor map is again the restriction of the $TM$-projection to $\mathbbm{L}$, and the Lie bracket is the restriction to $\Gamma(\mathbbm{L})$ of the $\Omega$-twisted Courant bracket. The integration problem for twisted Dirac structures has been discussed in \cite{BCWZ04}. 
	
	\begin{example}[Cartan-Dirac structure]
		From \cite[Example 4.2]{SW01} any Lie group with a bi-invariant non-degenerate inner product determines a twisted Dirac structure. Let $G$ be a Lie group and let $(-,-)_{\mathfrak{g}}$ be a bi-invariant non-degenerate inner product on the Lie algebra $\mathfrak{g}$ of $G$. We denote by $\theta$ and $\widetilde{\theta}$ the left and right \emph{Maurer-Cartan forms} on $G$ respectively, i.e., $\theta=g^{-1}dg$ and $\widetilde{\theta}= dgg^{-1}$. The \emph{Cartan $3$-form} is the bi-invariant form 
		\begin{equation}
			\label{eq:cartan3form}
			\Omega=\frac{1}{12}(\theta, [\theta, \theta])_{\mathfrak{g}} = \frac{1}{12}(\widetilde{\theta}, [\widetilde{\theta}, \widetilde{\theta}])_{\mathfrak{g}} \in \Omega^3(G).
		\end{equation}
		For any $a\in \mathfrak{g}$, we consider the sections of $\mathbbm{T}G$, given by
		\begin{equation}
			\label{eq:dirac_sections}
			\left(\overleftarrow{a}-\overrightarrow{a}, \frac{1}{2}(\overleftarrow{a}-\overrightarrow{a})\right),
		\end{equation}
		where $\overleftarrow{a}$ and $\overrightarrow{a}$ are the left and right-invariant vector fields corresponding to $a\in \mathfrak{g}$ respectively (see Definition \ref{def:right_invariant_vector_fields}), and we are using the non-degenerate inner product to identify $TG$ with $T^{\ast}G$. The value of the sections \eqref{eq:dirac_sections} form a subbundle $\mathbbm{L}\subseteq \mathbbm{T}G$ that is Lagrangian with respect to \eqref{eq:pairing} and involutive with respect to the $\Omega$-twisted Courant bracket $\leftq-,-\rightq_{\Omega}$. Hence $\mathbbm{L}$ is an $\Omega$-twisted Dirac structure on $G$ that, following \cite{BCWZ04}, is called the \emph{Cartan-Dirac structure} on $G$.
	\end{example}
	
	As showed in \cite[Section 7.2]{BCWZ04} the Cartan-Dirac structure integrates to the \emph{AMM-groupoid} (see \cite[Section 2.3]{Xu03}): the Lie group $G$ acts on itself by conjugation, then we can consider the action groupoid $G\ltimes G$ (Example \ref{ex:action_groupoid}). The latter comes with a $+1$-shifted symplectic structure: the $3$-form on $G$ is the Cartan $3$-form \eqref{eq:cartan3form} and the $2$-form on $G\times G$ is given by
	\begin{equation*}
		\omega_{(g,x)}= -\frac{1}{2} \left((Ad_x\pr_1^{\ast}\theta, \pr_1^{\ast}\theta) + (\pr_1^{\ast}\theta, \pr_2^{\ast} (\theta+\widetilde{\theta})\right), \quad (g,x)\in G\times G,
	\end{equation*}
	where $Ad$ is the adjoint action on $\mathfrak{g}$ and $\pr_i\colon G\times G \to G$ is the projection on the $i$-th factor, $i=1,2$. We recall that the AMM-groupoid encodes the Lie group valued momentum map theory of Alekseev–Malkin–Meinrenken \cite{AMM98}.

	\chapter{Shifted symplectic Atiyah structures}\label{ch:ssas}
We begin this chapter by recalling the definition of a contact structure on a manifold and introducing a symplectic-inspired perspective on Contact Geometry, expressed through a dictionary that connects Symplectic Geometry to Contact Geometry. Then, in the other sections, we apply this dictionary to the definitions of $0$ and $+1$-shifted symplectic structures in order to get a preliminary definition of shifted contact structure.

\section{Contact structures and Symplectic-to-Contact Dictionary}
\label{sec:c_manifolds}

Contact structures are in many respects the odd-dimensional counterparts of symplectic structures. In this section we recall the definition of contact structures on manifolds and we present a symplectic-like point of view on Contact Geometry.

Let $M$ be a manifold. 
\begin{definition}\label{def:contact_manifold}
	A \emph{contact structure} on $M$ is a hyperplane distribution $K\subseteq TM$ which is \emph{maximally non-integrable}, i.e., the \emph{curvature}
	\begin{equation*}
		R_K\colon \wedge^2K \to TM/K, \quad (X,Y)\mapsto [X,Y] \operatorname{mod} K,
	\end{equation*}
	is non-degenerate.
\end{definition} 
Notice that, in Definition \ref{def:contact_manifold}, since $K$ has codimension $1$, the normal bundle $TM/K$ is a well defined line bundle over $M$ and so it makes sense to require that $R_K$ is non-degenerate. 

Dually, let $L\to M$ be a line bundle over $M$.
\begin{definition}
	\label{def:contact_form}
	A \emph{contact form} on $L$ is an $L$-valued differential $1$-form $\theta\in \Omega^1(M,L)$ such that $\theta$ is nowhere-zero and its curvature
	\begin{equation*}
		R_{\theta}\colon \wedge ^2 K_{\theta} \to L, \quad (X,Y)\mapsto -\theta([X,Y])
	\end{equation*}
	is non-degenerate, where $K_{\theta}=\ker\theta$.
\end{definition}  
Notice that, in Definition \ref{def:contact_form}, since $\theta$ is nowhere-zero, $K_{\theta}$ is a well defined hyperplane distribution on $M$.

\begin{rem}
	The notions of contact structures and contact forms are equivalent up to line bundle automorphisms. Indeed, if $K$ is a hyperplane distribution, then the projection on the normal line bundle $TM\to TM/K$ is a nowhere-zero line bundle-valued differential $1$-form. Conversely, if $\theta\in \Omega^1(M,L)$ is a nowhere-zero $L$-valued $1$-form, then its kernel $K_{\theta}$ is a hyperplane distribution. Under this correspondence we get that $R_{K_{\theta}} = -R_{\theta}$, then the non-degeneracy conditions for the two curvatures agree. Notice that, when two contact $1$-forms differ for a nowhere-zero function then the kernels of the $1$-forms (and so the contact structures induced) agree. But, any nowhere-zero function determines a line bundle automorphism. Because of this equivalence we will often use the terminology ``contact structure'' for both of the notions. 
\end{rem}

\begin{rem}
	Usually, in the literature, the terminology ``contact form'' is used for an ordinary $1$-form $\theta\in\Omega^1(M)$ such that $\ker \theta$ is a contact structure. This is the case when the line bundle in Definition \ref{def:contact_form} is the trivial one, i.e., $L\cong \mathbbm{R}_M=M\times \mathbbm{R}$. Contact structure coming from the kernel of a plain $1$ form are usually called \emph{coorintable}. For conceptual reasons it seems better to consider the case of a generically non-trivial line bundle $L$. Indeed, in this case, we consider not just coorientable structures. An example of non coorientable structure is given by the canonical constant structure on the jet bundle discussed in Example \ref{ex:jet} below.
\end{rem}

For each contact form $\theta\in \Omega^1(M,L)$, it is possible to extend its curvature $R_{\theta}$ to an $L$-valued differential $2$-form on $M$ using a connection. Indeed, if $\nabla$ is a connection on $L$, then the curvature $R_{\theta}$ agrees with the restriction of $d^{\nabla}\theta\colon \wedge^2 TM \to L$ to $K_\theta$, where $d^{\nabla}$ is the connection differential associated to $\nabla$.

We now present the canonical example of contact structure, i.e., the analogous in Contact Geometry of the canonical symplectic structure on the cotangent bundle.
\begin{example}[{\cite[Example 4.6]{To17}}] \label{ex:jet}
	Let $L\to M$ be a line bundle and let $J^1L$ be the first jet bundle of $L$. We denote by $\pi\colon J^1L \to M$ and $\pr \colon J^1L\to L$ the canonical projections. The line bundle $\pi^\ast L\to J^1L$ comes together with a canonical contact structure $\theta\in \Omega^1(J^1L,\pi^\ast L)$. The $1$-form $\theta$ is defined by setting
	\begin{equation}
		\label{eq:can_form}
		\theta(\xi):=(d\pr -d\lambda\circ d\pi)(\xi), \quad \xi\in T_\alpha J^1L,
	\end{equation}
	where $\alpha=j^1_x\lambda \in J^1L$, with $x=\pi(\alpha)\in M$, and $\lambda \in \Gamma(L)$. The right hand side in \eqref{eq:can_form} is a tangent vector to $L$ at $\pr(\alpha)= \lambda(x)$, which is vertical which respect to the line bundle projection $L\to M$. So it identifies with an element of $L_x=(\pi^\ast L)_\alpha$, and further as it is to see, e.g., working in local trivialization charts, the right hand side in \eqref{eq:can_form} does not depend on the choice of $\lambda\in \Gamma(L)$ such that $j^1_x\lambda=\alpha$. The kernel of $\theta$ is a hyperplane distribution called the \emph{Cartan distribution} on $J^1L$.
\end{example}

\subsection{Symplectic-to-Contact Dictionary}
\label{sec:dictionary}
We are interested in the Atiyah algebroid of a line bundle $L\to M$ (see Section \ref{sec:Atiyah_algebroid}). In this case the vector bundle $\operatorname{End}(L)$ is the trivial line bundle $\mathbbm{R}_M$ and Sequence \eqref{eq:Spencer} reduces to
\begin{equation}\label{eq:Spencer_L}
	\begin{tikzcd}
		0 \arrow[r] &\mathbbm R_M \arrow[r] &DL \arrow[r, "\sigma"] &TM \arrow[r] &0
	\end{tikzcd}.
\end{equation}
For any connection $\nabla$ on $L$, the associated left splitting $f_{\nabla}\colon DL \to \mathbbm{R}_M$ is a fiber-wise linear function on $DL$. As discussed at the end of Section \ref{sec:Atiyah_algebroid}, $DL \cong \operatorname{Hom}(J^1L,L)$, and so $J^1L\cong \operatorname{Hom}(DL,L)$, where $J^1L$ is the first jet bundle of $L$.

Atiyah forms on $L$ have an alternative description that is often very useful. First of all, notice that, for any $k$, the map $\sigma^{\ast}\colon \Omega^k(M,L) \to \OA^k(L)$ defined by
\begin{equation*}
	(\sigma^{\ast}\theta)(\Delta_1, \dots, \Delta_k) =\theta\big(\sigma(\Delta_1), \dots, \sigma(\Delta_k)\big) \quad \text{for all } \Delta_1, \dots ,\Delta_k \in \Gamma(DL), 
\end{equation*}
is injective because of the surjectivity of $\sigma$.
\begin{rem}
	\label{rem:sigma}
	The map $\sigma^\ast\colon \Omega^k(M,L)\to \OA^k(L)$ satisfies $\im \sigma^{\ast}= \ker \iota_{\mathbbm{I}}$. Indeed, for any $k$, an Atiyah $k$-form $\omega\in \Omega^k(L)$ is in $\im \sigma^\ast$ if and only if there exists an $L$-valued $k$-form $\eta\in \Omega^k(M,L)$ such that $\omega= \eta \circ \wedge^\bullet\sigma$, with $\wedge^\bullet\sigma\colon \wedge^\bullet DL\to \wedge^\bullet TM$, the algebraic map induced by $\sigma$. The latter is true if and only if the restriction of $\omega$ to $\ker (\wedge^\bullet\sigma)$ is zero. However, $\ker (\wedge^\bullet\sigma)$ is the ideal generated by the identity derivation $\mathbbm{I}$, whence the claim.
\end{rem}
From the injectivity of $\sigma^\ast $ and Remark \ref{rem:sigma}, the following is a short exact sequence of $C^{\infty}(M)$-modules
\begin{equation}
	\label{ses_atiyah}
	\begin{tikzcd}
		0 \arrow[r] &\Omega^{\bullet}(M,L) \arrow[r, "\sigma^{\ast}"] &\OA^{\bullet}(L) \arrow[r] &\Omega^{\bullet-1}(M,L)\arrow[r] &0,
	\end{tikzcd}
\end{equation}
where the projection $\OA^{\bullet}(L)\to \Omega^{\bullet -1}(M,L)$ maps an Atiyah form $\omega \in \OA^{\bullet}(L)$ to the form $\omega_0\in \Omega^{\bullet -1}(M,L)$ uniquely determined by $\sigma^{\ast}\omega_0=\iota_{\mathbbm{I}}\omega$. The map 
\[
\Omega^{\bullet -1}(M,L)\to \OA^{\bullet}(L), \quad \tau\mapsto \dA(\sigma^{\ast}\tau),
\]
is a canonical $\mathbbm{R}$-linear right splitting of \eqref{ses_atiyah}. Accordingly, there is an $\mathbbm{R}$-linear isomorphism
\begin{equation}
	\label{eq:components}
	\Omega^{\bullet-1}(M,L)\oplus \Omega^{\bullet}(M,L) \to \OA^{\bullet}(L), \quad
	(\omega_0, \omega_1) \mapsto \omega:= \dA\sigma^{\ast}\omega_0 + \sigma^{\ast}\omega_1.
\end{equation}
The two forms $\omega_0$ and $\omega_1$ will be called the \emph{components} of $\omega$, and we write $\omega \rightleftharpoons (\omega_0, \omega_1)$.
\begin{rem}
	\label{rem:componentes_omega_closed}
	If $\omega\rightleftharpoons (\omega_0,\omega_1)$, then $\dA\omega \rightleftharpoons (\omega_1,0)$. Indeed, by Equation \eqref{eq:components}, $\omega= \dA\sigma^{\ast}\omega_0 + \sigma^{\ast}\omega_1$ and so $\dA\omega= \dA\sigma^\ast \omega_1$. In particular $\omega$ is $\dA$-closed if and only if $\omega_1=0$.
\end{rem}

Let $(F, f) \colon (L \to M) \to (L' \to N)$ be a regular VB morphism between line bundles (see Remark \ref{rem:regular_VB_morphisms} for the definition of regular VB morphism). In this case we call $(F, f)$ an \emph{LB morphism} (LB for line bundle). It is easy to see that
\begin{equation*}
	F^{\ast}(\sigma^{\ast}\theta)= \sigma^{\ast}(F^{\ast}\theta), \quad \theta\in \Omega^{k}(N,L').
\end{equation*}
Namely, let $\delta_1, \dots , \delta_k\in D_xL$, from \eqref{eq:DFcommutes} we get
\begin{align*}
	F^{\ast}(\sigma^{\ast} \theta) (\delta_1, \dots, \delta_k) &= F_x^{-1}\left((\sigma^{\ast}\theta) (DF(\delta_1), \dots, DF(\delta_k))\right) \\
	&= F_x^{-1}\left(\theta(\sigma(DF(\delta_1)), \dots, \sigma(DF(\delta_k)))\right) \\
	&= F_x^{-1}\left(\theta(df(\sigma(\delta_1)), \dots, df(\sigma(\delta_k)))\right) \\
	&= (F^{\ast}\theta)\left(\sigma(\delta_1), \dots, \sigma(\delta_k)\right) \\
	&= \sigma^{\ast}(F^{\ast}\theta)(\delta_1, \dots, \delta_k).
\end{align*}
It follows that, if $\omega\rightleftharpoons (\omega_0, \omega_1)$ is an Atiyah form on $L'$, then 
\begin{equation}
	\label{eq:Atiyah_components}
	F^{\ast}\omega\rightleftharpoons (F^{\ast}\omega_0, F^{\ast}\omega_1).
\end{equation}

The next proposition provides another description of the relation between an Atiyah $2$-form and its components.
\begin{prop}\label{prop:ker_omega_comp}
	Let $\omega\in \OA^2(L)$ be an Atiyah $2$-form and let $(\omega_0,\omega_1)\in \Omega^1(M,L)\oplus\Omega^2(M,L)$ be its components. Let $\delta, \delta'\in D_xL$, and set $v=\sigma(\delta), v'=\sigma(\delta')\in T_xM$, $x \in M$. Then, for any connection $\nabla$ on $L$,
	\begin{equation*}
		\omega(\delta, \delta')= \omega_1(v, v') + d^{\nabla}\omega_0(v,v') +f_{\nabla}(\delta) \omega_0(v') - f_{\nabla}(\delta') \omega_0(v),
	\end{equation*}
	where $d^\nabla \colon \Omega^\bullet (M, L) \to \Omega^{\bullet +1} (M, L)$ is the connection differential. In particular, when $\omega$ is $\dA$-closed (i.e.~$\omega_1=0$), we have
	\begin{equation}
		\label{eq:omegaandcomponents}
		\omega(\delta, \delta')=d^{\nabla}\omega_0(v,v') +f_{\nabla}(\delta) \omega_0(v') - f_{\nabla}(\delta') \omega_0(v).
	\end{equation}
\end{prop}
\begin{proof}
	Let $\delta, \delta', v, v'$ be as in the statement, and let $\Delta,\Delta'\in \Gamma(DL)$ be such that $\Delta_x=\delta,\Delta'_x=\delta'$. Set $V := \sigma (\Delta),V' := \sigma (\Delta') \in \mathfrak{X}(M)$ so that $V_x = v, V'_x = v'$. Then
	\begin{align*}
		\dA\sigma^{\ast}\omega_0(\delta,\delta')&=\delta(\sigma^{\ast}\omega_0(\Delta')) - \delta'(\sigma^{\ast}\omega_0(\Delta)) - \sigma^{\ast}\omega_0([\Delta, \Delta']_x) \\
		&=\delta(\omega_0(V')) - \delta'(\omega_0(V)) - \omega_0([V,V']_x)\\
		&= (\delta-\nabla_v)(\omega_0(v')) - (\delta'-\nabla_{v'})(\omega_0(v))\\
		&\quad + \nabla_v(\omega_0(V')) -\nabla_{v'}(\omega_0(V)) -\omega_0([V,V']_x)\\
		&=d^{\nabla}\omega_0(v,v') +f_{\nabla}(\delta)\omega_0(v') - f_{\nabla}(\delta')\omega_0(v).
	\end{align*}
	Now, the claim follows from $\omega=\sigma^{\ast}\omega_1+\dA\sigma^{\ast}\omega_0$.
\end{proof}

The following corollary is an easy consequence of Proposition \ref{prop:ker_omega_comp}.
\begin{coroll}
	\label{coroll:kernelofAtiyahform}
	Let $\omega\in \OA^2(L)$ and $\omega\rightleftharpoons(\omega_0,\omega_1)$. Then a derivation $\delta\in D_xL$ is in the kernel of $\omega$ if and only if
	\begin{itemize}
		\item[\emph{i)}] $\sigma(\delta)\in \ker(\omega_0)$, and
		\item[\emph{ii)}] $\iota_{\sigma(\delta)}(\omega_1+d^{\nabla}\omega_0) + f_{\nabla}(\delta)\omega_0=0$
	\end{itemize}
	for some, hence any, connection $\nabla$ on $L$.
\end{coroll}

Let $L \to M$ be a line bundle.
\begin{definition}
	A \emph{symplectic Atiyah form} is an Atiyah $2$-form $\omega\in \OA^2(L)$ such that $\omega$ is \emph{$\dA$-closed}, i.e., $\dA \omega=0$, and \emph{non-degenerate}, i.e., its flat map, again denoted by $\omega \colon DL\to J^1L$, $\delta \mapsto \omega (\delta, -)$ is a vector bundle isomorphism.
\end{definition} 
There is a relationship between symplectic Atiyah forms and contact structures \cite{V18, VW20} given by the following
\begin{theo}[{\cite[Proposition 3.6]{V18}}]
	\label{prop:corrispondenza}
	The assignment $\theta\mapsto \omega\rightleftharpoons (\theta,0)$ establishes a bijection between contact forms and symplectic Atiyah forms.
\end{theo}

\begin{rem}
	\label{rem:Acoordinates}
	In what follows it will be sometimes useful to express an Atiyah form in coordinates. In order to do that we consider $(U,x^1, \dots, x^n)$ a chart on $M$ around a point $x\in M$ and le t$\lambda$ be a nowhere-zero section of $L$ locally defined on $U$. Consider the derivations $\delta_i^\lambda\in DL$, $i=1,\dots, n$, defined by setting 
	\[
	\delta_i^\lambda(f\lambda)= \frac{\partial f}{\partial x^i}\lambda,
	\]
	for all $f\in C^\infty(U)$. In other words $\delta_i^\lambda$ is the derivation whose symbol is $\tfrac{\partial}{\partial x^i}$ and whose action on $\lambda$ is zero. The derivations $\delta_i^\lambda$ together with the identity derivation $\mathbbm{I}\in D_xL$ form a local frame of $DL$. Hence a local frame for $(DL)^\ast$ is given by $\mathbbm{I}^\ast$ and $(\delta_i^\lambda)^\ast$. Moreover, for any $i=1,\dots,n$, $(\delta_i^\lambda)^\ast$ agrees with the pullback $\sigma^\ast(dx^i)$ of $dx^i$ along the symbol map. Summarizing
	\[
	\big(\mathbbm{I}^\ast, \sigma^\ast(dx^1), \dots, \sigma^\ast(dx^n)\big)
	\]
	is a local frame of $(DL)^\ast$ and an Atiyah $k$-form $\omega\in \OA^k(L)= \Gamma(\wedge^k(DL)^\ast\otimes L)$ locally is given by
	\begin{align*}
		\omega&= \left(\sum A_{i_1, \dots, i_k} \sigma^\ast (dx^{i_1})\wedge \dots\wedge \sigma^\ast(dx^{i_k}) \right. \\
		&\quad \left. + \sum B_{i_1,\dots ,i_{k-1}} \mathbbm{I}^\ast \wedge\sigma^\ast(dx^{i_1})\wedge \dots\wedge \sigma^\ast(dx^{i_{k-1}}) \right)\otimes \lambda,
	\end{align*}
	where $A_{i_1,\dots, i_k}, B_{i_1,\dots ,i_{k-1}}\in C^\infty(U)$. %It is clear that the components $\omega_0$ and $\omega_1$ of $\omega$ are locally given by
%	\[
%	\omega_0=\left(\sum B_{i_1,\dots ,i_{k-1}} dx^{i_1}\wedge \dots \wedge dx^{i_{k-1}}\right)\otimes \lambda,
%	\]
%	and
%	\[
%	\omega_1= \left(\sum A_{i_1,\dots, i_k}dx^{i_1}\wedge \dots \wedge dx^{i_k}\right)\otimes \lambda. \qedhere
%	\]
\end{rem}

Theorem \ref{prop:corrispondenza} motivates the following
\begin{principle}[Symplectic-to-Contact Dictionary]
	\label{pr:dictionary}
	Every natural construction/statement in Symplectic Geometry translates into an analogous construction/statement in Contact Geometry. The translation is obtained making the following substitutions (that play the role of Symplectic-to-Contact Dictionary):
	\begin{align*}
		C^\infty(M)&\leadsto \Gamma(L) \\
		TM&\leadsto DL\\
		T^\ast M &\leadsto J^1L.
	\end{align*}
\end{principle}

Another different but equivalent (and closely related) approach to Contact Geometry is given by \emph{homogeneous symplectic structures} \cite{BGG17}. Let $\mathbbm R^\times$ be the multiplicative group of non-zero reals. A \emph{homogeneous manifold} is a principal $\mathbbm{R}^\times$-bundle. Notice that usually in the literature the terminology \emph{homogeneous manifold} is used for different objects. However, as we are not using these other objects, there will be no risk of confusion. The category of line bundles and LB morphisms, i.e., regular VB morphisms between line bundles, is equivalent to the category of homogeneous manifolds and $\mathbbm{R}^\times$-equivariant smooth maps. Namely, if $L\to M$ is a line bundle, then its dual with the zero section removed $\widetilde{L}:= L^{\ast}\setminus 0$ is an $\mathbbm{R}^\times$-principal bundle over $M$. The principal action is given by the fiberwise scalar multiplication. In the other direction if $P$ is a homogeneous manifold, put $M=P/\mathbbm{R}^\times$, while the line bundle $L\to M$ is given by the VB associated to the tautological representation of $\mathbbm R^\times = \operatorname{GL}(\mathbbm R, 1)$ on $\mathbbm R$. For any homogeneous manifold $P$ we denote by $h$ the principal action of $\mathbbm R^\times$ on $P$. A differential form $\omega$ on $P$ is \emph{homogeneous of degree $k$} if $h_\varepsilon^\ast \omega = \varepsilon^k \omega$ for all $\varepsilon \in \mathbbm R^\times$. Similarly a vector field $X$ on $P$ is \emph{homogeneous of degree $k$} if $h_\varepsilon^\ast X=\varepsilon^k X$ for all $\varepsilon\in \mathbbm{R}^\times$.

Let $L\to M$ be a line bundle and let $\widetilde{L}$ be the associated homogeneous manifold with $\pi\colon \widetilde{L}\to M$ the projection. We recall from \cite[Section 2]{VW20} the relation between differential calculus on $L$ and the one on the homogeneous manifold $\widetilde{L}$. A section $\lambda\in \Gamma(L)$ can be seen as a smooth function $\widetilde{\lambda}$ on $\widetilde{L}$ by setting 
\begin{equation*}
	\widetilde{\lambda}(\alpha):= \langle \alpha, \lambda_x\rangle, \quad \alpha \in \widetilde{L},
\end{equation*}
where $x=\pi(\alpha)\in M$ and $\langle -,-\rangle \colon L^{\ast} \times_M L \to \mathbbm{R}$ is the duality pairing. The function $\widetilde{\lambda}$ is homogeneous of degree $1$, and the assignment $\lambda\mapsto \widetilde{\lambda}$ is a bijection between sections of $L$ and degree $1$ homogeneous functions on $\widetilde{L}$.

Moreover, for each derivation $\Delta$ on $L$, there exists a unique vector field $\widetilde{\Delta}$ on $\widetilde{L}$ such that $\widetilde{\Delta \lambda}= \widetilde{\Delta}(\widetilde{\lambda})$, for all $\lambda\in \Gamma(L)$. The vector field $\widetilde{\Delta}$ is homogeneous of degree $0$, and the assignment $\Delta \mapsto \widetilde{\Delta}$ is a bijection between derivations of $L$ and degree $0$ homogeneous vector fields on $\widetilde{L}$. 

Finally, for every Atiyah $k$-form $\omega\in \OA^k(L)$, there exists a unique differential $k$-form $\widetilde{\omega}\in \Omega^k(\widetilde{L})$ such that $\widetilde{\omega}(\widetilde{\Delta}_1, \dots , \widetilde{\Delta}_k)= \widetilde{\omega(\Delta_1, \dots, \Delta_k)}$, for all $\Delta_1, \dots, \Delta_k\in \Gamma(DL)$. The $k$-form $\widetilde{\omega}$ is homogeneous of degree $1$ and the assignment $\omega \mapsto \widetilde{\omega}$ is a bijection between Atiyah forms on $L$ and degree $1$ homogeneous differential forms on $\widetilde{L}$. 

Let $P$ be a homogeneous manifold.
\begin{definition}
	A \emph{homogeneous symplectic structure} on a homogeneous manifold $P$ is a symplectic structure $\omega$ on $P$ which is additionally homogeneous of degree $1$. A homogeneous manifold equipped with a homogeneous symplectic structure is called a \emph{homogeneous symplectic manifold}.
\end{definition}

Homogeneous symplectic manifolds are another description of contact manifolds. In fact there is a bijection between symplectic Atiyah forms on a line bundle $L$ and degree $1$ homogeneous symplectic manifolds over $\widetilde{L}$. This is essentially given by Theorem \ref{prop:corrispondenza} and \cite[Proposition B.5]{BSTV20}.
\begin{theo}
	The assignment $\omega \mapsto \widetilde{\omega}$ establishes a bijection between symplectic Atiyah forms on $L$ and degree $1$ homogeneous symplectic structures on $\widetilde{L}$.
\end{theo}
%\begin{proof}
%	Let $L\to M$ be a line bundle. We already explained that there is a bijection between Atiyah $2$-forms $\omega\in \OA^2(L)$ and degree $1$ homogeneous $2$-forms $\widetilde{\omega}\in \Omega^2(\widetilde{L})$. We only need to prove that $\omega$ is symplectic if and only if $\widetilde{\omega}$ is so. The closure condition follows from the observation that $\widetilde{\dA\omega}=d\widetilde{\omega}$. The non-degeneracy condition is obvious
%\end{proof}

\begin{rem}
	Notice that if $\widetilde{\omega}$ is a degree $1$ homogeneous symplectic form on $\widetilde{L}$, then the corresponding contact form on $M$ is given by $\iota_{\mathcal E} \widetilde{\omega}$, seen as an $L$-valued $1$-form on $M$, where $\mathcal E$ is the infinitesimal generator of the principal action $h$ (i.e., the fundamental vector field corresponding to the generator $1$ in the Lie algebra $\mathbbm R$ of $\mathbbm R^\times$). Notice that $\mathcal{E}$ is just the Euler vector field.
\end{rem}

\section{Shifted Atiyah $2$-forms}\label{sec:Ashifted}

The main goal of this chapter is to apply the Symplectic-to-Contact Dictionary (Principle \ref{pr:dictionary}) to the definitions of $0$ and $+1$-shifted symplectic structures (Sections \ref{sec:0_shifted} and \ref{sec:1_shifted_s}), in order to derive a preliminary definition of shifted contact structure expressed in terms of Atiyah forms. In this section, we begin with introducing the concept of \emph{shifted Atiyah form}. Consider the Atiyah VBG
\[
	\begin{tikzcd}
		DL \arrow[r,shift left=0.5ex] \arrow[r, shift right=0.5 ex] \arrow[d] &DL_M \arrow[d] \\
		G \arrow[r,shift left=0.5ex] \arrow[r, shift right=0.5 ex] & M
	\end{tikzcd}
\]
of an LBG $(L\rightrightarrows L_M;G\rightrightarrows M)$, see Section \ref{sec:Atiyah_VBG}. Recall that the nerve of $DL$ is the simplicial VB $DL^{(\bullet)}\to G^{(\bullet)}$ obtained applying the functor $D$ to the nerve of the LBG $L$. We can consider the following complex
\begin{equation}
	\label{eq:Acomplex_partial}
	\begin{tikzcd}
		0 \arrow[r] & \OA^\bullet(L_M) \arrow[r, "\partial"] & \OA^\bullet(L) \arrow[r, "\partial"] & \OA^\bullet (L^{(2)}) \arrow[r, "\partial"] & \cdots,
	\end{tikzcd}
\end{equation}
where the differential $\partial$ is
\[
\partial = \sum_{i=0}^k (-1)^i \partial_i^\ast\colon \OA^\bullet(L^{(k-1)})\to\OA^\bullet(L^{(k)}),
\] 
the alternating sum of the pullbacks of Atiyah forms along the face maps $\partial_i\colon L^{(k)}\to L^{(k-1)}$ (see Equation \ref{eq:pullback_Atiyah_forms}).

In the next remark, we prove in which precise sense Complex \eqref{eq:Acomplex_partial} is Morita invariant. This property, in turn, allows us to view \eqref{eq:Acomplex_partial} as an object on the LB-stack $[L_M/L]\to [M/G]$.
\begin{rem}
	\label{rem:partial_Atiyah_Morita_invariant}
	An LBG morphism $(F,f)\colon (L'\rightrightarrows L'_N;H\rightrightarrows N)\to (L\rightrightarrows L_M;G\rightrightarrows M)$ determines a VBG morphism $(DF,f)\colon DL'\to DL$ and then a simplicial VB morphism $(DF,f)\colon (DL'^{(\bullet)} \to H^{(\bullet)})\to (DL^{(\bullet)} \to G^{(\bullet)})$ between the nerves of $DL'$ ad $DL$. Then the pullbacks of Atiyah forms determine a cochain map
	\begin{equation}
		\label{eq:Atiyah_complex}
		\begin{tikzcd}
			0 \arrow[r] & \OA^{\bullet}(L_M) \arrow[r, "\partial"] \arrow[d, "F^\ast"] &\OA^{\bullet}(L)\arrow[r, "\partial"] \arrow[d, "F^\ast"] &\OA^{\bullet}(L^{(2)})\arrow[r,"\partial"] \arrow[d, "F^\ast"] &\cdots \\
			0 \arrow[r] & \OA^{\bullet}(L'_N) \arrow[r, "\partial"] &\OA^{\bullet}(L')\arrow[r, "\partial"] &\OA^{\bullet}(L'^{(2)})\arrow[r,"\partial"] &\cdots
		\end{tikzcd}.
	\end{equation}
	When $f$ is a Morita map, or equivalently $(DF,f)$ is a VB-Morita map (see Proposition \ref{prop:DFMorita}), the cochain map $F^\ast$ is a quasi-isomorphism. Indeed, using the vector space decomposition $$\OA^\filleddiamond (L^{(\bullet)})\cong \Omega^{\filleddiamond -1 }(G^{(\bullet)}, L^{(\bullet)}) \oplus \Omega^{\filleddiamond}(G^{(\bullet)}, L^{(\bullet)}),$$ given by the $\mathbbm{R}$-linear isomorphism \eqref{eq:components}, the cochain map \eqref{eq:Atiyah_complex} is equivalent to
	\begin{equation}
		\label{eq:sum_complex}
		\resizebox{\textwidth}{!}{
			\begin{tikzcd}[ampersand replacement=\&]
				0 \arrow[r] \& \Omega^{\bullet-1}(M,L_M)\oplus \Omega^{\bullet}(M,L_M)  
				\arrow[r, "\partial\oplus \partial"] \arrow[d, "F^\ast\oplus F^\ast"] 
				\&\Omega^{\bullet-1}(G,L)\oplus \Omega^{\bullet}(G,L)
				\arrow[r, "\partial\oplus \partial"] \arrow[d, "F^\ast\oplus F^\ast"] 
				\&\Omega^{\bullet-1}(G^{(2)},L^{(2)})\oplus \Omega^{\bullet}(G^{(2)},L^{(2)})
				\arrow[r,"\partial\oplus \partial"] \arrow[d, "F^\ast\oplus F^\ast"] 
				\&\cdots \\
				0 \arrow[r] \&  \Omega^{\bullet-1}(N,L'_N)\oplus\Omega^{\bullet}(N,L'_N) 
				\arrow[r, "\partial\oplus \partial"] 
				\&\Omega^{\bullet-1}(H,L')\oplus \Omega^{\bullet}(H,L')
				\arrow[r, "\partial\oplus \partial"] 
				\&\Omega^{\bullet-1}(H^{(2)},L'^{(2)})\oplus \Omega^{\bullet}(H^{(2)},L'^{(2)})
				\arrow[r,"\partial\oplus \partial"] 
				\&\cdots
			\end{tikzcd}
		},
	\end{equation}
	
	where $$\partial=\sum_{i=0}^k (-1)^i \partial_i^\ast\colon \Omega^{\filleddiamond}(G^{(k-1)}, L^{(k-1)})\to \Omega^{\filleddiamond}(G^{(k)}, L^{(k)})$$ is the alternating sum of the pullbacks of VB-valued forms along the face maps $\partial_i\colon L^{(k)}\to L^{(k-1)}$ (see Equation \eqref{eq:pullback_vector_valued_forms}). On the other hand, the cochain map 
	\begin{equation*}
		\begin{tikzcd}
			0 \arrow[r] & \Omega^\bullet (M,L_M) \arrow[r, "\partial"] \arrow[d, "F^\ast"] & \Omega^\bullet(G,L) \arrow[r, "\partial"] \arrow[d, "F^\ast"]& \Omega^\bullet (G^{(2)}, L^{(2)}) \arrow[r, "\partial"] \arrow[d, "F^\ast"] & \cdots \\
			0 \arrow[r] & \Omega^\bullet (N, L'_N)\arrow[r, "\partial"] & \Omega^\bullet(H,L') \arrow[r, "\partial"] & \Omega^\bullet(H^{(2)}, L'^{(2)}) \arrow[r, "\partial"] & \cdots 
		\end{tikzcd}
	\end{equation*}
	is a quasi-isomorphism (because of Remark \ref{rem:vv_forms_Morita_equiv} below), then the cochain map \eqref{eq:sum_complex} (and so the cochain map \eqref{eq:Atiyah_complex}) is a quasi-isomorphism as well. In this sense the complex $(\OA^\filleddiamond (L^{(\bullet)}), \partial)$ is Morita invariant up to quasi-isomorphisms.
\end{rem}

Similarly to Section \ref{sec:shifted_structures}, we use complex \eqref{eq:Acomplex_partial} to introduce the notions of \emph{basic} and \emph{multiplicative} Atiyah forms on an LBG. Let $(L\rightrightarrows L_M)$ be an LBG.
\begin{definition}
	\label{def:multiplicative_Atiyah_forms}
	An Atiyah form $\omega\in \OA^\bullet(L_M)$ on $L_M$ is \emph{basic} if it is a $0$-cocycle in \eqref{eq:Acomplex_partial}.	An Atiyah form $\omega\in \OA^\bullet(L)$ on $L$ is \emph{multiplicative} if it is a $1$-cocycle in \eqref{eq:Acomplex_partial}. We denote by $\Omega^k_{D,\operatorname{mult}}(L)$ the space of multiplicative Atiyah $k$-forms on $L$.
\end{definition}

More generally, using \eqref{eq:Acomplex_partial}, we can give the following 
\begin{definition}
	\label{def:shif_Atiyah_forms}
	A \emph{$k$-shifted Atiyah $m$-form} on an LBG $L$ is a $\partial$-closed Atiyah $m$-form on $L^{(k)}$.
\end{definition}
Notice that, $0$ and $+1$-shifted Atiyah forms on $L$ are just basic and multiplicative Atiyah forms respectively.
%\begin{rem}
%	From Remark \ref{rem:partial_Atiyah_Morita_invariant} the notion of shifted Atiyah form is Morita invariant, then the $\partial$-cohomology class $[\omega]$ of a $k$-shifted Atiyah $m$-form $\omega\in \OA^m(L^{(k)})$ can be considered as a $k$-shifted Atiyah $m$-form on the LB-stack $[L_M/L]\to [M/G]$.
%\end{rem}

The differential $\partial$ commutes with the Atiyah differential $\dA$, giving rise to the following double complex:
\begin{equation}
	\label{eq:Atiyah_double_complex}
	\begin{tikzcd}
		&\vdots &\vdots &\vdots
		\\
		0\arrow[r] &\OA^2(L_M) \arrow[u, "\dA"] \arrow[r, "\partial"] &\OA^2(L)\arrow[u,"\dA"] \arrow[r,"\partial"] &\OA^2(L^{(2)})\arrow[u,"\dA"] \arrow[r,"\partial"] &\cdots
		\\
		0\arrow[r] &\OA^1(L_M) \arrow[u, "\dA"]\arrow[r,"\partial"] &\OA^1(L)\arrow[u,"\dA"]\arrow[r,"\partial"] &\OA^1(L^{(2)})\arrow[u,"\dA"]\arrow[r,"\partial"] &\cdots
		\\
		0\arrow[r] &\Gamma(L_M) \arrow[u, "\dA"]\arrow[r,"\partial"] &\Gamma(L)\arrow[u,"\dA"] \arrow[r,"\partial"]& \Gamma(L^{(2)})\arrow[u,"\dA"]\arrow[r,"\partial"] &\cdots
		\\
		&0 \arrow[u] &0\arrow[u] &0\arrow[u]
	\end{tikzcd},
\end{equation}
that is the analogous to the BSS double complex \eqref{eq:BSS} for Atiyah forms. The associated total complex is the complex $(\tot(L)^{\bullet}, D)$ where $$\tot(L)^n= \bigoplus_{k+l=n}\OA^l(L^{(k)}),$$ and $D = \partial + (-1)^k\dA$. Notice that the bottom row of \eqref{eq:Atiyah_double_complex} is the underlying complex of the differential graded module given by the representation $L_M$ (Remark \ref{rem:section_representation}).
\begin{rem}
	%\label{rem:total_complex_Morita_invariant}
	The complex $(\tot(L)^{\bullet}, D)$ is Morita invariant up to quasi-isomorphims. Indeed, the pullback $F^\ast$ along an LBG morphism $(F,f)\colon (L'\rightrightarrows L'_N; H\rightrightarrows N) \to (L\rightrightarrows L_M; G\rightrightarrows M)$ determines a cochain map of double complexes, hence a cochain map between the total complexes $(\tot (L)^\bullet, D)$ and $(\tot (L')^\bullet, D)$. When $f$ is a Morita map, or equivalently, $(DF,f)$ is a VB-Morita map, then, by Remark \ref{rem:partial_Atiyah_Morita_invariant}, $F^\ast$ is a quasi-isomorphism between all the rows, and so $F^\ast\colon (\tot (L)^\bullet, D) \to (\tot (L')^\bullet, D)$ is a quasi-isomorphism as well.
\end{rem}

We are particularly interested in $k$-shifted Atiyah $2$-forms. Let $L$ be an LBG.
\begin{definition}
	\label{def:closed_shifted_Atiyah}
	A \emph{closed $k$-shifted Atiyah $2$-form} on $L$ is a $(k+1)$-tuple $(\omega_k, \dots, \omega_0)$ consisting of Atiyah forms on $L$, where, for any $i=0,\dots, k$,
	\[
		\omega_i\in \OA^{2+k-i}(L^{(i)}),
	\] 
	and such that $D(\omega_k, \dots, \omega_0)=0$.
\end{definition} 

From Definition \ref{def:closed_shifted_Atiyah} it follows that, if $(\omega_k, \dots, \omega_0)$ is a closed $k$-shifted Atiyah $2$-form, then, in particular, $\omega_k$ is a $k$-shifted Atiyah $2$-form.

In the next sections we introduce the notion of $0$ and $+1$-shifted symplectic Atiyah structures. Here we just want to anticipate that a $k$-shifted symplectic Atiyah structure on an LBG $L$ will be a closed $k$-shifted Atiyah $2$-form on $L$ that is non-degenerate in a suitable sense similar to symplectic case. 
In particular the non-degeneracy condition of a $k$-shifted Atiyah $2$-form will be expressed in terms of an appropriate cochain map between the core complex of $DL$,
\begin{equation*}
	\begin{tikzcd}
		0 \arrow[r] & A \arrow[r, "\mathcal{D}"] & DL_M \arrow[r] & 0,
	\end{tikzcd}
\end{equation*}
which is concentrated in degrees $-1,0$, and its $L_M$-twisted dual complex (see Example \ref{ex:twisted_dual_RUTH})
\begin{equation*}
	\begin{tikzcd}
		0 \arrow[r] & J^1L_M \arrow[r, "\mathcal{D}^{\dagger}"] & A^{\dagger} \arrow[r] & 0,
	\end{tikzcd}
\end{equation*}
which is concentrated in degrees $0,1$, shifted by $k$. Notice that there are non trivial cochain maps between those for $k=0,1,2$ only. For this reason, for any LBG, it only makes sense to consider $k$-shifted symplectic Atiyah structures with $k=0,1,2$.

\section{$0$-shifted Symplectic Atiyah Structures} \label{sec:0_shifted_A}
In this section we apply the Symplectic-to-Contact Dictionary (Principle \ref{pr:dictionary}) to the definition of $0$-shifted symplectic structure (Definition \ref{def:0_shifted_sympl}) and prove the analogous Morita invariance.

By Definition \ref{def:closed_shifted_Atiyah}, a closed $0$-shifted Atiyah $2$-form on an LBG $(L\rightrightarrows L_M;G\rightrightarrows M)$ is a $\partial$-closed and $\dA$-closed Atiyah $2$-form $\omega\in \OA^2(L_M)$ on $L_M$. The condition $\partial\omega=0$ simply means that $s^{\ast}\omega= t^{\ast}\omega$. This implies that the flat map of $\omega$, that we denote again by $\omega\colon DL_M\to J^1L_M$, is a cochain map between the core complex of $DL$ and its $L_M$-twisted dual complex
\begin{equation}
	\label{eq:complex_Atiyah_0-shifted}
	\begin{tikzcd}
		0 \arrow[r] & A \arrow[r, "\mathcal{D}"] \arrow[d] & DL_M \arrow[r] \arrow[d, "\omega"] & 0 \arrow[r] \arrow[d] & 0 \\
		0 \arrow[r] & 0 \arrow[r] & J^1L_M \arrow[r, "\mathcal{D}^{\dagger}"'] & A^{\dagger} \arrow[r] & 0
	\end{tikzcd}.
\end{equation}
Indeed, for any $a\in A_x$ and $\delta=Dt(\tilde{\delta})\in D_xL_M$, with $\tilde{\delta}\in D_xL$ and $x\in M$, we get
\begin{equation*}
	\langle \omega(\mathcal{D}_a), \delta \rangle =\omega(\mathcal{D}_a,\delta)= t^{\ast}\omega(a,\tilde{\delta})=s^{\ast}\omega(a,\tilde{\delta})=\omega(Ds(a), Ds(\tilde{\delta}))=0.
\end{equation*}
Hence $\omega\circ \mathcal{D}=0$. Changing the roles of $a$ and $\delta$, we also obtain $\mathcal{D}^{\dagger}\circ \omega=0$.

Let $(L\rightrightarrows L_M;G\rightrightarrows M)$ be an LBG.
\begin{definition}
	\label{def:0-shif_Atiyah}
	A \emph{$0$-shifted symplectic Atiyah structure} on $L$ is an Atiyah $2$-form $\omega\in \OA^2(L_M)$, such that $\dA\omega=0$, $\partial\omega=0$ and $\omega$ is \emph{non-degenerate}, meaning that, for any $x\in M$, the value of the cochain map \eqref{eq:complex_Atiyah_0-shifted} at the point $x$
	\begin{equation*}
		\begin{tikzcd}
			0 \arrow[r] & A_x \arrow[r, "\mathcal{D}"] \arrow[d] & D_xL_M \arrow[r] \arrow[d, "\omega"] & 0 \arrow[r] \arrow[d] & 0 \\
			0 \arrow[r] & 0 \arrow[r] & J^1_xL_M \arrow[r, "\mathcal{D}^{\dagger}"'] & A_x^{\dagger} \arrow[r] & 0
		\end{tikzcd}
	\end{equation*}
	is a quasi-isomorphism.
\end{definition}

Notice that, from the non-degeneracy condition, $\mathcal{D}\colon A_x\to D_xL_M$ has to be injective, for all $x\in M$. However, unlike the symplectic case, $G$ is not necessarily a foliation groupoid (see Remark \ref{rem:foliation_groupoid} for the definition of foliation groupoid).

Now we want to prove that the notion of $0$-shifted symplectic Atiyah structure is Morita invariant in a similar sense as for $0$-shifted symplectic structures. In order to do that we need an ``Atiyah version'' of the results proved in Section \ref{sec:0_shifted}. We start with an appropriate version of Lemma \ref{lemma:omega_orbit}.

\begin{lemma}
	\label{lemma:omega_Atiyah_orbit}
	Let $(L\rightrightarrows L_M;G\rightrightarrows M)$ be an LBG and let $\omega\in \OA^2(L_M)$ be such that $\partial \omega=0$. If the cochain map \eqref{eq:complex_Atiyah_0-shifted} is a quasi-isomorphism at the point $x\in M$, then it is a quasi-isomorphism at all points in the orbit through $x$.
\end{lemma}
\begin{proof}
	Let $x,y\in M$ be two points in the same orbit and let $g\colon x\to y\in G$. Let $h\colon s^\ast TM\to TG$ be an Ehresmann connection (see Definition \ref{def:Ehresmann_connection}). As discussed in Example \ref{ex:Atiyah_RUTH}, $h$ determines a RUTH of $G$ on $A[-1]\oplus DL_M$. For simplicity, we denote by $g_D.$ the quasi actions of $G$ on $A$ and on $DL_M$ determined by $h$. The fibers of the Atiyah VBG $DL$ at the points $x$ and $y$ are related by the cochain map
	\begin{equation}
		\label{eq:Atiyah_cochain_map}
		\begin{tikzcd}
			0 \arrow[r] & A_x\arrow[r, "\mathcal{D}"] \arrow[d, "g_D."'] & D_xL_M\arrow[r] \arrow[d, "g_D."] & 0\\
			0 \arrow[r] & A_y\arrow[r, "\mathcal{D}"'] & D_yL_M\arrow[r] & 0
		\end{tikzcd}.
	\end{equation}
	The $L_M$-twisted dual complexes are related by the cochain map 
	\begin{equation}
		\label{eq:twisted_Atiyah_cochain_map}
		\begin{tikzcd}
			0 \arrow[r] & J^1_y L_M\arrow[r, "\mathcal{D}^\dagger"] \arrow[d, "g_\dagger^{-1}."'] & A_y^\dagger\arrow[r] \arrow[d, "g_\dagger^{-1}."] & 0 \\
			0 \arrow[r] & J^1_x L_M\arrow[r, "\mathcal{D}^\dagger"'] & A_x^\dagger\arrow[r] & 0
		\end{tikzcd},
	\end{equation}
	where we are denoting by $g^{-1}_{\dagger}.$ the first structure operator of the jet RUTH (Example \ref{ex:jet_RUTH}) corresponding to the right-horizontal lift $h^\dagger$ defined by Equation \eqref{eq:splitting_twisted_dual}.
	
	The cochain maps \eqref{eq:Atiyah_cochain_map} and \eqref{eq:twisted_Atiyah_cochain_map} fit in the diagram
	\begin{equation}
		\label{eq:Atiyah_orbit}
		{\scriptsize
		\begin{tikzcd}
			0 \arrow[rr] && A_x \arrow[rr, "\mathcal{D}"] \arrow[dd] \arrow[dr, "g_D."] && D_xL_M \arrow[rr] \arrow[dd, "\omega", near start] \arrow[dr, "g_D."] && 0 \arrow[rr] \arrow[dd] \arrow[dr] && 0 \\
			&0 \arrow[rr, crossing over] && A_y \arrow[rr, "\mathcal{D}", near start, crossing over] && D_yL_M \arrow[rr, crossing over] && 0 \arrow[rr] && 0 \\
			0 \arrow[rr] && 0 \arrow[rr] && J^1_xL_M \arrow[rr, "\mathcal{D}^{\dagger}"', near start] && A_x^{\dagger} \arrow[rr] && 0 \\
			&0 \arrow[rr] && 0 \arrow[rr] \arrow[ul] \arrow[from=uu, crossing over] && J^1_yL_M \arrow[rr, "\mathcal{D}^{\dagger}"'] \arrow[ul, "g^{-1}_\dagger."] \arrow[from=uu, "\omega", near start, crossing over] && A_y^{\dagger} \arrow[rr] \arrow[ul, "g^{-1}_\dagger."] \arrow[from=uu, crossing over] && 0
		\end{tikzcd}}.
	\end{equation}
	Diagram \eqref{eq:Atiyah_orbit} commutes. Indeed, for any $\delta,\delta'\in D_xL_M$, we have
	\[
	\begin{aligned}
		\langle g^{-1}_\dagger. \omega(g_D.\delta) ,\delta'\rangle&= g^{-1}.\left(\omega(g_D.\delta, g_D.\delta')\right)\\
		&= s\left(t_g^{-1}\left(\omega\left(Dt\left(h^D_g(\delta)\right), Dt\left(h^D_g(\delta')\right)\right)\right)\right)\\
		&= s\left((t^\ast\omega) \left(h^D_g(\delta), h^D_g(\delta')\right)\right)\\
		&= s\left((s^\ast\omega) \left(h^D_g(\delta), h^D_g(\delta')\right)\right)\\
		&=\omega\left(Ds\left(h^D_g(\delta)\right), Ds\left(h^D_g(\delta')\right)\right)\\
		&= \omega(\delta,\delta')\\
		&=\langle \omega(\delta),\delta'\rangle,
	\end{aligned}
	\]
	where we used that the $G$-action of $g^{-1}$ on $L_M$ is given by $s\circ t_g^{-1}\colon L_{M,y}\to L_{M,x}$.
	
	By Remark \ref{rem:quasi_actions_quis}, $g_D.$ and $g^{-1}_\dagger$ are quasi-isomorphisms. Hence the cochain map determined by $\omega$ is a quasi-isomorphism at the point $x$ if and only if it is so at the point $y$ as claimed.
\end{proof}
	
	If $\omega\in \OA^2(L_M)$ is such that $\partial \omega=0$ and $F\colon L'\to L$ is a VB-Morita map, then $\partial F^\ast \omega= F^\ast \partial \omega=0$ and so $\omega$ and $F^\ast\omega$ induce cochain maps between the core complex of $DL$ and its $L_M$-twisted dual and between the core complex of $DL'$ and its $L'_N$-twisted dual, respectively. The latter cochain maps are related by the following ``Atiyah version'' of Proposition \ref{prop:non_degeneracy_0-shifted}.
\begin{prop}
	\label{prop:non_degeneracy_Atiyah_0-shifted}
	Let $(F,f)\colon (L'\rightrightarrows L'_N;H\rightrightarrows N) \to (L\rightrightarrows L_M;G\rightrightarrows M)$ be a VB-Morita map between LBGs and let $\omega\in \OA^2(L_M)$ be such that $\partial \omega=0$. Then the cochain map determined by $\omega$ is a quasi-isomorphism at all points in $M$ if and only if the cochain map determined by $F^\ast \omega$ is so at all points in $N$.
\end{prop}
\begin{proof}
	For any $y\in N$, the flat maps $F^\ast\omega\colon D_yL'_N \to J^1_y L'_N$ and $\omega\colon D_{f(y)} L_M \to J^1_{f(y)} L_M$ fit in the following diagram
	\begin{equation}
		\label{eq:Atiyah_Morita_inv}
		\resizebox{\textwidth}{!}{
			\begin{tikzcd}[ampersand replacement=\&]
			0 \arrow[rr] \&\& A_{H,y} \arrow[rr, "\mathcal{D}_H"] \arrow[dd] \arrow[dr, "df"] \&\& D_yL'_N \arrow[rr] \arrow[dd, "F^\ast\omega", near start] \arrow[dr, "DF"] \&\& 0 \arrow[rr] \arrow[dd] \arrow[dr] \&\& 0 \\
			\&0 \arrow[rr, crossing over] \&\& A_{G,f(y)} \arrow[rr, "\mathcal{D}_G", near start, crossing over] \&\& D_{f(y)}L_M \arrow[rr, crossing over] \&\& 0 \arrow[rr] \&\& 0 \\
			0 \arrow[rr] \&\& 0 \arrow[rr] \&\& J^1_yL'_N \arrow[rr, "\mathcal{D}_H^{\dagger}"', near start] \&\& A_{H,y}^{\dagger} \arrow[rr] \&\& 0 \\
			\&0 \arrow[rr] \&\& 0 \arrow[rr] \arrow[ul] \arrow[from=uu, crossing over] \&\& J^1_{f(y)}L_M \arrow[rr, "\mathcal{D}_G^{\dagger}"'] \arrow[ul, "DF^\dagger"] \arrow[from=uu, "\omega", near start, crossing over] \&\& A_{G,f(y)}^{\dagger} \arrow[rr] \arrow[ul, "df^\dagger"] \arrow[from=uu, crossing over] \&\& 0
		\end{tikzcd}},
	\end{equation}
	where $\mathcal{D}_H$ and $\mathcal{D}_G$ are the core-anchors of $DL'$ and $DL$ respectively.
	
	Diagram \eqref{eq:Atiyah_Morita_inv} is commutative. Indeed, for any $\delta,\delta'\in D_yL'_N$, we have
	\begin{align*}
	\langle DF^\dagger(\omega (DF(\delta))),\delta'\rangle &= F_y^{-1}\left(\omega(DF(\delta))(DF(\delta'))\right)\\
	&= F_y^{-1}\left(\omega(DF(\delta), DF(\delta'))\right)\\
	&= (F^\ast \omega) (\delta,\delta')\\
	&= \langle (F^\ast\omega)(\delta),\delta'\rangle.
	\end{align*}

	Since $f$ is Morita, by Proposition \ref{prop:DFMorita} and Corollary \ref{coroll:twisted_dual_VB-Morita_map}, the cochain maps determined by $DF$ and $DF^\dagger$ in Diagram \eqref{eq:Atiyah_Morita_inv} are quasi-isomorphisms. Then the cochain map determined by $F^\ast\omega$ is a quasi-isomorphism at the point $y\in N$ if and only if the cochain map determined by $\omega$ is a quasi-isomorphism at the point $f(y)\in M$. Hence, if the cochain map determined by $\omega$ is a quasi-isomorphism at all points in $M$, then the one determined by $F^\ast \omega$ is so at all points in $N$. For the converse, assume that the cochain map determined by $F^\ast\omega$ is a quasi-isomorphism at all points $y\in N$. Then the one determined by $\omega$ is a quasi-isomorphism at the points $f(y)\in M$. Since $f$ is essentially surjective, then, for any $x\in M$ there exists $y\in N$ such that $f(y)$ and $x$ are in the same orbit. Finally, use Lemma \ref{lemma:omega_Atiyah_orbit}.
\end{proof}

We are now ready to prove that $0$-shifted symplectic Atiyah structures are Morita invariant. The next result is analogous to Proposition \ref{prop:0-shifted_Morita_inv}.
\begin{prop}
	\label{prop:0-shifted_Atiyah_Morita_inv}
	Let $F\colon L'\to L$ be a VB-Morita map. The pullback $F^\ast\colon \OA^\filleddiamond (L^{(\bullet)})\to \OA^\filleddiamond (L'^{(\bullet)})$ induces a bijection between $0$-shifted symplectic Atiyah structures on $L$ and on $L'$.
\end{prop}
\begin{proof}
	From Remark \ref{rem:partial_Atiyah_Morita_invariant}, the cochain map $F^\ast \colon (\OA^\filleddiamond (L^{(\bullet)}), \partial) \to (\OA^\filleddiamond (L'^{(\bullet)}), \partial)$ is a quasi-isomorphism. Then $F^\ast$ maps bijectively $\partial$-closed Atiyah $2$-forms on $L_M$ to $\partial$-closed Atiyah $2$-forms on $L'_N$. Moreover, $F^\ast$ preserves $\dA$-closed Atiyah forms, i.e., $\dA\omega=0$ if and only if $\dA F^{\ast}\omega=0$. Indeed, if $\dA\omega=0$, then 
	\[
	\dA F^\ast\omega=F^\ast \dA\omega=0.
	\]
	Conversely, if $\dA F^\ast \omega=0$, then $F^\ast \dA\omega=0$. But, $\partial \dA\omega=\dA\partial \omega=0$ and, since $F^\ast$ is an isomorphism on $\ker \partial$, then $\dA\omega=0$.
	
	Finally, from Proposition \ref{prop:non_degeneracy_Atiyah_0-shifted}, $\omega$ is a $0$-shifted symplectic Atiyah structure on $L$ if and only if $F^\ast\omega$ is so on $L'$.
\end{proof}

The following Definition makes sense in view of Proposition \ref{prop:0-shifted_Atiyah_Morita_inv} .
\begin{definition}
	A \emph{$0$-shifted symplectic Atiyah structure} on an LB-stack $[L_M/L]\to [M/G]$ is a $0$-shifted symplectic Atiyah structure on an LBG $(L\rightrightarrows L_M; G\rightrightarrows M)$ presenting $[L_M/L]\to [M/G]$.
\end{definition}

\section{$+1$-shifted Symplectic Atiyah Structures}\label{sec:A+1}
In this section we apply the Symplectic-to-Contact Dictionary (Principle \ref{pr:dictionary}) to the notion of $+1$-shifted symplectic structure (see Definition \ref{def:+1_shifted_sympl}). Similarly to Section \ref{sec:1_shifted_s}, we begin by introducing the notion of $+1$-shifted symplectic Atiyah structures. Then, in Section \ref{sec:Amult}, we explore properties of multiplicative Atiyah $2$-forms on a Lie groupoid, proving, among other things, that the non-degeneracy condition on a $+1$-shifted symplectic Atiyah structure is a Morita invariant condition. In Section \ref{sec:Asme} we introduce an appropriate notion of Morita equivalence between Lie groupoids equipped with a $+1$-shifted symplectic structure.

By Definition \ref{def:closed_shifted_Atiyah}, a closed $+1$-shifted Atiyah $2$-form on an LBG $(L\rightrightarrows L_M;G\rightrightarrows M)$ is a $D$-closed pair $(\omega, \Omega)\in \OA^2(L)\oplus \OA^3(L_M)$. This simply means that
\begin{equation*}
	\partial \omega=0, \quad \dA\omega=\partial \Omega, \quad \dA\Omega=0,
\end{equation*}
i.e., $\omega$ is multiplicative, $\omega$ is $\dA$-closed up to the $\partial$-coboundary of $\Omega$, and $\Omega$ is $\dA$-closed.

We will prove in Remark \ref{rem:Acochain_map} that, by multiplicativity, the flat map of $\omega$, again denoted by $\omega\colon DL\to J^1L$, induces a cochain map between the core complex of $DL$ and its $L_M$-twisted dual shifted by $+1$
\begin{equation}
	\label{eq:complex_Atiyah_1-shifted}
	\begin{tikzcd}
		0 \arrow[r] & A \arrow[r, "\mathcal{D}"] \arrow[d, "\omega"'] & DL_M \arrow[r] \arrow[d, "\omega"] & 0 \\
		0 \arrow[r] & J^1L_M \arrow[r, "\mathcal{D}^\dagger"'] & A^\dagger \arrow[r] & 0
	\end{tikzcd}.
\end{equation}

Let $(L\rightrightarrows L_M;G\rightrightarrows M)$ be an LBG.
\begin{definition}
	A \emph{$+1$-shifted symplectic Atiyah structure} on $L$ is a pair $(\omega, \Omega)\in \OA^2(L)\oplus \OA^3(L_M)$ such that $\partial \omega=0$, $\dA\omega=\partial \Omega$, $\dA\Omega=0$, and $\omega$ is \emph{non-degenerate}, meaning that, for any $x\in M$, the value of the cochain map \eqref{eq:complex_Atiyah_1-shifted} at the point $x$
	\begin{equation}
		\label{eq:Atiyah_nondegenerate_1}
		\begin{tikzcd}
			0 \arrow[r] & A_x \arrow[r, "\mathcal{D}"] \arrow[d, "\omega"'] & D_xL_M \arrow[r] \arrow[d, "\omega"] & 0 \\
			0 \arrow[r] & J^1_xL_M \arrow[r, "\mathcal{D}^\dagger"'] & A_x^\dagger \arrow[r] & 0
		\end{tikzcd}
	\end{equation}
	is a quasi-isomorphism.

	A \emph{$+1$-shifted symplectic Atiyah LBG} is a triple $(L, \omega, \Omega)$ where $L$ is an LBG and $(\omega, \Omega)$ is a $+1$-shifted symplectic Atiyah structure on $L$.
\end{definition}
\begin{rem}
	\label{rem:Atiyah_quasi_iso_-1_0}
	Notice that the map $\omega\colon D_xL_M\to A_x^\dagger$ is just the opposite of $\omega^\dagger\colon D_xLM\to A_x^\dagger$, namely the $L_M$-twisted dual of $\omega\colon A_x\to J^1_x L_M$. Indeed, for any $a\in A_x$ and $\delta\in D_xL_M$, we have
	\[
	\langle -\omega^\dagger(\delta),a\rangle =- \omega(a)(\delta)= -\omega(a,\delta)= \omega(\delta,a)= \langle \omega(\delta),a\rangle. 
	\]
	Then, \eqref{eq:Atiyah_nondegenerate_1} is an isomorphism on the $(-1)$-cohomology if and only if it is an isomorphism on the $0$-cohomology.
\end{rem}

By Remark \ref{rem:jet_VBG}, the complex
\begin{equation*}
	\begin{tikzcd}
		0 \arrow[r] & J^1L_M \arrow[r, "\mathcal{D}^{\dagger}"] & A^\dagger \arrow[r] & 0,
	\end{tikzcd}
\end{equation*}
concentrated in degrees $-1,0$, is the core complex of the jet VBG $J^1L$. Then, by Theorem \ref{theo:caratterizzazioneVBmorita}, $\omega$ is non-degenerate if and only if the VBG morphism $\omega\colon DL\to J^1L$ (see Proposition \ref{prop:Atiyah_omega_VBGmorphism}) is a VB-Morita map. Hence, we can give this alternative, but equivalent, definition of $+1$-shifted symplectic Atiyah structures:
\begin{definition}
	A \emph{$+1$-shifted symplectic Atiyah structure} on $L$ is a pair $(\omega, \Omega)\in \OA^2(L)\oplus \OA^3(L_M)$, such that $\partial \omega=0$, $\dA\omega=\partial \Omega$, $\dA\Omega=0$, and $\omega$ is \emph{non-degenerate}, i.e., $\omega\colon DL\to J^1L$ is a VB-Morita map.
\end{definition}

\subsection{Multiplicative Atiyah $2$-forms}\label{sec:Amult}
In this subsection we establish several properties of multiplicative Atiyah $2$-forms, analogous to those proven in Section \ref{sec:mult} for plain multiplicative forms. In particular, we focus on the non-degeneracy condition of $+1$-shifted symplectic Atiyah structures $(\omega,\Omega)$. In fact, this only depends on $\omega$. We establish that the non-degeneracy condition is a Morita invariant property. We actually present two distinct proofs of the latter fact: the first follows an approach similar to that used in the $0$-shifted case, while the second relies exclusively on VBG morphisms and linear natural isomorphisms.

Let $(L\rightrightarrows L_M;G\rightrightarrows M)$ be an LBG. By Definition \ref{def:multiplicative_Atiyah_forms} an Atiyah $2$-form $\omega\in \OA^2(L)$ on $L$ is multiplicative if it is closed with respect to the differential $\partial$, i.e.,
\begin{equation}
	\label{eq:Atiyah_mult}
	m^{\ast} \omega = \pr_1^{\ast} \omega + \pr_2^\ast \omega \in \OA^2(L^{(2)}),
\end{equation}
where $\pr_i\colon L^{(2)}\to L$ is the projection on the $i$-th factor, with $i=1,2$. 

The next Proposition is the analogue of Proposition \ref{prop:formule}.
\begin{prop}
	\label{prop:Atiyah_formule}
	Let $\omega\in \OA^2(G)$ be a multiplicative Atiyah $2$-form on the LBG $(L\rightrightarrows L_M;G\rightrightarrows M)$. Then
	\begin{itemize}
		\item[i)] the pullback $u^\ast \omega$ of $\omega$ along the unit map $u\colon L_M\to L$ is zero;
		\item[ii)] the pullback $i^\ast \omega$ of $\omega$ along the inversion map $i\colon L\to L$ is $-\omega$.
%		\item[iii)] for any $a,b\in \Gamma(A)$, we have
%		\[
%		\omega(\overrightarrow{\Delta^a}_g,\overrightarrow{\Delta^b}_g) = - \omega(\overleftarrow{\Delta^a}_{g^{-1}},\overleftarrow{\Delta^b}_{g^{-1}})^{-1}, \quad \omega(\overrightarrow{\Delta^a},\overleftarrow{\Delta^b})=0,
%		\]
%		where $\overrightarrow{\Delta^a},\overrightarrow{\Delta^b}$ and $\overleftarrow{\Delta^a},\overleftarrow{\Delta^b}$ are the right and left-invariant derivations respectively, generated by $a$ and $b$ (see \textcolor{red}{cita}).
	\end{itemize}
\end{prop}
\begin{proof}
	In this proof we use that all the face maps of the nerve of $L$ are regular VB morphisms, i.e., fiberwise isomorphisms (see Remark \ref{rem:face_maps_trivial_coreVBG}). Let $\delta,\delta'\in D_xL_M$, with $x\in M$. Then, applying both sides of Equation \eqref{eq:Atiyah_mult} to the pair $\big((\delta,\delta),(\delta',\delta')\big)$, with $(\delta,\delta),(\delta',\delta')\in D_{(x,x)}L^{(2)}$, we get
	\[
		m_{(x,x)}^{-1}\omega(\delta,\delta') = \pr_{1,(x,x)}^{-1}\omega(\delta,\delta')+\pr_{2,(x,x)}^{-1}\omega(\delta,\delta')\in L^{(2)}_{(x,x)},
	\]
	so
	\[
		\omega(\delta,\delta') = m \left(\pr_{1,(x,x)}^{-1} \omega(\delta,\delta')\right) + m \left(\pr_{2,(x,x)}^{-1} \omega(\delta,\delta')\right)\in L_x.
	\]
	Applying the source map and remembering that $s\circ m= s\circ \pr_2\colon L^{(2)}\to L_M$ we get
	\[
		s\left(\omega(\delta,\delta')\right)= s\left(\pr_2 \left(\pr_{1,(x,x)}^{-1} \omega(\delta,\delta')\right)\right) + s\left(\omega(\delta,\delta')\right)\in L_{M,x},
	\]
	but $\pr_2\circ \pr_{1,(x,x)}^{-1}= t_x^{-1}\circ s\colon L_x\to L_x$, so we have
	\[
		0=s\left(\pr_2 \left(\pr_{1,(x,x)}^{-1} \omega(\delta,\delta')\right)\right) = s\left(t_x^{-1}\big(s(\omega(\delta,\delta'))\big)\right)= s(\omega(\delta,\delta'))\in L_{M,x}
	\]
	and then $\omega(\delta,\delta')=0\in L_x$ and $i)$ is proved. 
	
	In order to prove $ii)$ let $\delta,\delta'\in D_gL$, with $g\in G$. Applying both sides of Equation \eqref{eq:Atiyah_mult} to the pair $\big((\delta, \delta^{-1}), (\delta',\delta'^{-1})\big)$, with $\delta, \delta'\in D_gL$, we get
	\begin{equation}
		\label{eq:Atiyah_inv_omega}
		m_{(g,g^{-1})}^{-1}\omega(\delta\cdot \delta^{-1}, \delta'\cdot \delta'^{-1}) = \pr_{1,(g,g^{-1})}^{-1}\omega(\delta,\delta') +\pr_{2,(g,g^{-1})}^{-1}\omega(\delta^{-1}, \delta'^{-1})\in L^{(2)}_{(g,g^{-1})},
	\end{equation}
	but $\delta\cdot \delta^{-1}= Dt(\delta), \delta'\cdot \delta'^{-1}=Dt(\delta')\in D_{t(g)}L_M$ and, from $i)$, Equation \eqref{eq:Atiyah_inv_omega} reduces to
	\begin{align*}
		-\omega(\delta,\delta')&= \pr_1\left(\pr_{2,(g,g^{-1})}^{-1} \omega(\delta^{-1}, \delta'^{-1})\right)\\
		&=s_g^{-1}\left(t\left(\omega(\delta^{-1}, \delta'^{-1})\right)\right)\\
		&= \omega(\delta^{-1}, \delta'^{-1})^{-1}\in L_g,
	\end{align*}
	where we used that $\pr_1\circ \pr_{2,(g,g^{-1})}^{-1}= s_g^{-1}\circ t\colon L_{g^{-1}}\to L_g$.
%	The first part of $iii)$ follows from $ii)$ noting that, for any $g\in G$, 
%	\[
%		\overrightarrow{\Delta^a}_g= a_{t(g)}\cdot 0^{DL}_g= \left( 0^{DL}_{g^{-1}}\cdot a^{-1}_{s(g^{-1})}\right)^{-1}=(\overleftarrow{\Delta^a}_{g^{-1}})^{-1}
%	\] 
%	and the same for $b$. 
%	
%	Finally, for any $g\in G$, applying \eqref{eq:Atiyah_mult} to $(\overrightarrow{\Delta^a}_g, 0^{DL}_{s(g)}), (0^{DL}_g, \overleftarrow{\Delta^b}_{s(g)})\in D_{(g,s(g))}L^{(2)}$, we get
%	\[
%		m_{(g,s(g))}^{-1}\omega(\overrightarrow{\Delta^a}_g, \overleftarrow{\Delta^b}_g) = \pr_{1,(g,s(g))}^{-1}\omega(\overrightarrow{\Delta^a}_g, 0^{DL}_g) + \pr_{2,(g,s(g))}^{-1} \omega(0^{DL}_{s(g)}, \overleftarrow{\Delta^b}_{s(g)})=0,
%	\]
%	where we used that $\overrightarrow{\Delta^a}_g \cdot 0^{DL}_{s(g)}= \overrightarrow{\Delta^a}_g$ and $0^{DL}_g\cdot \overleftarrow{\Delta^b}_{s(g)}= \overleftarrow{\Delta^b}_g$.
\end{proof}

Using Proposition \ref{prop:Atiyah_formule} we can prove that the flat map of a multiplicative Atiyah form is a VBG morphism, in analogy with what happens for plain multiplicative forms.
\begin{prop}
	\label{prop:Atiyah_omega_VBGmorphism}
	Let $(L\rightrightarrows L_M ;G\rightrightarrows M)$ be an LBG and let $\omega\in \OA^2(L)$ be a multiplicative Atiyah $2$-form. The flat map, again denoted by $\omega\colon DL\to J^1L$, is a VBG morphism covering the identity $\operatorname{id_G}\colon G\to G$ from the Atiyah VBG $DL$ to the jet VBG $J^1L$.
\end{prop}
\begin{proof}
	First notice that, from the canonical splitting $DL|_M=A\oplus DL_M$ and from $i)$ in Proposition \ref{prop:Atiyah_formule} we have that the restriction of $\omega\colon DL\to J^1L$ to $DL_M$ takes values in $A^\dagger$ and the map $\omega\colon DL_M\to A^\dagger$ is well-defined.
	
	The maps $\omega\colon DL\to J^1L$ and $\omega\colon DL_M\to A^\dagger$ are VB morphisms covering the identities $\operatorname{id}_G$ and $\operatorname{id}_M$ respectively. Next, we need to prove that $\omega\colon (DL\rightrightarrows DL_M)\to (J^1L\to A^\dagger)$ is a Lie groupoid morphism. We use the definition of the structure maps of the twisted dual VBG (Example \ref{ex:twisted_dual}) for the jet VBG $J^1L$. For any $\delta\in D_gL$ and $a\in A_{s(g)}$, with $g\in G$, we have
	\begin{align*}
		\langle s(\omega(\delta)), a\rangle &= -s\left\langle \omega(\delta), 0^{DL}_g\cdot (\overrightarrow{\Delta^a}_{s(g)})^{-1}\right\rangle = -s\left(\omega\left(\delta, 0^{DL}_g\cdot (\overrightarrow{\Delta^a}_{s(g)})^{-1}\right)\right)\\
		&= -s\left(\omega\left(\delta\cdot Ds(\delta), 0^{DL}_g\cdot (\overrightarrow{\Delta^a}_{s(g)})^{-1}\right)\right) \\
		& = -s\left(m\left(\pr_{2,(g,s(g))}^{-1}\omega\left(Ds(\delta), (\overrightarrow{\Delta^a}_{s(g)})^{-1}\right)\right)\right) \\
		&= -\omega\left(Ds(\delta)^{-1}, (\overrightarrow{\Delta^a}_{s(g)})^{-1}\right)\\
		&= \omega(Ds(\delta), a) \\
		&= \langle \omega(Ds(\delta)),a\rangle,
	\end{align*}
	where we used that $Ds(\delta)\in D_{s(g)}L_M$ and then its inverse agrees with $Ds(\delta)$ and $ii)$ in Proposition \ref{prop:Atiyah_formule}. Hence $\omega$ commutes with the source maps.
	
	For any $\delta\in D_gL$ and $a\in A_{t(g)}$, with $g\in G$, we have
	\begin{align*}
		\langle t(\omega(\delta)), a\rangle &= t\left\langle \iota_\delta\omega, \overrightarrow{\Delta^a}_{t(g)}\cdot 0^{DL}_g\right\rangle = t\left(\omega\left(\delta, \overrightarrow{\Delta^a}_{t(g)}\cdot 0^{DL}_g\right)\right)\\
		&=t\left( \omega\left(Dt(\delta)\cdot \delta, \overrightarrow{\Delta^a}_{t(g)}\cdot 0^{DL}_g\right)\right) = t\left(m\left(\pr_{1,(t(g),g)}^{-1}\omega\left(Dt(\delta), a\right)\right)\right)\\
		&=\omega\left(Dt(\delta), a\right) = \langle \omega(Dt(\delta)), a\rangle,
	\end{align*}
	and $\omega$ commutes with the target maps.
	
	For any $\delta\in D_xL$ and $\delta'\in D_xL_M$, with $x\in M$, we have
	\begin{align*}
		\langle u(\omega(\delta')), \delta\rangle = \omega( \delta', \pr_A(\delta)) = \omega(\delta',\delta) = \langle \omega(\delta'),\delta\rangle,
	\end{align*}
	where we used that $\delta=\pr_A(\delta) + Ds(\delta)$. But, from $i)$ in Proposition \ref{prop:Atiyah_formule}, $\omega(\delta',Ds(\delta))=0$. Hence $\omega$ commutes with the unit maps.
%	
%	For any $\delta\in D_gL$ and $\delta'\in D_{g^{-1}}L$, with $g\in G$, we have
%	\begin{align*}
%		\langle (\omega(\delta))^{-1}, \delta'\rangle = -s_{g^{-1}}^{-1}\left(t\langle \omega(\delta), \delta'^{-1}\rangle\right) = -s_{g^{-1}}^{-1}\left(t\left(\omega(\delta, \delta'^{-1})\right)\right)= \omega(\delta^{-1}, \delta') = \langle \omega(\delta^{-1}), \delta'\rangle,
%	\end{align*}
%	where we used $ii)$ in Proposition \ref{prop:Atiyah_formule}. Hence $\omega$ commutes with the inverse maps.
	
	Finally, for any $(\delta,\delta'), (\tilde{\delta},\tilde{\delta}')\in D_{(g,g')}L^{(2)}$, with $(g,g')\in G^{(2)}$, we have
	\begin{align*}
		\langle \omega(\delta) \cdot \omega(\delta'), \tilde{\delta}\cdot\tilde{\delta}'\rangle &= s_{gg'}^{-1}\left(g'^{-1}.s\langle \omega(\delta), \tilde{\delta}\rangle + s\langle \omega(\delta'), \tilde{\delta}'\rangle\right) \\
		&= s_{gg'}^{-1}\left(g'^{-1}.s\left(\omega(\delta,\tilde{\delta})\right) + s\left(\omega(\delta',\tilde{\delta}')\right)\right) \\ 
		&=t_{gg'}^{-1}\left(t\left(\omega(\delta,\tilde{\delta})\right)\right) + s_{gg'}^{-1}\left(s\left(\omega(\delta'.\tilde{\delta}')\right)\right) \\
		&= m\left(\pr_{1,(g,g^{-1})}^{-1}\omega(\delta,\tilde{\delta})+ \pr_{2,(g,g^{-1})}^{-1}\omega(\delta',\tilde{\delta}')\right) \\
		&= \omega(\delta\cdot \delta', \tilde{\delta}\cdot \tilde{\delta}') \\
		&= \langle \omega(\delta\cdot \delta'), \tilde{\delta}\cdot \tilde{\delta}'\rangle ,
	\end{align*}
	where we used Equation \eqref{eq:Atiyah_mult} and that, for any $\lambda\in L_g$, we have
	\begin{align*}
		s_{gg'}^{-1}\big(g'^{-1}.s(\lambda)\big)&= t_{gg'}^{-1} \left(t_{gg'} \left(s_{gg'}^{-1}\left( g'^{-1}.s \left(t_g^{-1} \left(t(\lambda)\right)\right)\right)\right)\right) \\
		&= t_{gg'}^{-1}\left( gg'. \left( g'^{-1}. \left(g^{-1}. \circ t(\lambda)\right)\right)\right)\\
		&=t_{gg'}^{-1}\left( t(\lambda)\right).
	\end{align*}
	Hence $\omega$ commutes with the multiplication maps and this concludes the proof.
\end{proof}

\begin{rem}
	\label{rem:Acochain_map}
	From Proposition \ref{prop:Atiyah_omega_VBGmorphism}, the flat map of $\omega$ determines a VBG morphism
	\begin{equation*}
		\scriptsize
		\begin{tikzcd}
			DL \arrow[rr, shift left=0.5ex] \arrow[rr, shift right=0.5ex] \arrow[dd] \arrow[dr, "\omega"] & &DL_M \arrow[dd] \arrow[dr, "\omega"] \\
			&  J^1L \arrow[rr, shift left= 0.5ex, crossing over] \arrow[rr, shift right =0.5ex, crossing over] & &A^{\dagger} \arrow[dd]\\
			G \arrow[rr, shift left=0.5ex] \arrow[rr, shift right=0.5ex] \arrow[dr, equal] & & M \arrow[dr, equal]\\ 
			&  G \arrow[from=uu, crossing over]\arrow[rr, shift left= 0.5ex] \arrow[rr, shift right =0.5ex] & &M
		\end{tikzcd}.
	\end{equation*}
	Hence it induces a cochain map between the core complex of $DL$ and the core complex of $J^1L$. 
\end{rem}

Before investigating the relation between Morita maps and the cochain maps induced by multiplicative Atiyah $2$-forms, we connect multiplicative Atiyah $2$-forms to the quasi-actions coming from the Atiyah and jet VBGs.
\begin{prop}
	Let $\omega\in \OA^2(L)$ be a multiplicative $2$-form on the LBG $L$. Denote by $g_D.$ the first structure operator of the Atiyah RUTH (Example \ref{ex:Atiyah_RUTH}) determined by an Ehresmann connection $h\colon s^\ast TM \to TG$, with $g\colon x\to y\in G$, then 
	\begin{equation}
		\label{eq:Atiyah_omega_adjointRUTH}
		\omega(g_D.a, g_D.\delta)= t\left(\omega\left(h^D_g(\mathcal{D}_a), h^D_g(\delta)\right)\right)+ g. \omega(a,\delta)\in L_y,
	\end{equation}
	for all $a\in A_x$ and $\delta\in D_xL_M$, where $g.$ is the $G$-action on $L_M$.
\end{prop}
\begin{proof}
	Let $g\colon x\to y\in G$. By Example \ref{ex:Atiyah_RUTH}, the quasi action $g_D.$ on $A$ agrees with the quasi action $g_T.$ of the adjoint RUTH on $A$ (Example \ref{ex:adjointRUTH}). Then, by Equation \eqref{eq:quasi_action_onA}, for any $a\in A_x$, we have
	\begin{equation*}
		g_D.a= h_g(\rho(a)) \cdot a \cdot 0^{TG}_{g^{-1}}= h_g^D(\mathcal{D}_a) \cdot a \cdot 0_{g^{-1}}^{DL}  \in A_y.
	\end{equation*}
	Remembering that $g_D.\delta= Dt(h^D_g(\delta))= h_g^D(\delta)\cdot h_g^D(\delta)^{-1}\in D_yL_M$, we have
	\begin{align*}
		\omega(g_D.a, g_D.v) &= \omega\left(h^D_g(\mathcal{D}_a) \cdot a \cdot 0^{DL}_{g^{-1}}, h^D_g(\delta)\cdot h^D_g(\delta)^{-1}\right)\\
		&=m\left(\pr_{1,(g,g^{-1})}^{-1}\omega\left(h_g(\mathcal{D}_a) \cdot a, h^D_g(\delta)\cdot \delta\right)\right) \\
		&=t\left(\omega\left(h^D_g(\mathcal{D}_a) \cdot a, h^D_g(\delta)\cdot \delta\right)\right) \\
		&=t\left(m\left(\pr_{1,(g,x)}^{-1}\omega\left(h^D_g(\mathcal{D}_a), h^D_g(\delta)\right)\right)\right) + t\left(m\left(\pr_{2,(g,x)}^{-1} \omega\left(a, \delta\right)\right)\right) \\
		&=t\left(\omega\left(h^D_g(\mathcal{D}_a), h^D_g(\delta)\right)\right)+ t\left(s_g^{-1}\left(\omega\left(a, \delta\right)\right)\right) \\
		&=t\left(\omega\left(h^D_g(\mathcal{D}_a), h^D_g(\delta)\right)\right)+ g. \omega(a,\delta),
	\end{align*}
	where we used that, from $t_y\circ m_{(g,g^{-1})}= t_g\circ \pr_{1,(g,g^{-1})}$, we get $$m\circ \pr_{1,(g,g^{-1})}^{-1} = t_y^{-1}\circ t_g=t_g.$$ From $t_g\circ m_{(g,x)}= t_g\circ \pr_{1,(g,x)}$, we get
	\[
		m\circ \pr_{1,(g,x)}^{-1}= t_g^{-1}\circ t_g= \operatorname{id}_G.
	\]
	And, from $s_g\circ m_{(g,x)}= s_x\circ \pr_{2,(g,x)}$, we get
	\[
		m\circ \pr_{2,(g,x)}^{-1}= s_g^{-1}\circ s_x.\qedhere
	\]
\end{proof}

The quasi-isomorphism of a cochain map induced by a multiplicative Atiyah $2$-form is a condition on the orbits of the Lie groupoid. This is proved in the next result that is analogous to Lemma \ref{lemma:non-deg_orbit} and follows from similar arguments that we write down explicitly for the reader's convenience.
\begin{lemma}
	\label{lemma:Atiyah_non-deg_orbit}
	Let $(L\rightrightarrows L_M;G\rightrightarrows M)$ be an LBG and let $\omega\in \OA^2(L)$ be a multiplicative Atiyah $2$-form on $L$. If the cochain map \eqref{eq:Atiyah_nondegenerate_1} is a quasi-isomorphism at a point $x\in M$, then it is a quasi-isomorphism at all points in the orbit through $x$.
\end{lemma}
\begin{proof}
	Let $x,y\in M$ be two points in the same orbit and let $g\colon x\to y\in G$. As discussed in the proof of Lemma \ref{lemma:omega_Atiyah_orbit}, the fibers over the points $x$ and $y$ are related by the cochain map \eqref{eq:Atiyah_cochain_map}:
	\begin{equation*}
		\begin{tikzcd}
			0 \arrow[r] & A_x\arrow[r, "\mathcal{D}"] \arrow[d, "g_D."'] & D_xL_M\arrow[r] \arrow[d, "g_D."] & 0\\
			0 \arrow[r] & A_y\arrow[r, "\mathcal{D}"'] & D_yL_M\arrow[r] & 0
		\end{tikzcd}
	\end{equation*}
	given by the Atiyah RUTH (Example \ref{ex:Atiyah_RUTH}), and the dual complexes are related by the cochain map \eqref{eq:twisted_Atiyah_cochain_map}, given by the jet RUTH (Example \ref{ex:jet_RUTH}):
	\begin{equation*}
		\begin{tikzcd}
			0 \arrow[r] & J^1_y L_M\arrow[r, "\mathcal{D}^\dagger"] \arrow[d, "g^{-1}_\dagger."'] & A_y^\dagger\arrow[r] \arrow[d, "g^{-1}_\dagger."] & 0 \\
			0 \arrow[r] & J^1_x L_M\arrow[r, "\mathcal{D}^\dagger"'] & A_x^\dagger\arrow[r] & 0
		\end{tikzcd}.
	\end{equation*}
	These cochain maps fit in the diagram
	\begin{equation}
		\label{eq:Atiyah_diagram_orbit}
		{\scriptsize
		\begin{tikzcd}
			0 \arrow[rr] & & A_x \arrow[rr, "\mathcal{D}"] \arrow[dd, "\omega"' near start] & & D_xL_M \arrow[rr] \arrow[dd, "\omega"' near start] & & 0 \\
			& 0 \arrow[rr, crossing over] & & A_y \arrow[rr, crossing over, "\mathcal{D}" near start] \arrow[from=ul, "g_D."] & & D_yL_M \arrow[rr] \arrow[from=ul, "g_D."]& & 0 \\
			0 \arrow[rr] & & J^1_xL_M \arrow[rr, "\mathcal{D}^\dagger"' near end] \arrow[from= dr, "g^{-1}_\dagger."] & & A_x^{\dagger} \arrow[rr] \arrow[from=dr, "g^{-1}_\dagger."] & & 0 \\
			& 0 \arrow[rr] & & J^1_yL_M \arrow[rr,  "\mathcal{D}^\dagger"'] \arrow[from=uu, crossing over, "\omega"' near start]  & & A_y^{\dagger} \arrow[rr] \arrow[from=uu, crossing over, "\omega" near start] & & 0
		\end{tikzcd}}.
	\end{equation}
	Diagram \eqref{eq:Atiyah_diagram_orbit} does not commute, but, by Equation \eqref{eq:Atiyah_omega_adjointRUTH}, it commutes in cohomology. Moreover, by Remark \ref{rem:quasi_actions_quis}, $g_D.$ and $g_D^\dagger.$ are quasi-isomorphisms. Hence the cochain map determined by $\omega$ is a quasi-isomorphism at $x$ if and only if it is so at $y$ and the statement is proved.
\end{proof}
Pullbacks of multiplicative Atiyah forms along LBG morphisms are again multiplicative Atiyah forms, then we have the following result that is the analogous of Proposition \ref{prop:+1-shifted_Morita_map}.
\begin{prop}
	\label{prop:+1-shifted_Atiyah_Morita_map}
	Let $(F,f)\colon (L'\rightrightarrows L'_N;H\rightrightarrows N)\to (L\rightrightarrows L_M;G\rightrightarrows M)$ be a VB-Morita map between LBGs and let $\omega\in \OA^2(L)$ be a multiplicative Atiyah $2$-form on $L$. Then $F^{\ast}\omega$ determines a quasi-isomorphism between the fibers over all points in $N$ if and only if $\omega$ determines a quasi-isomorphism over all points in $M$.
\end{prop}
\begin{proof}
	By Proposition \ref{prop:DFMorita} and Theorem \ref{theo:caratterizzazioneVBmorita}, for any $y\in N$, if we set $x=f(y)\in M$, then the cochain map
	\begin{equation}
		\label{eq:Df_quis}
		\begin{tikzcd}
			0 \arrow[r] & A_{H,y}\arrow[r, "\mathcal{D}_H"] \arrow[d, "df"'] & D_yL'_N \arrow[r] \arrow[d, "DF"] & 0\\
			0 \arrow[r] & A_{G,x}\arrow[r, "\mathcal{D}_G"'] & D_xL_M \arrow[r] & 0
		\end{tikzcd}
	\end{equation}
	is a quasi-isomorphism. Hence, the cochain map
	\begin{equation}
		\label{eq:dualDf_quis}
		\begin{tikzcd}
			0 \arrow[r] & J^1_xL_M \arrow[r, "\mathcal{D}_G^\dagger"] \arrow[d, "DF^\dagger"'] & A_{G,x}^\dagger \arrow[r] \arrow[d, "df^\dagger"] & 0\\
			0 \arrow[r] & J^1_yL'_N \arrow[r, "\mathcal{D}_H^\dagger"'] & A_{H,y}^\dagger \arrow[r] & 0
		\end{tikzcd},
	\end{equation}
	given by the twisted dual of the cochain map \eqref{eq:Df_quis}, is a quasi-isomorphism as well.

	The quasi-isomorphisms \eqref{eq:Df_quis} and \eqref{eq:dualDf_quis} fit in the following diagram
	\begin{equation}
		\label{eq:Df}
		{\scriptsize
		\begin{tikzcd}
			0 \arrow[rr] & & A_{H,y} \arrow[rr, "\mathcal{D}_H"] \arrow[dd] & & D_yL'_N \arrow[rr] \arrow[dd] & & 0 \\
			& 0 \arrow[rr, crossing over] & & A_{G,x} \arrow[rr, crossing over, "\mathcal{D}_G" near start] \arrow[from=ul, "df"] & & D_xL_M \arrow[rr] \arrow[from=ul, "DF"]& & 0 \\
			0 \arrow[rr] & & J^1_yL'_N \arrow[rr, "\mathcal{D}_H^\dagger"' near end] \arrow[from= dr, "DF^\dagger"] & & A_{G,y}^{\dagger} \arrow[rr] \arrow[from=dr, "df^\dagger"'] & & 0 \\
			& 0 \arrow[rr] & & J^1_xL_M \arrow[rr,  "\mathcal{D}_G^\dagger"'] \arrow[from=uu, crossing over]  & & A_{G,x}^{\dagger} \arrow[rr] \arrow[from=uu, crossing over] & & 0
		\end{tikzcd}}
	\end{equation}
	where the vertical arrows are the cochain maps induced by $F^\ast \omega$ and $\omega$. Diagram \eqref{eq:Df} commutes. Indeed, for any $a\in A_{H,y}$ and $\delta\in D_yL'_N$, we get
	\begin{align*}
		\langle F^\ast \omega(a),\delta\rangle &= F^\ast\omega(a,\delta)= F_y^{-1}\left(\omega(df(a), DF(\delta))\right) \\
		&= F_y^{-1}\langle(\omega(df(a)),DF(\delta)\rangle)=\left\langle DF^\dagger\big(\omega(df(a))\big),\delta\right\rangle,
	\end{align*}
	and the left square in \eqref{eq:Df} commutes. Exchanging the roles of $a$ and $\delta$ we see that the right square in \eqref{eq:Df} commutes as well.
	
	Since \eqref{eq:Df_quis} and \eqref{eq:dualDf_quis} are quasi-isomorphisms, the cochain map induced by $F^\ast \omega$ is a quasi-isomorphism at the point $y\in N$ if and only if so is the one induced by $\omega$ at the point $x=f(y)\in M$. Hence, if the cochain map induced by $\omega$ is a quasi-isomorphism at all points in $M$, then the one induced by $F^\ast\omega$ is a quasi-isomorphism at all points in $N$. For the converse, if the cochain map induced by $F^\ast \omega$ is a quasi-isomorphism at all points in $N$, then the one induced by $\omega$ is a quasi-isomorphism at all points $f(y)\in M$. But, since $f$ is essentially surjective, it follows that, for any $x\in M$ there exists $y\in N$ such that $f(y)$ and $x$ are in the same orbit. The statement now follows from Lemma \ref{lemma:Atiyah_non-deg_orbit}.
\end{proof}

By Proposition \ref{prop:Atiyah_omega_VBGmorphism}, every multiplicative Atiyah $2$-form determines a VBG morphism. If $\omega\in \OA^2(L)$ is multiplicative, then $\omega$ determines the VBG morphism $\omega\colon DL\to J^1L$, and, if $F\colon L'\to L$ is an LBG morphism, then $F^\ast \omega\in \OA^2(L')$ is also multiplicative, and it determines the VBG-morphim $F^\ast \omega\colon DL'\to J^1L'$. Hence, by Theorem \ref{theo:caratterizzazioneVBmorita}, Proposition \ref{prop:+1-shifted_Atiyah_Morita_map} says that, if $f\colon H\to G$ is a Morita map, then $\omega\colon DL\to J^1L$ is a VB-Morita map if and only if $F^\ast \omega\colon DL'\to J^1L'$ is so. Proposition \ref{prop:Atiyah_omega_VB-Morita_if_fomega} below is the analogue of Proposition \ref{prop:omega_VB-Morita_if_fomega} and it provides an alternative and more conceptual proof of the latter fact using linear natural isomorphisms (see Definition \ref{def:LNT}).

\begin{prop}
	\label{prop:Atiyah_omega_VB-Morita_if_fomega}
	Let $(F,f)\colon (L'\to H)\to (L\to G)$ be a VB-Morita map and let $\omega\in \OA^2(L)$ be a multiplicative Atiyah form. Then $\omega\colon DL\to J^1L$ is a VB-Morita map if and only if $F^\ast \omega\colon DL' \to J^1L'$ is so.
\end{prop}
\begin{proof}
	The VBG morphisms $F^\ast\omega\colon DL'\to J^1L'$ and $\omega\colon DL \to J^1L$ fit in the following diagram
	\begin{equation}
		\label{eq:Apentagon}
		\begin{tikzcd}
			& DL' \arrow[rr, "DF"] \arrow[dl, "F^\ast\omega"'] & & DL\arrow[dr, "\omega"]\\
			J^1L' &&&& J^1L \\
			&& f^\ast J^1L \arrow[ull, "DF^\dagger"] \arrow[urr, "\pr_2"']
		\end{tikzcd},
	\end{equation}
	where $DF\colon DL'\to DL$ is a VB-Morita map because of Proposition \ref{prop:DFMorita}, and $J^1L' \xleftarrow{DF^\dagger} f^\ast J^1L \xrightarrow{ \pr_2} J^1L$ is the Morita equivalence discussed in Remark \ref{rem:twisted_dualVBG_VB-Morita_equivalent}. 
	
	Since $DF^\dagger\colon f^\ast J^1L \to J^1L'$ is a VB-Morita map covering the identity $\operatorname{id}_H$, by Proposition \ref{prop:VB_morita_on_identity}, there exist a VBG morphism $\mathcal{F}\colon J^1L' \to f^\ast J^1L$ and two linear natural isomorphisms $(T,\operatorname{id}_H)\colon DF^\dagger \circ \mathcal{F}\Rightarrow \operatorname{id}_{J^1L'}$ and $(T',\operatorname{id}_H)\colon \mathcal{F}\circ DF^\dagger \Rightarrow \operatorname{id}_{f^\ast J^1L}$. In particular, by Theorem \ref{theo:VBtransformation}, $\mathcal{H}'=T'- \mathcal{F}\circ DF^\dagger\colon f^\ast A_G^\dagger\to f^\ast J^1L_M$ is a smooth map covering the identity $\operatorname{id}_N$ that makes $F'=\mathcal{F}\circ DF^\dagger$ homotopic to $\operatorname{id}_{f^\ast J^1L'}$.
	%for any $y\in N$, there exists a homotopy $\mathcal{H'}$ between $F'=F\circ df^\ast$ and $\operatorname{id}$
	%\begin{equation*}
	%	\begin{tikzcd}
		%		0 \arrow[r] & T_x^\ast M \arrow[r, "\rho_G^\ast"] \arrow[d, "\operatorname{id}", shift left=0.5ex] \arrow[d, "F'"', shift right=0.5ex] & A_{G,x}^\ast \arrow[r] \arrow[d, "\operatorname{id}", shift left=0.5ex] \arrow[d, "F'"', shift right=0.5ex] \arrow[dl, "\mathcal{H'}"'] & 0\\
		%		0 \arrow[r] & T_x^\ast M \arrow[r, "\rho_G^\ast"'] & A_{G,x}^\ast \arrow[r] & 0
		%	\end{tikzcd},
	%\end{equation*}
	%where $x=f(y)\in M$.

	Replacing $DF^\dagger$ with $\mathcal{F}$ in Diagram \eqref{eq:Apentagon} we get the diagram
	\begin{equation*}
		\label{eq:Apentagon2}
		\begin{tikzcd}
			& DL' \arrow[rr, "DF"] \arrow[dl, "F^\ast\omega"'] & & DL\arrow[dr, "\omega"]\\
			J^1L' &&&& J^1L \\
			&& f^\ast J^1L \arrow[from=ull, "\mathcal{F}"'] \arrow[urr, "\pr_2"']
		\end{tikzcd}.
	\end{equation*}
	The latter commutes up to a linear natural isomorphism. In order to see this, first notice that both $K'=\pr_2\circ \mathcal{F}\circ F^\ast \omega$ and $K= \omega \circ DF$ are VBG morphisms from $DL'$ to $J^1L$ covering $f\colon H\to G$. In order to use Theorem \ref{theo:VBtransformation}, we need a VB morphism $\mathcal{H}\colon DL'_N\to J^1L_M$ covering $f$ that makes $K'$ homotopic to $K$.
	%to define a smooth family of homotopies $\mathcal{H}$ between $K$ and $K'$:
	%\begin{equation*}
	%	\begin{tikzcd}
		%		0 \arrow[r] & A_{H,y} \arrow[r, "\rho_H"] \arrow[d, "K", shift left=0.5ex] \arrow[d, "K'"', shift right=0.5ex] & T_yN \arrow[r] \arrow[d, "K", shift left=0.5ex] \arrow[d, "K'"', shift right=0.5ex] \arrow[dl, "\mathcal{H}"'] & 0\\
		%		0 \arrow[r] & T_x^\ast M \arrow[r, "\rho_G^\ast"'] & A_{G,x}^\ast \arrow[r] & 0
		%	\end{tikzcd},
	%\end{equation*}
	%where $y\in N$ and $x=f(y)\in M$. We define $\mathcal{H}\colon T_yN \to T_X^\ast M$ by setting
%	%\[
%	%	\mathcal{H}(v)= \mathcal{H'}(\omega(df(v))), \quad v\in T_yN.
%	%\]
%	
%	%For any $a\in A_{H,y}$, we have
%	%\begin{align*}
%	%	\mathcal{H}(\rho_H(a))&= \mathcal{H'}(\omega(df(\rho_H(a)))) \\
%	%	&=\mathcal{H'}(\omega(\rho_G(df(a))))\\
%	%	&=\mathcal{H'}(\rho_G^\ast(\omega(df(a))))\\
%	%	&= F(df^\ast(\omega(df(a)))) - \omega(df(a)) \\
%	%	&= F((f^\ast\omega)_\flat(a))- \omega(df(a))\\
%	%	&= (K'-K)(a),
%	%\end{align*}
%	%and the homotopy condition in degree $-1$ is proved.
%	
%	%For any $v\in T_yN$, we have
%	%\begin{align*}
%	%	\rho_G^\ast (\mathcal{H}(v))&= \rho_G^\ast(\mathcal{H'}(\omega(df(v))))\\
%	%	&= F(df^\ast(\omega(df(v)))) - \omega(df(v))\\
%	%	&= F((f^\ast\omega)_\flat (v))- \omega (df(v))\\
%	%	&= (K'-K)(v),
%	%\end{align*}
%	%and the homotopy condition in degree $0$ is proved. 
%	
	The composition $\omega\circ DF\colon DL'_N\to A_G^\dagger$ is a VB morphism covering $f\colon N\to M$. Consider the VB morphism from $DL'_N$ to $f^\ast A_G^\dagger$ covering the identity $\operatorname{id}_N$, again denoted by $\omega\circ DF$, defined by setting
	\[
		(\omega\circ DF)(\delta')= (y, \omega(DF(\delta'))), \quad \delta'\in D_yL'_N,
	\]
	with $y\in N$. Now we set $\mathcal{H}= \pr_2\circ \mathcal{H'}\circ \omega \circ DF$. Then $K'$ is homotopic to $K$ through $\mathcal{H}$. Indeed, for any $\delta'\in D_hL'$, with $h\in H$, we have
	\begin{align*}
		(K'-K)(\delta')&= (\pr_2\circ \mathcal{F} \circ F^\ast \omega - \omega\circ DF)(\delta')\\
		&= (\pr_2\circ \mathcal{F} \circ DF^\dagger)(\omega(DF(\delta')))- \omega(DF(\delta'))\\
		&=(\pr_2\circ J_{\mathcal{H'}}- \pr_2\circ \operatorname{id}_{f^\ast J^1L})(\omega(DF(\delta'))) - \omega(DF(\delta')) \\
		&=\pr_2(J_{\mathcal{H}'}(\omega(DF(\delta'))))\\
		&= J_{\mathcal{H}}(\delta').
	\end{align*}
	
	Finally, $\mathcal{F}\colon J^1L'\to f^\ast J^1L$ is a VB-Morita map because of Remark \ref{rem:VB_morita_natural_iso}. Then $\omega\colon DL\to J^1L$ is a VB-Morita map if and only if $K$ is so (Lemma \ref{lemma:two-out-of-three-VBG}), if and only if $K'$ is so (Remark \ref{rem:VB_morita_natural_iso}), if and only if $F^\ast \omega\colon DL'\to J^1L'$ is so (Lemma \ref{lemma:two-out-of-three-VBG} again), whence the claim. 
\end{proof}

\subsection{Symplectic Atiyah Morita equivalence}\label{sec:Asme}
In this section we prove that the notion of $+1$-shifted symplectic Atiyah structure is Morita invariant. In order to do this we introduce a suitable notion of \emph{symplectic Atiyah Morita equivalence} that uses \emph{gauge transformations} of $+1$-shifted symplectic Atiyah structures in the same spirit as in the shifted symplectic case. Then, in Proposition \ref{prop:Atiyah_symplecticMoritaequiv}, we prove that any LBG VB-Morita equivalent to a $+1$-shifted Atiyah symplectic LBG, is a $+1$-shifted symplectic LBG itself. This allows us to define $+1$-shifted symplectic Atiyah LB-stacks. 

As in Section \ref{sec:sme} we start introducing gauge transformations of $+1$-shifted symplectic Atiyah structures. Let $(L\rightrightarrows L_M;G\rightrightarrows M)$ be an LBG.
\begin{definition}
	The \emph{gauge transformation} of a $+1$-shifted symplectic Atiyah structure $(\omega, \Omega)$ on $L$ by an Atiyah $2$-form $\alpha\in \OA^2(L_M)$ is the pair $(\omega + \partial \alpha, \Omega + \dA\alpha)\in \OA^2(L)\oplus \OA^3(L_M)$.
\end{definition}

The next Proposition is the analogue of Proposition \ref{prop:gauge_transf_1-shift}:
\begin{prop}
	\label{prop:Atiyah_gauge_transf_1-shift}
	Let $(L, \omega, \Omega)$ be a $+1$-shifted symplectic Atiyah LBG. The gauge transformation of $(\omega, \Omega)$ by an Atiyah $2$-form $\alpha\in \OA^2(L_M)$ is a $+1$-shifted symplectic Atiyah structure on $L$.
\end{prop}
\begin{proof}
	Since $\Omega$ is $\dA$-closed, the Atiyah $3$-form $\Omega+\dA\alpha\in \OA^3(L_M)$ is $\dA$-closed, and since $\omega$ is multiplicative, $\omega+\partial \alpha$ is multiplicative as well. Moreover,
	\begin{equation*}
		\dA(\omega+\partial \alpha)= \dA\omega + \dA\partial \alpha = \partial\Omega+ \partial \dA\alpha = \partial(\Omega+\dA\alpha).
	\end{equation*}
	We need to prove that $\omega+\partial \alpha$ is non-degenerate. Notice that, for any $x\in M$, the cochain map induced by $\omega+\partial \alpha$ between the fibers of $DL$ and $J^1L$ at $x\in M$ is
	\begin{equation*}
		\begin{tikzcd}
			0 \arrow[r] & A_x \arrow[r, "\mathcal{D}"] \arrow[d, "\omega + \partial \alpha"'] & D_xL_M \arrow[r] \arrow[d, "\omega + \partial \alpha"] &0 \\
			0 \arrow[r] & J^1_xL_M \arrow[r, "\mathcal{D}^{\dagger}"'] & A_x^{\dagger} \arrow[r] &0
		\end{tikzcd},
	\end{equation*}
	but $\partial \alpha = s^\ast \alpha - t^\ast \alpha \colon A_x \to D_x L_M$ vanishes on $\ker \mathcal{D}$, showing that $\omega$ and $\omega+\partial \alpha$ do actually induce the same map in the $(-1)$-cohomologies, hence, by Remark \ref{rem:Atiyah_quasi_iso_-1_0}, in $0$-cohomologies as well.
\end{proof}
\begin{rem}
	Unlike for the symplectic case, a $+1$-shifted symplectic Atiyah structure can always be gauge transformed to one with the $3$-form being zero. Indeed, let $(\omega, \Omega)$ be a $+1$-shifted symplectic Atiyah structure on an LBG $L$. Then, $\dA\Omega=0$. But the complex $(\OA^\bullet(L_M), \dA)$ is acyclic, so there exists $\alpha\in \OA^2(L_M)$ such that $\Omega=\dA \alpha$. Now the gauge transformation of $(\omega,\Omega)$ by $-\alpha$ is the pair $(\omega-\dA\alpha,0)$.
\end{rem}

The next step consists in proving that the pullbacks of $+1$-shifted symplectic Atiyah structures along VB-Morita maps are again $+1$-shifted symplectic Atiyah structures.
\begin{prop}
	\label{prop:Atiyah_pullback_1shifted}
	Let $(L, \omega, \Omega)$ be a $+1$-shifted symplectic Atiyah LBG and let $(F,f)\colon (L'\rightrightarrows L'_N;H\rightrightarrows N) \to (L\rightrightarrows L_M;G\rightrightarrows M)$ be a VB-Morita map. Then $(L', F^{\ast}\omega, F^{\ast}\Omega)$ is a $+1$-shifted symplectic Atiyah LBG as well.
\end{prop}
\begin{proof}
	Since the Atiyah differential is natural and $F\colon L'\to L$ is an LBG morphism, we get $\dA F^{\ast}\Omega=F^\ast \dA\omega=0$, $$\partial F^{\ast}\omega= F^{\ast}\partial\omega=0,$$so $F^\ast \omega$ is multiplicative, and
	\begin{equation*}
		\dA F^\ast \omega= F^\ast \dA\omega= F^\ast \partial \Omega=\partial F^\ast \Omega.
	\end{equation*}
	Finally, from Proposition \ref{prop:+1-shifted_Atiyah_Morita_map}, or Proposition \ref{prop:Atiyah_omega_VB-Morita_if_fomega}, we have that $F^\ast \omega$ is non-degenerate.
\end{proof}

We are now ready to introduce Morita equivalence between $+1$-shifted symplectic Atiyah LBGs.
\begin{definition}
	Two $+1$-shifted symplectic Atiyah LBGs $(L_1, \omega_1, \Omega_1)$ and $(L_2, \omega_2, \Omega_2)$ are \emph{symplectic Atiyah Morita equivalent} is there exist an LBG $L'$ and VB-Morita maps
	\begin{equation*}
		\begin{tikzcd}
			& L' \arrow[dl, "F_1"'] \arrow[dr, "F_2"] \\
			L_1 & & L_2
		\end{tikzcd}
	\end{equation*}
	such that the $+1$-shifted symplectic Atiyah structures $(F_1^\ast \omega_1, F_1^\ast \Omega_1), (F_2^\ast \omega_2, F_2^\ast \Omega_2)$ agree up to a gauge transformation.
\end{definition}
 
 The next proposition is analogous to Proposition \ref{prop:sympl_Morita_equiv} and establishes that symplectic Atiyah Morita equivalence is indeed an equivalence relation.
\begin{prop}
	\label{prop:Atiyah_sympl_Morita_equiv}
	Symplectic Atiyah Morita equivalence is an equivalence relation.
\end{prop}
\begin{proof}
	Reflexivity and simmetry are obvious. For the transitivity, let $(L_1, \omega_1, \Omega_1), (L_2, \omega_2, \Omega_2)$ and $(L_3, \omega_3, \Omega_3)$ be $+1$-shifted symplectic Atiyah LBGs and let
	\begin{equation*}
		\begin{tikzcd}[column sep=-3]
			& (L', \alpha_1) \arrow[dl, "F_1"'] \arrow[dr, "F_2"]  & & (L'', \alpha_2)   \arrow[dl, "K_1"'] \arrow[dr, "K_2"]  &\\
			(L_1, \omega_1, \Omega_1) & & (L_2, \omega_2, \Omega_2) & & (L_3, \omega_3, \Omega_3)
		\end{tikzcd}
	\end{equation*}
	be two symplectic Atiyah Morita equivalences, i.e.,
	\begin{equation}
		\label{eq:Atransitivity}
		\begin{aligned}
			F_2^{\ast}\omega_2 - F_1^\ast\omega_1 &= \partial \alpha_1, \quad F_2^\ast\Omega_2 - F_1^\ast\Omega_1 = \dA\alpha_1, \\
			K_2^\ast\omega_3 -K_1^\ast\omega_2& = \partial \alpha_2, \quad
			K_2^\ast \Omega_3 -K_1^\ast\Omega_2 = \dA\alpha_2.
		\end{aligned}
	\end{equation}
	As discussed in Proposition \ref{prop:VBequiv_relation}, the homotopy fiber product $L' \times_{L_2}^h L''$ of $L'\xrightarrow{F_2} L_2 \xleftarrow{K_1} L''$ (see Example \ref{ex:homotopy_pullback}) exists and it is an LBG (see Remark \ref{rem:hom_LBG}). Moreover it comes with two VB-Morita maps $ L' \xleftarrow{F} L' \times_{L_2}^h L'' \xrightarrow{K} L''$ fitting in the following diagram: 
	\begin{equation}
		\label{eq:Ahomot_fiber_prod_sympl}
		\begin{tikzcd}[column sep=-3]
			& &L' \times_{L_2}^h L'' \arrow[dl, "F"'] \arrow[dr, "K"]  & & \\
			& \big(L', \alpha_1\big) \arrow[dl, "F_1"'] \arrow[dr, "F_2"]  & & \big(L'', \alpha_2\big)   \arrow[dl, "K_1"'] \arrow[dr, "K_2"]  &\\
			\big(L_1, \omega_1, \Omega_1\big) & & \big(L_2, \omega_2, \Omega_2\big) & & \big(L_3, \omega_3, \Omega_3\big)
		\end{tikzcd},
	\end{equation}
	where every VBG morphism is a VB-Morita map. The middle square in \eqref{eq:Ahomot_fiber_prod_sympl} commutes up to a linear natural isomorphism $T\colon F_2\circ F \Rightarrow K_1\circ K$ (see Example \ref{ex:lni_homotopy}). Consider $\beta=-T^\ast \omega_2$, which is an Atiyah $2$-form on the base of the homotopy pullback $L' \times_{L_2}^h L''$. From the naturality of $T$ we have that
	\begin{equation*}
		F_2\circ F = m\circ \left(i\circ T \circ t, m\circ (K_1\circ K, T \circ s) \right),
	\end{equation*}
	then, from the multiplicativity of $\omega_2$, it follows that 
	\begin{align*}
		(F_2\circ F)^\ast \omega_2 &= \left(m\circ \left(i\circ T \circ t, m\circ (K_1\circ K, T \circ s) \right)\right)^\ast \omega_2 \\
		&=\left(i\circ T \circ t, m\circ (K_1\circ K, T \circ s) \right)^\ast (m^\ast \omega_2)\\
		&= \left(i\circ T \circ t, m\circ (K_1\circ K, T \circ s) \right)^\ast (\pr_1^\ast\omega_2 +\pr_2^\ast \omega_2)\\
		& =(i\circ T\circ t)^\ast \omega_2 + \left(m\circ (K_1\circ K, T \circ s)\right)^\ast\omega_2\\
		& = (i\circ T\circ t)^\ast \omega_2 + (K_1\circ L, T \circ s)^\ast(m^\ast\omega_2)\\
		&=(i\circ T\circ t)^\ast \omega_2 + (K_1\circ K, T \circ s)^\ast(\pr_1^\ast\omega_2 +\pr_2^\ast\omega_2)\\
		&=(i\circ T\circ t)^\ast \omega_2 + (K_1\circ K)^\ast \omega_2 + (T\circ s)^\ast \omega_2\\
		&= -(T\circ t)^\ast \omega_2  + (K_1\circ K)^\ast \omega_2 + (T\circ s)^\ast \omega_2,
	\end{align*}
	where, in the last step, we used that $i^\ast\omega=-\omega$ from point $ii)$ in Proposition \ref{prop:Atiyah_formule}. Hence
	\begin{equation}
		\label{eq:Apartialbeta}
		(K_1\circ K)^\ast \omega_2 - (F_2\circ F)^\ast \omega_2 = (T\circ t)^\ast \omega_2 -(T\circ s)^\ast \omega_2=t^\ast T^\ast \omega_2 - s^\ast T^\ast \omega_2= \partial \beta.
	\end{equation}
	Moreover, we get 
	\begin{equation}
		\label{eq:Adifferentialbeta}
		\begin{aligned}
			(K_1\circ K)^\ast \Omega_2 - (F_2\circ F)^\ast \Omega_2 &=(t\circ T)^\ast \Omega_2 - (s\circ T)^\ast \Omega_2 \\
			&=T^\ast t^\ast \Omega_2 - T^\ast s^\ast \Omega_2 = -T^\ast \partial\Omega_2 =- \dA T^\ast \omega_2= \dA\beta.
		\end{aligned}
	\end{equation}
	From \eqref{eq:Atransitivity} and \eqref{eq:Apartialbeta} we have
	\begin{align*}
		(K_2\circ K)^\ast \omega_3 - (F_1\circ F)^\ast \omega_1 &= K^\ast (K_2^\ast\omega_3) - F^\ast(F_1^\ast\omega_1) \\
		&= K^\ast\left(K_1^\ast\omega_2 +\partial \alpha_2\right) - F^\ast \left(F_2^\ast \omega_2 -\partial \alpha_1\right) \\
		&=K^\ast(K_1^\ast \omega_2) + K^\ast (\partial \alpha_2) -F^\ast(F_2^\ast\omega_2) +F^\ast (\partial \alpha_1) \\
		&= (K_1\circ K)^\ast \omega_2 - (F_2\circ F)^\ast \omega_2 + \partial (K^\ast \alpha_2) +\partial (F^\ast \alpha_1) \\
		&=\partial \beta +\partial (K^\ast \alpha_2) + \partial (F^\ast \alpha_1) \\
		&=\partial (\beta + K^\ast\alpha_2 +F^\ast\alpha_1),
	\end{align*}
	and, from \eqref{eq:Atransitivity} and \eqref{eq:Adifferentialbeta} we get
	\begin{align*}
		(K_2\circ K)^\ast \Omega_3 - (F_1\circ F)^\ast \Omega_1 &= K^\ast (K_2^\ast\Omega_3) - F^\ast(F_1^\ast\Omega_1) \\
		&=K^\ast\left(K_1^\ast\Omega_2 +\dA \alpha_2\right) - F^\ast \left(F_2^\ast \Omega_2 -\dA \alpha_1\right) \\
		&=K^\ast(K_1^\ast \Omega_2) + K^\ast (\dA\alpha_2) -F^\ast(F_2^\ast\Omega_2) +F^\ast (\dA\alpha_1) \\
		&=(K_1\circ K)^\ast \Omega_2 - (F_2\circ F)^\ast \Omega_2 + \dA K^\ast \alpha_2 + \dA F^\ast \alpha_1 \\
		&=\dA\beta + \dA (K^\ast \alpha_2) + \dA (F^\ast \alpha_1) \\
		& = \dA(\beta + K^\ast \alpha_2 + F^\ast \alpha_1).
	\end{align*}
	Hence the $+1$-shifted symplectic Atiyah structures $((F_1\circ F)^\ast\omega_1, (F_1\circ F)^\ast \Omega_1)$ and $((K_2\circ K)^\ast\omega_3, (K_2\circ K)^\ast \Omega_3) $ agree up to the gauge transformation by $\beta+K^\ast \alpha_2 + F^\ast \alpha_1$.
\end{proof}

The next result is the ``Atiyah version'' of Proposition \ref{prop:symplecticMoritaequiv}. In particular, we prove that if an LBG is VB-Morita equivalent to a $+1$-shifted Atiyah symplectic LBG, then it possesses a $+1$-shifted symplectic Atiyah structure.
\begin{prop}
	\label{prop:Atiyah_symplecticMoritaequiv}
	Let $(L_1, \omega_1, \Omega_1)$ be a $+1$-shifted symplectic Atiyah LBG and let $L_2$ be a VB-Morita equivalent LBG. Then there exists a $+1$-shifted symplectic Atiyah structure $(\omega_2, \Omega_2)$ on $L_2$ such that $(L_1, \omega_1, \Omega_1)$ and $(L_2, \omega_2, \Omega_2)$ are symplectic Atiyah Morita equivalent. Moreover such $(\omega_2, \Omega_2)$ is unique up to gauge transformations.
\end{prop}
\begin{proof}
	The LBGs $L_1$ and $L_2$ are VB-Morita equivalent. Then, by Corollary \ref{cor:VB_Mor_equiv_LBGs}, there exist an LBG $L'$ and two VB-Morita maps
	\begin{equation*}
		\begin{tikzcd}
			& L' \arrow[dr, "F_2"] \arrow[dl, "F_1"'] \\
			L_1 & & L_2
		\end{tikzcd}.
	\end{equation*}
	Consider the double complex
	\begin{equation}
		\label{eq:AtruncatedBSS}
		{\scriptsize
		\begin{tikzcd}
			&\vdots&\vdots 
			\\
			0\arrow[r] &\OA^4(L'_N) \arrow[u, "\dA"] \arrow[r, "\partial"] &\Omega^4_{D,\operatorname{mult}}(L')\arrow[u, "\dA"] \arrow[r] & 0
			\\
			0\arrow[r] &\OA^3(L'_N) \arrow[u, "\dA"] \arrow[r, "\partial"] &\Omega^3_{D,\operatorname{mult}}(L')\arrow[u,"\dA"] \arrow[r] & 0
			\\
			0\arrow[r] &\OA^2(L'_N) \arrow[u, "\dA"]\arrow[r,"\partial"] &\Omega^2_{D,\operatorname{mult}}(L)\arrow[u,"\dA"] \arrow[r]& 0
			%\\
			%0\arrow[r] & \Gamma(L'_N) \arrow[u, "\dA"]\arrow[r,"\partial"] &\Gamma_{\operatorname{mult}}(L')\arrow[u,"\dA"] \arrow[r]& 0
			\\
			&0 \arrow[u] &0\arrow[u] 
		\end{tikzcd}}
	\end{equation}
	obtained from the double complex \eqref{eq:Atiyah_double_complex} associated to $L'$ by deleting the first two rows and every column except the first two. We denote by $(\ttot(L'), D)$ the total complex of the double complex \eqref{eq:AtruncatedBSS} (likewise for $L_1$ and $L_2$). From Remark \ref{rem:partial_Atiyah_Morita_invariant}, every row in the double complex \eqref{eq:AtruncatedBSS} is Morita invariant up to quasi-isomorphisms, then, from standard spectral sequence arguments, the double complex \eqref{eq:AtruncatedBSS} and its total complex are Morita invariant up to quasi-isomorphisms. In other words, the pullback maps $$F_1^\ast \colon (\ttot (L_1), D) \to (\ttot (L'), D)$$ and $$F_2^\ast \colon (\ttot (L_2), D) \to (\ttot (L'), D)$$ are quasi-isomorphisms. The pair $(\omega_1,\Omega_1)$ is a $3$-cocycle in the total complex $(\ttot(L_1), D)$, then there exists a $3$-cocycle $(\omega_2, \Omega_2)\in \OA^2(L_2)\oplus \OA^3(L_{M_2})$ in $(\ttot(L_2), D)$ such that the cohomology classes of $(F_1^\ast\omega_1, F_1^\ast \Omega_1)$ and $(F_2^\ast \omega_2, F_2^\ast\Omega_2)$ in $(\ttot(L'), D)$ agree, where $L_{M_2}$ is the side bundle of $L_2$. This means that there exist $\alpha\in \OA^2(L'_N)$ such that
	\[
		F_1^\ast \Omega_1 - F_2^\ast \Omega_2 = \dA\alpha, \quad \text{and}\quad F_1^\ast \omega_1 - F_2^\ast \omega_2= \partial \alpha.
	\]
	%Since $F_2^\ast \colon (\OA^\filleddiamond(L_2^{(\bullet)}), \partial)\to (\OA^\filleddiamond (L'^{(\bullet)}), \partial)$ is a quasi-isomorphism, then there exists $\widetilde{\beta}\in \OA^1(L_2)$, such that $\partial \widetilde{\beta}=0$ and $F_2^\ast \widetilde{\beta}$ and $\beta$ are $\partial$-cohomologous, i.e., there exists $\gamma\in \OA^1(L'_N)$ such that $F_2^\ast \widetilde{\beta}= \beta + \partial \gamma$.
	%
	%We set $\omega_2=\widetilde{\omega}_2 + \dA\widetilde{\beta}$ and $\Omega_2=\widetilde{\Omega}_2$. The pair $(\omega_2, \Omega_2)\in \OA^2(L_2)\oplus\OA^3(L_{M_2})$ is a $3$-cocycle in $(\ttot (L_2), D)$. Indeed,
	%\[
	%\partial \omega_2=\partial \widetilde{\omega}_2 + \dA\partial \widetilde{\beta} =0, \quad \dA\Omega_2= \dA\widetilde{\Omega}_2=0, \quad \dA\omega_2= \dA\widetilde{\omega}_2 = \partial \widetilde{\Omega}_2 = \partial \Omega_2.
	%\]
	Then  $(F_1^\ast\omega_1, F_1^\ast \Omega_1)$ and $(F_2^\ast \omega_2, F_2^\ast\Omega_2)$ agree up to the gauge transformation by $\alpha$. %Indeed, 
	%\begin{align*}
	%	F_1^\ast \omega_1 - F_2^\ast\omega_2 &= F_1^\ast\omega_1 - F_2^\ast \widetilde{\omega}_2 + \dA F_2^\ast \widetilde{\beta} \\
	%	&=\partial \alpha -\dA \beta + \dA\beta + \dA \partial \gamma \\
	%	&= \partial (\alpha+ \dA\gamma),
	%\end{align*}
	%and
	%\begin{equation*}
	%	F_1^\ast \Omega_1 - F_2^\ast \Omega_2= F_1^\ast \Omega_1 - F_2^\ast \widetilde{\Omega}_2 = \dA\alpha = \dA(\alpha + \dA \gamma).
	%\end{equation*}
	Since $(\omega_1,\Omega_1)$ is a $+1$-shifted symplectic Atiyah structure, then, by Proposition \ref{prop:Atiyah_pullback_1shifted}, the pullback $(F_1^\ast\omega_1, F_1^\ast\Omega_1)$ is so, by Proposition \ref{prop:Atiyah_gauge_transf_1-shift}, $(F_2^\ast\omega_2, F_2^\ast\Omega_2)$ is so, by Proposition \ref{prop:+1-shifted_Atiyah_Morita_map}, $(\omega_2, \Omega_2)$ is so, and, the $+1$-shifted symplectic Atiyah LBGs $(L_1,\omega_1,\Omega_1)$ and $(L_2, \omega_2,\Omega_2)$ are symplectic Atiyah Morita equivalent.
	
	For the second part of the statement, let $(\omega_2', \Omega_2')$ be another $+1$-shifted symplectic Atiyah structure on $L_2$ such that $(L_2, \omega_2', \Omega_2')$ and $(L_1,\omega_1, \Omega_1)$ are symplectic Atiyah Morita equivalent through $L'$. Then $(\omega_2, \Omega_2)$ and $(\omega_2', \Omega_2')$ are in the same cohomology class of the total complex $(\ttot(L_2), D)$. This means that there exists $\alpha\in \OA^2(L_{M_2})$ such that
	\[
	\omega_2-\omega_2'= \partial \alpha, \quad \text{and} \quad \Omega_2-\Omega_2'=\dA \alpha. 
	\]
	Hence $(\omega_2, \Omega_2)$ and $(\omega'_2, \Omega'_2)$ agree up to the gauge transformation by $\alpha$.
\end{proof}

Proposition \ref{prop:Atiyah_symplecticMoritaequiv} motivates the following
\begin{definition}
	A $+1$-shifted symplectic Atiyah LB-stack is a symplectic Atiyah Morita equivalence class of $+1$-shifted symplectic Atiyah LBGs.
\end{definition}

\begin{rem}
	It is clear by the proof of Proposition \ref{prop:Atiyah_symplecticMoritaequiv} that the pullback of a VB-Morita map $F\colon L'\to L$ determines a bijection between the cohomology classes in the total complex of the truncated double complex \eqref{eq:AtruncatedBSS} of $+1$-shifted symplectic Atiyah structures on $L$ and $L'$. Moreover two $+1$-shifted symplectic Atiyah LBGs $(L_1,\omega_1,\Omega_1)$ and $(L_2,\omega_2, \Omega_2)$ are symplectic Atiyah Morita equivalent if and only if $L_1$ and $L_2$ are VB-Morita equivalent and this equivalence maps the cohomology class $[(\omega_1, \Omega_1)]$ in $(\ttot(L_1), D)$ to the cohomology class $[(\omega_2,\Omega_2)]$ in $(\ttot(L_2), D)$. Hence, a $+1$-shifted symplectic Atiyah structure on an LB-stack $[L_M/L]\to [M/G]$ can be seen as a cohomology class $[(\omega,\Omega)]$ in $(\ttot(L), D)$ of a $+1$-shifted symplectic Atiyah structure $(\omega,\Omega)$ on an LBG $L$ presenting $[L_M/L]$.
\end{rem}
	\chapter{Shifted contact structures}\label{ch:scs}

This final chapter is the core of the thesis. Here we introduce the notion of shifted contact structure. As seen in Section \ref{sec:contact}, a contact structure on a manifold $M$ is essentially defined in five steps: a line bundle $L$ over $M$, an $L$-valued $1$-form $\theta$, the kernel of $\theta$, the curvature of $\theta$, and a non-degeneracy condition. For both the $0$ and the $+1$-shifted cases we follow the same five steps: we begin by considering a line bundle over a Lie groupoid, i.e., an LBG. We introduce a notion of LBG-valued $1$-form $\theta$. We make sense of the kernel of $\theta$ in this higher setting, and define a curvature associated to $\theta$. Finally, we introduce a non-degeneracy condition. 

In the shifted context, the kernel of a $1$-form $\theta$ is given by a cochain complex, analogous to tangent complex of a Lie groupoid. Accordingly, the curvature is given by a cochain map. We define all these notions in such a way that they are Morita invariant in an appropriate sense, and so they are well-defined notions on differentiable stacks. Moreover, we show that our notion of shifted contact structure agrees with the corresponding shifted Atiyah symplectic structure, providing a strong motivation for our definition.

The chapter is divided into three parts. In the first section, we define shifted LBG-valued $1$-forms. In the second section we focus on $0$-shifted contact structures, and in the third one we discuss $+1$-shifted contact structures.

\section{Shifted LBG-valued $1$-forms}\label{sec:shifted_vv_forms}
In this section, we introduce the notions of \emph{basic} and \emph{multiplicative} LBG-valued $1$-form. We begin with an LBG and construct a cochain complex of LB-valued forms on its nerve. This provides a new notion of multiplicative LBG-valued $1$-form which actually agrees with the one introduced in \cite{CSS15}. Finally, we introduce the concept of shifted LBG-valued $1$-form.

We begin by considering an LBG
\begin{equation*}
	\begin{tikzcd}
		L \arrow[r,shift left=0.5ex] \arrow[r, shift right=0.5 ex] \arrow[d] &L_M \arrow[d] \\
		G \arrow[r,shift left=0.5ex] \arrow[r, shift right=0.5 ex] & M
	\end{tikzcd}.
\end{equation*}
By Lemma \ref{lemma:VBG_morphism_regular}, the nerve $L^{(\bullet)}\to G^{(\bullet)}$ of $L$ is a simplicial line bundle whose face maps are all LB morphisms. Then we can construct the complex
\begin{equation}
		\label{eq:complex_valuedform}
		\begin{tikzcd}
			0 \arrow[r]& \Omega^{\bullet}(M,L_M) \arrow[r, "\partial"] & \Omega^{\bullet}(G,L)\arrow[r, "\partial"] & \Omega^{\bullet}(G^{(2)},L^{(2)})\arrow[r, "\partial"] & \cdots
		\end{tikzcd}
\end{equation}
where the differential $\partial$ is
\[
	\partial = \sum_{i=0}^k (-1)^i \partial_i^\ast\colon \Omega^\bullet(G^{(k-1)},L^{(k-1)})\to\Omega^\bullet(G^{(k)},L^{(k)}),
\] 
the alternating sum of the pullbacks of VB-valued forms along the face maps $\partial_i\colon L^{(k)}\to L^{(k-1)}$ (see Equation \ref{eq:pullback_vector_valued_forms}).

In the next remark, we explain in which precise sense Complex \eqref{eq:complex_valuedform} is Morita invariant. This property, in turn, allows us to view \eqref{eq:complex_valuedform} as an object on the LB-stack $[L_M/L]\to [M/G]$.
\begin{rem}
	\label{rem:vv_forms_Morita_equiv}
	The complex \eqref{eq:complex_valuedform} is Morita invariant up to quasi-isomorphisms. This is proved in \cite[Theorem 4.4]{DE19} by noting that, by Corollary \ref{coroll:VBMorita}, any VB-Morita map $F\colon (L'\rightrightarrows L'_N)\to (L\rightrightarrows L_M)$ between LBGs is given by two regular VB morphisms $F\colon L'\to L$ and $F\colon L'_N\to L_M$, so that the LBG $L'$ is just the pullback VBG (see Example \ref{ex:pullback_VBG}) of $L$ along a Morita map. But, the pullback along a VB-Morita map between VBGs induces a quasi-isomorphism between the associated complexes of VB-valued differential forms.
\end{rem}

Similarly as in Sections \ref{sec:shifted_structures} and \ref{sec:Ashifted}, we use Complex \eqref{eq:complex_valuedform} to introduce the notions of \emph{basic} and \emph{multiplicative} LB-valued forms on an LBG. Let $(L\rightrightarrows L_M)$ be an LBG.
\begin{definition}
	\label{def:multiplicative_vv_forms}
	An $L_M$-valued form $\theta\in \Omega^\bullet(M,L_M)$ is \emph{basic} it is a $0$-cocycle in \eqref{eq:complex_valuedform}.
	An $L$-valued form $\theta\in \Omega^\bullet(G,L)$ is \emph{multiplicative} if it is a $1$-cocycle in \eqref{eq:complex_valuedform}. We denote by $\Omega^k_{\operatorname{mult}}(G,L)$ the space of multiplicative $L$-valued $k$-forms.
\end{definition}

\begin{rem}
	\label{rem:theta_basic}
	Let $\theta\in \Omega^1(M,L_M)$ be a basic $1$-form on an LBG $L$. Then $(s^\ast\theta, \theta)\colon (TG\rightrightarrows TM)\to (L\rightrightarrows L_M)$ is a VBG morphism covering the identity $\operatorname{id}_G$.
\end{rem}
\begin{rem}
	\label{rem:theta_morphisms}
	Let $\theta\in \Omega^1(G,L)$ be a multiplicative $1$-form on an LBG $L$. We already explained in Section \ref{sec:trivial_core} that $L\rightrightarrows L_M$ encodes the same information as the trivial side bundle VBG $L\rightrightarrows 0$, with core $L_M$. In fact $\theta$ can be seen as a VBG morphism from the tangent VBG $TG\rightrightarrows TM$ to $L\rightrightarrows 0$ covering the identity $\operatorname{id}_G$. Finally, if we consider the LBG $L\rightrightarrows L_M$, then an $L$-valued $1$-form is multiplicative in the sense of Definition \ref{def:multiplicative_vv_forms} if and only if it is multiplicative in the sense of \cite{DE19}.
\end{rem}

In \cite[Definition 2.1]{CSS15} the authors define a multiplicative form on a Lie groupoid $G\rightrightarrows M$ with values in a representation $E_M$ as a form $\theta\in \Omega^\bullet(G,t^\ast E_M)$ satisfying an appropriate condition. In the next proposition we prove that this definition of multiplicative VB-valued forms agrees with ours under the isomorphism $L\cong t^\ast L_M$.

\begin{prop}
	A $k$-form $\theta\in \Omega^k(G,L)$ is multiplicative if and only if $\theta$ satisfies the formula
	\begin{equation}
		\label{eq:formula_molt}
		t\big(\theta(v_1 w_1, \dots, v_k w_k)\big)= t\big(\theta(v_1, \dots, v_k)\big) + g . t\big(\theta(w_1, \dots , w_k)\big),
	\end{equation}
	for all $(v_1,w_1), \dots , (v_k, w_k)\in T_{(g,h)}G^{(2)}$, with $(g,h)\in G^{(2)}$.
\end{prop}
\begin{proof}
	For any $(v_1,w_1), \dots , (v_k, w_k)\in T_{(g,h)}G^{(2)}$, with $(g,h)\in G^{(2)}$, we have
	\begin{align*}
		(m^{\ast}\theta)_{(g,h)}\big((v_1,w_1), \dots, (v_k,w_k)\big)&=m_{(g,h)}^{-1}\big(\theta(v_1w_1, \dots , v_kw_k)\big) \\
		&= m_{(g,h)}^{-1}\left(t_{gh}^{-1}\left(t\big(\theta(v_1w_1, \dots , v_kw_k)\big)\right)\right),
	\end{align*}
	and 
	\begin{align*}
		&(\pr_1^{\ast}\theta+\pr_2^{\ast}\theta)_{(g,h)}\big((v_1,w_1), \dots, (v_k,w_k)\big)\\
		&\quad = \pr_{1,(g,h)}^{-1}\big(\theta(v_1, \dots,v_k)\big) + \pr_{2,(g,h)}^{-1}\left(t_h^{-1}\big(t(\theta(w_1, \dots, w_k))\big)\right) \\
		&\quad =\pr_{1,(g,h)}^{-1}\big(\theta(v_1, \dots,v_k)\big) + \pr_{1,(g,h)}^{-1}\left(s_g^{-1}\big(t\big(\theta(w_1, \dots, w_k)\big)\big)\right)\\
		&\quad= \pr_{1,(g,h)}^{-1}\left(t_g^{-1}\Big(t\big(\theta(v_1, \dots ,v_k)\big) + g. t\big(\theta(w_1, \dots , w_k)\big)\Big)\right)\\
		&\quad= m_{(g,h)}^{-1}\left(t_{gh}^{-1}\Big(t\big(\theta(v_1, \dots ,v_k)\big) + g. t\big(\theta(w_1, \dots , w_k)\big)\Big)\right),
	\end{align*}
	where we used that $t\circ \pr_2= s\circ \pr_1$ and $t\circ \pr_1= t\circ m$. Then, $\theta$ is multiplicative if and only if, for any $(v_1,w_1), \dots , (v_k, w_k)\in T_{(g,h)}G^{(2)}$, with $(g,h)\in G^{(2)}$, we have
	\begin{align*}
		m_{(g,h)}^{-1}\left(t_{gh}^{-1}\left(t\big(\theta(v_1w_1, \dots , v_kw_k)\big)\right)\right)&= m_{(g,h)}^{-1}\left(t_{gh}^{-1}\Big(t\big(\theta(v_1, \dots ,v_k)\big) \right. \\
		&\quad \left. + g. t\big(\theta(w_1, \dots , w_k)\big)\Big)\right),
	\end{align*}
	whence the claim.
\end{proof}

\begin{rem}
	\label{rem:mult_Atiyah_components}
	From Equation \ref{eq:Atiyah_components}, it follows that an Atiyah form is $\partial$-closed if and only if its components are so. In particular, an Atiyah form on $L$ is multiplicative if and only if its components are so.
\end{rem}

\begin{definition}
	\label{def:shifted_vv_forms}
	A \emph{$k$-shifted LBG-valued $m$-form} on an LBG $L$ is a $\partial$-closed $L^{(k)}$-valued $m$-form.
\end{definition}
Notice that $0$ and $+1$-shifted LBG-valued forms on $L$ are just basic and multiplicative LB-valued forms on $L$.

%\begin{rem}
%	From Remark \ref{rem:vv_forms_Morita_equiv} the notion of shifted VB-valued form is Morita invariant up to quasi-isomorphisms. Then the $\partial$-cohomology class $[\theta]$ of a $k$-shifted VB-valued $m$-form $\theta\in \Omega^m(G^{(k)}, L^{(k)})$ can be considered as a $k$-shifted VB-valued $m$-form on the LB-stack $[L_M/L]\to [M/G]$.
%\end{rem}

Unlike the case of Atiyah forms, there is no analogue of the BSS double complex in the case of VB-valued forms unless we choose a flat multiplicative connection.

In the next sections we will introduce the notions of $0$ and $+1$-shifted contact structure. Here we just want to anticipate that, in both the $0$ and the $+1$-shifted cases, the structure is determined by an LBG-valued $1$-form that satisfies an appropriate non-degeneracy condition. The latter condition is formulated in terms of a quasi-isomorphism between a complex (that plays the role of kernel of the $1$-form) and its twisted dual. As in the symplectic and symplectic Atiyah cases, this cochain map is non-zero only for $k$-shifted structures with $k=0,1$, or $2$.

\section{$0$-Shifted Contact Structures}\label{sec:0-shifted_cs}

In this section, we define $0$-shifted contact structures. As already mentioned, we start with an LBG $(L\rightrightarrows L_M;G\rightrightarrows M)$, and consider a basic $L_M$-valued $1$-form $\theta\in \Omega^1(M,L_M)$. The key difference from the case of standard contact $1$-forms on manifolds is that $\theta$ is not required to be nowhere-zero. In fact, the latter property is not Morita invariant. As we do not assume $\theta\neq 0$, the kernel of $\theta$ is not a well-defined VB: its rank can vary from point to point. To address this, we introduce a suitable replacement for $\ker \theta$ and prove its Morita invariance in Section \ref{sec:mk_0}. The next step is to define the curvature in a Morita invariant way, which we cover in Section \ref{sec:mc_0}. Finally, in Section \ref{sec:def_0}, we present the ultimate definition, prove its Morita invariance, and establish its equivalence with the notion of $0$-shifted symplectic Atiyah structure (see Section \ref{sec:0_shifted_A}). We conclude the section with several examples. 

\subsection{Morita Kernel}\label{sec:mk_0}
In this subsection, we focus on constructing the kernel of a basic $1$-form. We show that the role of the kernel is played by a RUTH of the underlying Lie groupoid, which is concentrated in degrees $-1$, $0$ and $1$. Additionally, we prove that when the $1$-form is nowhere-zero the kernel simplifies and is described by a RUTH concentrated in degrees $-1$ and $0$. Finally, we discuss the Morita invariance of the kernel, ensuring its well-definedness as a context on differentiable stacks.

Let $(L\rightrightarrows L_M; G\rightrightarrows M)$ be an LBG. By Definition \ref{def:multiplicative_vv_forms}, a basic $L$-valued $1$-form is an $L_M$-valued $1$-form $\theta\in \Omega^1(M,L_M)$ that is $\partial$-closed. This means that $s^{\ast}\theta=t^{\ast}\theta$, and implies that $\im \rho \subseteq \ker \theta$. We now rephrase this inclusion, in an apparently oversophisticated way, by noting that it is equivalent to $\theta \colon T M \to L_{M}$ being a cochain map between the core complexes of the tangent groupoid $TG$ and that of the LBG $L$:
\begin{equation}\label{eq:theta_cm}
	\begin{tikzcd}
		0 \arrow[r] &A \arrow[r, "\rho"] \arrow[d] & T M \arrow[r] \arrow[d, "\theta"]  & 0 \\
		0 \arrow[r] & 0 \arrow[r] & L_{M} \arrow[r]  &0
	\end{tikzcd}.
\end{equation}

We would like to take the kernel of $\theta$. However, first of all, $\ker \theta$ might not be a well-defined (smooth) distribution, as its dimension might jump (see Example \ref{exmpl:0-shifted_nonreg} below). More importantly, it turns out that, in order to have Morita invariance of the kernel, it is actually necessary to take the \emph{homotopy kernel} of $\theta$ seen as the cochain map \eqref{eq:theta_cm}. In other words, we have to take the mapping cocone of  \eqref{eq:theta_cm} which is (up to a conventional sign):
\begin{equation}
	\label{eq:Morita0-kernel}
	\begin{tikzcd}
		0 \arrow[r] & A\arrow[r, "\rho"] & TM \arrow[r, "\theta"] & L_{M}\arrow[r] &0
	\end{tikzcd},
\end{equation}
where the non-trivial terms are concentrated in degrees $-1,0,1$.

The complex \eqref{eq:Morita0-kernel} is not just a cochain complex. Actually it can be promoted to a RUTH of $G$. We discuss this in the following
\begin{rem}\label{rem:MK_RUTH}
	Let $(L \rightrightarrows L_M ; G \rightrightarrows M)$ be an LBG and let $\theta \in \Omega^1 (M, L_M)$ be such that $\partial \theta = 0$. By Remark \ref{rem:theta_basic}, $\theta$ determines a VBG morphism $(s^\ast\theta,\theta)$ from $TG$ to $L$ and so it corresponds to a RUTH morphism. We now describe this RUTH morphism in detail using a more direct approach. By Example \ref{ex:adjointRUTH}, choosing an Ehresmann connection $h$ on $G$ (see Definition \ref{def:Ehresmann_connection}), the core complex of $TG$
	\begin{equation*}
		\begin{tikzcd}
			0 \arrow[r] & A \arrow[r, "\rho"] & TM \arrow[r]&0
		\end{tikzcd}
	\end{equation*}
	can be promoted to the adjoint RUTH $(C(G; A \oplus TM), \partial^{\mathrm{Ad}})$. On the other hand. we have the RUTH $(C(G; L_M), \partial^{L})$ on the core complex of $L$
	\[
	\begin{tikzcd}
		0 \arrow[r] & L_M \arrow[r] & 0
	\end{tikzcd}
	\]
	coming from the $G$-action on $L_M$. The $1$-form $\theta$ can be seen as a degree $0$ graded vector bundle morphism $A \oplus TM \to 0 \oplus L_M$ (Diagram \eqref{eq:theta_cm}) that we call $\Phi_0$.
	Then, setting $\Phi_k = 0$ for $k > 0$, we have that
	\[
	\Phi\colon \left(C(G; A \oplus TM), \partial^{\mathrm{Ad}} \right) \to \left(C(G; L_M), \partial^{L}\right)
	\]
	is a RUTH morphism. To show this we have to use \eqref{eq:struct_sect_RUTH_mor}. For $k=0$, we have that $\theta\circ \rho=0$ is satisfied because of $\partial\theta=0$. For $k=1$, let $v\in T_{s(g)}M$, with $g\in G$. Then, we have
	\begin{equation}\label{eq:theta_g_T}
		\theta (g_T.v)= \theta\big(dt\big(h_g(v)\big)\big)= t\left(t^\ast\theta\big(h_g(v)\big)\right)= t\left(s^\ast\theta\big(h_g(v)\big)\right)= t\left(s_g^{-1}\big(\theta(v)\big)\right)= g.\theta(v),
	\end{equation}
	where $g_T.$ is the quasi action of $G$ on $TM$ and $g.$ is the $G$-action on $L_M$. Then $\theta$ commutes with the $1$-st structure operators of the RUTH. For $k>1$, \eqref{eq:struct_sect_RUTH_mor} is trivially true.
	
	The mapping cone of $\Phi$ (see Example \ref{ex:mp_RUTH}) now completes Complex \eqref{eq:Morita0-kernel} to a RUTH on $A \oplus TM \oplus L_M$.
\end{rem}

\begin{definition}\label{def:MK}
	The RUTH on $A\oplus TM \oplus L_M$ obtained in Remark \ref{rem:MK_RUTH} will be called the \emph{Morita kernel} of $\theta$. The value of \eqref{eq:Morita0-kernel} at the point $x \in M$ is the \emph{Morita kernel at $x$}.
\end{definition}

Before discussing why the Morita kernel of a basic $1$-form $\theta$ serves as the appropriate replacement for the kernel of $\theta$, we make two straightforward remarks on this construction.
\begin{rem}\label{rem:0_inv_mc}
	Denote by $\{R_k\}_{k \geq 0}$ the structure operators of the Morita kernel. We stress for future reference that, from Remark \ref{rem:quasi_actions_quis}, it follows that, for every $g \colon x \to y$ in the groupoid $G$, the vertical arrows in
	\begin{equation}
		\label{eq:mor_ker_orbit}
		\begin{tikzcd}
			0 \arrow[r] & A_x \arrow[r, "\rho"] \arrow[d, "R_1(g)"] & T_xM \arrow[r, "\theta"] \arrow[d, "R_1(g)"] & L_{M,x} \arrow[r] \arrow[d, "R_1(g)"] & 0 \\
			0 \arrow[r] & A_y \arrow[r, "\rho"'] & T_yM \arrow[r, "\theta"'] & L_{M,y} \arrow[r] & 0
		\end{tikzcd}
	\end{equation}
	form a quasi-isomorphism (actually $R_1$ induces a $G$-action in the $R_0$-cohomology). In particular, the complexes at different points of the same orbit of $G$ are quasi-isomorphic. 
\end{rem}
\begin{rem}
	Notice that, unfortunately, the equivalence between RUTHs concentrated in non-negative degrees and higher VBGs (see Remark \ref{rem:higherVBGs}) does not apply to the Morita kernel. Indeed, the latter should be thought of as being concentrated in degrees $-1,0,1$. Only in this way it comes with an obvious morphism of RUTHs into the adjoint RUTH and can be duly seen as a \emph{homotopy kernel}. This informal discussion suggests that the Morita kernel is actually a \emph{higher derived stack} rather than a \emph{higher stack}.
\end{rem} 

In the next remark, we discuss the case when $\theta$ is a nowhere-zero basic form.
\begin{rem}\label{rem:Mor_ker_theta_neq_0}
	When $\theta_x \neq 0$ for all $x \in M$, then $K_\theta := \ker \theta \subseteq TM$ is a well-defined corank $1$ distribution on $M$. In this case, the Morita kernel can be replaced, up to quasi-isomorphisms, by a simpler RUTH. Specifically, the structure operators of the adjoint RUTH, restrict to the $2$-term complex of vector bundles
	\begin{equation}\label{eq:ker_RUTH}
		\begin{tikzcd}
			0 \arrow[r] & A \arrow[r, "\rho"] & K_{\theta} \arrow[r]&0
		\end{tikzcd}.
	\end{equation}
	We only need to check that the quasi-action of $G$ on $TM$ restricts to $K_\theta$. For any $v\in K_{\theta,s(g)}$, with $g\in G$, we have
	\[
	\theta(g_T.v)=\theta\left(dt\big(h_g(v)\big)\right)= t\left((t^\ast \theta)\big(h_g(v)\big)\right)= t\left(s^\ast\theta\big(h_g(v)\big)\right)= t\left(s_g^{-1}\big(\theta(v)\big)\right) = 0,
	\]
	and so $g_T.v\in K_{\theta,t(g)}$.
	
	Additionally, the cochain map
	\begin{equation}\label{eq:Mor_ker_theta_neq_0}
		\begin{tikzcd}
			0 \arrow[r] &A \arrow[r, "\rho"] \arrow[d, equal]& K_{\theta} \arrow[r] \arrow[d, swap, "\mathrm{in}"] & 0 \arrow[r] \arrow[d] & 0 \\
			0 \arrow[r] & A \arrow[r, "\rho"'] & TM \arrow[r, "\theta"'] & L_{M} \arrow[r] &0
		\end{tikzcd},
	\end{equation}
	where $\mathrm{in}\colon K_\theta \to TM$ denotes the inclusion, is a strict RUTH quasi-isomorphism, i.e., a strict RUTH morphism that is a quasi-isomorphism between the underlying complexes. Indeed, letting $\Phi_0$ be the cochain map \eqref{eq:Mor_ker_theta_neq_0} and $\Phi_k=0$ for $k\ge 0$, then all the equations \eqref{eq:struct_sect_RUTH_mor} are trivially satisfied.
\end{rem}

We conclude this subsection by proving the Morita invariance of the Morita kernel. This should also clarify the terminology. Notice that the pullback of a basic $1$-form along an LBG morphism is also a basic $1$-form, basic $1$-forms determine their corresponding Morita kernels, and if the LBG morphism is a VB-Morita map, then the Morita kernels are related by a strict RUTH morphism.
\begin{prop}\label{prop:mor_ker_mor_inv}
	Let $(F,f)\colon (L'\rightrightarrows L'_N; H\rightrightarrows N)\to (L\rightrightarrows L_M; G\rightrightarrows M)$ be a VB-Morita map between LBGs, and let $\theta\in \Omega^1(M;L_M)$ be an $L_M$-valued $1$-form such that $\partial\theta=0$. Then, for all $y \in N$, $(F, f)$ induces a quasi-isomorphism between the Morita kernel of $\theta' :=F^{\ast}\theta \in \Omega^1(N,L'_N)$ at the point $y\in N$ and the Morita kernel of $\theta$ at the point $f(y)\in M$.
\end{prop}
\begin{proof}
	Denote by $A_H, A_G$ the Lie algebroids of $H, G$ and let $\rho_H, \rho_G$ be their anchors.
	Now, fix $y \in N$, and consider the following diagram of cochain complexes and cochain maps:	
	\begin{equation}\label{eq:funct_mapping_cone}
		{\scriptsize
			\begin{tikzcd}
				0 \arrow[rr] & & A_{H,y} \arrow[rr, "\rho_H" near start] \arrow[dd] \arrow[dr, "df"] & & T_yN \arrow[rr] \arrow[dd, "\,\theta'" near end] \arrow[dr, "df"] &  & 0 \\
				& 0 \arrow[rr, crossing over] & & A_{G,f(y)} \arrow[rr, crossing over, "\rho_G" near start] & & T_{f(y)}M \arrow[rr, crossing over] \arrow[dd, "\,\theta" near end]&  & 0 \\
				0 \arrow[rr] & &0\arrow[rr] & & L'_{N, y} \arrow[rr] \arrow[dr, "F"]&  & 0 \\
				& 0 \arrow[rr] & & 0 \arrow[from=uu, crossing over] \arrow[rr] \arrow[ul, equal]& & L_{M, f(y)} \arrow[from=uu, crossing over] \arrow[rr]  &  & 0 
		\end{tikzcd}}.
	\end{equation}
	Since $df\colon TH\to TG$ is a VBG morphism, then it determines a RUTH morphism between the adjoint RUTHs of $H$ and $G$. Moreover, $F\colon L'\to L$ determines a morphism from the representation of $H$ on $L'$ and the representation of $G$ on $L$. Then, the cohain map between the mapping cones (of the back and front squares):
	\begin{equation}
		\label{eq:mc_quis}
		\begin{tikzcd}
			0 \arrow[r] & A_{H,y} \arrow[r, "\rho_H"] \arrow[d, "df"] & T_yN \arrow[r, "\theta'"] \arrow[d, "df"] & L'_{N,y} \arrow[r] \arrow[d, "F"] & 0 \\
			0 \arrow[r] & A_{G,f(y)} \arrow[r, "\rho_G"'] & T_{f(y)}M \arrow[r, "\theta"'] & L_{M,f(y)} \arrow[r] & 0
		\end{tikzcd}
	\end{equation}
	can be promoted to a RUTH morphism.
	Finally, as $(F, f)$ is a VB-Morita map, then $f$ is a Morita map and, by Example \ref{ex:df_VB_Morita}, $(df, f) \colon TH \to TG$ is a VB-Morita map between the tangent VBGs. It follows that the diagonal maps in the upper square of \eqref{eq:funct_mapping_cone} are a quasi-isomorphism. As $F$ is a VB-Morita map then, by Corollary \ref{coroll:VBMorita}, it is an LB morphism, and the diagonal maps of the lower square are also a quasi-isomorphism. Then, from standard Homological Algebra, the cochain map \eqref{eq:mc_quis} is a quasi-isomorphism as well.
\end{proof}

\subsection{Morita Curvature}\label{sec:mc_0}
Having discussed the kernel of a basic $L_M$-valued $1$-form $\theta$ (Section \ref{sec:mk_0}), we now focus on the curvature of $\theta$. We show that the role of the curvature is played by a RUTH morphism between the Morita kernel (Definition \ref{def:MK}) and its twisted dual RUTH (Example \ref{ex:twisted_dual_RUTH}). Additionally, we prove that when the $1$-form is nowhere-zero, the curvature simplifies. Finally, we discuss the Morita invariance of the curvature, ensuring its well-definedness as a structure on differentiable stacks.

Let $\theta \in \Omega^1 (M, L_M)$ be a basic $1$-form on the LBG $(L\rightrightarrows L_M;G\rightrightarrows M)$. Applying the $L_M$-twisted dual RUTH construction to the Morita kernel we get that the complex
\begin{equation}\label{eq:adjoint_mor_ker_0}
	\begin{tikzcd}
		0 \arrow[r] & \mathbbm{R}_M \arrow[r, "\theta^{\dag}"] & T^{\dag}M \arrow[r, "\rho^{\dag}"] & A^{\dag} \arrow[r] & 0
	\end{tikzcd},
\end{equation}
where $\mathbbm{R}_M:=M \times \mathbbm{R}$ is the trivial line bundle over $M$, can be completed to a RUTH of $G$. Notice that the quasi action determined on $\mathbbm{R}_M$ is just the identity. Indeed, if $r\in \mathbbm{R}$ and $\lambda \in L_{M,t(g)}$, with $g\in G$, then we have
\[
	R_1^\dagger(g)(r)(\lambda)=g. \left(r(R_1(g^{-1})(\lambda))\right)= g.(rg^{-1}. \lambda)= r\lambda.
\]

In order to define a RUTH morphism
\[
\Phi \colon C(G; A \oplus TM \oplus L_M) \to  C(G; \mathbbm R_M \oplus T^\dag M \oplus A^\dag)
\]
between the Morita kernel and its twisted dual RUTH, we consider the $L_M$-valued $2$-form $d^\nabla\theta$ on $M$, where $\nabla$ is a connection on $L_M$ and $d^\nabla$ is the connection differential discussed in Remark \ref{rem:definition_differential_nabla}. We start with the following

\begin{lemma}
	Let $\theta\in \Omega^1(M,L_M)$ be a $0$-shifted $L$-valued $1$-form and let $\nabla$ be a connection on $L_M$. Then
	\begin{equation}
		\label{eq:d_nabla_0}
		\partial d^\nabla \theta = \eta_\nabla \wedge s^\ast \theta= \eta_\nabla\wedge t^\ast \theta,
	\end{equation}
	where $\eta_\nabla=s^\ast \nabla -t^\ast \nabla\in \Omega^1(G)$ is the $1$-form discussed right before Lemma \ref{lemma:restriction_fnabla}.
\end{lemma}
\begin{proof}
	The proof is an easy computation:
	\begin{align*}
		s^\ast d^\nabla\theta -t^\ast d^\nabla\theta&= d^{s^\ast \nabla}s^\ast \theta- d^{t^\ast \nabla} t^\ast \theta \\
		&= d^{s^\ast \nabla}s^\ast \theta- d^{t^\ast \nabla} s^\ast \theta \\
		&= (s^\ast \nabla -t^\ast \nabla)\wedge s^\ast \theta\\
		&= \eta_\nabla\wedge s^\ast \theta.\qedhere
	\end{align*}
\end{proof}

Using Equation \eqref{eq:d_nabla_0}, we are now ready to define a cochain map between the underlying complexes of the Morita kernel and its twisted dual.
\begin{prop}
	\label{prop:mc_0-shifted}
	Let $\theta\in \Omega^1(M,L_M)$ be such that $\partial\theta=0$ and let $\nabla$ be a connection on $L_M$. The vertical arrows in the diagram
	\begin{equation}
		\label{eq:0moritakernel}
		\begin{tikzcd}
			0 \arrow[r] & A \arrow[r, "\rho"] \arrow[d, "-F_\nabla"] & TM \arrow[r, "\theta"] \arrow[d, "d^{\nabla}\theta"] & L_{M} \arrow[r] \arrow[d, "F_{\nabla}^{\dag}"] & 0 \\
			0 \arrow[r] & \mathbbm{R}_M \arrow[r, "\theta^{\dag}"'] & T^{\dag}M \arrow[r, "\rho^{\dag}"'] &A^{\dag} \arrow[r] & 0
		\end{tikzcd}
	\end{equation}
	form a cochain map (see Diagram \eqref{eq:fnabla} for the definition of $F_\nabla\colon A\to \mathbbm{R}_M$).
\end{prop}
\begin{proof}
	From the skew-symmetry of $d^{\nabla}\theta$ it is enough to prove that the first square in \eqref{eq:0moritakernel} commutes. To do this let $a\in A_x$ and $v\in T_xM$, with $x\in M$. Then, we have
	\begin{align*}
		d^\nabla\theta(\rho(a), v)= t\left(\big(t^\ast d^\nabla\theta\big) (a,v)\right)&= t\left(\big(s^\ast d^\nabla\theta\big)(a,v)- \big(\eta_\nabla\wedge s^\ast \theta\big) (a,v)\right)\\
		&= -\eta_\nabla(a)\theta(v)= -F_\nabla(a)\theta(v), 
	\end{align*}
	where we used Equation \eqref{eq:d_nabla_0} and Lemma \ref{lemma:restriction_fnabla}. 
\end{proof}
\begin{rem}
	Proposition \ref{prop:mc_0-shifted} can also be proved using Atiyah forms instead of Equation \eqref{eq:d_nabla_0}. Indeed, we can consider the Atiyah form $\omega \rightleftharpoons(\theta,0)\in \OA^2(L_M)$. As $\partial \theta = 0$, by Remark \ref{rem:mult_Atiyah_components}, we also have $\partial \omega = 0$. It follows that $\im\mathcal{D}\subseteq \ker\omega$. Now, applying Corollary \ref{coroll:kernelofAtiyahform}.\emph{ii)} to $\mathcal{D}_a$, with $a\in A$, we have that the first square in \eqref{eq:0moritakernel} commutes, and so the second square commutes as well.
\end{rem}

Next we prove that the cochain map \eqref{eq:0moritakernel}, that we denote by $\Phi_0$, is the $0$-th component of a RUTH morphism
\[
\Phi \colon C(G; A \oplus TM \oplus L_M) \to  C(G; \mathbbm R_M \oplus T^\dag M \oplus A^\dag)
\]
between the Morita kernel and its twisted dual RUTH. 

In order to do that we define the map
\[
\Phi_1 \colon G \to \operatorname{Hom}\left(s^\ast (A \oplus TM \oplus L_M), t^\ast (\mathbbm R_M \oplus T^\dag M \oplus A^\dag)\right)
\]
consisting of two summands $\Phi_1 \colon G \to \operatorname{Hom} (s^\ast TM, t^\ast \mathbbm R_M)$ and $\Phi_1 \colon G \to \operatorname{Hom}(s^\ast L_M, t^\ast T^\dag M)$, given by
\[
\Phi_1 (g) v := \nabla_{R_1^T (g) v}- R_1^D (g) \nabla_v  \in \operatorname{End} L_{M, t(g)} \cong \mathbbm R, \quad g \in G, \quad v \in T_{s(g)} M,
\]
where $R_1^D$ denotes the first structure operator of the Atiyah RUTH (see Example \ref{ex:Atiyah_RUTH}), and
\[
\big\langle \Phi_1 (g) \lambda, w \big\rangle := g. \big((\Phi_1 (g^{-1})w)  \lambda  \big) \in L_{M, t(g)}, \quad g \in G, \quad \lambda \in L_{M,s(g)}, \quad w \in T_{t(g)}M.
\]

%\begin{lemma}
%	For any $v\in T_{s(g)}M$, with $g\in G$, we have that
%	\begin{equation*}
%		\Phi_1(g)v= \eta_\nabla(h_gv),
%	\end{equation*}
%	where $h\colon s^\ast TM\to TG$ is the Ehresmann connection defining the adjoint RUTH.
%\end{lemma}
%\begin{proof}
%	Let $v\in T_{s(g)}M$, with $g\in G$. Then
%	\[
%	\eta_\nabla(h_gv)\lambda = (s^\ast \nabla- t^\ast \nabla)_{h_gv}\lambda= 
%	\]
%\end{proof}

Before proving that, in this way, we get a RUTH morphism, we need the following
\begin{lemma}
	Let $(L\rightrightarrows L_M; G\rightrightarrows M)$ be an LBG and let $\omega\in \OA^2(L_M)$ be a basic Atiyah $2$-form. Then, for any $\delta\in D_{s(g)}L_M$, with $g\in G$, we have
	\begin{equation}
		\label{eq:symbol_delta}
		f_\nabla (g_D. \delta) - g_D.f_\nabla (\delta ) = g_D.\nabla _{\sigma (\delta)}-\nabla_{g_T.\sigma (\delta)},
	\end{equation}
	and, for any $\delta\in D_{s(g)}L_M$ and $\delta'\in D_{t(g)}L_M$, we have
	\begin{equation}
		\label{eq:dnabla_action}
		\omega (g_D.\delta, \delta')= g. \omega(\delta, g^{-1}_D.\delta') .
	\end{equation}
\end{lemma}
\begin{proof}
	For any $\delta\in D_{s(g)} L_M$, with $g\in G$, we have
	\[
	\begin{aligned}
		f_\nabla (g_D. \delta) & = g_D.\delta - \nabla_{\sigma (g_D.\delta)} = g_D.(\delta - \nabla_{\sigma (\delta)}) + g_D.\nabla _{\sigma (\delta)}-\nabla_{g_T.\sigma (\delta)} \\
		& = g_D.f_\nabla(\delta)+ g_D.\nabla _{\sigma (\delta)}-\nabla_{g_T.\sigma (\delta)},
	\end{aligned}
	\]
	and Equation \eqref{eq:symbol_delta} easily follows. If $\delta'\in D_{t(g)}L_M$, then we have
	\begin{align*}
		\omega(g_D.\delta,\delta')&= t\left((t^\ast\omega)(h_g^D(\delta),h_{g^{-1}}^D(\delta')^{-1} )\right)\\
		& = t\left(s_g^{-1}\omega \big(\delta, Dt(h_{g^{-1}}^D(\delta'))\big)\right)\\
		& = g.\omega(\delta, g^{-1}_D.\delta'),
	\end{align*}
	where $h\colon s^\ast TM\to TG$ is an Ehresmann connection.
\end{proof}

\begin{prop}\label{prop:MR_RUTH_morph}
	The maps $\{\Phi_0, \Phi_1\}$ are the components of a RUTH morphism
	\[
	\Phi \colon C(G; A \oplus TM \oplus L_M) \to  C(G; \mathbbm R_M \oplus T^\dag M \oplus A^\dag),
	\]
	where all the other components are trivial.
\end{prop}
\begin{proof}
	We prove the RUTH morphism identities \eqref{eq:struct_sect_RUTH_mor}, where the $R_k$ are the structure operators of the Morita kernel, and the $R'_k = R^\dag_k$ are the structure operators of the $L_M$-twisted dual RUTH. For simplicity of notation, for any $g \in G$, we use again the notation $g_T. := R_1^T(g)$ and $g_D. := R_1^D (g)$.
	
	For $k=0$, Equation \eqref{eq:struct_sect_RUTH_mor} becomes $\Phi_0 \circ R_0 = R^{\dagger}_0 \circ \Phi_0$.
	This is true because, by Proposition \ref{prop:mc_0-shifted}, the $0$-th component of the Morita curvature is a cochain map.
	
	For $k=1$, Equation \eqref{eq:struct_sect_RUTH_mor} says that, for any $g\in G$, we have
	\begin{equation*}
		\Phi_0 \circ R_1(g) - R_1^{\dagger}(g) \circ \Phi_0 = R_0^{\dagger} \circ \Phi_1(g) + \Phi_1 (g) \circ R_0,
	\end{equation*}
	or, in other words, for any $a\in A_{s(g)}$,
	\begin{equation}\label{eq:whstf}
		-F_{\nabla}\big(g_T.a\big) + g_D.F_{\nabla}(a)= \nabla_{g_T.\rho(a)} - g_D.\nabla_{\rho(a)}\in \mathbbm{R}, 
	\end{equation}
	and, for any $v\in T_{s(g)}M$ and $v' \in T_{t(g)}M$,
	\begin{equation}
		\label{eq:hsfet}
		\begin{aligned}
			d^{\nabla}\theta(g_T. v, v') - g.d^{\nabla}\theta (v, g^{-1}_T. v') &=
			\big(\nabla_{g_T.v}-g_D. \nabla_v\big)\theta(v') \\
			&\quad + \big(\nabla_{g^{-1}_T.v'}- g_D^{-1}. \nabla_{v'}\big) g . \theta(v)\in L_{M,t(g)}.
		\end{aligned}
	\end{equation}
	
	Equation \eqref{eq:whstf} follows by applying Equation \eqref{eq:symbol_delta} to $\delta = \mathcal D_a$, with $a\in A_{s(g)}$, and noting that, since the quasi actions $g_D.$ determine a cochain map between the core complexes of $DL$ at $s(g)$ and $t(g)$, then $g_D.\mathcal{D}_a= \mathcal{D}_{g_T.a}$.
	
	In order to prove \eqref{eq:hsfet}, let $\omega\rightleftharpoons (\theta,0)\in \OA^2(L_M)$, and use Equation \eqref{eq:omegaandcomponents} to compute
	\begin{equation*}
		d^{\nabla}\theta(g_T.v, v')= \omega(g_D. \delta, \delta') - f_\nabla(g_D.\delta)\theta(v') + f_\nabla(\delta')\theta(g_T.v)\in L_{M,t(g)},
	\end{equation*}
	where $\delta\in D_{s(g)}L_M$ and $\delta'\in D_{t(g)}L_M$ are such that $v = \sigma (\delta), v' = \sigma (\delta')$. Similarly
	\begin{equation*}
		d^{\nabla}\theta(v, g_T^{-1}.v')= \omega(\delta, g^{-1}_D.\delta') - f_\nabla(\delta)\theta(g^{-1}_T.v') + f_\nabla(g^{-1}_D.\delta')\theta(v)\in L_{M,s(g)},
	\end{equation*}
	where we used that the symbol intertwines the structure operators $R^D$ with the structure operators $R^T$ (see Example \ref{ex:Atiyah_RUTH}) to see that $\sigma(g^{-1}_D.\delta')= g^{-1}_T.v'$.
	But, from $\partial \theta=0$, so $\partial \omega=0$, it easily follows that  $\theta(g_T.v)=g.\theta(v)$ (see \eqref{eq:theta_g_T}), and $\omega(g_D. \delta, \delta') = g.\omega(\delta, g^{-1}_D.\delta')$ (see \eqref{eq:dnabla_action}), whence
	\begin{equation}
		\label{eq:central_commutes}
		\begin{aligned}
			d^{\nabla}\theta(g_T. v, v') - g.d^{\nabla}\theta (v, g^{-1}_T. v') &= \Big(g_D.f_\nabla(\delta)- f_\nabla(g_D.\delta)\Big)\theta(v') \\
			&\quad  + \Big(g^{-1}_D.f_\nabla(\delta')-f_\nabla(g^{-1}_D.\delta')	\Big)g. \theta(v).
		\end{aligned}
	\end{equation}
	Using \eqref{eq:symbol_delta} again, we get \eqref{eq:hsfet}.
	
	For $k=2$, Equation \eqref{eq:struct_sect_RUTH_mor} says that, for any $(g, g') \in G^{(2)}$,	\begin{equation*}
		\Phi_0\circ R_2(g, g') - \Phi_1(g)\circ R_1(g') = R_1^{\dagger}(g)\circ \Phi_1(g') + R_2^{\dagger}(g,g')\circ \Phi_0 - \Phi_1(gg').
	\end{equation*}
	This means that, for any $v\in T_{s(g')}M$,
	\begin{equation}
		\label{eq:7.2}
		\begin{aligned}
			&F_{\nabla}\big(R_2^T(g,g')v\big)+ \big(\nabla_{g_{T}.(g'_{T}.v)}-g_{D}. \nabla_{g'_{T}. v}\big)\\
			&\quad + g_D.\big(\nabla_{g'_{T}.v}- g'_{D}.\nabla_v\big) - \big(\nabla_{(gg')_T.v}- (gg')_D. \nabla_v\big) = 0,
		\end{aligned}
	\end{equation}
	where we used that $R_2$ is zero on $L_M$, and that, for any $\lambda\in L_{M,s(g')}$ and $v\in T_{t(g)} M$,
	\begin{equation}\label{eq:2536dh}
		\begin{aligned}
			&(gg')_D.\Big(g'{}^{-1}_D.\big(\nabla_{g^{-1}_{T}.v} - g^{-1}_{D}. \nabla_v\big) + F_{\nabla}\big(R_2^T(g'{}^{-1}, g^{-1})v\big) \\ &+ \big(\nabla_{g'{}^{-1}_{T}. (g_{T}^{-1}.v)} - g'{}^{-1}_{D}. \nabla_{g^{-1}_{T}.v}\big) - \big(\nabla_{(gg')^{-1}_T. v} - (gg')^{-1}_D .\nabla_v\big) \Big) gg'. \lambda=0.
		\end{aligned}
	\end{equation}
	But, from Equation \eqref{eq:struct_sect_RUTH} for $k=2$, we get that the structure operators of the adjoint RUTH and Atiyah RUTH satisfy the following equations:
	\begin{equation*}
		g_{T}.(g'_{T}.v) - (gg')_T.v = \rho\big(R_2^T(g,g')v\big), \quad v \in T_{s(g')} M,
	\end{equation*}
	and
	\begin{equation*}
		g_{D}.(g'_{D}.\delta) - (gg')_D.\delta= \mathcal D_{R_2^D(g,g')\delta} = \mathcal D_{R_2^T(g,g')\sigma(\delta)}, \quad \delta \in D_{s(g')} L_M,
	\end{equation*}
	where we also used \eqref{eq:gst36}. It follows that the left hand side of \eqref{eq:7.2} is
	\begin{equation*}
		\begin{aligned}
			& F_{\nabla}\big(R_2^T(g,g')v\big)+ \big(\nabla_{g_{T}.(g'_{T}.v)}-g_{D}. \nabla_{g'_{T}. v}\big)\\
			&\quad \quad + g_D.\big(\nabla_{g'_{T}.v}- g'_{D}.\nabla_v\big) - \big(\nabla_{gg'_T.v}- gg'_D. \nabla_v\big) \\
			&\quad  = \big(\mathcal D_{R_2^T(g,g')v} - \nabla_{\rho(R_2^T(g,g')v)}\big) + \big(\nabla_{g_{T}.(g'_{T}.v)}-g_{D}. \nabla_{g'_{T}. v}\big) \\
			&\quad \quad +g_D. \big(\nabla_{g'_{T}.v}- g'_{D}.\nabla_v\big) - \big(\nabla_{(gg')_T.v}- (gg')_D. \nabla_v\big)\\
			&\quad = g_{D}.(g'_D. \nabla_v)- g_D.\nabla_{g_T'.v} - g_D. \big( g'_{D}.\nabla_v - \nabla_{g'_{T}.v}\big) \\
			& \quad =0.
		\end{aligned}
	\end{equation*}
	Clearly, \eqref{eq:2536dh} follows from \eqref{eq:7.2}.
	
	For $k=3$, we have to prove that
	\begin{equation}
		\label{eq:bho}
		\begin{aligned}
			-R_2(g_1g_2,g_3) + R_2(g_1,g_2g_3)&= R_0\circ R_3(g_1,g_2,g_3) - R_1(g_1)\circ R_2(g_2,g_3) \\
			& \quad+ R_2(g_1,g_2)\circ R_1(g_3) - R_3(g_1,g_2,g_3)\circ R_0,
		\end{aligned}
	\end{equation}
	for all $(g_1,g_2,g_3)\in G^{(3)}$. But, again using that $L_M$ is a plain representation (no higher homotopies), we have that \eqref{eq:bho} is trivially satisfied. As, for degree reasons, there are no higher identities to check. 
	Finally, it is clear from the definition of $\Phi_1$ that $\Phi_1(x)=0$ for all $x\in M$, this concludes the proof.
\end{proof}

%\begin{rem}
%	Equation \eqref{eq:hsfet} can also be proved using Equation \eqref{eq:d_nabla_0} instead of Atiyah forms. Indeed, for any $v\in T_{s(g)}M$ and $v'\in T_{t(g)}M$, with $g\in G$, we have
%	\begin{align*}
%		d^\nabla \theta (g_T.v,v')&= d^\nabla\theta(dt(h_gv), dt((h_{g^{-1}}v')^{-1})) \\
%		&=t\left((t^\ast d^\nabla \theta) (h_g v, (h_{g^{-1}}v')^{-1})\right)\\
%		&=t\left((s^\ast d^\nabla \theta) (h_g v, (h_{g^{-1}}v')^{-1}) - (\eta_\nabla\wedge s^\ast\theta)(h_gv, (h_{g^{-1}}v')^{-1})\right)\\
%		&=t\left(s_g^{-1}(d^\nabla\theta(v,g^{-1}_T.v')) - \eta_\nabla(h_gv)s_g^{-1}(\theta(g^{-1}_T.v')) + \eta_\nabla((h_{g^{-1}}v')^{-1})s_g^{-1}(\theta(v))\right)\\
%		&=g.d^\nabla\theta(v,g^{-1}_T.v')-  \eta_\nabla(h_gv) g.\theta(g^{-1}_T.v') +\eta_\nabla((h_{g^{-1}}v')^{-1})g.\theta(v)\\
%		&=g.d^\nabla\theta(v,g^{-1}_T.v')- \left(g_D.\nabla_{v}-\nabla_{g_T.v} \right)\theta(v') +\left(\nabla_{g^{-1}_T.v'}-g^{-1}_D.\nabla_{v'}\right) g.\theta(v)\\
%		&=g.d^\nabla\theta(v,g^{-1}_T.v')+ \left(\nabla_{g_T.v} -g_D.\nabla_{v} \right)\theta(v') +\left(\nabla_{g^{-1}_T.v'}-g^{-1}_D.\nabla_{v'}\right) g.\theta(v).
%	\end{align*}
%	On the other hand we can prove Equation \eqref{eq:whstf} simply using Lemma \ref{lemma:restriction_fnabla}. Indeed, for any $a\in A_{s(g)}$, with $g\in G$, we have
%	\[
%		F_\nabla(g_T.a)= \eta_\nabla(g_T.a) = (s^\ast\nabla-t^\ast\nabla)_{g_T.a}= -\nabla_{g_T.\rho(a)},
%	\]
%	and
%	\[
%		F_\nabla(a)=\eta_\nabla(a)= (s^\ast\nabla-t^\ast\nabla)_a= -\nabla_{\rho(a)}.
%	\]
%\end{rem}

\begin{definition}\label{def:Mor_curv_0-shift}
	The RUTH morphism $\Phi$ in Proposition \ref{prop:MR_RUTH_morph} is called the \emph{Morita curvature of} $\theta$. The value of the cochain map \eqref{eq:0moritakernel} at the point $x\in M$ is the \emph{Morita curvature at $x$}.
\end{definition}

Before discussing why the Morita curvature of a basic $1$-form serves as the appropriate replacement for the curvature of the $1$-form, we make the following
\begin{rem}\label{rem:MC_RUTH}
	We stress here for future reference that, for every arrow $g \colon x \to y$ in the groupoid $G$, the diagram
	\begin{equation}\label{eq:mor_curv_orbit}
		{\scriptsize
			\begin{tikzcd}
				0 \arrow[rr] & & A_x \arrow[rr] \arrow[dd, pos=0.2, swap, "-F_{\nabla}"] \arrow[dr, "R_1(g)"] & & T_xM \arrow[rr] \arrow[dd, pos=0.2, swap, "d^{\nabla}\theta"] \arrow[dr, "R_1(g)"] & & L_{M,x} \arrow[rr] \arrow[dd, pos=0.2, swap, "F_{\nabla}^\dagger"] \arrow[dr, "R_1(g)"] & & 0 \\
				& 0 \arrow[rr, crossing over] & & A_y \arrow[rr, crossing over] & & T_yM \arrow[rr, crossing over] & & L_{M,y} \arrow[rr] & & 0 \\
				0 \arrow[rr] & & \mathbbm{R} \arrow[rr] & & T_x^{\dag}M \arrow[rr] \arrow[dr, "R_1^{\dag}(g)"] & & A_x^{\dag} \arrow[rr] \arrow[dr, "R_1^{\dag}(g)"] & & 0 \\
				& 0 \arrow[rr] & & \mathbbm{R} \arrow[from=uu, crossing over, pos=0.2, "-F_{\nabla}"] \arrow[rr] \arrow[ul,equal] & & T_y^{\dag}M \arrow[from=uu, crossing over, pos=0.2, "d^{\nabla}\theta"] \arrow[rr] & & A^{\dag}_y \arrow[from=uu, crossing over, pos=0.2, "F_{\nabla}^\dagger"] \arrow[rr] & & 0 
		\end{tikzcd}},
	\end{equation}
	although not being commutative, does actually define a commutative diagram in the cohomology of the Morita kernel (and its twisted dual RUTH). Indeed, for any $a\in A_x$, it follows from Equation \eqref{eq:symbol_delta}, that
	\[
		F_\nabla(g_T.a) = F_\nabla(a)+ g_D.\nabla_{\rho(a)} -\nabla_{g_T.\rho(a)}\in \mathbbm{R}, 
	\]
	where we used that the action of $g$ on endomorphisms is the identity (see Example \ref{ex:Atiyah_RUTH}). Then, when $a\in \ker \rho$, then $g_T.a$ is in $\ker \rho$ as well and the leftmost square commutes. The rightmost square is just the twisted dual of the leftmost one, and so it also commutes in cohomology. Finally, the central square commutes in cohomology because of Equation \eqref{eq:central_commutes}.
	
	 As the diagonal arrows are quasi-isomorphisms (Remark \ref{rem:0_inv_mc}), we conclude that the Morita curvatures at different points of the same orbit of $G$ are related by isomorphisms in cohomology. 
\end{rem}

In the next remark we discuss the case when $\theta$ is nowhere-zero.
\begin{rem}
	If $\theta_x\neq 0$ for all $x\in M$, then, as discussed in Remark \ref{rem:Mor_ker_theta_neq_0}, the Morita kernel can be replaced by~\eqref{eq:ker_RUTH}, with its RUTH structure, up to quasi-isomorphisms. Accordingly the Morita curvature simplifies in this case. Namely, the standard curvature can also be seen as a (necessarily strict) RUTH morphism:
	\begin{equation}\label{eq:Mor_curv_theta_neq_0}
		\begin{tikzcd}
			0 \arrow[r] & A \arrow[r, "\rho"] \arrow[d] & K_{\theta} \arrow[r] \arrow[d, "R_{\theta}"'] & 0 \arrow[r] \arrow[d] & 0 \\
			0 \arrow[r] & 0 \arrow[r] & K_{\theta}^{\dag} \arrow[r, "\rho^{\dag}"'] & A^{\dag} \arrow[r] & 0
		\end{tikzcd}.
	\end{equation}
	
	To show that $R_{\theta}$ is a RUTH morphism, denote by $\{R^T_k\}_{k \geq 0}$ the structure operators of the adjoint RUTH and by $\{R^{T^\dag}_k\}_{k \geq 0}$ the structure operators of the $L_M$-twisted dual RUTH. By degree reasons, it is enough to prove that, for any $g \colon x\to y$ in $G$,
	\[
	R_\theta \circ R^T_1(g) = R^{T^\dag}_1 (g) \circ R_\theta \colon K_{\theta, x} \to K^\dag_{\theta, y}. 
	\]
	This boils down to the following identity
	\begin{equation}\label{eq:R_RUTH}
		d^\nabla \theta \big(v, dt (h_{g^{-1}}(v'))\big) = g^{-1}.d^\nabla \theta \big(dt(h_{g}(v)), v'\big)\in L_{M,x},
	\end{equation}
	for all $v \in K_{\theta, x}$ and all $v' \in K_{\theta, y}$, where $h$ is the Ehresmann connection defining the adjoint RUTH. But Equation \eqref{eq:R_RUTH} agrees with %, let $\omega \rightleftharpoons (\theta, 0) \in \OA^2 (L_M)$, let $\delta \in D_x L_M, \delta' \in D_y L_M$ be such that $\sigma (\delta) = v, \sigma (\delta') = v'$, and let $h^D \colon s^\ast DL_M \to DL$ be as in Example \ref{ex:Atiyah_RUTH} \eqref{eq:SES_Appendix}. Now, using Equation \eqref{eq:omegaandcomponents}, and $\partial \theta = 0$ (hence $\partial \omega = 0$) we find: \textcolor{red}{provala prima e la richiami. è la stessa di prima}
	%\begin{align*}
	%	d^\nabla \theta \big(v, dt (h_{g^{-1}}(v'))\big) &= \omega \big(\delta, Dt(h^D_{g^{-1}}(\delta'))\big)=\omega\big(Dt(\tilde \delta),  Dt(h^D_{g^{-1}}(\delta'))\big) \\
	%	&=t\big((t^\ast \omega)(\tilde \delta,h^D_{g^{-1}}(\delta'))\big) =t\big((s^\ast \omega)(\tilde \delta,h^D_{g^{-1}}(\delta'))\big) \\
	%	&=g^{-1}.\omega\big(Ds(\tilde \delta),Ds(h^D_{g^{-1}}(\delta'))\big) = g^{-1}.\omega\big(Dt(h^D_g(\delta)),\delta'\big) \\
	%	&=g^{-1}. d^\nabla \theta \big(dt(h_{g}(v)), v'\big)
	%\end{align*}
	%where $\tilde\delta= Di(h^D_g(\delta)) \in D_{g^{-1}}L$ and we used that $u,v\in K_{\theta}$.
	Equation \eqref{eq:central_commutes} in the case when $v$ and $v'$ are in the kernel of $\theta$.
	
	The RUTH morphism $R_\theta$ is related to the Morita curvature by the following diagram:
	\begin{equation}\label{diag:mor_curv}
		{\scriptsize
			\begin{tikzcd}
				0 \arrow[rr] & & A_x \arrow[rr] \arrow[dd] \arrow[dr, shift left=0.3ex, dash] \arrow[dr, shift right=0.3ex, dash] & & K_{\theta,x} \arrow[rr] \arrow[dd, pos=0.2, swap, "R_{\theta}"] \arrow[dr, "\mathrm{in}"] & & 0 \arrow[rr] \arrow[dd] \arrow[dr] & & 0 \\
				& 0 \arrow[rr, crossing over] & & A_x \arrow[rr, crossing over] & & T_xM \arrow[rr, crossing over] & & L_{M,x} \arrow[rr] & & 0 \\
				0 \arrow[rr] & & 0 \arrow[rr] & & K_{\theta,x}^{\dagger} \arrow[rr] & & A_x^{\dagger} \arrow[rr] & & 0 \\
				& 0 \arrow[rr] & & \mathbbm{R} \arrow[from=uu, crossing over, pos=0.2, "-F_\nabla"] \arrow[rr] \arrow[ul] & & T_x^{\dagger}M \arrow[from=uu, crossing over, pos=.2, "d^\nabla \theta"] \arrow[rr] \arrow[ul, "\mathrm{in}^\dag"] & & A^{\dagger}_x \arrow[from=uu, crossing over, pos=0.2, "F_{\nabla}^\dag"] \arrow[rr] \arrow[ul, shift left=0.3ex, dash] \arrow[ul, shift right=0.3ex, dash] & & 0 
		\end{tikzcd}}.
	\end{equation}
	Diagram \eqref{diag:mor_curv} commutes. Indeed, for any $v,v'\in K_{\theta,x}$, with $x\in M$, we have $$d^\nabla\theta(v,v')= \theta([V,V']_x)=R_{\theta}(v)(v'),$$ where $V,V'\in \Gamma(K_\theta)$ are such that $V_x=v$ and $V'_x=v'$. By Remark \ref{rem:Mor_ker_theta_neq_0}, in Diagram \eqref{diag:mor_curv} the diagonal arrows on the top, hence on the bottom as well, form a quasi-isomorphism, showing that the Morita curvature of $\theta$ is a quasi-isomorphism if and only if the vertical arrows in \eqref{eq:Mor_curv_theta_neq_0} form a quasi-isomorphism, i.e.~$\rho$ is injective, hence $G$ is a foliation groupoid (see Remark \ref{rem:foliation_groupoid}), and $\im \rho = \ker R_\theta$.
	%Notice however, that there are $0$-shifted contact structures $\theta$ for which $\theta_x = 0$ for some $x$ (but not all $x$, see Example \ref{exmpl:0-shifted_nonreg} below).
\end{rem}

We conclude this subsection by proving that the Morita curvature of a basic $1$-form on an LBG is Morita invariant. To this end, we first show that it is independent of the choice of the connection $\nabla$, up to homotopies between the $0$-th components. More precisely, we have the following
\begin{prop}\label{prop:0_MC_conn}
	Let $\nabla$ and $\nabla'$ be two connections on $L_M$. Then there is a homotopy between the $0$-th components of the corresponding Morita curvatures.
\end{prop}
\begin{proof}
	The difference $h = \alpha_{\nabla, \nabla'} := \nabla - \nabla'$ can be seen as a $1$-form on $M$, so $h \in \Omega^1(M)$. The diagonal arrows in the diagram
	\begin{equation*}
		{\footnotesize
			\begin{tikzcd}
				0 \arrow[rr] & & A \arrow[rr, "\rho"] \arrow[dd, "-F_\nabla", shift left=0.5ex] \arrow[dd, "-F_{\nabla'}"', shift right=0.5ex] & &  TM \arrow[ddll, "h"] \arrow[rr, "\theta"] \arrow[dd, "d^{\nabla}\theta", shift left=0.5ex] \arrow[dd, "d^{\nabla'}\theta"', shift right=0.5ex] & & L_{M} \arrow[ddll, "-h^\dag"] \arrow[r] \arrow[dd, "F^\dag_\nabla", shift left=0.5ex] \arrow[dd, "F^\dag_{\nabla'}"', shift right=0.5ex] & 0 \\ \\
				0 \arrow[rr] & &  \mathbbm{R}_M \arrow[rr, "\theta^{\dag}"'] & &  T^{\dag}M \arrow[rr, "\rho^{\dag}"'] & & A^{\dag} \arrow[r] & 0
		\end{tikzcd}}
	\end{equation*}
	form a homotopy between the $0$-th components of the Morita curvatures. Indeed,  in degree $-1$, for every $x \in M$, take $a\in A_x$ and $\lambda \in \Gamma (L_M)$. Then
	\begin{align*}
		(F_{\nabla'}-F_\nabla)(a)\lambda_x&= 
		(\mathcal{D}_a -\nabla'_{\rho(a)} -\mathcal{D}_a + \nabla_{\rho(a)}) 
		\lambda \\
		&= (\nabla_{\rho(a)}-\nabla'_{\rho(a)})\lambda \\
		&= h(\rho(a))\lambda_x,
	\end{align*}
	i.e. $F_{\nabla'}-F_\nabla=h\circ \rho$. In degree $0$, take $v,v'\in T_xM$, and let $V, V'\in \mathfrak{X}(M)$ be such that $V_x=v$ and $V'_x=v'$. Then
	\begin{align*}
		(d^{\nabla}\theta - d^{\nabla'}\theta)(v,v')&= \nabla_v\theta(V') - \nabla_{v'}\theta(V)- \theta([V,V']_x)\\
		&\quad -\nabla'_v\theta(V') + \nabla'_{v'}\theta(V)+ \theta([V,V']_x) \\
		&= (\nabla_v - \nabla'_v)\theta(v') - (\nabla_{v'} -\nabla'_{v'})\theta(v)\\
		&=h(v)\theta(v')- h(v') \theta(v)\\
		&= \big(\theta^{\dag}(h(v)) - h^{\dag}(\theta(v))\big) (v')\\
		&=\big(\theta^\dagger\circ h -h^\dagger \circ \theta\big)(v,v'),
	\end{align*}
	i.e. $d^\nabla \theta - d^{\nabla'}\theta = \theta^\dag \circ h - h^\dag \circ \theta$. 
	
	For degree $+1$, just notice that the last triangle is the opposite of the twisted transpose of the first one. %Indeed, for any $\lambda \in L_{M,x}$ and $a\in A_x$ we have
%	\begin{align*}
%		(F_{\nabla}^\dagger - F_{\nabla'}^\dagger)(\lambda)(a)= 
%	\end{align*}
\end{proof}

Not only the Morita curvature is independent of the connection (up to homotopies of the $0$-th components), it is actually Morita invariant in an appropriate sense. In order to show this let $(F,f)\colon (L'\rightrightarrows L'_N; H\rightrightarrows N) \to (L\rightrightarrows L_M; G\rightrightarrows M)$ be a VB-Morita map between LBGs and let $\theta\in\Omega^1(M,L_M)$ be such that $\partial\theta=0$. Moreover, let $\nabla$ be a connection on $L_M$. Denote $\theta' = F^\ast \theta \in \Omega^1 (N, L'_N)$. Finally, denote by $\nabla' = f^\ast \nabla$ the pull-back connection. Then we have the following

\begin{prop}\label{prop:mor_curv_mor_inv}
	Let $(F,f), \theta', \theta, \nabla', \nabla$ be as above.
	Then the Morita curvatures of $\theta', \theta$ (with respect to $\nabla', \nabla$) at the points $y\in N$ and $f(y)\in M$ are related by quasi-isomorphisms.
\end{prop}
\begin{proof}
	We consider the diagram
	\begin{equation}\label{diag:mor_curv_mor_inv}
		\resizebox{\textwidth}{!}{
			\begin{tikzcd}[ampersand replacement=\&]
				0 \arrow[rr] \& \& A_{H,y} \arrow[rr] \arrow[dd] \arrow[dr, "df"] \& \& T_yN \arrow[rr] \arrow[dd] \arrow[dr, "df"] \& \& L'_{N,y} \arrow[rr] \arrow[dd] \arrow[dr, "F"] \& \& 0 \\
				\& 0 \arrow[rr, crossing over] \& \& A_{G,f(y)} \arrow[rr, crossing over] \& \& T_{f(y)}M \arrow[rr, crossing over] \& \& L_{M,f(y)} \arrow[rr] \& \& 0 \\
				0 \arrow[rr] \& \& \mathbbm{R} \arrow[rr] \& \& T_y^{\dag}N \arrow[rr] \& \& A_{H,y}^{\dag} \arrow[rr] \& \& 0 \\
				\& 0 \arrow[rr] \& \& \mathbbm{R} \arrow[from=uu, crossing over] \arrow[rr] \arrow[ul,equal] \& \& T_{f(y)}^{\dag}M \arrow[from=uu, crossing over] \arrow[rr] \arrow[ul, "df^{\dag}"] \& \& A^{\dag}_{G,f(y)} \arrow[from=uu, crossing over] \arrow[rr] \arrow[ul, "df^{\dag}"] \& \& 0 
		\end{tikzcd}},
	\end{equation}
	where the front vertical arrows are the Morita curvature of $\theta$ determined by the connection $\nabla$ at the point $f(y)$, and the back vertical arrows are the Morita curvature of $\theta'$ determined by $\nabla'$ at the point $y$. We know already from Proposition \ref{prop:mor_ker_mor_inv} that the diagonal arrows on the top, hence on the bottom as well, of the diagram form quasi-isomorphisms.
	Finally, diagram \eqref{diag:mor_curv_mor_inv} commutes. Indeed, for any $\delta\in D_yL'_N$ and $\Lambda\in \Gamma(L_M)$ we set $\sigma(\delta)=v\in T_yN$ and $\Lambda_{f(y)}=\lambda\in L_{M,f(y)}$, so that
	\begin{align*}
		f_{\nabla}\left(DF(\delta)\right)\lambda &= \left(DF(\delta) - \nabla_{df(v)}\right)\Lambda 
		= F\left(\delta (F^{\ast}\Lambda)- \nabla'_{v} (F^{\ast}\Lambda)\right)\\
		& = F\left(f_{\nabla'}(\delta)F_y^{-1}(\lambda)\right) = f_{\nabla'}(\delta)\lambda,
	\end{align*}
	then $f_{\nabla} \circ DF=f_{\nabla'}$. In particular, when $\delta=\mathcal{D}_a$, with $a\in A_{H,y}$, we have that 
	\[
		F_{\nabla}(df(a))= f_\nabla(\mathcal{D}_{df(a)})= f_\nabla(DF(\mathcal{D}_a))= f_{\nabla'}(\mathcal{D}_a)= F_{\nabla'}(a),
	\]
	and the leftmost and rightmost vertical squares in Diagram \eqref{diag:mor_curv_mor_inv} commute. Moreover, obviously $d^{\nabla'}\theta'= df^{\dagger}\circ d^{\nabla}\theta \circ df$ and the central vertical square in Diagram \eqref{diag:mor_curv_mor_inv} commutes as well.
\end{proof}

\subsection{Definition and Examples}\label{sec:def_0}
In this final subsection, we introduce the ultimate definition of $0$-shifted contact structure, prove Morita invariance, and establish equivalence with $0$-shifted symplectic Atiyah structures. Lastly, we provide some examples.

Let $(L \rightrightarrows L_M; G \rightrightarrows M)$ be an LBG.
\begin{definition}\label{def:0-shift_cont}
	A \emph{$0$-shifted contact structure} on $L$ is an shifted $1$-form $\theta\in \Omega^1(M,L_M)$ such that the Morita curvature at $x$ is a quasi-isomorphism for all points $x\in M$.
\end{definition}

The notion of $0$-shifted contact structure is Morita invariant in the sense of the following
\begin{theo}\label{theo:Mor_inv_0-shift}
	Let $(F,f)\colon (L'\rightrightarrows L'_N; H\rightrightarrows N) \to (L\rightrightarrows L_M; G\rightrightarrows M)$ be a VB-Morita map of LBGs. Then the assignment $\theta \mapsto F^\ast \theta$ establishes a bijection between $0$-shifted contact structures on $L$ and $0$-shifted contact structures on $L'$.
\end{theo}
\begin{proof}
	According to Remark \ref{rem:vv_forms_Morita_equiv} the assignment $\theta \mapsto F^\ast \theta$ establishes a bijection between $\partial$-closed $L_M$-valued $1$-forms on $M$ and $\partial$-closed $L'_N$-valued $1$-forms on $N$. It remains to prove that $\theta$ is a $0$-shifted contact structure if and only if so is $F^\ast \theta$. First, by Proposition \ref{prop:0_MC_conn}, the Morita curvatures do not depend on the choice of the connections up to homotopies. Then, if the Morita curvature of $\theta$ on $L_M$ is a quasi-isomorphism at all points in $M$, then, by Proposition \ref{prop:mor_curv_mor_inv}, the Morita curvature of $\theta'$ is a quasi-isomorphism at all points in $N$. Vice-versa, if the Morita curvature of $\theta'$ is a quasi-isomorphism at all points in $N$, then, by Proposition \ref{prop:mor_curv_mor_inv}, the Morita curvature of $\theta$ is a quasi-isomorphism at the points $f(y)$, with $y\in N$. Finally, since $f$ is essentially surjective, from Remark \ref{rem:MC_RUTH} (see diagram \eqref{eq:mor_curv_orbit}), it follows that the Morita curvature of $\theta$ is a quasi-isomorphism at all points in $M$.
\end{proof}

\begin{rem}
	For the last part of the proof of Theorem \ref{theo:Mor_inv_0-shift}, one can avoid using diagram \eqref{eq:mor_curv_orbit} and RUTHs, by working only with Morita maps which are surjective and submersive on objects (see Remark \ref{rem:Morita_maps} on the relation between Morita maps and Morita maps which are surjective submersions on bases). 
\end{rem}

Theorem \ref{theo:Mor_inv_0-shift} motivates the following
\begin{definition}
	A \emph{$0$-shifted contact structure} on an LB-stack $[L_M/L]\to [M/G]$ is a $0$-shifted contact structure on an LBG $L$ presenting the LB-stack $[L_M/L]$.
\end{definition} 

The last result of this section establishes a bijection between $0$-shifted contact structures and $0$-shifted symplectic Atiyah forms (see Definition \ref{def:0-shif_Atiyah}) which, in our opinion, represents a strong motivation for Definition \ref{def:0-shift_cont}.
\begin{theo}\label{theor:0-contact=0-Atiyah}
	Let $(L\rightrightarrows L_M; G\rightrightarrows M)$ be an LBG. The assignment $\theta\mapsto \omega \rightleftharpoons (\theta,0)$ establishes a bijection between $0$-shifted contact structures and $0$-shifted symplectic Atiyah forms on $L$.
\end{theo}
\begin{proof}
	Choose once for all a connection $\nabla$ on $L_M$. By Remark \ref{rem:mult_Atiyah_components} we already know that $\theta$ is $\partial$-closed if and only if $\omega$ is so. It remains to prove that the Morita curvature of $\theta$ is a quasi-isomorphism at the point $x\in M$ if and only if the cochain map \eqref{eq:0_shifted_Atiyah_non_deg} is a quasi-isomorphism. Notably, the mapping cones of the Morita curvature and that of \eqref{eq:0_shifted_Atiyah_non_deg} do actually agree. Namely, the mapping cone of \eqref{eq:0_shifted_Atiyah_non_deg} is
	\begin{equation}
		\begin{tikzcd}\label{eq:map_con_omega}
			0 \arrow[r] & A_x \arrow[r, "-\mathcal{D}"] & D_xL_M \arrow[r, "\omega"] & J^1_xL_M \arrow[r, "\mathcal{D}^{\dag}"] & A_x^{\dag} \arrow[r] & 0
		\end{tikzcd}.
	\end{equation}
	Under the direct sum decomposition $DL_M \cong TM\oplus\mathbbm{R}$, $\delta\mapsto (\sigma(\delta), f_{\nabla}(\delta))$, the map $\mathcal{D}\colon A \to DL_M$ becomes
	\begin{equation*}
		(\rho, F_\nabla)\colon A \to TM\oplus \mathbbm{R}, \quad a\mapsto \big(\rho(a), F_\nabla(a)\big),
	\end{equation*}
	and $\omega \colon DL_M \to J^1 L_M $ becomes
	\begin{equation*}
		\begin{pmatrix}
			d^{\nabla}\theta & \theta^{\dag} \\
			-\theta &0
		\end{pmatrix}
		\colon TM\oplus \mathbbm{R} \to T^{\dag}M \oplus L_{M}.
	\end{equation*}
	Indeed, using Equation \eqref{eq:omegaandcomponents}, for any $v,v'\in T_xM$ and $r,r'\in \mathbbm{R}$, with $x\in M$, we have
	\begin{align*}
		\omega(\nabla_v+r\mathbbm{I}_x,\nabla_{v'}+ r'\mathbbm{I}_x)& = d^\nabla\theta(v,v') + r\theta(v') - r'\theta(v) \\
		&=\begin{pmatrix}
			\iota_v d^\nabla\theta + r\theta^\dagger, & -\theta(v)
		\end{pmatrix} 
		\begin{pmatrix}
			v' \\ r'
		\end{pmatrix} \\
		&= \begin{pmatrix}
			v, & r
		\end{pmatrix}
		\begin{pmatrix}
			d^{\nabla}\theta & \theta^{\dag} \\
			-\theta &0
		\end{pmatrix}
		\begin{pmatrix}
			v' \\ r'
		\end{pmatrix}.
	\end{align*}
	Hence, the mapping cone \eqref{eq:map_con_omega} becomes
	\begin{equation}\label{eq:mc_mc}
		\resizebox{\textwidth}{!}{
			\begin{tikzcd}[ampersand replacement=\&,column sep=huge,row sep=large]
				0 \arrow[r] \& A_x \arrow[r, "{(-\rho, -F_\nabla)}"] \& T_x M \oplus \mathbbm{R} \arrow[r, 
				"
				{\left(\begin{smallmatrix}
						d^\nabla \theta & \theta^\dag \\
						-\theta & 0
					\end{smallmatrix}\right)}
				"
				] \& T_x^{\dag}M\oplus L_{M,x} \arrow[r, "\rho^\dag + F_\nabla^\dag"] \& A_x^\dag \arrow[r] \& 0
			\end{tikzcd}
		},
	\end{equation}
	which is exactly the mapping cone of the Morita curvature of $\theta$ at the point $x\in M$. As the mapping cone is acyclic if and only if the cochain map is a quasi-isomorphism (Remark \ref{rem:mapping_cone}), this concludes the proof.	
\end{proof} 

\begin{rem}
	Let $(L\rightrightarrows L_M; G \rightrightarrows M)$ be an LBG equipped with a $0$-shifted contact structure $\theta$ and let $\omega$ be the corresponding $0$-shifted symplectic Atiyah form. It follows from the exactness of the sequence \eqref{eq:map_con_omega} and from $\omega$ being skew-symmetric, that $\mathbf{d} := 2 \dim M - \dim G$ is odd. Indeed, $\omega$ is non-degenerate on $D_xL_M/\im \mathcal{D}$, then $D_xL_M/\im \mathcal{D}$ is even-dimensional. On the other hand $\mathcal{D}$ is injective, then
	\[
		\dim (\im \mathcal{D})= \dim A= \dim G - \dim M,
	\]
	and $\dim D_xL_M= \dim M +1$ (see Section \ref{sec:Atiyah_algebroid}). Summarizing $2\dim M+1 -\dim G$ is even and so $\mathbf{d}$ is odd.
	
	The integer $\mathbf d$ is sometimes referred to as the \emph{dimension of the differentiable stack $[M/G]$} \cite[Definition 2.27]{BXu11}. When $G$ is a foliation groupoid, then $\mathbf d$ agrees with the dimension of the leaf space. 
\end{rem}

\begin{rem}\label{rem:hom_0-shift_sympl_grp}
	Let $Q \rightrightarrows P$ be a Lie groupoid and let $h$ be a principal action of $\mathbbm R^\times$ on $Q$ by Lie groupoid isomorphisms. Let $G = Q/\mathbbm R^\times$, and $M = P/\mathbbm R^\times$. It is clear that $G \rightrightarrows M$ is a Lie groupoid. Moreover, the line bundles $L \to G, L_M \to M$ associated to the principal $\mathbbm R^\times$-actions on $Q, P$ fit in an obvious LBG $(L \rightrightarrows L_M; G \rightrightarrows M)$. Now let $\theta \in \Omega^1 (M, L_M)$ and, using that $\Gamma(L_M)$ embeds in $C^\infty(P)$ as the $C^\infty(M)$-submodule formed by homogeneous functions of degree $1$, we can interpret $\theta$ as an homogeneous $1$-form on $P$ of degree $1$. Denote the latter by $\Theta \in \Omega^1 (P)$. The assignment $\theta \mapsto d\Theta$ establishes a bijection between $0$-shifted contact structures on $L$ and homogeneous $0$-shifted symplectic structures of degree $1$ on $Q \rightrightarrows P$. This can be easily seen, e.g., using Theorem \ref{theor:0-contact=0-Atiyah} and the relationship between degree $1$ homogeneous differential forms and Atiyah forms, explained in Section \ref{sec:dictionary}.
\end{rem}

We conclude this section with some examples.

\begin{example}\label{exmpl:0-shifted_nonreg}
	Unlike for $0$-shifted symplectic structures, a Lie groupoid supporting a $0$-shifted contact structure $\theta$ needs not be regular. The rank of the anchor drops exactly at points $x$ where $\theta_x = 0$. Here we discuss a $0$-shifted contact structure $\theta$ such that $\theta_x$ is generically non-zero, but $\theta_x = 0$ along a submanifold (of positive codimension). Begin with $M = \mathbbm R^2$ with standard coordinates $(x, y)$ and the vector field $X = y \tfrac{\partial}{\partial y}$. Denote by $A \to M$ the Lie algebroid defined as follows. Set $A = \mathbbm R_{M} = \mathbbm R^2 \times \mathbbm R$, the anchor $\rho \colon A \to TM$ maps $1_M$, the constant function equal to $1$, to the vector field $X$, and the Lie bracket of two functions $f, g \in \Gamma (A) = C^\infty (M)$ is
	\[
	[f, g] := f X(g) - g X(f).
	\]
	The Lie algebroid $A$ is integrable (as all rank $1$ Lie algebroids) and it integrates to the Lie groupoid $G \rightrightarrows M$ defined by the flow $\Phi^X \colon (x, y; \varepsilon) \mapsto (x, e^\varepsilon y)$ of $X$. Hence $G = M \times \mathbbm R$ with coordinates $(x, y; \varepsilon)$, and its structure maps are: source and target are defined by
	\[
	s(x, y; \varepsilon) = (x, y), \quad t(x, y; \varepsilon) = (x, e^\varepsilon y).
	\]
	The multiplication is given by
	\[
		m\big((x', e^{\varepsilon'} y'; \varepsilon), (x', y'; \varepsilon')\big) = (x', y'; \varepsilon + \varepsilon').	
	\]
	Finally, the unit and the inverse are given by
	\[
	u (x, y) = (x, y; 0), \quad i(x, y; \varepsilon) = (x, e^\varepsilon y; - \varepsilon).
	\]
	We let $G$ act on the trivial line bundle $\mathbbm R_M$ with coordinates $(x, y; r)$ via
	\[
	(x, y; \varepsilon). (x,y; r) := (x, e^\varepsilon y; e^{\varepsilon}r).
	\]
	The infinitesimal action $\mathcal D \colon A \to D \mathbbm R_M$ then maps $1_M$ to $\mathcal D_{1_M} = X - \mathbbm I$. In particular $\mathcal D$ is injective. The $\mathbbm R_M$-valued $1$-form
	\[
	\theta = ydx \otimes 1_M
	\]
	is a $0$-shifted contact structure on the LBG $G \ltimes \mathbbm R_M$. To see this, first notice that $\theta$ is invariant under the $G$-action on $\mathbbm R_M$. Equivalently, $\partial \theta = 0$. For the ``non-degeneracy of the Morita curvature'' we prefer to use Theorem \ref{theor:0-contact=0-Atiyah}. So let $\omega \rightleftharpoons (\theta, 0)$ be the $\partial$-closed Atiyah $2$-form on $\mathbbm R_M$ corresponding to $\theta$. From Remark \ref{rem:Acoordinates}, we have
	\[
	\omega = \left(\sigma^\ast(dy) \wedge \sigma^\ast(dx) - y \sigma^\ast(dx) \wedge \mathbbm I^\ast\right)  \otimes 1_M
	\]
	where we are denoting by $(dx, dy, \mathbbm I^\ast)$ the basis of $\Gamma ((D \mathbbm{R}_M)^\ast)$ dual to the basis $(\tfrac{\partial}{\partial x}, \tfrac{\partial}{\partial y}, \mathbbm I)$ of $\Gamma (D\mathbbm R_M)$. It follows that
	\[
	\ker \omega = \operatorname{Span} (\mathcal D_{1_M}) = \im \mathcal D. 
	\]
	We conclude that $\omega$ is a $0$-shifted symplectic Atiyah form, hence $\theta$ is a $0$-shifted contact structure. Clearly $\theta_{(x, y)} = 0$ exactly when $y = 0$.
\end{example}

The next example is a significant generalization of Example \ref{exmpl:0-shifted_nonreg}.

\begin{example}
	Let $h$ be an action of $\mathbbm R^\times$ on a manifold $P$.
	%and let $\omega \in \Omega^2 (M)$ be a degree $1$ homogeneous symplectic form, i.e.~$\omega$ is symplectic and, moreover, $h_t^\ast \omega = t \omega$ for all $t \in \mathbbm R^\times$. 
	Denote by $\mathcal E$ the infinitesimal generator of $h$. The Lie group $\mathbbm R^\times$ acts on the trivial line bundle $\mathbbm R_P$ as follows: $\varepsilon.(p, r) := (h_\varepsilon(p), \varepsilon r)$, and we will need to consider both the action groupoid $G := \mathbbm R^\times \ltimes P \rightrightarrows P$ and the action LBG $L := \mathbbm R^\times \ltimes \mathbbm R_P \rightrightarrows \mathbbm R_P$. We know from Section \ref{sec:dictionary} that, when $h$ is a principal action, then degree $1$ homogeneous symplectic forms on $P$ correspond bijectively to contact forms on $P/\mathbbm R^\times$ with values in the line bundle $\mathbbm R_P / \mathbbm R^\times$.
	%It then follows from homogeneity that $\mathcal L_{\mathcal E} \omega = \omega$.
	We want to show that, when the action $h$ is \emph{not} principal, yet degree $1$ homogeneous symplectic forms on $P$ correspond bijectively to $0$-shifted contact structures on $L$. To see this begin with a $1$-form $\theta \in \Omega^1 (P) = \Omega^1 (P, \mathbbm R_P)$. The condition $\partial \theta=0$ is equivalent to $\theta$ being homogeneous of degree $1$ (beware that here $\partial$ is the differential on forms with values in the nerve of $L$, not the differential on ordinary forms). In this case $\omega = d \theta \in \Omega^2 (P)$ is also homogeneous of degree $1$, and $\theta = i_{\mathcal E} \omega$. Fix the trivial connection $\nabla$ on $\mathbbm R_P$ and notice that, with this choice, the sequence \eqref{eq:mc_mc} boils down to
	\begin{equation}
		{\footnotesize
			\begin{tikzcd}[ampersand replacement=\&,column sep=huge,row sep=large]
				0 \arrow[r] \& \mathbbm R \arrow[r, "{(-\mathcal E, \operatorname{id}_{\mathbbm R})}"] \& T_x P \oplus \mathbbm{R} \arrow[r, 
				"
				{\left(\begin{smallmatrix}
						\omega & \theta \\
						-\theta & 0
					\end{smallmatrix}\right)}
				"
				] \& T_x^{\ast}P\oplus \mathbbm R \arrow[r, "i_{\mathcal E} - \operatorname{id}_{\mathbbm R}"] \& \mathbbm R \arrow[r] \& 0
			\end{tikzcd}
		},
	\end{equation}
	which is exact if and only if $\omega$ is non-degenerate. This shows that the assignment $\theta \mapsto d\theta$ establishes a bijection between $0$-shifted contact structures on $L$ and degree $1$ homogeneous symplectic structures on $P$ whose inverse is given by $\omega \mapsto i_{\mathcal E} \omega$. 	
\end{example}

\begin{example}
	Let $L_B \to B$ be a line bundle. The unit groupoid $L_B \rightrightarrows L_B$ is an LBG over the unit groupoid $B \rightrightarrows B$. A $0$-shifted contact structure on $L_B \rightrightarrows L_B$ is the same as an ordinary $L_B$-valued contact $1$-form $\theta \in \Omega^1 (B, L_B)$.  Now, let $M \to B$ be a surjective submersion, and let $G = M \times_B M \rightrightarrows M$ be the corresponding submersion groupoid (example \ref{ex:submersion_groupoid}). Set $L_M = M \mathbin{\times_B} L_B$. Then $G$ acts on $L_M$ in the obvious way. The associated action groupoid $L = G \ltimes L_M$ is an LBG over $G$. Moreover, the projection $\pi \colon L \to L_B$ is actually a VB-Morita map onto the unit LBG. It follows from Theorem \ref{theo:Mor_inv_0-shift} that the $0$-shifted contact structures on $L$ are exactly the pull-backs $\pi^\ast \theta$ of some $L_B$-valued contact $1$-form $\theta$ on $B$. In other words, those $\theta \in \Omega^1 (M, L_M)$ for which there is a well-defined \emph{contact reduction} under the projection $L_M \to L_B$.  
\end{example}

\begin{example}[Contact structures on orbifolds]
	Let $G \rightrightarrows M$ be a proper and \'etale groupoid %(in particular the isotropy groups $G_x$, $x \in M$, are finite). The orbit space $X := M/G$ is an orbifold and $G$ defines an orbifold atlas on $X$ as follows. Let $x \in M$, and let $U \subseteq M$ be an open neighborhood of $x$ such that the restricted groupoid $G_U = (s \times t)^{-1} (U \times U) \rightrightarrows U$ identifies with an action groupoid $G_x \ltimes U \rightrightarrows U$, where $G_x$ acts (linearly) on $U$ via a diffeomorphism $U \cong T_x M$. The projection $U \to U/G_U \subseteq X$, together with the $G_x$-action on $U$, can be seen as an orbifold chart, and $X$ is covered by such charts. If $U, V \to X$ are two such charts, and $x \in U \cap V$, then a \emph{chart compatibility} is provided by any open subset $W \subseteq (s \times t)^{-1}(U \times V)$ such that $s \colon W \to U$ and $t \colon W \to V$ are both embeddings around $x$. 
	(see Example \ref{ex:orbifolds} for how to get an orbifold $X$ from a proper and \'etale groupoid). Let $L \rightrightarrows L_M$ be an LBG over $G \rightrightarrows M$, so that the orbit space $L_X := L_M / L$ is a line bundle (in the category of orbifolds) over $X$. Then $0$-shifted contact structures on $L$ are equivalent to contact structures on $X$ (i.e.~group invariant contact structures on each chart which are additionally preserved by chart compatibilities \cite[Definition 2.3.1]{H13}) with normal line bundle given exactly by $L_X$. From Theorem \ref{theo:Mor_inv_0-shift} the same is true for any \emph{orbifold groupoid}, i.e.~a proper foliation groupoid, presenting the orbifold $X$. %(remember that a groupoid is a foliation groupoid if and only if it is Morita equivalent to a proper, \'etale groupoid). We leave the obvious details to the reader. 
\end{example}

\section{$+1$-Shifted Contact Structures}
\label{sec:contact}

In this section we introduce \emph{$+1$-shifted contact structures}. The discussion will parallel that of Section \ref{sec:0-shifted_cs}. Given a multiplicative LBG-valued $1$-form, we define its \emph{Morita kernel} and its \emph{Morita curvature}. We will explain what does it mean for the Morita curvature to be ``non-degenerate'' in a Morita invariant way, and we will get our ultimate definition. We will conclude with two examples. Although the situation is more involved for $+1$-shifted contact structures, there are some features which are surprisingly simpler here than in the $0$-shifted case. Namely, the Morita kernel is a plain VBG and the Morita curvature is a plain VBG morphism in the $+1$-shifted case (no need to use RUTHs).

\subsection{Multiplicative LBG-valued $1$-forms}
\label{sec:mult_vv}
In this subsection we focus on multiplicative LBG-valued $1$-forms. Specifically, we establish several computational rules that will prove useful in the subsequent discussion.

Let $(L\rightrightarrows L_M;G\rightrightarrows M)$ be an LBG. By Definition \ref{def:multiplicative_vv_forms}, an $L$-valued $1$-form $\theta\in \Omega^1(G,L)$ is multiplicative if it is closed with respect to the differential $\partial$, i.e.,
\begin{equation}
	\label{eq:vv_mult}
	m^\ast \theta = \pr_1^\ast \theta +\pr_2^\ast \theta,
\end{equation}
where $\pr_i\colon L^{(2)}\to L$ is the projection on the $i$-th factor.

The next proposition is analogous to Propositions \ref{prop:formule} and \ref{prop:Atiyah_formule}.
\begin{prop}
	\label{prop:vv_formule}
	Let $\theta\in \Omega^1(G,L)$ be a multiplicative $L$-valued $1$-form on the LBG $L$. Then
	\begin{itemize}
		\item[i)] the pullback $u^\ast \theta$ of $\theta$ along the unit map $u\colon L_M\to L$ is zero;
		\item[ii)] the pullback $i^\ast \theta$ of $\theta$ along the inverse map $i\colon L\to L$ is $-\theta$.
%		\item[iii)] for any $a,b\in \Gamma(A)$, we have
%		\[
%		\omega(\overrightarrow{\Delta^a}_g,\overrightarrow{\Delta^b}_g) = - \omega(\overleftarrow{\Delta^a}_{g^{-1}},\overleftarrow{\Delta^b}_{g^{-1}})^{-1}, \quad \omega(\overrightarrow{\Delta^a},\overleftarrow{\Delta^b})=0,
%		\]
%		where $\overrightarrow{\Delta^a},\overrightarrow{\Delta^b}$ and $\overleftarrow{\Delta^a},\overleftarrow{\Delta^b}$ are the right and left-invariant derivations respectively, generated by $a$ and $b$ (see \textcolor{red}{cita}).
	\end{itemize}
\end{prop}
\begin{proof}
	In this proof we use that all the face maps of the nerve of $L$ are fiberwise isomorphisms (see Remark \ref{rem:face_maps_trivial_coreVBG}). Let $v\in T_xM$, with $x\in M$. Then, applying \eqref{eq:vv_mult} to $(v,v)\in T_{(x,x)}G^{(2)}$, we get
	\[
	m_{(x,x)}^{-1}\theta(v) = \pr_{1,(x,x)}^{-1}\theta(v)+\pr_{2,(x,x)}^{-1}\theta(v),
	\]
	then
	\[
	\theta(v) = m \left(\pr_{1,(x,x)}^{-1} \theta(v)\right) + m \left(\pr_{2,(x,x)}^{-1} \theta(v)\right).
	\]
	Applying the source map and remembering that $s\circ m= s\circ \pr_2$ we get
	\[
	s\left(\theta(v)\right)= s\left(\pr_2 \left(\pr_{1,(x,x)}^{-1} \theta(v)\right)\right) + s\left(\theta(v)\right),
	\]
	but $\pr_2\circ \pr_{1,(x,x)}^{-1}= t_x^{-1}\circ s$, so we have
	\[
	0=s\left(\pr_2 \left(\pr_{1,(x,x)}^{-1} \theta(v)\right)\right) = s\left(t_x^{-1}(s(\theta(v)))\right)= s(\theta(v))
	\]
	and then $\theta(v)=0$ and $i)$ is proved. 
	
	In order to prove $ii)$ let $v\in T_gG$, with $g\in G$. Applying \eqref{eq:vv_mult} to the pair $(v, v^{-1})\in T_{(g,g^{-1})}G^{(2)}$, we get
	\begin{equation}
		\label{eq:vv_inv_omega}
		m_{(g,g^{-1})}^{-1}\theta(v\cdot v^{-1}) = \pr_{1,(g,g^{-1})}^{-1}\theta(v) +\pr_{2,(g,g^{-1})}^{-1}\theta(v^{-1}),
	\end{equation}
	but $v\cdot v^{-1}= t(v)\in T_{t(g)}M$ and, from $i)$, Equation \eqref{eq:vv_inv_omega} reduces to
	\[
	-\theta(v)= \pr_1\left(\pr_{2,(g,g^{-1})}^{-1} \theta(v^{-1})\right)=s_g^{-1}\left(t\left(\theta(v^{-1})\right)\right)= \left(\theta(v^{-1})\right)^{-1},
	\]
	whence the statement.
%	The first part of $iii)$ follows from $ii)$ noting that, for any $g\in G$, 
%	\[
%	\overrightarrow{\Delta^a}_g= a_{t(g)}\cdot 0^{DL}_g= \left( 0^{DL}_{g^{-1}}\cdot a^{-1}_{s(g^{-1})}\right)^{-1}=(\overleftarrow{\Delta^a}_{g^{-1}})^{-1}
%	\] 
%	and the same for $b$. 
%	
%	Finally, for any $g\in G$, applying \eqref{eq:Atiyah_mult} to $(\overrightarrow{\Delta^a}_g, 0^{DL}_{s(g)}), (0^{DL}_g, \overleftarrow{\Delta^b}_{s(g)})\in D_{(g,s(g))}L^{(2)}$, we get
%	\[
%	m_{(g,s(g))}^{-1}\omega(\overrightarrow{\Delta^a}_g, \overleftarrow{\Delta^b}_g) = \pr_{1,(g,s(g))}^{-1}\omega(\overrightarrow{\Delta^a}_g, 0^{DL}_g) + \pr_{2,(g,s(g))}^{-1} \omega(0^{DL}_{s(g)}, \overleftarrow{\Delta^b}_{s(g)})=0,
%	\]
%	where we used that $\overrightarrow{\Delta^a}_g \cdot 0^{DL}_{s(g)}= \overrightarrow{\Delta^a}_g$ and $0^{DL}_g\cdot \overleftarrow{\Delta^b}_{s(g)}= \overleftarrow{\Delta^b}_g$.
\end{proof}

\begin{rem}
	Notice that Proposition \ref{prop:vv_formule} is also consequence of Proposition \ref{prop:Atiyah_formule} and the relation between $L$-valued forms and Atiyah forms on $L$. (see Section \ref{sec:dictionary}). For example, if $\omega\rightleftharpoons (\theta,0)\in \OA^2(L)$, then, by Remark \ref{rem:mult_Atiyah_components}, $\omega$ is multiplicative, and, from point $i)$ in Proposition \ref{prop:Atiyah_formule}, it follows that $u^\ast\omega=0$, but, by Equation \eqref{eq:Atiyah_components}, the components of $u^\ast\omega$ are $u^\ast\theta$ and $0$. It follows that $u^\ast\theta=0$. Many results in this section are consequence of the analogous results for Atiyah forms but we prefer to prove them without using Atiyah forms which might seem exotic objects to some readers.
\end{rem}

\begin{rem}
	As we explained in Section \ref{sec:VB-Morita}, an LB-stack is not given just by LBGs. Indeed, there are VBGs that are not LBG but they are Morita equivalent to a LBG. Let $L$ be a LBG and let $V$ be a VBG groupoid Morita equivalent to $L$. From \cite{DE19}, $\theta\in \Omega^1(G,V)$ is multiplicative if it determines a VBG morphism from $TG$ to $V$. Moreover, by Proposition \ref{prop:Morita_map_to _LBG}, $V$ is isomorphic to the direct sum $V'\oplus L'$, where $V'$ is a VBG, whose core-anchor is an isomorphism, and $L'$ is an LBG. Then we can decompose $\theta$ in a $V'$-valued $1$-form $\theta'$ and an $L'$-valued $1$-form. Since the core-anchor of $V'$ is an isomorphism, then $\theta'$ does not play any role and so we reduce again to consider LBG-valued $1$-forms.
\end{rem}

Let $\nabla$ be a connection on $L_M$. We consider again the $1$-form $\eta_\nabla\in \Omega^1(G)$ given by the difference $s^{\ast}\nabla - t^{\ast}\nabla$ between the pull-back connections on $L \to G$ along the source and the target $s, t \colon L \to L_M$. It will be useful in the next discussion to consider the $L$-valued $2$-form $d^{t^\ast\nabla}\theta\in \Omega^2(G,L)$. The latter is not multiplicative as explained in the following
\begin{lemma}
	Let $\nabla$ be a connection on $L_M$. Then $d^{t^\ast \nabla}\theta$ satisfies the following identity:
	\begin{equation}\label{eq:partial_dtheta}
		\partial \big(d^{t^\ast \nabla}\theta\big) = -\pr_1^\ast \eta_\nabla \wedge \pr_2^\ast \theta.
	\end{equation}
\end{lemma}
\begin{proof}
	The proof is a computation:
	\begin{align*}
		m^{\ast}d^{t^{\ast}\nabla}\theta &= d^{(tm)^{\ast}\nabla} m^{\ast}\theta\\
		&= d^{(tm)^{\ast}\nabla}(\pr_1^{\ast}\theta + \pr_2^{\ast}\theta)\\
		&= \big(d^{(tm)^{\ast}\nabla}\circ \pr_1^{\ast}\big )\theta + \big(d^{(t\pr_1)^{\ast}\nabla}\circ \pr_2^{\ast}\big)\theta \\
		&=\pr_1^{\ast}d^{t^{\ast}\nabla}\theta + \pr_2^{\ast}d^{t^{\ast}\nabla}\theta + \big(d^{(tm)^{\ast}\nabla} \circ \pr_2^{\ast} - \pr_2^{\ast}\circ d^{t^{\ast}\nabla}\big)\theta \\
		&=\pr_1^{\ast}d^{t^{\ast}\nabla}\theta + \pr_2^{\ast}d^{t^{\ast}\nabla}\theta +\big((d^{(tm)^{\ast}\nabla}- d^{(t\pr_2)^{\ast}\nabla})\circ \pr_2^{\ast}\big )\theta\\
		&= \pr_1^{\ast}d^{t^{\ast}\nabla}\theta + \pr_2^{\ast}d^{t^{\ast}\nabla}\theta + \big((t\pr_1)^{\ast}\nabla - (s\pr_1)^{\ast}\nabla\big)\wedge \pr_2^{\ast}\theta\\
		&=\pr_1^{\ast}d^{t^{\ast}\nabla}\theta + \pr_2^{\ast}d^{t^{\ast}\nabla}\theta + \pr_1^{\ast}(t^{\ast}\nabla - s^{\ast}\nabla)\wedge \pr_2^{\ast}\theta\\
		&=\pr_1^{\ast}d^{t^{\ast}\nabla}\theta + \pr_2^{\ast}d^{t^{\ast}\nabla}\theta - \pr_1^{\ast}\eta_{\nabla} \wedge \pr_2^{\ast}\theta,
	\end{align*}
	whence the claim.
\end{proof}

We now aim to analyze how care $d^{t^\ast\nabla}$ and $\eta_\nabla$ affected by the choice of a different connection. For any two connections $\nabla$ and $\nabla'$ on $L_M$ we set $\alpha_{\nabla, \nabla'}=\nabla-\nabla'\in \Omega^1(M)$, the plain $1$-form on $M$ defined by setting
\[
\alpha_{\nabla, \nabla'}(v)= \nabla_v-\nabla'_v\in \ker\sigma \cong \mathbbm{R},
\]
for all $v\in T_xM$, with $x\in M$.
\begin{lemma}
	Let $\nabla$ and $\nabla'$ be two connections on $L_M$ and let $\theta$ be a multiplicative $L$-valued $1$-form. Then
	\begin{equation}
		\label{eq:eta_and_alpha}
		\eta_{\nabla'} -\eta_\nabla= -s^\ast \alpha_{\nabla, \nabla'} + t^\ast \alpha_{\nabla, \nabla'},
	\end{equation}
	and
	\begin{equation}
		\label{eq:d_two_connections}
		d^{t^\ast\nabla'}\theta - d^{t^\ast\nabla}\theta = -t^\ast \alpha_{\nabla, \nabla'}\wedge \theta.
	\end{equation}
\end{lemma}
\begin{proof}
	The proof amounts to two easy computations:
	\begin{equation*}
		\eta_{\nabla'} -\eta_\nabla= s^\ast \nabla' -t^\ast \nabla' -s^\ast \nabla +t^\ast \nabla = -s^\ast(\nabla-\nabla') + t^\ast(\nabla-\nabla')= -s^\ast \alpha_{\nabla, \nabla'} + t^\ast \alpha_{\nabla, \nabla'},
	\end{equation*}
	and
	\begin{equation*}
		d^{t^\ast\nabla'}\theta -d^{t^\ast\nabla}\theta = (t^\ast\nabla'-t^\ast\nabla)\wedge\theta = t^\ast(\nabla'-\nabla)\wedge \theta= -t^\ast\alpha_{\nabla, \nabla'}\wedge \theta.\qedhere
	\end{equation*}
\end{proof}

\begin{rem}
	Let $\nabla'$ be another connection on $L_M$. By Equation \eqref{eq:eta_and_alpha}, $\eta_{\nabla} - \eta_{\nabla'} = \partial \alpha_{\nabla, \nabla'}$, where $\alpha_{\nabla, \nabla'} = \nabla - \nabla'\in \Omega^1(M)$ is a $1$-form on $M$. In other words the $\partial$-cohomology class $\varkappa_L := [\eta_{\nabla}]$ of $\eta_\nabla$ is independent of $\nabla$ and it is a ``characteristic class'' attached to the LBG $L$. Clearly $\varkappa_L$ is the obstruction to the existence of a $G$-invariant connection on $L_M$, i.e.~a connection $\nabla$ such that $s^\ast \nabla = t^\ast \nabla$.
\end{rem}

We conclude this subsection with some identities involving $d^{t^\ast\nabla}\theta$ analogous to the one of Proposition \ref{prop:vv_formule}.
\begin{prop}
	\label{prop:dnabla_formule}
	Let $\theta\in \Omega^1(G,L)$ be a multiplicative $1$-form on the LBG $(L\rightrightarrows L_M;G\rightrightarrows M)$, and let $\nabla$ be a connection on $L_M$. Then the $L$-valued $2$-form $d^{t^\ast \nabla}\theta\in \Omega^2(G,L)$ satisfies the following identities: 
	\begin{itemize}
		\item[i)] the pullback $u^\ast d^{t^\ast \nabla}\theta$ of $d^{t^\ast\nabla}\theta$ along the unit map $u\colon L_M\to L$ is zero;
		\item[ii)] the pullback $i^\ast d^{t^\ast \nabla}\theta$ of $d^{t^\ast\nabla}\theta$ along the inverse map $i\colon L\to L$ is equal to $-d^{t^\ast \nabla}\theta + \eta_\nabla\wedge \theta$.
	\end{itemize}
\end{prop}
\begin{proof}
	For any $v,w\in T_xM$, with $x\in M$, applying Equation \eqref{eq:partial_dtheta} to $(v,v), (w,w)\in T_{(x,x)}G^{(2)}$, we get
	\begin{equation}
		\label{eq:gggg}
		\begin{aligned}
			m_{(x,x)}^{-1}\left(d^{t^\ast\nabla}\theta (v,w)\right) &= \pr_{1,(x,x)}^{-1}\left(d^{t^\ast\nabla}\theta(v,w)\right) + \pr_{2,(x,x)}^{-1}\left(d^{t^\ast\nabla}\theta(v,w)\right) \\
			&\quad - \eta_\nabla(v)\pr_{2,(x,x)}^{-1}\big(\theta(w)\big)+ \eta_\nabla(w)\pr_{2,(x,x)}^{-1}\big(\theta(v)\big) \\
			&= \pr_{1,(x,x)}^{-1}\left(d^{t^\ast\nabla}\theta(v,w)\right) + \pr_{2,(x,x)}^{-1}\left(d^{t^\ast\nabla}\theta(v,w)\right),
		\end{aligned}
	\end{equation}
	where we used that $u^\ast\eta_{\nabla}=0$ (see Remark \ref{rem:eta_nabla_formule}), or $u^\ast\theta=0$ (Propostion \ref{prop:vv_formule}$.i)$). Applying $s\circ m$ to \eqref{eq:gggg}, we get
	\begin{equation*}
		s\left(m\left(\pr_{1,(x,x)}^{-1}\left(d^{t^\ast\nabla}\theta(v,w)\right)\right)\right)=0,
	\end{equation*}
	and so $d^{t^\ast\nabla}\theta(v,w)=0$ and $i)$ is proved.
	
	For any $v,w\in T_gG$, with $g\in G$, applying Equation \eqref{eq:partial_dtheta} to the pairs $(v,v^{-1}), (w.w^{-1})\in T_{(g,g^{-1})}G^{(2)}$, we have
	\begin{align*}
		m_{(g,g^{-1})}^{-1}&\left(d^{t^\ast\nabla}\theta \big(t(v),t(w)\big)\right) \\
		 &= \pr_{1,(g,g^{-1})}^{-1}\left(d^{t^\ast\nabla}\theta(v,w)\right) + \pr_{2,(g,g^{-1})}^{-1}\left(d^{t^\ast\nabla}\theta\big(v^{-1},w^{-1}\big)\right) \\
		&\quad - \eta_\nabla(v)\pr_{2,(g,g^{-1})}^{-1}\big(\theta(w^{-1})\big)+ \eta_\nabla(w)\pr_{2,(g,g^{-1})}^{-1}\big(\theta(v^{-1})\big).
	\end{align*}
	From point $i)$ and applying $\pr_1$, we get
	\begin{align*}
		i^\ast d^{t^\ast\nabla}\theta (v,w)&=
		\left(d^{t^\ast\nabla}\theta\big(v^{-1},w^{-1}\big)\right)^{-1} \\
		&=s_g^{-1}\left(t\left(d^{t^\ast\nabla}\theta\big(v^{-1},w^{-1}\big)\right)\right) \\
		&= \pr_1\left(\pr_{2,(g,g^{-1})}^{-1}\left(d^{t^\ast\nabla}\theta\big(v^{-1},w^{-1}\big)\right)\right) \\
		&=-d^{t^\ast\nabla}\theta(v,w) + \eta_\nabla(v) \left(\theta(w^{-1})\right)^{-1} - \eta_\nabla(w) \left(\theta(v^{-1})\right)^{-1}\\
		&=-d^{t^\ast\nabla}\theta(v,w) + \eta_\nabla(v)\theta(w) - \eta_\nabla(w) \theta(v)\\
		&= \left(-d^{t^\ast\nabla}\theta + \eta_\nabla\wedge \theta\right)(v,w),
	\end{align*}
	where we used Proposition \ref{prop:vv_formule}$.ii)$.
\end{proof}

\subsection{Morita Kernel}\label{sec:Morita_kernel}
The first step in providing a definition of $+1$-shifted contact structure is identifying the appropriate analogue of the kernel for a multiplicative LBG-valued $1$-form. Following the strategy outlined in Section \ref{sec:0-shifted_cs}, we show that the role of the kernel of a multiplicative $1$-form is played by a RUTH. Unlike the $0$-shifted case, this RUTH is concentrated in degrees $-1$ and $0$, then it corresponds to a VBG $\mk_{\theta}$. We prove that $\mk_{\theta}$ simplifies when the $1$-form is nowhere-zero. Lastly, we establish the Morita invariance of $\mk_{\theta}$, ensuring its well-definedness as a structure on differentiable stacks.

Let $(L\rightrightarrows L_M;G\rightrightarrows M)$ be an LBG and let $\theta \in \Omega^1(G,L)$ be a multiplicative $1$-form (see Definition \ref{def:multiplicative_vv_forms}). The kernel of $\theta$ is not a well-defined subgroupoid of $TG\rightrightarrows TM$ in general, because its dimension jumps at points $g\in G$ where $\theta_g=0$. The $1$-form $\theta$ induces a trivial cochain map
\begin{equation}
	\label{eq:cm_theta}
	\begin{tikzcd}
		0 \arrow[r] & A\arrow[r, "\rho"] \arrow[d, "\ell_\theta"'] & TM\arrow[r]\arrow[d] & 0\\
		0 \arrow[r] & L_M \arrow[r] & 0 \arrow[r] & 0
	\end{tikzcd}
\end{equation}
between the core complex of $TG$ and the core complex of $L$ shifted by $+1$, where $\ell_\theta\colon A\to L_M$ is the restriction of $\theta$ to the Lie algebroid $A$. Similarly as we did for $0$-shifted contact structures, inspired again by the homotopy kernel in Homological Algebra, we consider the mapping cone of the cochain map \eqref{eq:cm_theta} which is (up to a conventional sign):
\begin{equation}
	\label{eq:mc_1_shifted}
	\begin{tikzcd}
		0 \arrow[r] & A\arrow[r, "(\rho {,}\ell_\theta)"] & TM\oplus L_M \arrow[r] &0.
	\end{tikzcd}
\end{equation}

The complex \eqref{eq:mc_1_shifted} is not just a complex but it is a RUTH of $G$. We discuss this in the following
\begin{rem}
	\label{rem:MK_RUTH_1}
	By Remark \ref{rem:theta_morphisms}, $\theta$ determines a VBG morphism from $TG\rightrightarrows TM$ to $L\rightrightarrows 0$, then it follows that it could be promoted to a RUTH morphism between the core complexes of $TG\rightrightarrows TM$ and $L\rightrightarrows 0$. We now describe this RUTH morphism in detail using a more direct approach. We complete the cochain map \eqref{eq:cm_theta} to a RUTH morphism between the adjoint RUTH (see Example \ref{ex:adjointRUTH}) and the plain representation induced by $L$ and then we apply the mapping cone construction (see Example \ref{ex:mp_RUTH}). We choose one for all an Ehresmann connection $h\colon s^\ast TM\to TG$ (see Definition \ref{def:Ehresmann_connection}). The $1$-form $\theta$ is a degree $0$ graded vector bundle morphism $A[-1]\oplus TM\to L_M[-1]\oplus 0$ that we call $\Phi_0$. We define $\Phi_1\colon G\to \operatorname{Hom}(s^\ast TM,t^\ast L_M)$ by setting
	\begin{equation*}
		\Phi_1(g)v= t\big(\theta\big(h_g(v)\big)\big), \quad v\in T_{s(g)}M,
	\end{equation*} 
	with $g\in G$. Then, setting $\Phi_k=0$ for $k>1$ we have that
	\[
		\Phi\colon \left(C(G; A[-1] \oplus TM), \partial^{\mathrm{Ad}} \right) \to \left(C(G; L_M[-1]), \partial^{L}\right)
	\]
	is a RUTH morphism. To show this, first notice that, by Proposition \ref{prop:vv_formule}$.i)$, $\Phi_1(x)=0$ for all $x\in M$, then we have to use \eqref{eq:struct_sect_RUTH_mor}. For $k=0$ we already know that $\theta$ determines a trivial cochain map.
	
	For $k=1$, Equation \eqref{eq:struct_sect_RUTH_mor} says that, for any $g\in G$, we have
	\[
		\Phi_0\circ R_1(g) - \Phi_1(g)\circ R_0 = R_1'(g) \circ \Phi_0 + R'_0\circ \Phi_1(g).
	\]
	The latter in our case, using $R_0'=0$, simply means that for any $a\in A_{s(g)}$ we have
	\[
		\theta(g_T.a)- t\big(\theta(h_g(v))\big)= g.\theta(a).
	\]
	Using Equation \eqref{eq:quasi_action_onA} and the multiplicativity of $\theta$, we have
	\begin{align*}
		\theta(g_T.a)&= \theta\left(h_g\big(\rho(a)\big)\cdot a\cdot 0_{g^{-1}}\right)= m\left(\pr_{1,(g,g^{-1})}^{-1}\theta\big(h_g\rho(a)\cdot a\big)\right) \\
		&=t\left(\theta\big(h_g\rho(a)\cdot a\big)\right)= t\left(m\left(\pr_{1,(g,s(g))}^{-1}\theta\big(h_g\rho(a)\big) + \pr_{2,(g,s(g))}^{-1}\theta(a)\right)\right)\\
		&= t\left(\theta\big(h_g\rho(a)\big)\right) + g.\theta(a).
	\end{align*}
	
	For $k=2$, Equation \eqref{eq:struct_sect_RUTH_mor} says that, for any $(g,g')\in G^{(2)}$, we have
	\begin{align*}
		\Phi_0\circ R_2(g,g')-\Phi_1(g)\circ R_1(g') +\Phi_2(g,g')\circ R_0 &= R_0'\circ \Phi_2(g,g') + R_1'(g)\circ \Phi_1(g')\\
		&\quad  + R_2'(g,g') \circ \Phi_0 -\Phi_1(gg').
	\end{align*}
	The latter in our case, using $\Phi_2=0$ and $R_2'=0$, simply means that, for any $v\in T_{s(g')}M$ we have
	\begin{equation*}
		\theta(R_2(g,g')v) -t(\theta(h_g(g'_T.v)))= g.t(\theta(h_gv)) -t(\theta(h_{gg'}v)).
	\end{equation*}
	But we have
	\begin{align*}
		\theta(R_2(g,g')v)&= \theta\left(\left(h_{gg'}v - h_g\big(t(h_{g'}v)\big)\right)\cdot 0_{gg'}^{-1}\right) \\
		&=m\left(\pr_{1,(gg',g'{}^{-1}g^{-1})}^{-1} \theta\left(h_{gg'}v-h_g\big(t(h_{g'}v)\big)\cdot h_{g'}v\right)\right)\\
		&=t(\theta(h_{gg'}v)) -t\left(\theta\left(h_g(t(h_{g'}v))\cdot h_{g'}v\right)\right)\\
		&=t(\theta(h_{gg'}v)) -t\left(m\left(\pr_{1,(g,g')}^{-1}\theta\big(h_g(t(h_{g'}v))\big)+ \pr_{2,(g,g')}^{-1}\theta(h_{g'}v)\right)\right)\\
		&=t\big(\theta(h_{gg'}v)\big) -t\big(\theta(h_g(g'_T,v))\big) -g.t\big(\theta(h_{g'}v)\big),
	\end{align*}
	and the case $k=2$ is proved. For degree reasons, the equalities for $k>2$ are trivially true and so $\Phi$ is a RUTH morphism.
	
	Applying the mapping cone we get a RUTH on $A[-1]\oplus(TM\oplus L_M)$ whose structure operators $\{R_k\}_{k\geq 0}$ are the following: the $0$-th structure operator is given by
	\begin{align*}
		R_0(a)= \big(\rho(a), \ell_{\theta}(a)\big)\in T_xM\oplus L_{M,x}, \quad a\in A_x,
	\end{align*}
	with $x\in M$. The $1$-st structure operator is given by
	\begin{align*}
		R_1(g)(a)&= g_T.a\in A_{t(g)}, \\
		R_1(g)(v,\lambda)&= \big(g_T.v, t(\theta(h_gv))+ g.\lambda\big)\in T_{t(g)}M \oplus L_{M,t(g)}, 
	\end{align*}
	for all $a\in A_{s(g)}$, $v\in T_{s(g)}M$ and $\lambda\in L_{M,s(g)}$, with $g\in G$. Finally, the $2$-nd structure operator is given by
	\begin{align*}
		R_2(g,g')(v,\lambda)= R^T_2(g,g')(v)\in A_{t(g)}, \quad v\in T_{s(g')}M, \, \lambda \in L_{M,s(g')},
	\end{align*}
	with $(g,g')\in G^{(2)}$, where $R_2^T$ is the second structure operator of the adjoint RUTH.
\end{rem}

Unlike the $0$-shifted case, the complex \eqref{eq:cm_theta} is concentrated in degrees $-1$ and $0$. Then, by Remark \ref{rem:VBG-RUTH}, \eqref{eq:cm_theta} is just the core complex of a VBG. We begin with defining this VBG and proceed to demonstrate that it corresponds to the one derived from the aforementioned RUTH. This relation allows us to view that VBG as the kernel for the multiplicative $1$-form $\theta$. The total and side bundles are simply given by the direct sums $TG\oplus L$ and $TM\oplus L_M$. The structure maps are defined as follows: source and target are given by
\begin{align*}
	s(v,\lambda)=\big(s(v),s(\lambda)\big), \quad t(v,\lambda)=\big(t(v), t\big(\lambda +\theta(v)\big)\big), \quad v\in TG,\quad \lambda \in L,
\end{align*}
the unit map is given by
\begin{equation*}
	u(v,\lambda)=\big(u(v), u(\lambda)\big), \quad v\in TM,\quad \lambda \in L_M,
\end{equation*}
the inverse map is given by
\begin{equation*}
	(v,\lambda)^{-1}= \left(v^{-1}, s_{g^{-1}}^{-1}\left(t\big(\lambda+ \theta(v)\big)\right)\right), \quad v\in T_gG, \quad \lambda\in L_g.
\end{equation*}
Finally the multiplication map is given by
\begin{equation*}
	(v,\lambda)\cdot(v',\lambda')=\left(vv', s_{gg'}^{-1}\big(s(\lambda')\big)\right), \quad (v,v')\in T_{(g,g')}G^{(2)}, \quad (\lambda, \lambda')\in L^{(2)}_{(g,g')}.
\end{equation*}

Fist, we prove that the latter is a well-defined VBG in the following
\begin{prop/def}
	For any multiplicative $1$-form $\theta\in \Omega^1(G,L)$, $TG\oplus L\rightrightarrows TM \oplus L_M$ with the structure maps defined above is a VBG over $G\rightrightarrows M$ that we call the \emph{Morita kernel} of $\theta$ and we also denote by $\mk_\theta$.
\end{prop/def}
\begin{proof}
	The proof is a straightforward computation to check the VBG axioms. We explain the details for completeness. First, by definition, all the maps defined are VB morphisms, then we just have to check that $TG\oplus L\rightrightarrows TM\oplus L_M$ with the structure maps defined above is a Lie groupoid. Let $(v,v')\in T_{(g,g')}G^{(2)}$ and $(\lambda,\lambda')\in L_{(g,g')}^{(2)}$. Then
	\begin{align*}
		s\big((v,\lambda)\cdot(v',\lambda')\big)&= s\left(vv', s_{gg'}^{-1}\big(s(\lambda')\big)\right)\\
		&=\big(s(vv'), s(\lambda')\big)= \big(s(v'), s(\lambda')\big)= s(v',\lambda'),
	\end{align*}
	and 
	\begin{align*}
		t\big((v,\lambda)\cdot(v',\lambda')\big)&= t\left(vv', s_{gg'}^{-1}\big(s(\lambda')\big)\right)=\left(t(vv'), t\left(s_{gg'}^{-1}\big(s(\lambda')\big) + \theta(vv')\right)\right)\\
		&= \big(t(v), t(\lambda + \theta(v))\big)=t(v,\lambda),
	\end{align*}
	where we used that
	\begin{align*}
		t\left(s_{gg'}^{-1}\big(s(\lambda')\big) + \theta(vv')\right)&= t\big(s_{gg'}^{-1}\big(s(\lambda')\big)\big) + t\left(m\left(\pr_{1,(g,g')}^{-1}\theta(v) + \pr_{2,(g,g')}^{-1}\theta(v')\right)\right)\\
		&=t\big(s_{gg'}^{-1}\big(s(\lambda')\big)\big) + t\big(\theta(v)\big) + t\left(s_{gg'}^{-1}\big(s\big(\theta(v')\big)\big)\right)\\
		&=t\big(\theta(v)\big) + t\left(s_{gg'}^{-1}\big(s\big(\lambda'+ \theta(v')\big)\big)\right)\\
		&=t\big(\theta(v)\big) + gg'.\left(g'{}^{-1}. t\big(\lambda'+\theta(v')\big)\right)\\
		&= t\big(\theta(v)\big) + g.s(\lambda) \\
		&=t\big(\theta(v)+ \lambda\big).
	\end{align*}
	
	For any $(v,v',v'')\in T_{(g,g',g'')}G^{(3)}$ and $(\lambda,\lambda',\lambda'')\in L_{(g,g',g'')}^{(3)}$, we have
	\begin{align*}
		\big((v,\lambda)\cdot (v',\lambda')\big)\cdot (v'',\lambda'')&= \left(vv', s_{gg'}^{-1}\big(s(\lambda')\big)\right)\cdot (v'',\lambda'')\\
		&=\left((vv')v'', s_{(gg')g''}^{-1}\big(s(\lambda'')\big)\right)\\
		&=\left(v(v'v''), s_{g(g'g'')}^{-1}\big(s(\lambda'')\big)\right)\\
		&=(v,\lambda)\cdot \left(v'v'', s_{g'g''}^{-1}\big(s(\lambda'')\big)\right)\\
		&=(v,\lambda)\cdot \big((v',\lambda')\cdot (v'',\lambda'')\big).
	\end{align*}
	
	For any $v\in T_gG$ and $\lambda\in L_g$, we have
	\begin{align*}
		(v,\lambda)\cdot u\big(s(v,\lambda)\big)= (v,\lambda)\cdot \big(s(v), s(\lambda)\big)= \left(v\cdot s(v), s_{g}^{-1}\big(s(\lambda)\big)\right)= (v,\lambda),
	\end{align*}
	and
	\begin{align*}
		u\big(t(v,\lambda)\big)\cdot (v,\lambda)=\big(t(v), t\big(\lambda+\theta(v)\big)\big)\cdot (v,\lambda) = \big(t(v)\cdot v, s_g^{-1}\big(s(\lambda)\big)\big)= (v,\lambda). 
	\end{align*}
	
	Finally, for any $v\in T_gG$ and $\lambda\in L_g$, we have
	\begin{equation*}
		s((v,\lambda)^{-1})= s\left(v^{-1}, s_{g^{-1}}^{-1}\left(t\big(\lambda+ \theta(v)\big)\right)\right)= \big(t(v), t\big(\lambda + \theta(v)\big)\big)= t(v,\lambda),
	\end{equation*}
	and
	\begin{align*}
		t\left((v,\lambda)^{-1}\right)&= t\left(v^{-1}, s_{g^{-1}}^{-1}\left(t\big(\lambda+ \theta(v)\big)\right)\right)\\ &=\left(t(v^{-1}), t\left(s_{g^{-1}}^{-1}\big(t(\lambda+\theta(v))\big)+ \theta(v^{-1})\right)\right)=\big(s(v), s(\lambda)\big)= s(v,\lambda),
	\end{align*}
	where we used that
	\begin{align*}
		t\left(s_{g^{-1}}^{-1}\big(t\big(\lambda+\theta(v)\big)\big)+ \theta(v^{-1})\right)&= s\left(t_g^{-1}\big(t\big( \lambda+\theta(v)\big)\big)\right)- s\big(\theta(v)\big)\\
		&= s\big(\lambda+\theta(v)\big) - s\big(\theta(v)\big)= s(\lambda).
	\end{align*}
	Moreover, we have
	\begin{align*}
		(v,\lambda)\cdot (v,\lambda)^{-1} &= (v,\lambda)\cdot \left(v^{-1}, s_{g^{-1}}^{-1}\big(t\big(\lambda+\theta(v)\big)\big)\right) \\
		&=\big(v\cdot v^{-1}, s_{gg^{-1}}^{-1}\big(t\big(\lambda+\theta(v)\big )\big)\big)\\
		&=\big(t(v), t\big(\lambda+\theta(v)\big)\big)\\
		&=t(v,\lambda),
	\end{align*}
	and
	\begin{align*}
		(v,\lambda)^{-1}\cdot (v,\lambda)&= \left(v^{-1}, s_{g^{-1}}^{-1}\big(t(\lambda+\theta(v))\big)\right) \cdot (v,\lambda)\\
		&=\left(v^{-1}\cdot v, s_{g^{-1}g}^{-1}\big(s(\lambda)\big)\right)\\
		&=\big(s(v), s(\lambda)\big)\\
		&=s(v,\lambda). \qedhere
	\end{align*}
\end{proof}
 
\begin{rem}
	\label{rem:stucture_Morita_kernel}
	Let $(L\rightrightarrows L_M;G\rightrightarrows M)$ be an LBG and let $\theta\in \Omega^1(G,L)$ be a multiplicative $1$-form. The RUTH associated with the Morita kernel $\mk_\theta$ is the RUTH discussed in Remark \ref{rem:MK_RUTH_1}. Indeed, first the core of $\mk_\theta$ agrees with $A$, the Lie algebroid of $G$: for any $x\in M$, the source map $s_x\colon L_x\to L_{M,x}$ is an isomorphism, then the kernel of the source $s_x\colon T_xG\oplus L_x\to T_xM\oplus L_{M,x}$ agrees with the kernel of $s_x\colon T_xG\to T_xM$. It follows that the short exact sequence \eqref{eq:core_SES} associated with $\mk_\theta$ is
	\begin{equation*}
		\begin{tikzcd}
			0 \arrow[r] & t^\ast A\arrow[r] & TG\oplus L\arrow[r] & s^\ast (TM\oplus L_M)\arrow[r] &0,
		\end{tikzcd}
	\end{equation*} 
	and, as, for any $g\in G$, the source $s_g\colon L_g\to L_{M,s(g)}$ is an isomorphism, then a right-horizontal lift (Definition \ref{def:right_horizontal_splitting}) on $\mk_{\theta}$ is equivalent to an Ehresmann connection on $G$ (Definition \ref{def:Ehresmann_connection}). In other words, if $h\colon s^\ast TM\to TG$ is an Ehresmann connection then the map $s^\ast (TM\oplus L_M) \to TG\oplus L$, denoted again by $h$, given by
	\[
		h(g,v,\lambda) = \big(h_gv, s_g^{-1}(\lambda)\big), \quad v\in T_{s(g)}M, \, \lambda\in L_{M,s(g)},
	\]
	is a right-horizontal lift on $\mk_\theta$, and any right-horizontal lift on $\mk_\theta$ is of this form.
	
	Choose one for all an Ehresmann connection $h$ on $G$, hence a right-horizontal lift $h$ on $\mk_\theta$. Following Remark \ref{rem:VBG-RUTH}, the structure operators $\{R_k\}_{k\geq 0}$ of the RUTH on the graded VB $A[-1]\oplus(TM\oplus L_M)$ are the following: the $0$-th structure operator is the core-anchor
	\[
		R_0(a)= t(a,0^L_x)= \big(t(a), t(0^L_x+\theta(a))\big)= \big(\rho(a), \ell_{\theta}(a)\big)\in T_xM\oplus L_{M,x}, \quad a\in A_x,
	\]
	with $x\in M$. The $1$-st structure operator is given by
	\begin{align*}
		R_1(g)(a) &= h_g\left(t(a,0^L_{s(g)})\right)\cdot \big(a,0^L_{s(g)}\big)\cdot 0_{g^{-1}}^{TG\oplus L}\\
		&= h_g\big(\rho(a), t\big(\theta(a)\big)\big)\cdot \big(a\cdot 0^{TG}_{g^{-1}}, 0^L_{g^{-1}}\big) \\
		&=\left(h_g \rho(a), s_g^{-1}\big(t(\theta(a))\big)\right)\cdot \big(a\cdot 0^{TG}_{g^{-1}}, 0^L_{g^{-1}}\big)\\
		& = \left(h_g\rho(a)\cdot a\cdot 0^{TG}_{g^{-1}}, 0_{t(g)}^L\right)\\
		& = g_T.a\in A_{t(g)},
	\end{align*}
	for all $a\in A_{s(g)}$, with $g\in G$, and
	\begin{align*}
		R_1(g)(v,\lambda)&= t\left(h_g(v,\lambda)\right)= t\left(h_gv, s_g^{-1}(\lambda)\right)\\
		& = \big(t(h_gv), t(\theta(h_gv))+ t(s_{g}^{-1}(\lambda))\big)\\
		& = (g_T.v, t(\theta(h_gv))+ g.\lambda)\in T_{t(g)}M\oplus L_{M,t(g)},
	\end{align*}
	for all $v\in T_{s(g)}$ and $\lambda\in L_{M,s(g)}$, with $g\in G$. Finally the $2$-nd structure operator is given by
	\begin{align*}
		R_2(g,g')(v,\lambda)&= \left(h_{gg'}(v,\lambda) - h_{g}(t(h_{g'}(v,\lambda)))\cdot h_{g'}(v,\lambda)\right)\cdot 0^{TG\oplus L}_{{(gg')}^{-1}} \\
		&=\left(h_{gg'}(v,\lambda) - h_g(t(h_{g'}v,s_{g'}^{-1}(\lambda)))\cdot h_{g'}(v,\lambda)\right) \cdot 0^{TG\oplus L}_{{(gg')}^{-1}} \\
		&=\left(h_{gg'}(v,\lambda) - h_g(g'_T.v,g'.\lambda)\cdot h_{g'}(v,\lambda)\right) \cdot 0^{TG\oplus L}_{{(gg')}^{-1}} \\
		&=\left((h_{gg'}v,s_{gg'}^{-1}(\lambda)) - (h_g (g'_T.v), s_g^{-1}(g'.\lambda)) \cdot (h_{g'}v, s_{g'}^{-1}(\lambda))\right) \cdot 0^{TG\oplus L}_{{(gg')}^{-1}} \\
		&= \left((h_{gg'}v, s_{gg'}^{-1}(\lambda))- (h_g(g'_T.v)\cdot h_{g'}v, s_{gg'}^{-1}(\lambda))\right)\cdot 0^{TG\oplus L}_{{(gg')}^{-1}} \\
		&=\left(h_{gg'}v - h_g(g'_T.v)\cdot h_{g'}v, 0^L_{gg'}\right) \cdot 0^{TG\oplus L}_{{(gg')}^{-1}} \\
		&=\left(\left(h_{gg'}v - h_g(g'_T.v)\cdot h_{g'}v\right)\cdot 0^{TG}_{(gg')^{-1}}, 0^L_{t(g)}\right)\\
		&= R_2(g,g')v\in A_{t(g)} ,
	\end{align*}
	for all $v\in T_{s(g')}M$ and $\lambda\in L_{M,s(g')}$, with $(g,g')\in G^{(2)}$. So the structure operators agree with the ones in Remark \ref{rem:MK_RUTH_1} and the Morita kernel $\mk_\theta$ is the VBG associated with the RUTH defined in Remark \ref{rem:MK_RUTH_1} as claimed.
\end{rem}

Unlike the $0$-shifted case, we now have a more intrinsic description of the kernel of $\theta$. This allows us to state all the analogous results directly using the Morita kernel and VB-Morita maps, rather than relying on cochain complexes and cochain maps. We start by examining the case when the multiplicative $1$-form $\theta$ is nowhere-zero, to further motivate why $\mk_\theta$ serves as the appropriate replacement for the kernel of $\theta$.
\begin{rem}\label{rem:+1_Mor_ker_regular}
	In the case when $\theta_g\neq 0$ for all $g\in G$, then $K_\theta := \ker \theta \rightrightarrows TM$ is a VB subgroupoid of $TG \rightrightarrows TM$ with core $C = A \cap \ker \theta= \ker \ell_\theta$ (see, e.g., \cite[Lemma 3.6]{CSS15}). Actually, the inclusion $\mathrm{in}\colon K_{\theta} \to \mk_{\theta}$ is a VB-Morita map. Indeed, it is a VBG morphism over the identity, which is a Morita map, and the induced cochain map on fibers over $x\in M$ is
	\begin{equation}
		\label{eq:inclusionMorita}
		\begin{tikzcd}
			0\arrow[r]& C_x\arrow[r, "\rho"] \arrow[d] & T_xM \arrow[d] \arrow[r] &0 \\
			0 \arrow[r] &A_x \arrow[r,"(\rho {,} \ell_\theta)"'] & T_xM \oplus L_{M,x} \arrow[r] &0
		\end{tikzcd},
	\end{equation}
	where the vertical arrows are the inclusions. Clearly \eqref{eq:inclusionMorita} is a quasi-isomorphism for all $x$. Indeed,
	\[
	\ker(\rho,\ell_\theta)= \ker\rho \cap \ker \ell_\theta= \ker(\rho|_C),
	\]
	and we get an isomorphism in the degree $-1$ cohomology. On the other hand, if $v\in T_xM$ is such that there exists $a\in A_x$ with $\rho(a)=v$ and $\theta(a)=0 $, then $v$ is in the image of $\rho|_C$. And, if $(v,\lambda)\in T_xM\oplus L_{M,x}$, then, since $\theta$ is nowhere zero, there exists $a\in A_x$ such that $\theta(a)=\lambda$. The elements $(v,\lambda)$ and $(v-\rho(a), 0)$ are in the same class in the quotient $T_xM\oplus L_{M,x}/\im (\rho,\ell_\theta)$. This discussion proves that the quotient map $T_xM/\im\rho \to T_xM\oplus L_{M,x}/\im (\rho,\ell_\theta)$ is an isomorphism. This motivates replacing $K_\theta$ by $\mk_\theta$ in the general case.
\end{rem}

We conclude by discussing the Morita invariance of the Morita kernel through two propositions. In the first one, we show that the Morita kernel only depends on the $\partial$-cohomology class of $\theta$ in the complex \eqref{eq:complex_valuedform}, up to VB-Morita equivalences. In the second proposition, we establish that the Morita kernels of a multiplicative $1$-form $\theta$ and the pullback $F^\ast\theta$ of $\theta$ along a VB-Morita map $F$ between LBGs are related by a VB-Morita map.
\begin{prop}
	\label{prop:invarianceoftheta}
	Let $\theta$, $\theta' \in \Omega^1(G,L)$ be multiplicative and $\partial$-cohomologous $1$-forms: $\theta - \theta' =\partial \alpha$ for some $\alpha \in \Omega^1(M,L_M)$. Then the map $\mathsf A \colon \mk_\theta \to \mk_{\theta'}$ defined by setting 
	\[
		\mathsf A(v,\lambda)= \big(v, \lambda +s^{\ast}\alpha(v)\big), \quad v\in TG, \, \lambda\in L,
	\] 
	and
	\[
		\mathsf{A} (v,\lambda)= \big(v, \lambda+\alpha(v)\big), \quad v\in TM, \lambda \in L_M,
	\]
	is a VB-Morita map.
\end{prop}
\begin{proof}
	First $\mathsf A$ is a VBG morphism covering the identity $\operatorname{id}_G$. Indeed, for any $(v,\lambda)\in T_gG\oplus L_g$, with $g\in G$, we have
	\begin{align*}
		s\big(\mathsf{A}(v,\lambda)\big)= s\big(v,\lambda+ s^\ast\alpha(v)\big)= \big(s(v), s(\lambda)+ \alpha(s(v))\big)= \mathsf{A}\big(s(v), s(\lambda)\big) = \mathsf{A}\big(s(v,\lambda)\big),
	\end{align*}
	and
	\begin{align*}
		t\big(\mathsf{A}(v,\lambda)\big)&= t\big(v, \lambda+s^\ast\alpha(v)\big) = \big(t(v), t\big(\lambda + s^\ast\alpha(v) + \theta'(v)\big)\big) \\
		&=\big(t(v), t\big(\lambda + t^\ast\alpha(v) + \theta(v)\big)\big) = \big(t(v), t(\lambda + \theta(v)) + \alpha(t(v))\big) \\
		&=\mathsf{A}\big(t(v), t(\lambda+\theta(v))\big)= \mathsf{A}\big(t(v,\lambda)\big).
	\end{align*}
	
	For any composable arrows $(v,\lambda)\in T_gG\oplus L_g$ and $(v',\lambda')\in T_{g'}G\oplus L_{g'}$, we have
	\begin{align*}
		\mathsf{A}(v,\lambda)\cdot \mathsf{A}(v',\lambda')&= \big(v, \lambda+ s^\ast\alpha(v)\big)\cdot \big(v', \lambda'+s^\ast\alpha(v')\big)\\
		&= \left(vv', s_{gg'}^{-1}\big(s(\lambda'+s^\ast\alpha(v'))\big)\right) \\
		&=\left(vv', s_{gg'}^{-1}\big(s(\lambda')\big) +s_{gg'}^{-1}\big(\alpha(s(v'))\big)\right) \\
		&= \left(vv', s_{gg'}^{-1}\big(s(\lambda')\big)+s^\ast\alpha(vv')\right) \\
		&=\mathsf{A}\left(vv', s_{gg'}^{-1}\big(s(\lambda')\big) \right) \\
		&=\mathsf{A}\big((v,\lambda)\cdot (v',\lambda')\big).
	\end{align*}
	
	Finally, for any $(v,\lambda)\in T_xM\oplus L_{M,x}$, with $x\in M$, we have
	\begin{align*}
		u\big(\mathsf{A}(v,\lambda)\big)= \big(u(v), u(\lambda +\alpha(v))\big)&= \big(u(v), u(\lambda)+ s_x^{-1}(\alpha(v))\big)\\ 
		&= \big(u(v), u(\lambda) +s^\ast\alpha (u(v))\big)= \mathsf{A}\big(u(v,\lambda)\big).
	\end{align*}
%	
%	Finally, for any $(v,\lambda)\in T_gG\oplus L_g$, with $g\in G$, we have
%	\begin{align*}
%		\big(\mathsf{A}(v,\lambda)\big)^{-1}&= \big(v, \lambda+s^\ast\alpha(v)\big)^{-1}\\
%		&= \left(v^{-1}, s_{g^{-1}}^{-1}\big(t(\lambda +s^\ast\alpha(v)+\theta'(v))\big)\right)\\
%		&=\left(v^{-1}, s_{g^{-1}}^{-1}\big(t(\lambda+\theta(v)+t^\ast\alpha(v))\big)\right)\\
%		&= \left(v^{-1}, s_{g^{-1}}^{-1}\big(t(\lambda +\theta(v))\big) + s_{g^{-1}}^{-1}\big(\alpha(t(v))\big)\right)\\
%		&= \left(v^{-1}, s_{g^{-1}}^{-1}\big(t(\lambda +\theta(v))\big) +s^\ast\alpha (v^{-1}) \right)\\
%		&= \mathsf{A}\left(v^{-1}, s_{g^{-1}}^{-1}\big(t(\lambda +\theta(v))\big)\right)\\
%		&= \mathsf{A}\left(v,\lambda\right)^{-1}.
%	\end{align*}
	
	The cochain map induced by $\mathsf{A}$ on the fibers over $x \in M$ is
	\begin{equation}\label{eq:Phi_alpha}
		\begin{tikzcd}
			0 \arrow[r] & A_x \arrow[r, "(\rho{,} \ell_\theta)"] \arrow[d, equal] & T_xM\oplus L_{M,x} \arrow[r] \arrow[d, "\mathsf A"] & 0 \\
			0 \arrow[r] & A_x \arrow[r, "(\rho{,} \ell_{\theta'})"'] & T_xM\oplus L_{M,x} \arrow[r]  & 0
		\end{tikzcd}.
	\end{equation} 
	But, from $\theta - \theta' = \partial \alpha$, we get $\ell_{\theta'} = \ell_\theta + \alpha \circ \rho$, then $$\ker (\rho,\ell_\theta)= \ker \rho\cap \ker \ell_\theta= \ker \rho \cap \ker \ell_{\theta'}= \ker(\rho,\ell_{\theta'})$$ and the cochain map \eqref{eq:Phi_alpha} is a quasi-isomorphism in degree $-1$. Moreover, $\mathsf{A}\colon T_xM\oplus L_{M,x}\to T_xM\oplus L_{M,x}$ is a bijection. In particular, the cochain map \eqref{eq:Phi_alpha} is a quasi-isomorphism, and the claim follows from Theorem \ref{theo:caratterizzazioneVBmorita}.
\end{proof}

\begin{prop}\label{prop:+1_Mor_ker_Mor_inv}
	Let $(F,f)\colon (L'\rightrightarrows L'_N; H\rightrightarrows N)\to (L\rightrightarrows L_M; G\rightrightarrows M)$ be a VB-Morita map between LBGs, and let $\theta\in \Omega^1(G,L)$ be a multiplicative $L$-valued $1$-form. Then the map $\mathsf F \colon \mk_{F^{\ast}\theta} \to \mk_{\theta}$ defined by setting $\mathsf F(v,\lambda)= \big(df(v), F(\lambda)\big)$ is a VB-Morita map.
\end{prop}
\begin{proof}
	$\mathsf F$ is a VBG morphism because $F$ and $df$ are so. We discuss this in detail for completeness. For any $v\in T_hH$ and $\lambda\in L'_h$, with $h\in H$, we have
	\[
		s\big(\mathsf{F}(v,\lambda)\big)= s\big(df(v), F(\lambda)\big)= \big(df(s(v)), F(s(\lambda))\big)= \mathsf{F}\big(s(v), s(\lambda)\big)= \mathsf{F}\big(s(v,\lambda)\big),
	\]
	and
	\begin{align*}
		t\big(\mathsf{F}(v,\lambda)\big)&= t\big(df(v), F(\lambda)\big)= \left(t(df(v)), t\big(F(\lambda) + \theta(df(v))\big)\right) \\
		&= \left(df(t(v)), F\left(t\big(\lambda + F_h^{-1}(\theta(df(v)))\big)\right)\right)\\
		&= \left(df(t(v)), F\left(t\big(\lambda + F^\ast\theta(v)\big)\right)\right)\\
		&=\mathsf{F}\left(t(v), t\big(\lambda+ F^\ast\theta(v)\big)\right) =\mathsf{F}\big(t(v,\lambda)\big).
	\end{align*}
	
	For any composable arrows $(v,\lambda)\in T_hH\oplus L'_h$ and $(v',\lambda')\in T_{h'}H\oplus L'_{h'}$, with $(h,h')\in H^{(2)}$, we have
	\begin{align*}
		\mathsf{F}(v,\lambda)\cdot \mathsf{F}(v',\lambda')&= \big(df(v),F(\lambda)\big)\cdot \big(df(v'),F(\lambda')\big)\\ &=\left(df(v)df(v'), s_{f(h)f(h')}^{-1}\big(s(F(\lambda'))\big)\right)\\
		&=\left(df(vv'), F\big(s_{hh'}^{-1}\big(s(\lambda')\big)\big)\right)= \mathsf{F}\left(vv',s_{hh'}^{-1}\big(s(\lambda')\big)\right)\\
		&=\mathsf{F}\big((v,\lambda)\cdot (v',\lambda')\big).
	\end{align*}
	
	Finally, for any $(v,\lambda)\in T_yN\oplus L'_{N,y}$, with $y\in N$, we have
	\begin{align*}
		u\big(\mathsf{F}(v,\lambda)\big)= u\big(df(v), F(\lambda)\big)= \big(df(u(v)), F(u(\lambda))\big)= \mathsf{F}\big(u(v), u(\lambda)\big)= \mathsf{F}\big(u(v,\lambda)\big).
	\end{align*}
%	
%	Finally, for any $(v,\lambda)\in T_hH\oplus L'_h$, we have
%	\begin{align*}
%		\big(\mathsf{F}(v,\lambda)\big)^{-1}&= \big(df(v), F(\lambda)\big)^{-1}\\
%		&= \left(df(v)^{-1}, s_{f(h)^{-1}}^{-1}\big(t(F(\lambda)+\theta(df(v)))\big)\right)\\
%		&=\left(df(v^{-1}), F\big(s_{h^{-1}}^{-1}(t(\lambda + F^\ast\theta(v)))\big)\right)\\
%		&=\mathsf{F}\left(v^{-1},s_{h^{-1}}^{-1}(t(\lambda + F^\ast\theta(v))) \right)\\
%		&=\mathsf{F}\big((v,\lambda)^{-1}\big).
%	\end{align*}
	Hence $\mathsf{F}$ is a VBG morphism covering $f$.
	
	Now denote by $A_G, A_H$ the Lie algebroids of $G,H$. Since, by Theorem \ref{theo:caratterizzazioneVBmorita}, $f$ is a Morita map, then from Example \ref{ex:df_VB_Morita} it follows that, for any $y \in N$, the vertical arrows in
	\begin{equation*}\label{eq:df_quasi_iso}
		\begin{tikzcd}
			0 \arrow[r] & A_{H,y} \arrow[r, "\rho_H"] \arrow[d, "df"']  & T_y N\arrow[r] \arrow[d, "df"] & 0 \\
			0 \arrow[r] & A_{G,f(y)} \arrow[r, "\rho_G"'] & T_{f(y)} M\arrow[r]  & 0
		\end{tikzcd}
	\end{equation*}
	form a quasi-isomorphism. The latter quasi-isomorphism is the diagonal arrows of the top square in the following commutative diagram
	\begin{equation}
		\label{diag:mc}
		{\scriptsize
			\begin{tikzcd}
				0 \arrow[rr] & & A_{H,y} \arrow[rr, "\rho_H" near start] \arrow[dd, "\,\ell_{F^\ast\theta}" near end] \arrow[dr, "df"] & & T_yN \arrow[rr] \arrow[dd] \arrow[dr, "df"] &  & 0 \\
				& 0 \arrow[rr, crossing over] & & A_{G,f(y)} \arrow[rr, crossing over, "\rho_G" near start] & & T_{f(y)}M \arrow[rr, crossing over] \arrow[dd]&  & 0 \\
				0 \arrow[rr] & &L'_{N, y} \arrow[rr] \arrow[dr, "F"] & & 0 \arrow[rr] &  & 0 \\
				& 0 \arrow[rr] & & L_{M,f(y)} \arrow[from=uu, crossing over, "\,\ell_\theta" near end] \arrow[rr] & & 0 \arrow[from=uu, crossing over] \arrow[rr] \arrow[ul, equal]  &  & 0 
		\end{tikzcd}}
	\end{equation}
	As $F$ is a VB-Morita map, by Corollary \ref{coroll:VBMorita}, the diagonal arrows of the lower square in Diagram \eqref{diag:mc} form a quasi-isomorphism. Then, from standard Homological Algebra, the induced cochain map between the mapping cones (of the back and front squares):
	\begin{equation}\label{eq:F_quasi_iso}
		\begin{tikzcd}
			0 \arrow[r] & A_{H,y} \arrow[r, "(\rho_H{,} \ell_{F^\ast \theta})"] \arrow[d, "df"']  & T_y N\oplus L'_{N,y} \arrow[r] \arrow[d, "\, (df{,} F)"] & 0 \\
			0 \arrow[r] & A_{G,f(y)} \arrow[r, "(\rho_G{,} \ell_{\theta})"'] & T_{f(y)} M\oplus L_{M,f(y)} \arrow[r]  & 0
		\end{tikzcd}
	\end{equation}
	is a quasi-isomorphism as well. But the cochain map \eqref{eq:F_quasi_iso} is exactly the cochain map induced by $\mathsf{F}$ between the fiber of $\mk_{F^{\ast}\theta}$ over $y\in N$ and the fiber of $\mk_\theta$ over $f(y)\in M$. The claim now follows from Theorem \ref{theo:caratterizzazioneVBmorita}.
\end{proof}

\subsection{Morita Curvature}\label{sec:Morita_curvature}
Let $\theta \in \Omega^1 (G, L)$ be a multiplicative $1$-form on the LBG $L$. Our next step towards a definition of $+1$-shifted contact structure is defining the ``curvature'' of $\theta$ in a Morita invariant way. As we have a description of the kernel of $\theta$ in terms of a VBG we can define the curvature directly as a VBG morphism from $\mc_\theta$ to its twisted dual VBG (Example \ref{ex:twisted_dual}). 

First, applying the twisted dual VBG construction to $\mk_{\theta}$ we obtain the VBG $\mk_{\theta}^{\dag}:= T^{\dag}G\oplus \mathbbm{R}_G \rightrightarrows A^{\dag}$, where $\mathbbm{R}_G=G\times \mathbbm{R} \to G$ is the trivial line bundle over $G$. The structure maps of $\mk_\theta^\dag$ are explicitly given by:
\begin{itemize}
	\item The source and the target of $(\psi,r)\in T_g^{\dag}G\oplus\mathbbm{R}$, $g \in G$, are given by
	\begin{align*}
		\big\langle s(\psi,r), a \big\rangle &= - s\big\langle \psi ,0_g \cdot a^{-1} \big\rangle - r \theta (a), & a\in A_{s(g)},\\
		\big\langle t(\psi,r), a' \big\rangle &= t \big\langle \psi , a' \cdot 0_g \big\rangle, & a'\in A_{t(g)}.
	\end{align*}
	\item The unit over $\psi\in A_x^{\dag}$, $x \in M$, is given by
	\begin{equation*}
		u(\psi)= (\psi \circ \pr_A, 0),
	\end{equation*}
	where $\pr_A \colon T_{x} G \to A_x$ is the projection with kernel $T_x M$.
	\item The multiplication between two composable arrows $(\psi,r)\in T_g^{\dag}G\oplus \mathbbm{R}$ and $(\psi',r')\in T_{g'}^{\dag}G \oplus \mathbbm{R}$, $(g, g') \in G^{(2)}$, is $(\Psi, r+r') \in T_{gg'}^{\dag}G \oplus\mathbbm{R}$ where $\Psi$ is given by
	\begin{equation*}
		\langle \Psi, vv'\rangle = s_{gg'}^{-1}\Big(g'^{-1} . s \langle \psi, v \rangle + s\big(r \theta (v') + \langle \psi', v' \rangle\big)\Big), \quad (v,v')\in T_{(g, g')}G^{(2)}.
	\end{equation*}
	\item The inverse of $(\psi, r)\in T_g^{\dag}G\oplus \mathbbm{R}$, $g \in G$, is $(\phi, -r)\in T_{g^{-1}}G\oplus\mathbbm{R}$ where
	\begin{equation*}
		\langle \phi, v\rangle = - s_{g^{-1}}^{-1}t\langle \psi, v^{-1} \rangle -r \theta (v), \quad v\in T_{g^{-1}}G.
	\end{equation*}
\end{itemize}
By Example \ref{ex:twisted_dual} again, the core of $\mk_{\theta}^{\dag}$ is $T^{\dag}M\oplus \mathbbm{R}_M$ and the core-anchor is $$ \rho^{\dag}+ \ell_\theta^{\dag}\colon T^{\dag}M\oplus \mathbbm{R}_M \to A^{\dag}.$$

The role of the curvature is played, in this case, by an appropriate VBG morphism $$\mc_\theta \colon \mk_{\theta}\to \mk_{\theta}^{\dag}.$$ Similarly as in the $0$-shifted case, in order to define $\mc_\theta$ we need a connection $\nabla$ on $L_M$. Consider again the $1$-form $\eta_\nabla\in \Omega^1(G)$ given by the difference $s^{\ast}\nabla - t^{\ast}\nabla$ between the pull-back connections on $L$ along the source and the target maps $s, t \colon L \to L_M$.

Consider the $L$-valued $2$-form on $d^{t^\ast\nabla}\theta \in \Omega^2(G,L)$, already discussed in Section \ref{sec:mult_vv}, and define the VB morphisms
\begin{equation*}
	\mc_{\theta}=
	\begin{pmatrix}
		d^{t^{\ast}\nabla}\theta & \eta_{\nabla}\\
		-\eta_{\nabla} & 0
	\end{pmatrix}
	\colon TG\oplus L \to T^\dagger G\oplus \mathbbm{R}_G, \quad 
	\begin{pmatrix}
		v \\
		\lambda
	\end{pmatrix} \mapsto \begin{pmatrix} \iota_v d^{t^{\ast}\nabla}\theta + \lambda \otimes \eta_{\nabla} \\
		-\eta_{\nabla}(v)
	\end{pmatrix},
\end{equation*}
and 
\begin{equation*}
	\mc_{\theta}
	\colon TM \oplus L_M \to A^{\dag}, \quad 
	\begin{pmatrix}
		v \\
		\lambda
	\end{pmatrix} \mapsto  \big(\iota_v d^{t^{\ast}\nabla}\theta\big) |_A + \lambda \otimes \eta_{\nabla}|_A .
\end{equation*}

In the next result we prove that the latter VB morphisms determine a VBG morphism.
\begin{prop/def}\label{prop:1-shift_MC}
	The VB morphisms $\mc_{\theta}$ form a VBG morphism that we call the \emph{Morita curvature of $\theta$}.
\end{prop/def}
\begin{proof}
	The proof is an easy but long computation using Equation \eqref{eq:partial_dtheta}. We explain it in detail for completeness. For any $(v,\lambda)\in T_gG\oplus L_g$ and $a\in A_{s(g)}$, with $g\in G$, we have
	\begin{equation}
		\label{eq:source_MC}
		\begin{aligned}
			\left\langle s\big(\mc_{\theta}(v,\lambda)\big),a \right\rangle&=\left\langle s\big(\iota_v d^{t^{\ast}\nabla}\theta + \lambda\otimes \eta_{\nabla}, -\eta_{\nabla}(v)\big), a\right\rangle\\
			&=-s\big(d^{t^{\ast}\nabla}\theta(v, 0_g^{TG}\cdot a^{-1})\big) -s\big(\eta_{\nabla}(0_g^{TG}\cdot a^{-1})\lambda\big) + \eta_{\nabla}(v)\theta(a)\\
			&=-s\big(d^{t^{\ast}\nabla}\theta(v, 0_g^{TG}\cdot a^{-1})\big) + \eta_{\nabla}(v)\theta(a) +\eta_{\nabla}(a)s(\lambda), 
		\end{aligned}
	\end{equation}
	where we used that $\eta_{\nabla}$ is multiplicative and that, by Remark \ref{rem:eta_nabla_formule}, $\eta_{\nabla}(a^{-1})= -\eta_\nabla(a)$. Applying Equation \eqref{eq:partial_dtheta} to $(v,s(v)),(0_g^{TG}, a^{-1})\in T_{(g,s(g))}G^{(2)}$, we have
	\begin{align*}
		m_{(g,s(g))}^{-1}&\left(d^{t^{\ast}\nabla}\theta\big(v,0_g^{TG}\cdot a^{-1}\big)\right)\\
		&= \pr_{1,(g,s(g))}^{-1}\big(d^{t^{\ast}\nabla} \theta(v, 0_g^{TG})\big) + \pr_{2,(g,s(g))}^{-1}\big(d^{t^{\ast}\nabla} \theta\big(s(v), a^{-1}\big)\big)  \\
		&\quad +\eta_{\nabla}(0_g^{TG})\pr_{2,(g,s(g))}^{-1}\theta(s(v)) - \eta_{\nabla}(v) \pr_{2,(g,s(g))}^{-1}\theta(a^{-1}) \\
		&= \pr_{2,(g,s(g))}^{-1}\big(d^{t^{\ast}\nabla} \theta\big(s(v), a^{-1}\big)\big) - \eta_{\nabla}(v)\pr_{2,(g,s(g))}^{-1}\theta(a^{-1}), 
	\end{align*}
	and so
	\begin{align*}
		s\big(d^{t^{\ast}\nabla}\theta\big(v,0_g^{TG}\cdot a^{-1}\big)\big)&= d^{t^{\ast}\nabla}\theta\big(s(v),a^{-1}\big) - \eta_{\nabla}(v)\theta(a^{-1}) \\
		&= - d^{t^{\ast}\nabla}\theta\big(s(v),a\big) + \eta_{\nabla}(v)\theta(a),
	\end{align*}
	where in the last step we applied Proposition \ref{prop:dnabla_formule}$.ii)$ to $(s(v), a)$ and used $u^\ast\eta_\nabla=0$ (Remark \ref{rem:eta_nabla_formule}) and $u^\ast\theta=0$ (Proposition \ref{prop:vv_formule}$.i)$). Replacing the latter result in \eqref{eq:source_MC} we have
	\begin{align*}
		\left\langle s\big(\mc_{\theta}(v,\lambda)\big),a \right\rangle &=d^{t^{\ast}\nabla}\theta\big(s(v),a\big) - \eta_{\nabla}(v)\theta(a)  + \eta_{\nabla}(v)\theta(a) +\eta_{\nabla}(a)s(\lambda) \\
		&= d^{t^{\ast}\nabla}\theta\big(s(v),a\big)+\eta_{\nabla}(a)s(\lambda) \\
		&=\left\langle \mc_\theta \big(s(v), s(\lambda)\big), a\right\rangle\\
		&= \left\langle \mc_\theta\big(s(v,\lambda)\big), a \right\rangle.
	\end{align*}
	Hence $s\circ \mc_\theta = \mc_\theta\circ s$. 
	
	For any $(v,\lambda)\in T_gG\oplus L_g$ and $a\in A_{t(g)}$, with $g\in G$, we have
	\begin{equation}
		\label{eq:target_MC}
		\begin{aligned}
			\left\langle t\big(\mc_{\theta}(v,\lambda)\big), a\right\rangle &=\left\langle t\big(\iota_v d^{t^{\ast}\nabla}\theta + \lambda\otimes \eta_{\nabla}, -\eta_{\nabla}(v)\big), a\right\rangle\\
			&=t\big(d^{t^{\ast}\nabla}\theta(v,a\cdot 0_g^{TG})\big) + t\big( \eta_{\nabla}(a\cdot 0_g^{TG})\lambda\big) \\
			&=t\big(d^{t^{\ast}\nabla}\theta(v,a\cdot 0_g^{TG})\big) + \eta_{\nabla}(a)t(\lambda),
		\end{aligned}
	\end{equation}
	where we used that $\eta_\nabla$ is multiplicative. Applying Equation \eqref{eq:partial_dtheta} to the pairs $(t(v),v)$, $(a,0_g^{TG})\in T_{(t(g),g)}G^{(2)}$ we have
	\begin{align*}
		m_{(t(g),g)}^{-1}&\left(d^{t^{\ast}\nabla}\theta\big(v,a\cdot 0_g^{TG}\big)\right)\\
		&= \pr_{1,(t(g),g)}^{-1}\left(d^{t^{\ast}\nabla} \theta\big(t(v), a\big)\right) + \pr_{2,(t(g),g)}^{-1}\left(d^{t^{\ast}\nabla} \theta\big(v, 0_g^{TG}\big)\right)  \\
		&\quad +\eta_{\nabla}(a)\pr_{2,(t(g),g)}^{-1}\theta(v) - \eta_{\nabla}(t(v))\pr_{2,(t(g),g)}^{-1}\theta(0_g^{TG}) \\
		&= \pr_{1,(t(g),g)}^{-1}\left(d^{t^{\ast}\nabla} \theta\big(t(v), a\big)\right) + \eta_{\nabla}(a)\pr_{2,(t(g),g)}^{-1}\theta(v),
	\end{align*}
	and so
	\begin{align*}
		t\left(d^{t^{\ast}\nabla}\theta\big(v,a\cdot 0_g^{TG}\big)\right)&=d^{t^{\ast}\nabla} \theta\big(t(v), a\big) + \eta_{\nabla}(a)t\big(\theta(v)\big).
	\end{align*}
	Replacing the latter in \eqref{eq:target_MC} we get
	\begin{align*}
		\left\langle t\big(\mc_\theta(v,\lambda)\big), a\right\rangle &= d^{t^{\ast}\nabla} \theta\big(t(v), a\big) + \eta_{\nabla}(a)t\big(\theta(v)\big) + \eta_{\nabla}(a)t(\lambda) \\
		&= d^{t^{\ast}\nabla} \theta\big(t(v), a\big) + \eta_\nabla(a) t\big(\lambda +\theta(v)\big)\\
		&= \left\langle \mc_\theta\big(t(v), t\big(\lambda + \theta(v)\big)\big), a\right\rangle \\
		&= \left\langle \mc_\theta\big(t(v,\lambda)\big), a\right\rangle.
	\end{align*}
	Hence $t\circ \mc_\theta= \mc_\theta\circ t$.
	
	For any composable arrows $(v,\lambda)\in T_gG\times L_g$, $(v',\lambda')\in T_{g'}G\times L_{g'}$ we have
	\begin{align*}
		\mc_\theta(v,\lambda)&\cdot \mc_{\theta}(v',\lambda')\\
		&= \big(\iota_vd^{t^{\ast}\nabla}\theta + \lambda\otimes\eta_{\nabla}, -\eta_{\nabla}(v) \big) \cdot \big(\iota_{v'}d^{t^{\ast}\nabla}\theta + \lambda' \otimes \eta_{\nabla}, -\eta_{\nabla}(v')\big) \\
		&= \big((\iota_vd^{t^{\ast}\nabla}\theta + \lambda\otimes\eta_{\nabla}) \cdot (\iota_{v'}d^{t^{\ast}\nabla}\theta + \lambda' \otimes \eta_{\nabla}), -\eta_\nabla(vv')\big).
	\end{align*}
	Now, for any $(w,w')\in T_{(g,g')}G^{(2)}$ we have
	\begin{equation}
		\label{eq:multiplication_MC}
	\begin{aligned}
		&\left\langle \big(\iota_vd^{t^{\ast}\nabla}\theta + \lambda\otimes\eta_{\nabla}\big)\cdot \big(\iota_{v'}d^{t^{\ast}\nabla}\theta + \lambda' \otimes \eta_{\nabla}\big), ww'\right\rangle \\
		&\quad=s_{gg'}^{-1}\left(g'^{-1}\cdot s\big(d^{t^{\ast}\nabla}\theta (v,w)+\lambda \eta_{\nabla}(w)\big)\right.\\
		&\quad \quad\left.+ s\big(-\eta_{\nabla}(v)\theta(w') + d^{t^{\ast}\nabla}\theta(v',w') + \lambda'\eta_{\nabla}(w')\big)\right)\\
		&\quad=s_{gg'}^{-1}\left(s\left(t_{g'}^{-1} \big(s\big(d^{t^{\ast}\nabla}\theta(v,w)\big)\big)\right)\right) + \eta_{\nabla}(w) s_{gg'}^{-1}\big(s(\lambda')\big) + \eta_{\nabla}(w) s_{gg'}^{-1}\big(s\big(\theta(v')\big)\big) \\
		&\quad \quad - \eta_{\nabla}(v) s_{gg'}^{-1}\big(s\big(\theta(w')\big)\big) + s_{gg'}^{-1}\big(s\big(d^{t^{\ast}\nabla}\theta(v',w')\big)\big) + \eta_{\nabla}(w') s_{gg'}^{-1}\big(s(\lambda')\big) \\
		&\quad = t_{gg'}^{-1}\big(t\big(d^{t^{\ast}\nabla}\theta(v,w)\big)\big)+ \eta_{\nabla}(w) s_{gg'}^{-1}\big(s(\lambda')\big) + \eta_{\nabla}(w) s_{gg'}^{-1}\big(s\big(\theta(v')\big)\big) \\
		&\quad \quad - \eta_{\nabla}(v) s_{gg'}^{-1}\big(s\big(\theta(w')\big)\big) + s_{gg'}^{-1}\big(s\big(d^{t^{\ast}\nabla}\theta(v',w')\big)\big) + \eta_{\nabla}(w') s_{gg'}^{-1}\big(s(\lambda')\big),
	\end{aligned}
	\end{equation}
	where we used that 
	\[
		g'{}^{-1}.s(\lambda)=g'{}^{-1}. t\big(\lambda' +\theta(v')\big) = s\big(\lambda' +\theta(v')\big),
	\]
	and
	\begin{align*}
		s_{gg'}^{-1}\circ s\circ t_{g'}^{-1}\circ s&= t_{gg'}^{-1}\circ t\circ s_{gg'}^{-1}\circ s\circ t_{g'}^{-1}\circ s\circ t_g^{-1}\circ t\\
		&= t_{gg'}^{-1}\circ gg'. \circ g'{}^{-1}.\circ g^{-1}. \circ t\\
		&= t_{gg'}^{-1}\circ t.
	\end{align*}
	Applying Equation \eqref{eq:partial_dtheta} to $(v,v')$,$(w,w')\in T_{(g,g')}G^{(2)}$ we get
	\begin{align*}
		d^{t^{\ast}\nabla}\theta\big(vv',ww'\big)=&m\left(\pr_{1,(g,g')}^{-1}\big(d^{t^{\ast}\nabla}\theta(v,w)\big)\right) + m\left(\pr_{2,(g,g')}^{-1}\big(d^{t^{\ast}\nabla}\theta(v',w')\big)\right) \\
		&\quad + \eta_{\nabla}(w) m\left(\pr_{2,(g,g')}^{-1}\theta(v')\right) - \eta_{\nabla}(v) m\left(\pr_{2,(g,g')}^{-1}\theta(w')\right),
	\end{align*}
	and, from $m\circ \pr_{1,(g,g')}^{-1}= t_{gg'}^{-1}\circ t$, we have
	\begin{align*}
		t_{gg'}^{-1}\left(t\big(d^{t^{\ast}\nabla}\theta(v,w)\big)\right) &= d^{t^{\ast}\nabla}\theta\big(vv',ww'\big) - s_{gg'}^{-1}\big(s(d^{t^{\ast}\nabla}\theta(v',w'))\big) \\
		& \quad - \eta_\nabla(w) s_{gg'}^{-1}\big(s(\theta(v'))\big) + \eta_{\nabla}(v)s_{gg'}^{-1}\big(s(\theta(w'))\big).
	\end{align*}
	Replacing the latter in \eqref{eq:multiplication_MC} we get
	\begin{align*}
		&\left\langle \big(\iota_vd^{t^{\ast}\nabla}\theta + \lambda\otimes\eta_{\nabla}\big)\cdot \big(\iota_{v'}d^{t^{\ast}\nabla}\theta + \lambda' \otimes \eta_{\nabla}\big), ww'\right\rangle \\ 
		& \quad =d^{t^{\ast}\nabla}\theta\big(vv',ww'\big) +\eta_{\nabla}(w) s_{gg'}^{-1}\big(s(\lambda')\big) + \eta_{\nabla}(w')s_{gg'}^{-1}\big(s(\lambda')\big)\\
		&\quad = d^{t^{\ast}\nabla}\theta\big(vv',ww'\big) + \eta_\nabla(ww')s_{gg'}^{-1}\big(s(\lambda')\big) \\
		&\quad = \left\langle \iota_{vv'} d^{t^\ast\nabla}\theta + s_{gg'}^{-1}\big(s(\lambda')\big)\otimes \eta_{\nabla} , ww'\right \rangle .
	\end{align*}
	Summarizing, we have
	\begin{align*}
		\mc_\theta(v,\lambda)\cdot \mc_\theta(v',\lambda')&= \big(\iota_{vv'} d^{t^\ast\nabla}\theta + s_{gg'}^{-1}\big(s(\lambda')\big)\otimes \eta_{\nabla}, -\eta_\nabla(ww')\big)\\
		&=\mc_{\theta}\big(vv', s_{gg'}^{-1}\big(s(\lambda')\big)\big)\\
		&= \mc_{\theta}\big((v,\lambda)\cdot (v',\lambda')\big).
	\end{align*}
%	Hence $\mc_{\theta}\circ m= m\circ \mc_\theta$.
	
	Finally, for any $(v,\lambda)\in T_xM\oplus L_{M,x}$, with $x\in M$, we have
	\begin{align*}
		\mc_\theta\big(u(v,\lambda)\big)= \mc_\theta\big(u(v), u(\lambda)\big)= \big(\iota_{u(v)}d^{t^\ast\nabla}\theta + u(\lambda)\otimes \eta_\nabla, 0\big).
	\end{align*}
	But, for any $w\in T_xG$, we have
	\begin{align*}
		\left\langle \iota_{u(v)}d^{t^\ast\nabla}\theta + u(\lambda)\otimes \eta_\nabla, w\right\rangle &=d^{t^{\ast}\nabla}\theta\big(u(v),w\big) + \eta_{\nabla}(w)u(\lambda)\\
		&= d^{t^\ast\nabla}\theta\big(u(v), \pr_A w\big) +\eta_\nabla(\pr_A w) u(\lambda)\\
		&= \left\langle \iota_{u(v)}d^{t^\ast\nabla}\theta + u(\lambda)\otimes\eta_\nabla, \pr_A w\right\rangle\\
		&= \left\langle \big(\iota_{u(v)}d^{t^\ast\nabla}\theta + \lambda \otimes \eta_\nabla\big)\circ \pr_A, w\right\rangle.
	\end{align*}
	Summarizing, we have
	\begin{align*}
		\mc_\theta\big(u(v,\lambda)\big)&= \big(\big(\iota_{u(v)}d^{t^\ast\nabla}\theta + \lambda\otimes \eta_\nabla\big)\circ \pr_A, 0\big)\\
		&= u\big(\iota_v d^{t^\ast\nabla}\theta + \lambda\otimes \eta_\nabla\big)\\
		&= u\big(\mc_\theta(v,\lambda)\big).
	\end{align*}
	Hence $\mc_\theta\circ u=u\circ \mc_\theta$, and this concludes the proof.
\end{proof}

\begin{rem}
	Proposition \ref{prop:1-shift_MC} can be also proved using Atiyah forms and Proposition \ref{prop:ker_omega_comp} instead of Equation \eqref{eq:partial_dtheta}. For instance, for the source, take $g \in G$, $a \in A_{s(g)}$, $v\in T_gG$ and $\lambda\in L_g$. Then, using that $\eta_\nabla$ is a multiplicative form, we get
	\begin{align}\label{eq:fdgstw}
		\Big\langle s\big(\mc_{\theta}(v,\lambda)\big), a\Big\rangle &=\left\langle s\big(\iota_v d^{t^{\ast}\nabla}\theta + \lambda\otimes \eta_{\nabla}, -\eta_{\nabla}(v)\big), a\right\rangle \nonumber\\
		&=-s\left(d^{t^{\ast}\nabla}\theta\big(v, 0_g^{TG}\cdot a^{-1}\big)\right) -s\left(\lambda\eta_{\nabla}\big(0_g^{TG}\cdot a^{-1}\big)\right) + \eta_{\nabla}(v)\theta(a) \nonumber\\
		&=-s\left(d^{t^{\ast}\nabla}\theta\big(v, 0_g^{TG}\cdot a^{-1}\big)\right) + \eta_{\nabla}(v)\theta(a) + \eta_{\nabla}(a) s(\lambda). 
	\end{align}
	In order to compute the first summand in \eqref{eq:fdgstw}, consider $\omega \rightleftharpoons (\theta, 0) \in \OA^2 (L)$, and apply Equation \eqref{eq:omegaandcomponents} to the case $\sigma (\delta) = v$ and $\delta' = 0^{DL}_g\cdot a^{-1}$. We get
	\begin{align*}
		-s\left(d^{t^{\ast}\nabla}\theta\big(v, 0_g^{TG}\cdot a^{-1}\big)\right)&= -s\big(\omega(\delta, 0^{DL}_g\cdot a^{-1})\big) + f_{t^{\ast}\nabla}(\delta)s\big(\theta(0^{TG}_g \cdot a^{-1})\big)\\
		&\quad  - f_{t^{\ast}\nabla}\big(0^{DL}_g\cdot a^{-1}\big)s\big(\theta(v)\big).
	\end{align*}
	But $\omega$ is a multiplicative Atiyah form, so
	\begin{align*}
		m_{(g,s(g))}^{-1}\big(\omega(\delta, 0^{DL}_g \cdot a^{-1})\big)&= m_{(g,s(g))}^{-1}\big(\omega(\delta \cdot Ds (\delta)), 0^{DL}_g \cdot a^{-1}\big)\\
		&= \pr_{2,(g,s(g))}^{-1}\big(\omega(Ds(\delta), a^{-1})\big),
	\end{align*}
	whence, using that $s \circ m = s \circ \pr_2$,
	\begin{equation}\label{eq:fsgdthe}
		s\big(\omega(\delta, 0^{DL}_g\cdot a^{-1})\big)= \omega\big(Ds(\delta), a^{-1}\big)= -\omega\big(Ds(\delta),a\big).
	\end{equation}
	From the multiplicativity of $\theta$ we also get
	\begin{equation}\label{eq:ywrdsg}
		s\big(\theta(0^{TG}_g \cdot a^{-1})\big)= \theta(a^{-1})= -\theta(a).
	\end{equation}
	Substituting \eqref{eq:fsgdthe} and \eqref{eq:ywrdsg} in \eqref{eq:fdgstw}, and using that $f_{t^\ast \nabla} = f_\nabla \circ Dt$ (see Remark \ref{rem:f_nabla_fiberwise}), so that $f_{t^{\ast}\nabla}(0^{DL}_g\cdot a^{-1})=0$, we get 
	\begin{align*}
		\Big\langle s\big(\mc_{\theta}(v,\lambda)\big), a\Big\rangle & = \omega\big(Ds(\delta),a\big)- f_{t^{\ast}\nabla}(\delta)\theta(a) + \eta_{\nabla}(v)\theta(a) + \eta_\nabla (a) s(\lambda)\\
		&= \omega\big(Ds(\delta), a\big) - f_{\nabla}\big(Ds(\delta)\big)\theta(a)  + \eta_\nabla (a) s(\lambda),
	\end{align*}
	where, for the last equality, we used that
	\begin{equation*}
		\eta_{\nabla}(v)- f_{t^{\ast}\nabla}(\delta)= (s^{\ast}\nabla)_v-(t^{\ast}\nabla)_v - \delta + (t^{\ast}\nabla)_v= - f_{s^{\ast}\nabla}(\delta)=-f_\nabla\big(Ds(\delta)\big).
	\end{equation*}
	On the other hand, applying Equation \eqref{eq:omegaandcomponents} to the case $\delta \leadsto Ds(\delta)$ and $\delta' \leadsto a\in D_{s(g)}L$ we get
	\begin{align*}
		\Big\langle \mc_{\theta}\big(s(v),s(\lambda)\big), a\Big\rangle & = d^{t^{\ast}\nabla}\theta \big(s(v), a\big) + \eta_{\nabla}(a) s(\lambda) \\
		& = \omega\big(Ds(\delta), a\big) - f_{t^{\ast}\nabla}\big(Ds(\delta)\big)\theta(a)\\
		&\quad + f_{t^{\ast}\nabla}(a)\theta\big(s(v)\big) + \eta_{\nabla}(a) s(\lambda)\\
		& = \omega\big(Ds(\delta), a\big) - f_{\nabla}\big(Ds(\delta)\big)\theta(a) + \eta_{\nabla}(a) s(\lambda) \\
		& = \Big\langle s\big(\mc_{\theta}(v,\lambda)\big), a\Big\rangle,
	\end{align*}
	where we used that $\theta(s(v))=0$ (because of point $i)$ in Proposition \ref{prop:vv_formule}). We conclude that $s\circ \mc_{\theta}= \mc_{\theta} \circ s$, as desired. Compatibility with all other structure maps is similar.
\end{rem}

As the Morita curvature $\mc_{\theta}\colon \mk_{\theta}\to \mk_{\theta}^{\dag}$ is a VBG morphism, it induces a cochain map on fibers:
\begin{equation}\label{eq:cm_fibers_MC_+1}
	\begin{tikzcd}
		0 \arrow[r] & A_x \arrow[r, "(\rho{,} \ell_\theta)"] \arrow[d, "\mc_{\theta}"']& T_xM\oplus L_{M,x} \arrow[r] \arrow[d, "\mc_{\theta}"] & 0 \\
		0 \arrow[r] & T_x^{\dag}M \oplus \mathbbm{R} \arrow[r, "\rho^{\dag}+ \ell_\theta^{\dag}"'] & A_x^{\dag} \arrow[r] & 0
	\end{tikzcd}, \quad x \in M.
\end{equation}
Notice that the map $\mc_{\theta}\colon T_xM\oplus L_{M,x}\to A_x^{\dag}$ is actually the opposite of the twisted transpose of the map $\mc_{\theta}\colon A_x\to T_x^{\dag}M \oplus \mathbbm{R}$. Indeed, the twisted transpose of $\mc_{\theta}\colon A_x\to T^\dagger _xM\oplus\mathbbm{R}$ maps a pair $(v,\lambda)\in T_xM\oplus L_{M,x}$ to a covector $\phi\in A_x^\dagger$ defined as follows: for any $a\in A_x$,
\begin{align*}
	\langle \phi,a\rangle &= \langle \mc_{\theta}(a), (v,\lambda)\rangle = \langle (\iota_a d^{t^\ast\nabla}\theta, -\eta_\nabla(a)), (v,\lambda)\rangle= d^{t^\ast\nabla}\theta(a,v) -\eta_\nabla(a)\lambda\\
	& =- d^{t^\ast\nabla}\theta(v,a) -\eta_\nabla(a)\lambda = \langle \iota_v d^{t^\ast\nabla}\theta - \lambda\otimes \eta_\nabla , a \rangle = \langle - \mc_\theta(v,\lambda), a \rangle,
\end{align*}
so $\phi=-\mc_\theta(v,\lambda)$. In particular, $\mc_\theta\colon A_x\to T_x^\dagger M\oplus \mathbbm{R}$ is an isomorphism in cohomology if and only if so is $\mc_\theta\colon T_xM\oplus L_{M,x}\to A_x^{\dagger}$.

\begin{rem}\label{rem:Spencer}
	A multiplicative VB-valued form on a Lie groupoid determines (and, under appropriate connectedness assumptions, is determined by) certain \emph{infinitesimal data}, the associated \emph{Spencer operator} \cite{CSS15} (see also \cite{DE19} for a more general case). For instance, the Spencer operator of a multiplicative $L$-valued $1$-form $\theta \in \Omega^1 (G, L)$ is the pair $(D_\theta, \ell_\theta)$ where $\ell_\theta \colon A \to L_M$, just as above, is the restriction of $\theta$ to $A$, while $D_\theta \colon \Gamma (A) \to \Omega^1 (M, L_M)$ is the differential operator defined by 
	\begin{equation}\label{eq:D_theta}
		D_\theta a = u^\ast \left( \mathcal L_{\overrightarrow{\Delta^a}} \theta \right), \quad a \in \Gamma (A),
	\end{equation}
	(see \cite{CSS15}, see also \cite[Section 10]{V18} for this precise version of the definition of $D_\theta$). Equation \eqref{eq:D_theta} requires some explanations. Here $\overrightarrow{ \Delta^a} \in \Gamma (DL)$ is the right invariant derivation corresponding to $a$, and $\mathcal L_{\overrightarrow{ \Delta^a}} \theta \in \Omega^1 (G, L)$ denotes the Lie derivative of $\theta$ along $\overrightarrow{ \Delta^a}$:
	\[
	\mathcal L_{\overrightarrow{ \Delta^a}} \theta (X) = \overrightarrow{ \Delta^a} \big( \theta (a) \big) - \theta \big([\sigma(\overrightarrow{ \Delta^a}), X]\big), \quad X \in \mathfrak X (G).
	\]
	Now, it is not hard to see that the map $\mc_\theta \colon A \to T^\dag M \oplus L_M$ takes a section $a \in \Gamma (A)$ to
	\[
	\left(D_\theta a - d^\nabla \ell_\theta (a), F_\nabla (a) \right) \in \Gamma (T^\dag M \oplus L_M).\qedhere
	\]
\end{rem}

%We now prove three ``invariance properties'' of the Morita curvature, explaining in which precise sense the Morita curvature is Morita invariant:
%\begin{enumerate}
%	\item the Morita curvature is independent of the connection $\nabla$ up to linear natural isomorphisms (see Proposition \ref{prop:+1_MC_conn});
%	\item the Morita curvature does only depend on the $\partial$-cohomology class of $\theta$ up to Morita equivalence (see Proposition \ref{prop:+1_MC_cohom_class});
%	\item the Morita curvatures of two $1$-forms related by a Morita equivalence are also related by a Morita equivalence (see Proposition \ref{prop:+1_MC_Mequiv}). 
%\end{enumerate}

In the remainder of this subsection, we prove that the Morita curvature simplifies to the standard curvature when $\theta$ is nowhere zero. Additionally, we establish the Morita invariance of the Morita curvature. 
\begin{rem}
	When $\theta_g \neq 0$ for all $g \in G$, then the Morita kernel can be replaced by the plain kernel $K_\theta$ up to Morita equivalence (see Remark \ref{rem:+1_Mor_ker_regular}). In this case, the standard curvature $R_\theta \colon K_\theta \to K_\theta^\dag$ is a VBG morphism and it is related to the Morita curvature by the following commutative diagram
	\begin{equation*}
		\begin{tikzcd}
			K_{\theta} \arrow[r, "\mathrm{in}"] \arrow[d, "R_{\theta}"'] & \mk_{\theta} \arrow[d, "\mc_\theta"] \\
			K_{\theta}^{\dag} & \mk_{\theta}^{\dag} \arrow[l, "\mathrm{in}^\dag"]
		\end{tikzcd}
	\end{equation*}
	where both the inclusion $\mathrm{in} \colon K_\theta \to \mk_\theta$ and its twisted transpose map are VB-Morita maps (see Remarks \ref{rem:+1_Mor_ker_regular} and \ref{rem:twisteddual_VBMorita_on_iso}). 
\end{rem}

We now address the Morita invariance of the Morita curvature, focusing on proving the following three key facts:
\begin{enumerate}
	\item the Morita curvature is independent of the connection $\nabla$, up to linear natural isomorphisms (see Proposition \ref{prop:+1_MC_conn});
	\item the Morita curvatures of two $1$-forms related by a Morita equivalence are Morita equivalent (see Proposition \ref{prop:+1_MC_Mequiv});
	\item the Morita curvature does only depend on the $\partial$-cohomology class of $\theta$, up to Morita equivalence (see Proposition \ref{prop:+1_MC_cohom_class}). 
\end{enumerate}

We begin with the first property.
\begin{prop}\label{prop:+1_MC_conn}
	Let $\nabla$ and $\nabla'$ be two connections on $L_M$. Then there is a linear natural isomorphism between the corresponding Morita curvatures.
\end{prop}
\begin{proof}
	We use Theorem \ref{theo:VBtransformation}. Denote by $\mc$ and $\mc'$ the Morita curvatures defined through $\nabla$ and $\nabla'$ respectively. The map 		
	\begin{equation*}
		\mathcal{H}=\begin{pmatrix}
			0 & \alpha_{\nabla, \nabla'}\\
			-\alpha_{\nabla, \nabla'} & 0
		\end{pmatrix}
		\colon TM\oplus L_M \to T^{\dag}M\oplus \mathbbm{R}_M
	\end{equation*}
	is a homotopy between the cochain maps determined by $\mc, \mc'$:
	\[
	\begin{tikzcd}
		0 \arrow[r] & A \arrow[d, "\mc'", shift left=0.5ex] \arrow[d, "\mc"', shift right=0.5ex]\arrow[r]& TM \oplus L_M \arrow[d, "\mc'", shift left=0.5ex] \arrow[d,"\mc"', shift right=0.5ex] \arrow[r] \arrow[dl, "\mathcal{H}"'] &0\\
		0 \arrow[r] &T^\dag M \oplus \mathbbm R_M\arrow[r] &A^\dag \arrow[r]&0
	\end{tikzcd}.
	\]
	In the rest of the proof for simplicity we denote $\alpha_{\nabla, \nabla'}$ simply by $\alpha$ and we repeatedly use Equations \eqref{eq:eta_and_alpha} and \eqref{eq:d_two_connections}. In degree $-1$, for any $a\in A_x$ we have 
	\begin{equation}
		\label{eq:alpha_difference}
		\begin{aligned}
			(\mc'-\mc)(a)&= \big(\iota_ad^{t^\ast\nabla'}\theta - \iota_a d^{t^\ast\nabla}\theta , -\eta_{\nabla'}(a) + \eta_\nabla(a)\big) \\
			&= \left(\iota_a\big(d^{t^\ast\nabla'}\theta - d^{t^\ast\nabla}\theta\big) , -\alpha(\rho(a))\right)\\
			&= \left(\iota_a\big(-t^\ast\alpha\wedge \theta\big), -\alpha(\rho(a))\right).
		\end{aligned}
	\end{equation} 
	Moreover, for any $v\in T_xM$ we have
	\begin{align*}
		\Big\langle\iota_a\big(-t^{\ast}\alpha\wedge\theta\big) , v\Big\rangle
		=(-t^{\ast}\alpha\wedge\theta) (a,v)
		=\alpha(v) \theta(a) 
		 = \big\langle\theta(a)\otimes \alpha, v\big\rangle,
	\end{align*}
	and so $\iota_a\big(-t^{\ast}\alpha\wedge\theta\big) = \theta(a)\otimes \alpha$. Replacing the latter equality in \eqref{eq:alpha_difference} we have
	\begin{align*}
		(\mc'-\mc)(a)= \big(\theta(a)\otimes \alpha, -\alpha(\rho(a))\big) = \mathcal{H}\big(\rho(a), \theta(a)\big)=\big( \mathcal{H}\circ(\rho\oplus\ell_\theta)\big)(a).
	\end{align*}
	In degree $0$ for any $(v,\lambda)\in T_xM\oplus L_{M,x}$ we have
	\begin{equation}
		\label{eq:in_degree_0}
		\begin{aligned}
			(\mc'-\mc)(v,\lambda)&= \big(\iota_v d^{t^\ast\nabla}\theta + \lambda\otimes \eta_{\nabla'} -\iota_vd^{t^\ast\nabla}\theta -\lambda\otimes\eta_\nabla, -\eta_{\nabla'}(v)-\eta_\nabla(v)\big)\\
			&= \left(\iota_v\big(d^{t^\ast\nabla'}\theta-d^{t^\ast\nabla}\theta\big) + \lambda\otimes (\eta_{\nabla'}-\eta_\nabla), \alpha(v)-\alpha(v)\right)\\
			&= \left(\iota_v\big(d^{t^\ast\nabla'}\theta-d^{t^\ast\nabla}\theta\big) + \lambda\otimes (-s^\ast\alpha+t^\ast\alpha), 0\right)\\
			&=\left(\iota_v\big(-t^\ast\alpha\wedge \theta\big)+ \lambda\otimes (-s^\ast\alpha+t^\ast\alpha), 0 \right),
		\end{aligned}
	\end{equation}
	Moreover, for any $a\in A_x$, we have
	\begin{align*}
		\big\langle \iota_v\big(-t^\ast\alpha\wedge \theta\big) + \lambda\otimes (-s^\ast\alpha+t^\ast\alpha), a \big\rangle
		&= \big(-t^\ast\alpha\wedge \theta\big)(v,a) -s^\ast\alpha(a)\lambda +t^\ast\alpha(a)\lambda \\
		&=\big(-t^\ast\alpha\wedge \theta\big)(v,a) + \alpha(\rho(a))\lambda\\
		&=-\alpha(v)\theta(a) +\alpha(\rho(a))\lambda\\
		&= \left\langle -\theta^\dagger(\alpha(v)) +\rho^\dagger(\lambda \otimes \alpha), a \right\rangle,
	\end{align*}
	and so $\iota_v\big(-t^\ast\alpha\wedge \theta\big) + \lambda\otimes (-s^\ast\alpha+t^\ast\alpha)= -\theta^\dagger(\alpha(v)) +\rho^\dagger(\lambda\otimes \alpha)$. Replacing the latter equation in \eqref{eq:in_degree_0} we get
	\begin{align*}
		(\mc'-\mc)(v,\lambda)&= -\theta^\dagger(\alpha(v)) +\rho^\dagger(\lambda\otimes \alpha)\\
		&=(\rho^\dagger+ \theta^\dagger) (\mathcal{H}(v,\lambda))= \big((\rho^\dagger + \theta^\dagger) \circ \mathcal{H}\big)(v,\lambda).
	\end{align*}
	
	In order to use Theorem \ref{theo:VBtransformation} now we have to prove Equation \eqref{eq:point_3}. For any $(v,\lambda)\in T_gG\oplus L_g$, with $g\in G$, we have
	\begin{equation}
		\label{eq:3}
		\begin{aligned}
			(&\mc'-\mc)(v,\lambda)\cdot \mathcal{H}(s(v,\lambda))\\
			&= \left(\iota_vd^{t^\ast\nabla'}\theta +\lambda\otimes \eta_{\nabla'}-\iota_v d^{t^\ast\nabla}\theta -\lambda\otimes \eta_\nabla, -\eta_{\nabla'}(v)+\eta_\nabla(v)\right)\cdot \mathcal{H}\big(s(v), s(\lambda)\big) \\
			&=\left(\iota_v\big(d^{t^\ast\nabla'}\theta-d^{t^\ast\nabla}\theta\big) + \lambda\otimes (t^\ast\alpha -s^\ast\alpha), \alpha(s(v))-\alpha(t(v))\right)\\
			&\quad \cdot \big(s(\lambda)\otimes \alpha, -\alpha(s(v))\big) \\
			&= \left(\iota_v\big(-t^\ast\alpha \wedge\theta\big) + \lambda\otimes (t^\ast\alpha -s^\ast\alpha), \alpha(s(v))-\alpha(t(v))\right) \cdot \big(s(\lambda)\otimes \alpha, -\alpha(s(v))\big) \\
			&=\left(\big(\iota_v\big(-t^\ast\alpha \wedge\theta\big) + \lambda\otimes (t^\ast\alpha -s^\ast\alpha)\big)\cdot s(\lambda)\otimes \alpha, -\alpha(t(v))\right).
		\end{aligned}
	\end{equation}
	On the other hand, 
	\begin{equation}
		\label{eq:4}
		\begin{aligned}
			\mathcal{H}(t(v,\lambda))\cdot 0_g^{T^\dagger G\oplus \mathbbm{R}_G}&= \mathcal{H}(t(v), t(\lambda+\theta(v)))\cdot \big(0_g^{T^\dagger G}, 0\big)\\
			&=\big(t(\lambda+\theta(v))\otimes \alpha, -\alpha(t(v))\big)\cdot \big(0_g^{T^\dagger G}, 0\big)\\
			&= \big((t(\lambda+\theta(v))\otimes \alpha)\cdot 0_g^{T^\dagger G}, -\alpha(t(v))\big).
		\end{aligned}
	\end{equation}
	For any $w\in T_gG$, we have
	\begin{equation}
		\label{eq:1}
		\begin{aligned}
			&\left\langle \big(\iota_v\big(-t^\ast\alpha \wedge\theta\big) + \lambda\otimes (t^\ast\alpha -s^\ast\alpha)\big)\cdot s(\lambda)\otimes \alpha, w \cdot s(w)\right\rangle\\
			&\quad = s_g^{-1}\big(s\langle \iota_v\big(-t^\ast\alpha \wedge\theta\big) + \lambda\otimes (t^\ast\alpha -s^\ast\alpha), w \rangle + s \langle s(\lambda)\otimes \alpha, s(w) \rangle \big)\\
			&\quad= s_g^{-1}\big(s\left(-\alpha(t(v))\theta(w) +\alpha(t(w))\theta(v) +\alpha(t(w)) \lambda -\alpha(s(w))\lambda\right)\\
			&\quad \quad + \alpha(s(w)) s(\lambda)\big)\\
			&\quad= -\alpha(t(v))\theta(w)+ \alpha(t(w))\theta(v) +\alpha(t(w))\lambda .
		\end{aligned}
	\end{equation}
	On the other hand,
	\begin{equation}
		\label{eq:2}
		\begin{aligned}
			&\left\langle \big(t(\lambda+\theta(v))\otimes \alpha\big)\cdot 0_g^{T^\dagger 	G}, t(w)\cdot w\right\rangle\\
			&\quad =s_g^{-1}\left(g^{-1}.s\langle t(\lambda+\theta(v))\otimes \alpha, t(w)\rangle + s\big(-\alpha(t(v))\theta(w)\big)\right) \\
			&\quad=s_g^{-1}\left(g^{-1}. (\alpha(t(w))t(\lambda +\theta(v))) - 	\alpha(t(v))s(\theta(w))\right)\\
			&\quad = \alpha(t(w))\lambda + \alpha(t(w))\theta(v) -\alpha(t(v))\theta(v),
		\end{aligned}		
	\end{equation}
	where we used that
	\[
		s_g^{-1}\circ g^{-1}. \circ t= s_g^{-1} \circ s_g\circ t_g^{-1}\circ t_g= \operatorname{id}_{L_g}.
	\]
	Hence, composing \eqref{eq:1} and \eqref{eq:2}, we get $\iota_v\big(-t^\ast\alpha \wedge\theta\big) + \lambda\otimes (t^\ast\alpha -s^\ast\alpha)\big)\cdot s(\lambda)\otimes \alpha = t(\lambda+\theta(v))\otimes \alpha)\cdot 0_g^{T^\dagger G}$, and composing \eqref{eq:3} and \eqref{eq:4} we get $(\mc'-\mc)(v,\lambda)\cdot \mathcal{H}(s(v,\lambda)) = \mathcal{H}(t(v,\lambda))\cdot 0_g^{T^\dagger G\oplus \mathbbm{R}_G}$ as desired, and this concludes the proof. 
\end{proof}

We now prove that the Morita curvature of a multiplicative $1$-form $\theta$ is a VB-Morita map if and only if so is the Morita curvature of the pullback $F^\ast\theta$ of $\theta$ along a VB-Morita map $F$.
\begin{prop}\label{prop:+1_MC_Mequiv}
	Let $(F,f)\colon (L'\rightrightarrows L'_N;H\rightrightarrows N)\to (L\rightrightarrows L_M;G\rightrightarrows M)$ be a VB-Morita map of LBGs, and let $\theta \in \Omega^1(G,L)$ be a multiplicative $L$-valued $1$-form. Moreover, let $\nabla$ be a connection on $L_M$. Denote $\theta' = F^\ast \theta$ and $\nabla' = f^\ast \nabla$. Then the Morita curvature $\mc_\theta$ is a VB-Morita map if and only if $\mc_{\theta'}$ is so.
\end{prop}
\begin{proof}
	The Morita curvatures $\mc_{\theta}$ and $\mc_{\theta'}$ fit in the following diagram
	\begin{equation}
		\label{eq:Cpentagon}
		\begin{tikzcd}
			& \mk_{\theta'} \arrow[rr, "\mathsf{F}"] \arrow[dl, "\mc_{\theta'}"'] & & \mk_\theta\arrow[dr, "\mc_{\theta}"]\\
			\mk_{\theta'}^\dagger &&&& \mk_\theta^\dagger \\
			&& f^\ast \mc_\theta^\dagger \arrow[ull, "\mathsf{F}^\dagger"] \arrow[urr, "\pr_2"']
		\end{tikzcd},
	\end{equation}
	where $\mathsf{F}$ is the VB-Morita map defined in Proposition \ref{prop:+1_Mor_ker_Mor_inv} and $\mk_{\theta'}^\dagger \leftarrow f^\ast\mk_\theta^\ddagger \to \mk_\theta^\dagger$ is the Morita equivalence discussed in Remark \ref{rem:twisted_dualVBG_VB-Morita_equivalent}. We follow the same strategy used in the symplectic and symplectic Atiyah cases: the VBG morphism $\mathsf{F}^\dagger$ is a VB-Morita map covering the identity, then, by Proposition \ref{prop:VB_morita_on_identity}, there exist a VBG morphism $\mathcal{F}\colon \mk_{\theta'}^\dagger \to f^\ast \mk_{\theta}^\dagger$ and two linear natural isomorphisms $T\colon \mathsf{F}^\dagger \circ \mathcal{F} \Rightarrow \operatorname{id}_{\mk_{\theta'}^\dagger}$ and $T'\colon \mathcal{F}\circ \mathsf{F}^\dagger \to \operatorname{id}_{f^\ast\mk_\theta^\dagger}$. In particular, by Theorem \ref{theo:VBtransformation}, $\mathcal{H}'=T'- \mathcal{F}\circ \mathsf{F}^\dagger\colon f^\ast A_G^\dagger\to f^\ast (T^\dagger M\oplus \mathbbm{R}_M)$ is a smooth map covering the identity $\operatorname{id}_N$ that makes $F'=\mathcal{F}\circ \mathsf{F}^\dagger$ homotopic to $\operatorname{id}_{f^\ast \mk_\theta^\dagger}$.
	
	Replacing $\mathsf{F}^\dagger$ with $\mathcal{F}$ in Diagram \eqref{eq:Cpentagon} we get the following diagram:
	\begin{equation*}
		\label{eq:Cpentagon2}
		\begin{tikzcd}
			&\mk_{\theta'} \arrow[rr, "\mathsf{F}"] \arrow[dl, "\mc_{\theta'}"'] & & \mk_\theta\arrow[dr, "\mc_\theta"]\\
			\mk_{\theta'}^\dagger &&&& \mk_{\theta}^\dagger \\
			&& f^\ast \mk_\theta^\dagger \arrow[from=ull, "\mathcal{F}"'] \arrow[urr, "\pr_2"']
		\end{tikzcd},
	\end{equation*}
	which commutes up to a linear natural isomorphism. In order to see this, first notice that both $K'=\pr_2\circ \mathcal{F}\circ \mc_{\theta'}$ and $K= \mc_\theta \circ \mathsf{F}$ are VBG morphisms from $\mk_{\theta'}$ to $\mk_\theta^\dagger$ covering $f\colon H\to G$. Then, in order to apply Theorem \ref{theo:VBtransformation}, we need a VB morphism $\mathcal{H}\colon TN\oplus L'_N\to T^\dagger M\oplus \mathbbm{R}_M$ covering $f$ that makes $K'$ homotopic to $K$.
	
	The composition $\mc_\theta\circ \mathsf{F}\colon TN\oplus L'_N\to A_G^\dagger$ is a VB morphism covering $f\colon N\to M$. Consider the VB morphism from $TN\oplus L'_N$ to $f^\ast A_G^\dagger$ covering the identity $\operatorname{id}_N$, again denoted by $\mc_\theta\circ \mathsf{F}$, defined by setting
	\[
	(\mc_\theta\circ \mathsf{F})(v,\lambda)= \big(y, \mc_\theta(\mathsf{F}(v,\lambda))\big), \quad (v,\lambda)\in T_yN\oplus L'_{N,y},
	\]
	with $y\in N$. Now we set $\mathcal{H}= \pr_2\circ \mathcal{H'}\circ \mc_\theta \circ \mathsf{F}$. Then $K'$ is homotopic to $K$ through $\mathcal{H}$. Indeed, for any $(v,\lambda)\in T_hH\oplus L'_h$, with $h\in H$, we have
	\begin{align*}
		(K'-K)(v,\lambda)&= (\pr_2\circ \mathcal{F} \circ \mc_{\theta'} - \mc_\theta\circ \mathsf{F})(v,\lambda)\\
		&= (\pr_2\circ \mathcal{F} \circ \mathsf{F}^\dagger)(\mc_\theta(\mathsf{F}(v,\lambda)))- \mc_\theta(\mathsf{F}(v,\lambda))\\
		&=(\pr_2\circ J_{\mathcal{H'}}- \pr_2\circ \operatorname{id}_{f^\ast \mk_\theta^\dagger})(\mc_\theta(\mathsf{F}(v,\lambda))) - \mc_\theta(\mathsf{F}(v,\lambda)) \\
		&=\pr_2(J_{\mathcal{H}'}(\mc_{\theta}(\mathsf{F}(v,\lambda))))\\
		&= J_{\mathcal{H}}(v,\lambda).
	\end{align*}
	
	Finally, $\mathcal{F}\colon \mk_{\theta'}^\dagger\to f^\ast \mk_\theta^\dagger$ is a VB-Morita map because of Remark \ref{rem:VB_morita_natural_iso}. Then we have that $\mc_\theta\colon \mk_\theta\to \mk_\theta^\dagger$ is a VB-Morita map if and only if $K$ is so (Lemma \ref{lemma:two-out-of-three-VBG}), if and only if $K'$ is so (Remark \ref{rem:VB_morita_natural_iso}), if and only if $\mc_{\theta'}\colon \mk_{\theta'}\to \mk_{\theta'}^\dagger$ is so (Lemma \ref{lemma:two-out-of-three-VBG} again), whence the claim. 
\end{proof}

Before going on, we need some observations. Up to this point, our discussion has been centered on the kernel and the curvature associated with a multiplicative $1$-form $\theta$. We now extend this framework by considering a pair $(\theta, \kappa)\in \Omega^1(G,L)\oplus\Omega^2(M,L_M)$, rather then just $\theta$. This extension is essential for proving the last Morita invariance property of the Morita curvature, thereby preparing the groundwork for introducing the notion of $+1$-shifted contact structure in the next subsection.
\begin{rem}\label{rem:kappa}
	Let $\theta \in \Omega^1 (G, L)$ be a multiplicative $1$-form. Suppose we are also given a $2$-form $\kappa \in \Omega^2 (M, L_M)$. Define a new VB morphism $\mc_{(\theta, \kappa)} \colon \mk_\theta \to \mk_\theta^\dag$ by 
	\[
	\mc_{(\theta, \kappa)} =
	\begin{pmatrix}
		d^{t^{\ast}\nabla}\theta + \partial \kappa & \eta_{\nabla}\\
		-\eta_{\nabla} & 0
	\end{pmatrix}.
	\]
	It is easy to see that $\mc_{(\theta, \kappa)}$ is a VBG morphism as well. Moreover, there is a linear natural isomorphism $\mc_\theta \Rightarrow \mc_{(\theta, \kappa)}$. This is a straightforward consequence of Theorem \ref{theo:VBtransformation}. Indeed we can consider the VB morphism
	\begin{equation*}
		\mathcal{H}:=
		\begin{pmatrix}
			-\kappa & 0\\
			0& 0
		\end{pmatrix}
		\colon TM\oplus L_M \to T^\dagger M\oplus \mathbbm{R}_M,
	\end{equation*}
	which is a homotopy between the cochain maps induced by $\mc_{(\theta, \kappa)}$ and $\mc_{\theta}$:
	\[
	\begin{tikzcd}
		0 \arrow[r] & A \arrow[d, "\mc_{(\theta, \kappa)}"', shift right=0.5ex] \arrow[d, "\mc_\theta", shift left=0.5ex] \arrow[r]& TM \oplus L_M \arrow[d, "\mc_{(\theta, \kappa)}"', shift right=0.5ex] \arrow[d, "\mc_\theta", shift left=0.5ex]  \arrow[r] \arrow[dl, "\mathcal{H}"'] &0\\
		0 \arrow[r] &T^\dag M \oplus \mathbbm R_M\arrow[r] &A^\dag \arrow[r]&0
	\end{tikzcd}.
	\]
	Indeed,
	\[
	\mc_{(\theta, \kappa)} -\mc_{\theta} =
	\begin{pmatrix}
		\partial \kappa & 0\\
		0& 0
	\end{pmatrix},
	\]
	so, for any $a\in A_x$, with $x\in M$, we have
	\begin{align*}
		\mathcal{H}\big((\rho,\ell_\theta)(a)\big)&= \mathcal{H}\big(\rho(a), \theta(a)\big)= (-\iota_{\rho(a)}\kappa, 0)\\
		& = (\iota_a\partial\kappa,0)= \big(\mc_{(\theta, \kappa)}-\mc_\theta\big)(a),
	\end{align*}
	where we used that, for any $w\in v\in T_xM$,
	\[
	\big\langle \iota_a\partial\kappa, w\big\rangle = \partial\kappa (a,w)= - \kappa(\rho(a), w) = \big\langle-\iota_{\rho(a)}\kappa, w\big\rangle.
	\]
	Moreover, for any $(v,\lambda)\in T_xM\oplus L_{M,x}$,
	\[
	(\rho^\dagger + \theta^\dagger)\big(\mathcal{H}(v,\lambda)\big)= (\rho^\dagger+\theta^\dagger)(-\iota_v\kappa, 0)= -\rho^\dagger(\iota_v\kappa)= \big(\mc_{(\theta, \kappa)}-\mc_\theta\big)(v,\lambda),
	\]
	where we used that, for any $a\in A_x$,
	\[
	\big\langle \iota_v\partial \kappa, a\big\rangle=\partial\kappa(v,a)= - \kappa(v,\rho(a))= \big\langle -\iota_v\kappa, \rho(a)\big\rangle = \big\langle -\rho^\dagger(\iota_v\kappa), a\big\rangle.  
	\]
	Finally Equation \eqref{eq:point_3} is satisfied. Indeed, for any $(v,\lambda)\in T_gG\oplus L_g$, with $g\in G$. we have
	\begin{align*}
	\big(\mc_{(\theta, \kappa)}-\mc_\theta\big)(v,\lambda)\cdot \mathcal{H}\big(s(v,\lambda)\big) &= \big(\iota_v\partial\kappa, 0\big)\cdot \big(-\iota_{s(v)}\kappa,0\big) \\
	&= \big(-\iota_v\partial\kappa \cdot \iota_{s(v)} \kappa, 0\big).
	\end{align*}
	On the other hand
	\[
	\mathcal{H}\big(t(v,\lambda)\big)\cdot 0_g^{T^\dagger G\oplus \mathbbm{R}_G} = \big(-\iota_{t(v)}\kappa, 0\big)\cdot \big(0_g^{T\dagger G}, 0\big)= \big(-\iota_{t(v)}\kappa \cdot 0_g^{T^\dagger G}, 0\big). 
	\]
	But, for any $w\in T_gG$,
	\begin{align*}
		\big\langle \iota_v \partial \kappa \cdot \iota_{s(v)}\kappa, w\cdot s(w)\big\rangle &= s_g^{-1}\left(s\big\langle\iota_v \partial \kappa, w\big\rangle + s\big\langle \iota_{s(v)}\kappa, s(w)\big\rangle \right)\\
		&=\partial\kappa(v,w) + s^\ast\kappa (v,w)\\
		&=-t^\ast\kappa(v,w).
	\end{align*}
	On the other hand
	\begin{align*}
		\big\langle -\iota_{t(v)}\kappa \cdot 0_g^{T^\dagger G}, t(w)\cdot w\big\rangle &= s_g^{-1}\left(g^{-1}. s\big\langle -\iota_{t(v)}, t(w)\big\rangle \right)\\
		&= -t_g^{-1}\big(\kappa\big(t(v), t(w)\big)\big)\\
		&=- t^\ast\kappa(v,w).
	\end{align*}
	It follows that $\mc_\theta$ and $\mc_{(\theta, \kappa)}$ are related by a linear natural isomorphism. In particular, by Remark \ref{rem:VB_morita_natural_iso}, one is VB-Morita if and only if the other one is so. We will refer to $\mc_{(\theta, \kappa)}$ as the \emph{Morita curvature} of the pair $(\theta, \kappa)$.
\end{rem}

\begin{rem}
	\label{rem:analogues_MC}
	The Morita curvature $\mc_{(\theta, \kappa)}$ of the pair $(\theta,\kappa)\in \Omega^1(G,L)\oplus\Omega^2(M,L_M)$, where $\theta$ is multiplicative, satisfies the obvious analogues of Propositions \ref{prop:+1_MC_conn} and \ref{prop:+1_MC_Mequiv}. Indeed, if $\nabla$ and $\nabla'$ are two connections on $L_M$ then there is a linear natural isomorphism between the corresponding Morita curvatures of the pair $(\theta, \kappa)$. The latter is induced by exactly the same VB morphism $\mathcal{H}$ discussed in the proof of Proposition \ref{prop:+1_MC_conn}. Moreover, the analogue of Proposition \ref{prop:+1_MC_Mequiv} follows from Remark \ref{rem:kappa} and Proposition \ref{prop:+1_MC_Mequiv} itself.
\end{rem}

%The next two propositions are straightforward.
Finally, we prove that the Morita curvatures of two different $\partial$-cohomologous multiplicative $1$-forms are related by VB-Morita maps proving the last property of the Morita curvature.
\begin{prop}\label{prop:+1_MC_cohom_class}
	Let $\theta, \theta' \in \Omega^1(G,L)$ be $\partial$-cohomologous multiplicative $1$-forms: $\theta -\theta' =\partial \alpha$, for some $\alpha\in \Omega^1(M,L_M)$, and let $\nabla$ be a connection on $L_M$. Then the Morita curvatures of $\theta\in \Omega^1(G,L)$ and $(\theta',d^\nabla\alpha)\in \Omega^1(G,L)\oplus\Omega^2(M,L_M)$ fit in the following commutative square
	\begin{equation*}
		\begin{tikzcd}
			\mk_{\theta} \arrow[r, "\mathsf A"] \arrow[d, "\mc_{\theta}"']& \mk_{\theta'}\arrow[d, "\mc_{(\theta'{,}d^\nabla\alpha)}"] \\
			\mk^{\dag}_{\theta} & \mk^{\dag}_{\theta'} \arrow[l, "\mathsf A^{\dag}"] 
		\end{tikzcd},
	\end{equation*}
	where $\mathsf A$ is the VB-Morita map defined in Proposition \ref{prop:invarianceoftheta} and $\mathsf{A}^\dagger$ is the twisted dual map of $\mathsf{A}$ discussed in Remark \ref{rem:twisted_dualF_on_iso}.
\end{prop}
\begin{proof}
	First, for any $(\phi,r)\in T_g^\dagger G\oplus\mathbbm{R}$ and $(v,\lambda)\in T_gG\oplus L_g$, with $g\in G$,
	\begin{align*}
		\left\langle \mathsf A^\dagger(\phi,r), (v,\lambda)\right\rangle& = \left\langle (\phi,r), \mathsf A(v,\lambda)\right\rangle = \left\langle (\phi,r), \big(v, \lambda +s^\ast\alpha(v)\big)\right\rangle \\
		&=\phi(v) +r\lambda +r s^\ast\alpha(v) = \left\langle \big(\phi+rs^\ast\alpha, r\big), (v,\lambda)\right\rangle, 
	\end{align*}
	and so $\mathsf A^\dagger(\phi,r)= (\phi+rs^\ast\alpha, r)$.
	
	Next, for any $(v,\lambda)\in T_gG\oplus L_g$, with $g\in G$,
	\begin{align*}
		\mathsf{A}^\dagger &\left(\mc_{(\theta', d^\nabla\alpha)}\big(\mathsf{A}(v,\lambda)\big)\right)\\
		&= \mathsf{A}^\dagger\left(\mc_{(\theta', d^\nabla\alpha)}\big(v, \lambda +s^\ast \alpha(v)\big)\right)\\
		&=\mathsf{A}^\dagger\left(\iota_v\big(d^{t^\ast\nabla}\theta'+ \partial(d^\nabla\alpha)\big) + \big(\lambda + s^\ast\alpha(v)\big)\otimes \eta_\nabla, - \eta_\nabla(v)\right)\\
		&=\left(\iota_v\big(d^{t^\ast\nabla}\theta'+ \partial(d^\nabla\alpha)\big) + \big(\lambda + s^\ast\alpha(v)\big)\otimes \eta_\nabla -\eta_\nabla(v)s^\ast\alpha, -\eta_\nabla(v)\right).
		%&= (\iota_vd^{t^\ast\nabla}\theta -\iota_vd^{t^\ast\nabla}\partial \alpha +\lambda \otimes \eta_\nabla +s^\ast\alpha(v)\otimes \eta_\nabla-\eta_\nabla(v)s^\ast\alpha . -\eta_\nabla(v)).
	\end{align*}
	Now, for any $w\in T_gG$, we have
	\begin{align*}
		&\left\langle \iota_v\big(d^{t^\ast\nabla}\theta'+ \partial(d^\nabla\alpha)\big) + \big(\lambda + s^\ast\alpha(v)\big)\otimes \eta_\nabla -\eta_\nabla(v)s^\ast\alpha , w\right\rangle \\
		&\quad= d^{t^\ast\nabla}\theta'(v,w) + \partial(d^\nabla\alpha)(v,w) +\eta_\nabla(w)\lambda + \eta_\nabla(w) s^\ast\alpha (v) -\eta_{\nabla}(v)s^\ast\alpha(w)\\
		&\quad=d^{t^\ast\nabla}\theta(v,w) -d^{t^\ast\nabla}\partial\alpha (v,w) +\partial(d^\nabla\alpha)(v,w) +\eta_\nabla(w)\lambda -(\eta_\nabla\wedge s^\ast\alpha) (v,w)\\
		&\quad=d^{t^\ast\nabla}\theta(v,w) +\eta_\nabla(w)\lambda\\
		&\quad=\left\langle \iota_v \big(d^{t^\ast\nabla}\theta\big) + \lambda \otimes \eta_\nabla , w\right\rangle,
	\end{align*}
	where we used that
	\begin{align*}
		d^{t^\ast\nabla}\partial\alpha +\eta_\nabla\wedge s^\ast \alpha&= d^{t^\ast\nabla} s^\ast\alpha -d^{t^\ast\nabla}t^{\ast}\alpha + (s^\ast\nabla -t^\ast\nabla)\wedge s^\ast\alpha \\
		&=d^{t^\ast\nabla} s^\ast\alpha -d^{t^\ast\nabla}t^{\ast}\alpha + d^{s^\ast\nabla} s^\ast \alpha - d^{t^\ast \nabla}s^\ast\alpha\\
		&= s^\ast(d^\nabla \alpha)- t^\ast(d^\nabla\alpha)\\
		&= \partial (d^\nabla \alpha).
	\end{align*}
	Summarizing, we have that 
	\begin{equation*}
		\mathsf{A}^\dagger \left(\mc_{(\theta', d^\nabla\alpha)}\big(\mathsf{A}(v,\lambda)\big)\right)= \left(\iota_v \big(d^{t^\ast\nabla}\theta\big) + \lambda \otimes \eta_\nabla , -\eta_{\nabla}(v)\right) = \mc_\theta(v,\lambda).
	\end{equation*}
	We conclude noting that, by Remark \ref{rem:kappa}, the Morita curvatures $\mc_{(\theta', d^\nabla\alpha)}$ and $\mc_{\theta'}$ are related by a linear natural isomorphism. Then the Morita curvature $\mc_\theta$ of $\theta$ is a VB-Morita map if and only if so is the Morita curvature $\mc_{\theta'}$of $\theta'$.
\end{proof}

\subsection{Definition and Examples}\label{sec:+1-def_exmpl}

In this final subsection, we introduce the ultimate definition of $+1$-shifted contact structures, examine their Morita invariance, and establish their equivalence with $+1$-shifted symplectic Atiyah structures. In particular, we define an appropriate notion of \emph{contact Morita equivalence}, which serves as the framework for analyzing the Morita invariance of $+1$-shifted contact structures. We establish that contact Morita equivalence is a well-defined equivalence relation, and we illustrate the relation between this equivalence and the VB-Morita one. To support the definition of $+1$-shifted contact structure, we provide two illustrative examples.

Let $(L \rightrightarrows L_M; G \rightrightarrows M)$ be an LBG.
\begin{definition}
	\label{def:+1_shift_LBG}
	A\emph{$+1$-shifted contact structure} on $L$ is a pair $(\theta, \kappa) \in \Omega^1(G,L)\oplus \Omega^2(M,L_M)$, such that $\theta$ is multiplicative, and its Morita curvature $\mc_{(\theta, \kappa)}$ is a VB-Morita map. A \emph{$+1$-shifted contact LBG} $(L,\theta, \kappa)$ is an LBG $L$ equipped with a $+1$-shifted contact structure $(\theta, \kappa)$.
\end{definition}

%\begin{rem}\label{rem:+1_shift_LBG}
%	We can use another approach defining a $+1$-shifted contact structure on an LBG $L$ more similar to the symplectic one. In other words we can define a \emph{$+1$-shifted contact structure} on $L$ as a pair $(\theta, \kappa)$ consisting of an $L$-valued multiplicative $1$-form $\theta\in \Omega^1(G,L)$ and an $L_M$-valued $2$-form $\kappa \in \Omega^2 (M, L_M)$ such that the Morita curvature $\mc_{(\theta, \kappa)}\colon \mk_{\theta}\to \mk_{\theta}^{\dag}$ is a VB-Morita map. But, the $2$-form $\kappa$ does not play any role. Indeed, there are no compatibility conditions with the LBG structure and, from Remark \ref{rem:kappa}, it follows that the Morita curvature $\mc_{(\theta, \kappa)}$ of $(\theta,\kappa)$ is a VB-Morita map if and only if the Morita curvature $\mc_\theta$ of $\theta$ is a VB-Morita map. For this reason we prefer do not include $\kappa$ in the definition of $+1$-shifted contact structure.	  
%\end{rem}

The next goal is proving that the notion of $+1$-shifted contact structure is Morita invariant. In order to see this we start introducing \emph{gauge transformations} of these $+1$-shifted contact structures: let $(L, \theta, \kappa)$ be a $+1$-shifted contact LBG.
\begin{definition}
	\label{def:gauge_1_cont}
	The \emph{gauge transformation} of $(\theta, \kappa)$ by a pair $(\beta,\gamma) \in \Omega^1(M,L_M)\oplus \Omega^2(M,L_M)$ is the pair $(\theta+\partial\beta, \kappa +\gamma)\in \Omega^1(G,L)\oplus \Omega^2(M,L_M)$.
\end{definition}

\begin{prop}
	\label{prop:gauge_1_cont}
	Let $(L,\theta, \kappa)$ be a $+1$-shifted contact LBG. The gauge transformation of $(\theta,\kappa)$ by a pair $(\beta,\gamma)\in \Omega^1(M,L_M)\oplus\Omega^2(M,L_M)$ is a $+1$-shifted contact structure on $L$.
\end{prop}
\begin{proof}
	We set $\theta'= \theta +\partial \beta$ and $\kappa'=\kappa+\gamma$. Since $\theta$ is multiplicative, then $\theta'$ is multiplicative as well. Moreover, by Remark \ref{rem:analogues_MC}, the Morita curvature of $(\theta,\kappa)$ ia a VB-Morita map if and only if the Morita curvature of $(\theta',\kappa')$ is so and the statement is proved.
\end{proof}

\begin{rem}
	\label{rem:gauge_transf}
	Any $+1$-shifted contact structure $(\theta, \kappa)\in \Omega^1(G,L)\oplus\Omega^2(M,L_M)$ can be gauge transformed into the $+1$-shifted contact structure $(\theta,0)$. It is enough to consider the gauge transformation by $(0,-\kappa)$.
\end{rem}
%\begin{rem}\label{rem:gauge_equiv_cont}
%	If we define a $+1$-shifted contact structure as a pair $(\theta,\kappa)$ as discussed in Remark \ref{rem:+1_shift_LBG}, then Definition \ref{def:gauge_1_cont} and Proposition \ref{prop:gauge_1_cont} should be adapted. We define the \emph{gauge transformation} of $(\theta, \kappa)$ by a pair $(\beta, \gamma ) \in \Omega^1 (M, L_M) \oplus \Omega^2 (M, L_M)$ as the pair $(\theta + \partial \beta, \kappa + \gamma)$.
%	The latter is a $+1$-shifted contact structure. Indeed, $\theta+\partial\beta$ is multiplicative. Moreover, since the Morita curvature $\mc_{(\theta, \kappa)}$ of $(\theta,\kappa)$ is a VB-Morita map, then, by Remark \ref{rem:kappa}, the Morita curvature $\mc_\theta$ of $\theta$ is a VB-Morita map as well. Hence, by Proposition \ref{prop:+1_MC_cohom_class}, the Morita curvature $\mc_{\theta+ \partial\beta}$ of $\theta+\partial\beta$ is a VB-Morita map and, by Rearmk \ref{rem:kappa} again, the Morita curvature $\mc_{(\theta+ \partial\beta, \kappa+\gamma)}$ of $(\theta+\partial \beta, \kappa+\gamma)$ is a VB-Morita map as well.
%\end{rem}

The notion of $+1$-shifted contact structure is Morita invariant in an appropriate sense. First of all we have the following

\begin{theo}\label{theor:Mor_inv_+1-shifted_cont}
	Let $(F, f)\colon (L' \rightrightarrows L'_N; H \rightrightarrows N) \to (L \rightrightarrows L_M; G \rightrightarrows M)$ be a VB-Morita map, let $\theta \in \Omega^1 (G, L)$ be a multiplicative $L$-valued $1$-form and $\kappa\in \Omega^2(M,L_M)$ an $L_M$-valued $2$-form. Then the pair $(F^\ast \theta, F^\ast\kappa)$ is a $+1$-shifted contact structure if and only if so is $(\theta, \kappa)$.
\end{theo}

\begin{proof}
	It is a consequence of Proposition \ref{prop:+1_MC_Mequiv}.
\end{proof}

%\begin{comment}
As we did for the symplectic and symplectic Atiyah cases (Propositions \ref{prop:+1-shifted_Morita_map} and \ref{prop:+1-shifted_Atiyah_Morita_map}), we can prove Theorem \ref{theor:Mor_inv_+1-shifted_cont} without using linear natural isomorphism (and so Proposition \ref{prop:+1_MC_Mequiv}), but using Theorem \ref{theo:caratterizzazioneVBmorita} in a similar way as we did for the $0$-shifted case (Theorem \ref{theo:Mor_inv_0-shift}). We explain it for completeness. First we need the following

%
%%	\begin{theo}\label{theor:Mor_inv_+1-shifted_cont}
%	%	Let $(F, f)\colon (L' \rightrightarrows L'_N; H \rightrightarrows N) \to (L \rightrightarrows L_M; G \rightrightarrows M)$ be a VB-Morita map. Then the assignment $\theta \mapsto F^\ast \theta$ establishes a bijection between $\partial$-cohomology classes $[\theta]$ of $+1$-shifted contact structures $\theta$ on $L$ and on $L'$.
%	%	\end{theo}
%
%%	
%%		\begin{proof}
%	%		As we already mentioned, from \cite[Theorem 8.8]{Drummond:DifferentialformsinVBG}, the assignment $\theta \mapsto F^\ast \theta$ establishes a bijection between $\partial$-cohomology classes of multiplicative $L$-valued and $L'$-valued $1$-forms. It remains to prove that $\theta$ is a $+1$-shifted contact structure if and only if so is $F^\ast \theta$. This follows from Proposition \ref{prop:+1_MC_Mequiv}, Lemma \ref{lemma:Moritacurvatureorbit} below, and the essential surjectivity of $f$.
%	%		\end{proof}

\begin{lemma}
	\label{lemma:Moritacurvatureorbit}
	Let $\theta\in \Omega^1(G,L)$ be a multiplicative $L$-valued $1$-form. If $\mc_{\theta}$ is a quasi-isomorphism on fibers over the point $x\in M$, then it is a quasi-isomorphism on fibers over every point in the orbit of $x$.
\end{lemma}
\begin{proof}
	Choose an Ehresmann connection $h\colon s^\ast TM\to TG$ on $G$ (see Definition \ref{def:Ehresmann_connection}). We already discussed in Remark \ref{rem:stucture_Morita_kernel} that a right horizontal lift of the sequence
	\begin{equation*}
		\begin{tikzcd}
			0 \arrow[r] & t^\ast A \arrow[r]& MK_\theta \arrow[r] & s^\ast \big(TM \oplus L_M\big) \arrow[r]  & 0 
		\end{tikzcd},
	\end{equation*}
	associated to the Morita kernel, is determined by $h$. Moreover, from the last part in Example \ref{ex:twisted_dual}, a right horizontal lift of the sequence \newpage
	\begin{equation*}
		\begin{tikzcd}
			0 \arrow[r] & t^\ast \big(T^\dag M \oplus \mathbbm R_M\big) \arrow[r]& MK^\dag_\theta \arrow[r] & s^\ast A^\dag \arrow[r]   & 0 
		\end{tikzcd}
	\end{equation*}
	associated to $\mk_{\theta}^\dagger$ is also determined by $h$.
	
	Denote by $\{R_k\}_{k \geq 0}, \{R^\dag_k\}_{k \geq 0}$ the structure operators of the $2$-term RUTHs from $MK_\theta, MK_\theta^\dag$ respectively. By Remark \ref{rem:quasi_actions_quis}, for any $g\colon x\to y\in G$, the quasi action $R_1(g)$ and $R_1^\dagger(g^{-1})$ are quasi isomorphisms. Moreover, they fit in the following diagram
	\begin{equation}
		\label{eq:orbit_mc}
		\resizebox{\textwidth}{!}{
		\begin{tikzcd}[ampersand replacement=\&]
			0 \arrow[rr] \& \& A_x \arrow[rr] \arrow[dd] \& \& T_xM\oplus L_{M,x} \arrow[rr] \arrow[dd] \& \& 0 \\
			\& 0 \arrow[rr, crossing over] \& \& A_y \arrow[rr, crossing over] \arrow[from=ul, "R_1(g)"] \& \& T_yM \oplus L_{M,y} \arrow[rr] \arrow[from=ul, "R_1(g)"]\& \& 0 \\
			0 \arrow[rr] \& \& T_x^{\dagger}M \oplus \mathbbm{R} \arrow[rr] \arrow[from=dr, "R^\dag_1(g^{-1})"'] \& \& A_x^{\dagger} \arrow[rr] \arrow[from=dr, "R^\dag_1(g^{-1})"'] \& \& 0 \\
			\& 0 \arrow[rr] \& \& T_y^{\dagger}M \oplus \mathbbm{R} \arrow[rr] \arrow[from=uu, crossing over]  \& \& A_y^{\dagger} \arrow[rr] \arrow[from=uu, crossing over] \& \& 0
	\end{tikzcd}},
	\end{equation}
	where the vertical arrows are the Morita curvature at $x$ and $y$. Diagram \eqref{eq:orbit_mc} does not commute. But it commutes in cohomology. Indeed, for any $a\in A_x$ and $(v,\lambda)\in T_xM\oplus L_{M,x}$, we have
%	\begin{align*}
%		\left\langle R_1^\dagger (g)\big(\mc_{\theta}(a)\big),(v,\lambda)\right\rangle &= \left\langle R_1^\dagger (g)\big(\iota_ad^{t^\ast\nabla}\theta , -\eta_\nabla(a)\big),(v,\lambda)\right\rangle\\
%		&=g.\left\langle \big(\iota_ad^{t^\ast\nabla}\theta , -\eta_\nabla(a)\big), R_1(g^{-1})(v,\lambda)\right\rangle \\
%		&=g.\left\langle \big(\iota_ad^{t^\ast\nabla}\theta , -\eta_\nabla(a)\big), (g^{-1}_T.v,g^{-1}.\lambda)\right\rangle \\
%		&= g.\left(d^{t^\ast\nabla}\theta(a,g^{-1}_T.v)\right) -\eta_\nabla(a )\lambda.
%	\end{align*}
%	On the other hand, we have
	\begin{equation}
		\label{eq:ultima}
		\begin{aligned}
			&\left\langle R_1^\dagger (g^{-1})\left(\mc_\theta\big(R_1(g)(a)\big), (v,\lambda)\right) \right\rangle\\
			&\quad = \left\langle R_1^\dagger (g^{-1})\big(\iota_{g_T.a} d^{t^\ast\nabla}\theta, -\eta_\nabla(g_T.a)\big) , (v,\lambda) \right\rangle\\
			&\quad = g^{-1}.\left\langle(\iota_{g_T.a} d^{t^\ast\nabla}\theta, -\eta_\nabla(g_T.a)\big), \big(g_T.v,t(\theta(h_gv))+g.\lambda\big) \right\rangle\\
			&\quad =g^{-1}.d^{t^\ast\nabla}\theta(g_T.a,g. v) - \eta_\nabla(g_T.a)s(\theta(h_gv)) - \eta_\nabla(g_T.a)\lambda.
		\end{aligned}
	\end{equation}
 %Moreover, let $\omega\rightleftharpoons (\theta,0)\in \OA^2(L)$ and let $\delta\in D_xL_M$ be such that $\sigma(\delta)=v$. Then by Equation \eqref{eq:omegaandcomponents} we have
%	\begin{equation*}
%		d^{t^\ast\nabla}\theta(a, v)= \omega(a, \delta) + f_\nabla(\delta)\theta(a),
%	\end{equation*}
%	and
%	\begin{equation*}
%		g^{-1}.\big(d^{t^\ast\nabla}\theta(g_T.a, g_T.v)\big)= g^{-1}.\omega(g_D.a, g_D.\delta) + f_\nabla(g_D.\delta)g^{-1}.\theta(g_T.a),
%	\end{equation*}
%	where, in both equations, we used that $u^\ast\theta=0$. By Equation \eqref{eq:Atiyah_omega_adjointRUTH}, we get
%	\begin{align*}
%		\omega(g_D.a, g_D.\delta)= t\left(\omega\left(h^D_g(\mathcal{D}_a), h^D_g(\delta)\right)\right)+ g. \omega(a,\delta)
%	\end{align*}
%	
	Applying Equation \eqref{eq:partial_dtheta} to $(h_g(\rho(a))\cdot a,0_{g^{-1}}^{TG}), (h_gv, (h_gv)^{-1})\in T_{(g,g^{-1})}G^{(2)}$, we have
	\begin{align*}
		&d^{t^\ast\nabla}\theta(g_T.a, g_T.v)\\
		&\quad= m\left(\pr_{1,(g,g^{-1})}^{-1} \big(d^{t^\ast\nabla}\theta(h_g(\rho(a))\cdot a, h_gv)\big) \right.\\
		&\quad \quad\left. - \eta_\nabla(h_g(\rho(a))\cdot a)\pr_{2,(g,g^{-1})}^{-1}\theta\big(\big(h_gv\big)^{-1}\big)\right)\\
		&\quad = t\left(d^{t^\ast\nabla}\theta(h_g(\rho(a))\cdot a, h_gv)\right) - \eta_\nabla(h_g(\rho(a))\cdot a) s\left(\theta\big((h_gv)^{-1}\big)\right)\\
		&\quad= t\left(d^{t^\ast\nabla}\theta(h_g(\rho(a))\cdot a, h_gv)\right) + \eta_\nabla(h_g(\rho(a)))t\left(\theta\big(h_gv\big)\right) + \eta_\nabla(a)t\big(\theta\big(h_gv\big)\big).
	\end{align*}
	But, applying Equation \eqref{eq:partial_dtheta} to $(\rho(a),a), (h_g(v), v)\in T_{(g,x)}G^{(2)}$, we get
	\begin{align*}
		&d^{t^\ast\nabla}\theta(h_g(\rho(a))\cdot a, h_gv)\\
		&\quad = m\left(\pr_{1,(g,x)}^{-1}\big(d^{t^\ast\nabla}\theta(h_g(\rho(a)),h_g(v))\big)\right. \\
		&\quad \quad \left. + \pr_{2,(g,x)}^{-1}\big(d^{t^\ast\nabla}\theta(a,v)\big) + \eta_\nabla(h_gv)\pr_{2,(g,x)}^{-1}\theta(a)\right)\\
		&\quad=d^{t^\ast\nabla}\theta(h_g(\rho(a)),h_g(v)) + s_g^{-1}\left(d^{t^\ast\nabla}\theta(a,v)\right) + \eta_\nabla(h_gv)s_g^{-1}\big(\theta(a)\big).
	\end{align*}
	Summarizing we have that
	\begin{align*}
		d^{t^\ast\nabla}\theta(g_T.a, g_T.v) &= g.d^{t^\ast\nabla}\theta(a,v) + t\left(d^{t^\ast\nabla}\theta(h_g(\rho(a)),h_g(v))\right) + \eta_\nabla(h_gv)g.\theta(a)\\
		&\quad + \eta_\nabla(h_g(\rho(a)))t\big(\theta(h_gv)\big) + \eta_\nabla(a) t\big(\theta(h_gv)\big),
	\end{align*}
	and so
	\begin{align*}
		g^{-1}.d^{t^\ast\nabla}\theta(g_T.a, g_T.v) &=d^{t^\ast\nabla}\theta(a,v) + s\left(d^{t^\ast\nabla}\theta(h_g(\rho(a)),h_g(v))\right) + \eta_\nabla(h_gv)\theta(a)\\
		&\quad + \eta_\nabla(h_g(\rho(a)))s\big(\theta(h_gv)\big) + \eta_\nabla(a) s\big(\theta(h_gv)\big).
	\end{align*}
	When $a\in \ker \rho\cap \ker \theta$ we get
	\begin{align*}
		g^{-1}.d^{t^\ast\nabla}\theta(g_T.a, g_T.v) = d^{t^\ast\nabla}\theta(a,v)+ \eta_\nabla(a) s\big(\theta(h_gv)\big).
	\end{align*}
	Replacing the latter equation in \eqref{eq:ultima} we have that, when $a\in \ker \rho\cap \ker \theta$,
	\begin{align*}
		\left\langle R_1^\dagger (g^{-1})\left(\mc_\theta\big(R_1(g)(a)\big), (v,\lambda)\right) \right\rangle &= d^{t^\ast\nabla}\theta(a,v) - \eta_\nabla(g_T.a)\lambda\\
		&= \left\langle (\iota_a d^{t^\ast\nabla}\theta , -\eta_\nabla(a)), (v,\lambda)\right\rangle \\
		&= \left\langle \mc_\theta(a), (v,\lambda) \right\rangle,
	\end{align*}
	where we used that, by Lemma \ref{lemma:restriction_fnabla} and Equation \eqref{eq:whstf},
	\begin{equation*}
		\eta_\nabla(g_T.a)= \eta_\nabla(a) ,
	\end{equation*}
	when $a$ is in the kernel of $\rho$. This concludes the proof.
\end{proof}

Using Lemma \ref{lemma:Moritacurvatureorbit} we have an alternative proof of Theorem \ref{theor:Mor_inv_+1-shifted_cont} more similar to the analogous proof in the $0$-shifted case. We discuss this in detail in the following
\begin{rem}
	Let $(F,f)\colon (L'\to H)\to (L\to G)$ be a VB-Morita map between LBGs, let $\theta\in \Omega^1(M,L_M)$ be a multiplicative $L_M$-valued $1$-form, and let $\nabla$ be a connection on $L_M$. Put $\theta'=F^\ast \theta\in \Omega^1(N,L'_N)$ and $\nabla'=f^\ast\nabla$. For any $y\in N$, the Morita curvatures of $\theta$ and $\theta'$, determined by $\nabla$ and $\nabla'$ respectively, fit in the following diagram
	\begin{equation}
		\label{diag:MC}
		\resizebox{\textwidth}{!}{
			\begin{tikzcd}[ampersand replacement=\&]
				0 \arrow[rr] \& \& A_{H,y} \arrow[rr] \arrow[dd, "\mc_{\theta'}", near start] \arrow[dr, "df"] \& \& T_yN\oplus L'_{N,y} \arrow[rr] \arrow[dd, "\mc_{\theta'}", near start] \arrow[dr, "(df {,} F)"] \& \&  0 \\
				\& 0 \arrow[rr, crossing over] \& \& A_{G,f(y)} \arrow[rr, crossing over] \& \& T_{f(y)}M\oplus L_{M,f(y)} \arrow[rr, crossing over]\& \& 0 \\
				0 \arrow[rr] \& \& T_y^{\dag}N\oplus \mathbbm{R} \arrow[rr] \& \& A_{H,y}^{\dag} \arrow[rr] \& \& 0 \\
				\& 0 \arrow[rr] \& \& T_{f(y)}^{\dag}M\oplus \mathbbm{R} \arrow[from=uu, crossing over, "\mc_\theta", near start] \arrow[rr] \arrow[ul, "(df^\dagger {,} \operatorname{id}_{\mathbbm{R}})"] \& \& A^{\dag}_{G,f(y)} \arrow[from=uu, crossing over, "\mc_\theta", near start] \arrow[rr] \arrow[ul, "df^{\dag}"] \& \& 0 
		\end{tikzcd}},
	\end{equation}
	where the diagonal arrows upstairs are the cochain map determined by $\mathsf{F}\colon \mk_{\theta'}\to \mk_{\theta}$ (see Proposition \ref{prop:+1_Mor_ker_Mor_inv}), and the diagonal arrows downstairs are the cochain map determined by $\mathsf{F}^\dagger\colon f^\ast\mk_{\theta}\to \mk_{\theta'}$ (see Remark \ref{rem:twisted_dual_VBGmorphism}).
	
	Diagram \eqref{diag:MC} commutes. Indeed, first we have
	\[
		d^{t^\ast\nabla'}\theta'= d^{f^\ast(t^\ast\nabla)} F^\ast \theta = F^\ast \big(d^{t^\ast\nabla}\theta\big).
 	\]
	Next, for any $a\in A_y$, we get
	\begin{align*}
		\mathsf{F}^\dagger \big(\mc_\theta(\mathsf{F}(a))\big)&=  \mathsf{F}^\dagger \left(\iota_{df(a)} d^{t^\ast\nabla}\theta, -\eta_{\nabla}(df(a))\right)\\
		& = \left(df^\dagger\big(\iota_{df(a)} d^{t^\ast\nabla}\theta\big), -\eta_{\nabla}(df(a))\right),
	\end{align*}
	and
	\begin{align*}
		\mc_{\theta'}(a)= \left(\iota_a d^{t^\ast\nabla'} \theta', -\eta_{\nabla'}(a)\right)= \left(\iota_a\big(F^\ast(d^{t^\ast\nabla}\theta)\big), -\eta_{\nabla}(df(a))\right),
	\end{align*}
	where, in the last step, we used Lemma \ref{lemma:restriction_fnabla} and Remark \ref{rem:f_nabla_fiberwise}. Finally, for any $v\in T_yN$, we have
	\begin{align*}
		\left\langle df^\dagger\big(\iota_{df(a)} d^{t^\ast\nabla}\theta\big), v\right \rangle &= F_y^{-1}\left(d^{t^\ast\nabla}\theta (df(a), df(v))\right)\\
		& = F^\ast (d^{t^\ast\nabla}\theta) (a,v) = \left\langle \iota_a \left(F^\ast (d^{t^\ast\nabla}\theta)\right), v\right\rangle,
	\end{align*}
	and so the left vertical square in Diagram \eqref{diag:MC} commutes. The right vertical one is just (up to a sign) the twisted dual of the right one, and so it commutes as well.
	
	By Theorem \ref{theo:caratterizzazioneVBmorita}, Proposition \ref{prop:+1_Mor_ker_Mor_inv} and Corollary \ref{coroll:twisted_dual_VB-Morita_map}, the diagonal arrows (upstairs and downstairs) form quasi-isomorphisms. Then $\mc_{\theta'}$ determines a quasi-isomorphism at the point $y\in N$ if and only if $\mc_{\theta}$ determines a quasi-isomorphism at the point $f(x)\in M$. Now Theorem \ref{theor:Mor_inv_+1-shifted_cont} follows from Lemma \ref{lemma:Moritacurvatureorbit},$f$ being essentially surjective and Theorem \ref{theo:caratterizzazioneVBmorita} again. 
\end{rem}

We now introduce Morita equivalence between $+1$-shifted contact LBGs.
\begin{definition}\label{def:cont_Mor_equiv}
	Two $+1$-shifted contact LBGs $(L_1, \theta_1, \kappa_1), (L_2, \theta_2, \kappa_2)$ are \emph{contact Morita equivalent} if there exist an LBG $L'$, and VB-Morita maps
	\[
	\begin{tikzcd}
		& L' \arrow[dl, "F_1"'] \arrow[dr, "F_2"] \\
		L_1 & & L_2
	\end{tikzcd}
	\]
	such that the $+1$-shifted contact structures $(F_1^\ast \theta_1, F_1^\ast\kappa_1), (F_2^\ast \theta_2, F_2^\ast\kappa_2)$ agree up to a gauge transformation.
\end{definition}

\begin{rem}
	Notice that the $2$-forms $\kappa_1, \kappa_2$ do not play any role in Definition \ref{def:cont_Mor_equiv}. We decided to keep them therein in view of the relationship between $+1$-shifted contact structures and $+1$-shifted Atiyah forms (see Theorem \ref{theor:+1_theta_omega_bij} below).
\end{rem}

In the next result we prove that contact Morita equivalence is actually an equivalence relation. We follow the same strategy showed in the symplectic and symplectic Atiyah cases (Proposition \ref{prop:sympl_Morita_equiv} and \ref{prop:Atiyah_sympl_Morita_equiv}).
\begin{prop}\label{prop:cont_Mor_equiv}
	Contact Morita equivalence is an equivalence relation.
\end{prop}
\begin{proof}
	As the $2$-forms $\kappa_1, \kappa_2$ don't play any role in Definition \ref{def:cont_Mor_equiv}, we simply ignore them, but they can be easily restored in the obvious way. Reflexivity and symmetry are obvious. For the transitivity, let $(L_1, \theta_1),  (L_2, \theta_2), (L_3, \theta_3)$ be $+1$-shifted contact LBGs, such that $(L_2,\theta_2)$ is contact Morita equivalent to $(L-1,\theta_1)$ and $(L_3,\theta_3)$. Then, there exist VB-Morita maps
	\[
	\begin{tikzcd}
		& L' \arrow[dl, "F_1"'] \arrow[dr, "F_2"]  & & L''   \arrow[dl, "K_1"'] \arrow[dr, "K_2"]  &\\
		L_1 & & L_2 & & L_3,
	\end{tikzcd}
	\]
	where $L'$ and $L''$ are LBGs, and LB-valued $1$-forms $\beta'\in \Omega^1(M',L'_{M'})$ and $\beta''\in \Omega^1(M'',L''_{M''})$, such that
	\begin{equation}\label{eq:Ctransitivity}
	F_2^\ast \theta_2 - F_1^\ast \theta_1 = \partial \beta' , \quad \text{and} \quad K_2^\ast \theta_3 - K_1^\ast \theta_2 = \partial \beta'' .
	\end{equation}	
	The \emph{homotopy fiber product} (see Example \ref{ex:homotopy_pullback_VBG}) $L'''$ of $L' \to L_2 \leftarrow L''$ is a LBG over $G''' \rightrightarrows M'''$, the homotopy fiber product of the bases (see Remark \ref{rem:hom_LBG}). Moreover, the projections $L' \leftarrow L''' \rightarrow L''$ are VB-Morita maps fitting in the following diagram:
	\begin{equation}\label{eq:homot_fiber_prod}
		\begin{tikzcd}
			& &L''' \arrow[dl, "F"'] \arrow[dr, "K"]  & & \\
			& L' \arrow[dl, "F_1"'] \arrow[dr, "F_2"]  & & L''   \arrow[dl, "K_1"'] \arrow[dr, "K_2"]  &\\
			L_1 & & L_2 & & L_3
		\end{tikzcd}
	\end{equation}
	The middle square in \eqref{eq:homot_fiber_prod} commutes up to a linear natural isomorphism $T \colon F_2 \circ F \Rightarrow K_1 \circ K$ (see Example \ref{ex:lni_homotopy}). It follows that
	\[
	(K_1 \circ K)^\ast \theta_2  - (F_2 \circ F)^\ast \theta_2 = \partial \beta
	\]
	where $\beta =- T^\ast \theta_2 \in \Omega^1 (M''', L'''_{M'''})$. Indeed, from the naturality of $T$, we have that
	\begin{equation*}
		F_2\circ F = m\circ \left(i\circ T \circ t, m\circ (K_1\circ K, T \circ s) \right),
	\end{equation*}
	then, from the multiplicativity of $\theta_2$, it follows that 
	\begin{align*}
		(F_2\circ F)^\ast \theta_2 &= \left(m\circ \left(i\circ T \circ t, m\circ (K_1\circ K, T \circ s) \right)\right)^\ast \theta_2 \\
		&=\left(i\circ T \circ t, m\circ (K_1\circ K, T \circ s) \right)^\ast (m^\ast \theta_2)\\
		&= \left(i\circ T \circ t, m\circ (K_1\circ K, T \circ s) \right)^\ast (\pr_1^\ast\theta_2 +\pr_2^\ast \theta_2)\\
		& =(i\circ T\circ t)^\ast \theta_2 + \left(m\circ (K_1\circ K, T \circ s)\right)^\ast\theta_2\\
		& = (i\circ T\circ t)^\ast \theta_2 + (K_1\circ L, T \circ s)^\ast(m^\ast\theta_2)\\
		&=(i\circ T\circ t)^\ast \theta_2 + (K_1\circ K, T \circ s)^\ast(\pr_1^\ast\theta_2 +\pr_2^\ast\theta_2)\\
		&=(i\circ T\circ t)^\ast \theta_2 + (K_1\circ K)^\ast \theta_2 + (T\circ s)^\ast \theta_2\\
		&= -(T\circ t)^\ast \theta_2  + (K_1\circ K)^\ast \theta_2 + (T\circ s)^\ast \theta_2,
	\end{align*}
	where, in the last step, we used that $i^\ast\theta_2=-\theta_2$ from $ii)$ in Proposition \ref{prop:vv_formule}. Hence
	\begin{equation}
		\label{eq:Cpartialbeta}
		(K_1\circ K)^\ast \theta_2 - (F_2\circ F)^\ast \theta_2 = (T\circ t)^\ast \theta_2 -(T\circ s)^\ast \theta_2=t^\ast T^\ast \theta_2 - s^\ast T^\ast \theta_2= \partial \beta,
	\end{equation}
	as announced.
	From \eqref{eq:Ctransitivity} and \eqref{eq:Cpartialbeta}, we have
	\begin{align*}
		(K_2\circ K)^\ast \theta_3 - (F_1\circ F)^\ast \theta_1 &= K^\ast (K_2^\ast\theta_3) - F^\ast(F_1^\ast\theta_1) \\
		&= K^\ast\left(K_1^\ast\theta_2 +\partial \beta''\right) - F^\ast \left(F_2^\ast \theta_2 -\partial \beta'\right) \\
		&=K^\ast(K_1^\ast \theta_2) + K^\ast (\partial \beta'') -F^\ast(F_2^\ast\theta_2) +F^\ast (\partial \beta') \\
		&= (K_1\circ K)^\ast \theta_2 - (F_2\circ F)^\ast \theta_2 + \partial (K^\ast \beta'') +\partial (F^\ast \beta') \\
		&=\partial \beta +\partial (K^\ast \beta'') + \partial (F^\ast \beta') \\
		&=\partial (\beta + K^\ast\beta'' +F^\ast\beta').
	\end{align*}

	Hence the $+1$-shifted contact structures $(F_1\circ F)^\ast\theta_1$ and $(K_2\circ K)^\ast\theta_3$ agree up to the gauge transformation by $\beta+K^\ast \beta'' + F^\ast \beta'$.
\end{proof}

The next result is analogous to Proposition \ref{prop:symplecticMoritaequiv} and \ref{prop:Atiyah_symplecticMoritaequiv}. In particular, we prove that if a LBG is VB-Morita equivalent to a $+1$-shifted contact LBG, then it possesses a $+1$-shifted contact structure as well.
\begin{theo}\label{theor:new_Mor_inv_+1-shifted_cont}
	Let $(L_1, \theta_1, \kappa_1)$ be a $+1$-shifted contact LBG and let $L_2$ be a VB-Morita equivalent LBG. Then there exists a $+1$-shifted contact structure $(\theta_2, \kappa_2)$ on $L_2$ such that $(L_1, \theta_1, \kappa_1)$ and $(L_2, \theta_2, \kappa_2)$ are contact Morita equivalent. Moreover the $+1$-shifted contact structure $(\theta_2, \kappa_2)$ is unique up to gauge transformations.
\end{theo}
\begin{proof}
	We will ignore $\kappa_1, \kappa_2$ as in the proof of Proposition \ref{prop:cont_Mor_equiv}. From Corollary \ref{cor:VB_Mor_equiv_LBGs} there exist an LBG $L'$ and VB-Morita maps
	\[
	\begin{tikzcd}
		& L' \arrow[dl, "F_1"'] \arrow[dr, "F_2"] \\
		L_1 & & L_2
	\end{tikzcd}
	\]
	From Remark \ref{rem:vv_forms_Morita_equiv} the assignment $F_1^\ast \colon \theta \mapsto F_1^\ast \theta$ establishes a bijection between $\partial$-cohomology classes of multiplicative $L_1$-valued and $L'$-valued $1$-forms (likewise for $F_2$). Therefore, there exists $\theta_2 \in \Omega^1 (G_2, L_2)$ and $\beta \in \Omega^1 (M', L'_{M'})$ such that $F_2^\ast \theta_2 - F_1^\ast \theta_1 = \partial \beta$. By Theorem \ref{theor:Mor_inv_+1-shifted_cont} and Proposition \ref{prop:gauge_1_cont} $\theta_2$ is a $+1$-shifted contact structure. For the uniqueness, let $\tilde{\theta}_2 \in \Omega^1 (G_2, L_2)$ be another $+1$-shifted contact structure on $L_2$ such that $(L_1, \theta_1), (L_2, \tilde \theta_2)$ are contact Morita equivalent through $L'$. Then, by Remark \ref{rem:vv_forms_Morita_equiv}, $\theta_2$ and $\tilde \theta_2$ are in the same $\partial$-cohomology class. This concludes the proof.
\end{proof}

Theorem \ref{theor:new_Mor_inv_+1-shifted_cont} motivates the following

\begin{definition}\label{def:+1_shift_stack}
	A \emph{$+1$-shifted contact structure} on the LB stack $[L_M/L]$ is a contact Morita equivalence class of $+1$-shifted contact LBGs $(L, \theta, \kappa)$ representing $[L_M/L]$.	
\end{definition}

\begin{rem}
	By Proposition \ref{prop:cont_Mor_equiv}, any LBG morphism determines a bijection between the $\partial$-cohomology classes of multiplicative $1$-forms whose Morita curvature is a VB-Morita map. Then, it is clear that two $+1$-shifted contact LBGs $(L_1, \theta_1, \kappa_1)$ and $(L_2, \theta_2, \kappa_2)$ are contact Morita equivalent if and only if $L_1$ and $L_2$ are VB-Morita equivalent and this equivalence maps the $\partial$-cohomology class $[\theta_1]$ to the $\partial$-cohomology class $[\theta_2]$. Hence, a $+1$-shifted contact structure on an LB-stack can be seen as a $\partial$-cohomology class $[\theta]$ of a multiplicative $1$-form $\theta\in \Omega^1(G,L)$ on a LBG $L$ presenting $[L_M/L]$, such that Morita curvature of $\theta$ is a VB-Morita map. So, at the level of the stack, the $2$-form $\kappa$ does not play any role.
\end{rem}

Finally, we discuss the relationship between $+1$-shifted contact structures and $+1$-shifted Atiyah forms. The following theorem is a stacky analogue of Theorem \ref{prop:corrispondenza}, and it is yet another motivation for Definitions \ref{def:+1_shift_LBG} and \ref{def:+1_shift_stack}.

\begin{theo}\label{theor:+1_theta_omega_bij}
	Let $(L \rightrightarrows L_M; G \rightrightarrows M)$ be an LBG. The assignment 
	\begin{align*}
	\Omega^\bullet (G, L) \oplus \Omega^{\bullet +1}(M, L_M) &\to \Omega^{\bullet + 1}_D(L)\oplus \OA^{\bullet+2}(L_M), \\
	(\theta, \kappa) &\mapsto \big(\omega \rightleftharpoons (\theta, \partial \kappa), \Omega \rightleftharpoons (\kappa, 0)\big),
	\end{align*}
	establishes a bijection between $+1$-shifted contact structures and $+1$-shifted symplectic Atiyah structures on $L$. This bijection intertwines gauge equivalence and contact/symplectic Morita equivalence. 
\end{theo}

\begin{proof}
	Let $\theta\in \Omega^1(G,L)$, and $\kappa \in \Omega^2 (M, L_M)$. Consider the Atiyah forms $\omega\rightleftharpoons 
	(\theta,\partial \kappa)\in \OA^2(L)$ and $\Omega \rightleftharpoons 
	(\kappa,0)\in \OA^3(L_M)$. First notice that 
	\begin{equation*}\label{eq:delta=0}
		\dA \omega = \partial \Omega \quad \text{and} \quad \dA \Omega = 0.
	\end{equation*}
	%		Conversely, if $\omega \in \OA^2 (L)$ and $\Omega \in \OA^3 (L_M)$ are such that \eqref{eq:delta=0}, then there clearly exist $\theta \in \Omega^1 (G,L)$ and $\kappa \in \Omega^2 (M, L_M)$ such that $\omega \rightleftharpoons (\theta, \partial \kappa), \Omega \rightleftharpoons (\kappa, 0)$. 
	Moreover, by Remark \ref{rem:mult_Atiyah_components}, $\theta$ is multiplicative if and only if so is 
	$\omega$. Even more, $(\theta, \kappa)$ is a $+1$-shifted contact structure
	on $L$ if and only if $(\omega, \Omega)$ is a $+1$-shifted symplectic 
	Atiyah structure. To prove the latter 
	claim it is enough to show that the Morita curvature of $(\theta, \kappa)$, equivalently of $(\theta, 0)$ (Remark \ref{rem:kappa}), 
	is a quasi-isomorphism on fibers if and only if so is $\omega$. To do this, we follow a similar 
	strategy as we did for $0$-shifted contact structures (see Theorem \ref{theor:0-contact=0-Atiyah}). So, 
	take $x \in M$. The mapping cone of \eqref{eq:Atiyah_nondegenerate_1} is
	\begin{equation}\label{eq:mapp_cone_+1-shift_At}
		\begin{tikzcd}
			0 \arrow[r] & A_x \arrow[r, "(-\mathcal{D} {,} \omega)"] & D_xL_M \oplus J^1_xL_M \arrow[r, "\omega + \mathcal{D}^{\dag}"] & A_x^{\dag} \arrow[r] &0
		\end{tikzcd}.
	\end{equation}
	Now, let $\nabla$ be a connection on $L_M$. Under the direct sum decompositions $DL_M\cong TM\oplus \mathbbm{R}_M$ and $J^1L_M\cong T^{\dag}M \oplus L_{M}$, \eqref{eq:mapp_cone_+1-shift_At} becomes
	\begin{equation*}
		\resizebox{\textwidth}{!}{
		\begin{tikzcd}[ampersand replacement=\&]
			0 \arrow[r] \& A_x \arrow[rrr, "(-\rho {,} -F_\nabla {,} d^{t^{\ast}\nabla}\theta {,} -\theta)"] \& \& \& T_xM \oplus \mathbbm{R} \oplus T_x^{\dag}M \oplus L_{M,x} \arrow[rrr, "d^{t^{\ast}\nabla}\theta + \theta^{\dag} + \rho^{\dag} + F^\dag_{\nabla}"] \& \& \& A_x^{\dag} \arrow[r] \& 0
		\end{tikzcd}}.
	\end{equation*}
	But this is exactly the mapping cone of \eqref{eq:cm_fibers_MC_+1}, whence the claim follows by Remark \ref{rem:mapping_cone}. 
	
	For the second claim, just notice that the bijection in the statement intertwines the gauge transformation (of $+1$-shifted contact structures) by $(\beta, \gamma)$ and the gauge transformation (of $+1$-shifted symplectic Atiyah forms) by $\alpha \rightleftharpoons (\beta, \gamma)$. Likewise for contact/symplectic Morita equivalence. 
\end{proof}
%
%%	\begin{rem}\label{rem:+1-shift_cont_+1-shift_Atiyah}
%	%	The bijection in Theorem \ref{theor:+1_theta_omega_bij} can be actually lifted to a bijection between $+1$-shifted contact structures and $+1$-shifted symplectic Atiyah forms on the LBG $L \rightrightarrows L_M$ (rather than on the LB stack $[L_M / L]$) but, to do this, one needs to slightly change the definition of a $+1$-shifted contact structure on $L$. The new definition is as follows: an \emph{extended $+1$-shifted contact structure} on $L$ is a pair $(\theta, \kappa)$ consisting of a $1$-form $\theta \in \Omega^1 (G, L)$ and a $2$-form $\kappa \in \Omega^2 (M, L_M)$ such that $\partial \theta = 0$ and, moreover, the Morita curvature $\mc_\theta$ of $\theta$ is a VB-Morita map. The $+1$-shifted symplectic Atiyah form corresponding to $(\theta, \kappa)$ is then $(\omega, \dA \sigma^\ast \kappa)$, where $\omega \rightleftharpoons (\theta, \partial \kappa)$. To see that $\mc_\theta$ is a VB-Morita map if and only if $\omega \colon DL \to J^1 L$ is a VB-Morita map, it is convenient to argue as in the proof of Proposition \ref{theor:+1_theta_omega_bij} but now replacing $\mc_\theta$ with the \emph{extended Morita curvature} $\mc_{(\theta, \kappa)} \colon \mk_\theta \to \mk_\theta^\dag$ defined as 
%	%	\[
%	%	\mc_{(\theta, \kappa)} =
%	%		\begin{pmatrix}
%		%			d^{t^{\ast}\nabla}\theta + \partial \kappa & \eta_{\nabla}\\
%		%			-\eta_{\nabla} & 0
%		%		\end{pmatrix},
%	%	\]
%	%	after noticing that $\mc_\theta$ is a VB-Morita map if and only if so is $\mc_{(\theta, \kappa)}$. We leave the details to the reader.
%	%	\end{rem}

\begin{rem}
	Let $(L\rightrightarrows L_M; G \rightrightarrows M)$ be an LBG equipped with a $+1$-shifted contact structure $(\theta, \kappa)$ and let $(\omega, \Omega)$ be the corresponding $+1$-shifted symplectic Atiyah structure. It then follows by dimension counting, from $\omega$ determining a quasi-isomorphism on fibers, that $\dim G = 2 \dim M +1$.
\end{rem}

\begin{rem}%\label{rem:hom_1-shift_sympl_grp}
	Let $Q \rightrightarrows P$ and $(L \rightrightarrows L_M; G \rightrightarrows M)$ be as in Remark \ref{rem:hom_0-shift_sympl_grp}. Let $\theta \in \Omega^1 (G, L)$ (resp.~$\kappa \in \Omega^2 (M, L_M)$) and let $\Theta \in \Omega^1 (Q)$ (resp.~$K \in \Omega^2 (P)$) be the corresponsing degree $1$ homogeneous $1$-form (resp.~$2$-form), given by interpreting $\Gamma (L)$ (resp.~$\Gamma (L_M)$) as degree $1$ homogeneous functions on $Q$ (resp.~$P$). Similarly as in the $0$-shifted case (Remark \ref{rem:hom_0-shift_sympl_grp}), the assignment $(\theta, \kappa) \mapsto (d\Theta, dK)$ establishes a bijection between $+1$-shifted contact structures on $L$ and homogeneous $+1$-shifted symplectic structures of degree $1$ on $Q \rightrightarrows P$. This can be seen using Theorem \ref{theor:+1_theta_omega_bij} and the relationship between degree $1$ homogeneous differential forms and Atiyah forms (see Section \ref{sec:dictionary}).
\end{rem}

We conclude this chapter with some examples of $+1$-shifted contact structures.

\begin{example}
	Let $(G, K)$ be a \emph{contact groupoid} in the sense of Dazord \cite{Da97}, i.e.~$K \subseteq TG$ is a contact distribution and a Lie subgroupoid of the tangent groupoid $TG \rightrightarrows TM$. Then $L = TG/K$ is an LBG and the projection $\theta \colon TG \to L$ is a multiplicative contact form (see, e.g., \cite{CSS15}, see also Section \ref{sec:dictionary}), hence a $+1$-shifted contact structure.
\end{example}

\begin{example}[Trivial $+1$-shifted contact structure]
	Let $L_M \to M$ be a line bundle. Consider the general linear groupoid $G := \operatorname{GL}(L_M) \rightrightarrows M$ of $L_M$ and denote by $(L \rightrightarrows L_M; G \rightrightarrows M)$ the LBG coming from the tautological action of $G$ on $L_M$. The core complex of $DL \rightrightarrows DL_M$ is
	\[
	\begin{tikzcd}
		0 \arrow[r] & DL_M \arrow[r, equal] & DL_M \arrow[r] & 0
	\end{tikzcd}.
	\]
	Given any $1$-form $\theta_M \in \Omega^1 (M, L_M)$, and any $2$-form $\kappa \in \Omega^2 (M, L_M)$, the pair $(\partial \theta_M , \kappa)$ is a $+1$-shifted contact structure on $L$, which is actually trivial up to gauge transformations.
\end{example}

\begin{example}[Prequantization of $+1$-shifted symplectic structures]
	Recall from \cite{BXu03,LGXu05} the notion of \emph{prequantization of a $+1$-shifted symplectic structure}. Let $G\rightrightarrows M$ be a Lie groupoid. An \emph{$S^1$-central extension of $G$} is a Lie groupoid $H\rightrightarrows M$ together with a Lie groupoid morphism $\pi\colon H\to G$ being the identity on objects, and an $S^1$-action on $H$, making $\pi\colon H\to G$ a principal \emph{$S^1$-bundle groupoid}, i.e.~the following compatibility between the principal action and the groupoid structure holds: if we denote by $\star$ the multiplication in $S^1$, and by a dot $.$ the $S^1$-action, then $(\phi .h)(\phi'.h')=(\phi \star \phi').hh'$ for all $\phi, \phi'\in S^1$, and $(h,h')\in H^{(2)}$.
	
	Let $\pi \colon H \to G$ be an $S^1$-central extension. Then a \emph{pseudoconnection} in $H$ is a pair $(\theta, \kappa)$ consisting of a $1$-form $\theta \in \Omega^1(H)$ on $H$ and a $2$-form $\kappa \in \Omega^2(M)$ on $M$, such that $\theta$ is a principal connection $1$-form on $H$. Finally, let $(\omega, \Omega)$ be a $+1$-shifted symplectic structure on $G$. A  \emph{prequantization of $(\omega, \Omega)$} is an $S^1$-central extension $\pi\colon H \to G$ with a pseudo-connection $(\theta, \kappa)$, such that 
	\begin{equation}\label{eq:prequant}
		\partial \theta=0, \quad d\theta=\partial \kappa - \pi^{\ast}\omega, \quad \text{and} \quad d\kappa=\Omega.
	\end{equation}
	A prequantization $(\theta, \kappa)$ exists if $(\omega, \Omega)$ is an \emph{integral cocycle} in the total complex of the Bott-Shulmann-Stasheff double complex \cite[Proposition 3.3]{BXu03}.
	
	We can regard $\theta$ as a $1$-form with values in the \emph{trivial LBG} $\mathbbm R_G \to G$, i.e.~the LBG corresponding to the trivial action of $G$ on the trivial line bundle $\mathbbm R_M$. Then, by the first condition in Equation \eqref{eq:prequant}, $\theta$ is multiplicative. By definition, $\theta_h\neq 0$ for all $h\in H$ and its curvature is $R_{\theta} = d\theta|_{K_\theta} \colon K_{\theta} \to K_{\theta}^{\dag}$ which is a VB-Morita map. To see this, denote by $A_H, A_G$ the Lie algebroids of $H, G$. The fiber of $K_\theta$ over $x \in M$ is
	\begin{equation}\label{eq:fiber_K_theta_prequant}
		\begin{tikzcd}
			0 \arrow[r] & A_{H, x} \cap K_{\theta, x} \arrow[r, "\rho"] & T_x M \arrow[r] & 0
		\end{tikzcd}.
	\end{equation}
	If we use $d\pi \colon TH \to TG$ to identify $K_\theta$ with $\pi^\ast H$, then \eqref{eq:fiber_K_theta_prequant} identifies with
	\begin{equation}
		\begin{tikzcd}
			0 \arrow[r] & A_{G, x} \arrow[r, "\rho"] & T_x M \arrow[r] & 0
		\end{tikzcd}.
	\end{equation}
	Now, from the second one of \eqref{eq:prequant}, the curvature $R_\theta$ determines the following cochain map on fibers:
	\begin{equation}\label{eq:R_theta_fiber_prequant}
		\begin{tikzcd}
			0 \arrow[r] & A_{G, x} \arrow[r, "\rho"] \arrow[d, "- \omega - \kappa \circ \rho"']& T_x M \arrow[r] \arrow[d, "- \omega- \rho^\ast \circ \kappa"] & 0 \\
			0 \arrow[r] & T_x^\ast M \arrow[r, "\rho^\ast"'] & A_{G, x}^\ast \arrow[r] & 0
		\end{tikzcd},
	\end{equation}
	As $(\omega, \Omega)$ is a $+1$-shifted symplectic structure, then $\omega \colon TG \to T^\ast G$ is a VB-Morita map and the vertical arrows in \eqref{eq:R_theta_fiber_prequant} are a quasi-isomorphism. We conclude that $R_\theta$ is a VB-Morita map as claimed and $(\theta, \kappa)$ is a $+1$-shifted contact structure.
\end{example}

\begin{example}[Dirac-Jacobi structures]
	$+1$-shifted symplectic structures are the global structures on Lie groupoids integrating Dirac structures twisted by a closed $3$-form, seen as infinitesimal structures on the corresponding Lie algebroid (see Section \ref{sec:twisted}). Similarly, $+1$-shifted contact structures are the global structures integrating \emph{Dirac-Jacobi structures} \cite{V18}.
	
	Let $L_M \to M$ be a line bundle. The \emph{omni-Lie algebroid} \cite{CL10} of $L_M$ is the VB $\mathbbm{D}L_M:= DL_M \oplus J^1L_M$ together with the following structures:
	\begin{itemize}
		\item the projection $\pr_D \colon \mathbbm DL_M \to DL_M$ onto $DL_M$;
		\item the non-degenerate, symmetric $L$-valued bilinear form of split signature:
		\begin{equation*}
			\big\langle\hspace{-3pt}\big\langle -, - \big\rangle\hspace{-3pt}\big\rangle \colon \mathbbm{D}L_M \otimes \mathbbm{D}L_M \to L_M, \quad \big\langle\hspace{-3pt}\big\langle (\delta ,\psi), (\delta' , \psi') \big\rangle\hspace{-3pt}\big\rangle= \langle \psi, \delta'\rangle + \langle \psi', \delta\rangle;
		\end{equation*}
		\item the bracket on sections:
		\begin{equation*}
			\big[\hspace{-3pt}\big[-,-\big]\hspace{-3pt}\big]\colon \Gamma (\mathbbm D L_M) \times \Gamma (\mathbbm D L_M)  \to \Gamma (\mathbbm DL_M),
		\end{equation*}
		defined by
		\[
			\big[\hspace{-3pt}\big[(\Delta,\psi), (\Delta',\psi')\big]\hspace{-3pt}\big]:= \Big([\Delta,\Delta'], \mathcal{L}_{\Delta}\psi' - \iota_{\Delta'}\dA\psi\Big),
		\]
		where $\mathcal L_\Delta$ is the \emph{Lie algebroid Lie derivative} along $\Delta$.
	\end{itemize}
	The omni-Lie algebroid is an instance of an \emph{$E$-Courant algebroid} \cite{CLS10}, an \emph{$AV$-Courant algebroid} \cite{Bl11} and a \emph{contact-Courant algebroid} \cite{Gr13}, and can be regarded as a \emph{contact version} of the standard Courant algebroid: the generalized tangent bundle $\mathbbm T M = TM \oplus T^\ast M$.
	
	A \emph{Dirac-Jacobi structure} on $L_M$ is a vector subbundle $\mathbbm L \subseteq \mathbbm DL_M$ which is Lagrangian with respect to the inner product $\langle\!\langle -, - \rangle\!\rangle$ and whose sections are preserved by the bracket $[\![-,-]\!]$ (\cite{V18}, see also \cite{Wa00,Wa04}). A Dirac-Jacobi structure is a contact version of a Dirac structure (see Definition \ref{def:dirac}). Any Dirac-Jacobi structure $\mathbbm L \to M$ is a Lie algebroid, with anchor given by $\rho := \sigma \circ \pr_D \colon \mathbbm L \to TM$, and Lie bracket given by the restriction of $[\![-,-]\!]$. Moreover $\mathbbm L$ acts on $L_M$ via $\pr_D \colon \mathbbm L \to DL_M$. Dirac-Jacobi structures encompass contact and pre-contact structures, flat line bundles, locally conformally symplectic and locally conformally pre-symplectic structures, Jacobi structures, Poisson and Dirac structures as distinguished examples. Additionally, generalized complex structures in odd dimensions (aka generalized contact structures \cite{VW16,SV20}) can be seen as certain complex Dirac-Jacobi structures. This shows the wide range of applications of their theory.
	
	One can also define \emph{twisted Dirac-Jacobi structures}, in the same spirit as twisted Dirac structures discussed in Section \ref{sec:twisted} (see \cite{dCP06} for the trivial line bundle case). To do this first notice that one can \emph{deform} the bracket $[\![-,-]\!]$ on sections of the omni-Lie algebroid $\mathbbm D L_M$ via a $\dA$-closed Atiyah $3$-form $\Omega \in \OA^3 (L_M)$ similarly to the what recalled for the Courant bracket in Section \ref{sec:twisted}: define the new \emph{deformed bracket}
	\begin{equation*}
		\big[\hspace{-3pt}\big[-,-\big]\hspace{-3pt}\big]_\Omega \colon \Gamma (\mathbbm D L_M) \times \Gamma (\mathbbm D L_M)  \to \Gamma (\mathbbm DL_M), 
	\end{equation*}
	defined by setting
	\[
		\big[\hspace{-3pt}\big[(\Delta,\psi), (\Delta',\psi')\big]\hspace{-3pt}\big]_\Omega := \big[\hspace{-3pt}\big[(\Delta,\psi), (\Delta',\psi')\big]\hspace{-3pt}\big] + \big(0, \iota_\Delta \iota_{\Delta'} \Omega\big).
	\]
	Then $(\mathbbm DL_M, \pr_D, \langle\!\langle -, - \rangle\!\rangle, [\![-,-]\!]_\Omega)$ is again a contact Courant algebroid. An \emph{$\Omega$-twisted Dirac-Jacobi structure} on $L_M$ is a vector subbundle of $\mathbbm L \subseteq \mathbbm D L_M$ which is Lagrangian with respect to $\langle\!\langle -, - \rangle\!\rangle$ and whose sections are now preserved by the deformed bracket $[\hspace{-3pt}[-,-]\hspace{-3pt}]_\Omega$. 
	
	\emph{Twisted Dirac-Jacobi structures integrate to $+1$-shifted contact structures} in the following sense. First of all, a \emph{Dirac-Jacobi algebroid} is a Lie algebroid $A \to M$ together with a Lie algebroid isomorphism $A \cong \mathbbm L \subseteq \mathbbm D L_M$ onto a twisted Dirac-Jacobi structure. Now, let $(L \rightrightarrows L_M; G \rightrightarrows M)$ be an LBG, and let $A$ be the Lie algebroid of $G$. Out of a $+1$-shifted contact structure $(\theta, \kappa)$ on $L$ one can construct a Dirac-Jacobi algebroid structure $A \cong \mathbbm L \subseteq \mathbbm D L_M$ twisted by $\Omega \rightleftharpoons (\kappa, 0)$ in a canonical way. Moreover, if $G$ is source-simply connected, then the latter construction establishes a one-to-one correspondence between $+1$-shifted contact structures on $L$ and twisted Dirac-Jacobi algebroid structures $A \cong \mathbbm L \subseteq \mathbbm D L_M$. In the untwisted case $\Omega = 0$ (i.e.~$\kappa = 0$) this is essentially \cite[Theorem 10.11]{V18} together with the simple remark that a \emph{pre-contact groupoid} in the sense of \cite{V18} is just a rephrasing of an LBG equipped with a $+1$-shifted contact structure of the form $(\theta, 0)$. In the twisted case, the proof is essentially the same and we omit it.
	
	Notice that, unlike the case of twisted Dirac structures, twisted Dirac-Jacobi structures are not really new structures with respect to untwisted Dirac-Jacobi structures (i.e.~$\Omega = 0$). The reason is essentially the acyclicity of the der-complex. Indeed, from $\dA \Omega = 0$, we get $\Omega = \dA B$, where $B = \iota_{\mathbbm I} \Omega \in \OA^2 (L_M)$. The map
	\[
	\mathbbm F_B \colon \mathbbm DL \to \mathbbm DL,\quad \mathbbm F_B (\delta, \psi) := \big(\delta, \psi - \iota_\delta B\big),
	\]
	bijectively transforms $\Omega$-twisted Dirac-Jacobi structures to untwisted Dirac-Jacobi structures. Moreover, for every $\Omega$-twisted Dirac-Jacobi structure $\mathbbm L$, the restriction $\mathbbm F_B \colon \mathbbm L \to \mathbbm F_B (\mathbbm L)$ is a Lie algebroid isomorphism identifying the infinitesimal actions on $L_M$. This shows that $\Omega$-twisted Dirac-Jacobi structures are essentially the same as untwisted Dirac-Jacobi structures, and it is actually the infinitesimal counterpart of the remark that every $+1$-shifted contact structure can be gauge transformed into a $+1$-shifted contact structure of the form $(\theta, 0)$ (Remark \ref{rem:gauge_transf}).
\end{example}
	
	\backmatter
	%%%%CAPITOLO BIBLIOGRAFIA
	\cleardoublepage
	
	%%%%%%%%%%%
	%\bibliographystyle{plain}
	\addcontentsline{toc}{chapter}{Bibliography}
	%\bibliography{bibliography}
	%\printbibliography[heading=bibintoc]
	

\end{document}